\documentclass[12pt,reqno,psamsfonts,english]{amsart}

\usepackage{mathabx}

\usepackage[initials]{amsrefs}
\usepackage{feynmp}
\usepackage{latexsym}
\usepackage{amsfonts}
\usepackage{amssymb}
\usepackage{amscd}
\usepackage{amsbsy}
\usepackage{amsmath}
\usepackage{bbm}
\usepackage{euscript}%[mathscr]
\usepackage{mathrsfs}

\usepackage{dashbox}
\usepackage{environ}

\usepackage{textcomp}

\usepackage{accents}

\usepackage[T1]{fontenc}
\usepackage[latin1]{inputenc}
\usepackage[dvips]{graphicx}
\usepackage{xcolor}
\usepackage{tikz}
\usetikzlibrary{shapes}
\usetikzlibrary{decorations.markings}
\usetikzlibrary{decorations.pathreplacing}
\usetikzlibrary{patterns}
\usetikzlibrary{matrix,arrows}
\usepgflibrary{decorations.pathmorphing}

\usepackage{tikz-cd}

\usepackage{url}

\DefineSimpleKey{bib}{myurl}

\newcommand\myurl[1]{\url{#1}}

\BibSpec{webpage}{%
  +{}{\PrintAuthors} {author}
  +{,}{ \textit} {title}
  +{}{ \parenthesize} {date}
  +{,}{ \myurl} {myurl}
  +{,}{ } {note}
  +{.}{ } {transition}
}
\usepackage[all]{xy}
\xyoption{knot}
\xyoption{graph}
\xyoption{tile}
\xyoption{arc}
\xyoption{poly}
\usepackage{setspace}

\def\ie{{\it i.e.\ }}
\def\eg{{\it e.g.\ }}
\def\cf{{\it cf.\ }}
\def\rhs{{\it r.h.s.\ }}
\def\lhs{{\it l.h.s.\ }}

\def\End{\mathop{{\rm End}}\nolimits}
\def\Der{\mathop{{\rm Der}}\nolimits}
\def\Hom{\mathop{{\rm Hom}}\nolimits}

\def\im{\mathop{{\rm im}}\nolimits}

\def\deg{ \mathop{{\rm deg}}\nolimits }
\def\p{^{\prime}}

\def\rank{\mathop{{\rm rank}}\nolimits}

\def\mod#1{\;(\promod #1)}

\def\End{\mathop{{\rm End}}\nolimits}

\def\pr#1#2{ \noindent{\em Proof of #1~\ref{#2}.} }

\def\qed{ \hfill $\Box$ }

\def\lrbc#1{ \left( #1 \right) }

\parskip=\medskipamount                 % space between paragraphs
\thicklines                     % thick straight lines for pictures
\def\inbar{\vrule height1.5ex width.4pt depth0pt}
\def\IC{\relax\,\hbox{$\inbar\kern-.3em{\rm C}$}}
\def\IN{\relax{\rm I\kern-.18em N}}
\def\IQ{\relax\,\hbox{$\inbar\kern-.3em{\rm Q}$}}
\def\IR{\relax{\rm I\kern-.18em R}}
\def\ZZ{\relax{\sf Z\kern-.4em Z}}
\newtheorem{theorem}{Theorem}[section]
\newtheorem{proposition}[theorem]{Proposition}
\newtheorem{corollary}[theorem]{Corollary}
\newtheorem{conjecture}[theorem]{Conjecture}
\newtheorem{lemma}[theorem]{Lemma}
\newtheorem{definition}[theorem]{Definition}
\newtheorem{remark}[theorem]{Remark}

\catcode`\@=11

\newif\if@fewtab\@fewtabtrue

\catcode`\@=11

\newif\if@fewtab\@fewtabtrue

{\count255=\time\divide\count255 by 60
\xdef\hourmin{\number\count255} \multiply\count255
by-60\advance\count255 by\time
\xdef\hourmin{\hourmin:\ifnum\count255<10 0\fi\the\count255}}
\def\ps@draft{\let\@mkboth\@gobbletwo
    \def\@oddhead{}
    \def\@oddfoot
      {\hbox to 7 cm{\footnotesize {\em Draft of \jobname:} \draftdate
       \hfil}\hskip -7cm\hfil\rm\thepage \hfil}
    \def\@evenhead{}\let\@evenfoot\@oddfoot}

\def\ceqno{\global\@fewtabfalse
    \ifcase\@eqcnt \def\@tempa{& & &}\or \def\@tempa{& &}
      \or \def\@tempa{&}
      \or\def\@tempa{}\fi\@tempa
{\rm(\theequation)}}

\def\aeqno#1{\global\@fewtabfalse
    \ifcase\@eqcnt \def\@tempa{& & &}\or \def\@tempa{& &}
      \or \def\@tempa{&}
      \or\def\@tempa{}\fi\@tempa
{\rm(\theequation,#1)}}

\def\label#1{\ifnum\draftcontrol=1
 \global\def\draftnote{$\scriptstyle #1$}\fi
 \@bsphack\if@filesw {\let\thepage\relax
   \def\protect{\noexpand\noexpand\noexpand}%
\xdef\@gtempa{\write\@auxout{\string
      \newlabel{#1}{{\@currentlabel}{\thepage}}}}}\@gtempa
   \if@nobreak \ifvmode\nobreak\fi\fi\fi
  \@esphack}

\def\alabel#1#2{\label{#1}\global\@fewtabfalse
    \ifcase\@eqcnt \def\@tempa{& & &}\or \def\@tempa{& &}
      \or \def\@tempa{&}
      \or\def\@tempa{}\fi\@tempa
{\hbox to 3cm{\phantom{\rm(\theequation,#2)} \draftnote
\hfil}\hskip -3cm {\rm(\theequation,#2)}}}

\def\clabel#1{\label{#1}\global\@fewtabfalse
    \ifcase\@eqcnt \def\@tempa{& & &}\or \def\@tempa{& &}
      \or \def\@tempa{&}
      \or\def\@tempa{}\fi\@tempa
{\hbox to 3cm{\phantom{\rm(\theequation)} \draftnote \hfil}\hskip
-3cm{\rm(\theequation)}}}

\def\eqnarray{\def\draftnote{{}}\global\@fewtabtrue
\stepcounter{equation}\let\@currentlabel=\theequation
\global\@eqnswtrue
\global\@eqcnt\z@\tabskip\@centering\let\\=\@eqncr
$$\halign to \displaywidth\bgroup\@eqnsel\hskip\@centering\@eqcnt\z@
  $\displaystyle\tabskip\z@{##}$&\global\@eqcnt\@ne
  \hskip 1\arraycolsep \hfil$\displaystyle{##}$\hfil
  &\global\@eqcnt\tw@ \hskip 1\arraycolsep
$\displaystyle\tabskip\z@{##}$ \hfil
\tabskip\@centering&\global\@eqcnt\thr@@\llap{##}\tabskip\z@ \cr}

\def\endeqnarray{\@@eqncr\egroup
      \global\advance\c@equation\m@ne$$\global\@ignoretrue}

\def\@eqnnum{\hbox to 3cm{\phantom{\rm(\theequation)} \draftnote
                         \hfil}\hskip -3cm {\rm(\theequation)}}

\def\@@eqncr{\let\@tempa\relax
    \ifcase\@eqcnt \def\@tempa{& & &}\or \def\@tempa{& &}
      \or \def\@tempa{&}
      \or\def\@tempa{}
\fi\@tempa \if@eqnsw \if@fewtab\@eqnnum\fi
\stepcounter{equation}\fi\global
\@eqnswtrue\global\@eqcnt\z@\global\@fewtabtrue\cr}

\def\draftcite#1{\ifnum\draftcontrol=1#1\else{}\fi}

\def\@lbibitem[#1]#2{\item{}\hskip -3cm \hbox to 2cm
{\hfil$\scriptstyle\draftcite{#2}$}\hskip
1cm[\@biblabel{#1}]\if@filesw
     {\def\protect##1{\string ##1\space}\immediate
      \write\@auxout{\string\bibcite{#2}{#1}}}\fi\ignorespaces}

\def\@bibitem#1{\item\hskip -3cm \hbox to 2cm
{\hfil $\scriptstyle\draftcite{#1}$}\hskip 1cm \if@filesw
\immediate\write\@auxout
       {\string\bibcite{#1}{\the\value{\@listctr}}}\fi\ignorespaces}

\def\draftdate{\number\month/\number\day/\number\year\ \ \ \hourmin }
 \global\def\draftcontrol{0}
\catcode`\@=12

\def\theequation{{\thesection.\arabic{equation}}}
\def\qq{\begin{eqnarray}}
\def\qqq{\end{eqnarray}}
\def\ee{\begin{eqnarray}}
\def\eee{\end{eqnarray}}
\def\rx#1{~(\ref{#1})}
\def\rxw#1{(\ref{#1})}
\def\ex#1{eq.\hspace*{-3pt}\rx{#1}}
\def\eex#1{eqs.\hspace*{-3pt}\rx{#1}}
\def\cx#1{~\cite{#1}}
\def\rw#1{~\ref{#1}}

\def\xlee#1{ \begin{eqnarray} \label{#1} }
\def\xeee{ \end{eqnarray} }
\def\ylee#1{ \begin{eqnarray}\nonumber }
\def\yeee{ \end{eqnarray} }
\def\zlee#1{ \begin{displaymath} }
\def\zleee{ \end{displaymath} }
\def\wlee#1{ $ }
\def\weee{ $ }

\def\fg#1{Fig.~\ref{#1}}

\newlength{\shiftwidth}
\def\shift#1{&&\hbox to \shiftwidth{\hfill $\displaystyle#1$}}
\newlength{\sshiftwidth}
\def\sshift#1{\lefteqn{\hbox to
\sshiftwidth{\hfill$\displaystyle#1$}}}
\def\qbezier{\bezier{120}}
\def\DottedCircle{
\bezier{4}(0.966,-0.259)(1.04,0)(0.966,0.259)
\bezier{4}(0.966,0.259)(0.897,0.518)(0.707,0.707)
\bezier{4}(0.707,0.707)(0.518,0.897)(0.259,0.966)
\bezier{4}(0.259,0.966)(0,1.04)(-0.259,0.966)
\bezier{4}(-0.259,0.966)(-0.518,0.897)(-0.707,0.707)
\bezier{4}(-0.707,0.707)(-0.897,0.518)(-0.966,0.259)
\bezier{4}(-0.966,0.259)(-1.04,0)(-0.966,-0.259)
\bezier{4}(-0.966,-0.259)(-0.897,-0.518)(-0.707,-0.707)
\bezier{4}(-0.707,-0.707)(-0.518,-0.897)(-0.259,-0.966)
\bezier{4}(-0.259,-0.966)(0,-1.04)(0.259,-0.966)
\bezier{4}(0.259,-0.966)(0.518,-0.897)(0.707,-0.707)
\bezier{4}(0.707,-0.707)(0.897,-0.518)(0.966,-0.259) }
\def\Endpoint[#1]{
\ifcase#1 \put(1,0){\circle*{0.15}}
\or\put(0.866,0.5){\circle*{0.15}}
\or\put(0.5,0.866){\circle*{0.15}} \or\put(0,1){\circle*{0.15}}
\or\put(-0.5,0.866){\circle*{0.15}}
\or\put(-0.866,0.5){\circle*{0.15}} \or\put(-1,0){\circle*{0.15}}
\or\put(-0.866,-0.5){\circle*{0.15}}
\or\put(-0.5,-0.866){\circle*{0.15}} \or\put(0,-1){\circle*{0.15}}
\or\put(0.5,-0.866){\circle*{0.15}}
\or\put(0.866,-0.5){\circle*{0.15}} \fi}
\def\Arc[#1]{
\thicklines         % this can be changed!
\ifcase#1 \bezier{25}(0.966,-0.259)(1.04,0)(0.966,0.259) \or
\bezier{25}(0.966,0.259)(0.897,0.518)(0.707,0.707) \or
\bezier{25}(0.707,0.707)(0.518,0.897)(0.259,0.966) \or
\bezier{25}(0.259,0.966)(0,1.04)(-0.259,0.966) \or
\bezier{25}(-0.259,0.966)(-0.518,0.897)(-0.707,0.707) \or
\bezier{25}(-0.707,0.707)(-0.897,0.518)(-0.966,0.259) \or
\bezier{25}(-0.966,0.259)(-1.04,0)(-0.966,-0.259) \or
\bezier{25}(-0.966,-0.259)(-0.897,-0.518)(-0.707,-0.707) \or
\bezier{25}(-0.707,-0.707)(-0.518,-0.897)(-0.259,-0.966) \or
\bezier{25}(-0.259,-0.966)(0,-1.04)(0.259,-0.966) \or
\bezier{25}(0.259,-0.966)(0.518,-0.897)(0.707,-0.707) \or
\bezier{25}(0.707,-0.707)(0.897,-0.518)(0.966,-0.259) \fi}
\def\DottedArc[#1]{
\ifcase#1 \bezier{4}(0.966,-0.259)(1.04,0)(0.966,0.259) \or
\bezier{4}(0.966,0.259)(0.897,0.518)(0.707,0.707) \or
\bezier{4}(0.707,0.707)(0.518,0.897)(0.259,0.966) \or
\bezier{4}(0.259,0.966)(0,1.04)(-0.259,0.966) \or
\bezier{4}(-0.259,0.966)(-0.518,0.897)(-0.707,0.707) \or
\bezier{4}(-0.707,0.707)(-0.897,0.518)(-0.966,0.259) \or
\bezier{4}(-0.966,0.259)(-1.04,0)(-0.966,-0.259) \or
\bezier{4}(-0.966,-0.259)(-0.897,-0.518)(-0.707,-0.707) \or
\bezier{4}(-0.707,-0.707)(-0.518,-0.897)(-0.259,-0.966) \or
\bezier{4}(-0.259,-0.966)(0,-1.04)(0.259,-0.966) \or
\bezier{4}(0.259,-0.966)(0.518,-0.897)(0.707,-0.707) \or
\bezier{4}(0.707,-0.707)(0.897,-0.518)(0.966,-0.259) \fi}
\def\Chord[#1,#2]{
\thinlines \ifnum#1>#2\Chord[#2,#1] \else\ifnum#1<#2 \ifcase#1
\ifcase#2 \or\qbezier(1,0)(0.516,0.138)(0.866,0.5)
\or\qbezier(1,0)(0.45,0.26)(0.5,0.866)
\or\qbezier(1,0)(0.327,0.327)(0,1)
\or\qbezier(1,0)(0.179,0.311)(-0.5,0.866)
\or\qbezier(1,0)(0.0536,0.2)(-0.866,0.5) \or\put(1, 0){\line(-2,
0){2}} \or\qbezier(1,0)(0.0536,-0.2)(-0.866,-0.5)
\or\qbezier(1,0)(0.179,-0.311)(-0.5,-0.866)
\or\qbezier(1,0)(0.327,-0.327)(0,-1)
\or\qbezier(1,0)(0.45,-0.26)(0.5,-0.866)
\or\qbezier(1,0)(0.516,-0.138)(0.866,-0.5) \fi \or\ifcase#2\or
\or\qbezier(0.866,0.5)(0.378,0.378)(0.5,0.866)
\or\qbezier(0.866,0.5)(0.26,0.45)(0,1)
\or\qbezier(0.866,0.5)(0.12,0.446)(-0.5,0.866)
\or\qbezier(0.866,0.5)(0,0.359)(-0.866,0.5)
\or\qbezier(0.866,0.5)(-0.0536,0.2)(-1,0) \or\put(0.866,
0.5){\line(-5, -3){1.73}}
\or\qbezier(0.866,0.5)(0.146,-0.146)(-0.5,-0.866)
\or\qbezier(0.866,0.5)(0.311,-0.179)(0,-1)
\or\qbezier(0.866,0.5)(0.446,-0.12)(0.5,-0.866)
\or\qbezier(0.866,0.5)(0.52,0)(0.866,-0.5) \fi \or\ifcase#2\or\or
\or\qbezier(0.5,0.866)(0.138,0.516)(0,1)
\or\qbezier(0.5,0.866)(0,0.52)(-0.5,0.866)
\or\qbezier(0.5,0.866)(-0.12,0.446)(-0.866,0.5)
\or\qbezier(0.5,0.866)(-0.179,0.311)(-1,0)
\or\qbezier(0.5,0.866)(-0.146,0.146)(-0.866,-0.5) \or\put(0.5,
0.866){\line(-3, -5){1}} \or\qbezier(0.5,0.866)(0.2,-0.0536)(0,-1)
\or\qbezier(0.5,0.866)(0.359,0)(0.5,-0.866)
\or\qbezier(0.5,0.866)(0.446,0.12)(0.866,-0.5) \fi
\or\ifcase#2\or\or\or \or\qbezier(0,1.)(-0.138,0.516)(-0.5,0.866)
\or\qbezier(0,1.)(-0.26,0.45)(-0.866,0.5)
\or\qbezier(0,1.)(-0.327,0.327)(-1,0)
\or\qbezier(0,1.)(-0.311,0.179)(-0.866,-0.5)
\or\qbezier(0,1.)(-0.2,0.0536)(-0.5,-0.866) \or\put(0, 1){\line(0,
-2){2}} \or\qbezier(0,1.)(0.2,0.0536)(0.5,-0.866)
\or\qbezier(0,1.)(0.311,0.179)(0.866,-0.5) \fi
\or\ifcase#2\or\or\or\or
\or\qbezier(-0.5,0.866)(-0.378,0.378)(-0.866,0.5)
\or\qbezier(-0.5,0.866)(-0.45,0.26)(-1,0)
\or\qbezier(-0.5,0.866)(-0.446,0.12)(-0.866,-0.5)
\or\qbezier(-0.5,0.866)(-0.359,0)(-0.5,-0.866)
\or\qbezier(-0.5,0.866)(-0.2,-0.0536)(0,-1) \or\put(-0.5,
0.866){\line(3, -5){1}}
\or\qbezier(-0.5,0.866)(0.146,0.146)(0.866,-0.5) \fi
\or\ifcase#2\or\or\or\or\or
\or\qbezier(-0.866,0.5)(-0.516,0.138)(-1,0)
\or\qbezier(-0.866,0.5)(-0.52,0)(-0.866,-0.5)
\or\qbezier(-0.866,0.5)(-0.446,-0.12)(-0.5,-0.866)
\or\qbezier(-0.866,0.5)(-0.311,-0.179)(0,-1)
\or\qbezier(-0.866,0.5)(-0.146,-0.146)(0.5,-0.866) \or\put(-0.866,
0.5){\line(5, -3){1.73}} \fi \or\ifcase#2\or\or\or\or\or\or
\or\qbezier(-1,0)(-0.516,-0.138)(-0.866,-0.5)
\or\qbezier(-1,0)(-0.45,-0.26)(-0.5,-0.866)
\or\qbezier(-1,0)(-0.327,-0.327)(0,-1)
\or\qbezier(-1,0)(-0.179,-0.311)(0.5,-0.866)
\or\qbezier(-1,0)(-0.0536,-0.2)(0.866,-0.5) \fi
\or\ifcase#2\or\or\or\or\or\or\or
\or\qbezier(-0.866,-0.5)(-0.378,-0.378)(-0.5,-0.866)
\or\qbezier(-0.866,-0.5)(-0.26,-0.45)(0,-1)
\or\qbezier(-0.866,-0.5)(-0.12,-0.446)(0.5,-0.866)
\or\qbezier(-0.866,-0.5)(0,-0.359)(0.866,-0.5) \fi
\or\ifcase#2\or\or\or\or\or\or\or\or
\or\qbezier(-0.5,-0.866)(-0.138,-0.516)(0,-1)
\or\qbezier(-0.5,-0.866)(0,-0.52)(0.5,-0.866)
\or\qbezier(-0.5,-0.866)(0.12,-0.446)(0.866,-0.5) \fi
\or\ifcase#2\or\or\or\or\or\or\or\or\or
\or\qbezier(0,-1.)(0.138,-0.516)(0.5,-0.866)
\or\qbezier(0,-1.)(0.26,-0.45)(0.866,-0.5) \fi
\or\ifcase#2\or\or\or\or\or\or\or\or\or\or
\or\qbezier(0.5,-0.866)(0.378,-0.378)(0.866,-0.5) \fi\fi\fi\fi}
\def\FullChord[#1,#2]{
\Endpoint[#1] \Endpoint[#2] \Arc[#1] \Arc[#2] \Chord[#1,#2] }
\def\EndChord[#1,#2]{
\Endpoint[#1] \Endpoint[#2] \Chord[#1,#2] }
\def\Picture#1{
\begin{picture}(2,1)(-1,-0.167)
#1
\end{picture}
}
\def\DottedChordDiagram[#1,#2]{
\Picture{\DottedCircle \FullChord[#1,#2]} }
\def\ZZ{ \mathbb{Z} }
\def\IQ{ \mathbb{Q} }
\def\IC{ \mathbb{C} }
\def\IR{ \mathbb{R} }

\def\bfx{ \mathbf{x} }
\def\bfy{ \mathbf{y} }

\def\hlf{ {1\over 2} }

\def\xS{ \mathbb{S} }
\def\xSv#1{ \xS^{#1} }
\def\xSo {\xSv{1} }
\def\xSt{\xSv{2} }
\def\xStSo{\xSt\times\xSo}

\def\xD{ \mathbb{D} }
\def\xDp{ \xD' }
\def\xB{ \mathbb{B} }
\def\xT{ \mathbb{T} }
\def\xTv#1{ \xT^{#1} }
\def\xTt{ \xTv{2} }

\def\yT{ \mathrm{T} }
\def\yTvv#1#2{ \yT_{#1,#2} }
\def\yTmn{ \yTvv{m}{n} }
\def\yTmmn{ \yTvv{m}{-n} }

\def\Hom{ \mathop{\mathrm{Hom}}\nolimits }

\def\xId{ \mathbbm{1} }

\def\hlf{ \frac{1}{2} }
\def\thlf{ \frac{3}{2} }
\def\frth{ \frac{1}{4} }
\def\shlf{ \tfrac{1}{2} }
\def\sfrth{ \tfrac{1}{4} }
\def\sthlf{ \tfrac{3}{2} }

\def\hem{\bullet}

\def\TQFT{TQFT}

\def\tKbr{Kauffman bracket}

\def\tJpol{Jones polynomial}

\def\tKhom{Khovanov homology}
\def\tKhoms{Khovanov homologies}
\def\tKhbr{Khovanov bracket}

\def\tJW{Jones-Wenzl}
\def\tJWp{\tJW\ projector}
\def\cJWp{categorified \tJWp}

\def\tTL{Temperley-Lieb}
\def\taTL{TL}
\def\tTLc{\tTL\ category}

\def\tTLt{\tTL\ tangle}
\def\taTLt{\taTL\ tangle}
\def\tTLa{\tTL\ algebra}

\def\trbr{torus braid}

\def\tmcn{multi-cone}

\def\thead{tail}

\def\Asplng{A-splicing}
\def\Bsplng{B-splicing}

\def\Arpl{A-replacement}
\def\Arpld{A-replaced}
\def\Brpl{B-replacement}
\def\Brpld{B-replaced}

\def\tBcr{B-circle}

\def\tadq{adequate}
\def\tBadq{B-\tadq}
\def\tBiadq{B-inadequate}
\def\tBdg{B-diagram}
\def\trBdg{reduced \tBdg}

\def\tinadq{inadequate}
\def\tnBadq{not \tBadq}

\def\tBrdc{B-reduction}

\def\uclrd{unicolored}

\def\splng{splicing}

\def\nsplcd{negatively spliced}

\def\twdfc{\twd\ deficit}
\def\twd{width}

\def\tbdgr{bi-degree}

\def\ZZt{ \ZZ_2 }

\def\ZZZtt{\ZZ\times\ZZ\times\ZZt}

\def\ZZtqqi{ \ZZ[q^{\pm 2} ] }

\def\QQ{ \mathbb{Q} }

\def\QQqqi{ \QQ[q^{\pm 1} ] }
\def\QQqqip{ \QQ[[q,q^{-1}] }

\def\ctfont{ \mathsf }
\def\chfont{ \mathbf }
\def\stfont{ \mathfrak }

\def\tcat#1{ \ddot{#1} }

\def\cTng{ \ctfont{Tng} }
\def\ctTng{ \tcat{\cTng} }

\def\cTL{ \ctfont{TL} }
\def\cTLp{ \cTL^+ }

\def\ctTL{ \tcat{\cTL} }

\def\caC{ \ctfont{C} }
\def\caCt{ \tilde{\caC} }
\def\ctaC{ \tcat{\caC} }

\def\caA{ \ctfont{A} }

\def\cKom{ \mathop{\mathbf{Kom}} }
\def\cKomm{ \cKom^+ }

\def\cKommA{ \cKomm(\caA) }

\def\chA{ \chfont{A} }
\def\chB{ \chfont{B} }

\def\stA{ \stfont{A} }

\def\oba{ \alpha }

\def\obO{ O }

\def\hmord#1{ |#1|_{\mathrm h} }
\def\hlmord#1{ \left| #1 \right|_{\mathrm h} }

\def\drsys{direct system}
\def\drlim{direct limit}

\def\SUv#1{ \SU({#1}) }
\def\SUt{ \SUv{2} }

\def\aTL{ \mathrm{TL} }
\def\aTLv#1{ \aTL_{#1} }
\def\aTLmn{ \aTLv{m,n} }

\def\pJ{ \mathrm{J} }
\def\pJv#1{ \pJ_{#1} }
\def\pJqv#1{ \pJv{#1}(q) }
\def\pJvv#1#2{ \pJ_{#1,#2} }
\def\pJqvv#1#2{ \pJvv{#1}{#2}(q) }
\def\pJqL{ \pJqv{\xL} }
\def\pJqLi{ \pJqvv{\xL}{\infty} }

\def\pJqaL{ \pJqvv{\xca}{\xL} }
\def\pJqaD{ \pJqvv{\xca}{\xD} }

\def\zpol{ \mathrm{T} }
\def\zpolnvv#1#2{ \zpol^{#1}_{#2} }
\def\zpolnqvv#1#2{ \zpolnvv{#1}{#2}(q) }
\def\zpolnqnL{ \zpolnqvv{\xnu}{\xL} }

\def\ypolvv#1#2{ \zpol_{#1,#2} }
\def\ypolqvv#1#2{ \ypolvv{#1}{#2}(q) }
\def\ypolqnL{ \ypolqvv{\xn}{\xL} }
\def\ypolqnD{ \ypolqvv{\xn}{\xD} }

\def\ypolv#1{ \zpol_{#1} }
\def\ypolqtv#1{ \ypolv{#1}(\xt,q) }
\def\ypolqtL{ \ypolqtv{\xL} }
\def\ypolqtD{ \ypolqtv{\xD} }

\def\xdmm{ - }

\def\xKbrv#1{ \langle #1 \rangle }
\def\xKbrBv#1{ \Big \langle\, #1 \,\Big \rangle }

\def\xKbrd{ \xKbrv{\xdmm} }

\def\xKhv#1{ \langle\!\langle #1 \rangle\!\rangle }

\def\xvKhv#1{ \left \langle\!\!\!\left\langle #1 \right\rangle\!\!\!\right\rangle }

\def\Cnv#1{ \boxed{#1} }
\def\sdff{ \circlearrowleft }
\def\Pcnv#1{ \boxed{#1}_{\,\displaystyle \sdff} }

\def\Conv#1{ \mathrm{Cone} (#1) }

\def\hteqv{ \sim }

\def\dlm{ \lim\limits_{\rightarrow}}

\def\xbul{ \bullet }

\def\Kh{\scriptscriptstyle{\mathrm{Kh}} }
\def\Hm{ \mathrm{H} }
\def\KHm{ \Hm^{\Kh} }

\def\KHmvv#1#2{ \KHm_{#1,#2} }
\def\KHmvb#1{ \KHmvv{#1}{\hem} }

\def\KHmib{\KHmvv{i}{\xbul} }

\def\tKHm{ \tilde{\Hm}^{\Kh} }
\def\tKHmvv#1#2{ \tKHm_{#1,#2} }
\def\tKHmvb#1{ \tKHmvv{#1}{\hem} }

\def\tlH{ \Hm^{\infty} }
\def\tlHvv#1#2{ \tlH_{#1,#2} }

\def\ttlH{ \Hm^{\boldsymbol{\infty}} }
\def\ttlH{ \Hm^{\thicksim} }
\def\ttlHvvv#1#2#3{ \ttlH_{#1,#2,#3} }

\def\hgrshv#1#2#3{ [#1,#2,#3] }

\def\dgh{ \deg_{\mathrm{h}} }
\def\dgq{ \deg_{\mathrm{q}} }
\def\dgb{ \deg_{\mathrm{b}} }

\def\tqdgr{$q$-degree}
\def\thdgr{$h$-degree}
\def\tbdgr{$b$-degree}

\def\tbgrd{$b$-grading}

\def\smxnzi{ \sum_{\xn=0}^\infty }

\def\wdv#1{ |#1|_{\mathrm{wd}} }
\def\swdv#1{ \left| #1 \right|_{\mathrm wd} }
\def\wdfcv#1{ |#1|_{\mathrm{df}} }

\def\sTL{ \stfont{T} }
\def\sTLv#1{ \sTL_{#1} }

\def\sTLa{ \sTLv{a} }
\def\sTLab{ \sTLv{a,b} }
\def\sTLabc{ \sTLv{a,b}^{\supset} }

\def\shcr{ \mathsf{h} }
\def\shfr{ \mathsf{q} }
\def\shcrh{\shcr^{\hlf} }
\def\shcrmh{ \shcr^{-\hlf} }

\def\clN{ N }
\def\clNo{ \clN + 1}

\def\xDclv#1{ \xD_{#1} }
\def\xDclN{ \xDclv{\clN} }
\def\xDclNo{ \xDclv{\clNo} }

\def\xDclvv#1#2{ \xD_{#1,#2} }

\def\xDNoa{ \xDclvv{\xca+1}{\incra} }
\def\xDNoamo{ \xDclvv{\xca+1}{\incra-1} }

\def\xDpclvv#1#2{ \xD'_{#1,#2} }

\def\xDpNoa{ \xDpclvv{\xca+1}{\incra} }
\def\xDpNof{ \xDpclvv{\xca+1}{\ncrD} }

\def\xtD{ \tilde{\xD} }
\def\xtDv#1{ \xtD_{#1} }
\def\xtDNo{ \xtDv{\xca+1} }
\def\xtDNo{ \xtDv{\xca} }
\def\xtDN{ \xtDv{\xca} }

\def\xtDNov#1{ \xtD_{\xca+1;#1} }
\def\xtDNob{ \xtDNov{\crcb} }
\def\xtDNobo{ \xtDNov{\crcb+1} }
\def\xtDNobv#1{ \xtDNov{\crcb,#1} }
\def\xtDNobg{ \xtDNobv{\prjg} }

\def\xtDNobgo{ \xtDNobv{\prjg+1} }
\def\xtDNobf{ \xtDNobv{\pncrb-1} }

\def\xhD{ \hat{\xD} }

\def\xhDNov#1{ \xhD_{\xca+1,#1} }
\def\xhDNob{ \xhDNov{\crcb} }

\def\xDs{ \xD_{\spmp} }
\def\xDscr{ \xD_{\spmp,\scs} }

\def\xDcir{ \xD_{\circ} }

\def\xDov#1{ \xD_{#1} }
\def\xDon{ \xDov{n} }
\def\xDom{ \xDov{m} }

\def\xBov#1{ \xB_{#1} }
\def\xBmpn{ \xBov{m+n} }

\def\eqdiam{equatorial diameter}

\def\ttngl{tangle}
\def\ttnglv#1#2{$(#1,#2)$-\ttngl}
\def\TLttnglv#1#2{\taTL\ \ttnglv{#1}{#2}}
\def\ttnglmn{\ttnglv{m}{n}}
\def\ttnglnn{\ttnglv{n} {n}}
\def\TLttnglmn{\TLttnglv{m}{n}}
\def\TLttnglnn{\TLttnglv{n}{n}}

\def\qi{ q^{-1} }
\def\qpqi{ q + \qi }
\def\mqpqi{ - (\qpqi) }

\def\tHom{ \widetilde{\Hom} }

\def\xlam{ \lambda }
\def\xKhl{ \xKhv{\xlam} }

\def\shm{ \mu }
\def\hgrshklm{ \hgrshv{k}{l}{\shm} }

\def\elcf{ f }
\def\elcfv#1{ f_{#1} }

\def\Zcat{ \mathrm{Z} }
\def\Zcatv#1{ \Zcat(#1) }

\def\gAB{ g_{AB} }

\def\xIdv#1{ \xId_{#1} }
\def\xIdn{ \xIdv{n} }
\def\xIdvv#1#2{ \xId_{#1,#2} }
\def\xIdnL{ \xIdvv{n}{\xL} }

\def\xL{ L }
\def\xLp{ \xL' }
\def\xLb{ \bar{\xL} }
\def\xLcv#1{ \xL_{#1} }
\def\xLcN{ \xLcv{\xca} }
\def\xLcNo{ \xLcv{\xca+1} }

\def\xLpcv#1{ \xLp_{#1} }
\def\xLpcN{ \xLpcv{\xca} }

\def\xD{ D }

\def\xczt{ \zeta }
\def\xcztv#1{ \xczt(#1) }
\def\xcztL{ \xcztv{\xL} }

\def\cmfont{ \mathbf }
\def\cmA{ \cmfont{A} }

\def\xca{ N }
\def\xnu{ \nu }

\def\sB{ \mathrm{B} }

\def\nBv#1{ \gvv{#1} }
\def\nBD{ \nBv{\xD} }

\def\ncr{ \mathrm{cr} }
\def\ncr{ \times }
\def\ncr{ \chi }
\def\ncrv#1{ \ncr_{\!#1} }
\def\ncrD{ \ncrv{\xD} }
\def\ncrL{ \ncrv{\xL} }

\def\ncrm{ \chi^{!} }
\def\ncrmv#1{ \ncrm_{\!#1} }
\def\ncrmL{ \ncrmv{\xL} }

\def\ncriv#1{ \ncr_{\!#1}^{\mathrm in} }
\def\ncriD{ \ncriv{\xD} }

\def\ncrt{ \tilde{\ncr} }
\def\ncrtv#1{ \ncrt_{#1} }
\def\ncrtD{ \ncrtv{\xD} }

\def\sBv#1{ \sB(#1) }
\def\sBD{ \sBv{\xD} }
\def\sBDp{ \sBv{\xDp} }

\def\xfs{ s }
\def\xfsv#1{ \xfs_{#1} }

\def\xfsvv#1#2{ \xfsv{#1}(#2) }
\def\xfsLN{ \xfsvv{\xL}{\xca} }

\def\xn{ n }
\def\xt{ t }

\def\degq{ \deg_q }

\def\qhlf{ q^{ \frac{1}{2} } }
\def\qmhlf{ q^{-\frac{1}{2} } }

\def\gl{\mathrm{ l} }
\def\glv#1{ \gl_{#1} }
\def\glD{ \glv{\xD} }
\def\glL{ \glv{\xL} }
\def\ge{ \mathrm{e} }
\def\gev#1{ \ge_{#1} }
\def\geD{ \gev{\xD} }
\def\geL{ \gev{\xL} }
\def\gv{\kappa}
\def\gvv#1{ \gv_{#1} }
\def\gvD{ \gvv{\xD} }
\def\gvL{ \gvv{\xL} }

\def\xfrm{ \phi }
\def\xfrmv#1{ \xfrm_{#1} }
\def\xfrmD{ \xfrmv{\xD} }

\def\pncr{ \pi }
\def\pncrv#1{ \pncr_{#1} }
\def\pncrb{ \pncrv{\crcb} }

\def\xM{ M }

\def\bdA{ A }

\def\Nsgn{ (-1)^{\xca} }

\def\brbet{ \beta }
\def\xnp{ n_+ }
\def\xnm{n_- }

\def\stlcl#1{ {\scriptstyle #1 } }

\def\drlmA{ \dlm \chA_i }

\def\xalg{ \IQ }

\def\mnf{ f }
\def\mnfv#1{ \mnf_{#1} }
\def\mnfN{ \mnfv{\xca} }
\def\mnfi{ \mnfv{\mathrm i} }
\def\mnff{ \mnfv{\mathrm f} }

\def\mntf{ \tilde{\mnf} }
\def\mntfv#1{ \mntf_{#1} }
\def\mntfN{ \mntfv{\xca} }

\def\incra{ \alpha }
\def\incrb{ \alpha\p }
\def\xlbv#1{ v_{#1} }
\def\xlba{ \xlbv{\incra} }
\def\xlbao{ \xlbv{\incra+1} }
\def\xlbb{ \xlbv{\incrb} }

\def\crcb{ \beta }
\def\crcbp{ \beta\p }
\def\ylbv#1{ c_{#1} }
\def\ylbb{ \ylbv{\crcb} }
\def\ylbbp{ \ylbv{\crcbp} }

\def\prjg{ \gamma }
\def\prjgp{ \gamma\p }

\def\zlbvv#1#2{ p_{#1,#2} }
\def\zlbbv#1{ \zlbvv{\crcb}{#1} }
\def\zlbbg{ \zlbbv{\prjg} }
\def\zlbbgp{ \zlbbv{\prjgp} }

\def\xmg{ g }
\def\xmgNv#1{ \xmg_{\xca,#1} }
\def\xmgNa{ \xmgNv{\incra} }
\def\xmgN{ \xmg_{\xca} }

\def\xtmg{ \tilde{\xmg} }
\def\xtmgv#1{ \xtmg_{#1} }
\def\xtmgbg{ \xtmgv{\crcb,\prjg} }

\def\ltrf{local transformation}
\def\Ltrf{Local transformation}

\def\lrpl{local replacement}
\def\Lrpl{Local replacment}

\def\ytng{ \tau }

\def\ytngsv#1{ \ytng_{#1} }
\def\ytngsi{ \ytngsv{\mathrm i} }
\def\ytngsf{ \ytngsv{\mathrm f} }
\def\ytngse{ \ytngsv{\mathrm e} }
\def\ytngsfp{ \ytngsf\p }
\def\ytngsc{ \ytngsv{\mathrm c} }
\def\ytngsci{ \ytngsv{\mathrm{c},\xki} }

\def\ytngki{ \xKhv{\ytngsi} }
\def\ytngkf{ \xKhv{\ytngsf} }
\def\ytngkc{ \xKhv{\ytngsc} }

\def\ytngkfp{ \xKhv{\ytngsf'}}

\def\xDsv#1{ \xD_{#1} }
\def\xDsi{ \xDsv{\mathrm i} }
\def\xDsf{ \xDsv{\mathrm f} }

\def\xDsc{ \xDsv{\mathrm c} }
\def\xDscv#1{ \xDsv{\mathrm{c},#1} }
\def\xDsci{ \xDscv{\xki} }
\def\xDscip{ \xDsci' }
\def\xDscipp{ \xDsci''}
\def\xDse{\xDsv{\mathrm e} }
\def\xDsep{ \xDse'}
\def\xDscp{ \xDsc\p }

\def\hbnd{ M_{\mathrm h} }

\def\sgmm{ s }

\def\yncr{ n_{\times} }
\def\yncrv#1{ \yncr(#1) }
\def\yncrpv#1{ \yncr'(#1) }
\def\yncrD{ \yncrv{\xD} }

\def\yncc{ n_{\circ} }
\def\ynccv#1{ \yncc(#1) }

\def\tdgpr{degree preserving}

\def\xhsh{ m }
\def\xhshv#1{ \xhsh_{#1} }
\def\xhshf{ \xhshv{\mathrm f}}

\def\xqsh{ n }
\def\xqshv#1{ \xqsh_{#1} }
\def\xqshf{ \xqshv{\mathrm f}}

\def\ztau{ \tau }

\def\svrt{ \mathfrak{V} }
\def\svrta{ \svrt_{\mathrm{ad}} }
\def\svrti{ \svrt_{\mathrm{in}} }
\def\xvrt{ v }
\def\spmp{ s }
\def\sipvr{ \spmp_{\xvrt} }
\def\spvr{ \sipvr }
\def\xabms{ |\!|\spmp|\!| }

\def\tstt{state}

\def\tstrt{strut}
\def\tstrtl{\tstrt\ line}
\def\tjmp{jumping}
\def\tjmpc{\tjmp\ circle}
\def\trlx{relaxed}
\def\trlxc{\trlx\ circle}

\def\crb{ c }

\def\nvcr#1{ n_{\mathrm #1} }
\def\njcr{ \nvcr{j} }
\def\nscr{ \nvcr{s} }
\def\nscrb{ n_{\mathrm{s},\crb} }
\def\nwcr{ \nvcr{w} }
\def\nwcrb{ n_{\mathrm{w},\crb} }

\def\njrc{ n_{\mathrm{r},c} }

\def\tstrght{straight}
\def\txstr{straight}
\def\txstrs{\txstr\ segment}
\def\tstrghtc{\tstrght\ circle}

\def\twndg{winding}
\def\twndgc{\twndg\ circle}

\def\ntwd{not widening}
\def\scs{ c }

\def\xtusc{ \tau_{s,c} }

\def\xE{ E }
\def\xEo{\xE^1}
\def\xEovv#1#2{ \xEo_{#1,#2} }
\def\xEoij{ \xEovv{i}{j} }

\def\xki{ k }
\def\yki{ k }
\def\yvki{ \yki(i) }
\def\ysvki{ \yki^2(i) }

\def\zmi{ i }
\def\zmj{ j }

\def\xfd{ d }
\def\xfdv#1{ \xfd_{#1} }
\def\xfdN{ \xfdv{\xca} }
\def\xti{ \tilde{i}}

\def\ptJqLi { \mathrm{J}_{\xL,\thicksim}(b,q)}

\def\Stosic{Sto\v si\'c}

\def\SUv#1{ \mathrm{SU}(#1) }

\def\xltone{I}
\def\xlttwo{II}
\def\xltthree{III}

\def\prmlt{ \mu }
\def\prmltv#1{ \prmlt_{#1} }
\def\prmltijg{ \prmltv{ij,\gamma} }
\def\prmltij{ \prmltv{ij} }
\def\prmltijk{ \prmltv{ij,\xki} }
\def\tprmlt{ \tilde{\prmlt} }
\def\tprmltv#1{ \tprmlt_{#1} }
\def\tprmltij{ \tprmltv{ij} }

\def\lumps{lump sum}

\def\mtot{ m^{\mathrm{tot}} }
\def\mtotv#1{ \mtot_{#1} }
\def\mtotja{ \mtotv{j,a} }

\def\betbr{ \beta }
\def\xLv#1{ \xL_{#1} }
\def\xLb{ \xLv{\betbr} }

\def\xcrsp{
\xygraph{
!{0;/r1.5pc/:}
[u(0.5)]
!{\xoverv}
}
}

\def\xpver{
\xygraph{
!{0;/r1.5pc/:}
[u(0.5)]
!{\xunoverv}
}
}

\def\xphor{
\xygraph{
!{0;/r1.5pc/:}
[u(0.5)]
!{\xunoverh}
}
}

\def\zoverv{
\begin{tikzpicture} \draw + (1,0) -- ++(0,1);
\draw  [line width=6pt, draw=black] (0,0) -- ++(1,1);
\draw (0,0) -- ++(1,1);
\end{tikzpicture}
 }

 \def\zoverv#1#2{
 \begin{tikzpicture}[scale=#1,baseline=11*#1-#2*#1]
 \path[use as bounding box] (-0.1,-0.1) rectangle (1.1,1.1);
 \draw (0,1) -- ++(1,-1);
 \draw[line width=6pt, draw=white] (0,0) -- ++(1,1);
 \draw (0,0) -- ++(1,1);
\end{tikzpicture}
 }

\def\zunoverv#1#2{
 \begin{tikzpicture}[scale=#1,baseline=11*#1-#2*#1]
 \path[use as bounding box] (-0.1,-0.1) rectangle (1.1,1.1);
 \draw (0,0) to [out=45,in=-45] (0,1);
 \draw (1,0) to [out=135,in=-135] (1,1);
\end{tikzpicture}
 }

 \def\zunoverh#1#2{
 \begin{tikzpicture}[scale=#1,baseline=11*#1-#2*#1]
 \path[use as bounding box] (-0.1,-0.1) rectangle (1.1,1.1);
 \draw (0,0) to [out=45,in=135] (1,0);
 \draw (0,1) to [out=-45,in=-135] (1,1);
\end{tikzpicture}
 }

\def\zcirc#1#2{
 \begin{tikzpicture}[scale=#1,baseline=11*#1-#2*#1]
 \path[use as bounding box] (-0.1,-0.1) rectangle (1.1,1.1);
 \draw (0.5,0.5) circle (0.5);
\end{tikzpicture}
 }

\def\zposfr#1#2{
 \begin{tikzpicture}[scale=#1,baseline=11*#1-#2*#1]
 \path[use as bounding box] (-0.1,-0.1) rectangle (0.9,1.1);
 \draw (0.6,0.25) to [out=180,in=-90] (0,1);
 \draw [line width=6pt,draw=white] (0,0) to [out=90,in=180] (0.6,0.75);
 \draw (0,0) to [out=90,in=180] (0.6,0.75);
 \draw (0.6,0.75) .. controls +(0:0.4) and +(0:0.4) .. (0.6,0.25);
\end{tikzpicture}
 }

 \def\zstvr#1#2{
 \begin{tikzpicture}[scale=#1,baseline=11*#1-#2*#1]
 \path[use as bounding box] (-0.1,-0.1) rectangle (0.6,1.1);
 \draw (0.25,0) -- (0.25,1);
\end{tikzpicture}
 }

 \def\zcposfr#1#2{
 \begin{tikzpicture}[scale=#1,baseline=11*#1-#2*#1-0.3cm]
 \path[use as bounding box] (-1.2,-0.1) rectangle ++(2.4,0.8);
 \begin{scope}[rotate=90]
 \draw [thkln] (0.6,0.25) to [out=180,in=-90] (0,1);
 \draw [line width=6pt,draw=white] (0,0) to [out=90,in=180] (0.6,0.75);
 \draw [thkln] (0,0) to [out=90,in=180] (0.6,0.75);
 \draw [thkln] (0.6,0.75) .. controls +(0:0.4) and +(0:0.4) .. (0.6,0.25);
 \draw [ptzer] (-0.6,-0.3) rectangle ++ (1.2,0.3);
 \draw [thkln] (0,-0.3) -- (0,-0.9) node [near end, above] {$\scriptstyle a$};
 \end{scope}
\end{tikzpicture}
 }

  \def\zcstvr#1#2{
 \begin{tikzpicture}[scale=#1,baseline=11*#1-#2*#1-0.3cm]
 \path[use as bounding box] (-0.9,-0.5) rectangle (0.9,0.5);
\draw [ptzer] (-0.15,-0.6) rectangle ++(0.3,1.2);
\draw [thkln] (-0.75,0) -- (-0.15,0) (0.15,0) -- (0.75,0) node [near end,above] {$\scriptstyle a$};
\end{tikzpicture}
 }

\def\cblth{1.3pt}
\def\ljwp{1pt}
\def\ocrw{6pt}

\tikzset{ptzer/.style={line width=\ljwp} }
\tikzset{ptzert/.style={line width=\ljwp,fill=white}}
\tikzset{ptone/.style={line width=\ljwp,fill=gray!30}}
\tikzset{pttwo/.style={line width=\ljwp,fill=white,pattern=horizontal lines} }
\tikzset{ptthr/.style={line width=\ljwp,pattern=north east lines} }
\tikzset{ptfour/.style={line width=\ljwp,pattern=crosshatch} }
\tikzset{vthln/.style={line width=\ljwp} }

\tikzset{menvone/.style={scale=0.5,baseline=-1.5} }
\tikzset{menvtwo/.style={scale=0.70,baseline=-1.5} }
\tikzset{menvthree/.style={scale=0.35, baseline=-1.5} }
\tikzset{menvfour/.style={scale=0.25, baseline=-1.5} }
\tikzset{menvzer/.style={scale=1,baseline=-1.5} }

\tikzset{menvrtone/.style={scale=0.45,rotate=-45,baseline=-3.5} }
\tikzset{menvrtsone/.style={scale=0.65,rotate=-45,baseline=-3.5} }
\tikzset{menvrtonep/.style={scale=0.5,baseline=-3.5} }

\tikzset{lnovr/.style={line width=6pt,color=white}}
\tikzset{lnovrtw/.style={line width=4pt,color=white}}
\tikzset{thkln/.style={line width=\cblth} }
\tikzset{thnln/.style={} }
\tikzset{thkc/.style={line width=0.6pt} }
\tikzset{bcrc/.style={line width=0.2cm,color=gray!40,opacity=0.5} }
\tikzset{bcrct/.style={line width=0.25cm,color=gray!20} } 

\pgfdeclaredecoration{ignore}{final}
{
\state{final}{}
}

\pgfdeclaremetadecoration{middle}{initial}{
    \state{initial}[
        width={(\pgfmetadecoratedpathlength - \the\pgfdecorationsegmentlength)/2},
        next state=middle
    ]
    {\decoration{moveto}}

    \state{middle}[
        width={\the\pgfdecorationsegmentlength},
        next state=final
    ]
    {\decoration{curveto}}

    \state{final}
    {\decoration{ignore}}
}

\tikzset{middle segment/.style={decoration={middle},decorate, segment length=#1}}

\def\xId{ \mathbbm{1} }

\def\hmpt{HOMFLY-PT}

\def\tsq{sequence}
\def\Tsq{Sequence}

\def\drv{derivation}

\def\rdc{relative derived category}
\def\forfun{forgetful functor}

\def\nstd{nested}

\def\qiso{$\Ffr$-isomorphism}
\def\qisoo{$\Ff_1$-isomorphism}
\def\qisot{$\Ff_2$-isomorphism}
\def\xqiso{quasi-isomorphism}
\def\qisc{$\Ffr$-isomorphic}
\def\xqisc{quasi-isomorphic}

\def\sdr{semi-direct}
\def\sdrpr{\sdr\ product}

\def\Wfl{$\aWp$-flat}

\def\ChE{Chevalley-Eilenberg}
\def\ChEr{\ChE\ resolution}

\def\Sgl{Soergel}
\def\Sglb{\Sgl\ \bmdl}
\def\Hchs{Hochschild}
\def\Hchsh{\Hchs\ homology}
\def\Hchshs{\Hchs\ homologies}

\def\tgr{triply graded}
\def\tglh{\tgr\ link homology}

\def\tgh{\tgr\ homology}

\def\enfn{endo-functor}
\def\stact{standard action}

\def\xstv{strand variable}
\def\brstv{braid \xstv}

\def\bmdl{bimodule}

\def\brgr{braid-graph}
\def\brwd{braid word}

\def\smgr{semi-group}
\def\brwdsg{\brwd\ \smgr}

\def\smpr{semi-free}

\def\cnn{connection}

\def\qhteq{quasi-homotopy equivalence}
\def\qhtet{quasi-homotopy equivalent}

\def\Wpeq{$\aWp$-equivariant}
\def\Wpec{$\aWp$-equivariance}

\def\weqxyp{$\aWp$-equivariant $\Qbxy$-projective}
\def\weqxyop{$\aWp$-equivariant $\Qxyo$-projective}

\def\xLv#1{ L_{#1} }
\def\xLm{ \xLv{m} }
\def\xLn{ \xLv{n} }

\def\xhLv#1{ \hat{L}_{#1} }
\def\xhLm{ \xhLv{m} }
\def\xhLn{ \xhLv{n} }

\def\Qv#1{ \IQ[#1] }
\def\Qx{ \Qv{x} }
\def\Qy{ \Qv{y} }
\def\Qxyo{ \Qv{x_1,y_1} }
\def\Qxyt{ \Qv{x_2,y_2} }
\def\Qbx{ \Qv{\bfx} }
\def\Qby{ \Qv{\bfy} }
\def\Qbz{ \Qv{\bfz} }
\def\Qbxy{ \Qv{\bfx,\bfy} }
\def\Qbxz{ \Qv{\bfx,\bfz} }

\def\Qbxw{ \Qv{\bfx,\bfw} }

\def\Qblc{ \Qv{\bflc} }

\def\bfv#1{ \mathbf{#1} }
\def\bfa{ \bfv{a} }
\def\bfx{ \bfv{x} }
\def\bfy{ \bfv{y} }
\def\bfz{ \bfv{z} }
\def\bfw{ \bfv{w} }
\def\bfp{ \bfv{p} }

\def\bfk{ \bfv{k} }
\def\bfth{ \boldsymbol{\theta} }
\def\bfsdf{ \boldsymbol{\sdf} }
\def\bfshf{ \boldsymbol{\shf} }
\def\bfsdfot{ \bfsdf_{12} }

\def\bfsdfp{ \bfsdf' }
\def\sdfp{ \sdf' }
\def\bsdfpv#1{ \bfsdfp(#1) }
\def\bsdfpx{ \bsdfpv{x} }
\def\bsdfpli{ \bsdfpv{\lc_i} }
\def\bsdfpxt{ \bsdfpv{x_2} }
\def\bsdfpyo{ \bsdfpv{y_1} }
\def\bsdfpxtz{ \bsdfpv{x_2 + \frz } }
\def\sdfpvv#1#2{ \sdfp_{#1}(#2) }
\def\spdpmv#1{ \sdfpvv{m}{#1} }
\def\spdpmx{ \spdpmv{x} }

\def\bsdfz{ \bfa_{\frz} }

\def\brnk{ k }

\def\hchi{ \chi }
\def\hchiv#1{ \hchi_{#1} }

\def\hchipo{ \hchiv{+} }
\def\hchine{ \hchiv{-} }

\def\hchivv#1#2{ \hchi_{#1;#2} }
\def\hchipov#1{ \hchivv{#1}{+} }
\def\hchinev#1{ \hchivv{#1}{-} }
\def\hchipoi{ \hchipov{i} }
\def\hchinei{ \hchinev{i} }

\def\setmod#1{\mathfrak{#1}}
\def\setA{ \setmod{A} }
\def\setT{ \setmod{T} }
\def\tsetT{ \setT }
\def\tcCat{ \cCat }

\def\catmod#1{ \mathbf{#1} }
\def\xcatmod#1{ \mathbf{#1} }
\def\funmod#1{ \mathcal{#1} }
\def\ctD{ \catmod{D} }
\def\ctDr{ \ctD_{\mathrm{r}} }

\def\ctDFfr{ \ctDr }
\def\ctA{ \xcatmod{A} }
\def\ctB{ \xcatmod{B} }
\def\ctKom{ \catmod{Com} }
\def\ctmod{ \catmod{mod} }
\def\ctgmod{ \catmod{gm} }
\def\cCat{ \catmod{C} }
\def\ctAd{ \catmod{Ad} }
\def\ctCh{ \catmod{Ch} }
\def\ctK{ \catmod{K} }

\def\cCatbxy{ \cCatv{\bfx,\bfy} }
\def\cCatbyz{ \cCatv{\bfy,\bfz} }
\def\cCatbxz{ \cCatv{\bfx,\bfz} }
\def\cCatbxw{ \cCatv{\bfx,\bfw} }

\def\tCat{ \widetilde{\cCat} }
\def\tCatv#1{ \tCat_{#1} }
\def\tCatbxy{ \tCatv{\bfx,\bfy} }

\def\Ffr{ \Ff }
\def\functs#1{ \funmod{#1} }
\def\Ff{ \functs{F} }
\def\Gf{ \functs{G} }
\def\Qf{ \functs{Q} }

\def\Ffuvv#1#2{ \Ff^{(#1)}_{#2} }
\def\Ffusv#1{ \Ffuvv{s}{#1} }

\def\Ffusxl{ \Ffusv{\bfx;\bflc} }

\def\Ffukv#1{ \Ffuvv{\bfk}{#1} }
\def\FfukXl{ \Ffukv{\bfvbr;\bflc} }
\def\Ffukxl{ \Ffukv{\bfx;\bflc} }
\def\Ffpukv#1{ \Ffuvv{\bfk'}{#1} }
\def\FfpukXl{ \Ffpukv{\bfvbr;\bflc} }
\def\Ffvav#1#2{ \Ff_{#1\rightarrow #2} }
\def\FfxaX{ \Ffvav{\bfx}{\bfvbr} }
\def\FfA{ \Ff_{\ctA} }
\def\FfB{ \Ff_{\ctB} }

\def\Ffvv#1#2{ \Gf_{#1=#2} }
\def\FfXv#1{ \Ffvv{\bfvbr}{#1} }
\def\Ffbev#1{ \Ffvv{\bfbe}{#1} }

\def\Ffbey{ \Ffbev{\bfy} }
\def\FfXx{ \FfXv{\bfx} }
\def\FfXy{ \FfXv{\bfy} }
\def\Ffav#1{ \Ffvv{\vbr}{#1} }
\def\Ffaxi{ \Ffav{x_i} }
\def\Ffaxj{ \Ffav{x_j} }

\def\Fffq{ \Ff_{/\simeq} }
\def\Fffq{ \Ff_{\heartsuit}}

\def\ctDW{ \ctD_{\aWp} }

\def\aWpgmd{ \aWp\sctgmod }

\def\uWplcgmdh{ \uWplc\sctgmodh }

\def\foth{ f^{(13)} }
\def\Foth{ \funmod{F}^{(13)} }
\def\Fothv#1{ \Foth_{#1} }
\def\Fothxy{ \Fothv{\bfx,\bfy} }
\def\Fothxw{ \Fothv{\bfx,\bfw} }

\def\otdr{ \stackrel{\mathrm{L}}{\otimes} }
\def\otdrv#1{ \otdr_{#1} }
\def\otdrvv#1#2{ {}_{\;\,#1\!\!}\otdrv{#2} }
\def\otdrxyt{ \otdrv{x_2,y_2} }
\def\otdrxyo{ \otdrv{x_1,y_1} }
\def\otdry{ \otdrv{\bfy} }
\def\otdrxy{ \otdrv{\bfx,\bfy} }
\def\otdruxy{ \otdrv{\tnsp{\bfx},\bfy} }
\def\otdruxxpy{ \otdrv{x_1,\tnsp{\bfx}',\bfy} }
\def\otdruxpyp{ \otdrv{\tnsp{\bfx}',\bfy'} }

\def\otdryz{ \otdrv{\bfy,\bfz} }
\def\otdrxw{ \otdrv{\bfx,\bfw} }
\def\otdrxyzw{ \otdrv{\bfx,\bfy,\bfz,\bfw} }
\def\otdrlv#1{\otdrvv{\mathrm{l}}{#1}}

\def\otdrly{ \otdrlv{\bfy} }

\def\otdrfvv#1#2{ {}_{\;\,#1}\otdrv{#2} }
\def\otdrflv#1{\otdrfvv{\mathrm{l}}{#1}}

\def\otdrfly{ \otdrflv{\bfy} }

\def\otdrx{ \otdrv{\bfx} }
\def\otdry{ \otdrv{\bfy} }

\def\otv#1{ \otimes_{#1} }

\def\oty{ \otv{\bfy} }
\def\otz{ \otv{\bfz} }
\def\otsz{ \otv{z} }

\def\otw{ \otv{\bfw} }
\def\otxw{ \otv{\bfx,\bfw} }
\def\otyz{ \otv{\bfy,\bfz} }
\def\otxyo{ \otv{x_1,y_1} }
\def\tnsp#1{ \bar{#1} }
\def\otuy{ \otv{\tnsp{\bfy}} }
\def\otux{ \otv{\tnsp{\bfx}} }

\def\ctQ{ \xcatmod{Q} }
\def\ctW{ \xcatmod{W} }
\def\ctQv#1{ \ctQ_{#1} }
\def\ctQxy{ \ctQv{\bfx,\bfy} }
\def\ctQyz{ \ctQv{\bfy,\bfz} }
\def\ctQxz{ \ctQv{\bfx,\bfz} }

\def\ctWv#1{ \ctW^+_{#1} }
\def\ctWxy{ \ctWv{\bfx,\bfy} }

\def\ctWxyt{ \ctWv{x_2,y_2} }
\def\ctWyz{ \ctWv{\bfy,\bfz} }
\def\ctWxz{ \ctWv{\bfx,\bfz} }
\def\ctWxw{ \ctWv{\bfx,\bfw} }
\def\ctWx{ \ctWv{\bfx} }
\def\ctWxp{ \ctWv{\bfx'} }

\def\ctWX{ \ctWv{\bfvbr} }

\def\ctWl{ \ctWv{\bflc} }
\def\ctvWv#1#2{ {}_{#1}\ctWv{#2} }
\def\ctXWv#1{ \ctvWv{\bfvbr}{#1} }
\def\ctxWv#1{ \ctvWv{\vbr}{#1} }
\def\ctYWv#1{ \ctvWv{\bfbe}{#1} }

\def\ctXWx{ \ctXWv{\bfx} }
\def\ctXWxy{ \ctXWv{\bfx,\bfy} }
\def\ctYWyz{ \ctYWv{\bfy,\bfz} }
\def\ctXWxz{ \ctXWv{\bfx,\bfz} }

\def\ctxWx{ \ctxWv{\bfx} }
\def\ctxWxy{ \ctxWv{\bfx,\bfy} }

\def\alA{ \alggmod{A} }
\def\alB{ \alggmod{B} }

\def\algmod#1{ \mathcal{#1}}
\def\malgmod#1{ #1 }
\def\salgmod#1{\mathscr{#1}}
\def\aA{ \malgmod{A} }
\def\aB{ \malgmod{B} }
\def\aBt{ \widetilde{\aB} }
\def\aC{ \malgmod{C} }
\def\aD{ \malgmod{D} }
\def\aE{ \malgmod{E} }
\def\aF{ \malgmod{F} }
\def\bulv#1{ #1_\bullet }
\def\aAb{ \bulv{\aA} }
\def\aBb{ \bulv{\aB} }
\def\aBtb{ \bulv{\aBt} }
\def\aCb{ \bulv{\aC} }
\def\aDb{ \bulv{\aD} }
\def\aEb{ \bulv{\aE} }

\def\aalmod#1{\mathcal{#1} }
\def\alA{ \aalmod{A} }
\def\alB{ \aalmod{B} }

\def\rsP{ \algmod{P} }
\def\rsPb{ \rsP^{\bullet} }
\def\rsPW{ \rsP_{\uWp} }
\def\rsPWb{ \rsPW^{\bullet} }

\def\mM{ M }
\def\mN{ N }
\def\mMv#1{ \mM_{#1} }

\def\lfsin#1#2{ {}^{\;\,#1}_{#2} }

\def\mvMv#1#2{ {}_{#1}\mMv{#2} }
\def\mXMv#1{ \mvMv{\bfvbr}{#1} }
\def\mXMx{ \mXMv{\bfx} }
\def\mXMxy{ \mXMv{\bfx,\bfy} }
\def\mMxy{ \mMv{\bfx,\bfy} }
\def\mMx{ \mMv{\bfx} }
\def\mMy{ \mMv{\bfy} }

\def\mNv#1{ \mN_{#1} }
\def\mNyz{ \mNv{\bfy,\bfz} }

\def\mNxy{ \mNv{\bfx,\bfy} }
\def\mNyz{ \mNv{\bfy,\bfz} }

\def\aI{ \malgmod{I} }
\def\aIsv#1{ \aI_{S_{#1}} }
\def\aIsn{ \aIsv{\nst} }
\def\uW{ \aalmod{W} }
\def\aW{ \mathfrak{W} }
\def\aWp{ \aW^+ }
\def\uWp{ \uW^+ }
\def\aWpeq{$\aWp$-equivariant}

\def\uWpv#1{ \uWp_{#1} }
\def\uWpx{ \uWpv{\bfx} }
\def\uWpsx{ \uWpv{x} }

\def\uWpy{ \uWpv{\bfy} }
\def\uWpxyo{ \uWpv{x_1,y_1} }
\def\uWpxyt{ \uWpv{x_2,y_2} }
\def\uWpxy{ \uWpv{\bfx,\bfy} }
\def\uWpyz{ \uWpv{\bfy,\bfz} }
\def\uWpxz{ \uWpv{\bfx,\bfz} }
\def\uWpxw{ \uWpv{\bfx,\bfw} }
\def\uWplc{ \uWpv{\bflc} }
\def\uWpX{ \uWpv{\bfvbr} }

\def\uWpxyt{ \uWpv{x_2,y_2} }

\def\uWpfv#1#2{ {}_{#1}\uWpv{#2} }
\def\uWpXv#1{ \uWpfv{\bfvbr}{#1} }
\def\uWpXxy{ \uWpXv{\bfx,\bfy} }
\def\uWpXx{ \uWpXv{\bfx} }

\def\uWpav#1{ \uWpfv{\vbr}{#1} }
\def\uWpaxy{ \uWpav{\bfx,\bfy} }

\def\mS{ \salgmod{S} }
\def\mSv#1{ \mS_{#1} }
\def\mSn{ \mSv{\nst} }
\def\mSmv#1{ \mSv{m;#1} }
\def\mSmxypp{ \mSmv{\bfx'',\bfy''} }
\def\mSmbxy{ \mSmv{\bfx,\bfy} }
\def\mSvv#1#2#3{ \mS_{#1;#2}^{#3} }
\def\mSmxyi{ \mSvv{m}{\bfx,\bfy}{i} }

\def\brgr{ \setmod{B} }
\def\brwgr{ \widetilde{\brgr} }

\def\brgrv#1{ \brgr_{#1} }
\def\brwgrv#1{ \brwgr_{#1} }

\def\brgrn{ \brgrv{\nst} }
\def\brgrns{ \brgrn^{(s)} }
\def\brwgrn{ \brwgrv{\nst} }

\def\lnks{ \setmod{L} }
\def\lnksv#1{ \lnks_{#1} }
\def\lnksn{ \lnksv{\nst} }
\def\lnksm{ \lnksv{\mlk} }

\def\sggn{ \sigma }
\def\sggnpm{ \sggn^{\pm 1} }
\def\sggnv#1{ \sggn_{#1} }
\def\sggni{ \sggnv{i} }
\def\sggnj{ \sggnv{j} }
\def\sggno{ \sggnv{1} }
\def\sggnt{ \sggnv{2} }

\def\sggxi{ \sggn^{(-1)} }
\def\sggxiv#1{ \sggxi_{#1} }
\def\sggxii{ \sggxiv{i} }
\def\sggxio{ \sggxiv{1} }

\def\sggi{ \sigma^{-1} }
\def\sggiv#1{ \sggi_{#1} }
\def\sggii{ \sggiv{i} }
\def\sggio{ \sggiv{1} }
\def\sggit{ \sggiv{2} }

\def\xarc{ \shortmid }
\def\xpar{ \shortparallel }
\def\xpdr{ \shortmid \cdot\shortmid }
\def\xcrs{\!\times }
\def\xthr{ \Asterisk }

\def\xcpr{ \not\, \parallel}
\def\xcpr{ \nparallel}
\def\xipr{ \smallsetminus\!\!\!\! \parallel}

\def\xcra{ \times\shortmid }
\def\xacr{ \shortmid\times }
\def\xtha{ \shortmid\shortmid\shortmid }

\def\xsta{ \bigstar }

\def\xbta{ \between\mid }

\def\xcrc{ \reflectbox{\varpropto} }
\def\xcrc{\rotatebox[origin=c]{180}{\ensuremath\varpropto}}
\def\xprc{ \shortmid \circ }
\def\xgen{ \text{\textcurrency} }

\def\mMv#1{ \mM_{#1} }
\def\mMarc{ \mMv{\xarc} }
\def\mMpr{ \mMv{\xpar} }
\def\mMprlv#1{ \mMv{\xpar;#1} }
\def\mMprlxyp{ \mMprlv{\bfx'',\bfy''} }
\def\mMcr{ \mMv{\xcrs} }
\def\mMcrlv#1{ \mMv{\xcrs;#1} }
\def\mMcrlxyp{ \mMcrlv{\bfx'',\bfy''} }
\def\mMth{ \mMv{\xthr} }
\def\mMpd{ \mMv{\xpdr} }

\def\mPcr{ \rsPWb(\mMcr) }
\def\mPpr{ \rsPWb(\mMpr) }

\def\dfv#1{ d_{#1} }
\def\dfpr{ \dfv{\xpar} }
\def\dfcr{ \dfv{\xcrs} }

\def\xcdfv#1{\xcdf_{#1} }
\def\xcdfvv#1#2{ \xcdfv{#1;#2} }
\def\xcdfprv#1{ \xcdfvv{\xpar}{#1} }
\def\xcdfcrv#1{ \xcdfvv{\xcrs}{#1} }
\def\xcdfzv#1{ \xcdfvv{0}{#1} }
\def\xcdfprm{ \xcdfprv{m} }
\def\xcdfcrm{ \xcdfcrv{m} }
\def\xcdfzm{ \xcdfzv{m} }
\def\xbcdfv#1{ \boldsymbol{\xcdf_{#1}} }
\def\xbcdfpr{ \xbcdfv{\xpar} }
\def\xbcdfcr{ \xbcdfv{\xcrs} }

\def\prthv#1{ \partial_{\theta_{#1} } }
\def\prtho{ \prthv{1} }
\def\prtht{ \prthv{2} }
\def\prthi{ \prthv{i} }

\def\aIv#1{ \aI_{#1} }

\def\aIpr{ \aIv{\xpar} }
\def\aIcr{ \aIv{\xcrs} }

\def\aIthr{ \aIv{\xthr} }
\def\aIcra{ \aIv{\xcra} }
\def\aIxyz{ \aIv{\bfx,\bfy,\bfz} }

\def\aIvv#1#2{ \aIv{#1;#2} }

\def\mMvrv#1#2{ \mMv{#1;#2} }
\def\mMprvr#1{ \mMvrv{\xpar}{#1} }
\def\mMcrvr#1{ \mMvrv{\xcrs}{#1} }
\def\mMpdvr#1{ \mMvrv{\xpdr}{#1} }
\def\mMthavr#1{ \mMvrv{\xtha}{#1} }
\def\mMcravr#1{ \mMvrv{\xcra}{#1} }
\def\mMacrvr#1{ \mMvrv{\xacr}{#1} }
\def\mMthrvr#1{ \mMvrv{\xthr}{#1} }
\def\mMcprvr#1{ \mMvrv{\xcpr}{#1} }
\def\mMiprvr#1{ \mMvrv{\xipr}{#1} }
\def\mMstvr#1{ \mMvrv{\xsta}{#1} }

\def\mMbtavr#1{ \mMvrv{\xbta}{#1} }

\def\mMbvrv#1#2{ \mMvrv{#1}{#2}^\bullet }
\def\mMcrcvr#1{ \mMbvrv{\xcrc}{#1} }
\def\mMprcvr#1{ \mMbvrv{\xprc}{#1} }
\def\mMgenvr#1{ \mMvrv{\xgen}{#1} }

\def\mMgenrxy{ \mMgenvr{\bfx,\bfy} }
\def\mMcrcxyo{ \mMcrcvr{x_1,y_1} }
\def\mMprcxyo{ \mMprcvr{x_1,y_1} }

\def\mMpdrxw{ \mMpdvr{\bfx,\bfw} }
\def\mMtharxw{ \mMthavr{\bfx,\bfw} }
\def\mMacrrxw{ \mMacrvr{\bfx,\bfw} }
\def\mMcrarxw{ \mMcravr{\bfx,\bfw} }
\def\mMthrrxw{ \mMthrvr{\bfx,\bfw} }
\def\mMcprrxw{ \mMcprvr{\bfx,\bfw} }
\def\mMiprrxw{ \mMiprvr{\bfx,\bfw} }
\def\mMstarxw{ \mMstvr{\bfx,\bfw} }

\def\mMbtarxw{ \mMbtavr{\bfx,\bfw} }

\def\mMcrrxy{ \mMcrvr{\bfx,\bfy} }
\def\mMprrxy{ \mMprvr{\bfx,\bfy} }

\def\prbv#1{ {}_{#1} }
\def\prba{ \prbv{\vbr} }
\def\maMcrrxy{ \prba\mMcrrxy }
\def\maMprrxy{ \prba\mMprrxy }

\def\plbu#1{ #1^{\mathrm{u} } }
\def\plbd#1{ #1^{\mathrm{d} } }
\def\plbm#1{ #1^{\mathrm{m} } }
\def\mMcrarxwu{ \plbu{\mMcrarxw} }
\def\mMcrarxwd{ \plbd{\mMcrarxw} }
\def\mMacrrxwm{ \plbm{\mMacrrxw} }

\def\mMvv#1#2{ \mM_{#1}^{#2} }

\def\mMcrv#1{ \mMvv{\xcrs}{#1} }

\def\mMcri{ \mMcrv{i} }

\def\mMarcxyo{ \mMlarcv{x_1,y_1} }
\def\mMarcxyos{ \mMarcxyo^{\mathrm{s}} }

\def\mMarcxyt{ \mMlarcv{x_2,y_2} }

\def\mMlvv#1#2{ \mM_{#1;#2} }
\def\mMlarcv#1{ \mMlvv{\xarc}{#1} }

\def\mMlpd#1{ \mMlvv{\xpdr}{#1} }

\def\mMvvv#1#2#3{ \mM_{#1;#2}^{#3} }

\def\mMcrvv#1#2{ \mMvvv{\xcrs}{#1}{#2} }

\def\mcnv#1#2{ \left( #1 , #2 \right) }
\def\mcnbv#1#2{ \big( #1, #2 \big) }

\def\mcnbvv#1#2#3{ \mcnbv{#1}{\bsdfv{#2}{#3}} }

\def\mcnnbp#1{ \mcnv{#1}{-\bfsdf} }

\def\mcnbpbp#1{ \mcnbv{#1}{\bfsdf} }
\def\mcnbnbp#1{ \mcnbv{#1}{-\bfsdf} }

\def\mcnbpbf#1{ \mcnbv{#1}{\bfshf} }
\def\mcnbnbf#1{ \mcnbv{#1}{-\bfshf} }

\def\mcnbpbpot#1{ \mcnbv{#1}{\bfsdfot} }

\def\bsdfuv#1{ \bfsdf^{#1} }
\def\bsdfui{ \bsdfuv{i} }

\def\mcnnpi#1{ \mcnv{#1}{-\bsdfui} }

\def\mgen{ v }

\def\mgenv#1{ \mgen_{#1} }
\def\mgenvv#1#2{ \mgen_{#1;#2} }
\def\mgenvvv#1#2#3{ \mgenvv{#1}{#2}^{\mathrm{#3}}}
\def\mgencravv#1#2{ \mgenvvv{\xcra}{#1}{#2} }
\def\mgencraxwv#1{ \mgencravv{\bfx,\bfw}{#1} }
\def\mgencraxwu{ \mgencraxwv{u} }
\def\mgencraxwd{ \mgencraxwv{d} }

\def\xcdf{ \nabla }

\def\xcdfv#1{ \xcdf_{#1} }
\def\xcdfm{ \xcdfv{m} }
\def\xcdfn{ \xcdfv{n} }

\def\dfpvv#1#2{ #1^{#2} \partial_{#1} }
\def\dfpmv#1{ \dfpvv{#1}{m+1} }

\def\shf{ \varphi }
\def\sdf{ \pi }
\def\sdfv#1#2{ \sdf(#1,#2) }
\def\bsdfv#1#2{ \bfsdf(#1,#2) }
\def\bsdfxot{ \bsdfv{x_1}{x_2} }

\def\bsdfxoyt{ \bsdfv{x_1}{y_2} }
\def\bsdfxyt{  \bsdfv{x_2}{y_2} }

\def\bsdfxtyo{ \bsdfv{x_2}{y_1} }

\def\bsdfxoyxo{ \bsdfv{x_1}{y+x_1} }

\def\bshfv#1{ \bfshf(#1) }
\def\bshfxo{ \bshfv{x_1} }
\def\bshfyo{ \bshfv{y_1} }
\def\bshfxoy{ \bshfv{x_1 + y} }

\def\shfvv#1#2{ \shf_{#1}(#2) }
\def\shfmv#1{ \shfvv{m}{#1} }
\def\shfmx{ \shfmv{x} }

\def\sdfvv#1#2#3{ \sdf_{#1}(#2,#3) }
\def\sdfmv#1#2{ \sdfvv{m}{#1}{#2} }
\def\sdfmyot{ \sdfmv{y_1}{y_2} }
\def\sdfmxot{ \sdfmv{x_1}{x_2} }

\def\ssdfvv#1#2#3#4{ \sdf_{#1}(#2,#3,#4) }
\def\ssdfmov#1#2#3{ \ssdfvv{m-1}{#1}{#2}{#3} }
\def\ssdfmv#1#2#3{ \ssdfvv{m}{#1}{#2}{#3} }

\def\spdf{ \rho }

\def\nst{ n }
\def\mlk{ m }

\def\trkv#1{$#1$-rank}
\def\trkq{\trkv{q}}
\def\tdmv#1{$#1$-dimension}
\def\tdmq{\tdmv{q}}

\def\tdgv#1{$#1$-degree}
\def\tdgq{\tdgv{q}}
\def\tdga{\tdgv{a}}

\def\tgrddv#1{$#1$-graded}
\def\tgrddq{\tgrddv{q}}

\def\tgrdv#1{$#1$-grading}
\def\tgrdq{\tgrdv{q}}
\def\tgrda{\tgrdv{a}}
\def\tgrdt{\tgrdv{t}}

\def\xdg{ \mathrm{deg} }
\def\dgv#1{ \xdg_{#1} }
\def\dgq{ \dgv{q} }

\def\fbr{ \mathop{\mathrm{br} } }
\def\fcl{ \mathop{\mathrm{cl} } }

\def\ctmv#1{ [\![ #1 ] \! ] }
\def\ctmdv#1{\ctmv{#1}^{\cdot} }
\def\ctmdvr#1#2{ \ctmdv{#1}_{#2} }

\def\ctmdvbxy#1{ \ctmdvr{#1}{\bfx,\bfy} }
\def\ctmdvbyz#1{ \ctmdvr{#1}{\bfy,\bfz} }
\def\ctmdvbxz#1{ \ctmdvr{#1}{\bfx,\bfz} }
\def\ctmdvbzw#1{ \ctmdvr{#1}{\bfz,\bfw} }
\def\ctmdvbxw#1{ \ctmdvr{#1}{\bfx,\bfw} }

\def\ctmvr#1#2{ \ctmv{#1}_{#2} }

\def\ctmvbzy#1{ \ctmvr{#1}{\bfz,\bfy} }

\def\ctmvbxy#1{ \ctmvr{#1}{\bfx,\bfy} }
\def\ctmvbyz#1{ \ctmvr{#1}{\bfy,\bfz} }
\def\ctmvbzy#1{ \ctmvr{#1}{\bfz,\bfy} }
\def\ctmvbzw#1{ \ctmvr{#1}{\bfz,\bfw} }
\def\ctmvbxz#1{ \ctmvr{#1}{\bfx,\bfz} }
\def\ctmvbxw#1{ \ctmvr{#1}{\bfx,\bfw} }

\def\ctmvbwy#1{ \ctmvr{#1}{\bfw,\bfy} }
\def\vctmv#1#2{ {}_{#1}\ctmv{#2} }
\def\vctmvr#1#2#3{ {}_{#1}\ctmvr{#2}{#3} }

\def\vctmdvr#1#2#3{ {}_{#1}\ctmdvr{#2}{#3} }

\def\actmvr#1#2{ \vctmvr{\bfvbr}{#1}{#2} }

\def\actmdvr#1#2{ \vctmdvr{\bfvbr}{#1}{#2} }

\def\actmvbxy#1{ \actmvr{#1}{\bfx,\bfy} }
\def\actmvbyz#1{ \actmvr{#1}{\bfy,\bfz} }

\def\actmdvbxy#1{ \actmdvr{#1}{\bfx,\bfy} }

\def\lctmv#1{ \vctmv{\bflc}{#1} }
\def\lctdmv#1{ \lctmv{#1}^{\cdot} }

\def\vlctmvr#1#2#3{ \lfsin{\mathrm{l}}{#1}\ctmvr{#2}{#3} }

\def\vlctmdvr#1#2#3{ \lfsin{\mathrm{l}}{#1}\ctmdvr{#2}{#3} }

\def\alctmvr#1#2{ \vlctmvr{\bfvbr}{#1}{#2} }

\def\alctmdvr#1#2{ \vlctmdvr{\bfvbr}{#1}{#2} }

\def\alctmvbxy#1{ \alctmvr{#1}{\bfx,\bfy} }

\def\alctmdvbxy#1{ \alctmdvr{#1}{\bfx,\bfy} }

\def\vrctmvr#1#2#3{ \lfsin{\mathrm{r}}{#1}\ctmvr{#2}{#3} }

\def\arctmvr#1#2{ \vrctmvr{\bfvbr}{#1}{#2} }

\def\arctmvbxy#1{ \arctmvr{#1}{\bfx,\bfy} }

\def\olds{ \check{\;} }

\def\zctmv#1{ \olds[\![ #1 ] \! ] }
\def\zctmdv#1{\zctmv{#1}^{\cdot} }
\def\zctmdvr#1#2{ \zctmdv{#1}_{#2} }

\def\zctmdvbxy#1{ \zctmdvr{#1}{\bfx,\bfy} }

\def\zctmvr#1#2{ \zctmv{#1}_{#2} }

\def\zctmvbxy#1{ \zctmvr{#1}{\bfx,\bfy} }

\def\dmmy{ - }
\def\bdmm{}

\def\Hhom{ \mathrm{HH} }
\def\Hhomv#1{ \Hhom_{#1} }
\def\Hhomxy{ \Hhomv{\bfx,\bfy} }
\def\Hhomxz{ \Hhomv{\bfx,\bfz} }
\def\Hhomtxy{ \Hhomv{\tnsp{\bfx},\bfy} }

\def\tHhomv#1{ \widehat{\Hhom}_{#1} }

\def\tHhomtxy{ \tHhomv{\tnsp{\bfx},\bfy} }

\def\Hml{ \mathrm{H} }

\def\HmWx{ \Hml_{\ctWx} }
\def\HmWX{ \Hml_{\ctWX} }
\def\HmWxp{ \Hml_{\ctWxp} }
\def\Htg{ \mathcal{H} }
\def\zHtg{ \olds\Htg }
\def\HmlK{ \Hml_{\ctKom} }

\def\mod{ \mathop{\mathrm{mod}}\, }

\def\smpe{ e }
\def\smpev#1{ \smpe_{#1} }
\def\smpeo{\smpev{1} }
\def\smpet{ \smpev{2} }

\def\xbbr#1{ \big( #1 \big) }

\def\botarcv#1#2{ \bigotimes_{j=#1}^{#2} \mMlarcv{x_j,y_j} }

\def\xqh{ \simeq_{\catmod{K}} }

\def\xcnA{ A }
\def\xcnAv#1{ \xcnA_{#1} }
\def\xcnAm{ \xcnAv{m} }
\def\xcnAn{ \xcnAv{n} }

\def\xcnhA{ \hat{\xcnA} }
\def\xcnhAv#1{ \xcnhA_{#1} }
\def\xcnhAm{ \xcnhAv{m} }

\providecommand{\mtre}[1]{\lVert#1\rVert}

\def\xcna{ \mathbf{a} }
\def\xcnavv#1#2{ \xcna^{#1}_{#2} }
\def\xcnav#1{ \xcna^{#1} }
\def\xcnacrz{ \xcnavv{\xcrs}{0} }
\def\xcnaprz{ \xcnavv{\xpar}{0} }
\def\xcnacro{ \xcnav{\xcrs} }
\def\xcnalno{ \xcnav{\shortmid} }
\def\xcnapro{ \xcnav{\xpar} }

\def\xcnao{ \xcnav{1} }

\def\xcnxb{ \mathbf{b} }

\def\xcnbvv#1#2{ \xcnxb^{#1} }

\def\xcnbcro{ \xcnbvv{\xcrs}{1} }

\def\xcnbpro{ \xcnbvv{\xpar}{1} }

\def\xcnbA{ \mathbf{A} }

\def\mMarcxyoy{\mMarcxyo\otimes \Qy}

\def\cBgbr#1{ \Big[ #1 \Big] }

\def\xoy{ 1 \otimes y }

\def\shv#1{ \langle #1\rangle }
\def\shvv#1#2{ #1 \shv{#2} }

\def\lc{ \lambda }
\def\bflc{ \boldsymbol{\lc} }

\def\vbr{ \alpha }
\def\be{ \beta }
\def\bfvbr{ \boldsymbol{\vbr} }
\def\bfbe{ \boldsymbol{\be} }

\def\brw{ \mathfrak{w} }
\def\smgb#1{ \check{#1} }

\def\brr{ \mathfrak{b} }
\def\brrws{ \smgb{\brr} }

\def\brrpmo{ \brr_{\pm 1} }

\def\uin{\mathrm{in}}

\def\undrl#1#2{\underaccent{\;\;\;\;\;\;#2}{#1}}
\def\smin{ \undrl{\sim}{\uin} }
\def\smout{ \sim }
\def\eqsmin{ \undrl{\simeq}{\uin} }
\def\eqsmout{ \simeq }

\def\br{ \beta }

\def\sctgmod{\!-\!\ctgmod}
\def\sctmod{\!-\!\ctmod}

\def\IQgmv#1{ \IQ\sctgmodv{#1} }

\def\IQgmt{ \IQgmv{2} }
\def\IQgmh{ \IQgmv{3} }

\def\sctgmodv#1{ \sctgmod_{#1} }

\def\sctgmodt{ \sctgmodv{2} }
\def\sctgmodh{ \sctgmodv{3} }
\def\sctgmodi{ \sctgmodv{i} }

\def\mQv#1{ \IQ_{#1} }
\def\tmQv#1{ \widetilde{\IQ}_{#1} }
\def\mQx{ \mQv{\bfx} }

\def\mQax{ \mQv{\vbr,x} }
\def\mQxy{ \mQv{\bfx,\bfy} }

\def\mQxw{ \mQv{\bfx,\bfw} }

\def\mQxyo{ \mQv{x_1,y_1} }
\def\mQxytzw{ \mQv{x_2,y_2,z,w} }
\def\mQxytz{ \mQv{x_2,y_2,z} }
\def\tmQxytz{ \tmQv{x_2,y_2,z} }

\def\ytmxt{ y_2 - x_2 }

\def\mQxyt{ \mQv{x_2,y_2} }

\def\xnum#1{ |#1| }

\def\xass{associate}
\def\vass#1{$#1$-\xass}
\def\vbrass{\vass{\bfvbr}}

\def\Dlv#1{ \Delta_{#1} }
\def\DlXx{ \Dlv{\bfvbr;\bfx} }
\def\Dlxy{ \Dlv{\bfx;\bfy} }
\def\Dlyz{ \Dlv{\bfy;\bfz} }
\def\Dlyw{ \Dlv{\bfy;\bfw} }

\def\Dlxw{ \Dlv{\bfx;\bfw} }
\def\Dlxyo{ \Dlv{x_1;y_1} }
\def\Dlxyt{ \Dlv{x_2;y_2} }

\def\Dlxytsz{ \Dlxyt[z]}
\def\Dlxytsoz{ \Dlxyt\otimes \Qfrz }
\def\Dlxyp{ \Dlv{\bfx';\bfy'} }

\def\xccl{ \mathfrak{c} }

\def\uWpxybm{$\uWpxy$-module}

\def\Kzc{ \mathrm{K} }
\def\Kzcv#1{ \Kzc(#1) }

\def\smmatr#1{
\bigl( \begin{smallmatrix}
#1
\end{smallmatrix} \bigr)
}

\def\ssmmatr#1{
( \begin{smallmatrix}
#1
\end{smallmatrix} )
}

\def\bfca{ \bfa }
\def\bfcA{ \mathbf{A} }

\def\fqsoA{ f_{\aA;\eqsmout} }
\def\fqsoB{ f_{\aB;\eqsmout} }
\def\fisoAB{ f_{\aA\aB;\cong} }

\def\encv#1{ \hat{#1} }

\def\yL{ L }

\def\liealg#1{ \mathfrak{#1} }
\def\lieg{ \liealg{g} }
\def\Uliev#1{ \mathrm{U}{#1} }
\def\Ulieg{ \Uliev{\lieg} }

\def\Der{ \mathrm{Der}}

\def\bbrs#1{ \bigl( #1 \bigr) }
\def\Bbrs#1{ \Bigl( #1 \Bigr) }

\def\bsbrs#1{ \bigl[ #1 \bigr] }

\def\Ffbrox{ \Ffvv{\brrws_1(\bfvbr)}{\bfx} }

\def\Ffbroz{ \Ffvv{\brrws_1(\bfvbr)}{\bfz} }

\def\xcnb#1#2{ #1 \sqcup #2 }
\def\idbrv#1{ \shortparallel_{#1} }
\def\idbro{ \idbrv{1} }
\def\idbrn{ \idbrv{\nst} }
\def\idbrno{ \idbrv{\nst-1} }

\def\yomo{ y_1 - x_1 }
\def\yomt{ y_1 - x_2 }
\def\yomot{ (\yomo)(\yomt) }

\def\ycna{ a^{\xpar} }
\def\ycnavv#1#2{ \ycna_{#1;#2}}
\def\ycnamv#1{ \ycnavv{m}{#1}}
\def\ycnamoo{ \ycnamv{11}}
\def\ycnamtt{ \ycnamv{22}}
\def\ycnamot{ \ycnamv{12}}
\def\ycnamto{ \ycnamv{21}}
\def\ycnamij{ \ycnamv{ij}}

\def\ycnb{ a^{\xcrs} }
\def\ycnbvv#1#2{ \ycnb_{#1;#2}}
\def\ycnbmv#1{ \ycnbvv{m}{#1}}
\def\ycnbmoo{ \ycnbmv{11}}
\def\ycnbmtt{ \ycnbmv{22}}
\def\ycnbmot{ \ycnbmv{12}}
\def\ycnbmto{ \ycnbmv{21}}
\def\ycnbmij{ \ycnbmv{ij}}

\def\spyom{ \bsdfpyo }
\def\spxtm{ \bsdfpxt }

\def\xgn{ v }
\def\bgn{ \mathbf{\xgn} }

\def\acpvv#1#2{ a_{#1;#2} }
\def\acpmv#1{ \acpvv{m}{#1} }
\def\acpmij{ \acpmv{ij} }

\def\vcp{ \vec{p} }

\def\crF{\mathrm{F}}
\def\crFv#1#2{ \crF_{#1}[#2] }
\def\crFmn#1{ \crFv{m,n}{#1} }

\def\xdgv#1{ \xdg_{\mathrm{#1}} }
\def\xdgq{ \xdgv{q} }
\def\xdga{ \xdgv{a} }
\def\xdgt{ \xdgv{t} }

\def\dgshv#1{ \mathsf{#1} }
\def\dgshq{ \dgshv{q} }
\def\dgsha{ \dgshv{a} }
\def\dgsht{ \dgshv{t} }

\def\dgrv#1{$#1$-degree}
\def\dgrq{\dgrv{q}}
\def\dgra{\dgrv{a}}
\def\dgrt{\dgrv{t}}

\def\grddv#1{$#1$-graded}
\def\grddq{\grddv{q}}

\def\shnst{ (\dgsha\dgsht^{-1})^{\frac{1}{2}\nst} }

\def\wdeq{\;\;=\;\;}
\def\wdcng{\;\;\cong\;\;}

\def\pthe#1{ \partial_{\theta_{#1}} }

\def\xxp{ \tilde{x} }
\def\xxpv#1{ \xxp_{#1} }
\def\xxpo{ \xxpv{1} }
\def\xxpt{ \xxpv{2} }
\def\xxph{ \xxpv{3} }
\def\xxpi{ \xxpv{i} }

\def\mMcrarxwv#1{ \mMcrarxw^{#1} }
\def\mMcrarxwo{ \mMcrarxwv{1} }
\def\mMcrarxwt{ \mMcrarxwv{2} }

\def\xprP{ P }
\def\xprPv#1{ \xprP_{#1}}
\def\xprPt{ \xprPv{2} }

\def\xxf{ f }
\def\xxg{g}
\def\xxfv#1{ \xxf_{#1} }
\def\xxfp{ \xxfv{+} }
\def\xxfm{ \xxfv{-} }

\def\xpp{ p }
\def\xppo{ \xpp_1 }
\def\xppt{ \xpp_2 }

\def\xxqp{ \tilde{q} }
\def\xyq{ q }

\def\xdmv#1{ \dim_{#1} }
\def\xdmq{ \xdmv{q} }

\def\xdoq{ d(q) }
\def\xrkv#1{ \rank_{#1} }

\def\xrkqv#1{ \xrkv{#1;q} }
\def\xrkqx{ \xrkqv{\bfx} }
\def\xrkqy{ \xrkqv{\bfy} }
\def\xrkq{ \xrkv{q} }

\def\qnmv#1{ [#1]_{q} }

\def\lasl{ \mathfrak{sl} }
\def\laslv#1{ \lasl(#1) }
\def\laslt{ \laslv{2} }
\def\laslk{ \laslv{k} }

\def\Sq{ \mathrm{Sq} }
\def\Sqv#1{ \Sq^{#1} }
\def\Sqo{ \Sqv{1} }
\def\Sqt{ \Sqv{2} }

\def\mMp{ \aWp }
\def\xkf{ \mathbbm{k} }

\def\frz{ z }
\def\frw{ w }
\def\Qfrz{ \Qv{\frz} }
\def\zQfrz{ z\Qfrz }

\def\zez{ |_{\frz=0} }

\def\spmf{f}
\def\spmg{g}
\def\spmh{h}

\def\mvcn#1{ \vcenter{#1}}

\usepackage{babel,xcolor,framed,marginnote,blindtext}
\colorlet{shadecolor}{blue!10}

\numberwithin{equation}{section}

\title[Witt algebra and link homology]
{Positive half of the Witt algebra acts on triply graded link homology}

\author[M.~Khovanov]{Mikhail Khovanov}
\address{M.~Khovanov\\
Mathematics Department,
Columbia University \\
2990 Broadway,
New York, NY 10027}
\email{khovanov@math.columbia.edu}
\thanks{The work of M.K. was supported in part by the NSF grants DMS-0739392 and DMS-1005750}
\author[L.~Rozansky]{Lev Rozansky}
\address{
L.~Rozansky\\
Department of Mathematics\\
University of North Carolina at Chapel Hill\\
CB \# 3250, Phillips Hall\\
Chapel Hill, NC 27599
}
\email{rozansky@math.unc.edu}
\thanks{The work of L.R. was supported in part by the NSF grants DMS-0808974 and DMS-1108727}

\date{May 7, 2013}

\input xy.sty
\input xyarrow.tex

\begin{document}
%\draft
%\begin{titlepage}
\maketitle
\begin{abstract}
The positive half of the Witt algebra is the Lie algebra spanned by vector fields $x^{m+1}\frac{d}{dx}$ acting as differentiations on the polynomial algebra $\IQ[x]$ upon which the Soergel bimodule construction of triply graded link homology is based. We show that this action of Witt algebra can be extended to the link homology.
\end{abstract}
%\end{titlepage}

\begin{spacing}{0.65}
\tableofcontents
\end{spacing}

\section{Introduction}
\subsection{Motivation}
Stable cohomological operations on cohomology of topological spaces exhibit the deepest structure,
that of the Steenrod algebra, when the ring of coefficients is $\ZZ/p$ for a prime $p$, while
nontrivial operations don't even exist when the coefficient ring is $\IQ$.
The work of Lipshitz and Sarkar~\cite{LS1, LS2}, see also Kriz, Kriz, and Po~\cite{KKP}, shows that
the Steenrod algebra acts, in the homological direction, on Khovanov $\laslt$ link homology with
coefficients in $\ZZ/p$, and
this action is already highly nontrivial when $p=2$ for the subalgebra generated by $\Sqo$, $\Sqt$,
see \cite{LS2,Seed}.

Gorsky~\cite{Go} and Gorsky, Oblomkov, Rasmussen~\cite{GOR} conjectured
that large algebras related
to affine Lie algebras act on $\laslt$ homology and triply graded homology of torus knots. In this
conjecture the ground ring is $\ZZ$ rather than torsion.  One wonders what part of this structure will
act on cohomology groups of arbitrary links.

A variety of spectral sequences acting between various link and 3-manifold homologies, including
Ozsvath-Szabo 3-manifold homology, Ozsvath-Rasmussen-Szabo link
homology and $\laslk$ homology indicates the existence of a rich structure of cohomological
operations in these theories. Baston-Seed spectral sequence~\cite{BS} from $\laslt$ homology of a link
to tensor product of homology groups of its components comes from one family of such operations.

Reduced Khovanov homology of alternating and quasi-alternating knots lies on one diagonal in the
bigrading plane; for general knots this diagonal direction is the preferred one for the growth of
homology group ranks. It would be strange if the groups living on the same diagonal
were completely unrelated, and we conjecture the existence of cohomological operations acting
along this direction. The knot Floer homology case is similar, and such homological operations,
if found, might help to prove the Fox conjecture that coefficients of the Alexander
polynomial of an alternating knot constitute a trapezoidal sequence.

Categorifications of quantum groups at roots of unity of prime order $p$ over a field $\xkf$ of characteristic
$p$, investigated in~\cite{KQY, QYE}, use derivation $\partial = x^2 \partial / \partial x$ acting
first on the polynomial algebra $\xkf[x]$ in one variable $x$, then on $\xkf[x_1, \dots, x_n]$,
and later on suitable endomorphism algebras of these polynomial spaces, including the nilHecke algebras,
KLR and Webster algebras. Characteristic $p$ is needed to make $\partial$ nilpotent, $\partial^p=0$,
and work in the hopfological setting, over the base monoidal triangulated category of stable graded
modules over the Hopf algebra $\xkf[\partial]/(\partial^p)$, generalizing homotopy category of complexes.

This fails in characteristic zero, but one can at least study the symmetries $L_m = x^{m+1}
\partial/\partial x$, $m\geq 0$, acting on $\xkf[x]$. Symmetries $L_m$ span the positive
half $\mMp$ of the Witt Lie algebra, and their action  induces an action of $\mMp$ on nilHecke, KLR, and other
algebras categorifying quantum groups and their representations. One might hope that the entire
Webster's categorification of Reshetikhin-Turaev link invariants for arbitrary simple Lie algebras
and their representations~\cite{Web1,Web2} can be done $\mMp$-equivariantly, leading to an action
of the latter on all these link homology theories.

In finite characteristic one can expect an action of a larger algebra.
Recent papers of Beliakova, Cooper~\cite{BC} and Kitchloo~\cite{Kit} exhibit a Steenrod algebra action on the nil-Hecke algebra. If this action extends naturally to Webster algebras, associated bimodules, and link homology in an invariant way, there will be a Steenrod algebra action on bigraded link homology groups. Curiously, the direction of that action (if it exists) in the bigrading plane will not match that of the Lipshitz-Sarkar Steenrod algebra action~\cite{LS1,LS2}. Overall, we expect the existence of very large algebras of cohomological operations on link homology for various homology theories and coefficient rings.

%
%
%In finite characteristic there should be, in addition, an action of the Steenrod algebra,
%% it might be reasonable to expect an action of a different algebra, which contains the Steenrod algebra,
%see the recent papers of Beliakova and Cooper~\cite{BC} and Kitchloo~\cite{Kit}. The direction of that action in the bigrading plane does not match the Lipshitz-Sarkar Steenrod algebra action~\cite{LS2}.
%%as well as the universal enveloping algebra of $\mMp$ and perhaps some additional generators.
%%
%%perhaps including, besides $L_m$, some divided powers
%%and the Steenrod algebra.

In this paper we work in the simplest instance of the categorified Schur-Weyl dual set-up. Instead
of looking at the action of $\mMp$ on rings categorifying quantum groups, we consider an action of
$\mMp$ on Soergel bimodules serving as building blocks for the Soergel category, which
categorifies the Hecke algebra of the symmetric group. We make this action compatible with
the Rouquier braid group action~\cite{Rou, Rousl2} via complexes of Soergel bimodules and with the Hochschild
homology description of triply-graded categorification of the HOMFLYPT polynomial,
resulting in an action of $\mMp$ (and its universal enveloping algebra) on the triply-graded
link homology of~\cite{KR2,Khtgd}.

\subsection{Notations and definitions}
Throughout the paper we use notation $\bfa=a_1,a_2,\ldots$ for a finite or infinite \tsq\ of elements of a set $\setA$ and $\xnum{\bfa}$ for the length of $\bfa$.
%The number of elements in the \tsq\ is denoted as $\xnum{\bfa}$.
We also use a shortcut $\bfa\in\setA$ for $a_i\in\setA$, $i\geq 1$.
In particular,
$\bfx=x_1,\ldots,x_n$ is a \tsq\ of variables with $\xnum{\bfx}=n$.
%; the number of variables is denoted $\xnum{\bfx} = n$.

All modules in this paper are $\ZZ$-graded; we refer to this grading as \tgrdq\ and denote it $\xdgq$. The \tdgq\ of all main variables such as $\bfx$, $\bfvbr$ and $\bflc$ is 2. We work over the base field $\IQ$ of rational numbers.
The positive half $\aWp$ of the Witt algebra is a \tgrddq\ Lie algebra with generators $\xLm$, $m\geq 0$, $\xdgq\xLm=2m$, and relations
\begin{equation}
\label{eq:cmrl}
[\xLm,\xLn] = (n-m) \xLv{m+n}.
\end{equation}
This algebra acts by differentiations on the polynomial algebra $\Qx$: $\xhLm =  \dfpmv{x}$.
% x^{m+1} \frac{d}{dx}$.
 More generally, it acts on the algebra $\Qbx$, $\xnum{\bfx}=\brnk$, as
$\xhLm = \sum_{i=1}^{\brnk}  \dfpmv{x_i}$. We refer to this action as the \stact.

Let $\uWp$ denote the universal enveloping algebra of $\aWp$. The \stact\ of $\uWp$ on $\Qbx$ allows us to define a \sdrpr\ $\uWpx = \Qbx\rtimes\uWp$ with the multiplication
\begin{equation}
\label{eq:scndrl}
(p\otimes \xLm)(q\otimes \xLn) = pq\otimes \xLm\xLn + p(\xhLm q)\otimes \xLn,
\end{equation}
where $p,q\in\Qbx$.
%%21
For two \tsq s of variables $\bfvbr$ and $\bfx$ we also use notation
\[
\uWpXx=\uWpv{\bfvbr,\bfx}=\Qv{\bfx,\bfvbr}\rtimes\uWp,
\]
if we want to separate $\bfvbr$ from $\bfx$ in order to emphasize that the variables $\bfvbr$ are treated differently from $\bfx$. Note that $\uWp$ acts on both $\bfx$ and $\bfvbr$.

We use a subscript $\bfx$, as in $\mMx$, to denote a $\uWpx$-module. Then for another \tsq\ of variables $\bfy$, $\xnum{\bfy}=\xnum{\bfx}$, $\mMy$ denotes the corresponding module over $\uWpy$. Similarly, we use notation
$\mXMx$ for a $\uWpXx$-module.
We denote by $\mQx$ the algebra $\Qbx$ considered as a $\uWpx$-module with the \stact\ of $\uWp$.

%For two modules $\mMxy$ and $\mNyz$ we define the $\uWpxz$-module $\mMxy\oty\mNyz$ as a usual tensor product over $\Qby$ followed by forgetting the $\Qby$ action, while the action of a generator $\xLm\in\aWp$ on it is defined with the help of the Leibnitz rule: $\xLm\otimes\xId + \xId\otimes\xLm$. We will also use the notation
%$\mMxy\otuy\mNyz$ for the $\uWpv{\bfx,\bfy,\bfz}$-module constructed in the same way, except that the $\Qby$-module structure is \emph{not} forgotten.
%
%For two lists of variables $\bfx$, $\bfy$ of equal length we use a notation $(\bfy-\bfx) = (y_1-x_1,\ldots,y_n-x_n)$ for the ideal generated by the pairwise differences of corresponding variables.
%We introduce a special `diagonal' $\uWpxy$-module $\Dlxy = \Qbxy/(\bfy-\bfx)$, which has an obvious property: for any $\uWpx$-module $\mMx$ there is an isomorphism
%$\mMx\oty\Dlxy\cong \mMy$.

Define the curvature of an infinite \tsq\ of polynomials $\bfa=a_0,a_1,\ldots\in\Qbx$ as a double \tsq
\begin{equation}
\label{eq:odcrv}
\crFmn{\bfa} = \xhLm a_n - \xhLn a_m - (n-m)\,a_{m+n}.
\end{equation}
A \tsq\ $\bfa$
%An infinite list of polynomials $\bfa=a_0,a_1,\ldots\in\Qbx$
is called \Wfl\ if its curvature is zero:
\begin{equation}
\label{eq:shprp}
\crFmn{\bfa} = 0,\qquad m,n\geq 0.
\end{equation}
%if the polynomials satisfy the condition
%\begin{equation}
%\label{eq:shprp}
%\xhLm a_n - \xhLn a_m = (n-m)\,a_{m+n}.
%\end{equation}
For example, the \tsq\
\begin{equation}
\label{eq:dfppr}
\bsdfpx\in\Qv{x},\qquad \spdpmx = (m+1)x^m,
\end{equation}
is obviously \Wfl.

%A \Wfl\ \tsq\ determines an automorphism $\autfa$ of $\uWpx$:
%\begin{equation}
%\label{eq:autfp}
%\autfa(p\otimes \xLm) = p \otimes (\xLm + a_m).
%\end{equation}
%This automorphism, in turn, determines an invertible \enfn\ $\shv{\bfa}$ of the graded module category $\uWpx-\ctgmod$ and its derived category $\ctD(\uWpx-\ctgmod)$: the action of $\uWpx$ on a module $\mMx\shv{\bfa}$ is the action on $\mMx$ modified by the automorphism $\autfa$.

A \Wfl\ \tsq\ determines an automorphism of $\uWpx$:
\begin{equation}
\label{eq:autfp}
p\otimes \xLm \longmapsto p \otimes (\xLm + a_m).
\end{equation}
This automorphism, in turn, determines an invertible \enfn\ $\shv{\bfa}$ of the category $\uWpx\sctgmod$
of  graded left $\uWpx$-modules and its derived category $\ctD(\uWpx\sctgmod)$: the action of $\uWpx$ on a module $\mMx\shv{\bfa}$ is the action on $\mMx$ modified by the automorphism~\eqref{eq:autfp}.

\subsection{Nested derived categories}
\label{sbs:ndc}
%\subsubsection{A definition of quasi-isomorphism by a forgetful functor}
%\subsection{Nested categories}
%We will often use a notation $\mMx$ for a $\uWpx$-module $\mM$. On one hand, the subscript $\bfx$ reminds of algebra generators and, on the other hand, the notation $\mMy$ will denote the same module over $\uWpy$, if $\xnum{\bfx}=\xnum{\bfy}$. Similarly, we use a notation $\mXMx$ for a module over $\uWpXx$.

In this paper we consider only small categories.
Let $\ctAd$ be an additive category. Following the notations of Weibel's book~\cite{Weibel}, $\ctCh(\ctAd)$ denotes the chain category, whose objects are complexes over $\ctAd$, its morphisms being chain maps, while $\ctK(\ctAd)$ denotes the homotopy category, its morphisms being chain maps modulo null-homotopies.

Since old and new constructions of \tglh\ involve `\nstd' homotopy and derived categories, we will use the following shortcut notations for derived categories of graded modules over algebras:
\begin{equation}
\label{eq:twocatnots}
\ctQv{\bfx} =\ctD( \Qbx\sctgmod),\qquad\ctWx = \ctD(\uWpx\sctgmod),
%%21
\qquad \ctXWx = \ctD(\uWpXx\sctgmod).
\end{equation}
%
%\subsection{Nested derived categories}
The categories~\eqref{eq:twocatnots}
%These categories
are additive, hence we may use the categories of complexes and homotopy categories over
%%21
them, such as
%$\ctCh(\ctXWx)$ and $\ctK(\ctXWx)$.
$\ctCh(\ctWx)$ and $\ctK(\ctWx)$.
We call these categories `\nstd'.
%The latter categories are `\nstd',
%%$\ctKom(\ctQv{\bfx})$ and $\ctKom(\ctXWx)$. The latter categories are nested,
%hence
%we use the notations $\smin$, $\smout$ and $\eqsmin$, $\eqsmout$ for homotopy equivalence and \qiso\ (to be defined shortly) within the inner and outer categories respectively. When we need to distinguish between complexes within the inner and outer categories, we use a square bracket notation $\bsbrs{\cdots \rightarrow M_i\rightarrow M_{i+1}\rightarrow\cdots}$ for a complex within the inner category and a boxed notation
%$\boxed{\cdots \rightarrow M_i\rightarrow M_{i+1}\rightarrow\cdots}$ for a complex within the outer category.

In addition to outer chain and homotopy categories, we also use `relative' \nstd\ derived categories $\ctDr(\ctWx)$ and $\ctDr(\ctXWx)$. First, consider a general setup. Suppose that two additive categories $\ctA$ and $\ctB$ are related by an additive `forgetful' functor $\Ffr\colon\ctA\rightarrow\ctB$ which is extended to the chain categories: $\Ffr\colon\ctCh(\ctA)\rightarrow\ctCh(\ctB)$.
\begin{definition}
\label{df:relder}
A morphism $f\in\Hom_{\ctCh(\ctA)}(\aAb,\aBb)$ between two complexes in $\ctCh(\ctA)$ is called a \emph{\qiso} if $\Ffr(f)\colon\Ffr(\aAb)\rightarrow\Ffr(\aBb)$ is an isomorphism in
$\ctK(\ctB)$.
%$\Ffr$ makes its cone contractible in $\ctCh(\ctB)$: $\Ffr\bigl( \Cone(f) \bigr) \sim 0$.
\end{definition}

If $f,g\colon\aAb\rightarrow\aBb$ are homotopic and $f$ is a \qiso, then $g$ is also a \qiso, hence the notion of \qiso\ extends to the homotopy category $\ctK(\ctA)$.
%
%If two complexes in $\ctCh(\ctA)$ are homotopy equivalent, then they are \qisc, hence the notion of \qiso\ extends to $\ctK(\ctA)$.

The collection of \qiso s in $\ctK(\ctA)$ is a saturated localizing system compatible with triangulation. Denote by $\ctDFfr(\ctA)$ the localization of $\ctK(\ctA)$ relative to \qiso s (the index $\mathrm{r}$ means "relative"). It is a triangulated category and the localization functor $\Qf\colon\ctK(\ctA)\rightarrow\ctDFfr(\ctA)$ is exact (see Lemma~5.3 and Proposition 5.5 in~\cite[Section 5]{dcr}). The image $\Qf(\aAb)$ of an object $\aAb$ in $\ctK(\ctA)$ is isomorphic to the zero object in $\ctDFfr(\ctA)$ iff $0\rightarrow\aAb$ is a \qiso. This is also equivalent to the existence of a distinguished triangle
\[
\xymatrix{
\aAb'
\ar[r]^-{f}
&
\aAb''
\ar[r]
&
\aAb
\ar[r]
&
\aAb'[1]
}
\]
%$(\aAb',\aAb'',\aAb,f,g,h)$
in $\ctK(\ctA)$
such that $f\colon \aAb'\rightarrow\aAb''$ is a \qiso\ (\cite[Lemma 6.7]{dcr}).
%
%The \rdc\ $\ctDFfr(\ctA)$ or simply $\ctDr(\ctA)$ is the result of taking a quotient of $\ctCh(\ctA)$ over the equivalence relation generated by \qiso s.

The chain, homotopy, and relative derived category are related by functors:
\begin{equation}
\label{eq:chofmps}
%\ctCh(\ctA)\longrightarrow\ctK(\ctA)\xrightarrow{\;\;\Qf\;\;}\ctDr(\ctA).
\xymatrix{\ctCh(\ctA)\ar@/^2pc/[rr]^-{\Fffq} \ar[r] & \ctK(\ctA) \ar[r]_-{\Qf} & \ctDr(\ctA),}
\end{equation}
where we denote by $\Fffq$ the composition of the quotient functor and the localization functor $\Qf$.
%%Both categories are related by the quotient functor $\Fffq\colon\ctKom(\ctA)\rightarrow\ctD(\ctA)$.

The definition of $\ctDFfr(\ctA)$ as the localization relative to \qiso s specializes to the definition of the derived category in two familiar cases.
%definition~\ref{df:relder} matches the standard definition of the derived category in two familiar cases.
The first case is when $\ctA$ is the category of modules over an algebra $\alA$ (over $\IQ$), $\ctB$ is the category of $\IQ$-vector spaces and $\Ffr$ forgets the $\alA$-module structure. Then $\ctDr(\ctA)$ is the derived category of $\alA$-modules. The second case is when $\ctA$ is again the category of modules over $\alA$,  $\ctB$ is the category of modules over its subalgebra $\alB\subset \alA$ and $\Ffr\colon\alA\sctmod\longrightarrow \alB\sctmod$ is the restriction functor. Then $\ctDr(\ctA)$ is the relative derived category $\ctD\bbrs{(\alA,\alB)\sctmod}$.

In this paper we use the `nested' version of the second example. We set $\ctA = \ctD(\alA\sctgmod)$, $\ctB=\ctD(\alB\sctgmod)$ (algebras $\alA$, $\alB$ are now graded), and $\Ffr$ is the restriction functor extended to these derived categories. The result is the nested relative derived category
$
%\ctDr(\ctA) =
\ctDr\bbrs{\ctD(\alA\sctgmod)}$. Note that its objects are triply graded: the first grading comes from $\alA$, the second is the homological grading of $\ctD$ and the third is the homological grading of $\ctDr$.

%%21
More specifically, we take $\alA = \uWpXx$, $\alB=\Qbx$, and $\Ffr$ restricts $\uWpXx$-modules to $\Qbx$-modules. Then we get the nested relative derived category $\ctDr(\ctXWx)$.
%In this paper we take $\ctA=\ctXWx$, $\ctB=\ctQv{\bfx}$ and $\Ffr$ restricts $\uWpXx$-modules to $\Qbx$-modules. We define the relative derived category $\ctDr(\ctXWx)$ as the derived category corresponding to this functor.
Overall, there is a chain of functors
%\begin{equation}
%\label{eq:chct}
%\xy
%\xymatrix"M"{
%\uWpXx\sctgmod
%\ar[r]
%&
%\ctXWx = \ctD(\uWpXx\sctgmod)
%\ar[r]
%&
%\ctCh(\ctXWx)
%\ar[r]
%\ar@/^2pc/[rr]^-{\Fffq}
%&
%\ctK(\ctXWx)
%\ar[r]_-{\Qf}
%&
%\ctDr(\ctXWx)
%}
% \POS
% "M1,1"*\frm{_\}},%+D*++!U\txt{abelian}
%% "M2,2"."M2,3"!C*\frm{^\}},+U*++!D\txt{A to B}
%% ,"M2,4"."M2,5"!C*\frm{^\}},+U*++!D\txt{C to D}
%% ,"M3,3"."M3,4"!C*\frm{_\}},+D*++!U\txt{F to G}
%\endxy
%\end{equation}
\begin{equation}
\label{eq:chct}
\begin{tikzpicture}[normal line/.style={->}]
  \matrix (m) [matrix of math nodes,row sep=3em,column sep=2em,minimum width=2em]
  {
    \uWpXx\sctgmod
  & \ctXWx = \ctD(\uWpXx\sctgmod)
  & \ctCh(\ctXWx)
  & \ctK(\ctXWx)
  & \ctDr(\ctXWx)\\};
%    \\
%     A_t & A \\};
     \path[normal line,font=\scriptsize]
     (m-1-1) edge (m-1-2)
     (m-1-2) edge (m-1-3)
     (m-1-3) edge (m-1-4)
     (m-1-4) edge node[auto] {$\Qf$} (m-1-5)
     (m-1-3.35) edge [bend left=20] node[auto] {$\Fffq$}  (m-1-5.145);
\draw[decorate,decoration={brace,mirror}]
  (m-1-1.south west) -- (m-1-1.south east) node[midway,auto,swap,font=\scriptsize] {abelian};
\draw[decorate,decoration={brace,mirror}]
  (m-1-2.south west) -- (m-1-2.south east) node[midway,auto,swap,font=\scriptsize] {derived};
\draw[decorate,decoration={brace,mirror}]
  (m-1-3.south west) -- (m-1-5.south east) node[midway,auto,swap,font=\scriptsize] {nested};
\end{tikzpicture}
\end{equation}
The category $\ctDr(\ctWx)$ is a particular case of $\ctDr(\ctXWx)$ when the \tsq\ $\bfvbr$ is empty. Often we use a combined \tsq\ of variables $\bfx,\bfy$ instead of $\bfx$, thus getting the categories $\ctDr(\ctXWxy)$ and
$\ctDr(\ctWxy)$.
%
%More specifically, we take $\alA = \uWpxy$, $\alB=\Qbxy$, $\Ffr$ restricts $\uWpxy$-modules to $\Qbxy$-modules, and we denote this nested relative derived category by
%$\ctDr(\ctWxy)$.
%%21
%The category $\ctDr(\ctWxy)$ is a particular case of $\ctDr(\ctXWxy)$ when the \tsq\ $\bfvbr$ is emply.

%Notations for homotopy equivalence and quasi-isomorphisms.
Since objects in nested categories are `complexes of complexes', we have to distinguish between homotopy equivalences in the inner and outer categories. Hence we use notations $\smin$ and $\eqsmin$ for homotopy equivalence and \xqiso\ in the inner category, such as $\ctCh(\uWpXx\sctgmod)$, and $\smout$, $\eqsmout$ for homotopy equivalence and \qiso\ in the outer category, such as $\ctCh(\ctXWx)$. When we need to distinguish between complexes in the inner and outer categories, we use a square bracket notation $\bsbrs{\cdots \rightarrow M_i\rightarrow M_{i+1}\rightarrow\cdots}$ for a complex in the inner category and a boxed notation
$\boxed{\cdots \rightarrow M_i\rightarrow M_{i+1}\rightarrow\cdots}$ for a complex in the outer category, with each $M_i$ itself a complex.

\subsection{Derived partial tensor products}
% and associated monoidal structure}

In this paper we use two versions of the partial tensor product:
%\[
%\xymatrix{
%&(\Qbxy-\ctmod)\times (\Qbyz-\ctmod)
%\ar[ld]_(0.6){\otuy}
%\ar[rd]^(0.6){\oty}
%\\
%\Qv{\bfx,\bfy,\bfz}
%&&
%\Qbxz
%}
%\]
\[
\xymatrix{
&(\uWpxy\sctgmod)\times (\uWpyz\sctgmod)
\ar[ld]_(0.6){\otuy}
\ar[rd]^(0.6){\oty}
\\
\uWpv{\bfx,\bfy,\bfz}\sctgmod
&&
\uWpxz\sctgmod
}
\]
Both versions correspond to the usual tensor product $\otv{\Qby}$, while the action of the $\aWp$ generators $\xLm$ is defined by the Leibnitz rule: $\xLm$ acts on a tensor product $\mMxy\otv{\Qby}\mNyz$ of $\uWp$-modules as $\xLm\otimes\xId + \xId\otimes\xLm$.
The difference between $\otuy$ and $\oty$ is that the former remembers the $\Qby$-module structure, while the latter forgets it. In other words, the functor $\oty$ is the composition of $\otuy$ and the functor of forgetting the $\Qby$-module structure. The following properties are shared by both products $\otuy$ and $\oty$, and we will formulate them mostly for $\oty$.

Tensor products $\oty$ and $\otuy$ extend in left-derived form to derived categories, \eg%
%
%left-derived version of the tensor product $\oty$ extends to derived categories
\begin{equation}
\label{eq:drtnpr}
\xymatrix{
\ctWxy\times \ctWyz
\ar[r]^-{\otdry}
&
\ctWxz
}
\end{equation}
and, further, to the chain and homotopy categories built over them:
\begin{equation}
\label{eq:tpKm}
\xymatrix{
\ctK(\ctWxy)\times \ctK(\ctWyz)
\ar[r]^-{\otdry}
&
\ctK(\ctWxz)
}
\end{equation}
Theorem~\ref{th:drtnpr} says
%We will show
that the latter tensor product descends directly to the nested relative derived categories
such as $\ctDr(\ctWxy)$ without the need for any additional derivation:
\begin{equation}
\label{eq:tpDr}
\xymatrix{
\ctDr(\ctWxy)\times \ctDr(\ctWyz)
\ar[r]^-{\otdry}
&
\ctDr(\ctWxz).
}
\end{equation}

%%21
%\begin{SpecialPar}
When extra variables $\bfvbr$ are present, we use the `left' version of the tensor product~\eqref{eq:drtnpr}:
%\begin{equation}
%\label{eq:lwfntso}
%\xymatrix@C=1.60cm@R=0.2cm{
%&
%\ctXWxy \times \ctYWyz
%\ar[dl]_-{\otdrfly}
%\ar[dr]^-{\otdrfry}
%\\
%\ctXWxz
%&&
%\ctYWxz
%}
%\end{equation}
\begin{equation}
\label{eq:lwfntso}
\xymatrix@C=1.60cm@R=0.2cm{
\ctXWxy \times \ctYWyz
\ar[r]^-{\otdrfly}
&
\ctXWxz
}
\end{equation}
The name `left' has nothing to do with derivation. It stems from the fact that
the product $\otdrly$ keeps the $\Qv{\bfvbr}$-module structure, while forgetting the $\Qv{\bfbe}$-module structure. This tensor product
also extends to nested relative derived categories
\[
\xymatrix@C=1.60cm@R=0.2cm{
\ctDr(\ctXWxy) \times \ctDr(\ctYWyz)
\ar[r]^-{\otdrfly}
&
\ctDr(\ctXWxz).
}
\]

 extends to the homotopy category $\ctK(\ctXWxy)$ and to the derived category $\ctDr(\ctXWxy)$.
%\end{SpecialPar}

For two finite \tsq s of variables $\bfx$, $\bfy$ of the same length we use notation
\[
(\bfy-\bfx) = (y_1-x_1,\ldots,y_n-x_n)
\]
for the ideal in $\Qbxy$ generated by the pairwise differences of corresponding variables.
We introduce a special `diagonal' $\uWpxy$-module
\[
\Dlxy = \mQxy/(\bfy-\bfx),
\]
which has an obvious property: for any $\uWpx$-module $\mMx$ there is an isomorphism
\linebreak
$\mMx\oty\Dlxy\cong \mMy$.
%We will show later that
This property persists at the derived level:
\[
\mMx\otdry\Dlxy\simeq\mMy.
\]
Recall that $\xLm 1 =0$, where $1$ is the unit element of $\mQxy$.

\subsection{Gradings}
\label{ss:grd}

For an algebra $\alA$, let $\alA\sctgmodi$ denote the category of $\ZZ^i$-graded $\alA$-modules. If $\alA$ is graded, then one grading matches the grading of $\alA$, while others have external origin. Sometimes we drop $i$ and use notation $\alA\sctgmod$.

The \tglh\ construction involves three gradings called $q$-, $a$- and $t$-gradings. The corresponding degrees are denoted $\xdgq$, $\xdga$ and $\xdgt$ and the degree shifting functors are denoted $\dgshq$, $\dgsha$ and $\dgsht$ (\eg $\dgshq^2$ is the functor that shifts the \dgrq\ up by two).

We have already introduced the \tdgq. The algebras $\Qbx$ and $\uWpXx$ and their modules are \grddq:
\begin{equation}
\label{eq:dgrs}
\xdgq \bfx =
%%21
\xdgq\bfvbr =
2,\qquad\xdgq \xLm = 2m.
\end{equation}
Within the nested categories
%%21
such as
%$\ctCh(\ctXWx)$ and $\ctDr(\ctXWx)$
$\ctCh(\ctWx)$ and $\ctDr(\ctWx)$
the \dgra\ is the homological degree in the inner category, while
%the \dgrt\ is the sum of homological degrees in both directions, which determines the parity sign factors.
the \dgrt\ is the homological degree in the outer category. The \dgra\ and the \dgrt\ may simultaneously take half-integer values, however, the total homological degree is their sum, hence all sign factors associated with homological parity are well-defined.

In order to avoid overloading the formulas, we will ignore the gradings everywhere, except in most important formulas defining the categorification of elementary braids~\eqref{eq:bcmps} and the application of \Hchsh~\eqref{eq:hho} and~\eqref{eq:hht}
which define our construction.

\subsection{Outline of the homology construction}

For a set $\setA$ and a small category $\cCat$ notation $f\colon\setA\rightarrow\cCat$ means a map from the set $\setA$ to the set of isomorphism classes of objects of $\cCat$. If $\setA$ is a \smgr\ and $\cCat$ has a monoidal structure, then we assume that $f$ is a homomorphism: it intertwines the product in $\setA$ with the product on the set of isomorphism classes induced by the monoidal structure.

For a topological object (a \brwd, a braid or a link) $\tau$ we use notation $\ctmv{\tau}$ for the associated object or complex constructed though \emph{categorification} of an algebraic invariant such as the \hmpt\ polynomial~\cite{HOMFLY, PT}. The upper-left `check' decoration $\zctmv{\tau}$ indicates that this is the original categorification of~\cite{Khtgd}. The upper-right dot $\ctmdv{\tau}$ indicates that %this is a `pre-categorification', that is,
a map $\ctmdv{\dmmy}\colon\tsetT\rightarrow\tcCat$ from a set of topological objects $\tsetT$ to the (isomorphism classes of) objects of a category $\tcCat$ has a built-in invariance: there is an equivalence relation within $\tsetT$ and an equivalence relation between the objects of $\tcCat$, so that if $\tau_1\sim\tau_2$, then $\ctmdv{\tau_1}\sim\ctmdv{\tau_2}$, and as a result there is an unmarked bracket map $\ctmv{\dmmy}$ between the two quotient sets $\tsetT/\sim$ and
$\tcCat/\sim$.
%
%the set $\setT$ to which $\tau$ belongs, has an equivalence relation $\sim$ and we will show that the bracket $\ctmdv{\dmmy}$ is invariant under it: if $\tau_1\sim\tau_2$, then $\ctmdv{\tau_1}\cong\ctmdv{\tau_2}$, so the bracket can be extended as $\ctmv{\dmmy}$ to the quotient set $\setT/\sim$.

\subsubsection{Original construction}
The following commutative diagram summarizes the \Sgl\ \bmdl~\cite{Sgl1,Sgl2} construction of the \tgh\ in~\cite{Khtgd}:
\begin{equation}
\label{eq:oldd}
\mvcn{
\xymatrix@C=4pc@R=4pc{
\brwgrn \ar @{->>}[r]^-{\fbr} \ar[d]^-{\zctmdvbxy{\dmmy}}
&
\brgrn
%\ar @{->>}[r]^-{\fcl}
\ar[d]^-{\zctmvbxy{\dmmy}}
\ar[r]^-{\fcl}
&
\lnks
%\ar@{^{(}->}[r]
\ar[dr]^-{\zHtg\bdmm}
\ar@{-->}[d]
%&
%\lnks
%\ar[d]^-{\zHtg(\dmmy)}
%\ar@{-->}[dl]
\\
\ctCh(\ctQxy)
\ar[r]
&
%\tCatbxy
\ctK(\ctQxy)
\ar[r]^-{\Hhom\bdmm}
& \ctK(\IQgmt) \ar[r]^-{\Hm\bdmm}
&\IQgmh
}
}
\end{equation}
%\[
%\ctmv{aswdf}_{\uWpXxy}
%\]
where the category $\ctQxy=  \ctD( \Qbx\sctgmod)$,  see also formula~\eqref{eq:twocatnots}.

The top row of the diagram is purely topological. $\brwgrn$ is the semigroup of $\nst$-strand \brwd s, that is, a semigroup freely generated by the elements $\sggni$ and $\sggxii$, $1\leq i\leq\nst -1$. $\brgrn$ is the braid group and the map $\fbr$ turns $\sggni$ and $\sggxii$ into a braid group generator $\sggni$  and its inverse $\sggni^{-1}$. $\lnks$ is the set of framed oriented links and the map $\fcl$ performs the circular closure of a braid.

%$\lnksn$ is a set of framed links presentable as a circular closure of an $\nst$-strand braid, and $\fcl$ performs this closure. $\lnks$ is the set of all framed links.

\Tsq s $\bfx$ and $\bfy$ appearing in $\ctQxy$ contain $\nst$ variables each, one per braid strand position: $\xnum{\bfx}=\xnum{\bfy}=\nst$.
 %= x_1,\ldots x_{\nst}$ and $\bfy = y_1,\ldots,y_{\nst}$.
The bracket
%$\zctmdv{\dmmy}$
$\zctmdvbxy{\dmmy}$
 maps \brwd s into Rouquier complexes~\cite{Rou} of \Sglb s. The category $\ctQxy$ has a monoidal structure coming from the derived tensor product over intermediate variables:
\begin{equation}
\label{eq:dtpr}
\ctQxy \times \ctQyz\xrightarrow{\;\;\otdr_{\Qby}\;\;} \ctQxz,
\end{equation}
and this structure extends to the chain category $\ctCh(\ctQxy)$. The bracket $\zctmdvbxy{\dmmy}$ turns the product in $\brwgrn$ into the tensor product~\eqref{eq:dtpr}, hence it is determined by its value on generators $\sggni$ and $\sggxii$.

In order to prove the existence of the map
%$\ctmv{\dmmy}$
$\zctmvbxy{\dmmy}$
taking braids to isomorphism classes of objects in $\ctK(\ctQxy)$
%acting on braids
one has to verify that if two \brwd s $\brw_1$ and $\brw_2$ represent the same braid, $\fbr(\brw_1)=\fbr(\brw_2)$, then their complexes are homotopy equivalent:
$\zctmdv{\brw_1} \smout \zctmdv{\brw_2}$.

The bracket
%$\zctmv{\dmmy}$
$\zctmvbxy{\dmmy}$
has a `variable sliding' property.
Let $\brrws\in S_{\nst}$ be the symmetric group element corresponding to a braid $\brr$, so that $x_i$ and $y_{\brrws(i)}$ are variables at opposite ends of the same braid strand. Then
%If the variables $x_i$ and $y_j$ correspond to the same \brstr\ of $\brr$,
endomorphisms
$
\encv{x}_i,\encv{y}_{\brrws(i)}\in\End_{\ctCh(\ctQxy)}\bigl(\zctmvbxy{\brr}\bigr)
$
of multiplication by $x_i$ and $y_{\brrws(i)}$ are homotopic to each other:
%. If $x_i$ and $y_j$ correspond to the same \brstr, then the corresponding endomorphisms are homotopic:
\begin{equation}
\label{eq:endsld}
\encv{x}_i \smout \encv{y}_{\brrws(i)}.
\end{equation}
%where $\brrws\in S_{\nst}$ is the symmetric group element corresponding to the braid $\brr$.

%$\IQgm$ denotes the category of graded vector spaces.
\Hchsh\ is a functor $\Hhom\colon\ctQxy\rightarrow\IQgmt$, and the functor $\Hhom$ in the diagram is its extension to the homotopy categories over $\ctQxy$ and $\IQgmt$. Finally, homology functor $\Hm(\dmmy)$ establishes an equivalence between categories $\ctK(\IQgmt)$ and $\IQgmh$, so there is no additional benefit in introducing a dashed arrow mapping links into complexes of $\IQ$-vector spaces. In order to prove the existence of the homology map $\zHtg(\dmmy)$, one has to verify that if the closures of two braids $\br_1$ and $\br_2$ represent isotopic framed oriented links, then their homologies are isomorphic: $\Hm\bigl(\Hhom(\zctmvbxy{\br_1})\bigr)\cong\Hm\bigl(\Hhom(\zctmvbxy{\br_2})\bigr)$.

Property~\eqref{eq:endsld} allows one to endow the link homology $\zHtg(\yL)$ of a $\mlk$-component link with the structure of a $\Qblc$-module, where the variables $\bflc$, $\xnum{\bflc}=\mlk$ are assigned bijectively to  link components.
%
%has an implication for the structure of the link homology. Suppose that a closure of a braid $\brr$ produces an $\mlk$-component link $\yL$. Assign the variables $\bflc$, $\xnum{\bflc}=\mlk$ to the link components. Then one can define the $\IQ[\bflc]$-module structure on $\zHtg(\yL)$.
Since $\Hhom(\dmmy)$ and $\Hm(\dmmy)$ are functors,
there is a map
\[
\End_{\ctK(\ctQxy)}\bigl(\zctmvbxy{\brr}\bigr) \longrightarrow
%\End_{\IQ-\ctgmod}\Bigl(\Hm\bigl(\Hhom(\zctmvbxy{\brr})\bigr)\Bigr) =
\End_{\IQ-\ctgmod}\bigl( \zHtg(\yL)\bigr),
\]
where $\yL = \fcl(\brr)$. This map turns $\encv{x}_i$ and $\encv{y}_i$ into linear operators acting on $\zHtg(\yL)$. For each link component $\yL_k$, $1\leq k\leq \mlk$ we choose a strand in the braid $\brr$ which is a part of this component and define the action of $\lc_k$ either as $\encv{x}_i$ or as $\encv{y}_{\brrws(i)}$, the variables $x_i$ and $y_{\brrws(i)}$ corresponding to the endpoints of that strand. Equation~\eqref{eq:endsld} guarantees that the action of $\bflc$ does not depend on the choice of  strands which represent link components.

\subsubsection{New construction}
The following commutative diagram outlines our modification of the previous \tgh\ construction:
%%22
%\begin{equation}
%\xymatrix@C=3.75pc@R=4.5pc{
%\brwgrn \ar@{->>}[r]^-{\fbr}
%\ar[d]_-{\ctmdvbxy{\dmmy}}
%&
%\brgrn
%\ar@{->>}[r]^-{\fcl}
%\ar[d]_-{\ctmvbxy{\dmmy}}
%&
%\brhgrn
%\ar[r]
%\ar[d]_-{\xctdmv{\dmmy}}
%&
%\lnks
%\ar[d]_-{\lctmv{\dmmy}}
%\ar
%%@/^1pc/
%[dr]^-{\Htg(\dmmy)}
%\\
%\ctCh\bigl(\ctWxy\bigr)
%\ar[r]^-{\Fffq}
%&
%\ctDr\bigl(\ctWxy\bigr)
%\ar[r]^-{\tHhomtxy(\dmmy)}
%&
%\ctWx
%\ar[r]^-{\Ffukxl}
%&
%\ctWl
%\ar[r]^-{\Hml(\dmmy)}
%&
%\uWplcgmd
%}
%\end{equation}
%The top line in the diagram~\eqref{eq:mndiag} is almost the same as in the diagram~\eqref{eq:oldd}.
%$\brhgrn$ is the set of closed $\nst$-strand braids and the map $\fcl$ performs a quotient by conjugation.
%
%
%
%, except that in the middle we had to restrict the braid group $\brgrn$ to the subset of braids $\brgrns\subset \brgrn$, which correspond to a permutation $s\in\mathrm{S}_{\nst}$. This restriction is necessary, because the permutation $s$ determines the replacement of variables $\bfx$ and braid strand variables $\bfvbr$ by link component variables $\bflc$ performed by the functors $\Ffukxl$ and $\FfukXl$. $\lnksm$ is the set of $\mlk$-component \emph{framed} links.
%
%\begin{SpecialPar}
\begin{equation}
\label{eq:mndiag}
\xymatrix@C=3.5pc@R=4pc{
\brwgrn \ar @{->>}[r]^-{\fbr}
\ar@/_3pc/[dd]_(0.3){\actmdvbxy{\dmmy}}
\ar[d]^-{\ctmdvbxy{\dmmy}}
&
\brgrn
\ar@/^3pc/[dd]^(0.25){\actmvbxy{\dmmy}}
\ar[d]_-{\ctmvbxy{\dmmy}}
&
\brgrns
\ar@{_{(}->}[l]
\ar[r]^-{\fcl}
\ar[dr]^(0.45){\lctdmv{\dmmy}}
&
\lnksm
%\ar[d]_-{\ctmv{\dmmy}}
\ar[d]^-{\lctmv{\dmmy}}
\ar
@/^1pc/
[dr]^-{\Htg\bdmm}
\\
\ctCh\bigl(\ctWxy\bigr)
\ar[d]^-{\FfXx}%{\dmmy\otux\DlXx}
\ar[r]^-{\Fffq}
&
\ctDr\bigl(\ctWxy\bigr)
\ar[d]_-{\FfXx}%{\dmmy\otux\DlXx}
\ar[r]^(0.675){\tHhomtxy\bdmm}|(0.4)\hole
&
\ctWx
\ar[d]_-{\FfxaX}
\ar[r]^-{\Ffukxl}
&
\ctWl
\ar[r]^-{\Hml\bdmm}
&
\uWplcgmdh
\\
\ctCh\bigl(\ctXWxy\bigr)
\ar[r]^-{\Fffq}
&
\ctDr\bigl(\ctXWxy\bigr)
%\ar@/_5.5pc/[rru]_(0.7){\FHhomsxy(\dmmy)}
\ar[r]^-{\Hhomxy\bdmm}
&
\ctWX
\ar[ur]_-{\FfukXl}
}
\end{equation}
The relative complexity of this diagram is due to the appearance of extra variables $\bfvbr$, $\xnum{\bfvbr}=\nst$, associated with braid strands and variables $\bflc$, $\xnum{\bflc}=\mlk$, associated with link components.
%\end{SpecialPar}

All categories in the diagram~\eqref{eq:mndiag} have $q$-, $a$- and $t$-gradings. The grading of nested categories is described in subsection~\ref{ss:grd}. \tgrdt\ is also the homological grading of derived categories and $\ctWx$, $\ctWl$ and $\ctWX$, while their modules have \tgrdq\ coming from~\eqref{eq:dgrs} as well as \tgrda.

% three gradings: . The nested categories have \tgrdq\

The top line in the diagram~\eqref{eq:mndiag} is almost the same as in the diagram~\eqref{eq:oldd}, except that in the middle we had to restrict the braid group $\brgrn$ to the subset of braids $\brgrns\subset \brgrn$ which correspond to a permutation $s\in\mathrm{S}_{\nst}$. This restriction is convenient, because the permutation $s$ determines the replacement of variables $\bfx$ and braid strand variables $\bfvbr$ by link component variables $\bflc$ performed by the functors $\Ffukxl$ and $\FfukXl$. $\lnksm$ is the set of $\mlk$-component \emph{framed} oriented links.

The bracket $\ctmdvbxy{\dmmy}$ is essentially the same as $\zctmdvbxy{\dmmy}$ in the previous diagram~\eqref{eq:oldd}: it turns a \brwd\  into a complex of \Sglb s (with the differential acting in the second homological direction) representing an object of the category $\ctCh\bigl(\ctWxy\bigr)$, while intertwining the products of \brwd s and braids with the tensor products~\eqref{eq:tpKm} and~\eqref{eq:tpDr}
\begin{equation}
\label{eq:smbrpr}
\ctmdvbxz{\brw_1\brw_2} \cong \ctmdvbxy{\brw_1}\oty\ctmdvbyz{\brw_2}
\end{equation}
(according to Remark~\ref{rm:tpsbm}, the ordinary tensor product in the \rhs is \qisc\ to the derived one with respect to $\ctWxy$).
However, this time \Sglb s are endowed with the $\ctWxy$-module structure coming from the \stact\ of $\aWp$ on $\bfx$ and $\bfy$.
As a result of this additional structure, the map $\ctmdvbxy{\dmmy}$ preserves the braid relation only up to \qiso, hence the definition of the map $\ctmvbxy{\dmmy}$
requires taking a quotient over \qiso s performed by the functor $\Fffq$.

The functor $\tHhomtxy$ in the diagram is a modified \Hchsh.
The \Hchsh\ specialized to our case
%In its standard form
%adopted to the original construction of link homology~\cite{Khtgd} the \Hchsh\
is the functor
%
%In its original form, the \Hchsh\ $\tHhomtxy(\dmmy)$ in the diagram is the functor
\begin{equation}
\label{eq:hho}
\xymatrix@C=5pc{
\ctWxy
\ar[r]^-{\Hhomtxy\bdmm}
&
\uWpx\sctgmodt,
}
\qquad
\Hhomtxy(\dmmy) =\shnst\, \Hml\bigl(\dmmy\otdrv{\tnsp{\bfx},\bfy}\Dlxy\bigr),
\end{equation}
where the bar over $\bfx$ indicates that we remember the action of the whole $\uWpx$, including $\Qbx$, on the \Hchsh.
%
%Note that we preserve the $\Qbx$-module structure.
This functor extends to the functor from nested relative derived category to an ordinary relative derived category:
\[
%\Hhomtxy\colon\ctDr(\ctWxy)\longrightarrow\ctD\bbrs{(\uWpx,\Qbx)\sctgmodt}.
\xymatrix@C=5pc{
\ctDr(\ctWxy)
\ar[r]^-{\Hhomtxy}
&
\ctD\bbrs{(\uWpx,\Qbx)\sctgmodt}.
}
\]
%$\ctDr(\ctWxy)$ and $\ctD\bbrs{(\uWpx,\Qbx)\sctgmodt}$.
The functor $\tHhomtxy$ is the composition of the latter functor with the standard forgetful functor from the relative to the absolute derived category:
\begin{equation}
\label{eq:hhhom}
\xymatrix@C=5pc{
\ctDr(\ctWxy)
\ar[r]_-{\Hhomtxy\bdmm}
\ar@/^2pc/[rr]^-{\tHhomtxy\bdmm}
&
\ctD\bbrs{(\uWpx,\Qbx)\sctgmodt}
\ar[r]
&
\ctWx.
}
\end{equation}
%where the second arrow represents an obvious functor from the relative to the absolute derived category. The \Hchsh\ $\tHhomtxy(\dmmy)$ in the diagram is the composition of these two functors.

Keeping the $\Qbx$-module structure within the \Hchsh\ might violate its trace-like property important for ensuring the invariance under the first Markov move. However, we will show that this problem disappears for \Sglb s after the application of the renaming functor $\Ffusxl$. The permutation $s\in S_n$ corresponding to a braid determines the split of initial strand positions $\{1,\ldots,\nst\}$ into cycle subsets:
\[
\{1,\ldots,\nst\} = \bigcup_{i=1}^{\mlk} \xccl_i,
\]
each subset corresponding to a component of the link constructed by the cyclic closure of the braid. We choose an element $k_i\in\xccl_i$ in each cycle subset thus forming a \tsq\ $\bfk$, $\xnum{\bfk}=\mlk$. The functor $\Ffukxl$ renames variables $x_{k_i}$ into $\lc_i$, while forgetting the action of all other variables $\{\bfx\}\setminus\{x_{k_1},\ldots,x_{k_{\mlk}}\}$. We will show that the application of this functor to the \Hchsh~\eqref{eq:hhhom} of the bracket of a braid
\begin{equation}
\label{eq:longcomp}
\Ffukxl\Bigl( \tHhomtxy \bigl( \ctmvbxy{\dmmy}\bigr)\Bigr)
\end{equation}
does not depend on the choice of the representatives $k_i$ within each cycle.

\subsubsection{Braid strand variables}

We introduce auxiliary variables $\bfvbr$, $\xnum{\bfvbr}=\nst$, into the categorification construction. %for two reasons.
%The first reason is that
They help to prove that the composition of maps~\eqref{eq:longcomp} does not depend on the choice of cycle representatives in the definition of $\Ffusxl$.
In addition, we will use them to capitalize on
%The second reason is that we will use they  strengthen
the property~\eqref{eq:endsld} of the original categorification complex $\zctmvbxy{\brr}$.
%
%a well-known property of the old categorification complex $\zctmvbxy{\brw}$. Namely, let $\encv{x}_i,\encv{y}_i\in\End_{\ctKom(\ctQxy)}\bigl(\zctmvbxy{\brw}\bigr)$ represent the multiplication by $x_i$ and $y_i$. Then $\encv{x}_i\sim\encv{y}_{\brws(i)}$, where $\brws\in S_{\nst}$ is the symmetric group element corresponding to the braid $\brw$.
%
%
%, because they help to prove that the composition of maps~\eqref{eq:longcomp}
%%\[\Ffusxl\Bigl( \tHhomtxy \bigl( \ctmvbxy{\dmmy}\bigr)\Bigr)\]
%does not depend on the choice of cycle representatives in the definition of $\Ffusxl$, and the proof seems more conceptual.

The \brstv s $\bfvbr$ are introduced through the functors
\begin{equation}
\label{eq:dblvr}
\xymatrix@C=1.5cm{
\uWpxy\sctgmod
\ar[r]^-{\FfXx}
&
\uWpXxy\sctgmod,}
%\FfXx\colon \uWpxy\shdsh\ctgmod\longrightarrow \uWpXxy\shdsh\ctgmod,
\end{equation}
which turn a $\uWpxy$-module $\mMxy$ into a $\uWpXxy$-module $\mXMxy=\mMxy$ by making generators $\bfvbr$ act as $\bfx$. This functor $\FfXx (\dmmy) = \dmmy\otux\DlXx$ essentially doubles variables $\bfx$ to variables $\bfx,\bfvbr$. We will also use a similarly defined  functor $\FfXy$ which doubles variables $\bfy$.
Functors $\FfXx$ appearing in the diagram~\eqref{eq:mndiag} are the extensions of the original functor~\eqref{eq:dblvr} to the nested categories.

%The \brstv s $\bfvbr$ are introduced though the functors $\dmmy\otux\DlXx$ which essentially `double' the variables $\bfx$: the modules $\mMxy$ and $\mXMxy=\mMxy\otux\DlXx$ are the same as a $\uWpxy$-modules, and the action of $\bfvbr$ on $\mXMxy$ is the same as that of $\bfx$.

The new feature of $\bfvbr$ is that these variables can `slide' along braid strands. Let $\brrws$ be a permutation associated with a braid $\brr$ and let $\brrws(\bfvbr)$ denote the variables $\bfvbr$ permuted by $\brrws$. In subsection~\ref{ss:strvs} we prove the following theorem:
\begin{theorem}
\label{th:totsld}
The result of doubling the action of $\bfx$ on $\ctmvbxy{\brr}$ with $\bfvbr$ is \qisc\ (in the outer category $\ctCh$) to the doubling of $\bfy$ with $\brrws(\bfvbr)$:
\begin{equation}
\label{eq:slprop}
%\ctmdvbxy{\brw}\otux\DlXx \eqsmout \ctmdvbxy{\brw}\otuy\Dlv{\brws(\bfvbr); \bfy }.
%\FfXx\bigl(\ctmdvbxy{\brw} \bigr) \eqsmout \FfXv{\brws^{-1}(\bfy)}\bigl(\ctmdvbxy{\brw} \bigr).
\FfXx\bigl(\ctmvbxy{\brr} \bigr) \eqsmout \Ffvv{\brrws(\bfvbr)}{\bfy}\bigl(\ctmvbxy{\brr} \bigr).
\end{equation}
\end{theorem}
%
%Let $\brws$ be a permutation associated with a \brwd\ $\brw$, and let
%$\brws(\bfvbr)$ denote the variables $\bfvbr$ permuted by $\brws$. We will prove that the result of doubling the action of $\bfx$ on $\ctmdvbxy{\brw}$ with $\bfvbr$ is \qisc\ (with respect to the outer category $\ctKom$) to the doubling of $\bfy$ with $\brws(\bfvbr)$:
%\begin{equation}
%\label{eq:slprop_3}
%%\ctmdvbxy{\brw}\otux\DlXx \eqsmout \ctmdvbxy{\brw}\otuy\Dlv{\brws(\bfvbr); \bfy }.
%%\FfXx\bigl(\ctmdvbxy{\brw} \bigr) \eqsmout \FfXv{\brws^{-1}(\bfy)}\bigl(\ctmdvbxy{\brw} \bigr).
%\FfXx\bigl(\ctmdvbxy{\brw} \bigr) \eqsmout \Ffvv{\brws(\bfvbr)}{\bfy}\bigl(\ctmdvbxy{\brw} \bigr).
%\end{equation}

Monoidal structure within the categories $\ctCh\bigl(\ctXWxy\bigr)$ and $\ctDr\bigl(\ctXWxy\bigr)$ is defined by the left tensor product $\otdrly$ of the diagram~\eqref{eq:lwfntso}. Since this tensor product does not impact the action of $\bfvbr$, the doubling functor $\FfXx$ intertwines the tensor products $\otdry$ and $\otdrly$: for any two modules $\mMxy$ and $\mNxy$,
\[
\FfXx(\mMxy)\otdrly\Ffbey(\mNyz)\cong
\FfXx(\mMxy\otdry\mNyz).
\]
As a consequence, the composite brackets
\[
\actmdvbxy{\dmmy} = \FfXx\bigl(\ctmdvbxy{\dmmy}\bigr),
%\otux \DlXx,
\qquad
\actmvbxy{\dmmy} = \FfXx\bigl(\ctmvbxy{\dmmy}\bigr)
% \otux \DlXx.
\]
intertwine the products within $\brwgrn$ and $\brgrn$ with $\otdrly$.

The \Hchsh\ in the presence of \brstv s is a functor
\begin{equation}
\label{eq:hht}
\xymatrix@C=1.5cm{
 \ctXWxy
\ar[r]^-{\Hhomxy}
&
\uWpX\sctgmod,
}
%\Hhomxy\colon \ctXWxy\longrightarrow \uWpX\sctgmod,
\qquad
\Hhomxy(\dmmy) = \shnst\,\Hml\bigl(\dmmy\otdrxy\Dlxy\bigr),
\end{equation}
extended to derived categories:
%$\Hhomxy\colon \ctDr\bigl(\ctXWxy\bigr)\rightarrow \ctWX.$
\begin{equation}
\label{eq:hhmdg}
\xymatrix@C=7pc{
\ctDr\bbrs{\ctXWxy}
\ar[r]^-{\dmmy\otdrxy\Dlxy}
\ar@/_2pc/[rr]_-{\Hhomxy}
&
\ctDr\bbrs{\ctWX}
\ar[r]^-{\shnst\,\HmWX}
&
\ctWX,
}
\end{equation}
where
%$\Hmlin$
$\HmWX$
is the homology taken with respect to the differential of the inner category $\ctWX$ in the nested category
$\ctDr\bbrs{\ctWX}$.

Note that the diagonal module $\Dlxy$ has no $\bfvbr$ action. The functor $\FfxaX$ renames the variables $\bfx$ into $\bfvbr$, and the commutativity of the subdiagram
\[
\xymatrix@R=1.5cm@C=1.75cm{
\ctDr\bigl(\ctWxy\bigr)
\ar[d]_-{\FfXx}%{\dmmy\otux\DlXx}
%\ar[r]^(0.675){\tHhomtxy(\dmmy)}|(0.45)\hole
\ar[r]^-{\tHhomtxy\bdmm}
&
\ctWx
\ar[d]_-{\FfxaX}
\ar[r]^-{\Ffukxl}
&
\ctWl
%\ar[r]_-{\Hml(\dmmy)}
\\
\ctDr\bigl(\ctXWxy\bigr)
%\ar@/_5.5pc/[rru]_(0.7){\FHhomsxy(\dmmy)}
\ar[r]^-{\Hhomxy\bdmm}
&
\ctWX
\ar[ur]_-{\FfukXl}
}
\]
of the diagram~\eqref{eq:mndiag} is obvious.

The trace-like property of this \Hchsh:
\[
\Hhomxz\Bigl(\actmvbxy{\brr_1}\otdrly\actmvbyz{\brr_2}\Bigr) \simeq
\Hhomxz\Bigl(
%\actmdvbxy{\brw_2}
\vctmvr{\brrws_1(\bfvbr)}{\brr_2}{\bfx,\bfy}
\otdrly\actmvbyz{\brr_1}\Bigr)
\]
is established with the help of the sliding property~\eqref{eq:slprop}. The sliding property
also leads to the following theorem proved in subsection~\ref{ss:prthrnind}:
\begin{theorem}
\label{th:rnind}
The action of the renaming functor $\FfukXl$ on the Hochschild homology of the braid bracket does not depend on the choice of braid strands $k_i$ representing link components: for two \tsq s of strands $\bfk$ and $\bfk'$ there is a \qiso
\[
\FfukXl\bigl(\Hhomxy(\actmvbxy{\brr}) \bigr)\eqsmout
\FfpukXl\bigl(\Hhomxy(\actmvbxy{\brr}) \bigr),
\]
\end{theorem}
\begin{corollary}
The object $\FfukXl\bigl(\Hhomxy(\actmvbxy{\brr}) \bigr)$ of $\ctWX$ is determined up to isomorphism only by the braid $\brr$, so there exists a map $\lctdmv{\dmmy}$ which makes the following diagram commutative:
\begin{equation*}
%\label{eq:mndiag}
\xymatrix@C=4pc@R=4pc{
\brgrn
\ar[d]_-{\actmvbxy{\dmmy}}
%\ar[drr]^-{\lctdmv{\dmmy}}
&
\brgrns
\ar@{_{(}->}[l]
%\ar[r]^-{\fcl}
\ar[dr]^-{\lctdmv{\dmmy}}
\\
\ctDr\bigl(\ctXWxy\bigr)
\ar[r]_-{\Hhomxy\bdmm}
&
\ctWX
\ar[r]_-{\FfukXl}
&
\ctWl
%\ar[r]_-{\Hml(\dmmy)}
%&
%\uWplcgmd
}
\end{equation*}
\end{corollary}

Finally, we show in subsection~\ref{ss:mmi} that the bracket $\lctdmv{\dmmy}$
%object $\FfukXl\bigl(\Hhomxy(\actmvbxy{\brr}) \bigr)$
is invariant under Markov moves up to a connection shift $\shv{\bsdfpli}$,  where $\bsdfpx$ is defined by~\eqref{eq:dfppr}, caused by the change in framing of the $i$-th link component, so this bracket is an invariant of the framed link $\xL$ constructed by closing the braid $\brr$.

\subsection{Results} Our main results are Theorems~\ref{th:first} and~\ref{th:scnd} which state the existence and uniqueness of the maps
\begin{equation*}
\xymatrix@C=2cm{
\brgrn \ar[r]^-{\actmvbxy{\dmmy}}
&
\ctDr\bigl(\ctXWxy\bigr)
},
\qquad
\xymatrix@C=1.5cm{
\lnksm \ar[r]^-{\ctmv{\dmmy}}
&
\ctWl
}
\end{equation*}
in the diagram~\eqref{eq:mndiag} after we define properly the map
%$\actmdvbxy{\dmmy}$ and $\Hhomxy(\dmmy)$
\begin{equation}
\label{eq:bstnf}
\xymatrix@C=2cm{
\brwgrn \ar[r]^-{\actmdvbxy{\dmmy}}
&
\ctCh\bigl(\ctXWxy\bigr)
}.
%,
%\qquad
%\xymatrix@C=2cm{
%\ctDr\bigl(\ctXWxy\bigr) \ar[r]^-{\Hhomxy{\dmmy}}
%&
%\ctWX
%}.
\end{equation}
Moreover, we show that if a framed link $\xLp$ is obtained from a link $\xL$ by increasing the framing of its $i$-th component by one, then
$
\ctmv{\xLp} \simeq\shvv{\ctmv{\xL}}{-\bsdfpli}.
$

\section{Preliminaries}

\subsection{Simplifying objects in a \rdc}
Consider a general definition of a \rdc\ $\ctDr(\ctA)$ described in subsection~\ref{sbs:ndc}:  $\ctA$ and $\ctB$ are additive categories related by an additive functor $\Ffr\colon\ctA\rightarrow\ctB$.
%%%%%%%%%%%%%%%%%%%%%%%%%%%%%
%\begin{SpecialPar}
%\begin{theorem}
%\label{pr:qsbc1}
%Suppose that for three complexes $\aAb$, $\aBb$ and $\aCb$ of $\ctCh(\ctA)$ there exists a sequence of chain maps
%\[
%\xymatrix{
%\aBb\ar[r]^-{f}
%&
%\aAb\ar[r]^-{g}
%&
%\aCb
%}
%\]
%such that $g\circ f = 0$ in $\ctCh(\ctA)$, and after the application of $\Ffr$ these maps become a part of an exact triangle in $\ctK(\ctB)$:
%\begin{equation}
%\label{eq:genspl}
%\xymatrix{
%\Ffr(\aBb)\ar[r]^-{\Ffr(f)}
%&
%\Ffr(\aAb)\ar[r]^-{\Ffr(g)}
%&
%\Ffr(\aCb)
%\ar[r]
%&
%\Ffr(\aBb)[1]
%}
%\end{equation}
%Then
%\begin{equation}
%\label{eq:smpqiso}
%\aAb \eqsmout
%\begin{cases}
%\aCb,&\text{if $\aBb$ is contractible,}
%\\
%\aBb,&\text{if $\aCb$ is contractible,}
%\end{cases}
%\end{equation}
%where $\eqsmout$ denotes an isomorphism in $\ctDr(\ctA)$.
%\end{theorem}
%
%\begin{proof}
%Condition $g\circ f = 0$ implies that $f$ and $g$ define obvious chain maps
%$\tilde{f}\colon\aBb[1] \rightarrow\Cone(g)$ and $\tilde{g}\colon\Cone(f)\rightarrow\aCb$. Exact triangle~\eqref{eq:genspl} implies that $\Ffr(\tilde{f})$ and $\Ffr(\tilde{g})$ are homotopy equivalences, hence by Definition~\ref{df:relder}  maps $\tilde{f}$ and $\tilde{g}$ are \qiso s. Now \qiso s~\eqref{eq:smpqiso} follow from the fact that if $\aBb$ is contractible then $\Cone(f)\smout \aAb$, and if $\aCb$ is contractible then $\Cone(g)\smout \aAb[1]$.
%\end{proof}
%\end{SpecialPar}
%%%%%%%%%%%%%%%%%%%%%%%%%%%%%%
\begin{theorem}
\label{pr:qsbc}
Suppose that for three complexes $\aAb$, $\aBb$ and $\aCb$ of $\ctCh(\ctA)$ there exists a sequence of chain maps
\[
\xymatrix{
\aBb\ar[r]^-{f}
&
\aAb\ar[r]^-{g}
&
\aCb
}
\]
%such that $g\circ f = 0$ in $\ctCh(\ctA)$, and after the application of $\Ffr$ this sequence splits:
such that after the application of $\Ffr$ it splits:
\[
\xymatrix@C=3pc{
\Ffr(\aBb)\ar[r]^-{\Ffr(f)}
\ar[dr]_-{(\xId,0)}
&
\Ffr(\aAb)\ar[r]^-{\Ffr(g)}
\ar[d]^-{\cong}
&
\Ffr(\aCb)
%\ar[r]
\\
&
\Ffr(\aBb)\oplus\Ffr(\aCb)
\ar[ur]_-{(0,\xId)}
}
\]
Then
\begin{equation}
\label{eq:smpqiso_1}
\aAb \eqsmout
\begin{cases}
\aCb,&\text{if $\Ffr(\aBb)$ is contractible,}
\\
\aBb,&\text{if $\Ffr(\aCb)$ is contractible,}
\end{cases}
\end{equation}
where $\eqsmout$ denotes
a \qiso\ in $\ctCh(\ctA)$.
%an isomorphism in $\ctDr(\ctA)$.
\end{theorem}
\begin{proof}
If $\Ffr(\aCb)$ is contractible, then $\Ffr(f)$ is a homotopy equivalence in $\ctCh(\ctB)$, hence $f$ is a \qiso\ by definition.

If $\Ffr(\aBb)$ is contractible, then $\Ffr(g)$ is a homotopy equivalence in $\ctCh(\ctB)$, hence $g$ is a \qiso\ by definition.
\end{proof}

\subsection{Derived partial tensor product}
%\subsection{Partial tensor product}
\subsubsection{Equivariant resolutions}
%The partial tensor product $\oty$ can be extended to the derived categories as a left derived functor
%\begin{equation}
%\label{eq:otnpr_1}
%\xymatrix{\ctWxy\times\ctWyz
%%\xrightarrow{\;\;\;\otdry\;\;\;}
%\ar[r]^-{\otdry}
%&
%\ctWxz,
%}
%\end{equation}
%which
The derived partial tensor product~\eqref{eq:drtnpr}
can be computed with the help of a projective resolution. It is sufficient to resolve one factor: for example,
\[
\mMxy\otdry\mNyz \simeq \rsPb(\mMxy)\oty\mNyz,
\]
the resolution complex $\rsPb(\mMxy)$ consisting of projective $\uWpxy$-modules.

\begin{definition}
Consider a \sdr\ product $\alA\rtimes\Ulieg$, where $\lieg\subset\Der(\alA)$ is a Lie algebra acting on an algebra $\alA$ by derivations.
A $\lieg$-equivariant $\alA$-projective resolution of an $\alA\rtimes\Ulieg$-module $\mM$ is a complex in $\ctCh(\alA\rtimes\Ulieg-\ctmod)$ which is a projective resolution of $\mM$ in $\ctCh(\alA-\ctmod)$.
\end{definition}
\begin{theorem}
\label{th:drortp}
If $\rsPWb(\mMxy)$ is a \weqxyp\ resolution of $\mMxy$, then
\[
\mMxy\otdry\mNyz \simeq \rsPWb(\mMxy)\oty\mNyz.
\]
\end{theorem}
In other words, a $\uWp$-equivariant resolution is sufficient for the computation of the derived partial tensor product $\otdry$.
\begin{proof}
If $\aCb$ is the \ChEr\ of the trivial $\uWp$-module, then the tensor product $\aCb\otimes\rsPWb(\mMxy)$ (with $\uWp$ acting on this tensor product by the Leibnitz rule) is a $\uWpxy$-projective resolution of $\mMxy$, hence
%\begin{multline*}
%\mMxy\otdry\mNyz \simeq \bigl(\aCb\otimes\rsPWb(\mMxy)\bigl)\oty\mNyz
%\\
%\cong\aCb\otimes\bigl(\rsPWb(\mMxy)\oty\mNyz\bigr) \simeq \rsPWb(\mMxy)\oty\mNyz.
%\end{multline*}
\[
\begin{split}
\mMxy\otdry\mNyz & \simeq \bigl(\aCb\otimes\rsPWb(\mMxy)\bigl)\oty\mNyz
\cong\aCb\otimes\bigl(\rsPWb(\mMxy)\oty\mNyz\bigr)
\\
&\simeq \rsPWb(\mMxy)\oty\mNyz.
\end{split}
\]
\end{proof}

\subsubsection{Partial tensor products of \smpr\ modules}
A $\uWpxy$-module $\mMxy$ is called \smpr\ if it is free as a $\Qbx$-module and as a $\Qby$-module.
All elementary \Sglb s introduced in subsection~\ref{ss:elsbm} are \smpr, and their partial derived tensor products can be replaced by the ordinary ones in view of the following corollary of Theorem~\ref{th:drortp}:
\begin{corollary}
\label{cr:smprtp}
If modules $\mMxy$ and $\mNxy$ are \smpr\ then
\[
\mMxy\otdry\mNyz \simeq \mMxy\oty\mNyz,
\]
and this tensor product is also \smpr\ as a $\uWpxz$-module.
\end{corollary}

By definition, a free finitely generated \tgrddq\ $\Qbx$-module $M$ has a presentation
\[
M = \bigoplus_{i\in\ZZ} m_i \dgshq^i\Qbx,
\]
where $\dgshq$ is the \tdgq\ shift functor and $m_i\in\ZZ$ are the multiplicities. We define the \trkq\ of $M$ as a Laurent polynomial
\begin{equation}
\label{eq:dfqrk}
\xrkq M = \sum_{i\in\ZZ} m_i q^i.
\end{equation}
Let $\xrkqx\mMxy$ denote the $\Qbx$ \trkq\ of a \smpr\ module $\mMxy$. The following is obvious:
\begin{proposition}
\label{pr:rkpr}
If modules $\mMxy$ and $\mNyz$ are \smpr\ then
\[
\xrkqx (\mMxy\oty\mNyz) = (\xrkqx\mMxy)(\xrkqy\mNyz).
\]
\end{proposition}

\subsubsection{Partial tensor product in nested derived categories}
\begin{theorem}
\label{th:drtnpr}
 There exists a unique (up to isomorphism) functor~\eqref{eq:tpDr} which makes the following diagram commutative:
\[
\xymatrix@C=1.25cm@R=1.25cm{
\ctCh(\ctWxy)\times \ctCh(\ctWyz)
\ar[r]^-{\otdry}
\ar[d]^-{\Fffq}
&
\ctCh(\ctWxz)
\ar[d]^-{\Fffq}
\\
\ctDr(\ctWxy)\times \ctDr(\ctWyz)
\ar[r]^-{\otdry}
&
\ctDr(\ctWxz)
}
\]
\end{theorem}

This theorem is a particular case of the following lemma:
\begin{lemma}
Suppose that we have two pairs of additive categories related by additive functors:
$\ctA_1 \xrightarrow{\Ff_1} \ctB_1$ and $\ctA_2\xrightarrow{\Ff_2} \ctB_2$. If for a functor $\ctA_1 \xrightarrow{\FfA}\ctA_2$ there exists a functor $\ctB_1\xrightarrow{\FfB}\ctB_2$ such that the following square is commutative:
\begin{equation}
\label{eq:lcmsq}
\mvcn{
\xymatrix{
\ctA_1
\ar[r]^-{\FfA}
\ar[d]_-{\Ff_1}
&
\ctA_2
\ar[d]^-{\Ff_2}
\\
\ctB_1
\ar[r]^-{\FfB}
&
\ctB_2
}
}
\end{equation}
(functors $\Ff_2\circ\FfA$ and $\FfB\circ\Ff_1$ are isomorphic),
then the functor $\FfA$ extends to the derived categories:
\[
\xymatrix{
\ctD_{\Ff_1}(\ctA_1) \ar[r]^-{\FfA}
&
\ctD_{\Ff_2}(\ctA_2).
}
\]
\end{lemma}
\begin{proof}
It is sufficient to show that if there is a \qisoo\ $\aAb\xrightarrow{f}\aAb'$ between the objects of $\ctCh(\ctA_1)$, then its image $\FfA(\aAb)\xrightarrow{\FfA(f)}\FfA(\aAb')$ is a \qisot.
%, that is, the $\Ff_2$ image of its cone is contractible.
Indeed, by Definition~\ref{df:relder} $\Ff_1(f)$ is an isomorphism in $\ctK(\ctB_1)$, hence $\FfB\circ\Ff_1(f)$ is an isomorphism in $\ctK(\ctB_2)$. Commutativity of the square~\eqref{eq:lcmsq} means that $\Ff_2\circ\FfA(f) = \FfB\circ\Ff_1(f)$, hence
$\Ff_2\circ\FfA(f)$ is an isomorphism in $\ctK(\ctB_2)$ and by Definition~\ref{df:relder} $\FfA(f)$ is a \qisoo\ in
$\ctD_{\Ff_2}(\ctA_2)$.
%
%\[
%\Ff_2\Bigl( \Cone\bigl(\FfA(f)\bigr)\Bigr) \cong
%\Cone\bigl(\Ff_2\circ\FfA(f) \bigr) \cong \Cone\bigl(\FfB\circ\Ff_1(f)\bigr) \cong
%\FfB\Bigl( \Cone \bigl(\Ff_1(f)\bigr) \Bigr) \sim 0.
%\]
\end{proof}
\begin{proof}[Proof of Theorem~\ref{th:drtnpr}]
The theorem follows from the previous lemma when we observe that the original derived tensor product~\eqref{eq:dtpr} plays the role of the functor $\FfB$.
\end{proof}

%
%
%*******************
%
%\subsubsection{$\bfvbr$-associated modules and their partial tensor products}
%
%The $\uWpXxy$-modules that we encounter here are of a special form. A $\uWpXxy$-module $\mXMxy$ is called left (resp. right) \vbrass\ and denoted as $\mlXMxy$ (resp. $\mrXMxy$) if it has the form
%\begin{equation}
%\label{eq:assms}
%\mlXMxy \cong
%\mMxy\otimes_{\Qbx} \Qv{\bfvbr,\bfx}/(\bfvbr-\bfx),\qquad
%\mrXMxy\cong
%\mMxy\otimes_{\Qby} \Qv{\bfvbr,\bfy}/(\bfvbr-\bfy),
%\end{equation}
%where the $\uWpxy$-module $\mMxy$ is the module $\mlXMxy$ in which the $\Qv{\bfvbr}$-module structure is forgotten.
%%where $\mMxy$ is a $\uWpxy$-module.
%In other words, a \vbrass\ module is essentially a $\uWpxy$-module and the generators $\bfvbr$ repeat either the $\bfx$ or $\bfy$ action.
%Note that the ordinary tensor products~\eqref{eq:assms} are the same as derived ones, because the right factors are free as $\Qbx$ and $\Qby$ modules.
%
%The tensor product~\eqref{eq:lwfntso} between two \vbrass\ modules has two obvious properties:
%\[
%\mlXMxy \otdrly \mYNyz \simeq \lfsinlX\bigl(\mMxy \otdry \mNyz\bigr),\qquad
%%\]
%%where $\mNyz$ is the module $\mYNyz$, in which the $\Qv{\bfbe}$-module structure is forgotten, and
%%\[
%\mrXMxy \otdrly \mlYNyz \simeq \mrXMxy \otdrly \mlYNyz
%\]
%
%
%
%
%*****************************
\subsection{$\uWpx$-modules and $\aWp$-connections}
\subsubsection{Notations for connections}
If two $\uWpx$-modules $M$ and $N$ have finite numbers of generators as $\Qbx$-modules, then we describe the homomorphisms between them by matrices relative to these generators: $M\xrightarrow {A} N$, where $A$ is a matrix with polynomial entries

If a $\uWpx$-module $M$ has a finite number of generators $\bgn=\xgn_1,\ldots,\xgn_k\in M$ as a $\Qbx$-module, then the action of $\aWp$ on $M$ can be described by a \tsq\ of $k\times k$ matrices with polynomial entries $\xcnbA=\xcnA_0,\xcnA_1,\ldots$, $\xcnA_m =\mtre{a_{m;ij}}$ giving the action of derivations
$\xcdfm$, representing the algebra generators $\xLm$, on module generators $\bgn$:
\[
\xcdfm \xgn_i = \sum_{j=1}^k a_{m;ij} \xgn_j.
\]
We refer to the \tsq\ $\xcnbA$ as \emph{connection} and use notation $\mcnv{M}{\xcnbA}$ for this $\uWpx$-module, reserving a simple notation $M$ for the case when all matrices $\xcnbA$ are zero. If $M$ has a single generator $\xgn$, that is, $M\cong \Qbx/\aI$, where $\aI$ is a $\aWp$-invariant $\Qbx$-ideal, then the connection matrices $\xcnbA$ are reduced to numbers $\xcna$ such that $\xLm\xgn=a_m\,\xgn$, and the notation for the module becomes $\mcnv{M}{\xcna}$.

\begin{proposition}
A $\Qbx$-homomorphism between two $\uWpx$-modules with single generators
\begin{equation*}
\label{eq:phmm}
\xymatrix{ f\colon\mcnv{M}{\bfa}
\ar[r]^-{p}
&
\mcnv{M'}{\bfa'}
}
\end{equation*}
is a $\uWpx$-homomorphism iff the polynomial $p\in\Qbx$ satisfies the condition
\begin{equation}
\label{eq:hmcnd}
\xhLm p = (a_m - a'_m)\,p\quad \mod \aI',
\end{equation}
where $\aI'\subset\Qbx$ is the ideal such that $M'\cong\Qbx/\aI'$.
\end{proposition}
\begin{proof}
It is sufficient to verify $\aWp$-equivariance of the action of $f$ on the generator $\xgn$ of $M$:
%\[
%f(\xcdfm \xgn) = pa_m\xgn',\qquad \xcdfm' f(\xgn) = (\xhLm p)\,\xgn' + p a'_m\,\xgn',
%\]
\[
\xcdfm' f(\xgn) - f(\xcdfm \xgn) = (\xhLm p)\,\xgn' + p a'_m\,\xgn' - pa_m\xgn' =
\bbrs{\xhLm p - p (a_m - a'_m)}\xgn',
\]
where $\xgn'$ is the generator of $M'$.
\end{proof}

\subsubsection{Flat \tsq s}
As we explained in the introduction, a \tsq\ of polynomials $\bfa\in\Qbx$ is called \emph{flat} if it satisfies the property~\eqref{eq:shprp}. A flat \tsq\ $\bfa$ determines an automorphism~\eqref{eq:autfp} of $\uWpx$ and, consequently, an \enfn\ $\shv{\bfa}$. Obviously, $\mcnv{M}{\xcnbA}\shv{\bfa} = \mcnv{M}{\xcnbA + \bfa\xId}$, where $\xId$ is the $k\times k$ identity matrix.

We have two examples of flat \tsq s. The first one is described in the introduction: it is
$\bsdfpx\in\Qv{x}$, where $\spdpmx = (m+1)x^m$, and its flatness is verified by a direct calculation.

The second example originates from a `gauge transformation'.
\begin{theorem}
\label{th:lgder}
If for a polynomial $p\in\Qbx$ there exists a \tsq\ of polynomials $\bfa\in\Qbx$ such that
\begin{equation}
\label{eq:admpl}
\xhLm p = a_m\,p
\end{equation}
for all $m\geq 0$, then the \tsq\ $\bfa$ is flat.
\end{theorem}
\begin{proof}
Consider the action of $\aWp$ on the field of fractions of $\Qbx$. Then the modified generators $\xLm' = \xLm + a_m$ are the result of `conjugation of $\xLm$ by $p$:
\[
\xhLm' = \widehat{p^{-1}}\circ \xhLm\circ \hat{p},
\]
where $\hat{p}$ denotes the operator of multiplication by $p$.
Hence the commutation relations between the generators $\xLm'$ are the same as those between $\xLm$.
\end{proof}

The polynomial $y-x\in\Qv{x,y}$ satisfies the condition~\eqref{eq:admpl}:
\[
\xhLm (y-x) = \sdfmv{x}{y}\,(y-x),
\]
where
\begin{equation}
\label{eq:dfpm}
\sdfmv{x}{y} = \frac{y^{m+1} - x^{m+1}}{y-x} = \sum_{i=0}^m x^i y^{m-1},
\end{equation}
hence $\bsdfv{x}{y}$ is a flat \tsq.

Note that
%$\mathfrak{F}_{\bfx}$
flat \tsq s for $\Qbx$ form a $\IQ$-vector space, since a linear combination of flat \tsq s is flat.

%Consider a $\Qbx$-homomorphism between two $\uWpx$-modules with single generators:

\subsubsection{Connections in Koszul complexes}
\label{sss:ksres}

For a \tsq\ of polynomials $\bfp\in\Qbx$, $\xnum{\bfx}=\nst$, $\xnum{\bfp}=k$, consider the quotient $\Qbx$-module $M=\Qbx/(\bfp)$. This module has a single generator $\xgn$, which represents $1\in\Qbx$.

Suppose that the ideal $(\bfp)$ is invariant under the \stact\ of $\aWp$ on $\Qbx$, that is, there exist
$k\times k$ matrices $\xcnbA= \xcnAv{0},\xcnAv{1},\ldots$
%\xcnAm$
with polynomial entries: $\xcnAm=\mtre{\acpmij}$, $\acpmij\in\Qbx$, $m\geq 0$, $1\leq i,j\leq k$, which satisfy the relation
%for all $i$
\begin{equation}
\label{eq:mdfc}
\xhLm \vcp = \xcnAm\, \vcp,\qquad \vcp=\begin{pmatrix}p_1 \\ \vdots \\ p_k\end{pmatrix},\qquad
\vcp\in\Qbx^k.
%\xLm p_i = \sum_{j=1}^k \acpmij\,p_j,
\end{equation}
%where $\vcp=\smmatr{p_1 \\ \vdots \\p_k}$ is a column of polynomials $\bfp$.
Then $(p)$ is invariant under the full action of $\uWpx$ and $M$ has a structure of $\uWpx$-module with trivial connection: $\xcdfm\xgn=0$.

As a $\Qbx$-module, $M=\Qbx/(\bfp)$ has an associated Koszul complex $\rsPb(M)=\Qv{\bfx,\bfth}$, where $\bfth$, $\xnum{\bfth}=k$, are odd variables ($\theta_i\theta_j + \theta_j\theta_i=0$, $\theta_i^2 = 0$) of homological degree $-1$ and the differential is
\begin{equation}
\label{eq:kszdf}
d = \sum_{i=1}^n p_i \pthe{i},
\end{equation}
where $\pthe{i}\theta_{j} = \delta_{ij}$.
In order to endow $\rsPb(M)$ with a \Wpeq\ structure, we define an action of the generators $\xLm$ on $\rsPb(M)$ by the formula
\begin{equation}
\label{eq:dfkro}
\xcdfm =
%\sum_{i=1}^{\nst} x_i^{m+1}\partial_{x_i}
\xhLm+ \xcnhAm,\qquad\xcnhAm=\sum_{i,j=1}^{k}\acpmij\, \theta_j \prthi.
\end{equation}
\begin{proposition}
The derivations~\eqref{eq:dfkro} commute with the differential $d$: $[d,\xcdfm]=0$.
\end{proposition}
\begin{proof}
The proposition is proved by a direct calculation:
%A direct calculation of the commutator $[d,\xcdfm]$ yields the following expression:
\[
- [d,\xhLm] = [d,\xcnhAm] =  \sum_{i,j=1}^{k} \acpmij\,p_j\prthi.
\]
\end{proof}
%
%[d,\xcdfm] = [d,\xhLm] + [d,\xcnhAm] = -\sum_{i,j=1}^{k} \acpmij\,p_j\prthi +  -\sum_{i,j=1}^{k} \acpmij\,p_j\prthi = 0
%\]

Define the curvature of a \tsq\ $\xcnbA$ of $k\times k$ matrices with polynomial entries as a double \tsq\ of matrices
\[
\crFmn{\xcnbA} = \xhLm\xcnAn - \xhLn\xcnAm + [\xcnAn,\xcnAm] - (n-m)\xcnAv{m+n}
\]
(\cf \ex{eq:odcrv}).
\begin{theorem}
The Koszul complex $\rsPb(M)=\Qv{\bfx,\bfth}$ with the differential~\eqref{eq:kszdf} and the action of $\aWp$ given by the derivations~\eqref{eq:dfkro} is \Wpeq\ if the matrices $\xcnbA$ satisfy the relation
\begin{equation}
\label{eq:kscnrl}
\crFmn{\xcnbA} = 0,\qquad m,n\geq 0.
%\bigl( \xhLm\xcnAn - \xhLn\xcnAm + [\xcnAn,\xcnAm]\bigr) = (n-m)\xcnAv{m+n}
\end{equation}
\end{theorem}
\begin{proof}
Is is easy to check that the relation~\eqref{eq:kscnrl} is equivalent to the condition $[\xcdfm,\xcdfn]=(n-m)\xcdfv{m+n}$, the latter guaranteeing that derivations $\xcdfm$ represent the algebra $\aWp$.
\end{proof}

\begin{remark}
The condition~\eqref{eq:mdfc} does not define the matrices $\xcnAm$ uniquely.
\end{remark}
\begin{remark}
Combining the structure relations~\eqref{eq:cmrl} with relations~\eqref{eq:mdfc} we find that the matrices $\xcnAm$ satisfy the relations
\begin{equation}
\label{eq:rlAp}
\crFmn{\xcnbA} \,\vcp = 0.
%\bigl( \xhLm\xcnAn - \xhLn\xcnAm + [\xcnAn,\xcnAm]\bigr)\,\vcp = (n-m)\xcnAv{m+n}\vcp.
\end{equation}
This relation is weaker than the condition~\eqref{eq:kscnrl} for the Koszul complex $\rsPb(M)$ to be \Wpeq.
\end{remark}

\section{Categorification of the \brwdsg\ by \Wpeq\ \Sglb s}
\subsection{Elementary \Wpeq\ \Sgl\ \bmdl s}
\label{ss:elsbm}
\def\nst{m}
For a positive integer $\nst$, the elementary \Wpeq\ \Sgl\ \bmdl\ $\mSn=\mSmbxy$ is a $\uWpxy$-module
\[
\mSn = \mQxy/\aIsn,\qquad\aIsn=\big(\smpev{1}(\bfy)-\smpev{1}(\bfx),\ldots,\smpev{\nst}(\bfy)-\smpev{\nst}(\bfx) \big),
\]
where $\xnum{\bfx}=\xnum{\bfy}=\nst$
%, the algebra $\Qbxy$ is considered as a $\uWpxy$-module with the \stact\ of $\aWp$
and the
$\Qbxy$-ideal $\aIsn$ is generated by the differences of elementary symmetric polynomials $\smpev{k}(\bfx) = \sum_{1\leq i_1<\ldots<i_k\leq\nst} x_{i_1}\cdots x_{i_k}$.
% being a symmetric polynomial of degree $i$.
%
%over $\Qbxy$, $\bfx = x_1,\ldots,x_\nst$, $\bfy = y_1,\ldots y_\nst$ with the \stact\ of $\aWp$, which is a quotient of $\Qbxy$ over the differences of symmetric polynomials:
%\[
%\mSn = \Qbxy/\aIsn,\qquad\aIsn=\big(\smpev{1}(\bfy)-\smpev{1}(\bfx),\ldots,\smpev{\nst}(\bfy)-\smpev{\nst}(\bfx) \big).
%\]
$\xLn \,\smpev{i}(\bfx)$ is also a symmetric polynomial, hence $\xLn \xbbr{\smpev{i}(\bfy) - \smpev{i}(\bfx)}\in\aIsn$, and $\mSn$ has a $\uWpxy$-module structure. %\Wpeq\ structure.
\def\nst{n}

For $\xnum{\bfx}=\xnum{\bfy}=\nst$ an extended elementary \Sglb\ $\mSmxyi$ ($i+m\leq \nst+1$) is a $\uWpxy$-module defined as the following tensor product:
\begin{equation}
\label{eq:exsbm}
\mSmxyi = \Dlv{\bfx',\bfy'}\otimes \mSmxypp \otimes \Dlv{\bfx''',\bfy'''},
\end{equation}
where $\bfx' = x_1,\ldots,x_{i-1}$, $\bfx'' = x_i,\ldots,x_{i+m-1}$, $\bfx''' = x_{i+m},\ldots,x_{\nst}$ and $\bfy'$, $\bfy''$ and $\bfy'''$ are defined analogously.

\begin{remark}
\label{rm:tpsbm}
%\Sglb\ is a direct summand of a tensor product of elementary ones.
All extended elementary \Sglb s are \smpr\ with \trkq s
\begin{equation}
\label{eq:rkesbm}
\xrkqx \mSmxyi = \xrkqy\mSmxyi = \qnmv{m}!,
\end{equation}
where
\[
\qnmv{m} = \frac{q^{2m}-1}{q^2-1},\qquad \qnmv{m}! = \prod_{i=1}^m \qnmv{i}.
\]
According to Corollary~\ref{cr:smprtp}, all other \Sglb s defined as their partial tensor products, are also \smpr, their derived partial tensor products are \xqisc\ to the ordinary ones and their ranks are determined by Proposition~\ref{pr:rkpr}
\end{remark}

First three \Wpeq\ elementary \Sgl\ \bmdl s are especially important and we denote them as
\begin{equation}
\label{eq:thsgbm}
\mMarc = \mSv{1},\qquad \mMcr = \mSv{2},\qquad \mMth = \mSv{3}.
\end{equation}

The \bmdl\ $\mMarc = \Dlv{x,y}=\mQv{x,y}/(y-x)$ represents the identity \enfn\ acting on the category of $\uWpsx$-modules: for any such module $\mM_x$ there is a canonical isomorphism
\[
\mM_y \cong
%\mMlpd{\bfx,\bfy}\otimes_{\Qx}
\mM_x \otimes_{\Qx} \mMlarcv{x,y}.
\]

Now we set $\nst = 2$, that is, $\bfx = x_1,x_2$ and $\bfy=y_1,y_2$, and consider two $\uWpxy$-modules: %$\mMcr$ and
\begin{equation}
\label{eq:dfsmdl}
\begin{aligned}
\mMpr & = \Dlxy
%\mMlarcv{x_1,y_1}\otimes\mMlarcv{x_2,y_2}
= \mQxy/\aIpr,
&&&&&\aIpr&= (y_1-x_1,y_2-x_2),
\\
\mMcr & = \mSv{2} = \mQxy/\aIcr,
&&&&&
\aIcr &=
(\spdf,y_1 y_2 - x_1 x_2),
\end{aligned}
\end{equation}
where $\spdf = y_1+y_2-x_1-x_2$.
Sometimes it is convenient to choose an `asymmetric' set of generators for the ideals:
%represent them as the following quotients:
\begin{equation}
\label{eq:asmi}
%\mMpr = \Qv{\bfx,\bfy}/(y_1+y_2-x_1-x_2,y_2-x_2),\qquad\!\!\! \mMcr = \Qv{\bfx,\bfy}/\big(y_1 + y_2 - x_1 - x_2,
%(y_2 - x_2)(y_2 - x_1)\big),
\aIpr=(\spdf,y_2-x_2),\qquad \aIcr = \big(\spdf,
(y_2 - x_2)(y_2 - x_1)\big).
\end{equation}

For brevity, we use notation $\bfsdf = \bsdfv{x_1}{x_2}$, where $\bfsdf = \sdf_0,\sdf_1,\ldots$. Note the relations between various polynomials $\sdf$ as they act on these \bmdl s:
\begin{align}
\label{eq:pepr}
\bfsdf & = \bsdfv{x_1}{x_2} = \bsdfv{y_1}{y_2} &  \mod \aIcr, %\text{in $\mMcr$,}
\\
\bfsdf & = \bsdfv{x_1}{x_2} = \bsdfv{y_1}{y_2} = \bsdfv{x_1}{y_2} & \mod \aIpr. %\text{in $\mMpr$.}
\end{align}

There are two important homomorphisms between the \uWpxybm s $\mMpr$ and $\mMcr$:% with \cnn:
\begin{equation}
\label{eq:mhoms}
%\hchine\colon \mcnnbp{\mMcr} \xrightarrow{\;\;1\;\;} \mcnnbp{\mMpr},\qquad
%\hchipo\colon \mcnpbp{\mMpr} \xrightarrow{\;\;\frac{1}{2}(y_2 - y_1 + x_2 - x_1)\;\;} \mMcr.
%
\hchine\colon \mMcr \xrightarrow{\;\;1\;\;}\mMpr,\qquad
\hchipo\colon \mMpr \xrightarrow{\;\;\frac{1}{2}(y_2 - y_1 + x_2 - x_1)\;\;} \mcnnbp{\mMcr}.
%\hchine\colon \mcnnp{\mMcr} \xrightarrow{\;\;1\;\;} \mcnnp{\mMpr}.
\end{equation}
Sometimes it is convenient to describe the second homomorphism equivalently as
\begin{equation}
\label{eq:scchp}
\hchipo\colon \mMpr \xrightarrow{\;\;y_2 - x_1\;\;} \mcnnbp{\mMcr}.
\end{equation}
%Corollary~\ref{cor:cnLm} says that a module with \cnn  $-\bfsdf$ is indeed \Wpeq, that is, these connections do not violate the commutation relations~\eqref{eq:cmrl}. Also it is obvious that
Obviously $\hchine$ commutes with \drv s of $\aWp$. The same is true for $\hchipo$: condition~\eqref{eq:hmcnd} is satisfied in view of the following chain of equalities (the first one is modulo $\aIcr$):
\begin{multline}
\shlf(y_2 - y_1 + x_2 - x_1)\,\sdfmv{x_1}{x_2} =
\shlf(y_2 - y_1)\,\sdfmv{y_1}{y_2} + \shlf(x_2 -x_1)\,\sdfmv{x_1}{x_2}
\\
=
\shlf(y_2^{m+1} - y_1^{m+1} + x_2^{m+1} - x_1^{m+1}) =\shlf \xLm\,(y_2 - y_1 + x_2 - x_1)\quad\mod \aIcr.
\end{multline}

\subsection{Categorification bracket}

In order to define the map
$\ctmdvbxy{\dmmy}\colon\brwgrn\rightarrow \ctCh\bigl(\ctWxy\bigr)$ which categorifies \brwd s, we introduce extended versions of \bmdl s  $\mMpr$ and $\mMcr$ as \uWpxybm s with $\xnum{\bfx}=\xnum{\bfy}=\nst$ for any $\nst\geq 2$.
% for a general number of variables $\nst\geq 2$.
%\Wpeq\ $\Qbxy$-module version of the \bmdl s $\mMpr$ and $\mMcr$.

First, we define the \bmdl
\[
\mMpd = \Dlxy = \botarcv{1}{\nst}. % = \Qbxy/(y_1-x_1,\ldots,y_{\nst}-x_{\nst}).
\]
which determines the identity \enfn\ in the category of \Wpeq\ $\Qbx$-modules: for any module $\mM_{\bfx}$ there is a canonical isomorphism
\begin{equation}
\label{eq:idmpr}
\mM_{\bfy} \cong \mM_{\bfx} \otimes_{\Qbx}\mMlpd{\bfx,\bfy}.
\end{equation}

Second, we define the \bmdl
\begin{equation}
\label{eq:bmpr1}
\mMcri =
%\mMcr^{n,i} =
\Dlv{\bfx',\bfy'}\otimes \mMcrlxyp
\otimes\Dlv{\bfx''',\bfy'''}
%\botarcv{1}{i-1} \otimes
% \mMcrbi \otimes \botarcv{i+2}{\nst}.
\end{equation}
(\cf the general definition~\eqref{eq:exsbm}).
Note that $\mMpd$ has a similar presentation
\begin{equation}
\label{eq:bmpr2}
\mMpd = \Dlv{\bfx',\bfy'}\otimes\mMprlxyp\otimes\Dlv{\bfx''',\bfy'''}.
%\botarcv{1}{k-1} \otimes
% \mMprbi \otimes \botarcv{k+2}{\nst}.
\end{equation}

%Let $\mMprbi$ and $\mMcrbi$ be the modules~\eqref{eq:dfsmdl} in which the indices 1 and 2 of the variables are replaced with $i$ and $i+1$.
%Then for $1\leq k\leq \nst$ we define the \bmdl s
%\begin{equation}
%\label{eq:bmpr1_1}
%\mMcri = \botarcv{1}{i-1} \otimes
% \mMcrbi \otimes \botarcv{i+2}{\nst}.
%\end{equation}
%Note that $\mMpd$ has a similar presentation
%\begin{equation}
%\label{eq:bmpr2_1}
%\mMpd = \botarcv{1}{k-1} \otimes
% \mMprbi \otimes \botarcv{k+2}{\nst}.
%\end{equation}
Finally, we define the $\aWp$-invariant homomorphisms
\begin{align}
\label{eq:hmio}
\hchinei\colon &
\mMcri \rightarrow \mMpd, \qquad
&
\hchinei & =
%\xId^{\otimes(i-1)}
\xId'
\otimes \hchine \otimes
\xId''',
%\xId^{\otimes(\nst-i-1)},
\\
\label{eq:hmit}
\hchipoi\colon &
\mMpd \rightarrow \mcnnpi{\mMcri}, \qquad
&
\hchipoi & =
%\xId^{\otimes(i-1)}
\xId'
\otimes \hchipo \otimes
\xId''',
%\xId^{\otimes(\nst-i-1)},
\end{align}
where
$\bsdfui  = \bsdfv{x_i}{x_{i+1}}$,
$\xId'$ and $\xId'''$ are the identity endomorphism of $\Dlv{\bfx',\bfy'}$ and $\Dlv{\bfx''',\bfy'''}$ and the formulas for $\hchinei$ and $\hchipoi$ are written relative to the presentations~\eqref{eq:bmpr1} and~\eqref{eq:bmpr2}.

%
%First, we define the modules
%\[
%\mMarcv{<i} = \bigotimes_{j=1}^{i-1} \Qv{x_j,y_j}/(y_j-x_j),\qquad
%\mMarcv{>i} = \bigotimes_{j=i+1}^{\nst} \Qv{x_j,y_j}/(y_j-x_j).
%\]
%
%
%Define the modules with differentiation
%\[
%\mMarcv{<i} = \bigotimes_{j=1}^{i-1} \Qv{x_j,y_j}/(y_j-x_j),\qquad
%\mMarcv{>i} = \bigotimes_{j=i+1}^{\nst} \Qv{x_j,y_j}/(y_j-x_j).
%\]
%Let $\mMprbi$ and $\mMcrbi$ be the modules~\eqref{eq:dfsmdl} in which the indices 1 and 2 of the variables are replaced with $i$ and $i+1$. Now we define the $\Qbx$ \bmdl s
%\begin{equation}
%\label{eq:nmdls}
%\mMprbi = \mMarcv{<i}\otimes\mMpri\otimes\mMarcv{>i+1},\qquad
%\mMcrbi = \mMarcv{<i}\otimes\mMcri\otimes\mMarcv{>i+1}.
%\end{equation}
%Obviously, $\mMpri\cong \Qbxy/(y_1-x_1,\ldots,y_{\nst} - x_{\nst})$, but we use the presentation~\eqref{eq:nmdls} in order to define the homomorphisms
%\begin{align}
%\label{eq:hmio.1}
%\hchinei\colon &
%\mcnnpi{\mMcri} \rightarrow \mcnnpi{\mMpri}, \qquad
%&
%\hchinei & = \xId^{\otimes(i-1)} \otimes \hchine \otimes \xId^{\otimes(\nst-i-1)},
%\\
%\label{eq:hmit.1}
%\hchipoi\colon &
%\mcnppi{\mMpri} \rightarrow \mMcri, \qquad
%&
%\hchipoi & = \xId^{\otimes(i-1)} \otimes \hchipo \otimes \xId^{\otimes(\nst-i-1)},
%\end{align}
%where
%$\sdfui  = \sdfv{x_i}{x_{i+1}}
%$ and $\xId$ is the identity homomorphism.
Now to elementary \brwd s $\sggni$ and $\sggxii$ we associate cones in the category $\ctCh$ of the homomorphisms~\eqref{eq:hmio} and~\eqref{eq:hmit}:
\begin{equation}
\begin{split}
\label{eq:bcmps}
\ctmdv{\sggni} &=(\dgsha\dgsht)^{-\frac{1}{2}}\dgshq^2\, \boxed{
\mMpd \xrightarrow{\;\;\hchipoi\;\;} \dgshq^{-2}\,\dgsht\mcnnpi{\mMcri}
}\; ,
%\qquad
\\
\ctmdv{\sggxii}
&
=(\dgsha\dgsht)^{\frac{1}{2}}\dgshq^{-2}\, \boxed{
\dgsht^{-1}\mMcri \xrightarrow{\;\;\hchinei\;\;} \mMpd
}\;.
\end{split}
\end{equation}
The action of the map $\ctmdv{\dmmy}$ on other elements of $\brwgrn$ is defined by the relation~\eqref{eq:smbrpr}.
%{eq:bmtp}.

%A  is a special graph with equal number of lower and upper legs. semi-group is a semi-group generated by the `identity line' and a blob
%\[
%asdf
%\]
%and by two operations. The first operation is the disjoint union $\xgr_1\# \xgr_2$ which places the  s $\xgr_1$ and $\xgr_2$ side-by-side and the second operation is the composition $\xgr_1\xcm \xgr_2$ which is defined when $\xgr_1$ and $\xgr_2$ have the same number of legs and places one atop the other.

%\section{Invariance}

\section{\Wpeq\ categorification of the braid group}
Throughout the paper, by generators of a $\uWpxy$-module $\mMxy$ we mean the generators of $\mMxy$ as a $\Qbxy$-module.

\subsection{Main theorem}

\begin{theorem}
\label{th:first}
There exists a unique homomorphism
$\ctmvbxy{\dmmy}\colon\brgrn \rightarrow \ctDr\bigl(\ctWxy\bigr)$ which makes the following diagram commutative:
\begin{equation*}
\xymatrix@R=1.5cm@C=1.25cm{
\brwgrn \ar @{->>}[r]^-{\fbr}
%\ar@/_3.5pc/[dd]_-{\actmdvbxy{\dmmy}}
\ar[d]_-{\ctmdvbxy{\dmmy}}
&
\brgrn
%\ar@/^3.5pc/[dd]^(0.25){\actmvbxy{\dmmy}}
\ar[d]^-{\ctmvbxy{\dmmy}}
%\brgrn \ar[d]^-{\actmvbxy{\dmmy}}
\\
\ctCh\bigl(\ctWxy\bigr)
%\ar[d]^-{\dmmy\otux\DlXx}
\ar[r]^-{\Fffq}
&
\ctDr\bigl(\ctWxy\bigr)
%\ar[d]_-{\dmmy\otux\DlXx}
%\ar[r]^(0.675){\tHhomtxy(\dmmy)}|(0.45)\hole
}
\end{equation*}
\end{theorem}
\begin{proof}
The uniqueness of $\ctmvbxy{\dmmy}$ follows from the surjectivity of the map $\fbr$.
In order to establish its existence we have to prove  that the complex $\ctmdv{\sggni}$ is the inverse of the complex $\ctmdv{\sggxii}$ with respect to the monoidal structure of $\ctK(\ctWxy)$, and that the map $\ctmdv{\dmmy}$ respects the braid relations. For obvious reasons we will refer to these properties as second and third Reidemeister moves invariance.
The factorization of \bmdl s~\eqref{eq:bmpr1} and~\eqref{eq:bmpr2} and the local nature of the homomorphisms~\eqref{eq:hmio}, \eqref{eq:hmit} implies that it is sufficient to prove the second Reidemeister move invariance for the 2-strand braid and the third Reidemeister move for the 3-strand braid. This will be done in the next two subsections.
\end{proof}

%
% In other words, we have to establish the invariance of the map $\ctmdv{\sggxii}$ with respect to the second and third Reidemeister moves, the second one involving only the similarly oriented strands.
%

\subsection{Second Reidemeister move invariance}

Let $\bfx = x_1,x_2$ , $\bfy = y_1,y_2$ and $\bfz = z_1,z_2$. Recall that by the definition of the map $\ctmdvbxy{\dmmy}$ the complex associated to the product of elementary \brwd s
$\sggno \sggxio$ is the tensor product of elementary complexes over the intermediate algebra:
\begin{equation}
\label{eq:dftwrd}
\ctmdvbxz{ \sggno \sggxio  } = \ctmdvbxy{\sggno} \oty \ctmdvbyz{\sggxio}
\end{equation}

\begin{theorem}
\label{th:rdmtwo}
There is a homotopy equivalence of complexes
\begin{equation}
\label{eq:fhteqt}
\ctmdvbxz{ \sggno \sggxio  } \sim \mMprvr{\bfx,\bfz}
\end{equation}
within the category $\ctCh(\ctWxy)$.
\end{theorem}

\begin{lemma}
\label{lm:smspl}
The following tensor product splits as a $\uWpxz$-module:
%\Wpeq\ $\Qbxz$-module:
\begin{equation}
\label{eq:spldcr}
\mMcrvr{\bfx,\bfy}\oty \mMcrvr{\bfy,\bfz} \cong
\mMcrvr{\bfx,\bfz} \oplus \mcnbpbp{\mMcrvr{\bfx,\bfz}}.
\end{equation}
%The generators of the modules $\mMcrvr{\bfy,\bfz}$ in the \rhs may be chosen as
%$\mgenvv{\xcrs}{\bfx,\bfz}^1 = \mgenvv{\xcrs}{\bfx,\bfy}\otimes\mgenvv{\xcrs}{\bfy,\bfz}$ and
%$\mgenvv{\xcrs}{\bfx,\bfz}^2 = (y_2 - y_1)\mgenvv{\xcrs}{\bfx,\bfz}^1$,%
%
%
The generators of the modules $\mMcrvr{\bfy,\bfz}$ in the \rhs may be chosen as
$\mgen_1 = \mgenv{\bfx,\bfy}\otimes\mgenv{\bfy,\bfz}$ and
$\mgen_2 = (y_2 - y_1)  \mgen_1$,
where $\mgenvv{\xcrs}{\bfx,\bfy}$ and $\mgenvv{\xcrs}{\bfy,\bfz}$ are the generators of
$\mMcrvr{\bfx,\bfy}$ and $\mMcrvr{\bfy,\bfz}$ respectively. The following diagram is commutative:
\begin{equation}
\label{eq:fcmdg}
\mvcn{
\xymatrix@C=1.2cm@R=1.2cm{
\mMprvr{\bfx,\bfy}\oty \mMcrvr{\bfy,\bfz}
\ar[r]^-{\hchipo \otimes \xId }
\ar[d]^-{\cong}
&
\mcnbnbp{\mMcrvr{\bfx,\bfy}\oty \mMcrvr{\bfy,\bfz} }
\ar[r]^-{\xId\otimes\hchine}
\ar[d]^-{\cong}
&
\mcnbnbp{\mMcrvr{\bfx,\bfy}\oty \mMprvr{\bfy,\bfz} }
\ar[d]^-{\cong}
\\
\mMcrvr{\bfx,\bfz}
\ar[r]^-{\hlf(x_2 - x_1,\xId)}
&
\mcnbnbp{\mMcrvr{\bfx,\bfz}} \oplus \mMcrvr{\bfx,\bfz}
\ar[r]^-{(\xId,z_2-z_1)}
&
\mcnbnbp{\mMcrvr{\bfx,\bfz} }
}
}
\end{equation}
\end{lemma}
\begin{proof}
%It is easy to establish the isomorphism~\eqref{eq:spldcr} as a splitting of $\Qbxy$-modules. Indeed,
%\[
%\mMcrvr{\bfx,\bfy}\otimes_{\Qby} \mMcrvr{\bfy,\bfz} \cong
%\mMcrvr{\bfx,\bfz}\otimes_{\Qbx} \Qbxy/\big(y_1 + y_2 - (x_1+x_2),y_1 y_2 - x_1 x_2\big)
%\]
%and
%\[
%\Qbxy/\big(y_1 + y_2 - (x_1+x_2),y_1 y_2 - x_1 x_2\big)
%\cong \Qbx \oplus \Qbx.
%\]
It is easy to establish the isomorphism~\eqref{eq:spldcr} in the category of $\Qbxz$-modules, the \rhs ones having generators $\mgen_1$ and $\mgen_2$. Indeed, the tensor product in the \lhs can be presented as the quotient $\Qv{\bfx,\bfy,\bfz}$-module
\[
\mMcrvr{\bfx,\bfy}\otimes_{\Qby} \mMcrvr{\bfy,\bfz} \cong \Qv{\bfx,\bfy,\bfz}/\aIxyz,
\]
where
\[
\aIxyz =
\begin{pmatrix}
\smpeo(\bfy)-\smpeo(\bfx)
\\
\smpet(\bfy)-\smpet(\bfx)
\\
\smpeo(\bfz)-\smpeo(\bfy)
\\
\smpet(\bfz)-\smpet(\bfy)
\end{pmatrix}
=
\begin{pmatrix}
\smpeo(\bfz)-\smpeo(\bfx)
\\
\smpet(\bfz)-\smpet(\bfx)
\\
\smpeo(\bfy)-\smpeo(\bfx)
\\
\smpet(\bfy)-\smpet(\bfx)
\end{pmatrix}
,
\]
while $\smpeo$ and $\smpet$ are elementary symmetric polynomials of degrees one and two.
The second choice of generators of the ideal $\aIxyz$ implies the isomorphism of $\Qbxz$-modules:%
%there are isomorphisms of $\Qbxz$-modules:
\begin{equation}
\label{eq:fvsiso}
\mMcrvr{\bfx,\bfy}\otimes_{\Qby} \mMcrvr{\bfy,\bfz} \cong
\mMcrvr{\bfx,\bfz}\otimes_{\Qbx} \mMcrvr{\bfx,\bfy}.
\end{equation}
At the same time, there is an isomorphism of $\Qbx$-modules
\begin{equation}
\label{eq:vsmspl}
\mMcrvr{\bfx,\bfy}
\cong \Qbx \oplus \Qbx
\end{equation}
%both isomorphisms being considered over $\Qbxy$,
and the generators of the latter modules $\Qbx$ can be chosen as $\mgen_1'=\mgenv{\bfx,\bfy}$ and
\\
$\mgen_2'=(y_2 - y_1)\mgenv{\bfx,\bfy}$, where $\mgenv{\bfx,\bfy}$ is the generator of $\mMcrvr{\bfx,\bfy}$.

The action of a $\aWp$ generator $\xLm$ on the generator $\mgen_2'$ of the second module in the \rhs of \ex{eq:vsmspl} is
\[
\xcdf_m\mgen_2' = \big( \xLm\, (y_2 - y_1)\big) \mgenv{\bfx,\bfy} = \sdfmv{y_1}{y_2}\, \mgen_2' =
\sdfmv{x_1}{x_2}\,\mgen_2',
\]
hence the splitting~\eqref{eq:vsmspl} is \Wpeq\ if we add connection $\bfsdf$ to the second summand in its \rhs:
\begin{equation}
\nonumber
%\label{eq:vsmspl}
\mMcrvr{\bfx,\bfy}
\cong \mQx \oplus \mcnbpbp{\mQx}.
\end{equation}
Combining this splitting with the isomorphism~\eqref{eq:fvsiso} we get the splitting~\eqref{eq:spldcr}.

In order to verify the commutativity of the diagram~\eqref{eq:fcmdg} we check the action of the upper horizontal homomorphisms on the generators of modules:
\begin{gather*}
(\hchipo \otimes \xId)\, (\mgenv{\bfx,\bfy}\otimes\mgenv{\bfy,\bfz}) =
\shlf(x_2 - x_1 + y_2 - y_1)\,\mgenv{\bfx,\bfy}\otimes\mgenv{\bfy,\bfz} = \shlf(x_2 - x_1)\,\mgen_1 + \shlf \mgen_2,
\\
(\xId\otimes\hchine)\, (\mgen_1) =
(\xId\otimes\hchine)\, (\mgenv{\bfx,\bfy}\otimes\mgenv{\bfy,\bfz}) = \mgenv{\bfx,\bfz},
\\
(\xId\otimes\hchine)\, (\mgen_2) =
(\xId\otimes\hchine)\, \big((y_2-y_1)\,\mgenv{\bfx,\bfy}\otimes\mgenv{\bfy,\bfz}\big) =
(y_2-y_1)\,\mgenv{\bfx,\bfy}\otimes\mgenv{\bfy,\bfz}
=(z_2-z_1)\,\mgenv{\bfx,\bfy}\otimes\mgenv{\bfy,\bfz}
\end{gather*}

%and consequently the splitting~\eqref{eq:spldcr} are \Wpeq.

\end{proof}

\begin{proof}[Proof of Theorem~\ref{th:rdmtwo}]
We prove the theorem by a sequence of isomorphisms and homotopy equivalences:
\begin{multline*}
\ctmdvbxz{ \sggno \sggxio  } = \ctmdvbxy{\sggno} \oty \ctmdvbyz{\sggxio}
\\[15pt]
\shoveleft{
=\;\;
\boxed{
\vcenter{
%\boxed{
\xymatrix@C=0.95cm@R=0.5cm{
%\xymatrix@C=0.2cm{
&
\mcnbnbp{\mMcrvr{\bfx,\bfy}\oty \mMcrvr{\bfy,\bfz} }
\ar[dr]^-{\xId\otimes\hchine}
\\
\mMprvr{\bfx,\bfy}\oty \mMcrvr{\bfy,\bfz}
\ar[ur]^-{\hchipo \otimes \xId }
\ar[dr]^-{\xId\otimes\hchine}
&&
\mcnbnbp{\mMcrvr{\bfx,\bfy}\oty \mMprvr{\bfy,\bfz} }
\\
&
\mMprvr{\bfx,\bfy}\oty \mMprvr{\bfy,\bfz}
\ar[ur]^-{-\hchipo \otimes \xId }
}
}
}
}
\displaybreak[0]
\\[15pt]
\shoveleft{
\cong\;\;
\boxed{
\vcenter{
%\boxed{
%\xymatrix@R=1cm@C=2cm{
\xymatrix@C=2cm@R=0.5cm{
&
\mcnbnbp{\mMcrvr{\bfx,\bfz}} \oplus \mMcrvr{\bfx,\bfz}
\ar[dr]^-{-(\xId,z_2-z_1)}
\\
\mMcrvr{\bfx,\bfz}
\ar[ur]^-{\hlf(x_2 - x_1,\xId) \;\;\;\;}
\ar[dr]^-{1}
&&
\mcnbnbp{\mMcrvr{\bfx,\bfz} }
\\
&
\mMprvr{\bfx,\bfz}
\ar[ur]_-{\;\;\;\;\;\;\hlf(x_2 - x_1 + z_2 - z_1) }
}
}
}
}
\displaybreak[0]
\\[15pt]
\shoveleft{
\sim\;\;
\boxed{
\vcenter{
%\boxed{
%\xymatrix@R=1cm@C=2cm{
\xymatrix@R=0.25cm@C=3cm{
\mcnbnbp{\mMcrvr{\bfx,\bfz}}
%\oplus \mMcrvr{\bfx,\bfz}
\ar[dr]^-{-\xId}
\\
%\mMcrvr{\bfx,\bfz}
%\ar[ur]^-{\hlf(x_2 - x_1,\xId) \;\;\;\;}
%\ar[dr]^-{1}
&
\mcnbnbp{\mMcrvr{\bfx,\bfz} }
\\
\mMprvr{\bfx,\bfz}
\ar[ur]_-{\;\;\;\;\;\;\hlf(x_2 - x_1 + y_2 - y_1) }
}
}
}%\hspace{-2in}
}
\displaybreak[0]
\\[15pt]
%\\
\shoveleft{
\sim \;\;\mMprvr{\bfx,\bfz}.
}
\hfill
\end{multline*}
Here the first isomorphism is the definition~\eqref{eq:dftwrd}, the second isomorphism is the definition of the tensor product of complexes~\eqref{eq:bcmps} and the third isomorphism comes from Lemma~\ref{lm:smspl}. The last two homotopy equivalences are contractions of a part of a cone within the homotopy category:
\[
\boxed{ A\longrightarrow B} \sim
\begin{cases}
A[1], & \text{if $B$ is contractible,}
\\
B, & \text{if $A$ is contractible.}
\end{cases}
\]
\end{proof}

\subsection{Third Reidemeister move invariance}

\def\cCatbxy{ \ctDr(\ctWxy) }
\def\cCatbxz{ \ctDr(\ctWxz) }
\def\cCatbyz{ \ctDr(\ctWyz) }
\def\cCatbxw{ \ctDr(\ctWxw) }
\def\xqh{ \eqsmout }

Let $\bfx = x_1,x_2,x_3$, $\bfy = y_1,y_2,y_3$, $\bfz = z_1,z_2,z_3$ and $\bfw=w_1,w_2,w_3$.
Recall that by definition of the map $\ctmdv{\dmmy}$
% in the diagram~\eqref{eq:mndiag}, the complex associated to the product of elementary \brwd s
%$\sggno \sggxio$ is the tensor product of elementary complexes over the intermediate algebra:
the categorification of the triple products of \brwd s appearing at both sides of the third Reidemeister move takes the form
%both sides of the third Reidemeister move takes the form
\begin{equation*}
\label{eq:dftwrd.2}
\ctmdvbxw{ \sggni \sggnj \sggni  } = \ctmdvbxy{\sggni} \oty \ctmdvbyz{\sggnj} \otz \ctmdvbzw{\sggni}.
\end{equation*}

\begin{theorem}
\label{th:rdmthr}
The following complexes are \qisc\ in the category $\ctCh(\ctWxw)$:
\begin{equation}
\label{eq:fhteq}
\ctmdvbxw{ \sggno \sggnt \sggno } \eqsmout \ctmdvbxw{ \sggnt \sggno \sggnt }.
\end{equation}
% $\cCatbxw$.
\end{theorem}

%The proof of this theorem is shortened with the help of a special involutive \enfn\ of the category $\cCatbxy$.
The algebra $\Qbxy$ has an automorphism $\foth$ which swaps simultaneously $x_1$ with $x_3$ and $y_1$ with $y_3$:
$\foth(x_1) = x_3$, $\foth(x_2)=x_2$, $\foth(x_3) = x_1$ and the same for $y_1$, $y_2$ and $y_3$. This automorphism generates an involutive \enfn\ $\Fothxy$ of the category $\ctCh(\ctWxy)$. The obvious isomorphism of \bmdl\ complexes
$\Fothxy\big(\ctmdvbxy{\sggno}\big) \cong \ctmdvbxy{\sggnt}$ and the commutativity of a diagram
\[
\xymatrix@R=1.5cm@C=1.5cm{
\cCatbxy \times \cCatbyz \ar[r]^-{\otdry}
\ar[d]_-{\Fothxy\times\Fothv{\bfy,\bfz}}
& \cCatbxz
\ar[d]^-{\Fothv{\bfx,\bfz}}
\\
\cCatbxy \times \cCatbyz \ar[r]^-{\otdry}
& \cCatbxz
}
\]
imply the isomorphism of complexes
$\ctmdvbxw{ \sggnt \sggno \sggnt } \cong \Fothxy\big( \ctmdvbxw{ \sggno \sggnt \sggno }\big)$, hence the homotopy equivalence~\eqref{eq:fhteq} follows from the homotopy equivalence
\begin{equation}
\label{eq:mrsht}
\ctmdvbxw{ \sggno \sggnt \sggno } \xqh
\Fothxw\big( \ctmdvbxw{ \sggno \sggnt \sggno }\big).
\end{equation}
In other words, it is sufficient to show that the \rhs of \ex{eq:fhteq} is symmetric under the transposition of indices 1 and 3 of the variables.
%This makes the proof twice as short as that of the original relation~\eqref{eq:fhteq}.

Before proceeding with the proof of \ex{eq:mrsht} we introduce new notations and establish decomposition of some \bmdl s appearing in the complex $\ctmdvbxw{ \sggno \sggnt \sggno }$. The simplest \bmdl s are
\begin{align*}
\mMtharxw & = \mMpdrxw =\Dlv{\bfx,\bfw},
%\Qbxw / (w_1 - x_1,w_2 - x_2,w_3-x_3),
\\
\mMcrarxw & = \mMcrvv{\bfx,\bfw}{1},
\\
\mMacrrxw & = \mMcrvv{\bfx,\bfw}{2},
\end{align*}
and the \Sgl\ \bmdl\ $\mMthrrxw$ of \ex{eq:thsgbm}. The action of the transposition \enfn\ $\Foth$ on these \bmdl s is obvious:
\[
\Fothxw(\mMtharxw) \cong \mMtharxw,\qquad
\Fothxw(\mMcrarxw) \cong \mMacrrxw,\qquad
\Fothxw(\mMthrrxw) \cong \mMthrrxw.
\]

\begin{lemma}
The $\uWpxw$-module
\[
\mMcprrxw = \mMcravr{\bfx,\bfy}\oty \mMthavr{\bfy,\bfz} \otz
\mMacrvr{\bfz,\bfw}
\]
has a single generator $\mgenvv{\xcpr}{\bfx,\bfw} = \mgenvv{\xcra}{\bfx,\bfy}\otimes
\mgenvv{\xtha}{\bfy,\bfz}\otimes\mgenvv{\xacr}{\bfz,\bfx}$ and it can be presented as a quotient
$\mMcprrxw \cong \mQxw/\aIvv{\xcpr}{\bfx,\bfw}$, where
\[\aIvv{\xcpr}{\bfx,\bfw} = \big(w_1 + w_2 + w_3 - (x_1 + x_2 + x_3), (w_1-x_1)(w_1 - x_2), (w_3 - x_3)(w_3 - x_2)\big).\]
\end{lemma}
The proof is obvious and we omit it.
The \bmdl\ $\mMiprrxw \cong \Fothxw(\mMcprrxw)$ has a similar property.

Consider three $\uWpxw$-modules defined as tensor products
\begin{align}
\label{eq:thtpr}
\mMcrarxwd & =
\mMcravr{\bfx,\bfy}\oty\mMthavr{\bfy,\bfz}\otz \mMthavr{\bfz,\bfw}
\\
\mMcrarxwu & =
\mMthavr{\bfx,\bfy}\oty\mMthavr{\bfy,\bfz}\otz \mMcravr{\bfz,\bfw}
\\
\mMbtarxw & = \mMcravr{\bfx,\bfy}\oty\mMthavr{\bfy,\bfz}\otz \mMcravr{\bfy,\bfz}
\end{align}
and the homomorphisms
\begin{equation}
\label{eq:twhmso}
\xymatrix@C=3cm{
\mMcrarxwd
&
\mMbtarxw
\ar[l]_-{
%\xfhmd=
\hchine\otimes\xId\otimes\xId}
\ar[r]^-{
%\xfhmu=
\xId\otimes\xId\otimes\hchine}
&
\mMcrarxwu
}
\end{equation}
Canonical isomorphisms
\[
\mMcrarxwd \cong\mMcrarxwu \cong \mMcrarxw
\]
imply that $\mMcrarxwd$ and $\mMcrarxwu$ have single generators
$\mgencraxwu$ and $\mgencraxwd$, which are tensor products of the generators of factor modules in the tensor products~\eqref{eq:thtpr}.
\begin{lemma}
\label{lm:thsplo}
The $\uWpxw$-module $\mMbtarxw$
%The following tensor product
%\[
%\mMbtarxw = \mMcravr{\bfx,\bfy}\oty\mMthavr{\bfy,\bfz}\otz \mMcravr{\bfy,\bfz}
%\]
splits:
% as a $\uWpxw$-module:
\begin{equation}
\label{eq:spldcra}
\mMbtarxw \cong
%\mMcravr{\bfx,\bfy}\otimes_{\Qby}\mMthavr{\bfy,\bfz}\otimes_{\Qbz} \mMcravr{\bfy,\bfz} \cong
\mMcrarxwo \oplus \mcnbpbpot{\mMcrarxwt},
\end{equation}
where
$\mMcrarxwo\cong\mMcrarxwt\cong\mMcrarxw$ as $\Qbxw$-modules, while the connection $\bfsdfot$ of
$\mMcrarxwt$ is
\begin{equation}
\label{eq:cpot}
\bfsdfot = \bsdfv{x_1}{x_2} = \bsdfv{w_1}{w_2}\qquad\mod \aIv{\xcra}.
\end{equation}
The generators of the modules $\mMcrarxwo$ and $\mMcrarxwt$ may be chosen as
\begin{equation}
\label{eq:gnssm}
\mgen_1 = \mgenv{\bfx,\bfy}\otimes\mgenv{\bfy,\bfz}\otimes\mgenv{\bfz,\bfw},\qquad
\mgen_2 = (y_2 - y_1)  \mgen_1 = (z_2-z_1)\mgen_1,
\end{equation}
where $\mgenvv{\xcrs}{\bfx,\bfy}$ and $\mgenvv{\xcrs}{\bfy,\bfz}$ are the generators of
$\mMcravr{\bfx,\bfy}$, $\mMthavr{\bfy,\bfz}$ and $\mMcravr{\bfz,\bfw}$ respectively.
The homomorphisms~\eqref{eq:twhmso}
%\[
%\xymatrix@C=3cm{
%\mMcrarxwd
%&
%\mMbtarxw
%\ar[l]_-{
%%\xfhmd=
%\hchine\otimes\xId\otimes\xId}
%\ar[r]^-{
%%\xfhmu=
%\xId\otimes\xId\otimes\hchine}
%&
%\mMcrarxwu
%}
%\]
are presented by matrices
\[
\hchine\otimes\xId\otimes\xId =
%\begin{pmatrix}
%1 & w_2 - w_1
%\end{pmatrix},
\bbrs{ 1\;\;\;w_2-w_1},
\qquad
\xId\otimes\xId\otimes\hchine =
%\begin{pmatrix}
%1 & x_2 - x_1
%\end{pmatrix}
\bbrs{ 1\;\;\;x_2-x_1}
\]
relative to the generators $\mgen_1$, $\mgen_2$ of the middle module and $\mgencraxwd$, $\mgencraxwu$ of the target modules.
%
%
%The following diagram is commutative:
%\begin{equation}
%\xymatrix@C=2.5cm@R=1.5cm{
%\mMcrarxwd
%\ar[d]^-{\cong}
%&
%\mMbtarxw
%\ar[l]_-{\hchine\otimes\xId\otimes\xId}
%\ar[r]^-{\xId\otimes\xId\otimes\hchine}
%\ar[d]^-{\cong}
%&
%\mMcrarxwu
%\ar[d]^-{\cong}
%\\
%\mMcrarxw
%&
%\mMcravr{\bfx,\bfw} \oplus \mcnbpbpot{\mMcravr{\bfx,\bfw}}
%\ar[l]_-{(1,w_2-w_1)}
%\ar[r]^-{(1,x_2 - x_1)}
%&
%\mMcrarxw
%}
%\end{equation}
%where
%\begin{align*}
%\mMcrarxwd & =
%\mMcravr{\bfx,\bfy}\oty\mMthavr{\bfy,\bfz}\otz \mMthavr{\bfz,\bfw}
%\\
%\mMcrarxwu & =
%\mMthavr{\bfx,\bfy}\oty\mMthavr{\bfy,\bfz}\otz \mMcravr{\bfz,\bfw}
%\end{align*}
\end{lemma}
\begin{proof}
This lemma follows easily from Lemma~\ref{lm:smspl} if we tensor multiply its formulas by the \bmdl s $\mMarc$ over the polynomial algebras of variables with index 3.
\end{proof}

Consider the $\uWpxw$-module
\[
\mMstarxw = \mMcravr{\bfx,\bfy}\oty \mMacrvr{\bfy,\bfz}\otz\mMcravr{\bfz,\bfw}.
%
% \mMthavr{\bfy,\bfz} \otz
%\mMacrvr{\bfz,\bfw}
\]
and the homomorphism
\begin{equation}
\label{eq:ohmso}
\xymatrix@C=2.5cm{
\mMstarxw
\ar[r]^-{\xId\otimes\hchine\otimes \xId}
&
\mMbtarxw
}
\end{equation}

\begin{lemma}
\label{lm:thsplt}
There is an exact sequence of $\uWpxw$-modules
\begin{equation}
\label{eq:exseq}
%\mMthrrxw\hookrightarrow\mMstarxw\twoheadrightarrow\mcnbpbpot{\mMcravr{\bfx,\bfw}},
0\longrightarrow\mMthrrxw\longrightarrow\mMstarxw\longrightarrow\mcnbpbpot{\mMcravr{\bfx,\bfw}}\longrightarrow 0,
\end{equation}
which splits in the category of $\Qbxw$-modules:
\begin{equation}
\label{eq:spldthcr}
\mMstarxw \cong \mMthrrxw \oplus \mMcrarxw.
\end{equation}
The generators of the modules in the \rhs may be chosen as
\begin{equation}
\label{eq:gnstr}
\mgenvv{\xthr}{\bfx,\bfw}=\mgenv{\bfx,\bfy}\otimes\mgenv{\bfy,\bfz}\otimes\mgenv{\bfz,\bfw},
\qquad
\mgenvv{\xcra}{\bfx,\bfw} =  y\mgenvv{\xthr}{\bfx,\bfw},
\end{equation}
where
\begin{equation}
\label{eq:dfy}
y = 2(z_3-y_2) =  y_1 - y_2 - x_1 - x_2 + 2w_3,
%= y_1 - (x_1+x_2) + w_3,
\end{equation}
while $\mgenv{\bfx,\bfy}$, $\mgenv{\bfy,\bfz}$ and $\mgenv{\bfz,\bfw}$ are generators of the \bmdl s $\mMcravr{\bfx,\bfy}$, $\mMacrvr{\bfy,\bfz}$ and $\mMcravr{\bfz,\bfw}$ respectively.
The homomorphism~\eqref{eq:ohmso} is described by the matrix
\begin{equation}
\label{eq:fmtr}
\xId\otimes\hchine\otimes \xId =
\begin{pmatrix}
1 & -x_1 - x_2 + 2w_3
\\
0 & 1
\end{pmatrix}
\end{equation}
relative to the generators~\eqref{eq:gnstr} of $\mMstarxw$ and generators~\eqref{eq:gnssm} of $\mMbtarxw$.
%
%
%The following diagram is commutative:
%\begin{equation}
%\label{eq:cmdgo}
%\xymatrix@R=1.5cm
%%@C=4.5cm{
%@C=2cm{
%\mMstarxw
%\ar[d]_-{\cong}
%\ar[r]^-{\xId\otimes\hchine\otimes \xId}
%&
%\mMbtarxw
%\ar[d]^-{\cong}
%\\
%\mMthrvr{\bfx,\bfw} \oplus \mcnbpbpot{\mMcravr{\bfx,\bfw}}
%\ar[r]^-{\xfhm}
%%{\left(\begin{smallmatrix}1 &\;\;\;\;\; - \shlf(x_1 + x_2) + w_3 \\ 0 & \;\;\;\;\;\shlf\end{smallmatrix}\right)}
%&
%\mMcravr{\bfx,\bfw} \oplus \mcnbpbpot{\mMcravr{\bfx,\bfw}}
%}
%\end{equation}
%where the homomorphism $\xfhm$ is represented by the matrix
%\begin{equation}
%\label{eq:fmtr1}
%\xfhm =
%\begin{pmatrix}1 &\;\;\;\;\; - x_1 - x_2 + 2w_3 \\ 0 & \;\;\;\;\;1\end{pmatrix}
%\end{equation}
%relative to the generators~\eqref{eq:gnstr} and~\eqref{eq:gnssm}.
\end{lemma}

\begin{proof}

The splitting~\eqref{eq:spldthcr} and the generators~\eqref{eq:gnstr} are well-known in the theory of \Sglb s, but we will derive these results here for completeness.
%Our main goal is to establish the splitting~\eqref{eq:spldthcr} of $\Qbxw$-modules with generators~\eqref{eq:gnstr}.
The rest of the lemma is an easy corollary: the matrix presentation~\eqref{eq:fmtr} and the connections in the submodule and quotient module in the exact sequence~\eqref{eq:exseq} follow from the action of the homomorphism $\xId\otimes\hchine\otimes \xId$ and the generators  $\xLm\in\aWp$ on the module generators~\eqref{eq:gnstr}.
%
%First of all, we establish the splitting~\eqref{eq:spldthcr} of $\Qbxw$-modules with generators~\eqref{eq:gnstr}. The matrix presentation~\eqref{eq:fmtr} of the homomorphism~\eqref{eq:ohmso}
%%commutativity of the diagram~\eqref{eq:cmdgo}
%follows from the action of its homomorphisms on module generators. Then we will check the action of $\aWp$ on the generators~\eqref{eq:gnstr} and thus establish connections in the submodule and quotient module in the exact sequence~\eqref{eq:exseq}.

%The splitting~\eqref{eq:spldthcr} is well-known, but we have to remind ourselves the formulas~\eqref{eq:gnstr} for the generators.
%
If we use the asymmetric formulas~\eqref{eq:asmi}, then
it is easy to see that the tensor product in the \lhs of \ex{eq:spldthcr} has the following presentation as a
%\Wpeq\
$\Qbxw$-module:
\begin{equation}
\label{eq:mdprq}
\mMstarxw
%\mMcravr{\bfx,\bfy}\otimes_{\Qby}\mMacrvr{\bfy,\bfz}\otimes_{\Qbz} \mMcravr{\bfy,\bfz}
 \cong
\Qv{\bfx,\bfw,y_1,y_2,z_2}/\aIv{1},
\end{equation}
where
\begin{equation}
\label{eq:fstid}
\aIv{1} =
\begin{pmatrix}
y_2 - (x_1 + x_2 - y_1)
\\
(y_1 - x_1)(y_1-x_2)
\\
z_2 - (w_1 + w_2 - y_1)
\\
(y_1-w_1)(y_1-w_2)
\\
z_2 + w_3 - (y_2 + x_3)
\\
z_2 w_3 - y_2 x_3
\end{pmatrix},
\end{equation}
and elements~\eqref{eq:gnstr} become
\begin{equation}
\label{eq:gnrpq}
\mgenvv{\xthr}{\bfx,\bfw}=1,
\qquad
\mgenvv{\xcra}{\bfx,\bfw} =  y.
\end{equation}
%

%
%
%We have to show that the quotient~\eqref{eq:mdprq} splits as a $\Qbxw$-module according to \ex{eq:spldthcr}, with generators being
%\begin{equation}
%\label{eq
%
%\end{equation}

After eliminating the variables $y_2$ and $z_2$ with the help of the first and third lines and replacing $y_1$ with $y$ in accordance with~\eqref{eq:dfy} we find the following presentation for the tensor product:
\[
%\mMcravr{\bfx,\bfy}\otimes_{\Qby}\mMacrvr{\bfy,\bfz}\otimes_{\Qbz} \mMcravr{\bfy,\bfz}
\mMstarxw
\cong
\Qv{\bfx,\bfw,y}/\aIv{2},
\]
where
\[
\aIv{2}
%=
%\begin{pmatrix}
%\smpev{1}(\bfw) - \smpev{1}(\bfx)
%\\
%\smpev{2}(\bfw)-\smpev{2}(\bfx)
%\\
%y_1(w_3-x_3) + w_1 w_2 - x_1 x_2
%\\
%y_1^2 - y_1(x_1+x_2) + x_1 x_2
%\end{pmatrix}
=
%\begin{pmatrix}
%\smpev{1}(\bfw) - \smpev{1}(\bfx)
%\\
%\smpev{2}(\bfw)-\smpev{2}(\bfx)
%\\
%(x_3 - w_3)\, y
%\\
%(y+x_2-w_3)(y+x_1-w_3)
%\end{pmatrix}.
\begin{pmatrix}
\smpev{1}(\bfw) - \smpev{1}(\bfx)
\\
\smpev{2}(\bfw)-\smpev{2}(\bfx)
\\
\xppo
%\xxph\, y
\\
\xppt
%y^2 + (\xxpo + \xxpt)\,y + \xxpo\xxpt
%(y+x_2-w_3)(y+x_1-w_3)
\end{pmatrix},
%=
%\begin{pmatrix}
%\smpev{1}(\bfw) - \smpev{1}(\bfx)
%\\
%\smpev{2}(\bfw)-\smpev{2}(\bfx)
%\\
%p_1
%\\
%p_2
%\end{pmatrix}
%,
\qquad
\xppo = \xxph\,y,\qquad \xppt = y^2 + 2 (\xxpo + \xxpt)\,y + 4 \xxpo\xxpt,
\]
and we used shortcut notations $\xxpi = x_i - w_3$, $i=1,2,3$.
The other two modules of the sequence~\eqref{eq:exseq} have similar presentations:
\begin{equation}
\label{eq:othtwm}
\mMthrrxw \cong\Qbxw/\aIthr,\qquad
\mMcrarxw \cong\Qbxw/ \aIcra,
\end{equation}
where
\[
\aIthr =
\begin{pmatrix}
\smpev{1}(\bfw) - \smpev{1}(\bfx)
\\
\smpev{2}(\bfw)-\smpev{2}(\bfx)
\\
\xxpo\xxpt\xxph
\end{pmatrix},
\qquad
\aIcra =
\begin{pmatrix}
\smpev{1}(\bfw) - \smpev{1}(\bfx)
\\
\smpev{2}(\bfw)-\smpev{2}(\bfx)
\\
\xxph
\end{pmatrix}
\]
and we used an asymmetric choice for the generators of $\aIthr$ (\cf \ex{eq:asmi}).

The modules~\eqref{eq:othtwm} are both generated by $1\in\Qbxw$.
The relation $\xppt=0$ appearing at the bottom of $\aIv{2}$ expresses $y^2$ in terms of lower powers of $y$, hence
$\mMstarxw$ is generated as a $\Qbxw$-module by its elements~\eqref{eq:gnrpq} and the action of the homomorphism~\eqref{eq:ohmso} is described by the matrix~\eqref{eq:fmtr} relative to these generators and the target module generators~\eqref{eq:gnssm}.

Consider two special homomorphisms defined by their action on module generators:
\begin{equation*}
\xymatrix@C=1.5cm{
\mMcrarxw \ar[r]^-{\xxfm =
\left(\begin{smallmatrix}0 \\1\end{smallmatrix}\right)
}
&
\mMstarxw
\ar[r]^-{\xxfp = \left(\begin{smallmatrix}0 &1\end{smallmatrix}\right)}
&
\mMcrarxw
}
\end{equation*}
The homomorphism $\xxfm$ is well-defined because $y\aIcra\subset \aIv{2}$. The homomorphism $\xxfp$ is well-defined, because, in view of the matrix presentation~\eqref{eq:fmtr}, it can be expressed as a composition
%
%Consider a composition of homomorphisms
$\xxfp = \xprPt\circ(\xId\otimes\hchine\otimes \xId$),
\begin{equation*}
%\label{eq:ohmso}
\xymatrix@C=2cm{
\mMstarxw
\ar[r]^-{\xId\otimes\hchine\otimes \xId}
\ar@/_1.5pc/[rr]_-{\xxfp}
&
\mMbtarxw
\ar[r]^-{\xprPt}
&
\mMcrarxw,
}
\end{equation*}
where $\xprPt$ is the projection on the second module in the decomposition~\eqref{eq:spldcra}. Since $\xxfp\circ\xxfm=\xId$,
%it follows that $\xxfm$ is injective, while
the $\Qbxw$-module $\mMstarxw$ decomposes:
\begin{equation}
\label{eq:dcmps}
\mMstarxw = \im \xxfm \oplus \ker\xxfp
%\cong \mMcrarxw \oplus \ker\xxfp
\end{equation}
Also it follows that $\xxfm$ is injective, hence
\[
\mMcrarxw \cong \im\xxfm = (y)\subset\mMstarxw,
\]
where $(y)\subset\mMstarxw$ is the $\Qbxw$-submodule generated by $y$.

The matrix form of $\xxfp$ indicates that $\xxfp(1)=0$, hence $(1)\subset\ker\xxfp$. At the same time, $1$ and $y$ generate $\mMstarxw$, hence $(1)+(y)=\mMstarxw$. Then it follows from the decomposition~\eqref{eq:dcmps} that $\ker \xxfp = (1)$, so~\eqref{eq:dcmps} becomes
\begin{equation}
\label{eq:dcmpt}
\mMstarxw = (1) \oplus (y).
\end{equation}

We have already established the isomorphism $(y)\cong \mMcrarxw$, so  it remains to show that
\begin{equation}
\label{eq:rmpr}
(1)\cong \mMthrrxw.
\end{equation}
We will use dimension counting arguments. For a \tgrddq\ module $\mM=\bigoplus_{i\in\ZZ}\mM_i$ whose grading components $\mM_i$ have finite dimension as $\IQ$-vector spaces, define \tdmq\ as
\[
\xdmq \mM = \sum_{i\in\ZZ} q^i\dim\mM_i.
\]
If $\mM$ is a free \tgrddq\ $\Qbx$-module
%of \trkq\ $\rkr(q)$ defined by~\eqref{eq:dfqrk},
then $\xdmq\mM = \xdoq\,\xrkq\mM$, where $\xrkq\mM$ is \trkq\ defined by~\eqref{eq:dfqrk}, while $\xdoq = (1-q^2)^{-3}$, because $\xnum{\bfx}=3$ and $\xdgq \bfx=2$.

According to Remark~\ref{rm:tpsbm}, the following $\Qbxw$-modules are free as $\Qbx$-modules with ranks
\[
\rank \mMthrrxw =\qnmv{3}!,\qquad
\rank \mMstarxw =\qnmv{2}^3,\qquad
\rank \mMcrarxw = \qnmv{2}.
\]
In view of the decomposition~\eqref{eq:dcmpt}, this means that
\[
\xdmq (1) = \xdoq\bbrs{\qnmv{2}^3 - q^2\qnmv{2} } = \xdoq\qnmv{3}!,
\]
the factor $q^2$ in $q^2\qnmv{2}$ being due to the fact that $\xdgq y = 2$ in accordance with~\eqref{eq:dgrs}.
Hence
\begin{equation*}
%\label{eq:qqdm}
\xdmq(1) = \xdmq \mMthrrxw.
\end{equation*}
%
%Hence $(y)\cong\mMcrarxw$ is a free $\Qbx$-module of rank two and in view of decomposition~\eqref{eq:dcmpt} $(1)$ is a free $\Qbx$-module of rank six.
%
A homomorphism
\begin{equation}
\label{eq:exsq}
\xymatrix@C=2cm{
\mMthrrxw
\ar[r]^-{\xxg = \left(\begin{smallmatrix}1 \\ 0 \end{smallmatrix}\right)}
&
\mMstarxw
%\ar[r]^-{\xxfp = \left(\begin{smallmatrix}0 &1\end{smallmatrix}\right)}
%&
%\mMcrarxw.
}
\end{equation}
is well-defined, because $\xxg(\aIthr)\subset\aIv{2}$ in view of relation $\xxpo\xxpt\xxph = \xxph\xppt - (y+\xxpo+\xxpt)\xppo$. A composition of $\xxg$ with the projection of $\mMthrrxw$ on $(1)$ is a surjective homomorphism
\[
\xymatrix{
\mMthrrxw
\ar@{->>}[r]
&
(1)
}
\]
Since $\mMthrrxw$ and $(1)$ have equal  $q$-dimensions, it follows that this is an isomorphism, and this proves~\eqref{eq:rmpr}.

Thus we established the splitting~\eqref{eq:spldthcr} with generators~\eqref{eq:gnstr}. Now we check the action of a $\aWp$ generator $\xLm$ on these generators within the module $\mMstarxw$ in presentation~\eqref{eq:mdprq}. The module~\eqref{eq:mdprq} has zero connection, hence $\xcdfm \mgenvv{\xthr}{\bfx,\bfw} =\xhLm 1 = 0$ and the module $\mMthrrxw$ with zero connection is a submodule of $\mMstarxw$ as a $\uWpxw$-module, the quotient module being $\mMcravr{\bfx,\bfw}$.

In order to find the connection of $\mMcravr{\bfx,\bfw}$, we compute the action of  $\xLm$ on its generator $y$ defined by \ex{eq:dfy} as $y = y_1 - y_2 + \xyq$, where $\xyq = - x_1 - x_2 + 2w_3\in\Qbxw$:
\begin{multline}
\label{eq:drscg}
\xcdfm \mgenvv{\xcra}{\bfx,\bfw} =\xhLm y =  \xhLm\bbrs{
(y_1 - y_2) + q
} = \sdfmyot (y_1 - y_2) + \xhLm q
\\
=
 \sdfmxot (y_1 - y_2) + \xhLm q
=  \sdfmxot \,y + \xxqp
= \sdfmxot  \mgenvv{\xcra}{\bfx,\bfw} + \xxqp,
\end{multline}
where
\[
\xxqp = -\sdfmxot\, q + \xhLm q \in\mMthrrxw\subset\mMstarxw.
\]
In deriving the formula~\eqref{eq:drscg} we used the relation
\[
\sdfmyot = \sdfmxot\qquad \mod \aIv{1},
\]
which is due to the first two rows in the presentation~\eqref{eq:fstid} of $\aIv{1}$ establishing the equality between symmetric polynomials in $x_1,x_2$ and in $y_1,y_2$.
Since $\xxqp\in\mMthrrxw\subset\mMstarxw$, the formula~\eqref{eq:drscg} shows that
$\xcdfm \mgenvv{\xcra}{\bfx,\bfw} = \sdfmxot \mgenvv{\xcra}{\bfx,\bfw}$ within the quotient module
$\mMcrarxw\cong \mMstarxw/\mMthrrxw$, hence this quotient has a connection $\bfsdfot$ defined by \ex{eq:cpot}.
\end{proof}
%%%%%%%%%%%%%%%%%%%%%%%%%%%%%%%%%%%%%%%%%%%%%%%%%%%
%\begin{SpecialPar}
%\begin{proof}[Proof of Proposition~\ref{pr:lcinid}]
%The relation~\eqref{eq:lcinid} holds iff
%there exist the polynomials $\xtqo,\xtqt\in\Qbxw$ such that
%\begin{equation}
%\label{eq:cndinid}
%r_0 + yr_1 = \xtqo p_1 + \xtqt p_2,
%\end{equation}
%For any polynomial $p\in\aAe$ one can always choose a different set of polynomials
%\[
%\xtqo' = \xtqo +  p p_2,\qquad \xtqt' = \xtqt - pp_1,
%\]
%still satisfying \ex{eq:cndinid}, hence one can assume that $\xdg_y\xtqo\leq 1$. Then $\deg_y\xtqt=0$, so
%$\xtqt\in\Qbxw$, while $\xtqo$ has the form
%\[
%\xtqo = -\xqo y +\xqz.
%\]
%Substituting this expression into \ex{eq:cndinid}, we find
%\[
%0 = r_0 + yr_1 - \xtqo p_1 - \xtqt p_2 = (\xpch\xqo -\xtqt) y^2
%+ (r_1 - \xpch\xqz-\xpco\xtqt )y + (r_0 - \xpct\xtqt).
%\]
%Since all coefficients in the \rhs belong to $\Qbxw$, this equation holds iff all of them are equal to zero, and this implies the relations~\eqref{eq:rlrrzo}.
%\end{proof}
%\end{SpecialPar}

%%%%%%%%%%%%%%%%%%%%%%%%%%%%%%%%%%%%%%%%%%%%%%%%
\begin{proof}[Proof of Theorem~\ref{th:rdmthr}]
As we have observed, the \qiso~\eqref{eq:fhteq} is equivalent to~\eqref{eq:mrsht}. In fact, it is easier to prove the same relation for the inverse generators:
\begin{equation}
\label{eq:mrshtz}
\ctmdvbxw{ \sggio \sggit \sggio } \eqsmout
\Fothxw\big( \ctmdvbxw{ \sggio \sggit \sggio }\big).
\end{equation}
The \lhs of this relation is presented by a complex of \Wpeq\ $\Qbxw$-modules:
\begin{multline}
\label{eq:bgthcm}
\ctmdvbxw{ \sggio \sggit \sggio }
=
\boxed{
\vcenter{
\xymatrix@C=1.5cm@R=1.5cm{
&
\mMcprrxw
\ar[r] \ar[rd]_(0.25){-1}
&
\mMcrarxwd
\ar[dr]
\\
\mMstarxw
\ar[ur] \ar[r] \ar[dr]
&
\mMbtarxw
\ar[ru]^(0.25){-1}
\ar[rd]
&
\mMacrrxwm
\ar[r]
&
\mMtharxw
\\
&
\mMiprrxw
\ar[ru] \ar[r]^-{-1}
&
\mMcrarxwu
\ar[ur]
}
}
}\;,
\end{multline}
where $\mMacrrxwm = \mMthavr{\bfx,\bfy}\oty\mMacrvr{\bfy,\bfz}\otz\mMthavr{\bfz,\bfw}$. All modules in this complex are tensor products of elementary \Sgl\ \bmdl s over the intermediate algebras $\Qby$ and $\Qbz$ and, in accordance with definitions~\eqref{eq:bcmps},~\eqref{eq:hmio} and~\eqref{eq:mhoms}, all arrows represent either the homomorphisms $1\otimes  1\otimes  1$ (unmarked arrows) or $-1\otimes 1\otimes 1$ (arrows marked by  $-1$).
%, the latter  being marked as $-1$.

According to lemmas~\ref{lm:thsplt} and~\ref{lm:thsplo},
there are chain maps
\begin{equation}
\label{eq:sqtwo}
\xymatrix{
\aBb \ar[r]
&
\ctmdvbxw{ \sggio \sggit \sggio }
\ar[r]
&
\aCb,
}
\end{equation}
where
\begin{equation}
\label{eq:cmpas}
\aBb \wdeq
\boxed{
\vcenter{
\xymatrix@C=1.5cm@R=1.5cm{
&
\mMcprrxw
\ar[r] \ar[rd]_(0.25){-1}
&
\mMcrarxw
\ar[dr]
\\
\mMthrrxw
\ar[ur] \ar[r] \ar[dr]
&
\mMcrarxw
\ar[ur]^(0.25){-1} \ar[dr]
&
\mMacrrxw
\ar[r]
&
\mMtharxw
\\
&
\mMiprrxw
\ar[ru] \ar[r]^-{-1}
&
\mMcrarxw
\ar[ur]
}
}
}
\end{equation}
all unmarked arrows representing homomorphisms $1$, and
\[
\aCb = \boxed{\xymatrix{
\mcnbpbpot{\mMcravr{\bfx,\bfw}}
\ar[r]^-{\xId}
&
\mcnbpbpot{\mMcravr{\bfx,\bfw}}
}
}\;,
\]
such that after the application of the \forfun\ $\Fffq$ the sequence~\eqref{eq:sqtwo} becomes a part of an exact triangle. Since the complex $\aCb$ is obviously contractible, by Theorem~\ref{pr:qsbc}
 there is a \qiso\
$
\ctmdvbxw{ \sggio \sggit \sggio } \smout \aBb.
$

An obvious change of generators in the sum of modules $\mMcrarxw\oplus \mMcrarxw$ appearing in the third column of the complex~\eqref{eq:cmpas} allows us to present it in the following form:
\begin{equation*}
\aBb
\wdcng
\boxed{
\vcenter{
\xymatrix@C=1.5cm@R=1.5cm{
&
\mMcprrxw
\ar[r] \ar[rd]_(0.25){-1}
&
\mMcrarxw
\ar[dr]
\\
\mMthrrxw
\ar[ur] \ar[r] \ar[dr]
&
\mMcrarxw
%\ar[ur]^(0.25){-1}
\ar[dr]
&
\mMacrrxw
\ar[r]
&
\mMtharxw
\\
&
\mMiprrxw
\ar[ru] \ar[r]^-{-1}
\ar[uur]^(0.2){-1}
&
\mMcrarxw
%\ar[ur]
}
}
}
\end{equation*}
This complex has a cone presentation: $\aBb\cong \boxed{\aDb\rightarrow\aEb}$, where
\begin{equation}
\label{eq:thbcmp}
\aDb
\wdeq
\boxed{
\vcenter{
\xymatrix@C=1.75cm@R=0.75cm{
&
\mMcprrxw
\ar[r] \ar[rdd]_(0.25){-1}
&
\mMcrarxw
\ar[dr]
\\
\mMthrrxw
\ar[ur] \ar[dr]
&
&
&
\mMtharxw
\\
&
\mMiprrxw
\ar[r]
\ar[uur]^(0.2){-1}
&
\mMacrrxw
\ar[ur]
}
}
}
\end{equation}
while
\[
\aEb = \boxed{
\xymatrix{
\mMcrarxw
\ar[r]^-{\xId}
& \mMcrarxw
}
}\;.
\]
Since $\aEb$ is contractible, there is a homotopy equivalence $\aBb \smout \aDb$, hence
\begin{equation}
\label{eq:thrcb}
\ctmdvbxw{ \sggio \sggit \sggio } \eqsmout \aCb.
\end{equation}
It is easy to verify that the complex~\eqref{eq:thbcmp} is symmetric under the transposition of indices $1\leftrightarrow 3$ in the variables $\bfx$ and $\bfw$, that is, $\Fothxw(\aCb)\cong\aCb$, hence the \qiso~\eqref{eq:mrshtz} follows from \ex{eq:thrcb}.
\end{proof}

\section{Link invariant as an object of the derived category of $\uWplc$-modules}

\subsection{Categorification of the braid group with \xstv s}
\label{ss:strvs}

Our immediate goal is to prove the \qiso~\eqref{eq:slprop} and its analogs. We begin by studying the sliding properties of a single \brstv\ $\vbr$.
Consider the functor
\[
%\Ffaxi \colon \uWp\sctgmod\longrightarrow \uWpXx\sctgmod
\xymatrix@C=1.25cm{
\uWpx\sctgmod \ar[r]^-{\Ffaxi} & \uWpXx\sctgmod,
}
\]
which defines the action of $\vbr$ on a $\uWpx$-module $\mMx$ as being equal to that of $x_i$. In other words,
\begin{equation}
\label{eq:fftp}
\Ffaxi = \dmmy\otv{\bar{x}_i}\Dlv{\vbr-x_i}.
\end{equation}
The functor $\Ffaxi$ extends along the chain of categories~\eqref{eq:chct}  to  functors
\[
\xymatrix@C=1.25cm{
\ctWx \ar[r]^-{\Ffaxi} & \ctxWx,
}
\qquad
\xymatrix@C=1.25cm{
\ctCh(\ctWx)
 \ar[r]^-{\Ffaxi} &
\ctCh(\ctxWx),
}
\qquad
\xymatrix@C=1.25cm{
\ctDr(\ctWx)
 \ar[r]^-{\Ffaxi} &
\ctDr(\ctxWx).
}
\]
%
%Consider the functor $\Ffaxi\colon\ctWx\rightarrow\ctxWx$ which defines the action of $\vbr$ on a $\uWpx$-module $\mMx$ as being equal to that of $x_i$. In other words,
%\begin{equation}
%\label{eq:fftp1}
%\Ffaxi = \dmmy\otv{\bar{x}_i}\Dlv{\vbr-x_i}.
%\end{equation}
\begin{theorem}
\label{th:slvbr}
For any braid $\brr$ there is an isomorphism in the category $\ctDr(\ctxWxy)$
\begin{equation}
\label{eq:ssqiso}
\Ffaxi(\ctmvbxy{\brr}) \eqsmout \Ffav{y_{\brrws(i)}}(\ctmvbxy{\brr}),
\end{equation}
where $\brrws\in \mathrm{S}_n$ is the permutation associated with $\brr$.
\end{theorem}

We will prove this theorem by establishing \qiso~\eqref{eq:ssqiso} for elementary braids $\sggii$  in the category $\ctCh(\ctxWxy)$ and then use multiplicativity property~\eqref{eq:smbrpr} in order to extend it to all braids.

\begin{lemma}
\label{lm:smpsl}
The \qiso~\eqref{eq:ssqiso} holds for the elementary braid $\sggii$.
%the brackets~\eqref{eq:bcmps}.
\end{lemma}
\begin{proof}
We will prove the \qiso\ for a 2-strand elementary braid $\brr = \sggi$, so that $\xnum{\bfx}=\xnum{\bfy}=2$.
%The case of $\brw=\sggn$ is proved similarly, and
The locality of the formulas~\eqref{eq:bcmps} extends this result to the general case of $\nst$-strand elementary braid $\sggii$.
Thus we want to establish a \qiso
\begin{equation}
\label{eq:smlqs}
\Ffav{x_1}(\ctmvbxy{\sggi}) \eqsmout \Ffav{y_2}(\ctmvbxy{\sggi}),
\end{equation}
where
\[
\ctmvbxy{\sggi} = \boxed{
\mMcrrxy \xrightarrow{\;\;\hchine\;\;} \mMprrxy
}\;.
\]
The case of $\vbr=x_2$ versus $\vbr=y_1$ is treated similarly.

The \qiso~\eqref{eq:smlqs} is established by the
%following diagram:
diagram of Fig.~\ref{fg:ssv},
\begin{figure}[h]
\[
\xymatrix@C=0.5cm@R=2cm{
%**[l]
\Ffav{x_1}(\ctmvbxy{\sggi})
&
&&&
\Ffav{x_1}(\mMcrrxy)
\ar[rr]^-{1}
&&&
\hspace{-1cm}\Ffav{x_1}(\mMprrxy)
\\
\aAb
\ar[u]^-{\fqsoA}
\ar[d]_-{\fisoAB}
&
\hspace{-1.25cm}
%\mcnaivv{\maMcrrxy}{x_1}
\mcnv{\maMcrrxy}{\bfca}
\ar[rrr]^-{\smmatr{\vbr-x_1\\1}}
\ar[d]_-{1}
&&&
\maMcrrxy \oplus
%\mcnaivv{\maMprrxy}{x_1}
\mcnv{\maMprrxy}{\bfca}
\ar[rrr]^-{\smmatr{1 & x_1-\vbr}}
\ar[d]_-{\smmatr{1 & x_1 - y_2 \\ 0 & 1}}
\ar[u]^-{\ssmmatr{ 1 & 0}}
&&&
\maMprrxy
\ar[d]_-{1}
\ar[u]^-{1}
\\
\aBb
\ar[d]_-{\fqsoB}
&
\hspace{-1.25cm}
%\mcnaivv{\maMcrrxy}{x_1}
\mcnv{\maMcrrxy}{\bfca}
\ar[rrr]^-{\smmatr{\vbr-y_2\\1}}
&&&
\mcnv{\maMcrrxy \oplus
%\mcnaivv{\maMprrxy}{x_1}
%\mcnv{\maMprrxy}{\bfca}
\maMprrxy}{\bfcA}
\ar[rrr]^-{\smmatr{1 & y_2-\vbr}}
\ar[d]_-{\ssmmatr{1 & 0}}
&&&
\maMprrxy
\ar[d]^-{1}
\\
\Ffav{y_2}(\ctmvbxy{\sggi})
&
&&&
\Ffav{y_2}(\mMcrrxy)
\ar[rr]^-{1}
&&&
\hspace{-0.5cm}\Ffav{y_2}(\mMprrxy)
}
\]
\caption{Sliding of a strand variable}
\label{fg:ssv}
\end{figure}
%where
in which
we denote $\prba \mMxy = \mMxy\otimes \Qv{\vbr}$.
Each row in this diagram represents a complex in $\ctCh(\ctXWxy)$ (the rows have to be `boxed' in our notations) and we used the following abbreviations for $\aWp$ connections: $\bfca = \bsdfv{\vbr}{x_1}$,
\[
\bfcA =
\begin{pmatrix}
0 & (x_1 - y_2)\,\bigl(\bsdfv{\vbr}{x_1} -\bsdfv{x_1}{y_2}\bigr)
\\
0 & \bsdfv{\vbr}{x_1}
\end{pmatrix}
\]

That $\fisoAB$ is an isomorphism in $\ctCh(\ctXWxy)$ is established by inspection. That $\fqsoA$ and $\fqsoB$ are \qiso s can be established similarly, but it is easier to observe this by explaining the origin of complexes $\aAb$ and $\aBb$.

There is an obvious \qiso\ $\Kzcv{\vbr-x}\eqsmout\Dlv{\vbr,x}$ within the category $\ctCh( \ctvWv{\vbr}{x})$ between a Koszul complex
\[
\Kzcv{\vbr - x} =
\boxed{
\xymatrix@C=1.5cm{
\mcnv{\mQax}{ \bsdfv{\vbr}{x} }
\ar[r]^-{\vbr-x}
&
\mQax
}
}
\]
and the diagonal \bmdl\ $\Dlv{\vbr,x}=\mQax/(\vbr-x)$, where, as usual, $\mQax$ is $\Qv{\vbr,x}$ viewed as a module over $\uWpfv{\vbr}{x}$.
The complex $\aAb$ is the Koszul resolution of the relation $a=x_1$ in the complex $\Ffav{x_1}(\ctmvbxy{\sggi})$:
\[
\aAb = \ctmvbxy{\sggi} \otv{x_1} \Kzcv{\vbr-x_1},\qquad
\aAb \eqsmout \Ffav{x_1}(\ctmvbxy{\sggi}).
\]
A similar \qiso\
holds between $\Ffav{y_2}(\ctmvbxy{\sggi})$ and $\aBtb=\ctmvbxy{\sggi} \otv{y_2} \Kzcv{\vbr-y_2}$. However,  $\aBtb$ is not isomorphic to $\aAb$ as a complex of $\uWpaxy$-modules, unless we `tweak' the $\aWp$-connection in $\aBtb$. This tweaking transforms $\aBtb$ into $\aBb$, while keeping it \qisc\ to $\Ffav{y_2}(\ctmvbxy{\sggi})$.
%
%$\aBb$ and
%$\aBtb=\Ffav{y_2}(\ctmvbxy{\sggi}) \otv{y_2} \Kzcv{\vbr-y_2}$
%if we forget about the $\aWp$-module structure. However $\aBtb$ is not isomorphic to $\aAb$ as a complex of $\uWpaxy$-modules, so we have to `tweak' the $\aWp$-connection within $\aBtb$ in order to transform it into $\aBtb$ which would be isomorphic to $\aAb$. This tweaking does not destroy the \qiso\ $\fqsoB$.
\end{proof}

%\begin{lemma}
%For any object $\mMxy$ of $\ctWxy$ and for any object $\mNyz$ of $\ctWyz$, there is a \qiso
%\[
%(\Ffav{y_i} \mMxy) \otdry \mNyz \simeq \mMxy\otdry (\Ffav{y_i}\mNyz).
%\]
%\end{lemma}
%\begin{proof}
%Let $\rsPb(\mMxy)$ be a resolution of $\mMxy$ which is a chain complex of $\uWpxy$-modules, each of which is projective as a $\Qby$-module. Then $\Ffav{y_i}\rsPb(\mMxy)$ is a similar resolution of $\Ffav{y_i}\mMxy$ and we have a sequence of \qiso s, which prove the lemma:
%\[
%(\Ffav{y_i} \mMxy) \otdry \mNyz
%\simeq \Ffav{y_i}\rsPb(\mMxy)\oty\mNyz \cong\rsPb(\mMxy)\oty\Ffav{y_i}\mNyz\simeq
%\mMxy\otdry\Ffav{y_i}\mNyz.
%\]
%\end{proof}
\begin{lemma}
\label{lm:smpsli}
\qiso~\eqref{eq:ssqiso} holds for the elementary braid $\sggni$.
\end{lemma}
\begin{proof}
Again, it is sufficient to prove the lemma for the 2-strand braid generator $\sggn$.
Consider the sequence of \qiso s:
\begin{multline*}
\Ffav{x_1}(\ctmvbxy{\sggn})\oty\ctmvbyz{\sggi} \eqsmout
\Ffav{x_1}\Dlv{\bfx,\bfz}\cong\Ffav{z_1}\Dlyz \eqsmout
\ctmvbxy{\sggn}\oty\Ffav{z_1}(\ctmvbyz{\sggi})
\\
\eqsmout\ctmvbxy{\sggn}\oty\Ffav{y_2}(\ctmvbyz{\sggi})
\cong(\Ffav{y_2}\ctmvbxy{\sggn})\oty\ctmvbyz{\sggi}.
\end{multline*}
Finally, we apply the tensor multiplication $\dmmy\otz\ctmvbzw{\sggn}$ to the first and last complex in this chain and use the \qiso\ $\ctmvbyz{\sggi}\otz\ctmvbzw{\sggn}\eqsmout \Dlyw$ as well as  $\dmmy\oty\Dlyw$ being the identity functor.
\end{proof}

\begin{proof}[Proof of Theorem~\ref{th:slvbr}]
We prove the theorem by induction over the length of the minimal \brwd\ presentation of a braid. Lemmas~\ref{lm:smpsl} and~\ref{lm:smpsli} prove \qiso~\eqref{eq:ssqiso} for elementary braids. Suppose that the theorem holds for braids of minimal length $k$. If a braid $\brr$ has minimal length $k+1$ then it can be presented as a composition of a length $k$ braid $\brr_1$ and an elementary braid $\brr_2$. Thus we have a sequence of isomorphisms and \qiso s:
\begin{equation*}
%\begin{multline*}
\begin{split}
\Ffaxi\ctmvbxz{\brr}
&\eqsmout
\Ffaxi\ctmvbxy{\brr_1\brr_2} \cong \Ffaxi\bigl(\ctmvbxy{\brr_1}\oty\ctmvbyz{\brr_2}\bigr)
\cong \bigl (\Ffaxi\ctmvbxy{\brr_1}\bigr)\oty \ctmvbyz{\brr_2}
\\
& \eqsmout
\bigl (\Ffav{y_{\brrws_1(i)}}\ctmvbxy{\brr_1}\bigr)\oty \ctmvbyz{\brr_2}
\cong\ctmvbxy{\brr_1}\oty \bigl(\Ffav{y_{\brrws_1(i)}}\ctmvbyz{\brr_2}\bigr)
\\
&\eqsmout\ctmvbxy{\brr_1}\oty \bigl(\Ffav{z_{\brrws_2\brrws_1(i)}}\ctmvbyz{\brr_2}\bigr)
%%%%%%%%%%%
\eqsmout \Ffav{z_{\brrws(i)}}(\ctmvbxz{\brr})
\end{split}
%\end{multline*}
\end{equation*}
\end{proof}

Theorem~\ref{th:totsld} is an obvious corollary of Theorem~\ref{th:slvbr}.

\subsection{Proof of Theorem~\ref{th:rnind}}
\label{ss:prthrnind}
%
%\begin{theorem}
%\label{th:rnind}
%The action of the renaming functor $\FfukXl$ on the Hochschild homologies of braid brackets does not depend on the choice of braid strands $k_i$ representing link components: for two lists of choices $\bfk$ and $\bfk'$ there is a \qiso
%\[
%\FfukXl\bigl(\Hhomxy(\actmvbxy{\brr}) \bigr)\eqsmout
%\FfpukXl\bigl(\Hhomxy(\actmvbxy{\brr}) \bigr)
%\]
%\end{theorem}
Theorem~\ref{th:rnind} is an obvious corollary of the following theorem:
\begin{theorem}
\label{th:slstvlc}
Suppose that $i$-th and $j$-th strands of a braid $\brr$ belong to the same component of the link constructed by circular closure. Then there is a \qiso\ in the category $\ctDr(\ctWX)$:
\begin{equation}
\label{eq:qslbrcl}
\bigl(\Ffaxi \ctmvbxy{\brr}\bigr) \otdrxy \Dlxy \eqsmout
\bigl(\Ffaxj \ctmvbxy{\brr}\bigr) \otdrxy \Dlxy
\end{equation}
\end{theorem}
Indeed, this theorem means that a strand variable can slide around the closed braid, thus moving between all strands belonging to the same link component.

We prove Theorem~\ref{th:slstvlc} with the help of the following lemma:
\begin{lemma}
For any object $\mMxy$ of $\ctWxy$ and any object $\mNyz$ of $\ctWyz$ there is a \qiso
\[
(\Ffav{y_i} \mMxy) \otdry \mNyz \simeq \mMxy\otdry (\Ffav{y_i}\mNyz).
\]
\end{lemma}
\begin{proof}
Let $\rsPb(\mMxy)$ be a resolution of $\mMxy$ which is a chain complex of $\uWpxy$-modules, each of which is projective as a $\Qby$-module. Then $\Ffav{y_i}\rsPb(\mMxy)$ is a similar resolution of $\Ffav{y_i}\mMxy$, and there is a sequence of \qiso s, which proves the lemma:
%\begin{multline*}
\begin{equation*}
\begin{split}
(\Ffav{y_i} \mMxy) \otdry \mNyz
& \simeq \Ffav{y_i}\rsPb(\mMxy)\oty\mNyz
\\
&\cong\rsPb(\mMxy)\oty\Ffav{y_i}\mNyz
%\\
\simeq
\mMxy\otdry\Ffav{y_i}\mNyz.
%\end{multline*}
\end{split}
\end{equation*}
\end{proof}
\begin{proof}[Proof of Theorem~\ref{th:slstvlc}]
Since $\brrws$ performs a cyclic permutation of strand indices corresponding to the same link component, it is sufficient to prove the \qiso~\eqref{eq:qslbrcl} in the case when $j = \brrws(i)$. The latter is proved by the following sequence of \qiso s:
%\begin{multline*}
\begin{equation*}
\begin{split}
\bigl(\Ffaxi \ctmvbxy{\brr}\bigr) \otdrxy \Dlxy \eqsmout
\bigl(\Ffav{y_{\brrws(i)}}\ctmvbxy{\brr}\bigr) \otdrxy \Dlxy \eqsmin
\ctmvbxy{\brr}\otdrxy \bigl(\Ffav{y_{\brrws(i)}}\Dlxy\bigr)
\\
\cong
\ctmvbxy{\brr}\otdrxy \bigl(\Ffav{x_{\brrws(i)}}\Dlxy\bigr)
\eqsmin
\bigl(\Ffav{x_{\brrws(i)}} \ctmvbxy{\brr}\bigr) \otdrxy \Dlxy
\end{split}
\end{equation*}
%\end{multline*}
\end{proof}

%
%
%Theorem~\ref{th:slstvlc} has two obvious corollaries. The first one is Theorem~\ref{th:totsld}. The second one
%\begin{corollary}
%Strand variables slide simultaneously through the braid categorification complex:
%\[
%asdf
%\]
%\end{corollary}

\subsection{Markov move invariance}
\label{ss:mmi}
\begin{theorem}
\label{th:scnd}
There exists a unique map $\lctmv{\dmmy}$ which makes the upper right triangle in the following diagram commutative:
\begin{equation}
%\label{eq:mndiag}
\mvcn{
\xymatrix@C=4pc@R=5pc{
\brgrn
\ar[d]_-{\actmvbxy{\dmmy}}
%\ar[drr]^-{\lctdmv{\dmmy}}
&
\brgrns
\ar@{_{(}->}[l]
\ar[r]^-{\fcl}
\ar[dr]_-{\lctdmv{\dmmy}}
&
\lnksm
\ar[d]^-{\lctmv{\dmmy}}
%\ar[dr]^-{\Htg(\dmmy)}
\\
\ctDr\bigl(\ctXWxy\bigr)
\ar[r]_-{\Hhomxy(\dmmy)}
&
\ctWX
\ar[r]_-{\FfukXl}
&
\ctWl
%\ar[r]_-{\Hml(\dmmy)}
%&
%\uWplcgmd
}
}
\end{equation}
Moreover, if a framed link $\xL'$ is isotopic to a framed link $\xL$, except that the $i$-th component of $\xL'$ has an extra unit of framing, then their brackets are related by the connection shift endofunctor:
\[
\ctmv{\xLp} \simeq\shvv{\ctmv{\xL}}{\bsdfpli},
\]
where $\bsdfpx$ is defined in~\eqref{eq:dfppr}.
\end{theorem}
\begin{proof}
The claim of this theorem follows from the invariance of $\lctdmv{\dmmy}$ under Markov moves up to the connection shift functor. We prove this invariance in the next two subsections.
\end{proof}
\subsubsection{First Markov move}
\begin{theorem}
\label{th:fmm}
For two $\nst$-strand braids $\brr_1$ and $\brr_2$ there is an isomorphism in $\ctWl$:
\begin{equation}
\label{eq:qhhm}
\lctdmv{\brr_1\brr_2}\eqsmout \lctdmv{\brr_2\brr_1}.
\end{equation}
\end{theorem}
This theorem follows easily from the next lemma. Consider the functor $\dmmy\otdrxy\Dlxy$ in the diagram~\eqref{eq:hhmdg}.
\begin{lemma}
There is a \qiso\ in the category $\ctCh\bbrs{\ctWX}$:
\begin{equation}
\label{eq:slhbr}
\FfXx\ctmvbxy{\brr_1\brr_2}\otdrxy\Dlxy\eqsmout\Ffbrox \ctmvbxy{\brr_2\brr_1}
\otdrxy\Dlxy.
\end{equation}
\end{lemma}
\begin{proof}
Consider a sequence of \qiso s:
%\begin{multline*}
\begin{equation*}
\begin{split}
\Ffbroz\ctmvbzy{\brr_2\brr_1}
&\eqsmout
\bbrs{\Ffbroz\ctmvbzw{\brr_2}}\otw\ctmvbwy{\brr_1}
\\
&\eqsmout
\Bbrs{\Ffbroz\ctmvbzw{\brr_2}\otimes\ctmvbxy{\brr_1}}
\otxw \Dlxw.
%\end{multline*}
\end{split}
\end{equation*}
We substitute this \qiso\ into the \rhs of \ex{eq:slhbr} and rename some variables:
\begin{equation}
\label{eq:veron}
\Ffbroz \ctmvbzy{\brr_2\brr_1}\otdryz\Dlyz \eqsmout
\Bbrs{\Ffbroz\ctmvbzw{\brr_2}\otimes\ctmvbxy{\brr_1}}\otdrxyzw
\bbrs{\Dlxw\otimes\Dlyz}.
\end{equation}
Next consider the following sequence of \qiso s
\begin{multline*}
\FfXx\ctmvbxw{\brr_1\brr_2} \eqsmout
\bbrs{\FfXx\ctmvbxz{\brr_1}}\otz\ctmvbzw{\brr_2} \eqsmout
\bbrs{\Ffbroz\ctmvbxz{\brr_1}}\otz \ctmvbzw{\brr_2}
\\
\eqsmout
\ctmvbxz{\brr_1}\otz\bbrs{\Ffbroz\ctmvbzw{\brr_2}}
\eqsmout
\Bbrs{\ctmvbxy{\brr_1}\otz\bbrs{\Ffbroz\ctmvbzw{\brr_2}}}
\otyz\Dlyz.
\end{multline*}
We substitute it into the \lhs of \ex{eq:slhbr} and rename some variables:
\begin{equation}
\label{eq:vertw}
\FfXx\ctmvbxw{\brr_1\brr_2}\otdrxw\Dlxy
\eqsmout
\Bbrs{\ctmvbxy{\brr_1}\otz\bbrs{\Ffbroz\ctmvbzw{\brr_2}}}\otdrxyzw
\bbrs{\Dlxw\otimes\Dlyz}.
\end{equation}
Comparing the \rhs of equations~\eqref{eq:veron} and~\eqref{eq:vertw}, we come to \ex{eq:slhbr}.
\begin{proof}[Proof of Theorem~\ref{th:fmm}]
Applying the inner homology $\HmWX$ of diagram~\eqref{eq:hhmdg} to both sides of~\eqref{eq:slhbr} we get the \xqiso\ of \Hchshs
\[
\Hhomxy\bbrs{\FfXx\ctmvbxy{\brr_1\brr_2}}
\sim
\Hhomxy\bbrs{\Ffbrox \ctmvbxy{\brr_2\brr_1}}
\]
in $\ctWX$. Applying the relabelling functor $\FfukXl$ to both sides we get the isomorphism~\eqref{eq:qhhm}.
\end{proof}
%
%
%\begin{multline*}
%\FfXx\ctmvbxw{\brr_1\brr_2} \eqsmout
%\bbrs{\FfXx\ctmvbxy{\brr_1}}\oty\ctmvbyw{\brr_2}
%\eqsmout \Bbrs{\FfXx\ctmvbxy{\brr_1}\otimes\ctmvbzw{\brr_2}}
%\otdryz \Dlyz.
%\end{multline*}
%Substituting this \qiso\ into the \lhs of \ex{eq:slhbr} and its analog into the \rhs of \ex{eq:slhbr} and renaming some variables we get the following \qiso s:
%\begin{align}
%\FfXx\ctmvbxw{\brr_1\brr_2}\otdrxw\Dlxw&\eqsmout
%\Bbrs{\FfXx\ctmvbxy{\brr_1}\otimes\ctmvbzw{\brr_2}}
%\otdrxyzw\bbrs{\Dlyz\otimes\Dlxw}
%\\
%\Ffbroz\ctmvbzy{\brr_2\brr_1}\otdryz\Dlyz&\eqsmout
%\Bbrs{\Ffbroz\ctmvbzw{\brr_2}\otimes\ctmvbxy{\brr_1}}
%\otdrxyzw\bbrs{\Dlyz\otimes\Dlxw}
%\end{align}
%
%\begin{multline*}
%\FfXx\ctmvbxz{\brr_1\brr_2} \eqsmout
%\bbrs{\FfXx\ctmvbxy{\brr_1}}\oty\ctmvbyz{\brr_2} \eqsmout
%\bbrs{\Ffbroy\ctmvbxy{\brr_1}}\oty \ctmvbyz{\brr_2}
%\\
%\eqsmout
%\ctmvbxy{\brr_1}\oty\bbrs{\Ffbroy\ctmvbyz{\brr_2}}
%\end{multline*}
%
%\[
%\FfXx\ctmvbxz{\brr_1\brr_2}\otdrxz\Dlxz \eqsmout
%\bigl(\FfXx\ctmvbxy{\brr_1}\oty\ctmvbyz{\brr_2}\bigr) \otdrxz\Dlxz \eqsmout
%\]
\end{proof}

\subsubsection{Second Markov move}

For a $\nst_1$-strand braid $\brr_1$ and a $\nst_2$-strand braid $\brr_2$ let $\xcnb{\brr_1}{\brr_2}$ denote the $(\nst_1+\nst_2)$-strand braid constructed by placing $\brr_1$ and $\brr_2$ side by side. In other words, $\xcnb{\brr_1}{\brr_2}$ is the result of applying the injection $\brgrv{\nst_1}\times\brgrv{\nst_2}\hookrightarrow\brgrv{\nst_1+\nst_2}$ to the pair $(\brr_1,\brr_2)$.

Let $\idbrn$ denote the $\nst$-strand identity braid. The second Markov move relates an $\nst$-strand braid $\brr$ with two $(\nst+1)$-strand braids
\begin{equation}
\label{eq:prbmp}
\brrpmo=(\xcnb{\sggnpm}{\idbrno})(\xcnb{\idbro}{\brr}),
\end{equation}
where $\sggn$ is the elementary 2-strand braid and $\sggi$ is its inverse. The cyclic closures of all three braids are isotopic links, and the first strand of $\brr$ as well as the first two strands of $\brrpmo$ belong to the same link component. We refer to it as the first component, so that its component variable is $\lc_1$. The framing of the first link component in the closure of $\brrpmo$ differs by $\pm 1$ from the one within the closure of $\brr$.

\begin{theorem}
\label{th:mmtwo}
For any $\nst$-strand braid $\brr$ there is a \xqiso\ in the category $\ctWX$:
\[
\lctdmv{\brrpmo}\simeq \lctdmv{\brr}\shv{\mp\bsdfpli}.
\]
\end{theorem}

The proof of this theorem is based on the local version of the second Markov move:
\begin{lemma}
\label{lm:smmtw}
Let $\xnum{\bfx}=\xnum{\bfy}=2$.
There is a \qiso\ in the category $\ctCh(\ctWxyt)$
\begin{equation}
\label{eq:frdm}
\ctmvbxy{ \sggnpm } \otdrxyo\Dlxyo
\eqsmout \Dlxyt\shv{\mp \bsdfpxt}
\end{equation}
\end{lemma}
\begin{proof}[Proof of Theorem~\ref{th:mmtwo}]
In this case it is convenient to employ the presentation of the bracket $\lctdmv{\dmmy}$ which does not use \xstv s:
\begin{equation}
\label{eq:lbrwa}
\lctdmv{\dmmy} = \Ffukxl \Bbrs{ \tHhomtxy\bbrs{\ctmvbxy{\dmmy}}}
=\Ffukxl\Bbrs{\HmWx\bbrs{\ctmvbxy{\dmmy}\otdruxy\Dlxy}}.
\end{equation}
The first two strands of the closure of $\brrpmo$ are parts of the same link component, so we can assume that the choice $\bfk$ of braid strands representing link components does not include the first strand. Hence applying the formula~\eqref{eq:lbrwa} to $\brrpmo$ we may forget $x_1$ immediately after taking the derived tensor product:
\begin{equation}
\label{eq:lbmpo}
\lctdmv{\brrpmo} = \Ffukxl\Bbrs{\HmWxp\bbrs{\ctmvbxy{\brrpmo}\otdruxxpy\Dlxy}},
\end{equation}
where $\bfx' = x_2,\ldots,x_{n+1}$.

The formula~\eqref{eq:prbmp} for $\brrpmo$ implies the tensor product presentation of its bracket:
\[
\ctmvbxy{\brrpmo} \eqsmout \ctmvr{\sggnpm}{x_1,x_2,y_1,z} \otsz
\ctmvr{\brr}{z,\bfx'',\bfy'},
\]
where $\bfx'' = x_3,\ldots,x_{n+1}$ and $\bfy'=y_2,\ldots,y_{n+1}$. Therefore,
%\begin{multline*}
\[
\begin{split}
\ctmvbxy{\brrpmo}\otdruxxpy\Dlxy
&\eqsmout
\bbrs{
\ctmvr{\sggnpm}{x_1,x_2,y_1,z} \otsz \ctmvr{\brr}{z,\bfx'',\bfy'}
}
\otdruxxpy
\bbrs{\Dlxyo\otimes\Dlxyp}
\\
&
\eqsmout
\Bbrs{
\bbrs{
\ctmvr{\sggnpm}{x_1,x_2,y_1,z}\otdrxyo\Dlxyo
}
\otsz
\ctmvr{\brr}{z,\bfx'',\bfy'}
}
\otdruxpyp
\Dlxyp
\end{split}
\]
%\end{multline*}
According to Lemma~\ref{lm:smmtw}
\[
\ctmvr{\sggnpm}{x_1,x_2,y_1,z}\otdrxyo\Dlxyo \eqsmout
\Dlv{x_2;z}\shv{\pm\bsdfpxt},
\]
while $\Dlv{x_2;z}\otsz\ctmvr{\brr}{z,\bfx'',\bfy'} \cong \ctmvr{\brr}{\bfx',\bfy'}$, hence
\[
\ctmvbxy{\brrpmo}\otdruxxpy\Dlxy \eqsmout
\bbrs{\ctmvr{\brr}{\bfx',\bfy'}\otdruxpyp\Dlxyp}\shv{\pm\bsdfpxt}
\]
Substituting this relation to \ex{eq:lbmpo} we find
%\begin{multline*}
\[
\begin{split}
\lctdmv{\brrpmo}
&\eqsmout
\Ffukxl\Bbrs{\HmWxp\bbrs{\ctmvr{\brr}{\bfx',\bfy'}\otdruxpyp\Dlxyp}\shv{\pm\bsdfpxt} }
\\
&\cong
\Ffukxl\Bbrs{\HmWxp\bbrs{\ctmvr{\brr}{\bfx',\bfy'}\otdruxpyp\Dlxyp} }\shv{\pm\bsdfpli}
\cong
\lctdmv{\brr}\shv{\pm\bsdfpxt}.
\end{split}
\]
%\end{multline*}
\end{proof}

\subsubsection{Proof of Lemma~\ref{lm:smmtw}}
%\begin{proof}[Proof of Lemma~\ref{lm:smmtw}]
According to Theorem~\ref{th:drortp}, in order to calculate the derived tensor product in the \lhs of \ex{eq:frdm} we can use a \weqxyop\ Koszul resolution of the $\uWpxyo$-module $\Dlxyo$:
\begin{equation}
\label{eq:kzarc}
\rsPb(\Dlxyo) = \bsbrs{
\xymatrix@C=1.5cm{
\mcnbvv{\mQxyo}{x_1}{y_1}
\ar[r]^-{y_1-x_1}
&
\mQxyo
}
}
\end{equation}

For any $\uWpxy$-module $M$ the tensor product
$M\otxyo\mcnbv{\Qxyo}{\bfa}$, $\bfa\in\Qbxy$ is isomorphic to $M\shv{\bfa} $ considered as a $\uWpxyt$-module, that is, the tensor product functor
\[
\xymatrix@C=4cm{
\ctWxy\ar[r]^{\xdmm\otxyo\mcnbv{\mQxyo}{\bfa}}
&\ctWxyt
}
\]
is isomorphic to the functor of shifting the $\uWp$ connection by $\bfa$ and forgetting the $\Qxyo$-module structure.
It is convenient to replace the forgotten variables $x_1$ and $y_1$ with new variables
\[
\frz = y_1 - x_2,\qquad\frw = y_1 - x_1.
%,\qquad \xLm\frz = ,\qquad\xLm\frw= .
\]
We present $\uWpxy$-modules $\mMpr$ and $\mMcr$ as quotients
%This time we present the modules as quotients
\[
\mMpr \cong \mQxytzw/(\spdf,\frw),\qquad
\mMcr \cong \mQxytzw/\bbrs{\spdf,\frz\frw},
\]
where $\spdf = \frw+\ytmxt$. Taking the quotient over $\spdf$ in these expressions implies that when $\Qv{\frw}$-module structure is forgotten, the variable $\frw$ can be eliminated from module generators with the help of the relation $\frw = x_2 - y_2$ in the quotient. Hence
\begin{equation}
\label{eq:mdl}
\begin{aligned}
\mMpr \otxyo\mQxyo &\cong  \mQxytz/(\ytmxt)\cong
\Dlxytsz,
\\
%\label{eq:smdl}
\mMcr\otxyo\mQxyo &\cong \tmQxytz,
%\mQxytz/\bbrs{(\ytmxt)\frz},
\end{aligned}
\end{equation}
where we used shortcut notations
\[
\tmQxytz = \mQxytz/\bbrs{(\ytmxt)\frz},\qquad \Dlxytsz=\Dlxytsoz
\]
and generators $\xLm$ of $\uWp$ have the \stact\ on $x_2$ and $y_2$, while
\begin{equation}
\label{eq:lmzac}
\xLm\,\frz
%= \bsdfv{y_1}{x_2}\,\frz,\qquad y_1 = x_2 + \frz.
=\bsdfz\,\frz,\qquad\text{
%where
$\bsdfz = \bsdfv{x_2}{y_1}=\bsdfv{x_2}{x_2+\frz}$}.
\end{equation}

$\Qfrz$ splits as a $\IQ$-vector space: $\Qfrz = \IQ \oplus \zQfrz$,
%the first summand being generated by $1$, while the second being generated by positive powers of $\frz$,
and there is a canonical isomorphism of $\Qfrz$-modules $\zQfrz\cong\Qfrz$. The $\uWpxyt$-modules~\eqref{eq:mdl} split accordingly:
\begin{align*}
\mMpr \otxyo\mQxyo &\cong \Dlxytsz \cong \Dlxyt \oplus \mcnbv{\Dlxytsz}{\bsdfz},
\\
\mMcr\otxyo\mQxyo &\cong  \tmQxytz\cong \mQxyt \oplus \mcnbv{\Dlxytsz}{\bsdfz},
\end{align*}
%where $\Dlxytsz=\Dlxytsoz$ and
the presence of connection $\bsdfz$ being due to~\eqref{eq:lmzac}.
If the $\uWp$ action on the \lhs modules is modified by a connection $\bfa$, whose elements $a_m$ depend on $z$, then the splitting remains only at the level of $\Qxyt$-modules, while $\uWpxyt$-modules form non-split exact sequences
\begin{gather}
\label{eq:fexseq}
\xymatrix{
0 \ar[r] &
\mcnbv{\Dlxytsz}{\bfa+\bsdfz}
%\bbrs{\Dlxyt\otimes\zQfrz}\shba
\ar[r]^-{\frz}
&
%\bbrs{\mMpr \otxyo\mQxyo}\shba
%\mcnbv{\mMpr\otxyo\mQxyo}{\bfa}
\mcnbv{\Dlxytsz}{\bfa}
\ar[r]^-{1}
&
%\Dlxyt\shbap,
\mcnbv{\Dlxyt}{\bfa\zez}\ar[r]
&
0,
}
\\
\label{eq:sexseq}
\xymatrix{
0 \ar[r] &
%\bbrs{\Dlxyt\otimes\zQfrz}\shba
\mcnbv{\Dlxytsz}{\bfa+\bsdfz}
\ar[r]^-{\frz}
&
%\mcnbv{\mMcr \otxyo\mQxyo}{\bfa}
\mcnbv{ \tmQxytz }{\bfa}
\ar[r]^-{1}
&
%\mQxyt\shbap,
\mcnbv{\mQxyt}{\bfa\zez}
\ar[r]
&
0.
}
\end{gather}
%where $\shbap = \shba|_{\frz=0}\in\Qxyo$.

%Now we prove the \qiso~\eqref{eq:frdm} for $\sggn^{+1}$.
%
Consider tensor products of the modules $\mMpr$ and
%$\mcnbv{\mMcr}{-\bsdfxot}$
$\mMcr$ with the Koszul resolution~\eqref{eq:kzarc}. We use shortcut notations
\begin{align*}
\bfsdf & = \bsdfv{x_1}{x_2} = \bsdfv{x_2}{y_2+\frz},
\\
\xcnapro & =\bsdfv{y_1}{y_1}= \bsdfpxtz,
\\
\xcnacro & = \bsdfv{x_1}{y_1} = \bsdfv{x_2+z}{y_2+z},
\\
\xcnalno & =\bsdfv{x_2}{y_2}.
\end{align*}
Since
\begin{gather*}
y_1-x_1=0,\quad
%\bsdfxoyo=\bsdfpv{y_1} = \xcnapro
\xcnacro = \xcnapro
\qquad\mod (\spdf,\frw),
\\
y_1 - x_1 = x_2 - y_2\quad\mod \spdf,
\end{gather*}
then in view of presentations~\eqref{eq:mdl} we find
\begin{align}
\mMpr \otxyo \rsPb(\Dlxyo) & \cong
%\dgsha \mcnbv{\Dlxyt\otimes\Qfrz}{
%%\bsdfpxtz
%\xcnapro
%} \oplus \bbrs{\Dlxyt\otimes\Qfrz}.
\bsbrs{
\xymatrix{ \mcnbv{\Dlxytsz}{
%\bsdfpxtz
\xcnapro
}
\ar[r]^-{0} &
\Dlxytsz
}
},
\\
\mMcr\otxyo \rsPb(\Dlxyo)&\cong
\bsbrs{
\xymatrix@C=1.5cm{
\mcnbv{\tmQxytz}{\xcnacro}
\ar[r]^-{y_2-x_2}
&
\tmQxytz
}
}.
\end{align}

Now we are ready to establish the \qiso s~\eqref{eq:frdm}.
\begin{proof}[Proof of \qiso~\eqref{eq:frdm} for $\sggn$]
Exact sequence~\eqref{eq:sexseq} of $\uWpxyt$-modules implies short exact sequence in the category of chain complexes $\ctCh(\uWpxyt\sctgmod)$ represented by the first column of the diagram (we omitted zeroes at the top and bottom):
\begin{equation}
\label{eq:extra}
\mvcn{
\xymatrix@R=1.2cm{
\bbrs{\mMpr \otxyo \rsPb(\Dlxyo)}
\shv{\bsdfz}
\ar[d]^-{\hchipo\otimes\xId= \frz}
\ar@{}[r]|(0.55){\cong}
&
\bigl[
\mcnbv{\Dlxytsz}{
\xcnapro + \bsdfz
}
\ar[d]^-{\frz}
\ar[r]^-{0} &
\mcnbv{\Dlxytsz}{\bsdfz}
\ar[d]^-{\frz}
\bigr]
\\
\mMcr\otxyo \rsPb(\Dlxyo)
\ar@{}[r]|(0.55){\cong}
\ar[d]^{\spmf}
&
\bigl[\mcnbv{\tmQxytz}{\xcnacro}
\ar[r]^-{y_2-x_2}
\ar[d]^-{1}
&
\tmQxytz
\ar[d]^-{1}
\bigr]
\\
\aAb
\ar@{}[r]|(0.55){=}
&
\bigl[\mcnbv{\Qxyt}{\xcnalno}
\ar[r]^-{y_2 - x_2}
&
\Qxyt\bigr],
}
}
\end{equation}
where $\aAb$ denotes the complex at the bottom right, $\spmf$ is a quotient homomorphism
and we used the relation $\xcnacro\zez =\xcnalno$ in the middle of the bottom row.

According to definition~\eqref{eq:bcmps} and due to relation
\[
\bfsdf = \bsdfz\quad\mod y_1-x_1,\qquad
\]
the complex in the \lhs of the \qiso~\eqref{eq:frdm}
corresponding to $\sggn$
%=\sggn^{+1}$
has a presentation as the cone of the morphism $\hchipo\otimes\xId$ in the category $\ctDr(\ctWxyt)$:
\[
\ctmvbxy{ \sggn } \otdrxyo\Dlxyo
\cong
\boxed{
\xymatrix@C=3pc{
\bbrs{\mMpr \otxyo \rsPb(\Dlxyo)}
\shv{\bsdfz}
%\vspace{2pt}
\ar[r]^-{\hchipo\otimes\xId}
&
\mMcr\otxyo \rsPb(\Dlxyo)
}
}\;\shv{-\bfsdf}.
\]

Since $\spmf\circ(\hchipo\otimes\xId)=0$ in the diagram~\eqref{eq:extra}, and due to relation
$\bfsdf\zez=\xcnalno$
there is a morphism in $\ctCh(\ctWxyt)$:
\begin{equation}
\label{eq:qmor}
\mvcn{
\xymatrix@C=3pc@R=2.5pc{
\bbrs{\mMpr \otxyo \rsPb(\Dlxyo)}
\shv{\bsdfz}
%\vspace{2pt}
\ar[r]^-{\hchipo\otimes\xId}
&
\mMcr\otxyo \rsPb(\Dlxyo)
\save[]+<2.8cm,0cm>*\txt<8pc>{$\shv{-\bfsdf}$}
\restore
\\
&
\aAb\;\shv{- \xcnalno}
\save "1,2". "1,1"
*[F]\frm{}
\ar"2,2"^-{\spmf}
\restore
}
}
\end{equation}
The exact sequence~\eqref{eq:extra} splits if we forget the $\uWp$-module structure, that is, it splits in the category $\ctCh(\Qxyt\sctgmod)$. Hence the morphism~\eqref{eq:qmor} is a
homotopy equivalence in $\ctCh(\Qxyt\sctgmod)$ and, as a consequence, it is a
\qiso\ in $\ctCh(\Qxyt\sctgmod)$, so there is a \qiso
\begin{equation*}
%\label{eq:qisoon}
\ctmvbxy{ \sggn } \otdrxyo\Dlxyo\eqsmout\aAb\shv{- \xcnalno}.
\end{equation*}
Now the \qiso~\eqref{eq:frdm} follows from the \xqiso\ in $\ctWxyt$
\[
\aAb\shv{-\xcnalno}\simeq\Dlxyt\shv{-\bsdfpxt}
\]
which is similar to \eqref{eq:kzarc}.
\end{proof}
\begin{proof}[Proof of \qiso~\eqref{eq:frdm} for $\sggn^{-1}$]
The exact sequence~\eqref{eq:fexseq} of $\uWpxyt$-modules leads to a short exact sequence
in the category of chain complexes $\ctCh(\uWpxyt\sctgmod)$ which is similar to~\eqref{eq:extra}:
\begin{equation}
\label{eq:extrtw}
\vcenter{
\xymatrix@R=1.2cm{
\bbrs{\mMpr \otxyo \rsPb(\Dlxyo)}
\shv{\bsdfz}
\ar[d]^-{\frz}
\ar@{}[r]|(0.55){\cong}
&
\bigl[
\mcnbv{\Dlxytsz}{
\xcnapro + \bsdfz
}
\ar[d]^-{\frz}
\ar[r]^-{0} &
\mcnbv{\Dlxytsz}{\bsdfz}
\ar[d]^-{\frz}
\bigr]
\\
\mMpr\otxyo \rsPb(\Dlxyo)
\ar@{}[r]|(0.55){\cong}
\ar[d]^{\spmg}
&
\bigl[\mcnbv{\Dlxytsz}{\xcnapro}
\ar[r]^-{0}
\ar[d]^-{1}
&
\Dlxytsz
\ar[d]^-{1}
\bigr]
\\
\aBb
\ar@{}[r]|(0.55){=}
&
\bigl[\mcnbv{\Dlxyt}{\bsdfpxt}
\ar[r]^-{0}
&
\Dlxyt\bigr],
}
}
\end{equation}
where $\aBb$ denotes the complex at the bottom right, $\spmg$ is a quotient homomorphism and we used a relation $\xcnapro\zez=\bsdfpxt$.
Combining short exact sequences~\eqref{eq:extra} and~\eqref{eq:extrtw} we construct a chain of morphisms in the category $\ctCh(\ctWxyt)$:
\begin{equation}
\label{eq:bgdga}
\vcenter{
\xymatrix@R=3pc@C=5pc{
\bbrs{\mMpr \otxyo \rsPb(\Dlxyo)}
\shv{\bsdfz}
\ar[d]^-{\xId}
\ar[r]^-{\frz}
&
\mMcr\otxyo \rsPb(\Dlxyo)
\ar[d]^-{\hchipo\otimes\xId=1}
\ar[r]^-{\spmf}
&
\aAb
\ar[d]^-{\spmh}
\\
\bbrs{\mMpr \otxyo \rsPb(\Dlxyo)}
\shv{\bsdfz}
\ar[r]^-{\frz}
&
\mMpr \otxyo \rsPb(\Dlxyo)
\ar[r]^-{\spmg}
&
\aBb
\save
"1,1". "2,1"
*[F]\frm{}
\restore
\save
"1,2". "2,2"
*[F]\frm{}
\restore
\save
"1,3". "2,3"
*[F]\frm{}
\restore
}
}
\end{equation}
where the chain map $\spmh$ is defined by the diagram
\begin{equation}
\label{eq:dgrdh}
\mvcn{
\xymatrix@R=1.2cm{
\aAb
\ar@{}[r]|(0.3){=}
\ar[d]^{\spmh}
&
\bigl[\mcnbv{\Qxyt}{\xcnalno}
\ar[r]^-{y_2-x_2}
\ar[d]^-{1}
&
\Qxyt
\ar[d]^-{1}
\bigr]
\\
\aBb
\ar@{}[r]|(0.3){=}
&
\bigl[\mcnbv{\Dlxyt}{\bsdfpxt}
\ar[r]^-{0}
&
\Dlxyt\bigr].
}
}
\end{equation}
Exact sequences of modules in the rows of the diagram~\eqref{eq:bgdga} split if we forget the $\uWp$-module structure, and the complex in the first column is contractible. Hence, according to Theorem~\ref{pr:qsbc}, the second and third columns are \qisc\ in $\ctCh(\ctWxyt)$. Since the second column is isomorphic to $\ctmvbxy{ \sggnpm } \otdrxyo\Dlxyo$, we have a \qiso
\begin{equation}
\label{eq:lstq}
\ctmvbxy{ \sggnpm } \otdrxyo\Dlxyo \eqsmout
\boxed{
\xymatrix{
\aAb\ar[r]^-{\spmh} & \aBb
}
}.
\end{equation}

Diagram~\eqref{eq:dgrdh} shows that the complex in the \rhs of \qiso~\eqref{eq:lstq} can be presented as a cone in $\ctCh(\ctWxyt)$:
\begin{equation}
\label{eq:qstwcn}
\boxed{
\xymatrix{
\aAb\ar[r]^-{\spmh} & \aBb
}
}
\eqsmout
\boxed{
\xymatrix{
\aDb \ar[r] & \mcnbv{\Dlxyt}{\bsdfpxt}
}
},
\end{equation}
where
\[
\aDb =
\boxed{
\vcenter{
\xymatrix@C=2.5pc{
\bigl[\mcnbv{\Qxyt}{\xcnalno}
\ar[r]^-{y_2-x_2}
%\ar[d]^-{1}
&
\Qxyt
\ar[d]^-{1}
\bigr]
\\
&
\Dlxyt
}
}
}
\]
The latter complex is contractible, because its upper row is \xqisc\ to $\Dlxyt$ in $\ctWxyt$, so
\[
\aDb\eqsmout
\boxed{
\xymatrix{
\Dlxyt \ar[r]^-{\xId} &
\Dlxyt
}
}.
\]
Hence the complex in the \rhs of~\eqref{eq:qstwcn} is \qisc\ to $\mcnbv{\Dlxyt}{\bsdfpxt}$ and together with~\eqref{eq:lstq} this proves the \qiso~\eqref{eq:frdm} for $\sggn^{-1}$.

\end{proof}

\end{document}

The module $\Qv{y}$ has a submodule $y\Qv{y}$ generated by $y^i$, $i\geq 1$. Accordingly, the $\uWpxyt$-module $\mcnbv{\Dlxyt\otimes\Qv{y}}{\spyom}$ has a submodule
$\mcnbv{\Dlxyt\otimes y\Qv{y}}{\bsdfxtyo+\spyom}$, the quotient being $\mcnbv{\Dlxyt}{\spxtm}$,
and the complex in the \rhs of \ex{eq:bcmo} has a contractible subcomplex
\[
\boxed{
\xymatrix{
\mcnbv{\Dlxyt\otimes\Qv{y}}{\bsdfxtyo+\spyom} \ar[r]^-{1}
&
\mcnbv{\Dlxyt\otimes\Qv{y}}{\bsdfxtyo+\spyom}
}
}
\]
If we forget the $\aWp$-module structure, then the complex in the \rhs of \ex{eq:bcmo} splits, hence by Theorem~\ref{pr:qsbc} we have a \qiso
\[
\aBb \eqsmout \mcnbv{\Dlxyt}{\spxtm}.
\]

\begin{multline*}
\ctmvbxy{ \sggni } \otdrxyo\Dlxyo
\\
\eqsmout
\boxed{
\xymatrix@C=2.5cm{
\bbrs{\Qv{x_2,y_2,y,\bfth},\xbcdfcr,\dfcr}
\ar[r]^-{1 + (y-1)\theta_2\prtht}
&
\bbrs{\Qv{x_2,y_2,y,\bfth},\xbcdfpr,\dfpr}
}
}
\end{multline*}

As $\Qbxy$-modules, the $\aWp$-equivariant Koszul resolutions $\mPpr$ and $\mPcr$ have the same form:
\[
\mPpr \cong \mPcr \cong \mQxy\otimes\Qv{\bfth},
\]
where $\bfth=\theta_1,\theta_2$ are odd variables, while the Koszul differentials~\eqref{eq:kszdf} are
\begin{equation}
\label{eq:dfkr_3}
\dfpr = \spdf\prtho + (\yomo)\prtht,\qquad
\dfcr = \spdf\prtho + \yomot\prtht.
\end{equation}
The action of the generators $\xLm$ of $\aWp$ on the resolutions has the form~\eqref{eq:dfkr}:
\begin{equation}
\label{eq:derkr}
\xcdfprm = \xcdfzm + \sum_{i,j=1,2}\ycnamij\theta_j\partial_{\theta_i},\qquad
\xcdfcrm = \xcdfzm + \sum_{i,j=1,2}\ycnbmij\theta_j\partial_{\theta_i},
\end{equation}
where $\xcdfzm$ represents the \stact\ of $\aWp$ on $\Qbxy$:
\[
\xcdfzm = \sum_{i=1,2}(x_i^{m+1}\partial_{x_i} + y_i^{m+1}\partial_{y_i}),
\]
while the coefficients $\ycnamij$ and $\ycnbmij$ are expressed in terms of divided differences
\begin{equation}
\label{eq:cfsa}
\begin{aligned}
\ycnamoo & = \sdfmv{x_2}{y_2} - \sdfmv{y_1}{y_2}+\sdfmv{x_1}{y_2},
&&&&&&&
\ycnamot & = (y_1 - x_2)\, \ycnbmot
\\
\ycnamto &= 0,
&&&&&&&
\ycnamtt & = \sdfmv{x_1}{y_1}
\end{aligned}
\end{equation}
\begin{equation}
\label{eq:cfsb}
\begin{aligned}
\ycnbmoo &= \ycnamoo,
&
\ycnbmot & = \ssdfmov{x_1}{y_1}{y_2} + \ssdfmov{x_2}{y_1}{y_2}
\\
\ycnbmto & = 0,
&
\ycnbmtt &= \sdfmv{x_1}{y_1} + \sdfmv{x_2}{y_1},
\end{aligned}
\end{equation}
and we define $\ssdfmv{x}{y}{z}$ as a divided difference
\[
\ssdfmv{x}{y}{z} = \frac{\sdfvv{m+1}{x}{z}-\sdfvv{m+1}{y}{z}}{x-y} =
\sum_{i,j,k\atop i + j + k = m}x^i y^j z^k.
\]

The action of the homomorphisms~\eqref{eq:hmsmds} on the resolutions is presented by the formulas
\begin{equation}
\label{eq:hmsres}
\hchine = 1 + (\yomt-1)\theta_2\prtht,\qquad
\hchipo = \yomt + (1 - \yomt)\theta_2\prtht
\end{equation}

A direct calculation shows that the derivations~\eqref{eq:derkr} with coefficients~\eqref{eq:cfsa} and~\eqref{eq:cfsb} satisfy the commutation relations~\eqref{eq:cmrl} and commute with the differentials~\eqref{eq:dfkr} and with the homomorphisms~\eqref{eq:hmsres}.

%\begin{proposition}
%The derivations~\eqref{eq:derkr} with coefficients~\eqref{eq:cfsa} and~\eqref{eq:cfsb} satisfy the commutation relations~\eqref{eq:cmrl} and commute with the differentials~\eqref{eq:dfkr} and with the homomorphisms.
%\end{proposition}
%\begin{proof}
%These commutation relations are established by direct computations and we leave them to the reader.
%\end{proof}
%
%
%\begin{align*}
%\xcdfprm& = \xcdfzm + \ycnamoo\theta_1\prtho + \ycnamtt\theta_2\prtht + \ycnamot\theta_2\prtho
%\\
%\xcdfcrm &= \xcdfzm + \ycnbmoo\theta_1\prtho + \ycnbmtt\theta_2\prtht + \ycnbmot\theta_2\prtho
%\end{align*}
%
%
%modules $\mMpr$ and $\mMcr$ the resolution has the form $\mQxy\otimes\IQ$

The use of \Wpeq\ resolutions allows us to replace the derived tensor product in the \rhs of \ex{eq:frdm} with the ordinary one. Since $\Qbxy\otv{x_1,y_1}\Dlv{x_1,y_1}\cong \Qbxy/(y_1-x_1)\cong\Qv{y_1,x_2,y_2}$, the impact of the tensor product with $\Dlv{x_1, y_1}$ is to replace $x_1$ with $y_1$ in all the formulas~\eqref{eq:cfsa}--\eqref{eq:hmsres} and `forget' about the $y_1$ action. It is convenient to replace $y_1$ with another forgotten variable $y = y_1 - x_2$. Thus the derived tensor products within the \rhs of \ex{eq:frdm} take the form
\[
\mPpr\otv{x_1,y_1}\Dlv{x_1,y_1} \cong \mPcr \otv{x_1,y_1}\Dlv{x_1,y_1} \cong \Qv{x_2,y_2,y,\bfth}
%\otimes\Qv{\bfth},
\]
while differentials~\eqref{eq:dfkr} become
\[
\dfpr = \dfcr=(y_2 - x_2)\prtho,
\]
the coefficients of the derivations~\eqref{eq:cfsa} and~\eqref{eq:cfsb} become
\begin{align*}
\ycnamoo &= \sdfmv{x_2}{y_2}, &\ycnamtt & = (m+1)(y+x_2)^m, &
\ycnamto = 0,
\\
\ycnbmoo &=\sdfmv{x_2}{y_2}, &\ycnbmtt &= (m+1)(y+x_2)^m + \sdfmv{x_2}{y+x_2}
\end{align*}
(the value of  and is unimportant)
and the homomorphisms~\eqref{eq:hmsres} become
\[
\hchine = 1 + (y-1)\theta_2\prtht,\qquad
\hchipo = y + (1 -y)\theta_2\prtht.
\]

Now we consider the cases of $\sggn$ and $\sggni$ in \ex{eq:frdm} separately. We begin with $\sggni$:
\begin{multline*}
\ctmvbxy{ \sggni } \otdrxyo\Dlxyo
\\
\eqsmout
\boxed{
\xymatrix@C=2.5cm{
\bbrs{\Qv{x_2,y_2,y,\bfth},\xbcdfcr,\dfcr}
\ar[r]^-{1 + (y-1)\theta_2\prtht}
&
\bbrs{\Qv{x_2,y_2,y,\bfth},\xbcdfpr,\dfpr}
}
}
\end{multline*}
The complex in the \rhs has a subcomplex
\[
\aBb =
\boxed{
\xymatrix@C=2.5cm{
\bbrs{\theta_2\Qv{x_2,y_2,y,\theta_1},\xbcdfcr,\dfcr}
\ar[r]^-{y}
&
\bbrs{\theta_2\Qv{x_2,y_2,y,\theta_1},\xbcdfpr,\dfpr}
}
}
\]

%\end{proof}
****************************
%\begin{bibdiv}
%\begin{biblist}
%
%\bib{Khtgd}{article}
%{
%author={Khovanov, Mikhail}
%title={Triply-graded link homology and Hochschild homology of Soergel bimodules}
%journal={International journal of mathematics}
%volume={18}
%number={8}
%year={2007}
%pages={869-885}
%eprint={math.GT/0510265 }
%}
%
%\end{biblist}
%\end{bibdiv}

\end{document}

by complexes of $\uWpXxy$ \Sglb s

\subsection{Elementary \Sglb s and the map $\actmdvbxy{\dmmy}$}

\subsubsection{General construction}
We are going to define the homomorphism $\actmdvbxy{\dmmy}\colon\brwgrn \rightarrow \ctKom\bigl(\ctXWxy\bigr)$ which converts the product of \brwd s into the left partial derived tensor product: $\actmdvbxy{\brw_1\brw_2 } \eqsmout \actmdvbxy{\brw_1} \otdrly \actmdvbxy{\brw_2}$.
We will define the left and right versions of the bracket $\actmdvbxy{\dmmy}$ in such a way that both brackets are left (resp. right) \vbrass:
\[
\alctmvbxy{\dmmy} = \ctmvbxy{\dmmy}\otdrx\bigl(\Qv{\bfvbr,\bfx}/(\bfvbr-\bfx)\bigr),\qquad
\arctmvbxy{\dmmy} = \ctmvbxy{\dmmy}\otdry\bigl(\Qv{\bfvbr,\bfy}/(\bfvbr-\bfy)\bigr)
\]
where $\ctmvbxy{\dmmy}$

and use the left version as the definition of $\actmdvbxy{\dmmy}$:
\[
\actmdvbxy{\dmmy} = \alctmdvbxy{\dmmy}.
\]
However, both versions will be \qisc\ after a permutation of the generators $\bfvbr$ by the symmetric group element corresponding to the \brwd:

\vspace*{1in}
*****************

\section{Triply graded link homology of a \brwd\ closure}

%\subsection{A complex of a \brwd}

We present a \Wpeq\ version of the \Sgl\ \bmdl\ construction~\cite{Khtgd} of the \tglh. In order to simplify notations we fix $m>0$ and consider the action of a single \drv\ $\xLm$. We will abbreviate the notation~\eqref{eq:dfpm}:
$\sdfv{x}{y} = \sdfmv{x}{y}$.
Since all connections for $\xLm$ will be of the form~\eqref{eq:dfpm}, then according to Corollary~\ref{cor:cnLm} the commutation relations
%~\eqref{eq:gtr}
are preserved and at every important stage in our construction we will have fully \Wpeq\ modules, complexes and homology.

Sometimes for a module $\mM$ over an algebra $\Qbx$, $\bfx = x_1,\ldots,x_n$ we use a special notation $\mM_{\bfx}$ in order to emphasize the set of variables $\bfx$. Then $\mM_{\bfy}$ would denote the same module over the algebra $\Qby$.

\subsection{A complex of a \brwd\ and a homology of its closure}

For a fixed positive integer $\nst$ we consider a graded algebra $\Qbx$, $\bfx=x_1,\ldots,x_\nst$ with the \stact\ of $\aWp$. All variables $x_i$ have the same `\tdgq': $\dgq x_i = 2$ and the algebra $\aWp$ is also graded: $\dgq \xLm = 2m$.

The objects of the bounded derived category of \tgrddq\ \aWpeq\
$\Qbx$-\bmdl s $\ctDW(\Qbxy-\ctgmod)$ are graded by homological degree which is called `\tdga'. This category is additive and we consider the homotopy category of complexes over it:
\begin{equation}
\label{eq:mncate}
\tCatbxy = \ctKom\big(\ctDW(\Qbxy-\ctgmod)\big).
\end{equation}
This category has a triple grading: in addition to \tgrdq\ and \tgrda\ coming from $\ctDW(\Qbxy-\ctgmod)$  it has a \tgrdt\ which is the homological grading of the complexes of $\ctKom$.
There is a forgetful functor from this category to the similar category which ignores the \Wpeq\ structure of modules:
$\Ff\colon \tCatbxy\rightarrow \ctKom\big(\ctD(\Qbxy-\ctgmod)\big)$

Rather than a homotopy category of complexes~\eqref{eq:mncat}, we need a derived category over $\ctDW(\Qbxy-\ctgmod)$ with respect to the $\aWp$ action:
\begin{equation}
\label{eq:mncat}
\cCatbxy = \ctD\big( \ctDW(\Qbxy-\ctgmod) \big)
\end{equation}
%consider a new category $\cCatbxy$.
Its objects are the same as those of $\tCatbxy$, but more pairs of objects are considered isomorphic. Namely, the isomorphism relation within $\tCatbxy$ is generated by \emph{\qhteq s}.
\begin{definition}
\label{df:qhmeq}
Two complexes $\aAb$ and $\aBb$ of the category $\cCatbxy$ are \qhtet: $\aAb\xqh \aBb$ if there is a  morphism $f\in\Hom_{\cCatbxy}(\aAb,\aBb)$ such that its cone is contractible within $\ctKom\big(\ctD(\Qbxy-\ctgmod)\big)$, that is,
\[
\Ff\left( \boxed{\aAb\xrightarrow{f}\aBb}\right)
%\hspace{-5.5in}
\sim 0.
\]
%then $\aAb$ and $\aBb$ are \qhtet.
\end{definition}
There is an equivalent `techical' definition:
\begin{definition}
Two complexes $\aAb$ and $\aBb$ of the category $\cCatbxy$ are \emph{\qhtet} if they are homotopy equivalent as objects of $\ctKom\big(\ctD(\Qbxy-\ctgmod)\big)$, that is $\Ff(\aAb) \sim \Ff(\aBb)$, and one of the homotopy equivalence chain maps $\aAb\xrightarrow{f} \aBb$ is \Wpeq, that is, it represents a morphism within $\tCatbxy$.
\end{definition}

We will use a particular case of \qhteq.
%\begin{proposition}
%\label{pr:qcontr}
%If there is a morphism between two complexes $f\in\Hom_{\cCatbxy}(\aAb,\aBb)$ such that its cone is contractible within $\ctKom\big(\ctD(\Qbxy-\ctgmod)\big)$, that is,
%\[
%\Ff\left( \boxed{\aAb\xrightarrow{f}\aBb}\right)
%%\hspace{-5.5in}
%\sim 0,
%\]
%then $\aAb$ and $\aBb$ are \qhtet.
%\end{proposition}
%
%More practically, we use the following:
\begin{proposition}
\label{pr:qsbc_1}
Suppose that a $\ctKom$-complex $\aAb$ of \Wpeq\ $\Qbxy$-modules representing an object of $\tCatbxy$ has a \Wpeq\ subcomplex $\aBb\subset\aAb$. Let $\aCb$ be the quotient complex: $\aCb = \aAb/\aBb$. If after forgetting the \Wpeq\ structure each module of $\aAb$ splits: $\Ff(\aA^i) \cong \Ff(\aB^i) \oplus \Ff(\aC^i)$ (that is, $\Ff(\aAb)$ has a structure of a cone: $\Ff(\aAb) \cong \boxed{\Ff(\aCb)\rightarrow\Ff(\aBb)}$), then
\[
\aAb \xqh
\begin{cases}
\aCb,&\text{if $\aBb$ is contractible,}
\\
\aBb,&\text{if $\aCb$ is contractibe.}
\end{cases}
\]
\end{proposition}
The category $\cCatbxy$ has a monoidal structure coming from derived tensor product of \bmdl s:
\begin{equation}
\label{eq:bmtp}
\cCatbxy \times \cCatbyz \xrightarrow{\;\;\otdr_{\Qby}\;\;} \cCatbxz.
\end{equation}
In fact, all $\Qbxy$-modules appearing in this paper are semi-free, that is, they are free as $\Qbx$-modules and as $\Qby$-modules, therefore in our computations the derived tensor product may be replaced by an ordinary one.

A group of \brwd s with $\nst$ strands $\brwgrn$ is a free \smgr\ on $2\nst-2$ generators $\sggni$, $\sggxii$, $i = 1,\ldots,\nst-1$.
There is a sequence of two obvious surjections
$\xymatrix{\brwgrn \ar @{->>}[r]^-{\fbr}&\brgrn \ar @{->>}[r]^-{\fcl} &\lnksn},$
where $\brgrn$ is a braid group and $\lnksn$ is a set of oriented links which can be constructed as circular closures of $\nst$-strand braids. The map $\fbr$ turns $\sggni$ and $\sggxii$ into a braid group generator $\sggni$  and its inverse $\sggni^{-1}$, while $\fcl$ performs the circular closure of a braid.

The weak categorification procedure defines the remaining maps in the following commutative diagram:
\begin{equation}
\label{eq:mndiag}
\xymatrix@C=4pc@R=3pc{
\brwgrn \ar @{->>}[r]^-{\fbr} \ar[dr]_-{\ctmdv{\dmmy}}&
\brgrn \ar @{->>}[rr]^-{\fcl} \ar[d]^-{\ctmv{\dmmy}}&&
\lnksn \ar[d]^-{\Hml(\dmmy)}
\\
& \cCatbxy \ar[r]^-{\Hhom(\dmmy)}
& \ctKom(\aWpgmd) \ar[r]^-{\HmlK(\dmmy)}
&\aWpgmd
}
\end{equation}
where $\ctKom(\aWpgmd)$ is the homotopy category of graded $\aWp$-modules (over $\IQ$), $\Hhom(\dmmy)$ is the \Hchs\ homology within the category $\ctDW(\Qbxy-\ctgmod)$ and $\HmlK(\dmmy)$ is the homology within the homotopy category $\ctKom$.
The maps $\ctmdv{\dmmy}$ and $\ctmv{\dmmy}$ have to convert the composition in $\brwgrn$ and $\brgrn$ into the \bmdl\ tensor product~\eqref{eq:bmtp}

We will define the map $\ctmdv{\dmmy}$ and then show that the maps $\ctmv{\dmmy}$ and $\Hml(\dmmy)$ exist. The uniqueness of the latter maps follows from the surjectivity of $\fbr$ and $\fcl$.

**********************************************

\section{\Wpeq\ categorification of links}

It is sufficient to verify a local version of the second Markov move, which involves a partial circular closure of a braid with two strands. Let $\sggn$ denote the generator of the 2-strand braid group, $\bfx=x_1,x_2$ and $\bfy=y_1,y_2$.
We introduce a notation
\[
\shfmx = \sdfmv{x}{x} = (m+1)x^m.
\]
Since $\bshfxo = \bshfyo \mod y_1 - x_1$, we simply use notation $\bfshf=\bshfxo=\bshfyo$ in that case.
\begin{theorem}[Markov move IIa]
\label{th:mmtwa}
There is a \qhteq\ of complexes
\[
\ctmvbxy{ \sggn } \otdrxyt \mMarcxyt \xqh
\mcnbnbf{\mMarcxyo}.
\]
\end{theorem}

\begin{proof}
By definition~\eqref{eq:bcmps} of the complex of an elementary braid, we have to prove the following:
\begin{equation}
\label{eq:frdacn}
\boxed{
\xymatrix@C=1.5cm{
\mMprcxyo \ar[r]^-{\hchipo\otimes\xId} & \mcnnbp{\mMcrcxyo}
}
}
\xqh
\mcnbnbf{\mMarcxyo},
\end{equation}
where the complexes of \Wpeq\ $\Qv{x_1,y_1}$-modules in the \lhs are the derived tensor  products
\[
\mMprcxyo = \mMprrxy \otdrxyt  \mMarcxyt,\qquad
\mMcrcxyo = \mMcrrxy \otdrxyt  \mMarcxyt.
\]
The latter can be computed with the help of partial (that is, with respect to $\Qv{x_2,y_2}$-module structure) resolution of the left modules, which are presented in the  following diagram together with the action of the homomorphism $\hchipo$:
% and $\hchine$:
\begin{equation}
\label{eq:reshomo}
\xymatrix@R=1.25cm@C=0.5cm{
\aF(\mMprrxy)
\ar[d]_-{\hchipo}
\ar @{}[r]|(0.52){=}
&
\cBgbr{
\mcnbv{\mMgenrxy}{\xcnao}
\ar[rrr]^-{y_2 - x_2}
\ar[d]_-{1}
&&&
\mcnbv{\mMgenrxy}{\xcnaprz}
}
\ar[d]^-{y_2 - x_1}
\ar[rr]^-{1}
&&
\mMprrxy
\ar[d]^-{y_2 - x_1}
\\
\aF\mcnnbp{\mMcrrxy}
\ar @{}[r]|(0.52){=}
&
\cBgbr{
\mcnbv{\mMgenrxy}{\xcnao}
\ar[rrr]^-{(y_2 - x_1)(y_2 - x_2)}
&&&
\mcnbv{\mMgenrxy}{\xcnacrz}
}
\ar[rr]^-{1}
&&
\mcnnbp{\mMcrrxy}
}
\end{equation}
where the \cnn s of the resolution modules are
\[
\xcnacrz = - \bsdfxot,\qquad
\xcnaprz = \bsdfxoyt - \bsdfxot,\qquad
%\xcnapro = \xcnacro
\xcnao = \bsdfxyt + \bsdfxoyt - \bsdfxot.
\]
The ordinary tensor multiplication $\otimes_{\Qv{x_2,y_2}}\mMarcxyt$ amounts to taking the quotient by $y_2 - x_2$. Since the resolution differentials in the diagram~\eqref{eq:reshomo} become zero upon taking this quotient, while
$\mMgenrxy/(y_2 - x_2) \cong \mMarcxyoy$, where we denote $y = y_2 - x_1$, then the cone in the \lhs of \ex{eq:frdacn} has a presentation
\[
\vcenter{
\boxed{
\xymatrix@R=1.5cm{
\mMprcxyo \ar[d]_-{\hchipo\otimes\xId}
\\
\mcnnbp{\mMcrcxyo}
}
}
}
\hspace{-4.75in}
%\hspace{-5.25in}
\simeq
\vcenter{
\boxed{
\xymatrix@R=1.5cm{
\cBgbr{
\mcnbv{\mMarcxyoy}{\xcna}
\ar[d]_-{\xId}
\ar@{}[r]|{\oplus}
&
\mMarcxyoy
\ar[d]^-{\xoy}
}
\\
\cBgbr{
\mcnbv{\mMarcxyoy}{\xcna}
\ar@{}[r]|(0.4){\oplus}
&
\mcnbv{\mMarcxyoy}{-\bsdfv{x_1}{y+x_1}}
}
}
}
}
\]
where $\xcna = \bsdfv{y+x_1}{y+x_1}$.
The cone in the \lhs splits and the first summand contracts, hence
\begin{equation}
\label{eq:smcno}
\vcenter{
\boxed{
\xymatrix@R=1.5cm{
\mMprcxyo \ar[d]_-{\hchipo\otimes\xId}
\\
\mcnnbp{\mMcrcxyo}
}
}
}
\hspace{-4.75in}
%\hspace{-5.25in}
\simeq
\vcenter{
\boxed{
\xymatrix@R=1.5cm{
\mMarcxyoy
\ar[d]^-{\xoy}
\\
\mcnbv{\mMarcxyoy}{-\bsdfxoyxo}
}
}
}
\end{equation}
Since $\xLm y = \bsdfxoyxo\,y$, the \Wpeq\ $\Qxyo$-module at the bottom of the \rhs has a submodule
\[
\mMarcxyos\cong\mMarcxyoy
\]
freely generated over $\Qxyo$ by the elements $1\otimes y^i$, $i\geq 1$, while the quotient is
\[
\mcnbv{\mMarcxyoy}{-\bsdfxoyxo}/\mMarcxyos \cong \mcnbpbf{\mMarcxyo}.
\]
If we forget the \Wpec\ structure, then the latter relation splits:
\begin{equation}
\label{eq:smspl}
\mMarcxyoy \cong \mMarcxyos \oplus \mMarcxyo,
\end{equation}
and the cone in the \rhs of \ex{eq:smcno} splits into a sum, the second summand of which is contractible:
\[
\vcenter{
\boxed{
\xymatrix@R=1.5cm{
\mMarcxyoy
\ar[d]^-{\xoy}
\\
\mMarcxyoy
}
}
}
%\hspace{-4.75in}
\hspace{-5in}
\cong
\mMarcxyo \oplus
\vcenter{
\boxed{
\xymatrix@R=1.5cm{
\mMarcxyoy
\ar[d]^-{\xId}
\\
\mMarcxyos
}
}
}
\]
Hence by Proposition~\ref{pr:qsbc} there is a \qhteq
\[
\vcenter{
\boxed{
\xymatrix@R=1.5cm{
\mMarcxyoy
\ar[d]^-{\xoy}
\\
\mcnbv{\mMarcxyoy}{-\bsdfxoyxo}
}
}
}
\hspace{-3.5in}
\xqh
\mcnbnbf{\mMarcxyo},
\]
which together with \ex{eq:smcno} implies \ex{eq:frdacn}.
\end{proof}

\begin{theorem}[Markov move IIb]
There is a \qhteq\ of complexes
\begin{equation}
\label{eq:frdacni}
\ctmvbxy{ \sggi } \otdrxyt \mMarcxyt \xqh
\mcnbpbf{\mMarcxyo}.
\end{equation}
\end{theorem}
\begin{proof}
The proof is similar to that of Theorem~\ref{th:mmtwa}, so we will just sketch the details.
We have to prove the following \qhteq:
\begin{equation}
\label{eq:frdacnb}
\boxed{
\xymatrix@C=1.5cm{
\mMcrcxyo \ar[r]^-{\hchine\otimes\xId} & \mMprcxyo
}
}
\xqh
\mcnbpbf{\mMarcxyo}.
\end{equation}
The free resolutions of \Wpeq\ \bmdl s together with the homomorphism $\hchine$ are presented by the diagram
\begin{equation}
\label{eq:reshomt}
\xymatrix@R=1.25cm@C=0.5cm{
\aF(\mMcrrxy)
\ar[d]_-{\hchine}
\ar @{}[r]|(0.52){=}
&
\cBgbr{
\mcnbv{\mMgenrxy}{\xcnacro}
\ar[rrr]^-{(y_2-x_1)(y_2 - x_2)}
\ar[d]_-{y_2 - x_1}
&&&
\mMgenrxy
}
\ar[d]^-{1}
\ar[rr]^-{1}
&&
\mMcrrxy
\ar[d]^-{1}
\\
\aF(\mMprrxy)
\ar @{}[r]|(0.52){=}
&
\cBgbr{
\mcnbv{\mMgenrxy}{\xcnapro}
\ar[rrr]^-{y_2 - x_2}
&&&
\mMgenrxy
}
\ar[rr]^-{1}
&&
\mMprrxy
}
\end{equation}
where the \cnn s of the resolution modules are
\[
\xcnapro = \bsdfxyt,\qquad
\xcnacro = \bsdfxyt + \bsdfxoyt.
%\xcnacrz = - \bsdfxot,\qquad
%\xcnaprz = \bsdfxoyt - \bsdfxot,\qquad
%%\xcnapro = \xcnacro
%\xcnao = \bsdfxyt + \bsdfxoyt - \bsdfxot.
\]
Note that these connections are different from the ones in the free resolutions~\eqref{eq:reshomo}.

After a tensor multiplication $\otimes_{\Qv{x_2,y_2}}\mMarcxyt$ of the diagram~\eqref{eq:reshomt}, which amounts to taking a quotient by $y_2-x_2$, the cone of resolutions splits:
\begin{equation}
\label{eq:smcns}
\vcenter{
\boxed{
\xymatrix@R=1.5cm{
\mMcrcxyo \ar[d]_-{\hchine\otimes\xId}
\\
\mMprcxyo
}
}
}
\hspace{-5.25in}
%\hspace{-5.25in}
\simeq
\vcenter{
\boxed{
\xymatrix@R=1.5cm{
\cBgbr{
\mcnbv{\mMarcxyoy}{\xcnbcro}
\ar[d]_-{\xoy}
\ar@{}[r]|(0.55){\oplus}
&
\mMarcxyoy
\ar[d]^-{\xId}
}
\\
\cBgbr{
\mcnbv{\mMarcxyoy}{\xcnbpro}
\ar@{}[r]|(0.55){\oplus}
&
\mMarcxyoy
}
}
}
}
\end{equation}
where
\[
\xcnbpro = \bshfxoy,\qquad
\xcnbcro = \bshfxoy + \bsdfxoyxo
\]
The second column in \rhs of the diagram contracts. The lower module in the first column has a $\aWp$-invariant submodule
\[
\mcnbv{\mMarcxyos}{\xcnbcro}\subset \mcnbv{\mMarcxyoy}{\xcnbpro},
\]
such that
\[
\mcnbv{\mMarcxyoy}{\xcnbpro}/\mcnbv{\mMarcxyos}{\xcnbcro} \cong
\mcnbpbf{\mMarcxyo}
\]
and if we forget the \Wpec\ structure, then there is a splitting~\eqref{eq:smspl}.
Hence by Proposition~\ref{pr:qsbc} there is a \qhteq
\[
\vcenter{
\boxed{
\xymatrix@R=1.5cm{
\mcnbv{\mMarcxyoy}{\xcnbcro}
\ar[d]_-{\xoy}
\\
\mcnbv{\mMarcxyoy}{\xcnbpro}
}
}
}
\hspace{-4.5in}
\xqh
\mcnbpbf{\mMarcxyo},
\]
which together with \ex{eq:smcns} implies \ex{eq:frdacni}.
\end{proof}

%
%
%
%
%
%%
%%
%%
%%
%\vspace{1in}
%\begin{equation}
%\label{eq:reshom.init}
%\xymatrix@R=1.25cm{
%\aF(\mMcrrxy)
%\ar[d]_-{\hchine}
%\ar @{}[r]|(0.52){=}
%&
%\Big( \mMgenrxy
%\ar[rr]^-{\ymi^2 - \xmi^2}
%\ar[d]_-{\ymi + \xmi}
%&&
%\mMgenrxy \Big)
%\ar[d]^-{1}
%\ar[rr]^-{1}
%&&
%\mMcrrxy
%\ar[d]^-{1}
%\\
%\aF(\mMprrxy)
%\ar[d]_-{\hchipo}
%\ar @{}[r]|(0.52){=}
%&
%\Big( \mMgenrxy
%\ar[rr]^-{\ymi - \xmi}
%\ar[d]_-{1}
%&&
%\mMgenrxy \Big)
%\ar[d]^-{\ymi + \xmi}
%\ar[rr]^-{1}
%&&
%\mMprrxy
%\ar[d]^-{\ymi + \xmi}
%\\
%\aF\mcnnbp{\mMcrrxy}
%\ar @{}[r]|(0.52){=}
%&
%\Big( \mMgenrxy
%\ar[rr]^-{\ymi^2 - \xmi^2}
%&&
%\mMgenrxy \Big)
%\ar[rr]^-{1}
%&&
%\mcnnbp{\mMcrrxy}
%}
%\end{equation}
%Here
%$
%\mMgenrxy = \Qbxy/(y_1 + y_2 - x_1 - x_2)
%$
%is a \Wpeq\ $\Qbxy$-module, which is free as a module over $\Qv{x_2,y_2}$, while
%
%\vspace{1cm}
%\[
%%\vcenter{
%\boxed{
%\xymatrix{
%\mMprrxy \ar[r]^-{\hchipo} &\mMcrrxy
%}
%}
%%}
%\sim
%\mcnbpbp{\mMarcxyo}.
%\]
%
%
%
%\[
%\mMcrc \mMprc \mMgen \mMipr
%\]

%\[
%\mMcra \mMtha = \mMcrv{1} \mMcpr \mMipr
%\]
%$\mathsf{tqa} \mathbbm{1}$
%asdf

\end{document}

Each crossing of a link diagram $\xD$ can be `spliced' in two ways, we call them \Asplng\ and \Bsplng:
\def\zpsh{1}
%\[
%\xy 0;/r.22pc/:
%(0,0)*{
%\begin{tikzpicture}[menvthree]
%\draw (-\zpsh,\zpsh) -- (\zpsh,-\zpsh);
%\draw [lnovr] (-\zpsh,-\zpsh) -- (\zpsh,\zpsh);
%\draw (-\zpsh,-\zpsh) -- (\zpsh,\zpsh);
%\end{tikzpicture}
%} = "1";
%(-30,-15)*{
%\begin{tikzpicture}[menvthree,rotate=90]
%\draw(-\zpsh,-\zpsh) to [out=45,in=-180] (0,-\zpsh*0.5) to [out=0,in=135] (\zpsh,-\zpsh);
%\draw(-\zpsh,\zpsh) to [out=-45,in=-180] (0,\zpsh*0.5) to [out=0,in=-135] (\zpsh,\zpsh);
%\end{tikzpicture}
%} = "2";
%(30,-15)*{
%\begin{tikzpicture}[menvthree]
%\draw(-\zpsh,-\zpsh) to [out=45,in=-180] (0,-\zpsh*0.5) to [out=0,in=135] (\zpsh,-\zpsh);
%\draw(-\zpsh,\zpsh) to [out=-45,in=-180] (0,\zpsh*0.5) to [out=0,in=-135] (\zpsh,\zpsh);
%\end{tikzpicture}
%} = "3";
%{\ar@{~>}"1"+(-9,-5);"2"+(9,3)_\rmA};
%{\ar@{~>}"1"+(9,-5);"3"+(-9,3)^\rmB};
%(-25,0)*{}
%\endxy
%\vspace*{0.3cm}
%\]
%
\[
\begin{tikzpicture}[menvthree,commutative diagrams/every diagram]
\node (ul) at ( -2,-0.5){};
\node (dl) at (-6,-2){};
\path[commutative diagrams/.cd, every arrow, every label]
(ul) edge[commutative diagrams/squiggly] node[pos=0.4,swap] {A} (dl);
\node (ur) at ( 2,-0.5){};
\node (dr) at (6,-2){};
\path[commutative diagrams/.cd, every arrow, every label]
(ur) edge[commutative diagrams/squiggly] node[pos=0.4] {B} (dr);
\begin{scope}
\draw (-\zpsh,\zpsh) -- (\zpsh,-\zpsh);
\draw [lnovr] (-\zpsh,-\zpsh) -- (\zpsh,\zpsh);
\draw (-\zpsh,-\zpsh) -- (\zpsh,\zpsh);
\end{scope}
\begin{scope}[xshift=-8cm,yshift=-3cm,rotate=90]
\draw(-\zpsh,-\zpsh) to [out=45,in=-180] (0,-\zpsh*0.5) to [out=0,in=135] (\zpsh,-\zpsh);
\draw(-\zpsh,\zpsh) to [out=-45,in=-180] (0,\zpsh*0.5) to [out=0,in=-135] (\zpsh,\zpsh);
\end{scope}
\begin{scope}[xshift=8cm,yshift=-3cm]
\draw(-\zpsh,-\zpsh) to [out=45,in=-180] (0,-\zpsh*0.5) to [out=0,in=135] (\zpsh,-\zpsh);
\draw(-\zpsh,\zpsh) to [out=-45,in=-180] (0,\zpsh*0.5) to [out=0,in=-135] (\zpsh,\zpsh);
\end{scope}
\end{tikzpicture}
%\vspace*{0.3cm}
\]

Let $\sBD$ denote the diagram which consists of two parts: the circles resulting from \Bsplng\ of all crossings of $\xD$ (\tBcr s) and segments connecting those circles at places where crossings were in $\xD$ (\tstrt s). Schematically, one passes from $\xD$ to $\sBD$ in the following way:
\def\zpsh{1}
%\[
%%\vspace*{0.3cm}
%\xy 0;/r.22pc/:
%(-20,0)*{
%\begin{tikzpicture}[menvthree]
%\draw (-\zpsh,\zpsh) -- (\zpsh,-\zpsh);
%\draw [lnovr] (-\zpsh,-\zpsh) -- (\zpsh,\zpsh);
%\draw (-\zpsh,-\zpsh) -- (\zpsh,\zpsh);
%\end{tikzpicture}
%} = "1";
%(20,0)*{
%\begin{tikzpicture}[menvthree]
%\draw(-\zpsh,-\zpsh) to [out=45,in=-180] (0,-\zpsh*0.5) to [out=0,in=135] (\zpsh,-\zpsh);
%\draw(-\zpsh,\zpsh) to [out=-45,in=-180] (0,\zpsh*0.5) to [out=0,in=-135] (\zpsh,\zpsh);
%\draw [densely dashed] (0,-\zpsh*0.5) -- (0,\zpsh*0.5);
%\end{tikzpicture}
%}="2";
%{\ar@{~>}"1"+(12,0);"2"+(-12,0)};
%\endxy
%%\vspace*{0.3cm}
%\]
\[
\begin{tikzpicture}[menvthree]
\node (l) at (-2.2,0) {};
\node (r) at (2.2,0) {};
\path[commutative diagrams/.cd, every arrow, every label]
(l) edge[commutative diagrams/squiggly] (r);
\begin{scope}[xshift=-4cm]
\draw (-\zpsh,\zpsh) -- (\zpsh,-\zpsh);
\draw [lnovr] (-\zpsh,-\zpsh) -- (\zpsh,\zpsh);
\draw (-\zpsh,-\zpsh) -- (\zpsh,\zpsh);
\end{scope}
\begin{scope}[xshift=4cm]
\draw(-\zpsh,-\zpsh) to [out=45,in=-180] (0,-\zpsh*0.5) to [out=0,in=135] (\zpsh,-\zpsh);
\draw(-\zpsh,\zpsh) to [out=-45,in=-180] (0,\zpsh*0.5) to [out=0,in=-135] (\zpsh,\zpsh);
\draw [densely dashed] (0,-\zpsh*0.5) -- (0,\zpsh*0.5);
\end{scope}
\end{tikzpicture}
\]
where the arcs in the right diagram are parts of \tBcr s and the dashed segment is a \tstrt.

A crossing of $\xD$ and its \tstrt\ in $\sBD$ are called \tBadq\ if the \tstrt\ connects two different \tBcr s. A framed diagram $\xD$ is \tBadq\ if all of its crossings are \tBadq. Finally, a framed link is \tBadq\ if it can be represented by a \tBadq\ framed diagram. An unframed link $\xL$ is called \tBadq\ if there is at least one framed \tBadq\ diagram that represents it. Note that if an unframed link is \tBadq, then, generally, it can not be represented by a \tBadq\ framed diagram for all framings.

All alternating links are \tBadq, but not all \tBadq\ links are alternating: an example of this is a torus knot $\yTmmn$, $n \geq m \geq 3$. More generally, a link constructed by closing a totally negative braid is \tBadq. Torus knots $\yTmn$, $n\geq m \geq 3$ provide examples of links which are not \tBadq.

Here are some notations associated with a link diagram $\xD$ throughout the paper:
\begin{center}
\begin{tabular}{c c l}
$\svrt$ & -- &a set of crossings (\tstrt s) in $\xD$ or in $\sBD$
\\
%$\svrta$ & a subset of \tBadq\ crossings (\tstrt s)
%\\
%$\svrti$ & a subset of B-\tinadq\ crossings (\tstrt s)
%\\
$\ncrD$ & -- &the number of crossings in $\xD$
\\
$\ncriD$ & -- & the number of B-\tinadq\ crossings in $\xD$
\\
$\gvD$ & -- &  the number of \tBcr s in $\sBD$
\\
$\xfrmD$ & -- & the framing number of $\xD$
\end{tabular}
\end{center}

The following is an easy corollary of the results of section 7.7 of~\cite{Kh99}:
\begin{theorem}
\label{thm:frm}
The numbers $\ncrv{\xL}$ and $\gvv{\xL}$ are topological invariants of a \tBadq\ framed link $\xL$, because they do not depend on the choice of representative \tBadq\ diagram $\xD$ for $\xL$. Moreover, if
\tBadq\ framed links $\xL$ and $\xLp$ differ only by framing, then
%If $\xD$ and $\xDp$ are \tBadq\ diagrams of the same link, then
\begin{equation}
\label{eq:afrm}
\ncrv{\xLp} - \ncrv{\xL} = \gvv{\xLp} - \gvv{\xL} = - (\xfrmv{\xLp} - \xfrmv{\xL}).
\end{equation}
\end{theorem}

For an unframed \tBadq\ link $\xL$ we define the minimal crossing number $\ncrmL$ as the minimum among the numbers $\ncrD$ for \tBadq\ framed diagrams $\xD$ representing $\xL$.

%In addition to \tBdg\ $\sBD$ of a link diagram $\xD$ we define its \tBgr\ $\BgrD$: its vertices correspond to the circles of $\sBD$ and its edges correspond to the \tstrt s of $\sBD$.

\subsection{The \tKbr\ and the Jones polynomial}
The \tKbr\ of a framed tangle diagram is defined by the \splng\ relation and the unknot normalization condition:
%For a framed link $\xL$ let $\pJqL$ denote its Jones polynomial defined by the \tKbr\ axiom
%\[
%\xKbrBv{\xcrsp}
%\;\;=\;\;
%\qfrth\;\;
%\xKbrBv{\xpver}
%\;\;+\;\;
%\qmfrth\;\;
%\xKbrBv{\xphor}
%\]
\begin{equation}
\label{eq:dkbr}
\xKbrBv{\zoverv{0.75}{1}}
\;=\;
\qhlf\xKbrBv{\zunoverv{0.75}{1}} \;+\; \qmhlf \xKbrBv{\zunoverh{0.75}{1}},
\qquad
\xKbrBv{\zcirc{0.75}{1}} \; = \; -(\qpqi).
\end{equation}
%and by the condition that a disjoint circle creates a factor $-(\qpqi)$.
Thus defined, the bracket is framing-dependent:
\[
\xKbrBv{
\zposfr{0.75}{1}
}\;= \; -q^{\frac{3}{2}}\xKbrBv{
\zstvr{0.75}{1}
}.
\]

The Jones polynomial of a framed link $\xL$ is the \tKbr\ of its diagram: $\pJqL=\xKbrv{\xL}$.

\subsection{Cables and coloring}

We introduce coloring of tangle and link components through cabling and \tJWp s. A cable of a strand is depicted by using a thicker line with the label indicating the number of strands, and the \tJWp\ is depicted by a box:
\[
\begin{tikzpicture}
\draw [thkln] (0,0) -- (1,0); \node [above] at (0.5,0) {$\stlcl{a}$};
\end{tikzpicture}
\;=\;
\begin{tikzpicture}[baseline=-4]
\draw (0,0.4) -- (1,0.4);
\draw (0,-0.4) -- (1,-0.4);
\node at (0.5,0.1) {$\vdots$};
\draw[decorate,decoration={brace, amplitude=4pt}]
        (1.2,0.4) -- (1.2,-0.4) node[midway, right=2pt]{$a$};
\end{tikzpicture},
\qquad
\begin{tikzpicture}[scale=0.75,baseline=-2.5]
\draw [ptzer] (-0.15,-0.6) rectangle ++(0.3,1.2);
\draw [thkln] (-.65,0) -- (-0.15,0) node [near start, above] {$\scriptstyle a$}
(0.15,0) -- (.65,0);
\end{tikzpicture}
\]

For a positive integer $\xca$ let $\pJqaL$ denote the \uclrd\ Jones polynomial of $\xL$, that is, all components of $\xL$ are colored by the same color $\xca$. A coloring of a link component by $\xca$ means that we assign the $(\xca+1)$-dimensional irreducible representation of $\SUt$ to it. Equivalently, the color $\xca$ means that the link component is $\xca$-cabled and we place the \tJWp\ on this cable.

In this paper we consider \uclrd\ links, that is, links, all of whose components are colored by the same number $\xca$.
Their colored Jones polynomial $\pJqaL$ is a Laurent polynomial of $q^2$ up to an overall factor: if $\xL$ is presented by a diagram $\xD$, then
\[
q^{\hlf\ncrD\xca^2 + \gvD \xca } \pJqaL \in\ZZtqqi.
\]

%
%\subsubsection{Quantum numbers and factorials}
%The variable $q$ in this paper is the square root of the variable $q$ used in most papers about the categorification of the Jones polynomial. In particular, for an integer $n$ the corresponding \tqnum\ and \tbqnum\ are defined as
%\[
%\qnum{n} = \frac{q^{-\frac{n}{2}} - q ^{\frac{n}{2}}}{\qmhlf-\qhlf},\qquad
%\bqnum{n} = \frac{1 - q^n}{1 - q},
%\]
%while the \tqfac\ and \tbqfac\ are defined as
%\[
%\qnumf{n} = \prod_{i=1}^n \qnum{i},\qquad
%\bqnumf{n} = \prod_{i=1}^n \bqnum{i}.
%\]
%We will use mostly the biased expressions since their behavior at $q\rightarrow 0$ is more transparent.
%%
%%they are more convenient at $q\rightarrow 0$.

\subsection{Homological notations}
Let $\caA$ be a finitely generated additive category: objects of $\caA$ are finite sums of elements of a finite set $\stA$.
%For an additive category $\caA$,
Let $\cKommA$ denote the homotopy category of its complexes bounded from below: an object of $\cKommA$ is a chain
\begin{equation}
\label{eq:complex}
\chA = ( \cdots \rightarrow A_{i+1}\rightarrow A_i \rightarrow\cdots\rightarrow A_m),
\end{equation}
where $A_i = \bigoplus_{\oba\in\stA} m_{i,\oba}\,\oba$ and $ m_{i,\oba}\in\ZZ_{\geq 0}$ are the multiplicities of generators. The notation $m$ for multiplicity is treated in this paper as an arbitrary constant, so the appearance of $m$ in different expressions does not imply that there is a relation between the multiplicities, unless it is stated specifically. The special multiplicities appearing in a presentation of the \cJWp\ are denoted by $\prmlt$.

We use a non-standard notation for the translation functor: $\shcr \chA = \chA[1]$, which allows us to define a functor $p(\shcr)$ for any polynomial $p(x)$ with integer non-negative coefficients. In particular, we use a functor ${i \brace j}_{\shcr}$ based on a combinatorial polynomial
\begin{equation}
\label{eq:combp}
{i \brace j}_{x} =
\frac
{
(1-x^{2i})(1-x^{2i-2})\cdots(1-x^{2i -2j + 2})
}
{
(1-x^2)(1-x^4)\cdots(1-x^{2j})
}.
\end{equation}

We also use a non-standard notation for the cone of two complexes:
\begin{equation}
\label{eq:cnnt}
\Cnv{\shcr \chA \xrightarrow{f} \chB} = \Conv{\chA\xrightarrow{f} \chB}.
\end{equation}
in order to emphasize the fact that the cone $\Conv{\chA\rightarrow\chB}$ can be presented as a sum $\shcr \chA \oplus \chB$ deformed by an extra differential $\chA\xrightarrow{f} \chB$. Moreover, when we work with bi-graded Khovanov complexes, there may be some confusion about which of two gradings is homological, but our non-standard notation\rx{eq:cnnt} specifies all degree shifts explicitly.

The homological order $\hmord{\obO}$ of an object $\obO\in\cKommA$ is the minimum number $m$, for which $\obO$ can be presented by a complex\rx{eq:complex}.

Consider a direct system of complexes of $\cKommA$:
$\chA_0\rightarrow\chA_1\rightarrow\cdots.$ If this system is `Cauchy', that is,
if for the cones $\chB_i =  \Conv{\chA_{i-1} \rightarrow \chA_{i}}$ there is a limit
$\lim_{i\rightarrow\infty}\hmord{\chB_i} = \infty$,
%$\lim_{i\rightarrow\infty}\hmord{ \Conv{\chA_i \rightarrow \chA_{i+1}} }=\infty$,
then there exists a direct limit $\drlmA$.

Since $\chA_{i}\hteqv \Conv{\shcr^{-1} \chB_i\rightarrow\chA_{i-1}}$, %where $\shcr$ is our non-standard notation for the translation functor $[-1]$.
the direct limit $\drlmA$ can be viewed as a result of attaching the complexes $\chB_i$ one after another to the initial complex $\chB_0=\chA_0$, hence we use the following notation for the complex $\drlmA$:
\begin{equation}
\label{eq:mtcn}
\drlmA \hteqv \boxed{
\cdots\rightarrow \chB_i \rightarrow \cdots
}_{\;i=0}^{\;\infty}
\end{equation}
In fact, if all $\chB_i$ are `homologially minimal' representatives of their equivalence classes, then the sum $\bigoplus_{i=0}^{\infty} \chB_i$ is well-defined (every chain object is finitely generated) and $\drlmA$ is homotopy equivalent to $\bigoplus_{i=0}^{\infty} \chB_i$ defomed by adding extra differentials $\chB_i\xrightarrow{f_{ij}}\chB_j$ for all pairs $i>j$.

We refer to the \rhs of \ex{eq:mtcn} as a \emph{\tmcn}, and we also use a similar notation for the complex\rx{eq:complex}: $\chA = \boxed{\cdots\rightarrow \shcr^iA_i\rightarrow\cdots}_{\;i=m}^{\;\infty}$. Note the use of the functor $\shcr$ to set explicitly the correct homological degree of the chain object $A_i$ in the \tmcn.

If a \tmcn\ $\chA$ is generated by complexes $\chB_a$:
\begin{equation}
\label{eq:lsfmc}
\chA = \Cnv{
\cdots\rightarrow \bigoplus_{j,a} m_{ij,a}\shcr^j\, \chB_a\rightarrow\cdots
}\;,
\end{equation}
(where $m_{ij,a}$ are multiplicities) but we do not care how those complexes are arranged within the \tmcn, then we use a `\lumps' notation
\[
\chA = \Pcnv{ \bigoplus_{j,a} \mtotja\,\shcr^j\,\chB_a
}, \qquad
\mtotja = \sum_i m_{ij,a},
\]
because, as a complex, $\chA$ is a sum of $\chB_a$ with total multiplicities $\sum_i m_{ij,a}$ deformed by an extra differential depicted as $\sdff$.

If the category $\caA$ is abelian, then we can compute the homology of the \tmcn\rx{eq:mtcn} with the help of the filtered complex spectral sequence. The $\xEo$ term of this spectral sequence is the sum of homologies of $\chB_i$: $\xEo = \bigoplus_{i=0}^\infty \Hm(\chB_i)$ and it is determined by the \lumps\ form of the \tmcn.
\begin{remark}
\label{rmk:bndss}
Since subsequent terms in the spectral sequence get only smaller, there is a bound on the homological order of the homology of\rx{eq:mtcn} in terms of homological orders of its constituent complexes:
\begin{equation*}
%\label{eq:mcssb}
\hlmord{
\Hm\left( \;\boxed{
\cdots\rightarrow \chB_i \rightarrow \cdots
}_{\;i=0}^{\;\infty}
\;
\right)
}
\geq
\min_{i} \hmord{\Hm(\chB_i)}.
\end{equation*}
In particular, for the \lumps\ \tmcn\rx{eq:lsfmc}
\[
\hmord{ \Hm(\chA) } \geq \min\{\hmord{\Hm(\chB_a)} + j\colon \mtotja\neq 0\}.
\]
\end{remark}

\subsection{Khovanov homology}
\hyphenation{ca-no-po-ly}
In defining Khovanov complexes~\cite{Kh99} for tangles we follow the cobordism based approach of D.~Bar-Natan~\cite{BN05}, albeit with a different grading convention. We still have two degrees: \thdgr\ $\dgh$ and \tqdgr\ $\dgq$, and we use the notations $\shcr$ and $\shfr$ for their translation functors (these functors increase the corresponding degrees by 1). The \tqdgr\ is the genuine homological degree: it takes values in $\ZZ$ and its parity determines the sign factors. The \thdgr\ is `pseudo-homological', it takes values in $\hlf\ZZ$ and it has no impact on signs, however it is the \thdgr\ shift functor $\shcr$ which is present explicitly in the \tKhbr.

In our notations, \tKhbr\ of a crossing and of the unknot are
\begin{equation}
\label{eq:dkhbr}
\xvKhv{
\zoverv{0.75}{1}
}
\;=\;
\Cnv{
\shcr^{\hlf}
\xvKhv{\zunoverv{0.75}{1}
}
\xrightarrow{\;\sgmm \; }
\shcr^{-\hlf}
\xvKhv{\zunoverh{0.75}{1}
}
}\;,\qquad
\xvKhv{
\zcirc{0.75}{1}
}
\;=\;
(\shfr + \shfr^{-1})\xalg,
\end{equation}
where $\shcr$ and $\shfr$ are degree shift functors, while
$\sgmm$ is the morphism corresponding to the saddle cobordism. Note that $\dgh\sgmm = -1$, while $\dgq\sgmm = 1$, so $\sgmm$ is odd.

Thus defined, \tKhbr\ is invariant under the first Reidemeister move only up to a degree shift:
\begin{equation}
\label{eq:frsh}
\xvKhv{
\zposfr{0.75}{1}
}\;= \; \shcr^{\hlf}\shfr \xvKhv{
\zstvr{0.75}{1}
}.
\end{equation}

Relations~\eqref{eq:dkhbr} transform into the relations~\eqref{eq:dkbr} after the substitution
$\shcr\mapsto q$, $\shfr\mapsto - q$, hence in our notations the graded Euler characteristic of \tKhom\ of a framed link equals its \tJpol:
\[
\pJqL = \sum_{i,j}(-1)^j q^{i+j} \dim\KHmvv{i}{j}(\xL).
\]

We will use the \tKhbr\ notation $\xKhv{-}$ very sparingly, because it clutters the pictures, especially when the diagrams are big. Nevertheless, we hope that the distinction between diagrams and their Khovanov complexes will be clear. Actually, we blur this distinction further by allowing the presence of \cJWp s within diagrams, since, strictly speaking, projectors are not diagrams but rather complexes within Bar-Natan's universal category.

\subsection{A \cJWp}
%
%We define \tKhbr\ of colored tangles by cabling tangle components and adding at least one  \cJWp\ to each tangle component. This means that we allow semi-infinite complexes which may extend infinitely far into positive homological degree. The \cJWp\ is depicted again as a rectangular box on a cabled line:
%$
%\begin{tikzpicture}[scale=0.5,baseline=-2.5]%[level distance=2]
%\draw [line width=1pt] (0,-0.5) rectangle (0.25,0.5);
%\draw [line width=0.75pt] (0.25,0) -- (0.75,0);
%\draw [line width=0.75pt] (-0.75,0) -- (0,0);
%\node [above] at (-0.5,0) {$\scriptstyle a$};
%\end{tikzpicture}
%$

An $(a,b)$-tangle is an embedding of circles and segments, the segment endpoints coinciding with initial $a$ points or final $b$ points. Imagine that the tangle goes from the bottom up. Depending on the position of its endpoints, the segment is either \tstrght, or a cap, or a cup. If one of its endpoints is initial and the other is final (so the segment goes straight through the tangle), then the segment is \emph{\tstrght}, if both endpoints are initial, then the segment is a \emph{cap}, and if both of its  the segments are final, then the segment is a \emph{cup}.

The \emph{\twd} $\wdv{\tau}$ of a tangle $\tau$ is the number of its \txstrs s.
An $(a,a)$ tangle has an equal number of cups and caps, we call this number a \emph{\twdfc} and denote it as $\wdfcv{\tau}$. Obviously, $\wdfcv{\tau} = \hlf(a-\wdv{\tau})$.

A \emph{\tTL} (\taTL) tangle is a flat tangle which contains no circles. Let $\sTLa$ be the set of all $(a,a)$ \taTLt s. The \cJWp\
\begin{tikzpicture}[scale=0.5,baseline=-2.5]%[level distance=2]
\draw [ptzer] (0,-0.5) rectangle (0.25,0.5);
\draw [thkln] (0.25,0) -- (0.75,0);
\draw [thkln] (-0.75,0) -- (0,0);
\node [above] at (-0.5,0) {$\scriptstyle a$};
\end{tikzpicture}
was constructed independently by Frenkel, Stroppel and Sussan~\cite{FSS11}, Cooper and Krushkal~\cite{CK10} and by the author~\cite{Ro11}. It satisfies three essential properties:
it is a projector:
\begin{equation}
\label{eq:cmppr}
\begin{tikzpicture}[menvtwo]%[scale=0.75,baseline=-2.5]
\draw [ptzer] (-0.15,-0.6) rectangle ++(0.3,1.2);
\draw [ptzer] (0.65,-0.6) rectangle ++(0.3,1.2);
%\draw [line width=\cblth] (-.65,0) -- (-0.15,0) node [near start, above] {$\scriptstyle a$}
%(0.15,0) -- (.65,0) (0.65+0.3) -- (0.65+0.3+0.5);
\draw [thkln] (-.65,0) -- (-0.15,0) node [near start, above] {$\scriptstyle a$}
(0.15,0) -- (.65,0) (0.65+0.3,0) -- (0.65+0.3+0.5,0);
\end{tikzpicture}
\;\hteqv\;
\begin{tikzpicture}[menvtwo]%[scale=0.75,baseline=-2.5]
\draw [ptzer] (-0.15,-0.6) rectangle ++(0.3,1.2);
\draw [thkln] (-.65,0) -- (-0.15,0) node [near start, above] {$\scriptstyle a$}
(0.15,0) -- (.65,0);
\end{tikzpicture}
\;,
\end{equation}
it annihilates cups and caps:
\begin{equation}
\label{eq:ccann}
\def\zext{0.4}
\begin{tikzpicture}[menvtwo]%[scale=0.75,baseline=-2.5]
\draw [ptzer] (-0.15,-0.6) rectangle ++(0.3,1.2);
\draw [thkln] (-.65-\zext,0) -- (-0.15,0) node [near start, above] {$\scriptstyle a$};
\draw [thkln] (0.15,0.45) -- (.65+\zext,0.45) node [near end,above] {$\scriptstyle b$};
\draw [thkln] (0.15,-0.45) -- (.65+\zext,-0.45) node [near end,below] {$\scriptstyle a-b-2$};
\draw (0.15,0.25) to [out=0,in=90] (0.5,0) to [out=-90,in=0] (0.15,-0.25);
\end{tikzpicture}
\;\hteqv\;
\begin{tikzpicture}[menvtwo,xscale=-1]%[scale=0.75,baseline=-2.5]
\draw [ptzer] (-0.15,-0.6) rectangle ++(0.3,1.2);
\draw [thkln] (-.65-\zext,0) -- (-0.15,0) node [near start, above] {$\scriptstyle a$};
\draw [thkln] (0.15,0.45) -- (.65+\zext,0.45) node [near end,above] {$\scriptstyle b$};
\draw [thkln] (0.15,-0.45) -- (.65+\zext,-0.45) node [near end,below] {$\scriptstyle a-b-2$};
\draw (0.15,0.25) to [out=0,in=90] (0.5,0) to [out=-90,in=0] (0.15,-0.25);
\end{tikzpicture}
\;\hteqv\; 0,
\end{equation}
and it
has a presentation as a cone of an identity braid and a complex
$
\begin{tikzpicture}[scale=0.4,baseline=-2.5]
\draw[ptone] (-0.15,-0.6) rectangle ++(0.3,1.2);
\draw [thkln] (-.65,0) -- (-0.15,0) node [near start, above] {$\scriptstyle a$}
(0.15,0) -- (.65,0);
\end{tikzpicture}
$
generated by \taTLt s with positive \twdfc\ and with non-negative \thdgr\ and \tqdgr\ shifts
\begin{equation}
\label{eq:projcn}
\begin{tikzpicture}[scale=0.75,baseline=-2.5]
\draw [ptzer] (-0.15,-0.6) rectangle ++(0.3,1.2);
\draw [thkln] (-.65,0) -- (-0.15,0) node [near start, above] {$\scriptstyle a$}
(0.15,0) -- (.65,0);
\end{tikzpicture}
\;\hteqv\;
\boxed{
\shcr
\begin{tikzpicture}[scale=0.75,baseline=-2.5]
%\draw[line width=\ljwp] (-0.15,-0.6) rectangle ++(0.3,1.2);
%\node at (0,-1.0) {$\scriptstyle \mathrm{s}$};
%\draw[line width=\ljwp] (-0.15,0) -- ++(0.15,0.6) -- ++(0.15,-0.6) -- ++(-0.15,-0.6) -- (-0.15,0);
%\draw[line width=\ljwp,fill=gray!30] (-0.15,-0.6) rectangle ++(0.3,1.2);
\draw[ptone] (-0.15,-0.6) rectangle ++(0.3,1.2);
\draw [thkln] (-.65,0) -- (-0.15,0) node [near start, above] {$\scriptstyle a$}
(0.15,0) -- (.65,0);
\end{tikzpicture}
\longrightarrow
\begin{tikzpicture}
\draw [thkln] (-0.5,0) -- (0.5,0) node [midway,above] {$\scriptstyle a$};
\end{tikzpicture}
}\;,
\end{equation}
where
\begin{equation}
\label{eq:grproj}
\begin{tikzpicture}[scale=0.75,baseline=-2.5]
%\draw[line width=\ljwp] (-0.15,-0.6) rectangle ++(0.3,1.2);
%\node at (0,-1.0) {$\scriptstyle \mathrm{s}$};
%\draw[line width=\ljwp] (-0.15,0) -- ++(0.15,0.6) -- ++(0.15,-0.6) -- ++(-0.15,-0.6) -- (-0.15,0);
%\draw[line width=\ljwp,fill=gray!30] (-0.15,-0.6) rectangle ++(0.3,1.2);
%\draw[line width=\ljwp,pattern=crosshatch] (-0.15,-0.6) rectangle ++(0.3,1.2);
\draw[ptone] (-0.15,-0.6) rectangle ++(0.3,1.2);
\draw [thkln] (-.65,0) -- (-0.15,0) node [near start, above] {$\scriptstyle a$}
(0.15,0) -- (.65,0);
\end{tikzpicture}
\;=\;
\boxed{
\cdots\longrightarrow\shcr^i\bigoplus_{\substack{0\le j \le i \\ \gamma\in\sTLa, \wdfcv{\gamma}>0}} \prmltijg\,\shfr^j\,\xKhv{\gamma}
\longrightarrow\cdots
}_{\;i=0}^{\;\infty}
\end{equation}
and $\prmltijg$ are multiplicities.

\subsection{\tKhbr\ of colored tangles}
We define \tKhbr\ of colored tangles by cabling tangle components and adding at least one  \cJWp\ to each tangle component. This means that we allow semi-infinite complexes which may extend infinitely far into positive homological degree.
%The \cJWp\ is depicted again as a rectangular box on a cabled line:
%$
%\begin{tikzpicture}[scale=0.5,baseline=-2.5]%[level distance=2]
%\draw [line width=1pt] (0,-0.5) rectangle (0.25,0.5);
%%%%%\path (1,0) node[draw,shape=rectangle] (br) {$\beta$};
%%\node[draw,shape=rectangle] at (1,0) (br) {$\beta$};
%\draw [line width=0.75pt] (0.25,0) -- (0.75,0);
%\draw [line width=0.75pt] (-0.75,0) -- (0,0);
%\node [above] at (-0.5,0) {$\scriptstyle a$};
%\end{tikzpicture}
%$

The colored \tKhbr\ is independent of the framing up to a degree shift:
\begin{equation}
\label{eq:cfrsh}
%\xvKhv{
\zcposfr{0.75}{1}
%}
\;= \; \shcr^{\hlf a^2}\shfr^a
%\xvKhv{
\zcstvr{0.75}{1}
%}
.
\end{equation}

\section{Results}
\subsection{Overview}
\subsubsection{Bounds on colored \tKhom}
Let $\xLcN$  denote the $\xca$-\uclrd\ version of the link $\xL$. For a \tBadq\ framed link $\xL$, a \emph{shifted} \tKhom\ of $\xLcN$ is defined by the formula
\begin{equation}
\label{eq:shdgt}
\tKHm(\xLcN) = \shcr^{\hlf\clN^2\ncrL} \shfr^{\gvL\xca} \KHm(\xLcN).
\end{equation}
In view of \ex{eq:afrm}, if links $\xL$ and $\xLp$ differ only by framing, then their shifted Khovanov homologies are isomorphic: $\tKHm(\xLpcN) = \tKHm(\xLcN)$.
\begin{theorem}
\label{thm:bndd}
There are bounds on degrees of shifted homology of a \tBadq\ link $\xL$:
$\tKHmvv{i}{j}(\xLcN) = 0$, if one of the following conditions is met:
%\begin{gather}
%i<0
%\\
%j< -\shlf (i + \ncrmL)
%\\
%i + j <0
%\\
%\text{$i+j= 0$ and $i>0$}
%\end{gather}
\begin{align}
i&<0
\label{eq:bd1}
\\
\label{eq:bd2}
j&< -\shlf i -\shlf \ncrmL
\\
\label{eq:bd3}
j& < -i
\\
\label{eq:bd4}
j& =-i \neq 0,
\end{align}
where $\ncrmL$ is the minimum crossing number of a diagram representing $\xL$.
Moreover,
\begin{equation}
\label{eq:endi}
\dim\tKHmvv{0}{0} = 1.
\end{equation}
%\begin{equation}
%\label{eq:hmlb}
%\begin{array}{cl}
%\tKHmvv{i}{\hem}(\xLcN)  = 0, &\text{if $i<0$,}
%\\
%\tKHmvv{i}{j}(\xLcN) = 0, & \text{if $j<-\hlf (i + \ncrmL)$},
%\\
%\,\tKHmvv{0}{0}(\xLcN) = \IQ
%\end{array}
%\end{equation}
\end{theorem}

\subsubsection{Tail of Khovanov homology}
The main result of this paper is a definition of special degree-preserving maps between shifted Khovanov homologies of \uclrd\ \tBadq\ links, such that these maps are isomorphisms at low \thdgr s.
\begin{theorem}
\label{thm:llviso}
A \tBadq\ diagram of a link $\xL$ determines a sequence of \tdgpr\ maps
\begin{equation}
\label{eq:splmps}
\tKHm(\xLcN)\xrightarrow{\mnfN}
\tKHm(\xLcNo)
\end{equation}
which are isomorphisms on $\tKHmvv{i}{\hem}(\xLcN)$  for $i\leq \xca-1$.
\end{theorem}
Let $\tlH(\xL)$ be the \drlim\ of the \drsys\ determined by the sequence of maps $\mnfN$, $\xca\in\ZZ_+$ : %\rx{eq:splmps}:
\begin{equation}
\label{eq:drlm}
\tlH(\xL) = \dlm \tKHm(\xLcN).
\end{equation}

%\begin{corollary}
%The \drsys\ of maps
%\end{corollary}
%
%The homologies and maps form a \drsys\
%\begin{equation}
%\label{eq:dirsls}
%\tKHm(\xLcv{1})\xrightarrow{\mnfv{1}}\cdots\xrightarrow{\mnfv{\xca-1}}
%\tKHm(\xLcN)\xrightarrow{\mnfN}\tKHm(\xLcNo)\xrightarrow{\mnfv{\xca+1}}\cdots
%\end{equation}
%and we define the \emph{\ttlhm} of a \tBadq\ link $\xL$ as its \drlim
%\begin{equation}
%\label{eq:drlm_o}
%\tlH(\xL) = \dlm \tKHm(\xLcN).
%\end{equation}
%\begin{theorem}
%\label{thm:llviso_o}
%For $i\leq\xca-1$, the map\rx{eq:splmps} restricted to $\tKHmvv{i}{\hem}$ is an isomorphism.
%\end{theorem}
%%%%%
\begin{corollary}
\label{cor:mpis}
The maps
\begin{equation}
\label{eq:shti}
\tKHm(\xLcN)\rightarrow\tlH(\xL)
\end{equation}
 associated with the \drlim\rx{eq:drlm} are isomorphisms on $\tKHmvv{i}{\hem}(\xLcN)$ for $i\leq\xca-1$, hence
the \drlim\ $\tlH(\xL)$ is finite-dimensional in every bi-degree, it satisfies the bounds $\tlHvv{i}{j}(\xL)=0$ at the conditions\rx{eq:bd1}--\rxw{eq:bd4} and
\begin{equation}
\label{eq:hdtl}
\dim\tlHvv{0}{0}(\xLcN) = 1.
\end{equation}
\end{corollary}
%\begin{corollary}
%The limit of the \drsys\rx{eq:dirsls} is finite-dimensional in every bi-degree:
%$\dim\tlHij(\xD)<\infty$, and it satisfies the same bounds\rx{eq:hmlb}:
%\begin{equation}
%\label{eq:hmlt}
%\begin{array}{cl}
%\tlHvv{i}{\hem}(\xLcN)  = 0, &\text{if $i<0$,}
%\\
%\tlHvv{i}{j}(\xLcN) = 0, & \text{if $j<-\hlf (i + \ncrmL)$},
%\\
%\,\tlHvv{0}{0}(\xLcN) = \IQ
%\end{array}
%\end{equation}
%\end{corollary}
%
\begin{remark}
To define the maps\rx{eq:splmps} we have to choose a diagram $\xD$ representing $\xL$, however we expect that the maps can be defined canonically, that is, independently of that choice.
\end{remark}

\subsubsection{Relation to the tail of the Jones polynomial}
The bounds\rx{eq:bd1} and\rx{eq:bd2} on $\tlH(\xL)$ mean that the graded Euler characteristic of the tail homology is well-defined, because in its presentation as an alternating sum of homology dimensions
\[
\pJqLi = \sum_{i,j} (-1)^j\,q^{i+j}\,\dim\tlHvv{i}{j}(\xD)
\]
there is only a finite number of non-trivial terms for any given value of $i+j$.

The bound\rx{eq:bd2} indicates that $\tKHmvv{i}{j}(\xLcN)$ and $\tlHvv{i}{j}(\xL)$ are trivial when $i+j< \hlf i - \hlf\ncrmL$, hence their high \thdgr s  contribute only to coefficients at high powers of $q$ in the graded Euler characteristic. Since the map\rx{eq:shti} is an isomorphism at low \thdgr s, we come to the following:
\begin{theorem}
The graded Euler characteristic of the tail homology determines the lower powers of $q$ in the \uclrd\ Jones polynomial of a \tBadq\ link:
\[
%\pJqLN
\pJqaL=
%\shcr^{\hlf\clN^2\ncrL} \shfr^{\gvL\xca}
(-1)^{\gvL\xca}
 q^{-\hlf\clN^2\ncrL-\gvL\xca} \left(\pJqLi + O\left(q^{\hlf\xca - \hlf\ncrmL}\right)\right).
\]
\end{theorem}
This means that the tail homology categorifies the tail of the \uclrd\ Jones polynomial of \tBadq\ links studied by C.~Armond~\cite{Arm11} and by S.~Garoufalidis and T.~Le~\cite{GL11}.

\subsubsection{Tail homology is determined by a \trBdg\ of a link}

%For a link diagram $\xD$ we define its \tBrdc\ $\xDp$

%In this subsection we assume that all diagrams are \tBadq.
A link diagram $\xDp$ is a \tBrdc\ of a link diagram $\xD$, if the diagram $\sBDp$ is constructed from $\sBD$ in two stages: at first stage for each pair of distinct \tBcr s of $\sBD$ connected by more that one \tstrt\ we remove all connecting \tstrt s but one; at the second stage we remove all \tBcr s which have only one \tstrt\ attached to them. Obviously, if $\xD$ is \tBadq, then so is $\xDp$.

A link $\xLp$ is a \tBrdc\ of a \tBadq\ link $\xL$, if $\xLp$ can be presented by a diagram which is a \tBrdc\ of a \tBadq\ diagram presenting $\xL$.

\begin{theorem}
\label{thm:rddg}
If a link $\xLp$ is a \tBrdc\ of a \tBadq\ link $\xL$, then their tail homologies are isomorphic: $\tlH(\xLp)\cong\tlH(\xL)$.
\end{theorem}
C.~Armond and O.~Dasbach~\cite{AD12} are working on a similar result for the tail of the \uclrd\ Jones polynomial.
\begin{corollary}
If $\xLb$ is a circular closure of a connected negative braid $\betbr$, then the tail homology of $\xLb$ is that of an unknot: $\tlH(\xLb)\cong\tlH(\bigcirc )$.
\end{corollary}
\begin{proof}
It is easy to see that $\xLb$ is \tBadq\ and a \trBdg\ of $\xLb$ consists of a single circle without \tstrt s.
\end{proof}
\begin{corollary}[Invariance under strut doubling]
If $\xL$ is a \tBadq\ link and $\xLp$ is constructed by performing a replacement
\begin{equation}
\label{eq:crdcr}
\def\mxsh{2}
\def\mysh{1.5}
\begin{tikzpicture}[menvfour]
\node (l) at (-3.2,0) {};
\node (r) at (3.2,0) {};
\path[commutative diagrams/.cd, every arrow, every label]
(l) edge[commutative diagrams/squiggly] (r);
\begin{scope}[xshift=-6cm]
\draw [thnln] (-\mxsh+0.15,\mysh) to [out=0,in=180] (\mxsh-0.15,-\mysh);
\draw [lnovr] (-\mxsh+0.15,-\mysh) to [out=0,in=180] (\mxsh-0.15,\mysh);
\draw [thnln] (-\mxsh+0.15,-\mysh) to [out=0,in=180] (\mxsh-0.15,\mysh);
\end{scope}
\begin{scope}[xshift=6cm]
\draw [thnln] (-\mxsh+0.15,\mysh) to [out=0,in=180] (\mxsh-0.15,-\mysh);
\draw [lnovr] (-\mxsh+0.15,-\mysh) to [out=0,in=180] (\mxsh-0.15,\mysh);
\draw [thnln] (-\mxsh+0.15,-\mysh) to [out=0,in=180] (\mxsh-0.15,\mysh);
\begin{scope}[xshift=4cm-0.3cm]
\draw [thnln] (-\mxsh+0.15,\mysh) to [out=0,in=180] (\mxsh-0.15,-\mysh);
\draw [lnovr] (-\mxsh+0.15,-\mysh) to [out=0,in=180] (\mxsh-0.15,\mysh);
\draw [thnln] (-\mxsh+0.15,-\mysh) to [out=0,in=180] (\mxsh-0.15,\mysh);
\end{scope}
\end{scope}
\end{tikzpicture}
\end{equation}
in a \tBadq\ diagram of $\xL$, then tail homologies of $\xL$ and $\xLp$ are isomorphic: $\tlH(\xL)\cong\tlH(\xLp)$.
\end{corollary}
\begin{proof}
Obviously, $\xL$ and $\xLp$ have the same \tBrdc.
\end{proof}
%
%\subsubsection{Invariance under \tstrt\ doubling}
%Suppose that a diagram of a link $\xLp$ is constructed from a \tBadq\ diagram of a link $\xL$ by the following replacement:
%\begin{equation}
%\label{eq:crdcr1}
%\def\mxsh{2}
%\def\mysh{1.5}
%\begin{tikzpicture}[menvfour]
%\node (l) at (-3.2,0) {};
%\node (r) at (3.2,0) {};
%\path[commutative diagrams/.cd, every arrow, every label]
%(l) edge[commutative diagrams/squiggly] (r);
%\begin{scope}[xshift=-6cm]
%\draw [thnln] (-\mxsh+0.15,\mysh) to [out=0,in=180] (\mxsh-0.15,-\mysh);
%\draw [lnovr] (-\mxsh+0.15,-\mysh) to [out=0,in=180] (\mxsh-0.15,\mysh);
%\draw [thnln] (-\mxsh+0.15,-\mysh) to [out=0,in=180] (\mxsh-0.15,\mysh);
%\end{scope}
%\begin{scope}[xshift=6cm]
%\draw [thnln] (-\mxsh+0.15,\mysh) to [out=0,in=180] (\mxsh-0.15,-\mysh);
%\draw [lnovr] (-\mxsh+0.15,-\mysh) to [out=0,in=180] (\mxsh-0.15,\mysh);
%\draw [thnln] (-\mxsh+0.15,-\mysh) to [out=0,in=180] (\mxsh-0.15,\mysh);
%\begin{scope}[xshift=4cm-0.3cm]
%\draw [thnln] (-\mxsh+0.15,\mysh) to [out=0,in=180] (\mxsh-0.15,-\mysh);
%\draw [lnovr] (-\mxsh+0.15,-\mysh) to [out=0,in=180] (\mxsh-0.15,\mysh);
%\draw [thnln] (-\mxsh+0.15,-\mysh) to [out=0,in=180] (\mxsh-0.15,\mysh);
%\end{scope}
%\end{scope}
%\end{tikzpicture}
%\end{equation}
%Obviously, $\xLp$ is also \tBadq, its \tBdg\ being obtained from that of $\xL$ by doubling the \tstrt, corresponding to the crossing.
%\begin{theorem}
%\label{thm:strdb}
%The tail homologies of $\xL$ and $\xLp$ are isomorphic: $\tlH(\xL)\cong\tlH(\xLp)$.
%\end{theorem}
The latter corollary is a categorification of a similar property of the tail of the \uclrd\ Jones polynomial observed by C.~Armond and O.~Dasbach~\cites{AD11,Arm11}. This property suggests that a single crossing plays the role of a \cJWp\ in the tail homology. In order to make this statement precise, tail homology has to be defined for knotted graphs, which may include both finite and infinite colors, so that the essential property of contracting cups/caps can be formulated. We hope to address this issue in a subsequent paper. Meanwhile, we prove in Appendix that for large $\xca$ a crossing of two $\xca$-cables is, indeed, homologically close to the \tJWp\ placed on two parallel $\xca$-cables.

\subsection{Technicalities}

We prove most statements of the previous subsection not just for \tBadq\ links, but for any link diagram. However, we conjecture that the results are trivial for \tBiadq\ diagrams, because their tail homology is trivial, if defined as a \drlim\ of the system that we construct. We expect that \tBiadq\ links also have tail homology, but the proof that the tail of their Khovanov homology stabilizes in the limit of large color requires new ideas.

\subsubsection{Shifted Khovanov homology}
Let $\xD$ be a diagram of a tangle which may include single lines, cables and \tJWp s. We define $\yncrv{\xD}$ to be the total number of single line crossings in $\xD$  (that is, a crossing between an $a$-cable and a $b$-cable contributes $ab$ to $\yncrv{\xD}$). The following is an obvious corollary of \ex{eq:dkhbr} and the fact that, according to\rx{eq:projcn} and\rx{eq:grproj},
$\hmord{
\begin{tikzpicture}[scale=0.4,baseline=-2.5]
\draw[line width=\ljwp] (-0.15,-0.6) rectangle ++(0.3,1.2);
\draw [line width=\cblth] (-.65,0) -- (-0.15,0) %node [near start, above] {$\scriptstyle a$}
(0.15,0) -- (.65,0);
\end{tikzpicture}
}=0$:
\begin{theorem}
\label{thm:smfr}
The complex $\xKhv{\xD}$ has a lower homological bound:
$\hmord{\xKhv{\xD}}\geq - \hlf\yncrv{\xD}.$
\end{theorem}

Let $\xD$ be a diagram of a link which may include single lines, cables and \tJWp s. Define $\ynccv{\xD}$ to be the total number of circles in the diagram constructed from $\xD$ by replacing the \tJWp s with identity braids and performing \Bsplng s on all crossings. Now we define the shifted Khovanov homology of $\xD$:
\[
\tKHm(\xD) = \shcr^{\hlf\yncrv{\xD}} \shfr^{\ynccv{\xD}} \KHm(\xD).
\]
The following is a particular case of Theorem\rw{thm:smfr}:
\begin{theorem}
\label{thm:eaest}
If $i<0$, then $\tKHmvv{i}{\hem}(\xD)=0$.
\end{theorem}

For a link diagram $\xD$ let $\xDclN$ denote the corresponding \uclrd\ diagram (that is, every link components is $\xca$-cabled and contains at least one \tJWp).
Then, obviously, $\yncrv{\xDclN} = \clN^2\ncrD$ and
$\ynccv{\xDclN} =\xca \gvD$, so
\begin{equation}
\label{eq:shdgd}
\tKHm(\xDclN) = \shcr^{\hlf\clN^2\ncrD} \shfr^{\gvD\xca} \KHm(\xDclN).
\end{equation}
Hence, if $\xD$ is \tBadq\ and represents a link $\xL$, then $\tKHm(\xDclN)$ coincides with the shifted homology defined by \ex{eq:shdgt}.

Now Theorem\rw{thm:bndd} is a corollary of Theorem\rw{thm:eaest} and the following:
\begin{theorem}
\label{thm:kqbnd}
A shifted homology of a \uclrd\ diagram $\xDclN$ has a bound:
%For any link diagram $\xD$, the shifted homology has a bound:
$\tKHmvv{i}{j}(\xDclN)=0$ if one of the following conditions is satisfied:
\begin{align}
%i&<0
%\label{eq:bd1a}
%\\
\label{eq:bd2a}
j&<-\shlf i - \shlf\ncrD - \sthlf\ncriD
\\
\label{eq:bd3a}
j& < -i - \ncriD
%\\
%\label{eq:bd4a}
%j& =-i \neq 0\quad\text{and $\xD$ is \tBadq.}
\end{align}
Moreover, if $\xD$ is \tBadq, then
\begin{equation}
\label{eq:cndbd}
\dim \tKHmvv{i}{-i} =
\begin{cases}
0,&\text{if $i>0$,}
\\
%\xalg,
1,
&\text{if $i=0$.}
\end{cases}
\end{equation}
\end{theorem}
%\begin{theorem}
%If $\xD$ is \tBadq, then $\dim\tKHmvv{0}{0}(\xD) = 1$.
%\end{theorem}
%%
%\begin{theorem}
%\label{thm:bound}
%If $i<0$ or $i+j<0$ then $\tKHmvv{i}{j}(\xDclN)=0$. If $\xD$ is \tBadq, then $\tKHmvv{0}{0}(\xDclN) = \IQ$.
%\end{theorem}
Theorem\rw{thm:llviso} is a special case of the following:
\begin{theorem}
\label{thm:leviso}
For any link diagram $\xD$ there is a sequence of \tdgpr\ maps
\begin{equation}
\label{eq:spmaps}
%\KHmijk(\xDclN) \rightarrow
\tKHm(\xDclN)\xrightarrow{\mnfN}
%\shcr^{(2\clN + 1)\ncrD}\shfr^{\gvD}
\tKHm(\xDclNo),
\end{equation}
which are isomorphisms on $\tKHmvv{i}{\hem}$ for $i\leq \xca-1$.
\end{theorem}
This theorem implies that the \drsys\ formed by maps\rx{eq:spmaps} has a limit
\begin{equation}
\label{eq:dglim}
\tlH(\xD) = \dlm \tKHm(\xDclN).
\end{equation}
which is finite-dimensional in every bi-degree.
\begin{conjecture}
If the diagram $\xD$ is not \tBadq, then the \drlim\rx{eq:dglim} is trivial: $\tlH(\xD)=0$.
\end{conjecture}

\subsection{Discussion}
We conjecture that \tBadq\ links have a tri-graded homology $\ttlH(\xL)$, which has an additional `\tbgrd', whose zero-degree part coincides with the tail homology $\tlH(\xL)$:
\[
\ttlH(\xL) = \bigoplus_{\substack{i,j\\k\geq 0}}\ttlHvvv{i}{j}{k}(\xL),\qquad
\ttlHvvv{i}{j}{0}(\xL) = \tlHvv{i}{j}(\xL).
\]
This homology should have a family of mutually anti-commuting differentials $\xfdN$, $\xca\in\ZZ_+$, with degrees $\dgq\xfdN=-1$, $\dgh\xfdN=1$ and $\dgb=-1$
such that homology of $\ttlH(\xL)$ with respect to $\xfdN$ matches the shifted Khovanov homology $\tKHm$ up to a level proportional to $\xca^2$, after the \tbdgr\ is converted into \thdgr:
\begin{equation}
\label{eq:trconj}
\tKHmvv{\xti}{\hem}(\xL) = \bigoplus_{i+\xca k = \xti}\Hm_{i,\hem,k}^{\xfdN}\big(
\ttlH(\xL)\big),\quad\text{if $\xti\geq a_{\xL}\xca^2$},
\end{equation}
where $a_{\xL}$ is a constant determined by $\xL$.

There are three reasons to formulate this conjecture. The first reason is that the proof of Theorem\rw{thm:llviso} is based on numerous long exact sequences\rx{eq:les}, in which the `correction homology' starts at homological degree proportional to $\xca$.

The second reason is that, according to Garoufalidis and Le~\cite{GL11}, the tail of the Jones polynomial of \tBadq\ links has a `telescopic' structure. They show that if $\xL$ is alternating, then there exists a family of Laurent series $\Phi_n(q) = \sum_m a_m q^m$ such that for any $k>0$ the combined series $F_k(q) = \sum_{n=0}^k \Phi_n(q)$ is a better approximation for the tail of the colored Jones polynomial that just the first term $\Phi_0(q)= \pJqLi $. With the help of the colored Kauffman bracket (\cf\rx{eq:colKhbr}) we can prove a similar result for all \tBadq\ links and we expect that the 2-variable series
\begin{equation}
\label{eq:dsrs}
\ptJqLi = \sum_{n=0}^\infty b^n \Phi_n(q)
\end{equation}
is the bi-graded Euler characteristic of the tri-graded homology $\ttlH(\xL)$.

The third reason for our conjecture comes from the paper by Gukov and \Stosic~\cite{GS11}. Based on QFT models of Khovanov homology, they suggest that its dependence on color should be similar to the dependence of the $\SUv{n}$ homology on $n$: this homology  may be presented as a homology of a special differential acting on $\SUv{N}$ homology if $N>n$. We suggest to go half step further. Ultimately, the $\SUv{N}$ homology may be presented, at least, conjecturally, from the tri-graded HOMFLY-PT homology with the help of special differentials $\xfdv{n}$ and we expect that a similar process may work for the tail homology.

We expect that the formation of a stable tail of a \uclrd\ \tBadq\ link is a general feature which originates in the tri-graded homology when the Young diagram describing the color has a very large value of one of the differences between the lengths of rows or columns. In particular, it could be easy to follow the tail formation in case when the diagram consists of a single very large column.

Witten suggested~\cite{Wi11} that a series of the form\rx{eq:dsrs} should represent the graded Euler characteristic of Khovanov homology in the background of a flat $\mathrm{U}(1)$-reducible $\SUt$ connection in the link complement. We conjecture that if a link can be presented as a circular closure of a totally negative braid, then the tail homology coincides with the one related to the flat $\mathrm{U}(1)$-reducible $\SUt$ connection.

\subsection{Acknowledgements}
The author thanks Eugene Gorsky for sharing the results of his unfinished research and, in particular, the conjecture about the structure of Khovanov homology of a colored unknot and its stabilization in the limit of high color.

This work was supported in part by the NSF grants DMS-0808974 and DMS-1108727.

\section{Five tools}

The proof of Theorems\rw{thm:kqbnd} and\rw{thm:leviso} requires five tools: \ltrf s, purging, braid straightening, colored \tKhbr\ and recurrence relations between \cJWp s. \Ltrf s relate homologies of similar diagrams. Purging gets rid of redundant \taTLt s in complexes, which contain \tJWp s, thus improving estimates of homological order. Straightening a braid is a simple observation that braiding within a cable attached to a projector results only in degree shifts. The colored \tKhbr\ is a special presentation of a crossing of two cables attached to \tJWp s. Finally, recurrence relations relate \tJWp s on $\xca$ and $\xca+1$ strands.

\subsection{\Lrpl s and \ltrf s}
A \emph{\lrpl} is a pair of tangles $\ytngsi$ and $\ytngsf$, which may contain single and cabled lines, as well as \tJWp s. Both tangles should have the same sets of incoming legs and the same sets of outgoing legs. Hence if an initial diagram $\xDsi$ contains the tangle $\ytngsi$ attached by its legs to the rest of the diagram, then we can construct a final diagram $\xDsf$ by replacing $\ytngsi$ with $\ytngsf$. If $\ytngsi$ or $\ytngsf$ is not an actual diagram, but rather a complex of diagrams within the universal category, the \lrpl\ still makes sense as a construction of $\xKhv{\xDsf}$ from $\xKhv{\xDsi}$.

A \emph{\ltrf} is a \lrpl\ together with a specified \tdgpr\ morphism $\ytngkfp\xrightarrow{\xmg}\ytngki$, where we use a shortcut $\ytngkfp= \shfr^{\xhshf}\shcr^{\xqshf}\ytngkf$\footnote{The unnatural direction of morphism is chosen for future convenience.} The morphism $\xmg$ determines the presentation of $\ytngki$ as a cone
\begin{equation}
\label{eq:cnrel}
\ytngki \hteqv
\boxed{
\ytngkc
\rightarrow\ytngkfp
}\;,
\qquad\text{where}\quad
\ytngkc = \boxed{
\shcr\ytngkfp\xrightarrow{\xmg}\ytngki
}\;.
%\qquad
%\ytngsfp = \shfr^{\xhshf}\shcr^{\xqshf}\ytngsf,\quad
%\ytngscp = \shfr^{\xhshc}\shcr^{\xqshc}\ytngsc.
\end{equation}
Up to a degree shift, the `correction' complex $\ytngkc$ may be the categorification complex of an actual tangle $\ytngsc$, or it may be just a convenient shortcut.

Let $\xDsi$ be a diagram of a link which contains $\ytngsi$ and let $\xDsf$ and $\xDsc$ be the diagrams constructed by replacing $\ytngsi$ with $\ytngsf$ and $\ytngsc$. The relations\rx{eq:cnrel} imply a long exact sequence
\begin{equation}
\label{eq:les}
\shcr^{-1}\KHm(\xDsc) \longrightarrow  \shfr^{\xhshf}\shcr^{\xqshf}\KHm(\xDsf) \xrightarrow{\;\;\xmg\;\;} \KHm(\xDsi)
\longrightarrow \KHm(\xDsc).
\end{equation}
For all \ltrf s considered in this paper, there are relations
\[
\xqshf = \shlf\yncrv{\xDsf} - \shlf\yncrv{\xDsi},\qquad
\xhshf = \ynccv{\xDsf} - \ynccv{\xDsi},
\]
hence \ex{eq:les} turns into the following sequence of \tdgpr\ maps between shifted homologies:
\[
\shcr^{-1} \KHm(\xDscp) \longrightarrow
%\shfr^{\xhshf}\shcr^{\xqshf}
\tKHm(\xDsf) \xrightarrow{\;\;\xmg\;\;} \tKHm(\xDsi)
\longrightarrow \KHm(\xDscp),
\]
where $
\KHm(\xDscp) = \shcr^{\shlf\yncrv{\xDsi}}\shfr^{\ynccv{\xDsi}}\KHm(\xDsc)$. This exact sequence implies the following:
%\begin{proposition}
%\label{prp:gestdg}
%If $\KHmib(\xDsc)=0$ for $i \leq \hbnd-1$, then
%the \tdgpr\ map
%\begin{equation}
%\label{eq:dgprmg}
%\tKHm(\xDsf) \xrightarrow{\;\xmg\;} \tKHm(\xDsi),
%\end{equation}
%\end{proposition}
\begin{proposition}
\label{prp:gestdg}
If $\KHmib(\xDsc)=0$ for $i \leq \hbnd-1$, then
the \tdgpr\ map
\begin{equation}
\label{eq:dgprmg}
\tKHm(\xDsf) \xrightarrow{\;\xmg\;} \tKHm(\xDsi),
\end{equation}
is an isomorphism on $\tKHmvv{i}{\hem}$ for
\begin{equation}
\label{eq:ineqo}
i\leq\hbnd+
%\xqshc+
\shlf\yncrv{\xDsi}-2.
\end{equation}
\end{proposition}

%
%
%%where we use a shortcut $\KHm(\xDsfp) = \shfr^{\xhshf}\shcr^{\xqshf}\KHm(\xDsf)$.
%
%
%The morphism $\xmg$ translates into a \tdgpr\ map
%\begin{equation}
%\label{eq:mpxmg}
%\KHm(\xDsfp) \xrightarrow{\xmg} \KHm(\xDsi),\qquad
%\xDsfp = \shfr^{\xhshf}\shcr^{\xqshf} \xDsf,
%\end{equation}
%which
%\begin{equation}
%\label{eq:2les}
%\shcr^{-1}\KHm(\xDscp) \longrightarrow \KHm(\xDsfp) \xrightarrow{\;\;\xmg\;\;} \KHm(\xDsi)
%\longrightarrow \KHm(\xDscp).
%\end{equation}

\subsection{Purging}

Purging is a process of using \ex{eq:ccann} to remove constituent \taTLt s of a complex, whose cups or caps are connected directly to a \tJWp. Let $\sTLab$ be the set of all (a,b) \taTLt s and let
$\sTLabc =\{ \gamma\in\sTLabc\colon \swdv{\gamma} = b\}$. In other words, $\sTLabc$ is a subset of $(a,b)$ \taTLt s, which contain no cups, but only caps and \txstrs s.
\begin{proposition}
\label{prp:purge}
There is a homotopy equivalence
\begin{multline}
\label{eq:bprg}
\boxed{
\cdots\longrightarrow\shcr^i\bigoplus_{\substack{j \\ \gamma\in\sTLab}} m_{ij,\gamma}\,\shfr^j\,
%\xKhv{\gamma}
\begin{tikzpicture}[menvtwo]
\draw [ptzer] (-0.6,-0.6) rectangle (0.6,0.6); \node at (0,0) {$\gamma$};
\draw [thkln] (-1.2,0) -- (-0.6,0) node [near start,below] {$\scriptstyle a$};
\draw [thkln] (1.2,0) -- (0.6,0) node [near start,below] {$\scriptstyle b$};
\end{tikzpicture}
\longrightarrow\cdots
}
%_{\;i=0}^{\;\infty}
\begin{tikzpicture}[menvtwo]
\path [use as bounding box] (-1,-1) rectangle (2,1);
\draw [ptzer] (-0.15,-0.6) rectangle (.15,.6);
\draw [thkln] (-0.75-0.22,0) -- (-0.15,0)
node [midway,below] {$\scriptstyle b$}
(0.15,0) -- (.75,0)
node [near end,below] {$\scriptstyle b$};
%\draw (-0.15,0.4) to [out=180,in=-90] (-0.6,0.8) to [out=90,in=180] (0,1.4)
%to [out=0,in=90] (0.6,0.8) to [out=-90,in=0] (0.15,0.4);
\end{tikzpicture}
\\
\hteqv\;
\boxed{
\cdots\longrightarrow\shcr^i\bigoplus_{\substack{j \\ \gamma\in\sTLabc}} m_{ij,\gamma}\,\shfr^j\,
%\xKhv{\gamma}
\begin{tikzpicture}[menvtwo]
\draw [ptzer] (-0.6,-0.6) rectangle (0.6,0.6); \node at (0,0) {$\gamma$};
\draw [thkln] (-1.2,0) -- (-0.6,0) node [near start,below] {$\scriptstyle a$};
\draw [thkln] (1.8,0) -- (0.6,0) node [midway,below] {$\scriptstyle b$}
(2.1,0) -- (2.7,0) node [near end,below] {$\scriptstyle b$};
\draw[ptzer] (1.8,-0.6) rectangle (2.1,0.6);
\end{tikzpicture}
\longrightarrow\cdots
}
\end{multline}
\end{proposition}

%$\sTLabc\subset\sTLab$ be the subset of tangles,

\subsection{Straightening a braid attached to a \tJWp}

\begin{theorem}
The categorification complex of the tangle composition of the \tJWp\ with a braid $\brbet$ is homotopy equivalent to the shifted complex of the \tJWp:
\[
%\xvKhv{
%\begin{tikzpicture}[menvtwo]%[baseline=-2.5]%[level distance=2]
%\draw [line width=1pt] (0,-0.5) rectangle (0.25,0.5);
%%%%%\path (1,0) node[draw,shape=rectangle] (br) {$\beta$};
%\node[draw,shape=rectangle,ptzer] at (1,0) (br) {$\beta$};
%\draw [line width=0.75pt] (0.25,0) -- (br);
%\draw [line width=0.75pt] (-0.75,0) -- (0,0);
%\node [above] at (-0.5,0) {$\scriptstyle a$};
%\draw [line width=0.75pt] (br) -- (1.5,0);
%\end{tikzpicture}
\begin{tikzpicture}[menvtwo]
\draw [ptzer] (-0.6,-0.6) rectangle (0.6,0.6); \node at (0,0) {$\brbet$};
\draw [thkln] (-2.7,0) -- (-2.1,0) node [near start,below] {$\scriptstyle a$} (-1.8,0) -- (-0.6,0)
 node [midway, below] {$\scriptstyle a$} (0.6,0) -- (1.2,0) node [near end, below] {$\scriptstyle a$};
\draw [ptzer] (-2.1,-0.6) rectangle ++(0.3,1.2);
\end{tikzpicture}
%}
\;\hteqv\;
\shcr^{\hlf(\xnp-\xnm)}
%\xvKhv{
\begin{tikzpicture}[menvtwo]
\draw [ptzer] (-0.15,-0.6) rectangle ++(0.3,1.2);
\draw [thkln] (-0.75,0) -- (-0.15,0) node [near start,below] {$\scriptstyle a$}
(0.15,0) -- (0.75,0);
\end{tikzpicture}
%\begin{tikzpicture}[menvtwo]%[baseline=-2.5]%[level distance=2]
%\draw [line width=1pt] (0,-0.5) rectangle (0.25,0.5);
%%%%%\path (1,0) node[draw,shape=rectangle] (br) {$\beta$};
%%\node[draw,shape=rectangle] at (1,0) (br) {$\beta$};
%\draw [line width=0.75pt] (0.25,0) -- (0.75,0);
%\draw [line width=0.75pt] (-0.75,0) -- (0,0);
%\node [above] at (-0.5,0) {$\scriptstyle a$};
%\end{tikzpicture}
%}
\]
where $\xnp$ ($\xnm$) is the number of positive (negative) elementary crossings in a presentation of $\brbet$.
\end{theorem}

\begin{proof}
It is sufficient to prove this equivalence in the case of $\xnp=1$ and $\xnm=0$ (the case of $\xnp=0$ and $\xnm=1$ is similar and other cases can be proved by consequent composition of elementary crossings).
Thus we replace the positive crossing by its Khovanov complex
\[
%\xvKhv{
\begin{tikzpicture}[xscale=0.6,yscale=-0.6,baseline=-2.5]
\draw [thkln] (-0.5,0) -- (0.1,0);
\draw (0.5,0.3) to [out=0,in=180] (2,-0.3);
\draw [line width=6pt, draw=white] (0.5,-0.3) to [out=0,in=180] (2,0.3);
\draw (0.5,-0.3) to [out=0,in=180] (2,0.3);
\draw [thkln] (0.5,0.6) -- (2,0.6);
\draw [thkln] (0.5,-0.6) -- (2,-0.6);
\draw [ptzer] (0.1,-0.8) rectangle (0.5,0.8);
\end{tikzpicture}
%}
\; = \;
\Cnv{
\shcrh\;
%\xvKhv{
\begin{tikzpicture}[scale=0.6,baseline=-2.5]
\draw [thkln] (-0.5,0) -- (0.1,0);
%\draw (0.5,0.3) -- (0.75,0.3) .. controls +(0:0.4) and +(0:0.4) .. (0.75,-0.3) -- (0.5,-0.3);
%\draw (2,0.3) -- (1.75,0.3) .. controls +(180:0.4) and +(180:0.4) .. (1.75,-0.3) -- %(2,-0.3);
\draw (0.5,0.3) -- (0.75,0.3) to [out=0,in=180] (1.25,0.15) to [out=0,in=180] (1.75,0.3) -- (1.75,0.3) -- (2,0.3);
\draw (0.5,-0.3) -- (0.75,-0.3) to [out=0,in=180] (1.25,-0.15) to [out=0,in=180] (1.75,-0.3) -- (1.75,-0.3) -- (2,-0.3);
%\draw (0.5,0.3) to [out=0,in=180] (2,-0.3);
%\draw [line width=6pt, draw=white] (0.5,-0.3) to [out=0,in=180] (2,0.3);
%\draw (0.5,-0.3) to [out=0,in=180] (2,0.3);
\draw [thkln] (0.5,0.6) -- (2,0.6);
\draw [thkln] (0.5,-0.6) -- (2,-0.6);
\draw [ptzer] (0.1,-0.8) rectangle (0.5,0.8);
\end{tikzpicture}
%}
\xrightarrow{\hspace{0.8cm}}
\shcrmh\;
%\xvKhv{
\begin{tikzpicture}[scale=0.6,baseline=-2.5]
\draw [thkln] (-0.5,0) -- (0.1,0);
\draw (0.5,0.3) -- (0.75,0.3) .. controls +(0:0.4) and +(0:0.4) .. (0.75,-0.3) -- (0.5,-0.3);
\draw (2,0.3) -- (1.75,0.3) .. controls +(180:0.4) and +(180:0.4) .. (1.75,-0.3) -- (2,-0.3);
%\draw (0.5,0.3) to [out=0,in=180] (2,-0.3);
%\draw [line width=6pt, draw=white] (0.5,-0.3) to [out=0,in=180] (2,0.3);
%\draw (0.5,-0.3) to [out=0,in=180] (2,0.3);
\draw [thkln] (0.5,0.6) -- (2,0.6);
\draw [thkln] (0.5,-0.6) -- (2,-0.6);
\draw [ptzer] (0.1,-0.8) rectangle (0.5,0.8);
\end{tikzpicture}
%}
}
\]
and observe that the second term in the resulting cone is contractible.
\end{proof}

This proposition has important special cases:
\begin{equation}
\label{eq:projtw}
%\xvKhv{
\begin{tikzpicture}[scale=0.75,baseline=-3]
\draw [thkln] (-0.55,0)
%node[above] {$\scriptstyle a+1$}
-- (0.2,0)
node [near start,above] {$\scriptstyle a+1$};
\draw  [thkln] (0.5,0.3) to [out=0,in=180] (1.25,-0.3)
node[right] {$\scriptstyle a$};
\draw [lnovr] (0.5,-0.3) to [out=0,in=180] (1.25,0.3);
\draw (0.5,-0.3) to [out=0,in=180] (1.25,0.3);
\draw [ptzer] (0.2,-0.6) rectangle (0.5,0.6);
\end{tikzpicture}
%}
\;\hteqv\;
\shcr^{\frac{a}{2}}
%\xvKhv{
\begin{tikzpicture}[scale=0.75,baseline=-3]
\draw [line width=\cblth] (-0.55,0)
%node[above] {$\scriptstyle a+1$}
-- (0.2,0)
node [midway,above] {$\scriptstyle a+1$\;};
\draw [line width=\cblth] (0.5,0) -- (1.25,0);
%\draw  [line width=\cblth] (0.5,0.3) to [out=0,in=180] (1.25,-0.3);
%\draw [line width=6pt, draw=white] (0.5,-0.3) to [out=0,in=180] (1.25,0.3);
%\draw (0.5,-0.3) to [out=0,in=180] (1.25,0.3);
\draw [line width=1pt] (0.2,-0.6) rectangle (0.5,0.6);
\end{tikzpicture}
%}
,\qquad
%\xvKhv{
\begin{tikzpicture}[scale=0.75,baseline=-3]
\draw [line width=\cblth] (-0.55,0)
%node[above] {$\scriptstyle a+1$}
-- (0.2,0)
node [near start,above] {$\scriptstyle a+1$};
\begin{scope}[yscale=-1]
\draw  [line width=\cblth] (0.5,0.3) to [out=0,in=180] (1.25,-0.3)
node [right] {$\scriptstyle a$};
\draw [line width=6pt, draw=white] (0.5,-0.3) to [out=0,in=180] (1.25,0.3);
\draw (0.5,-0.3) to [out=0,in=180] (1.25,0.3);
\end{scope}
\draw [line width=1pt] (0.2,-0.6) rectangle (0.5,0.6);
\end{tikzpicture}
%}
\;\hteqv\;
\shcr^{-\frac{a}{2}}
%\xvKhv{
\begin{tikzpicture}[scale=0.75,baseline=-3]
\draw [line width=\cblth] (-0.55,0)
%node[above] {$\scriptstyle a+1$}
-- (0.2,0)
node [midway,above] {$\scriptstyle a+1$\;};
\draw [line width=\cblth] (0.5,0) -- (1.25,0);
%\draw  [line width=\cblth] (0.5,0.3) to [out=0,in=180] (1.25,-0.3);
%\draw [line width=6pt, draw=white] (0.5,-0.3) to [out=0,in=180] (1.25,0.3);
%\draw (0.5,-0.3) to [out=0,in=180] (1.25,0.3);
\draw [line width=1pt] (0.2,-0.6) rectangle (0.5,0.6);
\end{tikzpicture}
%}
\end{equation}
and two similar cases with opposite powers of $\shcr$ when the cable runs over the single line.

%\section{Preliminaries}
%\section{The morphism $\mnfN$}
%\label{sct:mrph}

\subsection{Colored \tKhbr}
%\subsubsection{Single strand splicing}
\begin{theorem}[Single strand splicing]
\label{lm:sss}
A \tKhbr\ of the crossing of two equally colored strands can be presented as a cone:
\begin{equation}
\label{eq:colssKhbr}
\begin{tikzpicture}[scale=0.5,rotate=-45,baseline=-3.5]
\draw [thkln] (-1.5,0) -- (1.5,0);
\draw [lnovr] (0,-1.5) -- (0,1.5);
\draw [thkln] (0,-1.5) -- (0,1.5);
%
%\draw [line width=\ljwp] (-2.2,0) node[left]{$\scriptstyle a$} -- (-1.8,0);
%\draw [line width=\ljwp] (1.8,0) -- (2.2,0) node[right]{$\scriptstyle a$};
%\draw [line width=\ljwp] (0,1.8) -- (0,2.2) node[above]{$\scriptstyle a$};
%\draw [line width=\ljwp] (0,-1.8) -- (0,-2.2)node[below]{$\scriptstyle a$};
%
\draw [thkln] (-2.2,0)  -- (-1.8,0);
\node at (-2.6,0) {$\scriptstyle a+1$};
\draw [thkln] (1.8,0) -- (2.2,0);
\node at (2.6,0) {$\scriptstyle a+1$};
\draw [thkln] (0,1.8) -- (0,2.2);
\node at (0,2.6) {$\scriptstyle a+1$};
\draw [thkln] (0,-1.8) -- (0,-2.2);
\node at (0,-2.6) {$\scriptstyle a+1$};
\draw [thkln] (-1.8,-0.6) rectangle ++(0.3,1.2);
\draw [thkln] (1.5,-0.6) rectangle ++(0.3,1.2);
\draw [thkln] (-0.6,1.5) rectangle ++(1.2,0.3);
\draw [thkln] (-0.6,-1.5) rectangle ++(1.2,-0.3);
\end{tikzpicture}
%%%%%%%%%%%%%%%%%%
\hteqv
\;\;
\Cnv{
\shcr^{a+\hlf}
\begin{tikzpicture}[scale=0.5,rotate=-45,baseline=-3.5]
\draw  (-1.5,-0.3) to [out=0, in=90] (-0.3,-1.5);
\draw  (1.5,0.3) to [out=180, in=-90] (0.3, 1.5);
%\draw [line width=\ljwp] (1.5,-0.3) to [out=180, in=90] (0.3,-1.5);
\draw [thkln] (-1.5,0) -- (1.5,0);
\draw [line width=\ocrw,draw=white] (0,-1.5) -- (0,1.5);
\draw [thkln] (0,-1.5) -- (0,1.5);
\draw [thkln] (-2.2,0)  -- (-1.8,0);
\node at (-2.6,0) {$\scriptstyle a+1$};
\draw [thkln] (1.8,0) -- (2.2,0);
\node at (2.6,0) {$\scriptstyle a+1$};
\draw [thkln] (0,1.8) -- (0,2.2);
\node at (0,2.6) {$\scriptstyle a+1$};
\draw [thkln] (0,-1.8) -- (0,-2.2);
\node at (0,-2.6) {$\scriptstyle a+1$};
\draw [ptzer] (-1.8,-0.6) rectangle ++(0.3,1.2);
\draw [ptzer] (1.5,-0.6) rectangle ++(0.3,1.2);
\draw [ptzer] (-0.6,1.5) rectangle ++(1.2,0.3);
\draw [ptzer] (-0.6,-1.5) rectangle ++(1.2,-0.3);
\end{tikzpicture}
\longrightarrow
\shcr^{-(a+\hlf)}
\begin{tikzpicture}[scale=0.5,rotate=-45,baseline=-3.5]
%\node at (135:0.65) {$\scriptstyle a-i$};
%\node at (135:-0.65) {$\scriptstyle a-i$};
%\node at (45:1.5) {$\scriptstyle i$};
%\node at (45:-1.5) {$\scriptstyle i$};
\draw (-1.5,0.3) to [out=0, in=-90] (-0.3,1.5);
%\draw [line width=\ljwp] (-1.5,-0.3) to [out=0, in=90] (-0.3,-1.5);
%\draw [line width=\ljwp] (1.5,0.3) to [out=180, in=-90] (0.3, 1.5);
\draw (1.5,-0.3) to [out=180, in=90] (0.3,-1.5);
\draw [thkln] (-1.5,0) -- (1.5,0);
\draw [line width=\ocrw,draw=white] (0,-1.5) -- (0,1.5);
\draw [thkln] (0,-1.5) -- (0,1.5);
\draw [thkln] (-2.2,0)  -- (-1.8,0);
\node at (-2.6,0) {$\scriptstyle a+1$};
\draw [thkln] (1.8,0) -- (2.2,0);
\node at (2.6,0) {$\scriptstyle a+1$};
\draw [thkln] (0,1.8) -- (0,2.2);
\node at (0,2.6) {$\scriptstyle a+1$};
\draw [thkln] (0,-1.8) -- (0,-2.2);
\node at (0,-2.6) {$\scriptstyle a+1$};
\draw [ptzer] (-1.8,-0.6) rectangle ++(0.3,1.2);
\draw [ptzer] (1.5,-0.6) rectangle ++(0.3,1.2);
\draw [ptzer] (-0.6,1.5) rectangle ++(1.2,0.3);
\draw [ptzer] (-0.6,-1.5) rectangle ++(1.2,-0.3);
\end{tikzpicture}
}
\end{equation}
\end{theorem}
\begin{proof}
Split off a single strand from each crossing cable, apply the \tKhbr\ relation\rx{eq:dkhbr} to their crossing:
\[
\begin{tikzpicture}[scale=0.5,rotate=-45,baseline=-3.5]
\draw [thkln] (-1.5,0.4) -- (1.5,0.4);
\draw (-1.5,-0.4) -- (1.5,-0.4);
\draw [line width=\ocrw,draw=white] (-0.4,-1.5) -- (-0.4,1.5);
\draw [thkln] (-0.4,-1.5) -- (-0.4,1.5);
\draw [line width=\ocrw,draw=white] (0.4,-1.5) -- (0.4,1.5);
\draw  (0.4,-1.5) -- (0.4,1.5);
%
%\draw [line width=\ljwp] (-2.2,0) node[left]{$\scriptstyle a$} -- (-1.8,0);
%\draw [line width=\ljwp] (1.8,0) -- (2.2,0) node[right]{$\scriptstyle a$};
%\draw [line width=\ljwp] (0,1.8) -- (0,2.2) node[above]{$\scriptstyle a$};
%\draw [line width=\ljwp] (0,-1.8) -- (0,-2.2)node[below]{$\scriptstyle a$};
%
\draw [thkln] (-2.2,0)  -- (-1.8,0);
\node at (-2.6,0) {$\scriptstyle a+1$};
\draw [thkln] (1.8,0) -- (2.2,0);
\node at (2.6,0) {$\scriptstyle a+1$};
\draw [thkln] (0,1.8) -- (0,2.2);
\node at (0,2.6) {$\scriptstyle a+1$};
\draw [thkln] (0,-1.8) -- (0,-2.2);
\node at (0,-2.6) {$\scriptstyle a+1$};
\draw [ptzer] (-1.8,-0.6) rectangle ++(0.3,1.2);
\draw [ptzer] (1.5,-0.6) rectangle ++(0.3,1.2);
\draw [ptzer] (-0.6,1.5) rectangle ++(1.2,0.3);
\draw [ptzer] (-0.6,-1.5) rectangle ++(1.2,-0.3);
\end{tikzpicture}
\hteqv\;\;
\Cnv{
\shcrh
\begin{tikzpicture}[scale=0.5,rotate=-45,baseline=-3.5]
\draw [thkln] (-1.5,0.4) -- (1.5,0.4);
\draw (-1.5,-0.4) -- (-0.7,-0.4);
\draw (-0.7,-0.4) to [out=0,in=90] (0.4,-1.5);
\draw [line width=\ocrw,draw=white] (-0.4,-1.5) -- (-0.4,1.5);
\draw [thkln] (-0.4,-1.5) -- (-0.4,1.5);
%\draw [line width=\ocrw,draw=white] (0.4,-1.5) -- (0.4,1.5);
\draw  (0.4,0.7) -- (0.4,1.5);
\draw [line width=\ocrw,draw=white] (1.5,-0.4) to [out=180,in=-90] (0.4,0.7);
\draw (1.5,-0.4) to [out=180,in=-90] (0.4,0.7);
%
%\draw [line width=\ljwp] (-2.2,0) node[left]{$\scriptstyle a$} -- (-1.8,0);
%\draw [line width=\ljwp] (1.8,0) -- (2.2,0) node[right]{$\scriptstyle a$};
%\draw [line width=\ljwp] (0,1.8) -- (0,2.2) node[above]{$\scriptstyle a$};
%\draw [line width=\ljwp] (0,-1.8) -- (0,-2.2)node[below]{$\scriptstyle a$};
%
\draw [thkln] (-2.2,0)  -- (-1.8,0);
\node at (-2.6,0) {$\scriptstyle a+1$};
\draw [thkln] (1.8,0) -- (2.2,0);
\node at (2.6,0) {$\scriptstyle a+1$};
\draw [thkln] (0,1.8) -- (0,2.2);
\node at (0,2.6) {$\scriptstyle a+1$};
\draw [thkln] (0,-1.8) -- (0,-2.2);
\node at (0,-2.6) {$\scriptstyle a+1$};
\draw [ptzer] (-1.8,-0.6) rectangle ++(0.3,1.2);
\draw [ptzer] (1.5,-0.6) rectangle ++(0.3,1.2);
\draw [ptzer] (-0.6,1.5) rectangle ++(1.2,0.3);
\draw [ptzer] (-0.6,-1.5) rectangle ++(1.2,-0.3);
\end{tikzpicture}
%%%%%%%%%%%%%%%%%%%%%%
\longrightarrow
%%%%%%%%%%%%%%%%%%%%%%%
\shcrmh
\begin{tikzpicture}[scale=0.5,rotate=-45,baseline=-3.5]
\draw [thkln] (-1.5,0.4) -- (1.5,0.4);
\draw (-1.5,-0.4) -- (-0.7,-0.4);
%\draw (-0.7,-0.4) to [out=0,in=90] (0.4,-1.5);
%
%
%\draw [line width=\ocrw,draw=white] (0.4,-1.5) -- (0.4,1.5);
\draw  (0.4,0.7) -- (0.4,1.5);
%
%\draw [line width=\ocrw,draw=white] (1.5,-0.4) to [out=180,in=-90] (0.4,0.7);
%\draw (1.5,-0.4) to [out=180,in=-90] (0.4,0.7);
%
\draw (0.4,-1.5) to [out=90,in=180] (1.5,-0.4);
\draw [line width=\ocrw,draw=white] (-0.7,-0.4) to [out=0,in=-90] (0.4,0.7);
\draw (-0.7,-0.4) to [out=0,in=-90] (0.4,0.7);
\draw [line width=\ocrw,draw=white] (-0.4,-1.5) -- (-0.4,1.5);
\draw [thkln] (-0.4,-1.5) -- (-0.4,1.5);
%
%\draw [line width=\ljwp] (-2.2,0) node[left]{$\scriptstyle a$} -- (-1.8,0);
%\draw [line width=\ljwp] (1.8,0) -- (2.2,0) node[right]{$\scriptstyle a$};
%\draw [line width=\ljwp] (0,1.8) -- (0,2.2) node[above]{$\scriptstyle a$};
%\draw [line width=\ljwp] (0,-1.8) -- (0,-2.2)node[below]{$\scriptstyle a$};
%
\draw [thkln] (-2.2,0)  -- (-1.8,0);
\node at (-2.6,0) {$\scriptstyle a+1$};
\draw [thkln] (1.8,0) -- (2.2,0);
\node at (2.6,0) {$\scriptstyle a+1$};
\draw [thkln] (0,1.8) -- (0,2.2);
\node at (0,2.6) {$\scriptstyle a+1$};
\draw [thkln] (0,-1.8) -- (0,-2.2);
\node at (0,-2.6) {$\scriptstyle a+1$};
\draw [ptzer] (-1.8,-0.6) rectangle ++(0.3,1.2);
\draw [ptzer] (1.5,-0.6) rectangle ++(0.3,1.2);
\draw [ptzer] (-0.6,1.5) rectangle ++(1.2,0.3);
\draw [ptzer] (-0.6,-1.5) rectangle ++(1.2,-0.3);
\end{tikzpicture}
}
\]
and then use the relations\rx{eq:projtw} to bring both tangles to the form of \ex{eq:colKhbr}.
\end{proof}
%
%\subsubsection{\tPmcn s}
%Let $\obO$ be an object of a triangulated category. A \tPtri\ of $\obO$ is a collection of objects $\obOij$ which are trivial unless $i\geq 0$ and $0\leq j\leq i$, and a collection of homotopy equivalences
%\[
%\obOvv{i+1}{j}\hteqv
%\Cnv{
%\shcr^{2i}\obOvv{i}{j}\rightarrow\obOvv{i}{j-1}
%}\;.
%\]
%This relation means that any $\obOij$ is a \tmcn\ generated by $\obO$. An identity
%\[
%{i+1\brace j}_{x} =
%{i \brace j}_{x} + x^{2j} {i\brace j}_{x}
%\]
%for the combinatorial polynomial\rx{eq:combp}
%implies that for a triangulated category $\cKommA$ and for an object $\chA$ of\rx{eq:complex}
%a \tPmcn\ $(\chA)_{i,j}$ is a direct sum ${i \brace j}_{\shcr}\chA$ deformed by an extra differential, hence we use a notation
%\begin{equation}
%\label{eq:pmcn}
%(\chA)_{i,j} =\Pcnv{{i\brace j}_{\shcr}\chA }
%\end{equation}
%For the purpose of this paper we do not need to know the exact order in which constituent compelxes $\chA$ are arranged inside the \tmcn\rx{eq:pmcn}.

%\subsubsection{Colored \tKhbr}
%
\begin{theorem}[Colored \tKhbr]
\label{thm:colkhovbr}
A \tKhbr\ of the crossing of two equally colored strands can be presented as a \tmcn\ of crossingless colored tangles:
\begin{equation}
\label{eq:lmtcn}
\begin{tikzpicture}[menvrtone]
\draw [thkln] (-1.5,0) -- (1.5,0);
\draw [line width=\ocrw,draw=white] (0,-1.5) -- (0,1.5);
\draw [thkln] (0,-1.5) -- (0,1.5);
%
%\draw [line width=\ljwp] (-2.2,0) node[left]{$\scriptstyle a$} -- (-1.8,0);
%\draw [line width=\ljwp] (1.8,0) -- (2.2,0) node[right]{$\scriptstyle a$};
%\draw [line width=\ljwp] (0,1.8) -- (0,2.2) node[above]{$\scriptstyle a$};
%\draw [line width=\ljwp] (0,-1.8) -- (0,-2.2)node[below]{$\scriptstyle a$};
%
\draw [thkln] (-2.2,0)  -- (-1.8,0);
\node at (-2.6,0) {$\scriptstyle a$};
\draw [thkln] (1.8,0) -- (2.2,0);
\node at (2.6,0) {$\scriptstyle a$};
\draw [thkln] (0,1.8) -- (0,2.2);
\node at (0,2.6) {$\scriptstyle a$};
\draw [thkln] (0,-1.8) -- (0,-2.2);
\node at (0,-2.6) {$\scriptstyle a$};
\draw [ptzer] (-1.8,-0.6) rectangle ++(0.3,1.2);
\draw [ptzer] (1.5,-0.6) rectangle ++(0.3,1.2);
\draw [ptzer] (-0.6,1.5) rectangle ++(1.2,0.3);
\draw [ptzer] (-0.6,-1.5) rectangle ++(1.2,-0.3);
\end{tikzpicture}
\;\hteqv\;
\shcr^{-\hlf a^2}
\;
\Cnv{
\cdots\longrightarrow
\shcr^{\ysvki + n_i}
\begin{tikzpicture}[menvrtsone]
\node at (135:0.55) {$\scriptstyle a-\yvki$};
\node at (135:-0.55) {$\scriptstyle a-\yvki$};
\node at (45:1.5) {$\scriptstyle \yvki$};
\node at (45:-1.5) {$\scriptstyle \yvki$};
\draw [thkln] (-1.5,0.3) to [out=0, in=-90] (-0.3,1.5);
\draw [thkln] (-1.5,-0.3) to [out=0, in=90] (-0.3,-1.5);
\draw [thkln] (1.5,0.3) to [out=180, in=-90] (0.3, 1.5);
\draw [thkln] (1.5,-0.3) to [out=180, in=90] (0.3,-1.5);
%\draw [line width=\ljwp] (-1.5,0) -- (1.5,0);
%\draw [line width=\ocrw,draw=white] (0,-1.5) -- (0,1.5);
%\draw [line width=\ljwp] (0,-1.5) -- (0,1.5);
%
\draw [thkln] (-2.2,0)  -- (-1.8,0);
\node at (-2.6,0) {$\scriptstyle a$};
\draw [thkln] (1.8,0) -- (2.2,0);
\node at (2.6,0) {$\scriptstyle a$};
\draw [thkln] (0,1.8) -- (0,2.2);
\node at (0,2.6) {$\scriptstyle a$};
\draw [thkln] (0,-1.8) -- (0,-2.2);
\node at (0,-2.6) {$\scriptstyle a$};
\draw [ptzer] (-1.8,-0.6) rectangle ++(0.3,1.2);
\draw [ptzer] (1.5,-0.6) rectangle ++(0.3,1.2);
\draw [ptzer] (-0.6,1.5) rectangle ++(1.2,0.3);
\draw [ptzer] (-0.6,-1.5) rectangle ++(1.2,-0.3);
\end{tikzpicture}
\longrightarrow\cdots
}_{\;i=0}^{\;\infty}\;,
\end{equation}
such that
\begin{equation}
\label{eq:twineq}
n_i\geq 0,\qquad  i\geq 2^{\yki}-1
\end{equation}
and the \lumps\ form of this \tmcn\ is
\begin{equation}
\label{eq:colKhbr}
\begin{tikzpicture}[menvrtone]
\draw [thkln] (-1.5,0) -- (1.5,0);
\draw [line width=\ocrw,draw=white] (0,-1.5) -- (0,1.5);
\draw [thkln] (0,-1.5) -- (0,1.5);
%
%\draw [line width=\ljwp] (-2.2,0) node[left]{$\scriptstyle a$} -- (-1.8,0);
%\draw [line width=\ljwp] (1.8,0) -- (2.2,0) node[right]{$\scriptstyle a$};
%\draw [line width=\ljwp] (0,1.8) -- (0,2.2) node[above]{$\scriptstyle a$};
%\draw [line width=\ljwp] (0,-1.8) -- (0,-2.2)node[below]{$\scriptstyle a$};
%
\draw [thkln] (-2.2,0)  -- (-1.8,0);
\node at (-2.6,0) {$\scriptstyle a$};
\draw [thkln] (1.8,0) -- (2.2,0);
\node at (2.6,0) {$\scriptstyle a$};
\draw [thkln] (0,1.8) -- (0,2.2);
\node at (0,2.6) {$\scriptstyle a$};
\draw [thkln] (0,-1.8) -- (0,-2.2);
\node at (0,-2.6) {$\scriptstyle a$};
\draw [ptzer] (-1.8,-0.6) rectangle ++(0.3,1.2);
\draw [ptzer] (1.5,-0.6) rectangle ++(0.3,1.2);
\draw [ptzer] (-0.6,1.5) rectangle ++(1.2,0.3);
\draw [ptzer] (-0.6,-1.5) rectangle ++(1.2,-0.3);
\end{tikzpicture}
\;\hteqv\;
\shcr^{-\hlf a^2}\;
\Pcnv{
\bigoplus_{\xki=0}^a
\shcr^{\xki^2}
{a \brace \yki}_{\shcr}
%%%%%%%%%%%%%%%%%%
\begin{tikzpicture}[menvrtone]
\node at (135:0.65) {$\scriptstyle a-\yki$};
\node at (135:-0.65) {$\scriptstyle a-\yki$};
\node at (45:1.5) {$\scriptstyle \yki$};
\node at (45:-1.5) {$\scriptstyle \yki$};
\draw [thkln] (-1.5,0.3) to [out=0, in=-90] (-0.3,1.5);
\draw [thkln] (-1.5,-0.3) to [out=0, in=90] (-0.3,-1.5);
\draw [thkln] (1.5,0.3) to [out=180, in=-90] (0.3, 1.5);
\draw [thkln] (1.5,-0.3) to [out=180, in=90] (0.3,-1.5);
%\draw [line width=\ljwp] (-1.5,0) -- (1.5,0);
%\draw [line width=\ocrw,draw=white] (0,-1.5) -- (0,1.5);
%\draw [line width=\ljwp] (0,-1.5) -- (0,1.5);
%
\draw [thkln] (-2.2,0)  -- (-1.8,0);
\node at (-2.6,0) {$\scriptstyle a$};
\draw [thkln] (1.8,0) -- (2.2,0);
\node at (2.6,0) {$\scriptstyle a$};
\draw [thkln] (0,1.8) -- (0,2.2);
\node at (0,2.6) {$\scriptstyle a$};
\draw [thkln] (0,-1.8) -- (0,-2.2);
\node at (0,-2.6) {$\scriptstyle a$};
\draw [ptzer] (-1.8,-0.6) rectangle ++(0.3,1.2);
\draw [ptzer] (1.5,-0.6) rectangle ++(0.3,1.2);
\draw [ptzer] (-0.6,1.5) rectangle ++(1.2,0.3);
\draw [ptzer] (-0.6,-1.5) rectangle ++(1.2,-0.3);
\end{tikzpicture}
}
\end{equation}
\end{theorem}
\begin{proof}
We prove this theorem by induction over $a$. At $a=1$ it amounts to \tKhbr\rx{eq:dkhbr}. Suppose that it holds for some $a$ and consider the crossing of two $(a+1)$-cables. We split each cable into an $a$-cable and a single line and apply \eex{eq:lmtcn} and\rx{eq:colKhbr} to the crossing of $a$-cables:
\def\xscl{0.6}
\begin{equation}
\label{eq:spcrmt}
\begin{split}
\begin{tikzpicture}[scale=\xscl,rotate=-45,baseline=-3.5]
\draw [thkln] (-1.5,0.4) -- (1.5,0.4);
\draw (-1.5,-0.4) -- (1.5,-0.4);
\draw [line width=\ocrw,draw=white] (-0.4,-1.5) -- (-0.4,1.5);
\draw [thkln] (-0.4,-1.5) -- (-0.4,1.5);
\draw [line width=\ocrw,draw=white] (0.4,-1.5) -- (0.4,1.5);
\draw  (0.4,-1.5) -- (0.4,1.5);
\draw [thkln] (-2.2,0)  -- (-1.8,0);
\node at (-2.6,0) {$\scriptstyle a+1$};
\draw [thkln] (1.8,0) -- (2.2,0);
\node at (2.6,0) {$\scriptstyle a+1$};
\draw [thkln] (0,1.8) -- (0,2.2);
\node at (0,2.6) {$\scriptstyle a+1$};
\draw [thkln] (0,-1.8) -- (0,-2.2);
\node at (0,-2.6) {$\scriptstyle a+1$};
\draw [ptzer] (-1.8,-0.6) rectangle ++(0.3,1.2);
\draw [ptzer] (1.5,-0.6) rectangle ++(0.3,1.2);
\draw [ptzer] (-0.6,1.5) rectangle ++(1.2,0.3);
\draw [ptzer] (-0.6,-1.5) rectangle ++(1.2,-0.3);
\end{tikzpicture}
&\;\hteqv\;
\shcr^{-\hlf a^2}\;
\Pcnv{
\bigoplus_{\yki=0}^{a}
\shcr^{\yki^2}
{a \brace \yki}_{\shcr}
\begin{tikzpicture}[scale=\xscl,
rotate=-45,
baseline=-3.5]
\draw [thkln] (-1.5,0.4) to [out=0,in=-90] (-0.4,1.5);
\node at (-1,1) {$\scriptstyle a-\yki$};
\draw (-1.5,-0.4) -- (1.5,-0.4);
\draw [lnovrtw] (-1.5,0) to [out=0,in=90] (-0.4,-1.5);
\draw [thkln] (-1.5,0) to [out=0,in=90] (-0.4,-1.5) ;
%\draw [lnovr] (-0.4,-1.5) -- (-0.4,1.5);
%\draw [line width=\ljwp] (-0.4,-1.5) -- (-0.4,1.5);
%\draw [thkln] (0,-1.5) to [out=90,in=180] (1.5,0);
\draw [lnovrtw] (0,-1.5) to [out=90,in=-135] (0,0);
\draw [thkln] (0,-1.5) to [out=90,in=-135] (0,0) to [out=45,in=180] (1.5,0);
\draw [thkln]  (1.5,0.4) to [out = 180, in=-90] (0,1.5);
\draw [lnovrtw] (0.4,-1.5) -- (0.4,1.5);
\draw  (0.4,-1.5) -- (0.4,1.5);
%
%\draw [line width=\ljwp] (-2.2,0) node[left]{$\scriptstyle a$} -- (-1.8,0);
%\draw [line width=\ljwp] (1.8,0) -- (2.2,0) node[right]{$\scriptstyle a$};
%\draw [line width=\ljwp] (0,1.8) -- (0,2.2) node[above]{$\scriptstyle a$};
%\draw [line width=\ljwp] (0,-1.8) -- (0,-2.2)node[below]{$\scriptstyle a$};
%
\draw [thkln] (-2.2,0)  -- (-1.8,0);
\node at (-2.6,0) {$\scriptstyle a+1$};
\draw [thkln] (1.8,0) -- (2.2,0);
\node at (2.6,0) {$\scriptstyle a+1$};
\draw [thkln] (0,1.8) -- (0,2.2);
\node at (0,2.6) {$\scriptstyle a+1$};
\draw [thkln] (0,-1.8) -- (0,-2.2);
\node at (0,-2.6) {$\scriptstyle a+1$};
\draw [ptzer] (-1.8,-0.6) rectangle ++(0.3,1.2);
\draw [ptzer] (1.5,-0.6) rectangle ++(0.3,1.2);
\draw [ptzer] (-0.6,1.5) rectangle ++(1.2,0.3);
\draw [ptzer] (-0.6,-1.5) rectangle ++(1.2,-0.3);
\end{tikzpicture}
}
\\
&\;\hteqv\;
\shcr^{-\hlf a^2}\;
\Cnv{
\cdots\longrightarrow
\shcr^{\ysvki + n_i}
\begin{tikzpicture}[scale=\xscl,
rotate=-45,
baseline=-3.5]
\draw [thkln] (-1.5,0.4) to [out=0,in=-90] (-0.4,1.5);
\node at (-1+0.45,1-0.45) {$\scriptstyle a-\yvki$};
\draw (-1.5,-0.4) -- (1.5,-0.4);
\draw [lnovrtw] (-1.5,0) to [out=0,in=90] (-0.4,-1.5);
\draw [thkln] (-1.5,0) to [out=0,in=90] (-0.4,-1.5) ;
%\draw [lnovr] (-0.4,-1.5) -- (-0.4,1.5);
%\draw [line width=\ljwp] (-0.4,-1.5) -- (-0.4,1.5);
%\draw [thkln] (0,-1.5) to [out=90,in=180] (1.5,0);
\draw [lnovrtw] (0,-1.5) to [out=90,in=-135] (0,0);
\draw [thkln] (0,-1.5) to [out=90,in=-135] (0,0) to [out=45,in=180] (1.5,0);
\draw [thkln]  (1.5,0.4) to [out = 180, in=-90] (0,1.5);
\draw [lnovrtw] (0.4,-1.5) -- (0.4,1.5);
\draw  (0.4,-1.5) -- (0.4,1.5);
%
%\draw [line width=\ljwp] (-2.2,0) node[left]{$\scriptstyle a$} -- (-1.8,0);
%\draw [line width=\ljwp] (1.8,0) -- (2.2,0) node[right]{$\scriptstyle a$};
%\draw [line width=\ljwp] (0,1.8) -- (0,2.2) node[above]{$\scriptstyle a$};
%\draw [line width=\ljwp] (0,-1.8) -- (0,-2.2)node[below]{$\scriptstyle a$};
%
\draw [thkln] (-2.2,0)  -- (-1.8,0);
\node at (-2.6,0) {$\scriptstyle a+1$};
\draw [thkln] (1.8,0) -- (2.2,0);
\node at (2.6,0) {$\scriptstyle a+1$};
\draw [thkln] (0,1.8) -- (0,2.2);
\node at (0,2.6) {$\scriptstyle a+1$};
\draw [thkln] (0,-1.8) -- (0,-2.2);
\node at (0,-2.6) {$\scriptstyle a+1$};
\draw [ptzer] (-1.8,-0.6) rectangle ++(0.3,1.2);
\draw [ptzer] (1.5,-0.6) rectangle ++(0.3,1.2);
\draw [ptzer] (-0.6,1.5) rectangle ++(1.2,0.3);
\draw [ptzer] (-0.6,-1.5) rectangle ++(1.2,-0.3);
\end{tikzpicture}
\longrightarrow\cdots
}_{\;i=0}^{\;\infty}\;
\end{split}
\end{equation}
The categorification complex of a constituent tangle of the resulting \tmcn\ can be simplified:
\begin{equation}
\label{eq:sepbr}
\begin{split}
\begin{tikzpicture}[scale=\xscl,
rotate=-45,
baseline=-3.5]
\draw [thkln] (-1.5,0.4) to [out=0,in=-90] (-0.4,1.5);
\node at (-1,1) {$\scriptstyle a-\yki$};
\draw (-1.5,-0.4) -- (1.5,-0.4);
\draw [lnovrtw] (-1.5,0) to [out=0,in=90] (-0.4,-1.5);
\draw [thkln] (-1.5,0) to [out=0,in=90] (-0.4,-1.5) ;
%\draw [lnovr] (-0.4,-1.5) -- (-0.4,1.5);
%\draw [line width=\ljwp] (-0.4,-1.5) -- (-0.4,1.5);
%\draw [thkln] (0,-1.5) to [out=90,in=180] (1.5,0);
\draw [lnovrtw] (0,-1.5) to [out=90,in=-135] (0,0);
\draw [thkln] (0,-1.5) to [out=90,in=-135] (0,0) to [out=45,in=180] (1.5,0);
\draw [thkln]  (1.5,0.4) to [out = 180, in=-90] (0,1.5);
\draw [lnovrtw] (0.4,-1.5) -- (0.4,1.5);
\draw  (0.4,-1.5) -- (0.4,1.5);
%
%\draw [line width=\ljwp] (-2.2,0) node[left]{$\scriptstyle a$} -- (-1.8,0);
%\draw [line width=\ljwp] (1.8,0) -- (2.2,0) node[right]{$\scriptstyle a$};
%\draw [line width=\ljwp] (0,1.8) -- (0,2.2) node[above]{$\scriptstyle a$};
%\draw [line width=\ljwp] (0,-1.8) -- (0,-2.2)node[below]{$\scriptstyle a$};
%
\draw [thkln] (-2.2,0)  -- (-1.8,0);
\node at (-2.6,0) {$\scriptstyle a+1$};
\draw [thkln] (1.8,0) -- (2.2,0);
\node at (2.6,0) {$\scriptstyle a+1$};
\draw [thkln] (0,1.8) -- (0,2.2);
\node at (0,2.6) {$\scriptstyle a+1$};
\draw [thkln] (0,-1.8) -- (0,-2.2);
\node at (0,-2.6) {$\scriptstyle a+1$};
\draw [ptzer] (-1.8,-0.6) rectangle ++(0.3,1.2);
\draw [ptzer] (1.5,-0.6) rectangle ++(0.3,1.2);
\draw [ptzer] (-0.6,1.5) rectangle ++(1.2,0.3);
\draw [ptzer] (-0.6,-1.5) rectangle ++(1.2,-0.3);
\end{tikzpicture}
&
\;\hteqv\;
\shcr^{2\xki-a}
\begin{tikzpicture}[scale=\xscl,
rotate=-45,
baseline=-3.5]
\draw [thkln] (-1.5,0.4) to [out=0,in=-90] (-0.4,1.5);
\node at (-1.05,1.05) {$\scriptstyle a-\yki$};
\node at (1,-1) {$\scriptstyle a-\yki$};
\node at (1.,1.) {$\scriptstyle \yki$};
\node at (-1.,-1.) {$\scriptstyle \yki$};
\draw (-1.5,0) -- (1.5,0);
%\draw [lnovrtw] (-1.5,0) to [out=0,in=90] (-0.4,-1.5);
\draw [thkln] (-1.5,-0.4) to [out=0,in=90] (-0.4,-1.5) ;
%\draw [lnovr] (-0.4,-1.5) -- (-0.4,1.5);
%\draw [line width=\ljwp] (-0.4,-1.5) -- (-0.4,1.5);
%\draw [thkln] (0,-1.5) to [out=90,in=180] (1.5,0);
%\draw [lnovrtw] (0,-1.5) to [out=90,in=-135] (0,0);
\draw [thkln] (0.4,-1.5) to [out=90,in=180]  (1.5,-0.4);
\draw [thkln]  (1.5,0.4) to [out = 180, in=-90] (0.4,1.5);
\draw [lnovrtw] (0,-1.5) -- (0,1.5);
\draw  (0,-1.5) -- (0,1.5);
%
%\draw [line width=\ljwp] (-2.2,0) node[left]{$\scriptstyle a$} -- (-1.8,0);
%\draw [line width=\ljwp] (1.8,0) -- (2.2,0) node[right]{$\scriptstyle a$};
%\draw [line width=\ljwp] (0,1.8) -- (0,2.2) node[above]{$\scriptstyle a$};
%\draw [line width=\ljwp] (0,-1.8) -- (0,-2.2)node[below]{$\scriptstyle a$};
%
\draw [thkln] (-2.2,0)  -- (-1.8,0);
\node at (-2.6,0) {$\scriptstyle a+1$};
\draw [thkln] (1.8,0) -- (2.2,0);
\node at (2.6,0) {$\scriptstyle a+1$};
\draw [thkln] (0,1.8) -- (0,2.2);
\node at (0,2.6) {$\scriptstyle a+1$};
\draw [thkln] (0,-1.8) -- (0,-2.2);
\node at (0,-2.6) {$\scriptstyle a+1$};
\draw [ptzer] (-1.8,-0.6) rectangle ++(0.3,1.2);
\draw [ptzer] (1.5,-0.6) rectangle ++(0.3,1.2);
\draw [ptzer] (-0.6,1.5) rectangle ++(1.2,0.3);
\draw [ptzer] (-0.6,-1.5) rectangle ++(1.2,-0.3);
\end{tikzpicture}
\\
&
\;\hteqv\;
\shcr^{2\xki-a}\;
\boxed{
\,\shcr^{\hlf}
\begin{tikzpicture}[scale=\xscl,
rotate=-45,
baseline=-3.5]
\draw [thkln] (-1.5,0.3) to [out=0,in=-90] (-0.3,1.5);
\node at (-0.4,0.4) {$\scriptstyle a-\yki$};
\node at (0.4,-0.4) {$\scriptstyle a-\yki$};
\node at (1.15,1.15) {$\scriptstyle \yki+1$};
\node at (-1.15,-1.15) {$\scriptstyle \yki+1$};
%\draw (-1.5,0) -- (1.5,0);
\draw [thkln] (-1.5,-0.3) to [out=0,in=90] (-0.3,-1.5) ;
\draw [thkln] (0.3,-1.5) to [out=90,in=180]  (1.5,-0.3);
\draw [thkln]  (1.5,0.3) to [out = 180, in=-90] (0.3,1.5);
%\draw [lnovrtw] (0,-1.5) -- (0,1.5);
%\draw  (0,-1.5) -- (0,1.5);
%
\draw [thkln] (-2.2,0)  -- (-1.8,0);
\node at (-2.6,0) {$\scriptstyle a+1$};
\draw [thkln] (1.8,0) -- (2.2,0);
\node at (2.6,0) {$\scriptstyle a+1$};
\draw [thkln] (0,1.8) -- (0,2.2);
\node at (0,2.6) {$\scriptstyle a+1$};
\draw [thkln] (0,-1.8) -- (0,-2.2);
\node at (0,-2.6) {$\scriptstyle a+1$};
\draw [ptzer] (-1.8,-0.6) rectangle ++(0.3,1.2);
\draw [ptzer] (1.5,-0.6) rectangle ++(0.3,1.2);
\draw [ptzer] (-0.6,1.5) rectangle ++(1.2,0.3);
\draw [ptzer] (-0.6,-1.5) rectangle ++(1.2,-0.3);
\end{tikzpicture}
\longrightarrow
\shcr^{-\hlf}
\begin{tikzpicture}[scale=\xscl,
rotate=-45,
baseline=-3.5]
\draw [thkln] (-1.5,0.3) to [out=0,in=-90] (-0.3,1.5);
\node at (-0.4,0.4) {$\scriptstyle a-\yki+1$};
\node at (0.4,-0.4) {$\scriptstyle a-\yki+1$};
\node at (0.9,0.9) {$\scriptstyle \yki$};
\node at (-0.9,-0.9) {$\scriptstyle \yki$};
%\draw (-1.5,0) -- (1.5,0);
\draw [thkln] (-1.5,-0.3) to [out=0,in=90] (-0.3,-1.5) ;
\draw [thkln] (0.3,-1.5) to [out=90,in=180]  (1.5,-0.3);
\draw [thkln]  (1.5,0.3) to [out = 180, in=-90] (0.3,1.5);
%\draw [lnovrtw] (0,-1.5) -- (0,1.5);
%\draw  (0,-1.5) -- (0,1.5);
%
\draw [thkln] (-2.2,0)  -- (-1.8,0);
\node at (-2.6,0) {$\scriptstyle a+1$};
\draw [thkln] (1.8,0) -- (2.2,0);
\node at (2.6,0) {$\scriptstyle a+1$};
\draw [thkln] (0,1.8) -- (0,2.2);
\node at (0,2.6) {$\scriptstyle a+1$};
\draw [thkln] (0,-1.8) -- (0,-2.2);
\node at (0,-2.6) {$\scriptstyle a+1$};
\draw [ptzer] (-1.8,-0.6) rectangle ++(0.3,1.2);
\draw [ptzer] (1.5,-0.6) rectangle ++(0.3,1.2);
\draw [ptzer] (-0.6,1.5) rectangle ++(1.2,0.3);
\draw [ptzer] (-0.6,-1.5) rectangle ++(1.2,-0.3);
\end{tikzpicture}
}
\end{split}
\end{equation}
Here the first homotopy equivalence follows from \ex{eq:projtw} and the second one is the application of Khovanov bracket\rx{eq:dkhbr} to the crossing of two single lines. We substitute \ex{eq:sepbr} for every constituent tangle in both \tmcn s of \ex{eq:spcrmt}. The \lumps\ \tmcn\ transforms into the \rhs of \ex{eq:colKhbr} for the intersection of two $(a+1)$-cables with the help of a simple identity
%
%After substituting \ex{eq:sepbr} for every constituent tangle in the \tmcn\rx{eq:spcrmt}
%%and rearranging the nested \tmcn s associatively
%we get \ex{eq:colKhbr} for the intersection of two $(a+1)$ cables with the help of a simple identity
\[
{a \brace \yki-1}_{\shcr} + \shcr^{2\yki} {a\brace \yki}_{\shcr} =
{a+1\brace \yki}_{\shcr}.
\]
Associativity of the cone operation implies that the second \tmcn\ of \ex{eq:spcrmt} can be brought to the linear form of the \rhs of \ex{eq:lmtcn}, so it remains to verify inequalities\rx{eq:twineq}. The first inequality follows from the \lumps\ \tmcn\ formula\rx{eq:colKhbr} which we have just proved. Let us verify the second inequality for two tangles of the cone\rx{eq:sepbr} after they appear through the substitution in the second \tmcn\ of \ex{eq:spcrmt}. Since every constituent tangle of\rx{eq:spcrmt} is replaced by a cone of two tangles, the second tangle of the cone\rx{eq:sepbr} will appear at the \tmcn\ position $i'=2i$ and the second inequality of\rx{eq:twineq} for it obviously holds. The first tangle of \rx{eq:spcrmt} appears at the position $i'=2i+1$ and it carries $\yki'=\yki+1$. The inequality $2i'+1\geq 2^{\yki'}-1$ follows easily from the assumed inequality $i\leq 2^{\yki}-1$.
\end{proof}

\subsection{Recurrence relations for \cJWp s}

\begin{proposition}
A larger \tJWp\ absorbs a smaller one:
\begin{equation}
\label{eq:absorb}
\begin{tikzpicture}[scale=0.5,baseline=-1.5]
%\draw [xscale=-1]  (1, 0.9) to [out=180,in=0] (0,-0.4);
%\draw [line width=6pt, color=white] (-1+0.15,0) -- (0,0);
\draw [thkln] %(-1.75,0) -- (-1 - .15,0)
%node [near start,below] {$\scriptstyle \xca$}
%(-1+0.15,0) --
(-0.,0) -- (1-0.15,0)
node [near start,below] {$\scriptstyle \xca$}
 (1+.15,0) -- (1.75,0)
(2.05,0) -- (2.9,0)
node [near end, below] {$\scriptstyle \xca$};
%\draw [line width=6pt, color=white] (1, 0.9) to [out=180,in=0] (0,-0.4);
%\draw (1, 0.9) to [out=180,in=0] (0,-0.4);
\draw (0,0.9) -- (1.75,0.9) (2.05,0.9) -- (2.9,0.9);
%\draw [line width=\ljwp] (-1.15,-0.6) rectangle ++(0.3,1.2);
\draw [ptzer] (1 - 0.15,-0.6) rectangle ++(0.3,1.2);
%%%%%%%
\draw [ptzer] (1.75,-1.2) rectangle ++(0.3,2.4);
\end{tikzpicture}
\;\hteqv\;
\begin{tikzpicture}[scale=0.5,baseline=-1.5]
\draw [ptzer] (1.75,-0.6) rectangle ++(0.3,1.2);
\draw [thkln] (1.75-1.3,0) -- (1.75,0)
node [near start, below] {$\scriptstyle \xca+1$}
(2.05,0) -- (2.05+1.3,0)
node [near end, below] {$\scriptstyle \xca+1$};
\end{tikzpicture}\;.
\end{equation}
\end{proposition}
\begin{proof}
In view of \eex{eq:projcn} and\rx{eq:grproj} for $a=\xca$, this equivalence is a result of purging the smaller projector with the larger one.
\end{proof}

Let us introduce a shortcut notation:
\begin{equation}
\label{eq:xthrpr}
\begin{tikzpicture}[menvone]
\draw[pttwo] (-.15,-0.6) rectangle ++(0.3,1.2);
\draw[thkln] (-1.35,0) -- (-.15,0)
node[near start,below] {$\scriptstyle \xca+1$ }
(.15,0) -- (1.35,0)
node [near end, below] {$\scriptstyle \xca+1$};
\end{tikzpicture}
=
\begin{tikzpicture}[scale=0.5,baseline=-1.5]
\draw[ptone] (-.15,-1.2) rectangle ++(.3,2.4);
\draw[ptzer] (-1-.15,-0.6) rectangle ++(.3,1.2);
\draw[ptzer] (1-.15,-0.6) rectangle ++(.3,1.2);
\draw[vthln] (-2,0) -- (-1-.15,0)
 node[near start, below] {$\scriptstyle \xca$}
 (-1+.15,0) -- (-.15,0) (.15,0) -- (1-.15,0)  (1.15,0) -- (2,0)
 node[near end,below] {$\scriptstyle \xca$};
 \draw (-2,0.9) -- (-.15,0.9) (.15,0.9) -- (2,0.9);
\end{tikzpicture},
\end{equation}
where the complex
$
\begin{tikzpicture}[scale=0.4,baseline=-3]
\draw[ptone] (-.15,-0.6) rectangle ++(0.3,1.2);
\draw[vthln] (-0.75,0) -- (-.15,0) (.15,0) -- (.75,0);
\end{tikzpicture}
$
is defined by \ex{eq:projcn}.
\begin{proposition}
\label{pr:hteqvgr}
Thus defined, the complex
$
\begin{tikzpicture}[scale=0.4,baseline=-3]
\draw[pttwo] (-.15,-0.6) rectangle ++(0.3,1.2);
\draw[vthln] (-0.75,0) -- (-.15,0) (.15,0) -- (.75,0);
\end{tikzpicture}
$
has a \tmcn\ presentation
%There is a homotopy equivalence
\begin{equation}
\label{eq:swgrpr}
%\begin{tikzpicture}[scale=0.5,baseline=-1.5]
%\draw[ptone] (-.15,-1.2) rectangle ++(.3,2.4);
%\draw[ptzer] (-1-.15,-0.6) rectangle ++(.3,1.2);
%\draw[ptzer] (1-.15,-0.6) rectangle ++(.3,1.2);
%\draw[vthln] (-2,0) -- (-1-.15,0)
% node[near start, below] {$\scriptstyle \xca$}
% (-1+.15,0) -- (-.15,0) (.15,0) -- (1-.15,0)  (1.15,0) -- (2,0)
% node[near end,below] {$\scriptstyle \xca$};
% \draw (-2,0.9) -- (-.15,0.9) (.15,0.9) -- (2,0.9);
%\end{tikzpicture}
%%%%%%%%%%%%%%
\begin{tikzpicture}[menvone]
\draw[pttwo] (-.15,-0.6) rectangle ++(0.3,1.2);
\draw[thkln] (-1.35,0) -- (-.15,0)
node[near start,below] {$\scriptstyle \xca+1$ }
(.15,0) -- (1.35,0)
node [near end, below] {$\scriptstyle \xca+1$};
\end{tikzpicture}
\;
\hteqv
\;
\boxed{
\cdots\longrightarrow
\shcr^i \bigoplus_{j=0}^i
\prmltij\shfr^j
\begin{tikzpicture}[menvone]
%\draw [xscale=-1]  (1, 0.9) to [out=180,in=0] (0,-0.4);
\draw [line width=6pt, color=white] (-1+0.15,0) -- (0,0);
\draw [thkln] (-1.75,0) -- (-1 - .15,0)
node [near start, below] {$\scriptstyle \xca$}
(-1+0.15,0) -- (0,0) -- (1-0.15,0) (1+.15,0) -- (1.75,0)
node [near end, below] {$\scriptstyle \xca$};
%\draw [line width=6pt, color=white] (1, 0.9) to [out=180,in=0] (0,-0.4);
\draw (1, 0.9) to [out=180,in=90] (0.4,0.6) to [out=-90,in=180] (1-.15,0.3);
\draw [xscale=-1] (1, 0.9) to [out=180,in=90] (0.4,0.6) to [out=-90,in=180] (1-.15,0.3);
\draw (-1.75,0.9) -- (-1,0.9) (1,0.9) -- (1.75,0.9);
\draw [ptzer] (-1.15,-0.6) rectangle ++(0.3,1.2);
\draw [ptzer] (1 - 0.15,-0.6) rectangle ++(0.3,1.2);
\draw [dotted] (-0.6,-0.7) rectangle (0.6,1.2);
\end{tikzpicture}
\longrightarrow\cdots
}_{\;i=0}^{\infty},
\end{equation}
where $\prmltij$ are the multiplicities of the \tTLt\ inside the dotted box, with which it appears in the \rhs of \ex{eq:grproj}.
\end{proposition}
\begin{proof}
We purge the complex $
\begin{tikzpicture}[scale=0.4,baseline=-3]
\draw[pttwo] (-.15,-0.6) rectangle ++(0.3,1.2);
\draw[vthln] (-0.75,0) -- (-.15,0) (.15,0) -- (.75,0);
\end{tikzpicture}
$ with the help of two $\xca$-strand \tJWp s.
The tangle in the dotted box is the only
$(\xca+1,\xca+1)$ \taTLt\
%element of the set $\sTLa$ of \ex{eq:grproj}
which is not contracted when sandwiched between them.
%two $\xca$-strand \tJWp s.
\end{proof}

\begin{theorem}
\label{thm:projpr}
The $(\xca+1)$-strand \cJWp\ is homotopy equivalent to a cone
\begin{equation}
\label{eq:jwpar}
%%%%%%%%%%%%%%%
\begin{tikzpicture}[menvone]
\draw[ptzer] (-.15,-0.6) rectangle ++(0.3,1.2);
\draw[thkln] (-1.35,0) -- (-.15,0)
node[near start,below] {$\scriptstyle \xca+1$ }
(.15,0) -- (1.35,0)
node [near end, below] {$\scriptstyle \xca+1$};
\end{tikzpicture}
%}
\hteqv
\boxed{
\shcr
\begin{tikzpicture}[menvone]
\draw[pttwo] (-.15,-0.6) rectangle ++(0.3,1.2);
\draw[thkln] (-1.35,0) -- (-.15,0)
node[near start,below] {$\scriptstyle \xca+1$ }
(.15,0) -- (1.35,0)
node [near end, below] {$\scriptstyle \xca+1$};
\end{tikzpicture}
\longrightarrow
\begin{tikzpicture}[menvone]
\draw[color=white] (-0.15,-1.2) rectangle (0.15,1.2);
\draw[ptzer] (-0.15,-0.6) rectangle (0.15,0.6);
\draw [thkln] (-0.75,0) -- (-0.15,0)
node [near start, below] {$\scriptstyle \xca$} (0.15,0) -- (0.75,0)
node [near end, below] {$\scriptstyle \xca$};
\draw (-0.75,0.9) -- (0.75, 0.9);
\end{tikzpicture}
}
\end{equation}
\end{theorem}
\begin{proof}
Consider a sequence of homotopy equivalences
\begin{multline}
\label{eq:thrprj}
\begin{tikzpicture}[menvone]
\draw[ptzer] (-.15,-0.6) rectangle ++(0.3,1.2);
\draw[thkln] (-1.35,0) -- (-.15,0)
node[near start,below] {$\scriptstyle \xca+1$ }
(.15,0) -- (1.35,0)
node [near end, below] {$\scriptstyle \xca+1$};
\end{tikzpicture}
\;\hteqv\;
\begin{tikzpicture}[scale=0.5,baseline=-1.5]
\draw[ptzer] (-.15,-1.2) rectangle ++(.3,2.4);
\draw[ptzer] (-1-.15,-0.6) rectangle ++(.3,1.2);
\draw[ptzer] (1-.15,-0.6) rectangle ++(.3,1.2);
\draw[vthln] (-2,0) -- (-1-.15,0)
 node[near start, below] {$\scriptstyle \xca$}
 (-1+.15,0) -- (-.15,0) (.15,0) -- (1-.15,0)  (1.15,0) -- (2,0)
 node[near end,below] {$\scriptstyle \xca$};
 \draw (-2,0.9) -- (-.15,0.9) (.15,0.9) -- (2,0.9);
\end{tikzpicture}
\;
\hteqv
\;
\boxed{
\shcr
\begin{tikzpicture}[scale=0.5,baseline=-1.5]
\draw[ptone] (-.15,-1.2) rectangle ++(.3,2.4);
\draw[ptzer] (-1-.15,-0.6) rectangle ++(.3,1.2);
\draw[ptzer] (1-.15,-0.6) rectangle ++(.3,1.2);
\draw[vthln] (-2,0) -- (-1-.15,0)
 node[near start, below] {$\scriptstyle \xca$}
 (-1+.15,0) -- (-.15,0) (.15,0) -- (1-.15,0)  (1.15,0) -- (2,0)
 node[near end,below] {$\scriptstyle \xca$};
 \draw (-2,0.9) -- (-.15,0.9) (.15,0.9) -- (2,0.9);
\end{tikzpicture}
\longrightarrow
\begin{tikzpicture}[menvone]
\draw[ptzer] (-.15,-0.6) rectangle ++(0.3,1.2);
\draw[ptzer] (1-.15,-0.6) rectangle ++(0.3,1.2);
\draw[vthln] (-1,0) -- (-.15,0)
 node[near start, below] {$\scriptstyle \xca$}
(.15,0) -- (1-.15,0) (1+.15,0) --(2,0)
 node[near end, below] {$\scriptstyle \xca$};
 \draw (-1,0.9) -- (2,0.9);
\end{tikzpicture}
}
\\
\;\hteqv\;
\boxed{
\shcr
\begin{tikzpicture}[menvone]
\draw[pttwo] (-.15,-0.6) rectangle ++(0.3,1.2);
\draw[thkln] (-1.35,0) -- (-.15,0)
node[near start,below] {$\scriptstyle \xca+1$ }
(.15,0) -- (1.35,0)
node [near end, below] {$\scriptstyle \xca+1$};
\end{tikzpicture}
\longrightarrow
\begin{tikzpicture}[menvone]
 \draw [color=white] (-0.15,-1.2) rectangle ++ (0.3,2.4);
\draw[ptzer] (-.15,-0.6) rectangle ++(0.3,1.2);
\draw[vthln] (-1,0) -- (-.15,0)
 node[near start, below] {$\scriptstyle \xca$}
(.15,0) -- (1,0)
 node[near end, below] {$\scriptstyle \xca$};
 \draw (-1,0.9) -- (1,0.9);
\end{tikzpicture}
}
\end{multline}
The first homotopy equivalence comes from \ex{eq:absorb}, the second follows from \ex{eq:projcn} and the last one follows from \eex{eq:xthrpr} and\rx{eq:cmppr}.
%complex is homotopy equivalent to the right one in \ex{eq:jwpar} because of \ex{eq:swgrpr}.
\end{proof}

\begin{theorem}
\label{thm:jwind}
The $(\xca+1)$-strand \cJWp\ can be presented as the following cone:
\begin{equation}
\label{eq:jwind}
%\xvKhv{
%%%%%%%%%%%%%
%\begin{tikzpicture}[scale=0.5,baseline=-1.5]
% \path[use as bounding box] (-1.2,-0.6) rectangle ++(2.4,1.2);
%\draw [line width=\ljwp] (-0.15,-0.6) rectangle ++(0.3,1.2);
%\draw (-1,0.3) -- (-0.15,0.3) (0.15,0.3) -- (1,0.3);
%\draw [line width=\cblth] (-1,-0.3) -- (-0.15,-0.3) node [near start, below] {$\scriptstyle \xca$} (0.15,-0.3) -- (1,-0.3) node [near end, below] {$\scriptstyle \xca$};
%\end{tikzpicture}
%%%%%%%%%%%%%%%
\begin{tikzpicture}[menvone]
\draw[ptzer] (-.15,-0.6) rectangle ++(0.3,1.2);
\draw[thkln] (-1.35,0) -- (-.15,0)
node[near start,below] {$\scriptstyle \xca+1$ }
(.15,0) -- (1.35,0)
node [near end, below] {$\scriptstyle \xca+1$};
\end{tikzpicture}
%}
\hteqv
\boxed{
\shcr^{2\xca+1}\shfr^{2}
\begin{tikzpicture}[menvone]
\draw[pttwo] (-.15,-0.6) rectangle ++(0.3,1.2);
\draw[thkln] (-1.35,0) -- (-.15,0)
node[near start,below] {$\scriptstyle \xca+1$ }
(.15,0) -- (1.35,0)
node [near end, below] {$\scriptstyle \xca+1$};
\end{tikzpicture}
\longrightarrow
\shcr^{\xca}
\begin{tikzpicture}[scale=0.5,baseline=-1.5]
\draw [xscale=-1]  (1, 0.9) to [out=180,in=0] (0,-0.4);
\draw [line width=6pt, color=white] (-1+0.15,0) -- (0,0);
\draw [thkln] (-1.75,0) -- (-1 - .15,0)
node [near start,below] {$\scriptstyle \xca$}
(-1+0.15,0) -- (0,0) -- (1-0.15,0) (1+.15,0) -- (1.75,0)
node [near end, below] {$\scriptstyle \xca$};
\draw [line width=6pt, color=white] (1, 0.9) to [out=180,in=0] (0,-0.4);
\draw (1, 0.9) to [out=180,in=0] (0,-0.4);
\draw (-1.75,0.9) -- (-1,0.9) (1,0.9) -- (1.75,0.9);
\draw [ptzer] (-1.15,-0.6) rectangle ++(0.3,1.2);
\draw [ptzer] (1 - 0.15,-0.6) rectangle ++(0.3,1.2);
\end{tikzpicture}
}
\end{equation}
%\[
%\begin{tikzpicture}[scale=0.22,baseline=-1.5]
%%\draw [xscale=-1]  (1, 0.9) to [out=180,in=0] (0,-0.4);
%\draw [line width=6pt, color=white] (-1+0.15,0) -- (0,0);
%\draw [line width=\cblth] (-1.75,0) -- (-1 - .15,0)
%node [near start, below] {$\scriptstyle \xca$}
%(-1+0.15,0) -- (0,0) -- (1-0.15,0) (1+.15,0) -- (1.75,0)
%node [near end, below] {$\scriptstyle \xca$};
%%\draw [line width=6pt, color=white] (1, 0.9) to [out=180,in=0] (0,-0.4);
%\draw (1, 0.9) to [out=180,in=90] (0.4,0.6) to [out=-90,in=180] (1-.15,0.3);
%\draw [xscale=-1] (1, 0.9) to [out=180,in=90] (0.4,0.6) to [out=-90,in=180] (1-.15,0.3);
%\draw (-1.75,0.9) -- (-1,0.9) (1,0.9) -- (1.75,0.9);
%\draw [line width=\ljwp] (-1.15,-0.6) rectangle ++(0.3,1.2);
%\draw [line width=\ljwp] (1 - 0.15,-0.6) rectangle ++(0.3,1.2);
%\end{tikzpicture}
%\]
\end{theorem}

\begin{lemma}
\label{lm:grwind}
There is a homotopy equivalence
\begin{equation}
\label{eq:grwind}
\begin{tikzpicture}[scale=0.5,baseline=-1.5]
\draw [xscale=-1]  (1, 0.9) to [out=180,in=0] (0,-0.4);
\draw [line width=6pt, color=white] (-1+0.15,0) -- (0,0);
\draw [thkln] (-1.75+0.3,0) -- (-1 +0.3 - .15,0)
node [near start,below] {$\scriptstyle \xca$}
(-1+0.15,0) -- (0,0) -- (1,0) (1.3,0) -- (2.15,0)
(2.05,0) -- (2.15,0)
node [near end, below] {$\scriptstyle \xca$};
\draw [line width=6pt, color=white] (1, 0.9) to [out=180,in=0] (0,-0.4);
\draw (1, 0.9) to [out=180,in=0] (0,-0.4);
\draw (-1.75+0.3,0.9) -- (-1,0.9) (1.3,0.9) -- (2.15,0.9); % (2.05,0.9) -- (2.9,0.9);
\draw [pttwo] (1,-1.2) rectangle ++(0.3,2.4);
\end{tikzpicture}
\hteqv
\shcr^{\xca}\shfr^2
\begin{tikzpicture}[menvone]
\draw[pttwo] (-.15,-0.6) rectangle ++(0.3,1.2);
\draw[thkln] (-1.35,0) -- (-.15,0)
node[near start,below] {$\scriptstyle \xca+1$ }
(.15,0) -- (1.35,0)
node [near end, below] {$\scriptstyle \xca+1$};
\end{tikzpicture}
\end{equation}
\end{lemma}
\begin{proof}
Consider the composition of the line winding around the $\xca$-cable with the left portion of the complex
$
\begin{tikzpicture}[scale=0.4,baseline=-1.5]
%\draw [xscale=-1]  (1, 0.9) to [out=180,in=0] (0,-0.4);
\draw [line width=6pt, color=white] (-1+0.15,0) -- (0,0);
\draw [line width=\cblth] (-1.75,0) -- (-1 - .15,0)
%node [near start, below] {$\scriptstyle \xca$}
(-1+0.15,0) -- (0,0) -- (1-0.15,0) (1+.15,0) -- (1.75,0);
%node [near end, below] {$\scriptstyle \xca$};
%\draw [line width=6pt, color=white] (1, 0.9) to [out=180,in=0] (0,-0.4);
\draw (1, 0.9) to [out=180,in=90] (0.4,0.6) to [out=-90,in=180] (1-.15,0.3);
\draw [xscale=-1] (1, 0.9) to [out=180,in=90] (0.4,0.6) to [out=-90,in=180] (1-.15,0.3);
\draw (-1.75,0.9) -- (-1,0.9) (1,0.9) -- (1.75,0.9);
\draw [line width=\ljwp] (-1.15,-0.6) rectangle ++(0.3,1.2);
\draw [line width=\ljwp] (1 - 0.15,-0.6) rectangle ++(0.3,1.2);
\end{tikzpicture}
$
which generates the \tmcn\rx{eq:swgrpr}:
\begin{equation}
\label{eq:windjw}
\begin{tikzpicture}[scale=0.75,baseline=-1.5]
\draw [xscale=-1]  (1, 0.9) to [out=180,in=0] (0,-0.4);
\draw [line width=6pt, color=white] (-1+0.15,0) -- (0,0);
\draw [thkln] (-1.15,0) -- (-1 + .15,0)
node [near start,below] {$\scriptstyle \xca$}
(-1+0.15,0) -- (0,0) -- (1-0.15,0);
%(1+.15,0) -- (1.75,0);
%(2.05,0) -- (2.65,0) (2.95,0)  -- (3.55,0)
%node [near end, below] {$\scriptstyle \xca$};
\draw [line width=6pt, color=white] (1, 0.9) to [out=180,in=0] (0,-0.4);
\draw (1, 0.9) to [out=180,in=0] (0,-0.4);
\draw (-1.15,0.9) -- (-1,0.9); % (1,0.9) -- (1.75,0.9) (2.05,0.9) -- (3.55,0.9);
%\draw [line width=\ljwp] (-1.15,-0.6) rectangle ++(0.3,1.2);
\draw [line width=\ljwp] (1 - 0.15,-0.6) rectangle ++(0.3,1.2);
%%%%%%%
%\draw [line width=\ljwp] (1.75,-1.2) rectangle ++(0.3,2.4);
%\draw [line width=\ljwp] (1.75+0.9,-0.6) rectangle ++(0.3,1.2);
%%%%%%%%%%%%%%%%%%%%%%%%%%%%%%%%%%%%%%%%%%%
\begin{scope}[xshift=2cm]
%\draw [line width=6pt, color=white] (-1+0.15,0) -- (0,0);
\draw [thkln]
%(-1.75,0) -- (-1 - .15,0) node [near start, below] {$\scriptstyle \xca$}
(-1+0.15,0) -- (-0.15+0.3,0) % -- (1-0.15,0)
%(1+.15,0) -- (1.75,0)
node [near end, below] {$\scriptstyle \xca-1$};
%\draw [line width=6pt, color=white] (1, 0.9) to [out=180,in=0] (0,-0.4);
%\draw (1, 0.9) to [out=180,in=90] (0.4,0.6) to [out=-90,in=180] (1-.15,0.3);
\draw [xscale=-1] (1, 0.9) to [out=180,in=90] (0.4,0.6) to [out=-90,in=180] (1-.15,0.3);
%\draw (-1.75,0.9) -- (-1,0.9);
%\draw (1,0.9) -- (1.75,0.9);
%\draw [line width=\ljwp] (-1.15,-0.6) rectangle ++(0.3,1.2);
%\draw [line width=\ljwp] (1 - 0.15,-0.6) rectangle ++(0.3,1.2);
\end{scope}
\end{tikzpicture}
\;\hteqv\;
\begin{tikzpicture}[scale=0.75,baseline=-1.5]
\draw [xscale=-1]  (1, 0.9) to [out=180,in=0] (0,-0.4);
\draw [line width=4pt, color=white] (-1+0.15,0) -- (0,0);
\draw [thkln] (-1.75,0) -- (-1 - .15,0)
node [near start,below] {$\scriptstyle \xca$}
(-1+0.15,0) -- (0,0) -- (1+0.15,0);
%%%%%%%%%%%%
\draw [line width=4pt, color=white] (-1+0.15,0.3) -- (0,0.3);
\draw (-1+0.15,0.3)  -- (1+0.15,0.3);
%%%%%%%%%%%%%%
\draw [line width=6pt, color=white] (1, 0.9) to [out=180,in=0] (0,-0.4);
\draw (1, 0.9) to [out=180,in=0] (0,-0.4);
\draw (-1.75,0.9) -- (-1,0.9); % (1,0.9) -- (1.75,0.9) (2.05,0.9) -- (3.55,0.9);
\draw [ptzer] (-1.15,-0.6) rectangle ++(0.3,1.2);
%\draw [line width=\ljwp] (1 - 0.15,-0.6) rectangle ++(0.3,1.2);
%%%%%%%%%%%%%%%%%%%%%%%%%%%%%%%%%%%%%%%%%%%
\begin{scope}[xshift=2cm]
\draw [line width=\cblth]
(-1+0.15,0) -- (-0.15-0.2,0) % -- (1-0.15,0)
node [near end, below] {$\scriptstyle \xca-1$};
\draw [xscale=-1] (1, 0.9) to [out=180,in=90] (0.4,0.6) to [out=-90,in=180] (1-.15,0.3);
\end{scope}
\end{tikzpicture}
\;\hteqv\;
\shcr \shfr^{2}
\begin{tikzpicture}[menvtwo]
%\draw (0.15,0) -- (3,0);
\draw [thkln] (-0.75,0) -- (-0.15,0)
node [near start, below] {$\scriptstyle \xca$}
(0.15,0) -- (1,0);
\draw [lnovr] (0.15,0.3) to [out=0,in=180] (1,-0.4);
\draw (0.15,0.3) to [out=0,in=180] (1,-0.4);
\draw (1,-0.4) .. controls +(0:0.8) and +(0:1.2) .. (0.2,0.9) -- (-0.75,0.9);
\draw [lnovr] (1.1,0) -- (2,0);
\draw [thkln] (1,0) -- (2.2,0)
node [very near end, below] {$\scriptstyle \xca-1$};
%\draw [fill=white] (0,0.9) circle (3pt)
%node [above] {$\scriptstyle 2$};
\draw[ptzer] (-0.15,-0.6) rectangle (0.15,0.6);
\end{tikzpicture}
\;\hteqv\;
\shcr^{\xca}\shfr^{2}
\begin{tikzpicture}[menvtwo]
\draw [ptzer] (-0.15,-0.6) rectangle (0.15,0.6);
\draw [thkln] (-0.75,0) -- (-0.15,0)
node [near start, below] {$\scriptstyle \xca$} (0.15,0) -- (1,0)
node [near end,below] {$\scriptstyle \xca-1$};
\draw [xscale=-1,xshift=-1cm] (1, 0.9) to [out=180,in=90] (0.4,0.6) to [out=-90,in=180] (1-.15,0.3);
\draw (-0.75,0.9) -- (0,0.9);
\end{tikzpicture}
\end{equation}
Here the first equivalence is purely topological: the projector is moved left along the cable, the second equivalence uses \ex{eq:frsh} to remove two framing kinks on the single line and
% two equivalences are purely topological: the projector is moved left along the cable and the single line winding is simplified, while
 the third equivalence follows from \ex{eq:projtw}.
 % and\rx{eq:frsh}.
%
The equivalence\rx{eq:grwind} comes from applying equivalence\rx{eq:windjw} to every constituent complex in the \tmcn\rx{eq:swgrpr}.
%follows easily from the equivalence\rx{eq:windjw}.
\end{proof}
\begin{proof}[Proof of Theorem\rw{thm:jwind}]
Eq.\rx{eq:jwind} follows from a sequence of homotopy equivalences:
\begin{multline*}
\begin{tikzpicture}[scale=0.5,baseline=-1.5]
\draw [ptzer] (1.75,-0.6) rectangle ++(0.3,1.2);
\draw [thkln] (1.75-0.85-0.3,0) -- (1.75,0)
node [very near start, below] {$\scriptstyle \xca+1$}
(2.05,0) -- (2.9+0.3,0)
node [very near end, below] {$\scriptstyle \xca+1$};
\end{tikzpicture}
\;\hteqv\;
\shcr^{\xca}
\begin{tikzpicture}[scale=0.5,baseline=-1.5]
\draw [xscale=-1]  (1, 0.9) to [out=180,in=0] (0,-0.4);
\draw [line width=6pt, color=white] (-1+0.15,0) -- (0,0);
\draw [thkln] (-1.75+0.3,0) -- (-1 +0.3 - .15,0)
node [near start,below] {$\scriptstyle \xca$}
(-1+0.15,0) -- (0,0) -- (1,0) (1.3,0) -- (2.15,0)
(2.05,0) -- (2.15,0)
node [near end, below] {$\scriptstyle \xca$};
\draw [line width=6pt, color=white] (1, 0.9) to [out=180,in=0] (0,-0.4);
\draw (1, 0.9) to [out=180,in=0] (0,-0.4);
\draw (-1.75+0.3,0.9) -- (-1,0.9) (1.3,0.9) -- (2.15,0.9); % (2.05,0.9) -- (2.9,0.9);
\draw [ptzer] (1,-1.2) rectangle ++(0.3,2.4);
\end{tikzpicture}
\;\hteqv\;
\boxed{
\shcr^{\xca+1}
\begin{tikzpicture}[scale=0.5,baseline=-1.5]
\draw [xscale=-1]  (1, 0.9) to [out=180,in=0] (0,-0.4);
\draw [line width=6pt, color=white] (-1+0.15,0) -- (0,0);
\draw [thkln] (-1.75+0.3,0) -- (-1 +0.3 - .15,0)
node [near start,below] {$\scriptstyle \xca$}
(-1+0.15,0) -- (0,0) -- (1,0) (1.3,0) -- (2.15,0)
(2.05,0) -- (2.15,0)
node [near end, below] {$\scriptstyle \xca$};
\draw [line width=6pt, color=white] (1, 0.9) to [out=180,in=0] (0,-0.4);
\draw (1, 0.9) to [out=180,in=0] (0,-0.4);
\draw (-1.75+0.3,0.9) -- (-1,0.9) (1.3,0.9) -- (2.15,0.9); % (2.05,0.9) -- (2.9,0.9);
\draw [pttwo] (1,-1.2) rectangle ++(0.3,2.4);
\end{tikzpicture}
\longrightarrow
\shcr^{\xca}
\begin{tikzpicture}[scale=0.5,baseline=-1.5]
\draw [xscale=-1]  (1, 0.9) to [out=180,in=0] (0,-0.4);
\draw [line width=6pt, color=white] (-1+0.15,0) -- (0,0);
\draw [thkln] (-1.75+0.3,0) -- (-1 - .15+0.3,0)
node [near start,below] {$\scriptstyle \xca$}
(-1+0.15,0) -- (0,0) -- (1-0.15,0) (1+.15,0) -- (1.75,0)
node [near end, below] {$\scriptstyle \xca$};
\draw [line width=6pt, color=white] (1, 0.9) to [out=180,in=0] (0,-0.4);
\draw (1, 0.9) to [out=180,in=0] (0,-0.4);
\draw (-1.75+0.3,0.9) -- (-1,0.9) (1,0.9) -- (1.75,0.9);
%\draw [line width=\ljwp] (-1.15,-0.6) rectangle ++(0.3,1.2);
\draw [line width=\ljwp] (1 - 0.15,-0.6) rectangle ++(0.3,1.2);
\end{tikzpicture}
}
\\
\;\hteqv\;
\boxed{
\shcr^{2\xca+1}\shfr^{2}
\begin{tikzpicture}[menvone]
\draw[pttwo] (-.15,-0.6) rectangle ++(0.3,1.2);
\draw[thkln] (-1.35,0) -- (-.15,0)
node[near start,below] {$\scriptstyle \xca+1$ }
(.15,0) -- (1.35,0)
node [near end, below] {$\scriptstyle \xca+1$};
\end{tikzpicture}
\longrightarrow
\shcr^{\xca}
\begin{tikzpicture}[scale=0.5,baseline=-1.5]
\draw [xscale=-1]  (1, 0.9) to [out=180,in=0] (0,-0.4);
\draw [line width=6pt, color=white] (-1+0.15,0) -- (0,0);
\draw [thkln] (-1.75,0) -- (-1 - .15,0)
node [near start,below] {$\scriptstyle \xca$}
(-1+0.15,0) -- (0,0) -- (1-0.15,0) (1+.15,0) -- (1.75,0)
node [near end, below] {$\scriptstyle \xca$};
\draw [line width=6pt, color=white] (1, 0.9) to [out=180,in=0] (0,-0.4);
\draw (1, 0.9) to [out=180,in=0] (0,-0.4);
\draw (-1.75,0.9) -- (-1,0.9) (1,0.9) -- (1.75,0.9);
\draw [ptzer] (-1.15,-0.6) rectangle ++(0.3,1.2);
\draw [ptzer] (1 - 0.15,-0.6) rectangle ++(0.3,1.2);
\end{tikzpicture}
}
\end{multline*}
Here the first equivalence follows from \ex{eq:projtw}, the second equivalence follows from \ex{eq:jwpar}, the third equivalence follows from \eex{eq:grwind} and\rx{eq:cmppr}.
\end{proof}

\begin{theorem}
\label{thm:projcr}
The $(\xca+1)$-strand \cJWp\ is homotopy equivalent to a cone
\begin{equation}
\label{eq:jwcrs}
%%%%%%%%%%%%%%%
\begin{tikzpicture}[menvone]
\draw[ptzer] (-.15,-0.6) rectangle ++(0.3,1.2);
\draw[thkln] (-1.35,0) -- (-.15,0)
node[near start,below] {$\scriptstyle \xca+1$ }
(.15,0) -- (1.35,0)
node [near end, below] {$\scriptstyle \xca+1$};
\end{tikzpicture}
%}
\hteqv
\boxed{
\shcr^{2\xca}\shfr
\begin{tikzpicture}[menvone]
\draw[ptthr] (-.15,-0.6) rectangle ++(0.3,1.2);
\draw[thkln] (-1.35,0) -- (-.15,0)
node[near start,below] {$\scriptstyle \xca+1$ }
(.15,0) -- (1.35,0)
node [near end, below] {$\scriptstyle \xca+1$};
\end{tikzpicture}
\longrightarrow
\shcr^{\hlf}
\begin{tikzpicture}[menvone]
%\draw [xscale=-1]  (1, 0.9) to [out=180,in=0] (0,-0.4);
\draw [line width=6pt, color=white] (-1+0.15,0) -- (0,0);
\draw [thkln] (-1.75,0) -- (-1 - .15,0)
node [near start, below] {$\scriptstyle \xca$}
(-1+0.15,0) -- (0,0) -- (1-0.15,0) (1+.15,0) -- (1.75,0)
node [near end, below] {$\scriptstyle \xca$};
%\draw (1, 0.9) to [out=180,in=90] (0.4,0.6) to [out=-90,in=180] (1-.15,0.3);
%\draw [xscale=-1] (1, 0.9) to [out=180,in=90] (0.4,0.6) to [out=-90,in=180] (1-.15,0.3);
\draw [xscale=-1] (-1+0.15,0.3) -- (-0.5,0.3) to [out=0, in=180] (0.5,0.9) -- (1,0.9);
\draw [lnovr] (-0.5,0.3) to [out=0, in=180] (0.5,0.9);
\draw (-1+0.15,0.3) -- (-0.5,0.3) to [out=0, in=180] (0.5,0.9) -- (1,0.9);
\draw (-1.75,0.9) -- (-1,0.9) (1,0.9) -- (1.75,0.9);
\draw [ptzer] (-1.15,-0.6) rectangle ++(0.3,1.2);
\draw [ptzer] (1 - 0.15,-0.6) rectangle ++(0.3,1.2);
%\draw [dotted] (-0.6,-0.7) rectangle (0.6,1.2);
\end{tikzpicture}
}
\end{equation}
in which the complex
$
\begin{tikzpicture}[scale=0.4,baseline=-3]
\draw[ptthr] (-.15,-0.6) rectangle ++(0.3,1.2);
\draw[vthln] (-0.75,0) -- (-.15,0) (.15,0) -- (.75,0);
\end{tikzpicture}
$
has the following \tmcn\ presentation:
\begin{equation}
\label{eq:swnepr}
\begin{tikzpicture}[menvone]
\draw[ptthr] (-.15,-0.6) rectangle ++(0.3,1.2);
\draw[thkln] (-1.35,0) -- (-.15,0)
node[near start,below] {$\scriptstyle \xca+1$ }
(.15,0) -- (1.35,0)
node [near end, below] {$\scriptstyle \xca+1$};
\end{tikzpicture}
\;
\hteqv
\;
\boxed{
\cdots\longrightarrow
\shcr^i \bigoplus_{j=0}^i
\tprmltij \shfr^j
\begin{tikzpicture}[menvone]
%\draw [xscale=-1]  (1, 0.9) to [out=180,in=0] (0,-0.4);
\draw [line width=6pt, color=white] (-1+0.15,0) -- (0,0);
\draw [thkln] (-1.75,0) -- (-1 - .15,0)
node [near start, below] {$\scriptstyle \xca$}
(-1+0.15,0) -- (0,0) -- (1-0.15,0) (1+.15,0) -- (1.75,0)
node [near end, below] {$\scriptstyle \xca$};
%\draw [line width=6pt, color=white] (1, 0.9) to [out=180,in=0] (0,-0.4);
\draw (1, 0.9) to [out=180,in=90] (0.4,0.6) to [out=-90,in=180] (1-.15,0.3);
\draw [xscale=-1] (1, 0.9) to [out=180,in=90] (0.4,0.6) to [out=-90,in=180] (1-.15,0.3);
\draw (-1.75,0.9) -- (-1,0.9) (1,0.9) -- (1.75,0.9);
\draw [ptzer] (-1.15,-0.6) rectangle ++(0.3,1.2);
\draw [ptzer] (1 - 0.15,-0.6) rectangle ++(0.3,1.2);
\end{tikzpicture}
\longrightarrow\cdots
}_{\;i=0}^{\infty},
\quad
\tprmltij =
\begin{cases}
\prmltv{i-1,j-1} & \text{if $i\geq 1$},
\\
1 & \text{if $i=0$}.
\end{cases}
\end{equation}
\end{theorem}

\begin{lemma}
There is a homotopy equivalence
\begin{equation}
\label{eq:wdcr}
\begin{tikzpicture}[scale=0.5,baseline=-1.5]
\draw [xscale=-1]  (1, 0.9) to [out=180,in=0] (0,-0.4);
\draw [line width=6pt, color=white] (-1+0.15,0) -- (0,0);
\draw [thkln] (-1.75,0) -- (-1 - .15,0)
node [near start,below] {$\scriptstyle \xca$}
(-1+0.15,0) -- (0,0) -- (1-0.15,0) (1+.15,0) -- (1.75,0)
node [near end, below] {$\scriptstyle \xca$};
\draw [line width=6pt, color=white] (1, 0.9) to [out=180,in=0] (0,-0.4);
\draw (1, 0.9) to [out=180,in=0] (0,-0.4);
\draw (-1.75,0.9) -- (-1,0.9) (1,0.9) -- (1.75,0.9);
\draw [ptzer] (-1.15,-0.6) rectangle ++(0.3,1.2);
\draw [ptzer] (1 - 0.15,-0.6) rectangle ++(0.3,1.2);
\end{tikzpicture}
\;\hteqv\;
\boxed{
\shcr^{\xca}\shfr
\begin{tikzpicture}[menvone]
%\draw [xscale=-1]  (1, 0.9) to [out=180,in=0] (0,-0.4);
\draw [line width=6pt, color=white] (-1+0.15,0) -- (0,0);
\draw [thkln] (-1.75,0) -- (-1 - .15,0)
node [near start, below] {$\scriptstyle \xca$}
(-1+0.15,0) -- (0,0) -- (1-0.15,0) (1+.15,0) -- (1.75,0)
node [near end, below] {$\scriptstyle \xca$};
\draw (1, 0.9) to [out=180,in=90] (0.4,0.6) to [out=-90,in=180] (1-.15,0.3);
\draw [xscale=-1] (1, 0.9) to [out=180,in=90] (0.4,0.6) to [out=-90,in=180] (1-.15,0.3);
\draw (-1.75,0.9) -- (-1,0.9) (1,0.9) -- (1.75,0.9);
\draw [ptzer] (-1.15,-0.6) rectangle ++(0.3,1.2);
\draw [ptzer] (1 - 0.15,-0.6) rectangle ++(0.3,1.2);
%\draw [dotted] (-0.6,-0.7) rectangle (0.6,1.2);
\end{tikzpicture}
\longrightarrow
\shcr^{-\xca+\hlf}
\begin{tikzpicture}[menvone]
%\draw [xscale=-1]  (1, 0.9) to [out=180,in=0] (0,-0.4);
\draw [line width=6pt, color=white] (-1+0.15,0) -- (0,0);
\draw [thkln] (-1.75,0) -- (-1 - .15,0)
node [near start, below] {$\scriptstyle \xca$}
(-1+0.15,0) -- (0,0) -- (1-0.15,0) (1+.15,0) -- (1.75,0)
node [near end, below] {$\scriptstyle \xca$};
%\draw (1, 0.9) to [out=180,in=90] (0.4,0.6) to [out=-90,in=180] (1-.15,0.3);
%\draw [xscale=-1] (1, 0.9) to [out=180,in=90] (0.4,0.6) to [out=-90,in=180] (1-.15,0.3);
\draw [xscale=-1] (-1+0.15,0.3) -- (-0.5,0.3) to [out=0, in=180] (0.5,0.9) -- (1,0.9);
\draw [lnovr] (-0.5,0.3) to [out=0, in=180] (0.5,0.9);
\draw (-1+0.15,0.3) -- (-0.5,0.3) to [out=0, in=180] (0.5,0.9) -- (1,0.9);
\draw (-1.75,0.9) -- (-1,0.9) (1,0.9) -- (1.75,0.9);
\draw [line width=\ljwp] (-1.15,-0.6) rectangle ++(0.3,1.2);
\draw [line width=\ljwp] (1 - 0.15,-0.6) rectangle ++(0.3,1.2);
%\draw [dotted] (-0.6,-0.7) rectangle (0.6,1.2);
\end{tikzpicture}
}
\end{equation}
\end{lemma}
\begin{proof}
The lemma is proved by applying \tKhbr\ formula\rx{eq:dkhbr} to one of the elementary crossings in the \lhs diagram:
\def\cht{0.4}
\[
\begin{tikzpicture}[scale=0.75,baseline=-1.5]
\draw [xscale=-1]  (1, 0.9) to [out=180,in=0] (0,-0.4);
\draw [lnovr] (-1+0.15,0) -- (0,0);
\draw [lnovr] (-1+0.15,\cht) -- (0,\cht);
\draw [thkln] (-1.75,0) -- (-1 - .15,0)
node [near start,below] {$\scriptstyle \xca$}
(-1+0.15,0) -- (0,0) -- (1-0.15,0) (1+.15,0) -- (1.75,0)
node [near end, below] {$\scriptstyle \xca$};
\draw (-1+0.15,\cht) -- (1-0.15,\cht);
\draw [line width=6pt, color=white] (1, 0.9) to [out=180,in=0] (0,-0.4);
\draw (1, 0.9) to [out=180,in=0] (0,-0.4);
\draw (-1.75,0.9) -- (-1,0.9) (1,0.9) -- (1.75,0.9);
\draw [ptzer] (-1.15,-0.6) rectangle ++(0.3,1.2);
\draw [ptzer] (1 - 0.15,-0.6) rectangle ++(0.3,1.2);
\end{tikzpicture}
\hteqv
\boxed{
\shcr^{\hlf}
%%%%%%%%%%%%%%%%%%%%%%%
\begin{tikzpicture}[scale=0.75,baseline=-1.5]
%\draw [xscale=-1]  (1, 0.9) to [out=180,in=0] (0,-0.4);
\draw (0,-0.4) to [out=-180,in=-90] (-0.3,-0.05) to [out=90,in=-180] (0,\cht);
\draw [lnovr] (-1+0.15,0) -- (0,0);
%\draw [lnovr] (-1+0.15,0.3) -- (0,0.3);
\draw [thkln] (-1.75,0) -- (-1 - .15,0)
node [near start,below] {$\scriptstyle \xca$}
(-1+0.15,0) -- (0,0) -- (1-0.15,0) (1+.15,0) -- (1.75,0)
node [near end, below] {$\scriptstyle \xca$};
%\draw (-1+0.15,0.3) -- (1-0.15,0.3);
\draw (0,\cht) -- (1-0.15,\cht);
\draw (-1+0.15,\cht) to [out=0,in=-90] (-0.5,0.6) to [out=90,in=0] (-1,0.9);
\draw [line width=6pt, color=white] (1, 0.9) to [out=180,in=0] (0,-0.4);
\draw (1, 0.9) to [out=180,in=0] (0,-0.4);
\draw (-1.75,0.9) -- (-1,0.9) (1,0.9) -- (1.75,0.9);
\draw [ptzer] (-1.15,-0.6) rectangle ++(0.3,1.2);
\draw [ptzer] (1 - 0.15,-0.6) rectangle ++(0.3,1.2);
\end{tikzpicture}
%%%%%%%%%%%%%%%%%%%%%%
\longrightarrow
\shcr^{-\hlf}
%%%%%%%%%%%%%%%%%%%%%%
\begin{tikzpicture}[scale=0.75,baseline=-1.5]
\draw [xscale=-1]  (1-0.15, \cht) to [out=180,in=0] (0,-0.4);
\draw (-1,0.9) to [out=0,in=180] (0,\cht);
\draw [lnovr] (-1+0.15,0) -- (0,0);
%\draw [lnovr] (-1+0.15,0.3) -- (0,0.3);
\draw [thkln] (-1.75,0) -- (-1 - .15,0)
node [near start,below] {$\scriptstyle \xca$}
(-1+0.15,0) -- (0,0) -- (1-0.15,0) (1+.15,0) -- (1.75,0)
node [near end, below] {$\scriptstyle \xca$};
\draw (0,\cht) -- (1-0.15,\cht);
\draw [line width=6pt, color=white] (1, 0.9) to [out=180,in=0] (0,-0.4);
\draw (1, 0.9) to [out=180,in=0] (0,-0.4);
\draw (-1.75,0.9) -- (-1,0.9) (1,0.9) -- (1.75,0.9);
\draw [ptzer] (-1.15,-0.6) rectangle ++(0.3,1.2);
\draw [ptzer] (1 - 0.15,-0.6) rectangle ++(0.3,1.2);
\end{tikzpicture}
%%%%%%%%%%%%%%%%%%%%%%%
}
\]
The diagrams in the \rhs cone are reduced to those of \ex{eq:wdcr} with the help of \eex{eq:frsh} and\rx{eq:projtw}.
\end{proof}
\begin{proof}[Proof of Theorem\rw{thm:projcr}]
A substitution of \ex{eq:wdcr} into \ex{eq:jwind} yields the cone presentation\rx{eq:jwcrs} with
\[
\begin{tikzpicture}[menvone]
\draw[ptthr] (-.15,-0.6) rectangle ++(0.3,1.2);
\draw[thkln] (-1.35,0) -- (-.15,0)
node[near start,below] {$\scriptstyle \xca+1$ }
(.15,0) -- (1.35,0)
node [near end, below] {$\scriptstyle \xca+1$};
\end{tikzpicture}
=
\boxed{
\shcr\shfr
\begin{tikzpicture}[menvone]
\draw[pttwo] (-.15,-0.6) rectangle ++(0.3,1.2);
\draw[thkln] (-1.35,0) -- (-.15,0)
node[near start,below] {$\scriptstyle \xca+1$ }
(.15,0) -- (1.35,0)
node [near end, below] {$\scriptstyle \xca+1$};
\end{tikzpicture}
\longrightarrow
\begin{tikzpicture}[menvone]
\draw[color=white] (-0.15,-1.2) rectangle (0.15,1.2);
\draw [line width=6pt, color=white] (-1+0.15,0) -- (0,0);
\draw [thkln] (-1.75,0) -- (-1 - .15,0)
node [near start, below] {$\scriptstyle \xca$}
(-1+0.15,0) -- (0,0) -- (1-0.15,0) (1+.15,0) -- (1.75,0)
node [near end, below] {$\scriptstyle \xca$};
\draw (1, 0.9) to [out=180,in=90] (0.4,0.6) to [out=-90,in=180] (1-.15,0.3);
\draw [xscale=-1] (1, 0.9) to [out=180,in=90] (0.4,0.6) to [out=-90,in=180] (1-.15,0.3);
\draw (-1.75,0.9) -- (-1,0.9) (1,0.9) -- (1.75,0.9);
\draw [ptzer] (-1.15,-0.6) rectangle ++(0.3,1.2);
\draw [ptzer] (1 - 0.15,-0.6) rectangle ++(0.3,1.2);
\end{tikzpicture}
}
\]
and \ex{eq:swnepr} follows from \ex{eq:swgrpr}.
\end{proof}

If we connect the endpoints of the upper single line in \ex{eq:jwcrs} and apply the framing relation\rx{eq:frsh} to the last diagram in that relation, then we come to the following corollary
\begin{corollary}
There is a homotopy equivalence
\begin{equation}
\label{eq:htloop}
\begin{tikzpicture}[menvone]
\draw [ptzer] (-0.15,-0.6) rectangle (.15,.6);
\draw [thkln] (-0.75,0) -- (-0.15,0)
node [near start,below] {$\scriptstyle \xca$}
(0.15,0) -- (.75,0)
node [near end,below] {$\scriptstyle \xca$};
\draw (-0.15,0.4) to [out=180,in=-90] (-0.6,0.8) to [out=90,in=180] (0,1.4)
to [out=0,in=90] (0.6,0.8) to [out=-90,in=0] (0.15,0.4);
%\draw (-0.75,0.4) -- (-0.15,0.4) (0.15,0.4) -- (0.75,0.4);
\end{tikzpicture}
\;\hteqv\;
\boxed{
\shcr^{2\xca}\shfr
\begin{tikzpicture}[menvone]
\draw [color=white] (-0.15,-1) rectangle (0.15,1.6);
\draw [ptthr] (-0.15,-0.6) rectangle (.15,.6);
\draw [thkln] (-0.75,0) -- (-0.15,0)
node [near start,below] {$\scriptstyle \xca$}
(0.15,0) -- (.75,0)
node [near end,below] {$\scriptstyle \xca$};
\draw (-0.15,0.4) to [out=180,in=-90] (-0.6,0.8) to [out=90,in=180] (0,1.4)
to [out=0,in=90] (0.6,0.8) to [out=-90,in=0] (0.15,0.4);
%\draw (-0.75,0.4) -- (-0.15,0.4) (0.15,0.4) -- (0.75,0.4);
\end{tikzpicture}
\longrightarrow
\shfr^{-1}
\begin{tikzpicture}[menvone]
\draw [ptzer] (-0.15,-0.6) rectangle (.15,.6);
\draw [thkln] (-0.75,0) -- (-0.15,0)
node [near start,below] {$\scriptstyle \xca$}
(0.15,0) -- (.75,0)
node [near end,below] {$\scriptstyle \xca$};
\end{tikzpicture}
}
\end{equation}
where
\begin{equation}
\label{eq:grlp}
\begin{tikzpicture}[menvone]
\draw [color=white] (-0.15,-1) rectangle (0.15,1.6);
\draw [ptthr] (-0.15,-0.6) rectangle (.15,.6);
\draw [thkln] (-0.75,0) -- (-0.15,0)
node [near start,below] {$\scriptstyle \xca$}
(0.15,0) -- (.75,0)
node [near end,below] {$\scriptstyle \xca$};
\draw (-0.15,0.4) to [out=180,in=-90] (-0.6,0.8) to [out=90,in=180] (0,1.4)
to [out=0,in=90] (0.6,0.8) to [out=-90,in=0] (0.15,0.4);
%\draw (-0.75,0.4) -- (-0.15,0.4) (0.15,0.4) -- (0.75,0.4);
\end{tikzpicture}
\;
\hteqv
\;
\boxed{
\cdots\longrightarrow
\shcr^i \bigoplus_{j=0}^i
\tprmltij \shfr^j
\begin{tikzpicture}[menvone]
\draw [ptzer] (-0.15,-0.6) rectangle (.15,.6);
\draw [thkln] (-0.75,0) -- (-0.15,0)
node [near start,below] {$\scriptstyle \xca$}
(0.15,0) -- (.75,0)
node [near end,below] {$\scriptstyle \xca$};
\end{tikzpicture}
\longrightarrow\cdots
}_{\;i=0}^{\infty}.
%\quad
%\ymlt_{ij} =
%\begin{cases}
%\xmlt_{i-1,j-1} & \text{if $i\geq 1$},
%\\
%1 & \text{if $i=0$}.
%\end{cases}
\end{equation}
\end{corollary}

\section{The morphism $\mnfN$}
\label{sct:mrph}

\subsection{General setup}

For a link diagram $\xD$ we give a precise definition of a diagram (a complex) $\xDclN$. $\xDclN$ is constructed by first $\xca$-cabling all components of $\xD$ and then placing a \cJWp\ at every edge of $\xD$, an edge being a piece of $\xca$-cabled strand between two crossings.

The map $\mnfN$ of \ex{eq:spmaps} is a composition of many maps between \tKhoms\ of a sequence of diagrams related by \ltrf s, the first diagram in that sequence being $\xDclNo$ and the last being
%related to each other by \ltrf s
$\xDclN$ (recall that maps go backwards).

We use three types of \ltrf s, which are based on the following \lrpl s:
\begin{gather}
\def\mxsh{1.5}
\def\mysh{1.2}
%%%%%%%%%%%%%%%%%%%%%%%%%%%%%%%%%%%%%%%%%
\begin{tikzpicture}[menvone]
\node at (-12,0) {(I):};
\node (l) at (-1.8,0) {};
\node (r) at (1.8,0) {};
\path[commutative diagrams/.cd, every arrow, every label]
(l) edge[commutative diagrams/squiggly] (r);
\begin{scope}[xshift=-6cm]
\draw [thkln] (-\mxsh+0.15,\mysh) to [out=0,in=180] (\mxsh-0.15,-\mysh);
\draw [lnovr] (-\mxsh+0.15,-\mysh) to [out=0,in=180] (\mxsh-0.15,\mysh);
\draw [thkln] (-\mxsh+0.15,-\mysh) to [out=0,in=180] (\mxsh-0.15,\mysh);
\draw [bcrc] (-0.75-0.6-\mxsh,-\mysh) -- (0.75+0.6+\mxsh,-\mysh);
\draw [bcrc] (-0.75-0.6-\mxsh,\mysh) -- (0.75+0.6+\mxsh,\mysh);
%
%\draw [lnovr] (-0.75-\mxsh,0) to [out=0,in=180] (0,-\mysh) to [out=0,in=180] (0.75+\mxsh,0);
%\draw (-0.75-\mxsh,0) to [out=0,in=180] (0,-\mysh) to [out=0,in=180] (0.75+\mxsh,0);
%
\draw [ptzert] (-0.15-\mxsh,-0.6-\mysh) rectangle ++(.3,1.2);
\draw [thkln] (-0.75-0.6-\mxsh,-\mysh) -- (-0.15-\mxsh,-\mysh)
node [very near start,below] {$\scriptstyle \xca+1$};
\begin{scope}[xscale=-1]
\draw [ptzert] (-0.15-\mxsh,-0.6-\mysh) rectangle ++(.3,1.2);
\draw [thkln] (-0.75-0.6-\mxsh,-\mysh) -- (-0.15-\mxsh,-\mysh)
node [very near start,below] {$\scriptstyle \xca+1$};
\end{scope}
\begin{scope}[yscale=-1]
\draw [ptzert] (-0.15-\mxsh,-0.6-\mysh) rectangle ++(.3,1.2);
\draw [thkln] (-0.75-0.6-\mxsh,-\mysh) -- (-0.15-\mxsh,-\mysh)
node [very near start,above] {$\scriptstyle \xca+1$};
\end{scope}
\begin{scope}[yscale=-1,xscale=-1]
\draw [ptzert] (-0.15-\mxsh,-0.6-\mysh) rectangle ++(.3,1.2);
\draw [thkln] (-0.75-0.6-\mxsh,-\mysh) -- (-0.15-\mxsh,-\mysh)
node [near start,above] {$\scriptstyle \xca+1$};
\end{scope}
\end{scope}
%%%%%%%%%%%
\begin{scope}[xshift=6cm]
\draw [thkln] (-\mxsh+0.15,\mysh) to [out=0,in=180]
node [sloped, near end,above] {$\scriptstyle \xca$} (\mxsh-0.15,-\mysh);
\draw [lnovr] (-\mxsh+0.15,-\mysh) to [out=0,in=180] (\mxsh-0.15,\mysh);
\draw [thkln] (-\mxsh+0.15,-\mysh) to [out=0,in=180]
node [sloped, near start,above] {$\scriptstyle \xca$} (\mxsh-0.15,\mysh);
\draw (0.15-\mxsh,\mysh+0.35) -- (-0.15+\mxsh,\mysh+0.35);
\draw (0.15-\mxsh,-\mysh-0.35) -- (-0.15+\mxsh,-\mysh-0.35);
\draw [bcrc] (-0.75-0.6-\mxsh,-\mysh) -- (0.75+0.6+\mxsh,-\mysh);
\draw [bcrc] (-0.75-0.6-\mxsh,\mysh) -- (0.75+0.6+\mxsh,\mysh);
%
%\draw [lnovr] (-0.75-\mxsh,0) to [out=0,in=180] (0,-\mysh) to [out=0,in=180] (0.75+\mxsh,0);
%\draw (-0.75-\mxsh,0) to [out=0,in=180] (0,-\mysh) to [out=0,in=180] (0.75+\mxsh,0);
%
\draw [ptzert] (-0.15-\mxsh,-0.6-\mysh) rectangle ++(.3,1.2);
\draw [thkln] (-0.75-0.6-\mxsh,-\mysh) -- (-0.15-\mxsh,-\mysh)
node [very near start,below] {$\scriptstyle \xca+1$};
\begin{scope}[xscale=-1]
\draw [ptzert] (-0.15-\mxsh,-0.6-\mysh) rectangle ++(.3,1.2);
\draw [thkln] (-0.75-0.6-\mxsh,-\mysh) -- (-0.15-\mxsh,-\mysh)
node [very near start,below] {$\scriptstyle \xca+1$};
\end{scope}
\begin{scope}[yscale=-1]
\draw [ptzert] (-0.15-\mxsh,-0.6-\mysh) rectangle ++(.3,1.2);
\draw [thkln] (-0.75-0.6-\mxsh,-\mysh) -- (-0.15-\mxsh,-\mysh)
node [very near start,above] {$\scriptstyle \xca+1$};
\end{scope}
\begin{scope}[yscale=-1,xscale=-1]
\draw [ptzert] (-0.15-\mxsh,-0.6-\mysh) rectangle ++(.3,1.2);
\draw [thkln] (-0.75-0.6-\mxsh,-\mysh) -- (-0.15-\mxsh,-\mysh)
node [near start,above] {$\scriptstyle \xca+1$};
\end{scope}
\end{scope}
\end{tikzpicture}
\\
\label{eq:thtr}
%\vspace{0.3cm}
\\
\nonumber
\begin{tikzpicture}[menvtwo]
\node at (-4,0) {(II):};
\node (l) at (-0.8,0) {};
\node (r) at (0.8,0) {};
\path[commutative diagrams/.cd, every arrow, every label]
(l) edge[commutative diagrams/squiggly] (r);
\begin{scope}[xshift=-2cm]
\draw [ptzer] (-0.15,-0.6) rectangle (.15,.6);
\draw [thkln] (-0.75,0) -- (-0.15,0)
node [near start,below] {$\scriptstyle \xca$}
(0.15,0) -- (.75,0)
node [near end,below] {$\scriptstyle \xca$};
\draw [bcrc] (-0.75,0) -- (-0.15,0)  (0.15,0) -- (0.75,0);
\draw (-0.75,0.4) -- (-0.15,0.4) (0.15,0.4) -- (0.75,0.4);
\end{scope}
\begin{scope}[xshift=2cm]
\draw [ptzer] (-0.15,-0.6) rectangle (.15,.6);
\draw [thkln] (-0.75,0) -- (-0.15,0)
node [near start,below] {$\scriptstyle \xca$}
(0.15,0) -- (.75,0)
node [near end,below] {$\scriptstyle \xca$};
\draw [bcrc] (-0.75,0) -- (-0.15,0)  (0.15,0) -- (0.75,0);
\draw (-0.75,0.8) -- (0.75,0.8);% (0.15,-0.4) -- (0.75,-0.4);
\end{scope}
\end{tikzpicture}
\qquad
%%%%%%%%%%%%%%%%%%%%%%%%%%%%%%%%%%%%
\begin{tikzpicture}[menvtwo]
\node at (-4,0) {(II):};
\node (l) at (-0.8,0) {};
\node (r) at (0.8,0) {};
\path[commutative diagrams/.cd, every arrow, every label]
(l) edge[commutative diagrams/squiggly] (r);
\begin{scope}[xshift=-2cm]
\draw [ptzer] (-0.15,-0.6) rectangle (.15,.6);
\draw [thkln] (-0.75,0) -- (-0.15,0)
node [near start,below] {$\scriptstyle \xca$}
(0.15,0) -- (.75,0)
node [near end,below] {$\scriptstyle \xca$};
\draw [bcrc] (-0.75,0) -- (-0.15,0)  (0.15,0) -- (0.75,0);
\draw (-0.15,0.4) to [out=180,in=-90] (-0.6,0.8) to [out=90,in=180] (0,1.4)
to [out=0,in=90] (0.6,0.8) to [out=-90,in=0] (0.15,0.4);
\end{scope}
\begin{scope}[xshift=2cm]
\draw [ptzer] (-0.15,-0.6) rectangle (.15,.6);
\draw [thkln] (-0.75,0) -- (-0.15,0)
node [near start,below] {$\scriptstyle \xca$}
(0.15,0) -- (.75,0)
node [near end,below] {$\scriptstyle \xca$};
\draw [bcrc] (-0.75,0) -- (-0.15,0)  (0.15,0) -- (0.75,0);
%\draw (-0.15,0.4) to [out=180,in=-90] (-0.6,0.8) to [out=90,in=180] (0,1.4)
%to [out=0,in=90] (0.6,0.8) to [out=-90,in=0] (0.15,0.4);
\end{scope}
\end{tikzpicture}
\;.
\end{gather}
The thick gray lines in these pictures mark the \tBcr s of the diagram $\sBD$.

The transition from the diagram $\xDclNo$ to $\xDclN$ is performed in two stages. At the first stage we apply the first replacement of\rx{eq:thtr} to every crossing of $\xDclNo$. The result is the diagram $\xtDN$, which consists of two parts connected at \tJWp s. The first part is the $\xca$-cabled diagram $\xDclN$ and the second part consists of non-intersecting circles formed by single lines appearing in the final diagrams of replacements (I) of\rx{eq:thtr}. These single line circles go along the \tBcr s. We orient them clockwise and assume that in our pictures the clockwise orientation corresponds to the direction from the left to the right. The circles are attached to $\xDclN$ at the \tJWp s and those junctions have four possible forms:
\begin{equation}
\label{eq:frmprj}
\begin{tikzpicture}[menvtwo]
\draw [ptzer] (-0.15,-0.6) rectangle (.15,.6);
\draw [thkln] (-0.75,0) -- (-0.15,0)
node [near start,below] {$\scriptstyle \xca$}
(0.15,0) -- (.75,0)
node [near end,below] {$\scriptstyle \xca$};
\draw [bcrc] (-0.75,0) -- (-0.15,0)  (0.15,0) -- (0.75,0);
\draw (-0.75,0.4) -- (-0.15,0.4) (0.15,0.4) -- (0.75,0.4);
\end{tikzpicture},
\qquad
\begin{tikzpicture}[menvtwo]
\draw [ptzer] (-0.15,-0.6) rectangle (.15,.6);
\draw [thkln] (-0.75,0) -- (-0.15,0)
node [near start,above] {$\scriptstyle \xca$}
(0.15,0) -- (.75,0)
node [near end,above] {$\scriptstyle \xca$};
\draw [bcrc] (-0.75,0) -- (-0.15,0)  (0.15,0) -- (0.75,0);
\draw (-0.75,-0.4) -- (-0.15,-0.4) (0.15,-0.4) -- (0.75,-0.4);
\end{tikzpicture},
\qquad
\begin{tikzpicture}[menvtwo]
\draw [ptzer] (-0.15,-0.6) rectangle (.15,.6);
\draw [thkln] (-0.75,0) -- (-0.15,0)
node [near start,above] {$\scriptstyle \xca$}
(0.15,0) -- (.75,0)
node [near end,below] {$\scriptstyle \xca$};
\draw [bcrc] (-0.75,0) -- (-0.15,0)  (0.15,0) -- (0.75,0);
\draw (-0.75,-0.4) -- (-0.15,-0.4) (0.15,0.4) -- (0.75,0.4);
\end{tikzpicture},
\qquad
\begin{tikzpicture}[menvtwo]
\draw [thkln] (-0.75,0) -- (-0.15,0)
node [near start,below] {$\scriptstyle \xca$}
(0.15,0) -- (.75,0)
node [near end,above] {$\scriptstyle \xca$};
\draw [bcrc] (-0.75,0) -- (-0.15,0)  (0.15,0) -- (0.75,0);
\draw (-0.75,0.4) -- (-0.15,0.4) (0.15,-0.4) -- (0.75,-0.4);
\draw [ptzer] (-0.15,-0.6) rectangle (.15,.6);
\end{tikzpicture}.
\end{equation}

At the second stage of the transition from $\xDclNo$ to $\xDclN$ we remove the single circle lines
of $\xtDN$ one-by-one. In order to remove a particular circle we select an `initial' \tJWp\ on it and then detach the single lines from other projectors going clockwise. During this process, the single line between the initial and current \tJWp s are kept on the same side of the \tBcr s. If the current projector has the incoming and outgoing single lines on the opposite sides of the \tBcr\ (third and fourth type of\rx{eq:frmprj}) then, prior to detachment, we perform the following transformation
for the junction of the third type (and a similar transformation for the fourth type):
\begin{equation}
\label{eq:twflp}
\begin{tikzpicture}[menvtwo]
\begin{scope}[xshift=1.75cm]
\draw [thkln] (-1.15,0) -- (-0.15,0)
%node [very near start,below] {$\scriptstyle \xca$}
(0.15,0) -- (.75,0)
node [near end,below] {$\scriptstyle \xca$};
\draw [bcrc] (-4.25,0) -- (0.75,0);
%\draw [lnovr]  (-1.05,0.4) to [out=0,in=180] (-0.15,-0.4);
\draw % (-1.05,0.4) to [out=0,in=180] (-0.15,-0.4)
(0.15,0.4) -- (0.75,0.4);
\draw [ptzert] (-0.15,-0.6) rectangle (.15,.6);
\end{scope}
\begin{scope}[xshift=-1.75cm,xscale=-1]
\draw [thkln] (-1.15,0) -- (-0.15,0)
%node [very near start,below] {$\scriptstyle \xca$}
(0.15,0) -- (.75,0)
node [near end,above] {$\scriptstyle \xca$};
%\draw [lnovr]  (-1.05,0.4) to [out=0,in=180] (-0.15,-0.4);
\draw % (-1.05,0.4) to [out=0,in=180] (-0.15,-0.4)
(0.15,-0.4) -- (0.75,-0.4);
\draw [ptzert] (-0.15,-0.6) rectangle (.15,.6);
\end{scope}
\node (0,0) {$\cdots$};
\draw (-1.6,-0.4) -- (1.6,-0.4);
\end{tikzpicture}
\;\hteqv\;
\begin{tikzpicture}[menvtwo]
\begin{scope}[xshift=1.75cm]
\draw [bcrc] (-4.25,0) -- (0.75,0);
\draw [thkln] (-1.15,0) -- (-0.15,0)
%node [very near start,below] {$\scriptstyle \xca$}
(0.15,0) -- (.75,0)
node [near end,below] {$\scriptstyle \xca$};
\draw [lnovr]  (-1.05,0.4) to [out=0,in=180] (-0.15,-0.4);
\draw (-1.05,0.4) to [out=0,in=180] (-0.15,-0.4) (0.15,0.4) -- (0.75,0.4);
\draw [ptzert] (-0.15,-0.6) rectangle (.15,.6);
\end{scope}
\begin{scope}[xshift=-1.75cm,xscale=-1]
\draw [thkln] (-1.15,0) -- (-0.15,0)
%node [very near start,below] {$\scriptstyle \xca$}
(0.15,0) -- (.75,0)
node [near end,above] {$\scriptstyle \xca$};
\draw [lnovr]  (-1.05,0.4) to [out=0,in=180] (-0.15,-0.4);
\draw (-1.05,0.4) to [out=0,in=180] (-0.15,-0.4) (0.15,-0.4) -- (0.75,-0.4);
\draw [ptzert] (-0.15,-0.6) rectangle (.15,.6);
\end{scope}
\node (0,0) {$\cdots$};
\draw (-0.7,0.4) -- (0.7,0.4);
\end{tikzpicture}
\;\hteqv\;
\begin{tikzpicture}[menvtwo]
\begin{scope}[xshift=1.75cm]
\draw [bcrc] (-4.25,0) -- (0.75,0);
\draw [thkln] (-1.15,0) -- (-0.15,0)
%node [very near start,below] {$\scriptstyle \xca$}
(0.15,0) -- (.75,0)
node [near end,below] {$\scriptstyle \xca$};
%\draw [lnovr]  (-1.05,0.4) to [out=0,in=180] (-0.15,-0.4);
\draw % (-1.05,0.4) to [out=0,in=180] (-0.15,-0.4)
(0.15,0.4) -- (0.75,0.4);
\draw [ptzert] (-0.15,-0.6) rectangle (.15,.6);
\end{scope}
\begin{scope}[xshift=-1.75cm,xscale=-1]
\draw [thkln] (-1.15,0) -- (-0.15,0)
%node [very near start,below] {$\scriptstyle \xca$}
(0.15,0) -- (.75,0)
node [near end,above] {$\scriptstyle \xca$};
%\draw [lnovr]  (-1.05,0.4) to [out=0,in=180] (-0.15,-0.4);
\draw % (-1.05,0.4) to [out=0,in=180] (-0.15,-0.4)
(0.15,-0.4) -- (0.75,-0.4);
\draw [ptzert] (-0.15,-0.6) rectangle (.15,.6);
\end{scope}
\node (0,0) {$\cdots$};
\draw (-1.6,0.4) -- (1.6,0.4);
\end{tikzpicture},
\end{equation}
In these pictures the left projector is initial, the right projector is current, the first homotopy equivalence comes from the Reidemeister moves, while the second equivalence comes from \ex{eq:projtw}. Note that the single line between the initial and current projectors is kept always above the rest of the diagram.

After the single lines attached to the current projector are brought to the same side of the \tBcr, we detach the single line from that projector by the \lrpl\ (II) of\rx{eq:thtr}
%\[
%\xy 0;/r.22pc/:
%(-20,0)
%*{
%\begin{tikzpicture}[menvtwo]
%\draw [ptzer] (-0.15,-0.6) rectangle (.15,.6);
%\draw [thkln] (-0.75,0) -- (-0.15,0)
%node [near start,below] {$\scriptstyle \xca$}
%(0.15,0) -- (.75,0)
%node [near end,below] {$\scriptstyle \xca$};
%\draw (-0.75,0.4) -- (-0.15,0.4) (0.15,0.4) -- (0.75,0.4);
%\end{tikzpicture}
%}="1";
%(20,0)
%*{
%\begin{tikzpicture}[menvtwo]
%\draw [ptzer] (-0.15,-0.6) rectangle (.15,.6);
%\draw [thkln] (-0.75,0) -- (-0.15,0)
%node [near start,below] {$\scriptstyle \xca$}
%(0.15,0) -- (.75,0)
%node [near end,below] {$\scriptstyle \xca$};
%\draw (-0.75,0.8) -- (0.75,0.8);% (0.15,-0.4) -- (0.75,-0.4);
%\end{tikzpicture}
%}="2";
%{\ar@{~>}"1"+(12,0);"2"+(-12,0)};
%\endxy
%\]
and pass to the next projector on the single line.

The single line is kept above the rest of the diagram, so once it is detached from all projectors except the initial one, it can be contracted to a small loop attached to that initial projector with the help of Reidemeister moves. The final step is the removal of that loop by the replacement (III) of\rx{eq:thtr}. After all single line circles are removed, the diagram $\xtDN$ becomes $\xDclN$.

Our transition from $\xDclNo$ to $\xDclN$ is generally similar to that used by C.~Armond\cite{Arm11}, but the details are different. In particular, we do not replace $(\xca+1)$-cable crossings by projectors, but rather apply replacements (I) of\rx{eq:thtr} directly to the crossings.

\subsection{\Ltrf s generate isomorphisms at low \thdgr s}
\label{sct:hest}

We describe the \ltrf s related to replacements\rx{eq:thtr} and show that the corresponding maps\rx{eq:dgprmg} between shifted homologies are isomorphisms at low \thdgr s, thus proving Theorem\rw{thm:leviso}.

%\subsection{The \ltrf\rx{eq:frstst}}
\subsubsection{\Ltrf\ \xltone}
Set
\begin{equation}
\label{eq:trone}
\def\mxsh{1.5}
\def\mysh{1.2}
%\label{eq:tngcr}
\ytngsi =
\begin{tikzpicture}[menvone]
\draw [thkln] (-\mxsh+0.15,\mysh) to [out=0,in=180] (\mxsh-0.15,-\mysh);
\draw [lnovr] (-\mxsh+0.15,-\mysh) to [out=0,in=180] (\mxsh-0.15,\mysh);
\draw [thkln] (-\mxsh+0.15,-\mysh) to [out=0,in=180] (\mxsh-0.15,\mysh);
\draw [bcrc] (-0.75-0.6-\mxsh,-\mysh) -- (0.75+0.6+\mxsh,-\mysh);
\draw [bcrc] (-0.75-0.6-\mxsh,\mysh) -- (0.75+0.6+\mxsh,\mysh);
%
%\draw [lnovr] (-0.75-\mxsh,0) to [out=0,in=180] (0,-\mysh) to [out=0,in=180] (0.75+\mxsh,0);
%\draw (-0.75-\mxsh,0) to [out=0,in=180] (0,-\mysh) to [out=0,in=180] (0.75+\mxsh,0);
%
\draw [ptzert] (-0.15-\mxsh,-0.6-\mysh) rectangle ++(.3,1.2);
\draw [thkln] (-0.75-0.6-\mxsh,-\mysh) -- (-0.15-\mxsh,-\mysh)
node [very near start,below] {$\scriptstyle \xca+1$};
\begin{scope}[xscale=-1]
\draw [ptzert] (-0.15-\mxsh,-0.6-\mysh) rectangle ++(.3,1.2);
\draw [thkln] (-0.75-0.6-\mxsh,-\mysh) -- (-0.15-\mxsh,-\mysh)
node [very near start,below] {$\scriptstyle \xca+1$};
\end{scope}
\begin{scope}[yscale=-1]
\draw [ptzert] (-0.15-\mxsh,-0.6-\mysh) rectangle ++(.3,1.2);
\draw [thkln] (-0.75-0.6-\mxsh,-\mysh) -- (-0.15-\mxsh,-\mysh)
node [very near start,above] {$\scriptstyle \xca+1$};
\end{scope}
\begin{scope}[yscale=-1,xscale=-1]
\draw [ptzert] (-0.15-\mxsh,-0.6-\mysh) rectangle ++(.3,1.2);
\draw [thkln] (-0.75-0.6-\mxsh,-\mysh) -- (-0.15-\mxsh,-\mysh)
node [near start,above] {$\scriptstyle \xca+1$};
\end{scope}
\end{tikzpicture}\;,
%%%%%%%%%%
\quad
%%%%%%%%%%
\ytngsf=
\begin{tikzpicture}[menvone]
\draw [thkln] (-\mxsh+0.15,\mysh) to [out=0,in=180]
node [sloped, near end,above] {$\scriptstyle \xca$} (\mxsh-0.15,-\mysh);
\draw [lnovr] (-\mxsh+0.15,-\mysh) to [out=0,in=180] (\mxsh-0.15,\mysh);
\draw [thkln] (-\mxsh+0.15,-\mysh) to [out=0,in=180]
node [sloped, near start,above] {$\scriptstyle \xca$} (\mxsh-0.15,\mysh);
\draw (0.15-\mxsh,\mysh+0.35) -- (-0.15+\mxsh,\mysh+0.35);
\draw (0.15-\mxsh,-\mysh-0.35) -- (-0.15+\mxsh,-\mysh-0.35);
\draw [bcrc] (-0.75-0.6-\mxsh,-\mysh) -- (0.75+0.6+\mxsh,-\mysh);
\draw [bcrc] (-0.75-0.6-\mxsh,\mysh) -- (0.75+0.6+\mxsh,\mysh);
%
%\draw [lnovr] (-0.75-\mxsh,0) to [out=0,in=180] (0,-\mysh) to [out=0,in=180] (0.75+\mxsh,0);
%\draw (-0.75-\mxsh,0) to [out=0,in=180] (0,-\mysh) to [out=0,in=180] (0.75+\mxsh,0);
%
\draw [ptzert] (-0.15-\mxsh,-0.6-\mysh) rectangle ++(.3,1.2);
\draw [thkln] (-0.75-0.6-\mxsh,-\mysh) -- (-0.15-\mxsh,-\mysh)
node [very near start,below] {$\scriptstyle \xca+1$};
\begin{scope}[xscale=-1]
\draw [ptzert] (-0.15-\mxsh,-0.6-\mysh) rectangle ++(.3,1.2);
\draw [thkln] (-0.75-0.6-\mxsh,-\mysh) -- (-0.15-\mxsh,-\mysh)
node [very near start,below] {$\scriptstyle \xca+1$};
\end{scope}
\begin{scope}[yscale=-1]
\draw [ptzert] (-0.15-\mxsh,-0.6-\mysh) rectangle ++(.3,1.2);
\draw [thkln] (-0.75-0.6-\mxsh,-\mysh) -- (-0.15-\mxsh,-\mysh)
node [very near start,above] {$\scriptstyle \xca+1$};
\end{scope}
\begin{scope}[yscale=-1,xscale=-1]
\draw [ptzert] (-0.15-\mxsh,-0.6-\mysh) rectangle ++(.3,1.2);
\draw [thkln] (-0.75-0.6-\mxsh,-\mysh) -- (-0.15-\mxsh,-\mysh)
node [near start,above] {$\scriptstyle \xca+1$};
\end{scope}
\end{tikzpicture}\;,
\quad
%%%%%%%%%%
\ytngsc= \shcr^{\xca+\hlf}
\begin{tikzpicture}[menvone]
\draw [thkln] (-\mxsh+0.15,\mysh) to [out=0,in=180]
node [sloped, very near end,below] {$\scriptstyle \xca$} (\mxsh-0.15,-\mysh);
\draw [lnovr] (-\mxsh+0.15,-\mysh) to [out=0,in=180] (\mxsh-0.15,\mysh);
\draw [thkln] (-\mxsh+0.15,-\mysh) to [out=0,in=180]
node [sloped, very near start,below] {$\scriptstyle \xca$} (\mxsh-0.15,\mysh);
%\draw (0.15-\mxsh,\mysh+0.35) -- (-0.15+\mxsh,\mysh+0.35);
%\draw (0.15-\mxsh,-\mysh-0.35) -- (-0.15+\mxsh,-\mysh-0.35);
\draw (0.15-\mxsh,\mysh-0.35) to [out=0,in=90] (-\mxsh*0.5,0) to [out=-90,in=0] (0.15-\mxsh,-\mysh+0.35);
\draw [xscale=-1] (0.15-\mxsh,\mysh-0.35) to [out=0,in=90] (-\mxsh*0.5,0) to [out=-90,in=0] (0.15-\mxsh,-\mysh+0.35);
\draw [bcrc] (-0.75-0.6-\mxsh,-\mysh) -- (0.75+0.6+\mxsh,-\mysh);
\draw [bcrc] (-0.75-0.6-\mxsh,\mysh) -- (0.75+0.6+\mxsh,\mysh);
%
%\draw [lnovr] (-0.75-\mxsh,0) to [out=0,in=180] (0,-\mysh) to [out=0,in=180] (0.75+\mxsh,0);
%\draw (-0.75-\mxsh,0) to [out=0,in=180] (0,-\mysh) to [out=0,in=180] (0.75+\mxsh,0);
%
\draw [ptzert] (-0.15-\mxsh,-0.6-\mysh) rectangle ++(.3,1.2);
\draw [thkln] (-0.75-0.6-\mxsh,-\mysh) -- (-0.15-\mxsh,-\mysh)
node [very near start,below] {$\scriptstyle \xca+1$};
\begin{scope}[xscale=-1]
\draw [ptzert] (-0.15-\mxsh,-0.6-\mysh) rectangle ++(.3,1.2);
\draw [thkln] (-0.75-0.6-\mxsh,-\mysh) -- (-0.15-\mxsh,-\mysh)
node [very near start,below] {$\scriptstyle \xca+1$};
\end{scope}
\begin{scope}[yscale=-1]
\draw [ptzert] (-0.15-\mxsh,-0.6-\mysh) rectangle ++(.3,1.2);
\draw [thkln] (-0.75-0.6-\mxsh,-\mysh) -- (-0.15-\mxsh,-\mysh)
node [very near start,above] {$\scriptstyle \xca+1$};
\end{scope}
\begin{scope}[yscale=-1,xscale=-1]
\draw [ptzert] (-0.15-\mxsh,-0.6-\mysh) rectangle ++(.3,1.2);
\draw [thkln] (-0.75-0.6-\mxsh,-\mysh) -- (-0.15-\mxsh,-\mysh)
node [near start,above] {$\scriptstyle \xca+1$};
\end{scope}
\end{tikzpicture}\;,
\end{equation}
while $\ytngkfp=\shcr^{-\xca+\hlf}\ytngkf$. Theorem\rw{lm:sss} provides the exact triangle relation\rx{eq:cnrel}.

\begin{proposition}
\label{prp:bndfs}
Let $\xDsi$ be the diagram constructed by performing \lrpl s I of\rx{eq:thtr} on some vertices of $\xDclNo$ and let $\xDsf$ be the diagram constructed by performing the \lrpl\ I on the `current' vertex in $\xDsi$. Then the \tdgpr\ map\rx{eq:dgprmg} is an isomorphism on $\tKHmvv{i}{\hem}$ for $i\leq 2\xca-1$.
\end{proposition}
\begin{proof}
Let $\xDsc$ be the diagram constructed by performing the \lrpl\ $\ytngsi \rightsquigarrow\ytngsc$ on the current vertex. We estimate the homological order of $\KHm(\xDsc)$ with the help of Theorem\rw{thm:smfr}: since $\yncrv{\xDsc} = \yncrv{\xDsi}-2\xca-1$, then $\KHmvv{i}{\hem}(\xDsc)=0$ for $i\leq -\shlf\yncrv{\xDsi} + 2\xca$ (we took into account the shift $\shcr^{\xca+\hlf}$ of $\ytngsc$ in \ex{eq:trone}) and the claim of the theorem follows from Proposition\rw{prp:gestdg}.
\end{proof}
%
%Let $\xDsi$ be the diagram constructed by performing  \Arpl s\rx{eq:abrpl} on some vertices of $\xDclNo$, $\xDsf$ be the diagram constructed by performing the \lrpl\ I of\rx{eq:thtr} on the `current' vertex in $\xDsi$ and $\xDsc$ be the diagram constructed by performing the \lrpl\ $\ytngsi \rightsquigarrow\ytngsc$ on that vertex. Then the \tdgpr\ map\rx{eq:dgprmg} is an isomorphism on $\tKHmvv{i}{\hem}$ for $i\leq 2\xca-1$.
%
%In view of the relation
%\[\yncrv{\xDsc}=\yncrv{\xDsf} = \yncrv{\xDsi}-2\xca-1,\]
%this proposition is a special case of Proposition\rw{prp:gestdg} and inequality\rx{eq:ineqt}.%
%%
%%Applying the Proposition\rw{thm:eaest} to the diagram $\xDsc$ and observing that %\[\yncrv{\xDsc}=\yncrv{\xDsf} = \yncrv{\xDsi}-2\xca-1,\]
%% we find that the claim of the proposition is an easy corollary of Proposition\rw{prp:gestdg}.
%\end{proof}

\subsubsection{\Ltrf\ \xlttwo}
Set
\begin{equation}
\label{eq:lrdpr}
\ytngsi =
\begin{tikzpicture}[menvtwo]
\draw [bcrc] (-0.75,0) -- (0.75,0);
\draw [ptzert] (-0.15,-0.6) rectangle (.15,.6);
\draw [thkln] (-0.75,0) -- (-0.15,0)
node [near start,below] {$\scriptstyle \xca$}
(0.15,0) -- (.75,0)
node [near end,below] {$\scriptstyle \xca$};
\draw (-0.75,0.4) -- (-0.15,0.4) (0.15,0.4) -- (0.75,0.4);
\end{tikzpicture},
\qquad
\ytngsf =
\begin{tikzpicture}[menvtwo]
\draw [bcrc] (-0.75,0) -- (0.75,0);
\draw [ptzert] (-0.15,-0.6) rectangle (.15,.6);
\draw [thkln] (-0.75,0) -- (-0.15,0)
node [near start,below] {$\scriptstyle \xca$}
(0.15,0) -- (.75,0)
node [near end,below] {$\scriptstyle \xca$};
\draw (-0.75,0.8) -- (0.75,0.8);% (0.15,-0.4) -- (0.75,-0.4);
\end{tikzpicture},
\qquad
\ytngsc =
%\shcr
\shcr
\begin{tikzpicture}[menvone]
\draw [bcrc] (-1.35,0) -- (1.35,0);
\draw[pttwo] (-.15,-0.6) rectangle ++(0.3,1.2);
\draw[thkln] (-1.35,0) -- (-.15,0)
node[near start,below] {$\scriptstyle \xca+1$ }
(.15,0) -- (1.35,0)
node [near end, below] {$\scriptstyle \xca+1$};
\end{tikzpicture}\;,
%\qquad
%\ytngsfp = \ytngsf,\quad
%\ytngscp = \shcr\ytngsc.
\end{equation}
while $\ytngkfp=\ytngkf$. The exact triangle relation\rx{eq:cnrel} is provided by Theorem\rw{thm:projpr}.
%%%%%%%
\begin{proposition}
\label{prp:bndsc}
Let $\xDsi$ be a diagram constructed by removing some single line circles from $\xtDN$ and by detaching the `current' single line circle from the projectors which lie between the initial one and the current one and let $\xDsf$ be the diagram constructed from $\xDsi$ by detaching the single line from the current projector.
 %and let $\xDsc$ be the diagram constructed from $\xDsi$ by replacing the current projector ($\ytngsi$ of \ex{eq:lrdpr}) with the tangle complex $\ytngsc$ of \ex{eq:lrdpr}.
Then the \tdgpr\ map\rx{eq:dgprmg} is an isomorphism on $\tKHmvv{i}{\hem}$ for $i\leq \xca-1$.
\end{proposition}
The proof uses the following
\begin{lemma}
\label{lm:cmbnd}
The tangle
\def\mxsh{2}
\def\mysh{1.5}
\begin{equation}
\label{eq:tngcr}
\ztau \; = \;
\begin{tikzpicture}[menvone]
\draw [thkln] (-\mxsh+0.15,\mysh) to [out=0,in=180] (\mxsh-0.15,-\mysh);
\draw [lnovr] (-\mxsh+0.15,-\mysh) to [out=0,in=180] (\mxsh-0.15,\mysh);
\draw [thkln] (-\mxsh+0.15,-\mysh) to [out=0,in=180] (\mxsh-0.15,\mysh);
\draw [lnovr] (-0.75-\mxsh,0) to [out=0,in=180] (0,-\mysh) to [out=0,in=180] (0.75+\mxsh,0);
\draw (-0.75-\mxsh,0) to [out=0,in=180] (0,-\mysh) to [out=0,in=180] (0.75+\mxsh,0);
\draw [ptzer] (-0.15-\mxsh,-0.6-\mysh) rectangle ++(.3,1.2);
\draw [thkln] (-0.75-\mxsh,-\mysh) -- (-0.15-\mxsh,-\mysh)
node [near start,below] {$\scriptstyle \xca$};
\begin{scope}[xscale=-1]
\draw [ptzer] (-0.15-\mxsh,-0.6-\mysh) rectangle ++(.3,1.2);
\draw [thkln] (-0.75-\mxsh,-\mysh) -- (-0.15-\mxsh,-\mysh)
node [near start,below] {$\scriptstyle \xca$};
\end{scope}
\begin{scope}[yscale=-1]
\draw [ptzer] (-0.15-\mxsh,-0.6-\mysh) rectangle ++(.3,1.2);
\draw [thkln] (-0.75-\mxsh,-\mysh) -- (-0.15-\mxsh,-\mysh)
node [near start,above] {$\scriptstyle \xca$};
\end{scope}
\begin{scope}[yscale=-1,xscale=-1]
\draw [ptzer] (-0.15-\mxsh,-0.6-\mysh) rectangle ++(.3,1.2);
\draw [thkln] (-0.75-\mxsh,-\mysh) -- (-0.15-\mxsh,-\mysh)
node [near start,above] {$\scriptstyle \xca$};
\end{scope}
%\draw [thkln] (-\mxsh+0.15,\mysh) to [out=0,in=180] (\mxsh-0.15,-\mysh);
%\draw [lnovr] (-\mxsh+0.15,-\mysh) to [out=0,in=180] (\mxsh-0.15,\mysh);
%\draw [thkln] (-\mxsh+0.15,-\mysh) to [out=0,in=180] (\mxsh-0.15,\mysh);
\end{tikzpicture}
\end{equation}
has a homological bound
$\hmord{\ztau} \geq -\shlf\xca^2.$
\end{lemma}
\begin{remark}
\label{rmk:ignr}
This bound is better than the crude bound of Theorem\rw{thm:smfr}. In fact, it coincides with that bound, if we neglect the intersections between the single line and the $\xca$-cables.
\end{remark}
\begin{proof}[Proof of Lemma\rw{lm:cmbnd}]
Applying \ex{eq:colKhbr} to the $\xca$-cable crossing in $\ztau$ we get the presentation
\[
\def\mxsh{2}
\def\mysh{1.5}
\xKhv{\ztau} \hteqv
\shcr^{-\hlf\xca^2}\,
\Pcnv{
\bigoplus_{i=0}^{\xca}
\shcr^{i^2}{\xca \brace i}_{\shcr}\,
\ztau_i}
\qquad
\xKhv{\ztau_i}\;=\;
\begin{tikzpicture}[menvone]
\draw [thkln] (-\mxsh+0.15,\mysh-0.2) to [out=0,in=90] (-\mxsh+1.1,0)
% node [near end,right] {$\scriptstyle i$}
 to [out=-90,in=0] (-\mxsh+0.15,-\mysh+0.2) ;
\begin{scope}[xscale=-1]
\draw [thkln] (-\mxsh+0.15,\mysh-0.2) to [out=0,in=90] (-\mxsh+1.1,0) to [out=-90,in=0] (-\mxsh+0.15,-\mysh+0.2);
\end{scope}
\node at (-\mxsh+1.5,0.4) {$\scriptstyle i$};
\node at (\mxsh-1.5,0.4) {$\scriptstyle i$};
%\draw [lnovr] (-\mxsh+0.15,-\mysh) to [out=0,in=180] (\mxsh-0.15,\mysh);
%\draw [thkln] (-\mxsh+0.15,-\mysh) to [out=0,in=180] (\mxsh-0.15,\mysh);
\draw [thkln] (-\mxsh+0.15,\mysh+0.2) -- (\mxsh-0.15,\mysh+0.2) node [midway,above] {$\scriptstyle \xca-i$};
\draw [thkln] (-\mxsh+0.15,-\mysh-0.2) -- (\mxsh-0.15,-\mysh-0.2) node [midway, below] {$\scriptstyle \xca-i$};
\draw [lnovr] (-0.75-\mxsh,0) to [out=0,in=180] (0,-\mysh+0.7) to [out=0,in=180] (0.75+\mxsh,0);
\draw (-0.75-\mxsh,0) to [out=0,in=180] (0,-\mysh+0.7) to [out=0,in=180] (0.75+\mxsh,0);
\draw [ptzer] (-0.15-\mxsh,-0.6-\mysh) rectangle ++(.3,1.2);
\draw [thkln] (-0.75-\mxsh,-\mysh) -- (-0.15-\mxsh,-\mysh)
node [near start,below] {$\scriptstyle \xca$};
\begin{scope}[xscale=-1]
\draw [ptzer] (-0.15-\mxsh,-0.6-\mysh) rectangle ++(.3,1.2);
\draw [thkln] (-0.75-\mxsh,-\mysh) -- (-0.15-\mxsh,-\mysh)
node [near start,below] {$\scriptstyle \xca$};
\end{scope}
\begin{scope}[yscale=-1]
\draw [ptzer] (-0.15-\mxsh,-0.6-\mysh) rectangle ++(.3,1.2);
\draw [thkln] (-0.75-\mxsh,-\mysh) -- (-0.15-\mxsh,-\mysh)
node [near start,above] {$\scriptstyle \xca$};
\end{scope}
\begin{scope}[yscale=-1,xscale=-1]
\draw [ptzer] (-0.15-\mxsh,-0.6-\mysh) rectangle ++(.3,1.2);
\draw [thkln] (-0.75-\mxsh,-\mysh) -- (-0.15-\mxsh,-\mysh)
node [near start,above] {$\scriptstyle \xca$};
\end{scope}
\end{tikzpicture}
\]
The homological order of $\ztau_i$ can be estimated with the help of Theorem\rw{thm:smfr}: $\hmord{\xKhv{\ztau_i}} \geq - i$. Since the polynomial ${\xca \brace i}_{\shcr}$ has only non-negative powers of $\shcr$ and $i^2 - i \geq 0$ for all integer $i$, we come to the estimate of Lemma\rw{lm:cmbnd}.
\end{proof}

\begin{proof}[Proof of Proposition\rw{prp:bndsc}]
Let $\xDsc$ be the diagram constructed from $\xDsi$ by replacing the current projector ($\ytngsi$ of \ex{eq:lrdpr}) with the tangle complex $\ytngsc$ of \ex{eq:lrdpr}.
By Proposition\rw{prp:gestdg}, we have to prove the bound:
\[
\KHmvv{i}{\hem}(\xDsc)=0 \quad\text{for $i\leq-\shlf\yncrv{\xDsi}+\xca.$}
\]
estimate
\[\hmord{\xDsc}\geq-\hlf\yncrv{\xDsi} + \xca.\]
Since the complex $\ytngsc$ of \ex{eq:lrdpr} is a \tmcn\rx{eq:swgrpr} generated by an `elementary' tangle
\begin{equation}
\label{eq:eltng}
\ytngse \;=\;
\begin{tikzpicture}[menvone]
\draw [line width=6pt, color=white] (-1+0.15,0) -- (0,0);
\draw [line width=\cblth] (-1.75,0) -- (-1 - .15,0)
node [near start, below] {$\scriptstyle \xca$}
(-1+0.15,0) -- (0,0) -- (1-0.15,0) (1+.15,0) -- (1.75,0)
node [near end, below] {$\scriptstyle \xca$};
\draw (1, 0.9) to [out=180,in=90] (0.4,0.6) to [out=-90,in=180] (1-.15,0.3);
\draw [xscale=-1] (1, 0.9) to [out=180,in=90] (0.4,0.6) to [out=-90,in=180] (1-.15,0.3);
\draw (-1.75,0.9) -- (-1,0.9) (1,0.9) -- (1.75,0.9);
\draw [line width=\ljwp] (-1.15,-0.6) rectangle ++(0.3,1.2);
\draw [line width=\ljwp] (1 - 0.15,-0.6) rectangle ++(0.3,1.2);
\end{tikzpicture}
\end{equation}
then, according to Remark\rw{rmk:bndss},
it is sufficient to prove
\begin{equation}
\label{eq:scest}
%\hmord{\xDse}\geq -\hlf\yncrv{\xDsi} + \xca,
\KHmvv{i}{\hem}(\xDse) = 0\quad\text{for $i\leq -\shlf\yncrv{\xDsi} + \xca-1$,}
\end{equation}
where $\xDse$ is the diagram constructed by replacing $\ytngsi$ in $\xDsi$ with $\ytngse$.

Consider a tangle within $\xDse$ which consists of the right half of $\ytngse$ and the cable crossing which follows the current projector and transform its complex with the help of two homotopy equivalences:
\begin{equation}
\label{eq:prsld}
\def\mxsh{2}
\def\mysh{1.5}
\begin{tikzpicture}[menvone]
\draw [thkln] (-\mxsh+0.15,\mysh) to [out=0,in=180] (\mxsh-0.15,-\mysh);
\draw [lnovr] (-\mxsh+0.15,-\mysh) to [out=0,in=180] (\mxsh-0.15,\mysh);
\draw [thkln] (-\mxsh+0.15,-\mysh) to [out=0,in=180] (\mxsh-0.15,\mysh);
\draw (\mxsh-0.15,\mysh+0.3) to [out=180,in=0]
(-\mxsh, \mysh+0.9) to [out=180,in=90] (-\mxsh-0.6,\mysh+0.6) to [out=-90,in=180] (-\mxsh-.15,\mysh+0.3);
\draw (\mxsh+0.15,\mysh+0.3) -- (\mxsh+0.75,\mysh+0.3);
\draw [dashed] (-\mxsh-0.75,-\mysh-0.3) -- (-\mxsh-0.15,-\mysh-0.3)
(-\mxsh+0.15,-\mysh-0.3) -- (\mxsh-0.15,-\mysh-0.3) (\mxsh+0.15,-\mysh-0.3) -- (\mxsh+0.75,-\mysh-0.3);
%\draw [bcrc] (-\mxsh-0.75,\mysh) -- (\mxsh+0.75,\mysh);
%\draw [bcrc] (-\mxsh-0.75,-\mysh) -- (\mxsh+0.75,-\mysh);
\draw [ptzert] (-0.15-\mxsh,-0.6-\mysh) rectangle ++(.3,1.2);
\draw [thkln] (-0.75-\mxsh,-\mysh) -- (-0.15-\mxsh,-\mysh)
node [near start,above] {$\scriptstyle \xca$};
\begin{scope}[xscale=-1]
\draw [ptzert] (-0.15-\mxsh,-0.6-\mysh) rectangle ++(.3,1.2);
\draw [thkln] (-0.75-\mxsh,-\mysh) -- (-0.15-\mxsh,-\mysh)
node [near start,above] {$\scriptstyle \xca$};
\end{scope}
\begin{scope}[yscale=-1]
\draw [ptzert] (-0.15-\mxsh,-0.6-\mysh) rectangle ++(.3,1.2);
\draw [thkln] (-0.75-\mxsh,-\mysh) -- (-0.15-\mxsh,-\mysh)
node [at start,below] {$\scriptstyle \xca-1\;\;\;$};
\end{scope}
\begin{scope}[yscale=-1,xscale=-1]
\draw [ptzert] (-0.15-\mxsh,-0.6-\mysh) rectangle ++(.3,1.2);
\draw [thkln] (-0.75-\mxsh,-\mysh) -- (-0.15-\mxsh,-\mysh)
node [near start,below] {$\scriptstyle \xca$};
\end{scope}
\end{tikzpicture}
\;\hteqv\;
\begin{tikzpicture}[menvone]
\draw (\mxsh+0.15,\mysh+0.3) -- (\mxsh+0.75,\mysh+0.3);
\draw [dashed] (-\mxsh-0.75,-\mysh-0.3) -- (-\mxsh-0.15,-\mysh-0.3)
(-\mxsh+0.15,-\mysh-0.3) -- (\mxsh-0.15,-\mysh-0.3) (\mxsh+0.15,-\mysh-0.3) -- (\mxsh+0.75,-\mysh-0.3);
%
%\draw [bcrc] (-\mxsh-0.75,\mysh) -- (\mxsh+0.75,\mysh);
%\draw [bcrc] (-\mxsh-0.75,-\mysh) -- (\mxsh+0.75,-\mysh);
\draw (\mxsh-0.15,\mysh+0.3) to [out=180,in=90]
(\mxsh-1.2,0) to [out=-90,in=180] (\mxsh-0.15,-\mysh+0.3);
\draw [thkln] (-\mxsh+0.15,\mysh) to [out=0,in=180] (\mxsh-0.15,-\mysh);
\draw [lnovr] (-\mxsh+0.15,-\mysh) to [out=0,in=180] (\mxsh-0.15,\mysh);
\draw [thkln] (-\mxsh+0.15,-\mysh) to [out=0,in=180] (\mxsh-0.15,\mysh);
%\draw [lnovr] (\mxsh-0.15,\mysh+0.3) to [out=180,in=90] (\mxsh-1.2,0);
%(-\mxsh, \mysh+0.9) to [out=180,in=90] (-\mxsh-0.6,\mysh+0.6) to [out=-90,in=180] %(-\mxsh-.15,\mysh+0.3);
%%%
\draw [ptzert] (-0.15-\mxsh,-0.6-\mysh) rectangle ++(.3,1.2);
\draw [thkln] (-0.75-\mxsh,-\mysh) -- (-0.15-\mxsh,-\mysh)
node [near start,above] {$\scriptstyle \xca$};
\begin{scope}[xscale=-1]
\draw [ptzert] (-0.15-\mxsh,-0.6-\mysh) rectangle ++(.3,1.2);
\draw [thkln] (-0.75-\mxsh,-\mysh) -- (-0.15-\mxsh,-\mysh)
node [near start,above] {$\scriptstyle \xca$};
\end{scope}
\begin{scope}[yscale=-1]
\draw [ptzert] (-0.15-\mxsh,-0.6-\mysh) rectangle ++(.3,1.2);
\draw [thkln] (-0.75-\mxsh,-\mysh) -- (-0.15-\mxsh,-\mysh)
node [at start,below] {$\scriptstyle \xca-1\;\;\;$};
\end{scope}
\begin{scope}[yscale=-1,xscale=-1]
\draw [ptzert] (-0.15-\mxsh,-0.6-\mysh) rectangle ++(.3,1.2);
\draw [thkln] (-0.75-\mxsh,-\mysh) -- (-0.15-\mxsh,-\mysh)
node [near start,below] {$\scriptstyle \xca$};
\end{scope}
\end{tikzpicture}
\;\hteqv\;
\shcr^{\hlf\xca}\,
\begin{tikzpicture}[menvone]
\draw (\mxsh+0.15,\mysh+0.3) -- (\mxsh+0.75,\mysh+0.3);
\draw [dashed] (-\mxsh-0.75,-\mysh-0.3) -- (-\mxsh-0.15,-\mysh-0.3)
(-\mxsh+0.15,-\mysh-0.3) -- (\mxsh-0.15,-\mysh-0.3) (\mxsh+0.15,-\mysh-0.3) -- (\mxsh+0.75,-\mysh-0.3);
\draw (\mxsh-0.15,\mysh-0.3) to [out=180,in=90]
(\mxsh-1.2,0) to [out=-90,in=180] (\mxsh-0.15,-\mysh+0.3);
\draw [thkln] (-\mxsh+0.15,\mysh) to [out=0,in=180] (\mxsh-0.15,-\mysh);
\draw [lnovr] (-\mxsh+0.15,-\mysh) to [out=0,in=180] (\mxsh-0.15,\mysh);
\draw [thkln] (-\mxsh+0.15,-\mysh) to [out=0,in=180] (\mxsh-0.15,\mysh);
%\draw [lnovr] (\mxsh-0.15,\mysh+0.3) to [out=180,in=90] (\mxsh-1.2,0);
%(-\mxsh, \mysh+0.9) to [out=180,in=90] (-\mxsh-0.6,\mysh+0.6) to [out=-90,in=180] %(-\mxsh-.15,\mysh+0.3);
%%%
\draw [ptzer] (-0.15-\mxsh,-0.6-\mysh) rectangle ++(.3,1.2);
\draw [thkln] (-0.75-\mxsh,-\mysh) -- (-0.15-\mxsh,-\mysh)
node [near start,above] {$\scriptstyle \xca$};
\begin{scope}[xscale=-1]
\draw [ptzer] (-0.15-\mxsh,-0.6-\mysh) rectangle ++(.3,1.2);
\draw [thkln] (-0.75-\mxsh,-\mysh) -- (-0.15-\mxsh,-\mysh)
node [near start,above] {$\scriptstyle \xca$};
\end{scope}
\begin{scope}[yscale=-1]
\draw [ptzer] (-0.15-\mxsh,-0.6-\mysh) rectangle ++(.3,1.2);
\draw [thkln] (-0.75-\mxsh,-\mysh) -- (-0.15-\mxsh,-\mysh)
node [at start,below] {$\scriptstyle \xca-1\;\;\;$};
\end{scope}
\begin{scope}[yscale=-1,xscale=-1]
\draw [ptzer] (-0.15-\mxsh,-0.6-\mysh) rectangle ++(.3,1.2);
\draw [thkln] (-0.75-\mxsh,-\mysh) -- (-0.15-\mxsh,-\mysh)
node [near start,below] {$\scriptstyle \xca$};
\end{scope}
\end{tikzpicture}
\end{equation}
The first equivalence comes from sliding the upper left projector down right along its $\xca$-cable, and the second equivalence comes from \ex{eq:projtw}. The dashed line indicates the possible presence of another single line which has not been removed yet, however, it plays no role in these calculations.

Let $\xDsep$ denote the diagram $\xDse$ in which the left tangle of \ex{eq:prsld} has been replaced by the right tangle, then
\begin{equation}
\label{eq:scsh}
\xDse \hteqv \shcr^{\hlf\xca}\xDsep.
\end{equation}
We would like to estimate $\hmord{\xDsep}$ with the help of Theorem\rw{thm:smfr}. In doing so we would have to take into account possible crossings coming from the stretch of the single line between the initial projector and the left projector of the tangle $\ytngse$ of \ex{eq:eltng} and $\xca$-cables participating in the crossings attached to the current single line circle. These new crossings are generated by the Reidemeister moves involved in the first homotopy equivalence of \ex{eq:twflp}: when a single line is flipped to the other side of the circle, it may come across the $\xca$-cable crossings, from which parts of this line originate through replacements I of\rx{eq:thtr} (see the picture\rx{eq:tngcr} of the tangle $\ztau$). However, Remark\rw{rmk:ignr} indicates that these crossings between the single line and the $\xca$-cables may be ignored when applying the estimate of Theorem\rw{thm:smfr}, so $\hmord{\xDsep} \geq -\hlf\yncrpv{\xDsep}$, where $\yncrpv{\xDsep}$ is the number of single line intersections within $\xDsep$, except those which we can ignore.

The cable intersection of the left tangle of \ex{eq:prsld} involves two $\xca$-cables, while the same intersection in the right tangle involves a $\xca$-cable and a $(\xca-1)$-cable, hence $\yncrpv{\xDsep} = \yncrv{\xDsi} - \xca$ and the inequality\rx{eq:scest} follows from \ex{eq:scsh}.
%$\hmord{\xDse}\geq-\hlf\yncrv{\xDsi}+\xca$.
\end{proof}

\subsubsection{\Ltrf\ \xltthree}
Set
\begin{equation}
\label{eq:lrpth}
\ytngsi=
\begin{tikzpicture}[menvtwo]
\draw [bcrc] (-0.75,0) -- (0.75,0);
\draw [ptzert] (-0.15,-0.6) rectangle (.15,.6);
\draw [thkln] (-0.75,0) -- (-0.15,0)
node [near start,below] {$\scriptstyle \xca$}
(0.15,0) -- (.75,0)
node [near end,below] {$\scriptstyle \xca$};
\draw (-0.15,0.4) to [out=180,in=-90] (-0.6,0.8) to [out=90,in=180] (0,1.4)
to [out=0,in=90] (0.6,0.8) to [out=-90,in=0] (0.15,0.4);
\end{tikzpicture},
\qquad
\ytngsf =
%\shfr^{-1}
\begin{tikzpicture}[menvtwo]
\draw [bcrc] (-0.75,0) -- (0.75,0);
\draw [ptzer] (-0.15,-0.6) rectangle (.15,.6);
\draw [thkln] (-0.75,0) -- (-0.15,0)
node [near start,below] {$\scriptstyle \xca$}
(0.15,0) -- (.75,0)
node [near end,below] {$\scriptstyle \xca$};
%\draw (-0.15,0.4) to [out=180,in=-90] (-0.6,0.8) to [out=90,in=180] (0,1.4)
%to [out=0,in=90] (0.6,0.8) to [out=-90,in=0] (0.15,0.4);
\end{tikzpicture},
\qquad
\ytngsc = \shcr^{2\xca}\shfr
%\shcr^{2\xca}\shfr
\begin{tikzpicture}[menvtwo]
\draw [bcrc] (-0.75,0) -- (-0.15,0) (0.15,0) -- (0.75,0);
\draw [color=white] (-0.15,-1) rectangle (0.15,1.6);
\draw [ptthr] (-0.15,-0.6) rectangle (.15,.6);
\draw [thkln] (-0.75,0) -- (-0.15,0)
node [near start,below] {$\scriptstyle \xca$}
(0.15,0) -- (.75,0)
node [near end,below] {$\scriptstyle \xca$};
\draw (-0.15,0.4) to [out=180,in=-90] (-0.6,0.8) to [out=90,in=180] (0,1.4)
to [out=0,in=90] (0.6,0.8) to [out=-90,in=0] (0.15,0.4);
%\draw (-0.75,0.4) -- (-0.15,0.4) (0.15,0.4) -- (0.75,0.4);
\end{tikzpicture},
%\qquad
%\ytngkfp = \shfr^{-1}\ytngkf,
%\quad
%\ytngscp = \shcr^{2\xca}\shfr\ytngsc.
\end{equation}
while $\ytngkfp = \shfr^{-1}\ytngkf$
and the cone relation\rx{eq:cnrel} is \ex{eq:htloop}.
\begin{proposition}
\label{prp:sckin}
Let $\xDsi$ be a diagram constructed by removing some single line circles from $\xtDN$ and by detaching the `current' single line circle from the all of its projectors, except the initial one, to which it is attached as in the picture\rx{eq:lrpth} of tangle $\ytngsi$. Let $\xDsf$ be the diagram $\xDsi$ from which this circle is completely removed.
Then the \tdgpr\ map\rx{eq:dgprmg} is an isomorphism on $\tKHmvv{i}{\hem}$ for $i\leq 2\xca-2$.
\end{proposition}
\begin{proof}
%By Proposition\rw{prp:gestdg} coupled with \ex{eq:ineqt}, we have to prove the estimate
%\[\hmord{\xDsc}\geq-\hlf\yncrv{\xDsi} + \xca.\]
%
Since the complex $\ytngsc$ of \ex{eq:lrpth} is a \tmcn\rx{eq:grlp} generated by the elementary tangle
$\ytngsf$ of \ex{eq:lrpth}, then, according to Remark\rw{rmk:bndss}, the claim of this proposition
would follow from the bound
\[
\KHmvv{i}{\hem}(\xDsf)=0\quad\text{for $i\leq-\shlf\yncrv{\xDsi}-1 $.}
\]
The latter follows from Theorem\rw{thm:smfr} coupled with an obvious relation
$\yncrv{\xDsi}=\yncrv{\xDsf}$.
\end{proof}

\subsection{Proof of Theorem\rw{thm:leviso}}
Of all three types of \ltrf s considered in Propositions\rw{prp:bndfs},\rw{prp:bndsc} and\rw{prp:sckin}, it is the \ltrf\rx{eq:lrdpr} which yields the weakest estimate of the homological degrees at which the map\rx{eq:spmaps} is an isomorphism, and this is the estimate of Theorem\rw{thm:leviso}\qed

%\section{Estimate of \tqdgr\ bounds of Khovanov and tail homology}
\section{Proof of Theorem\rw{thm:kqbnd}}

%\begin{proof}[Proof of the bounds\rx{eq:bd2a} and\rx{eq:bd3a}]
In order to compute the shifted homology $\tKHm(\xDclN)$, we apply the colored Khovanov bracket\rx{eq:colKhbr} to all crossings of $\xDclN$. As a result, this diagram turns into a \tmcn\ of flat diagrams of a special kind. Let $\svrt$ be the set of crossings of $\xD$.
%A sequence of numbers $\sipvr\in\{0,1,\ldots,\xca\}$, $\xvrt\in\svrt$
A \emph{\tstt} of $\xDclN$ is a map $\spmp\colon \svrt\rightarrow \{0,1,\ldots,\xca\}$, $\xvrt\mapsto\sipvr$. It
determines a diagram $\xDs$ constructed by performing the following \ltrf s at each crossing $\xvrt\in\svrt$:
\begin{equation*}
\def\mxsh{2}
\def\mysh{1.5}
\begin{tikzpicture}[menvone]
%\node (l) at (-1.8,0) {};
%\node (r) at (1.8,0) {};
%\path[commutative diagrams/.cd, every arrow, every label]
%(l) edge[commutative diagrams/squiggly] (r);
\node (l) at (-1.6,0) {};
\node (r) at (1.6,0) {};
\path[commutative diagrams/.cd, every arrow, every label]
(l) edge[commutative diagrams/squiggly] (r);
\begin{scope}[xshift=-5cm]
\draw [thkln] (-\mxsh+0.15,\mysh) to [out=0,in=180] (\mxsh-0.15,-\mysh);
\draw [lnovr] (-\mxsh+0.15,-\mysh) to [out=0,in=180] (\mxsh-0.15,\mysh);
\draw [thkln] (-\mxsh+0.15,-\mysh) to [out=0,in=180] (\mxsh-0.15,\mysh);
%
%\draw [lnovr] (-0.75-\mxsh,0) to [out=0,in=180] (0,-\mysh) to [out=0,in=180] (0.75+\mxsh,0);
%\draw (-0.75-\mxsh,0) to [out=0,in=180] (0,-\mysh) to [out=0,in=180] (0.75+\mxsh,0);
%
\draw [bcrc] (-\mxsh-0.75,\mysh) -- (\mxsh+0.75,\mysh);
\draw [bcrc] (-\mxsh-0.75,-\mysh) -- (\mxsh+0.75,-\mysh);
\draw [ptzert] (-0.15-\mxsh,-0.6-\mysh) rectangle ++(.3,1.2);
\draw [thkln] (-0.75-\mxsh,-\mysh) -- (-0.15-\mxsh,-\mysh)
node [near start,below] {$\scriptstyle \xca$};
\begin{scope}[xscale=-1]
\draw [ptzert] (-0.15-\mxsh,-0.6-\mysh) rectangle ++(.3,1.2);
\draw [thkln] (-0.75-\mxsh,-\mysh) -- (-0.15-\mxsh,-\mysh)
node [near start,below] {$\scriptstyle \xca$};
\end{scope}
\begin{scope}[yscale=-1]
\draw [ptzert] (-0.15-\mxsh,-0.6-\mysh) rectangle ++(.3,1.2);
\draw [thkln] (-0.75-\mxsh,-\mysh) -- (-0.15-\mxsh,-\mysh)
node [near start,above] {$\scriptstyle \xca$};
\end{scope}
\begin{scope}[yscale=-1,xscale=-1]
\draw [ptzert] (-0.15-\mxsh,-0.6-\mysh) rectangle ++(.3,1.2);
\draw [thkln] (-0.75-\mxsh,-\mysh) -- (-0.15-\mxsh,-\mysh)
node [near start,above] {$\scriptstyle \xca$};
\end{scope}
\end{scope}
%%%%%%%%%%%%%%%%%%%%%%%%%%%%%%%%%%%%
\begin{scope}[xshift=5cm]
\draw [bcrc] (-0.75-\mxsh,-\mysh) -- (0.75+\mxsh,-\mysh);
\draw [bcrc] (-0.75-\mxsh,\mysh) -- (0.75+\mxsh,\mysh);
\draw [thkln] (-\mxsh+0.15,\mysh-0.2) to [out=0,in=90] (-\mxsh+1.1,0)
% node [near end,right] {$\scriptstyle i$}
 to [out=-90,in=0] (-\mxsh+0.15,-\mysh+0.2) ;
\begin{scope}[xscale=-1]
\draw [thkln] (-\mxsh+0.15,\mysh-0.2) to [out=0,in=90] (-\mxsh+1.1,0) to [out=-90,in=0] (-\mxsh+0.15,-\mysh+0.2);
\end{scope}
\node at (-\mxsh+0.25,0) {$\scriptstyle \spvr$};
\node at (\mxsh-0.25,0) {$\scriptstyle \spvr$};
%\draw [lnovr] (-\mxsh+0.15,-\mysh) to [out=0,in=180] (\mxsh-0.15,\mysh);
%\draw [thkln] (-\mxsh+0.15,-\mysh) to [out=0,in=180] (\mxsh-0.15,\mysh);
\draw [thkln] (-\mxsh+0.15,\mysh+0.2) -- (\mxsh-0.15,\mysh+0.2) node [midway,above] {$\scriptstyle \xca-\spvr$};
\draw [thkln] (-\mxsh+0.15,-\mysh-0.2) -- (\mxsh-0.15,-\mysh-0.2) node [midway, below] {$\scriptstyle \xca-\spvr$};
%
%\draw [lnovr] (-0.75-\mxsh,0) to [out=0,in=180] (0,-\mysh+0.7) to [out=0,in=180] (0.75+\mxsh,0);
%\draw (-0.75-\mxsh,0) to [out=0,in=180] (0,-\mysh+0.7) to [out=0,in=180] (0.75+\mxsh,0);
%
\draw [ptzert] (-0.15-\mxsh,-0.6-\mysh) rectangle ++(.3,1.2);
\draw [thkln] (-0.75-\mxsh,-\mysh) -- (-0.15-\mxsh,-\mysh)
node [near start,below] {$\scriptstyle \xca$};
\begin{scope}[xscale=-1]
\draw [ptzert] (-0.15-\mxsh,-0.6-\mysh) rectangle ++(.3,1.2);
\draw [thkln] (-0.75-\mxsh,-\mysh) -- (-0.15-\mxsh,-\mysh)
node [near start,below] {$\scriptstyle \xca$};
\end{scope}
\begin{scope}[yscale=-1]
\draw [ptzert] (-0.15-\mxsh,-0.6-\mysh) rectangle ++(.3,1.2);
\draw [thkln] (-0.75-\mxsh,-\mysh) -- (-0.15-\mxsh,-\mysh)
node [near start,above] {$\scriptstyle \xca$};
\end{scope}
\begin{scope}[yscale=-1,xscale=-1]
\draw [ptzert] (-0.15-\mxsh,-0.6-\mysh) rectangle ++(.3,1.2);
\draw [thkln] (-0.75-\mxsh,-\mysh) -- (-0.15-\mxsh,-\mysh)
node [near start,above] {$\scriptstyle \xca$};
\end{scope}
\end{scope}
\end{tikzpicture}
\end{equation*}
The gray strips in $\xDs$ combine into \tBcr s in the background of this diagram.

The diagrams $\xDs$ for all \tstt s $\spmp$ generate a \tmcn\ presentation of $\xDclN$, hence $\tKHm(\xDclN)$ can be computed by spectral sequence, and its $E_1$ term is a sum of appropriately shifted homologies $\KHm(\xDs)$:
\[
\xEo = \shfr^{\xca\gvD}\bigoplus_{\substack{\spmp\\ k\geq 0 }}m_{\spmp,k}\,\shcr^{\xabms+k}\,  \KHm(\xDs),
\]
where $\xabms=\sum_{\xvrt\in\svrt}\spvr^2$.
Hence, a component $\xEoij$ of bi-degree $i,j$ (both are homological and have nothing to do with filtration!) has the form
\begin{equation}
\label{eq:eogr}
\xEoij = \bigoplus_{\substack{\spmp\\ k\geq 0 }}m_{\spmp,k} \KHmvv{i-\xabms-k}{j-\xca\gvD}(\xDs)
\end{equation}

As we already noted in Remark\rw{rmk:bndss}, further steps of spectral sequence may only reduce homology, hence $\xEoij=0$ implies $\tKHmvv{i}{j}(\xDclN)=0$. Moreover, all differentials have bi-grading (-1,1), hence  $\xEovv{i+1}{j-1}=\xEovv{i-1}{j+1}=0$ implies $\tKHmvv{i}{j}(\xDclN) = \xEoij$. These arguments imply that Theorem\rw{thm:kqbnd} follows from the proposition
\begin{proposition}
\label{prp:lm1}
$\xEoij=0$ if one of the following conditions is satisfied:
\begin{align}
\label{eq:bd1b}
i &<0,
\\
\label{eq:bd2b}
j&<- \shlf i - \shlf\ncrD - \sthlf\ncriD,
\\
\label{eq:bd3b}
j& < -i  - \ncriD - \xca\gvD.
%\\
%\label{eq:bd4a}
%-\xca \gvD-j& =i + \xabms > 0\quad\text{and $\xD$ is \tBadq.}
\end{align}
Moreover, if $\xD$ is \tBadq, then
\begin{equation}
\label{eq:bd4b}
\xEovv{i}{-i} =
\begin{cases}
0,&\text{if $i\neq 0$,}
\\
\xalg,&\text{if $i=0$.}
\end{cases}
\end{equation}
\end{proposition}
Thus we proved Theorem\rw{thm:leviso}\qed

In view of \ex{eq:eogr}, Proposition\rw{prp:lm1} follows from the next one:
\begin{proposition}
\label{prp:lm2}
$\KHmvv{i}{j}(\xDs)=0$ if one of the following two conditions is satisfied:
\begin{align}
\label{eq:bd1d}
i &<0,
\\
\label{eq:bd2d}
j&<- \shlf \xabms - \shlf\ncrD - \sthlf\ncriD - \xca \gvD,
\\
\label{eq:bd3d}
j& < -\xabms  - \ncriD - \xca\gvD.
%\\
%\label{eq:bd4a}
%-\xca \gvD-j& =i + \xabms > 0\quad\text{and $\xD$ is \tBadq.}
\end{align}
Furthermore, if a diagram $\xD$ is \tBadq, then $\KHmvv{i}{j}(\xDs)=0$ for
\begin{equation}
\label{eq:bd4d}
\KHmvv{0}{-\xabms-\xca\gvD}(\xDs) =
\begin{cases}
0,&\text{if $\xabms>0$},
\\
\IQ,&\text{if $\xabms=0$.}
\end{cases}
\end{equation}
\end{proposition}
\begin{proof}[Proof of Proposition\rw{prp:lm1}]
Conditions\rx{eq:bd1b}--\rxw{eq:bd3b} follow easily from the conditions\rx{eq:bd1d}--\rxw{eq:bd3d}. In order to prove \ex{eq:bd4b}, observe that according to \ex{eq:bd4b}, $\xEovv{i}{-i}$ is a sum of homologies $\KHmvv{i'}{j'}(\xDs)$ with $i' =i-\xabms-k$, $j'=-i-\xca\gvD$, hence
\[
j' = -\xabms -\xca\gvD - i' - k.
\]
Since $i'\geq 0$ by \ex{eq:bd1d} and $k\geq 0$ by \ex{eq:eogr}, then in view of the bound\rx{eq:bd3d} with $\ncriD=0$ we conclude that non-trivial contributing homology exists only for $i'=k=0$, so $j'= - \xabms-\xca\gvD$ and $i=\xabms$. Thus we proved that $\xEovv{i}{-i}$ is a sum of homologies $\KHmvv{0}{-\xabms-\xca\gvD}(\xDs)$ with $\xabms=i$, hence \ex{eq:bd4b} follows from \ex{eq:bd4d} and from the fact that the \tstt\ $\spmp$ with $\xabms=0$ is unique (it corresponds to \Bsplng\ all crossings in $\xDclN$) and its multiplicity in the presentation of $\xEovv{i}{-i}$ is one, because complete \Bsplng\ has multiplicity one in \ex{eq:colKhbr}.
\end{proof}

\begin{proof}[Proof of Proposition\rw{prp:lm2}]
First of all, we observe that the bound\rx{eq:bd1d} follows from the fact that $\xDs$ has no crossings, while the formulas\rx{eq:projcn},\rx{eq:grproj} for the \cJWp\ contain only non-negative shifts of \thdgr.

The proof of other bounds requires a simplification of the complex, whose homology yields Khovanov homology $\KHm(\xDs)$. We cut the diagram $\xDs$ into pieces (tangles), simplify their Khovanov complexes and then glue those complexes back together.

%
%
%
%Next steps in spectral sequence only reduce the homology, so it is sufficient to establish the bounds on \tqdgr\ for each homology $\KHm(\xDs)$ appearing in $E_1$.
%%
%According to \ex{eq:colKhbr}, each diagram $\xDs$ appears in the \tmcn\ with a \thdgr\ shift at least
%$\shcr^{-\hlf\xca^2 + \xabms}$, where
%$\xabms=\sum_{\xvrt\in\svrt}\spvr^2$,
%so it is sufficient for us to prove the following bound: $\KHmvv{i}{j}(\xDs) = 0$ if one of the following conditions holds:
%\begin{align}
%\label{eq:bd2xa1}
%j&<- \shlf \xabms - \shlf\ncrD - \sthlf\ncriD - \xca \gvD
%\\
%\label{eq:bd3xa1}
%j& < -\xabms  - \ncriD - \xca\gvD
%%\\
%%\label{eq:bd4a1}
%%-\xca \gvD-j& =i + \xabms > 0\quad\text{and $\xD$ is \tBadq.}
%\end{align}
%%\begin{equation}
%%\label{eq:dgqest}
%%\KHmvv{i}{j}(\xDs) = 0,\quad\text{if $j< - \shlf \xabms - \shlf\ncrD - \sthlf\ncriD - \xca \gvD$  }.
%%\end{equation}
%
%In order to prove the bounds\rx{eq:bd2a}--\rxw{eq:bd3a}, we cut each diagram into pieces (tangles) and simplify their Khovanov categorification complexes.

Consider a neighborhood of a \tBcr\ $c$ within a diagram $\xDs$ and cut in the middle all \tstrtl s which are attached to it. We are going to simplify the complex of the resulting colored tangle $\xtusc$ by inserting two extra \tJWp s in it and then purging all other (preexisting) projectors.

For $a\geq b$ let the box
$
\;
\begin{tikzpicture}[menvthree]%,rotate=90]
\draw [ptzer] (0.4,0.8) rectangle ++(-0.8,-1.6);
\draw [thkc] (-0.1,-0.2) arc (-90:90:0.2);
\draw [thkln] (-1.2,0) -- (-0.4,0)   node [near start,below] {$\scriptscriptstyle a$}
(0.4,0) -- (1.2,0)  node [near end,below] {$\scriptscriptstyle b$};
\end{tikzpicture}
\;
$
denote any \taTLt\ with the property
\[
\swdv{
\begin{tikzpicture}[menvone]
\draw [ptzer] (0.4,0.8) rectangle ++(-0.8,-1.6);
\draw [thkc] (-0.1,-0.2) arc (-90:90:0.2);
\draw [thkln] (-1,0) -- (-0.4,0) node [near start,below] {$\scriptstyle a$}
(0.4,0) -- (1,0) node [near end,below] {$\scriptstyle b$};
\end{tikzpicture}
} = b.
\]
In other words, a tangle
$
\;
\begin{tikzpicture}[menvthree]%,rotate=90]
\draw [ptzer] (0.4,0.8) rectangle ++(-0.8,-1.6);
\draw [thkc] (-0.1,-0.2) arc (-90:90:0.2);
\draw [thkln] (-1,0) -- (-0.4,0)  % node [near start,below] {$\scriptstyle a$}
(0.4,0) -- (1,0);  %node [near end,below] {$\scriptstyle b$};
\;
\end{tikzpicture}
$
contains no cups, but only caps and \txstrs s.
%%%%%%%%%%%%%%%%%%%%%%%%%%%%%%%%%%%%%%%
\begin{lemma}
The Khovanov categorification complex of the colored tangle diagram $\xtusc$ can be presented in the form
\begin{equation}
\label{eq:prcmp}
\xKhv{\xtusc} \hteqv
\boxed{
\cdots\longrightarrow\shcr^i
\bigoplus_{0\leq j\leq i}\shfr^j\left( \bigoplus_{\tau} m_{ij,\tau} \xKhv{\tau}\right)
%\bigoplus_{\substack{0\le j \le i \\ \gamma\in\sTLab,\; \wdv{\gamma}=b}} %m_{ij,\gamma}\,\shfr^j\,\xKhv{\gamma}
\longrightarrow\cdots
}_{\;i=0}^{\;\infty}
\end{equation}
where the diagrams $\tau$ are of one of two types depicted in \fg{fg:twodg}:
%%%%%%%%%%%%%%%%%%%%%%%%%%%%%%%%
%%%%%%%%%%%%%%%%%%%%%%%%%%%%%%%%
\begin{figure}%[h]
\begin{equation*}
%\label{eq:twpdgs}
\begin{tikzpicture}[menvtwo]
\draw [bcrct] (-1.25,0) -- (1.25,0) (1.55,0) to [out=0,in=90] (4,-2) to [out=-90,in=0] (0,-5)
(-1.55,0) to [out=180,in=90] (-4,-2) to [out=-90,in=180] (0,-5) ;
\draw [thkln] (-1.25,0) -- (1.25,0) (1.55,0) to [out=0,in=90] (4,-2) to [out=-90,in=0] (0,-5)
(-1.55,0) to [out=180,in=90] (-4,-2) to [out=-90,in=180] (0,-5) ;
\node at (0,-5.5) {$\scriptstyle  \xca_1$};
\node at (0,0.5) {$\scriptstyle \xca_2$};
\draw [ptzert] (-1.55,-0.6) rectangle ++(0.3,1.2);
\draw [ptzert] (1.55,-0.6) rectangle ++(-0.3,1.2);
%\draw [ptzert] (4.05,-0.6) rectangle ++(0.3,1.2);
%%%%
\draw [ptzer] (0.8,2) rectangle ++(-1.6,-0.8);
\draw [thkc] (0.2,0.1+1.6) arc (0:-180:0.2);
\draw [thkln]  (-0.5,2) -- (-0.5,2.6);
\draw [thkln]  (0.5,2) -- (0.5,2.6);
%\node at (0.05,2.3) {$\cdots$};
\node at (0.05,2.4) {$\scriptstyle \cdots$};
\draw [ptzer] (0.8,-2) rectangle ++(-1.6,0.8);
\draw [thkc] (0.2,-0.1-1.6) arc (0:180:0.2);
\draw [thkln]  (-0.5,-2) -- (-0.5,-2.6);
\draw [thkln]  (0.5,-2) -- (0.5,-2.6);
%\node at (0.05,-2.3) {$\cdots$};
\node at (0.05,-2.4) {$\scriptstyle \cdots$};
\draw [thkln] (-1.25,0.35) to [out=0,in=-90] (-0.5,1.2);
\draw [thkln,xscale=-1] (-1.25,0.35) to [out=0,in=-90] (-0.5,1.2);
\draw [thkln,yscale=-1] (-1.25,0.35) to [out=0,in=-90] (-0.5,1.2);
\draw [thkln,xscale=-1,yscale=-1] (-1.25,0.35) to [out=0,in=-90] (-0.5,1.2);
%%%%%%%
%\draw [thkln] (-1.55,0.35) to [out=180,in=-90] ++(-0.75,0.85) -- (-2.3,2.6);
%\draw [thkln,xscale=-1] (-1.55,0.35) to [out=180,in=-90] ++(-0.75,0.85) -- (-2.3,2.6);
%%%%%%%
%\draw [thkln,xshift=5.6cm] (-1.55,0.35) to [out=180,in=-90] ++(-0.75,0.85) -- (-2.3,2.6);
%\draw [thkln,xscale=-1,yscale=-1,xshift=-2.8cm] (-1.55,0.35) to [out=180,in=-90] ++(-0.75,0.85) -- (-2.3,2.6);
\draw [decorate,decoration={brace,amplitude=4pt},xshift=0,yshift=-5pt]
(0.5cm+5pt,-2.6) -- (-0.5cm-5pt,-2.6)
node [black,midway,yshift=-0.4cm]
{\scriptsize to struts};
\draw [decorate,decoration={brace,amplitude=4pt},xshift=0,yshift=5pt]
%(-2.3cm-5pt,2.6) -- (2.3cm+5pt,2.6)
(-0.5cm-5pt,2.6) -- (0.5cm+5pt,2.6)
node [black,midway,yshift=0.4cm]
{\scriptsize to struts};
\end{tikzpicture}
\qquad
\qquad
\begin{tikzpicture}[menvtwo]
\draw [bcrct] (-1.25,0) -- (1.25,0) (1.55,0) to [out=0,in=90] (4,-2) to [out=-90,in=0] (0,-5)
(-1.55,0) to [out=180,in=90] (-4,-2) to [out=-90,in=180] (0,-5) ;
\draw [thkln] (0,1.2) -- (0,-1.2) (1.55,0) to [out=0,in=90] (4,-2) to [out=-90,in=0] (0,-5)
(-1.55,0) to [out=180,in=90] (-4,-2) to [out=-90,in=180] (0,-5) ;
\node at (0,-5.5) {$\scriptstyle  \xca_1$};
\draw [ptzert] (-1.55,-0.6) rectangle ++(0.3,1.2);
\draw [ptzert] (1.55,-0.6) rectangle ++(-0.3,1.2);
%\draw [ptzert] (4.05,-0.6) rectangle ++(0.3,1.2);
%%%%
\draw [ptzer] (0.8,2) rectangle ++(-1.6,-0.8);
\draw [thkc] (0.2,0.1+1.6) arc (0:-180:0.2);
\draw [thkln]  (-0.5,2) -- (-0.5,2.6);
\draw [thkln]  (0.5,2) -- (0.5,2.6);
%\node at (0.05,2.3) {$\cdots$};
\node at (0.05,2.4) {$\scriptstyle \cdots$};
\draw [ptzer] (0.8,-2) rectangle ++(-1.6,0.8);
\draw [thkc] (0.2,-0.1-1.6) arc (0:180:0.2);
\draw [thkln]  (-0.5,-2) -- (-0.5,-2.6);
\draw [thkln]  (0.5,-2) -- (0.5,-2.6);
%\node at (0.05,-2.3) {$\cdots$};
\node at (0.05,-2.4) {$\scriptstyle \cdots$};
\draw [thkln] (-1.25,0.35) to [out=0,in=-90] (-0.5,1.2);
\draw [thkln,xscale=-1] (-1.25,0.35) to [out=0,in=-90] (-0.5,1.2);
\draw [thkln,yscale=-1] (-1.25,0.35) to [out=0,in=-90] (-0.5,1.2);
\draw [thkln,xscale=-1,yscale=-1] (-1.25,0.35) to [out=0,in=-90] (-0.5,1.2);
%%%%%%%%%%%
%\draw [thkln] (-1.55,0.35) to [out=180,in=-90] ++(-0.75,0.85) -- (-2.3,2.6);
%\draw [thkln,xscale=-1] (-1.55,0.35) to [out=180,in=-90] ++(-0.75,0.85) -- (-2.3,2.6);
%%%%%%%%%%%
%
%
%\draw [thkln,xshift=5.6cm] (-1.55,0.35) to [out=180,in=-90] ++(-0.75,0.85) -- (-2.3,2.6);
%\draw [thkln,xscale=-1,yscale=-1,xshift=-2.8cm] (-1.55,0.35) to [out=180,in=-90] ++(-0.75,0.85) -- %(-2.3,2.6);
\draw [decorate,decoration={brace,amplitude=4pt},xshift=0,yshift=-5pt]
(0.5cm+5pt,-2.6) -- (-0.5cm-5pt,-2.6)
node [black,midway,yshift=-0.4cm]
{\scriptsize to struts};
\draw [decorate,decoration={brace,amplitude=4pt},xshift=0,yshift=5pt]
%(-2.3cm-5pt,2.6) -- (2.3cm+5pt,2.6)
(-0.5cm-5pt,2.6) -- (0.5cm+5pt,2.6)
node [black,midway,yshift=0.4cm]
{\scriptsize to struts};
\end{tikzpicture}
\end{equation*}
\caption{A purged vicinity of a \tBcr}
\label{fg:twodg}
\end{figure}
%%%%%%%%%%%%%%%%%%%%%%%%%%%%%%%%
with $\xca_1,\xca_2\leq \xca$.
%A box with an arc denotes a `\ntwd' \tTLt: for $a \geq b$,
%\[
%\swdv{
%\begin{tikzpicture}[menvone]
%\draw [ptzer] (0.4,0.8) rectangle ++(-0.8,-1.6);
%\draw [thkc] (-0.1,-0.2) arc (-90:90:0.2);
%\draw [thkln] (-1,0) -- (-0.4,0) node [near start,below] {$\scriptstyle a$}
%(0.4,0) -- (1,0) node [near end,below] {$\scriptstyle b$};
%\end{tikzpicture}
%} = b.
%\]
%In other words, a tangle
%$
%\begin{tikzpicture}[menvthree,rotate=90]
%\draw [ptzer] (0.4,0.8) rectangle ++(-0.8,-1.6);
%\draw [thkc] (-0.1,-0.2) arc (-90:90:0.2);
%\draw [thkln] (-1,0) -- (-0.4,0)  % node [near start,below] {$\scriptstyle a$}
%(0.4,0) -- (1,0);  %node [near end,below] {$\scriptstyle b$};
%\end{tikzpicture}
%$
%contains no cups, but only caps and \txstrs s.
\end{lemma}

By gluing the complexes\rx{eq:prcmp} back together
%, yet without replacing the resulting single line circles by $(\shfr + \shfr^{-1})\xalg$ in accordance with %the second part of \ex{eq:dkhbr}. Thus
we get a complex
\[
\xKhv{\xDs} \hteqv
\boxed{
\cdots\longrightarrow\shcr^i
\bigoplus_{0\leq j\leq i}\shfr^j\left( \bigoplus_{\xDcir} m_{ij,\tau} \xKhv{\xDcir}\right)
\longrightarrow\cdots
}_{\;i=0}^{\;\infty}
\]
such that $\KHm(\xDs) = \Hm(\xKhv{\xDs})$. The `circle diagrams' $\xDcir$ which result from gluing the diagrams of \fg{fg:twodg} at the \tstrtl\ cutting points and replacing projectors with complexes\rx{eq:projcn} and\rx{eq:grproj}, consist of multiple single line circles.
%
%Consider first the idea of the proof. If we replace each \tJWp\ in $\xDs$ with its presentation\rx{eq:projcn}, then we get a semi-infinite complex of \emph{circle diagrams}, each consisting of multiple disjoint single-line circles accompanied by non-negative \thdgr\ and \tqdgr\ shifts coming from $\shcr^i\shfr^j$ of \ex{eq:grproj}.
In view of the second formula of \ex{eq:dkhbr}, the lowest \tqdgr\ in the homology $\KHm(\xDs)$ may be bounded by the highest number of circles in those circle diagrams.

The circles in circle diagrams are of three types. The first type is \emph{\tjmp} circles: they contain at least one \tstrt\ line. The circles of the second and third type stay within the same \tBcr. A \emph{\txstr} circle goes along its \tBcr, passing straight through each \tJWp\ on its way. A \emph{\twndg} circle changes its direction at least twice, because it contains at least one cup and one cap of a constituent \taTLt\ coming from one of projectors.

Let us prove the inequalities\rx{eq:bd2d} and\rx{eq:bd3d} by finding upper bounds for the numbers $\njcr$, $\nscr$ and $\nwcr$ of \tjmp, \txstr\ and \twndg\ circles respectively in a circle diagram.

We begin with $\njcr$.
A \tjmpc\ must contain at least two \tstrt\ lines of an \tadq\ crossing or at least one \tstrt\ line of an inadequate crossing, so the number of \tjmpc s $\njcr$ has a bound:
\begin{equation}
\label{eq:sbnd}
\njcr \leq \sum_{\xvrt\in\svrta}\spvr + 2\sum_{\xvrt\in\svrti}\spvr,
\end{equation}
where $\svrta,\svrti\subset\svrt$ are the subsets of \tBadq\ and \tBiadq\ crossings.
The obvious inequalities
\[
\spvr\leq\shlf\spvr^2+\shlf,\qquad 2\spvr\leq\shlf\spvr^2 +2,\qquad
\spvr\leq\spvr^2,\qquad 2\spvr\leq\spvr^2 + 1,
\]
(the third inequality uses the fact that $\spvr$ is integer) indicate that the bound\rx{eq:sbnd} implies to other bounds:
\begin{equation}
\label{eq:mbnd}
\njcr \leq \shlf\xabms + \shlf\ncrD + \sthlf\ncriD,\qquad
\njcr \leq \xabms + \ncriD,
\end{equation}
which means that the first three terms in the \rhs of the inequality\rx{eq:bd2d} and the first two terms in the inequality\rx{eq:bd3d} bound the negative contribution of \tjmpc s to the \tqdgr\ of $\KHm(\xDs)$.

Next we prove the bound
\begin{equation}
\label{eq:wbd}
\nscrb + \nwcrb \leq \xca,
\end{equation}
where $\nscrb$ and $\nwcrb$ are the numbers of \txstr\ and \twndg\ circles within any given \tBcr\ $\crb$. It implies the bound $\nscr + \nwcr \leq \xca \gvD$ and combined with the bounds\rx{eq:mbnd} they imply the bounds\rx{eq:bd2d} and\rx{eq:bd3d}. In order to prove the bound\rx{eq:wbd}, we observe that
%Establishing a bound $\njrc\leq\xca$ within diagrams\rx{eq:twpdgs} is easy.
in the first diagram a \tstrghtc\ contains one strand from the $\xca_1$-cable and one strand from the $\xca_2$-cable, while a \twndgc\ contains at least two strands from one of these cables, hence there is a bound
\begin{equation}
\label{eq:bndNt}
\nscrb+\nwcrb\leq\shlf(\xca_1+\xca_2)\leq\xca.
\end{equation}
The second diagram is treated similarly, if we set $\xca_2=0$ in the previous argument.
% contains only \twndgc s, each of them contains at least two strands of the $\xca_1$-cable, hence %there $\njrc\leq\hlf\xca_1\leq\xca$.
Thus we proved the bounds\rx{eq:bd2d} and\rx{eq:bd3d}.

It remains to prove \ex{eq:bd4d}. Since this time $\xD$ is \tBadq, the second inequality of\rx{eq:mbnd} becomes
$ \njcr \leq \xabms$. Since we consider only homology of zeroth \thdgr, then according to \eex{eq:projcn} and\rx{eq:grproj}, we may replace \tJWp s with identity braids, so there is only one circle diagram $\xDcir$ contributing to  $\KHmvv{0}{-\xabms-\xca\gvD}(\xDs)$, and this circle diagram has no \twndg\ circles: $\nwcr=0$. Furthermore, $\nscr\leq\xca\gvD$, but if $\xabms\neq 0$, then there is at least one pair of \tstrtl s in $\xDs$, so $\nscr<\xca$ and $\njcr+\nscr<\xabms + \xca\gvD$, hence $\KHmvv{0}{-\xabms-\xca\gvD}(\xDs) = 0$. If $\xabms=0$, then $\xDs$ has no \tstrtl s and consists of disjoint $\xca$-cabled circles, so the relevant circle diagram $\xDcir$
%. Again, at zeroth \thdgr\ we have to replace projectors with identity braids, so the relevant circle
%diagram
consists
of $\xca\gvD$ single-line circles, and $\KHmvv{0}{-\xca\gvD}(\xDs)=\IQ$ follows from the second equation of\rx{eq:dkhbr}.
\end{proof}

\begin{proof}[Proof of Lemma\rw{eq:prcmp}]
We prove the lemma by `purging' \cJWp s appearing in the tangle $\xtusc$.
%%%%%%%%
%Let $\tau$ be a flat tangle containing some \tJWp s. Purging of a particular $a$-cable projector in $\tau$ is the following procedure. %For a \tTL\ $(a,a)$-tangle $\gamma$, let
%Replace this projector with the presentation\rx{eq:projcn},\rx{eq:grproj}, then the complex $\xKhv{\tau}$ becomes a \tmcn:
%\begin{equation}
%\label{eq:tmtcn}
%\xKhv{\tau} \hteqv
%\boxed{
%\cdots\longrightarrow\shcr^{i+1}\bigoplus_{\substack{0\le j \le i \\ \gamma\in\sTLa, \wdfcv{\gamma}>0}} m_{ij,\gamma}\,\shfr^j\,\xKhv{\xtg}
%\longrightarrow\cdots\longrightarrow
%\shcr\bigoplus_{\gamma\in\sTLa, \wdfcv{\gamma}>0} m_{00,\gamma}\xKhv{\xtg}
% \longrightarrow\xKhv{\xtId}
%}
%\end{equation}
%%
%%, whose constituent complexes are $\xKhv{\xtg}$, where $\gamma$ are \tTL\ $(a,a)$-tangles %appearing in \ex{eq:grproj}, as well as the identity braid from \ex{eq:projcn}, while a
%Here $\xtg$ is a tangle $\xtg$ constructed from $\tau$ by replacing the projector with $\gamma$, while $\xId_a$ is the identity braid of $a$ strands.
% A \taTLt\ $\gamma$ may contain cups and caps in addition to \txstrs. If both ends of a cup or a cap are connected to another \tJWp\ in $\tau$, then the complex $\xKhv{\xtg}$ is contractible. Purging means that we contract all such complexes within the \tmcn\ presentation\rx{eq:tmtcn}
%
In order to bring the complex $\xKhv{\xtusc}$ to the form\rx{eq:prcmp} with diagrams $\tau$ depicted in \fg{fg:twodg}, we insert two extra \tJWp s side-by-side at any place on the cable which runs along the \tBcr. Then we go from the front one (relative to the clockwise orientation) to the back one in the clockwise direction, purging each preexisting projector that appears on our way. It is easy to prove by induction that after every projector purge we get a \tmcn\ presentation
\begin{equation*}
\xKhv{\xtusc} \hteqv
\boxed{
\cdots\longrightarrow\shcr^i
\bigoplus_{0\leq j\leq i}\shfr^j\left( \bigoplus_{\tau} m'_{ij,\tau} \xKhv{\tau}\right)
%\bigoplus_{\substack{0\le j \le i \\ \gamma\in\sTLab,\; \wdv{\gamma}=b}} %m_{ij,\gamma}\,\shfr^j\,\xKhv{\gamma}
\longrightarrow\cdots
}_{\;i=0}^{\;\infty}
\end{equation*}
whose constituent diagrams $\tau$ have one of two possible forms between the front projector and the first unpurged  projector (which lies in the pictures to the left of the dashed line) depicted in \fg{fg:prgd}.
\begin{figure}[h]
\[
\begin{tikzpicture}[menvtwo]
\draw [bcrct] (-1.55-0.6-0.5,0) -- (-1.55,0) (-1.25,0) -- (1.25,0) (1.55,0) -- (4.05,0) (4.35,0) -- (4.95+0.5,0);
\draw [thkln] (-1.55-0.6-0.5,0) -- (-1.55,0) (-1.25,0) -- (1.25,0) (1.55,0) -- (4.05,0) (4.35,0) -- (4.95+0.5,0);
\draw [ptzert] (-1.55,-0.6) rectangle ++(0.3,1.2);
\draw [ptzert] (1.55,-0.6) rectangle ++(-0.3,1.2);
\draw [ptzert] (4.05,-0.6) rectangle ++(0.3,1.2);
%%%%
\draw [ptzer] (0.8,2) rectangle ++(-1.6,-0.8);
\draw [thkc] (0.2,0.1+1.6) arc (0:-180:0.2);
\draw [thkln]  (-0.5,2) -- (-0.5,2.6);
\draw [thkln]  (0.5,2) -- (0.5,2.6);
%\node at (0.05,2.3) {$\cdots$};
\node at (0.05,2.4) {$\scriptstyle \cdots$};
\draw [ptzer] (0.8,-2) rectangle ++(-1.6,0.8);
\draw [thkc] (0.2,-0.1-1.6) arc (0:180:0.2);
\draw [thkln]  (-0.5,-2) -- (-0.5,-2.6);
\draw [thkln]  (0.5,-2) -- (0.5,-2.6);
%\node at (0.05,-2.3) {$\cdots$};
\node at (0.05,-2.4) {$\scriptstyle \cdots$};
\draw [thkln] (-1.25,0.35) to [out=0,in=-90] (-0.5,1.2);
\draw [thkln,xscale=-1] (-1.25,0.35) to [out=0,in=-90] (-0.5,1.2);
\draw [thkln,yscale=-1] (-1.25,0.35) to [out=0,in=-90] (-0.5,1.2);
\draw [thkln,xscale=-1,yscale=-1] (-1.25,0.35) to [out=0,in=-90] (-0.5,1.2);
%\draw [thkln] (-1.55,0.35) to [out=180,in=-90] ++(-0.75,0.85) -- (-2.3,2.6);
\draw [thkln,xscale=-1] (-1.55,0.35) to [out=180,in=-90] ++(-0.75,0.85) -- (-2.3,2.6);
\draw [thkln,xshift=5.6cm] (-1.55,0.35) to [out=180,in=-90] ++(-0.75,0.85) -- (-2.3,2.6);
\draw [thkln,xscale=-1,yscale=-1,xshift=-2.8cm] (-1.55,0.35) to [out=180,in=-90] ++(-0.75,0.85) -- (-2.3,2.6);
\draw [decorate,decoration={brace,amplitude=4pt},xshift=0,yshift=-5pt]
(5.1cm+5pt,-2.6) -- (-0.5cm-5pt,-2.6)
node [black,midway,yshift=-0.4cm]
{\scriptsize to struts};
\draw [decorate,decoration={brace,amplitude=4pt},xshift=0,yshift=5pt]
%(-2.3cm-5pt,2.6) -- (3.3cm+5pt,2.6)
(-0.5cm-5pt,2.6) -- (3.3cm+5pt,2.6)
node [black,midway,yshift=0.4cm]
{\scriptsize to struts};
\draw [dashed] (2.8,2.6) -- (2.8,-2.6);
\end{tikzpicture}
\;,\qquad
\begin{tikzpicture}[menvtwo]
\draw [bcrct] (-1.55-0.6-0.5,0) -- (-1.55,0) (-1.25,0) -- (1.25,0) (1.55,0) -- (4.05,0) (4.35,0) -- (4.95+0.5,0);
\draw [thkln] (-1.55-0.6-0.5,0) -- (-1.55,0) (0,1.2) -- (0,-1.2) (1.55,0) -- (4.05,0) (4.35,0) -- (4.95+0.5,0);
\draw [ptzert] (-1.55,-0.6) rectangle ++(0.3,1.2);
\draw [ptzert] (1.55,-0.6) rectangle ++(-0.3,1.2);
\draw [ptzert] (4.05,-0.6) rectangle ++(0.3,1.2);
%%%%
\draw [ptzer] (0.8,2) rectangle ++(-1.6,-0.8);
\draw [thkc] (0.2,0.1+1.6) arc (0:-180:0.2);
\draw [thkln]  (-0.5,2) -- (-0.5,2.6);
\draw [thkln]  (0.5,2) -- (0.5,2.6);
\node at (0.05,2.4) {$\scriptstyle \cdots$};
\draw [ptzer] (0.8,-2) rectangle ++(-1.6,0.8);
\draw [thkc] (0.2,-0.1-1.6) arc (0:180:0.2);
\draw [thkln]  (-0.5,-2) -- (-0.5,-2.6);
\draw [thkln]  (0.5,-2) -- (0.5,-2.6);
\node at (0.05,-2.4) {$\scriptstyle \cdots$};
\draw [thkln] (-1.25,0.35) to [out=0,in=-90] (-0.5,1.2);
\draw [thkln,xscale=-1] (-1.25,0.35) to [out=0,in=-90] (-0.5,1.2);
\draw [thkln,yscale=-1] (-1.25,0.35) to [out=0,in=-90] (-0.5,1.2);
\draw [thkln,xscale=-1,yscale=-1] (-1.25,0.35) to [out=0,in=-90] (-0.5,1.2);
%%%
%\draw [thkln] (-1.55,0.35) to [out=180,in=-90] ++(-0.75,0.85) -- (-2.3,2.6);
%%%
\draw [thkln,xscale=-1] (-1.55,0.35) to [out=180,in=-90] ++(-0.75,0.85) -- (-2.3,2.6);
\draw [thkln,xshift=5.6cm] (-1.55,0.35) to [out=180,in=-90] ++(-0.75,0.85) -- (-2.3,2.6);
\draw [thkln,xscale=-1,yscale=-1,xshift=-2.8cm] (-1.55,0.35) to [out=180,in=-90] ++(-0.75,0.85) -- (-2.3,2.6);
\draw [decorate,decoration={brace,amplitude=4pt},xshift=0,yshift=-5pt]
(5.1cm+5pt,-2.6) -- (-0.5cm-5pt,-2.6)
node [black,midway,yshift=-0.4cm]
{\scriptsize to struts};
\draw [decorate,decoration={brace,amplitude=4pt},xshift=0,yshift=5pt]
%(-2.3cm-5pt,2.6) -- (3.3cm+5pt,2.6)
(-0.5cm-5pt,2.6) -- (3.3cm+5pt,2.6)
node [black,midway,yshift=0.4cm]
{\scriptsize to struts};
\draw [dashed] (2.8,2.6) -- (2.8,-2.6);
\end{tikzpicture}
\]
\caption{Purging \tJWp s along a \tBcr}
\label{fg:prgd}
\end{figure}
In both diagrams the left projector on the grey strip is the front one, the middle projector is the first unpurged one and the right projector is the second unpurged one. It is not hard to see that if we purge the middleprojector, then we get similar diagrams with the third projector becoming the first unpurged one (the left diagram may become of either left or right type after the purge, while the right diagram remains of the same type). The \tqdgr\ shifts remain non-negative, because the purging does not produce any circles: it just makes explicit various line connections that were hidden inside the constituent \taTLt s of the purged projector.
\end{proof}

\section{Proof of invariance of the tail homology under \tBrdc}
\begin{proof}[Proof of Theorem\rw{thm:rddg}]
A removal of a \tBcr\ connected to the rest of the \tBdg\ by a single \tstrt\ corresponds to the first Reidemeister move, hence the invariance of the tail homology under this removal follows from the fact that tail homology of a \tBadq\ link is determined by shifted Khovanov homologies of its \uclrd\ diagrams and the latter are invariant under this type of first Reidemeister moves.

In view of Corollary\rw{cor:mpis}, the invariance of the tail homology under the removal of `extra' \tstrt s follows from the next lemma.
\end{proof}
\begin{lemma}
\label{lm:strm}
Suppose that two distinct \tBcr s of a link diagram $\xD$ are connected by multiple \tstrt s and the diagram $\xDp$ is constructed by removal of one of those \tstrt s. Then there exists a \tdgpr\ map $\tKHm(\xDclN')\xrightarrow{\xmg}\tKHm(\xDclN)$ which is an isomorphism on $\tKHmvv{i}{\hem}$ for $i\leq \xca-1$.
\end{lemma}
%\end{proof}
%The proof is based on a simple corollary of
The proof of this lemma is similar to proofs of Section\rw{sct:mrph}:
%and\rw{sct:hest}:
we show that $\xDclN$ can be constructed from $\xDclN'$ with the help of a \ltrf\ and prove the homological smallness of the correction diagram.

We need a simple corollary of Theorem\rw{thm:colkhovbr}
\begin{corollary}
The \tKhbr\ formula\rx{eq:lmtcn} for the colored crossing can be recast in the form
\def\mxsh{2}
\def\mysh{1.2}
\begin{equation}
\label{eq:ccrc}
\begin{tikzpicture}[menvone]
\draw [thkln] (-\mxsh+0.15,\mysh) to [out=0,in=180] (\mxsh-0.15,-\mysh);
\draw [lnovr] (-\mxsh+0.15,-\mysh) to [out=0,in=180] (\mxsh-0.15,\mysh);
\draw [thkln] (-\mxsh+0.15,-\mysh) to [out=0,in=180] (\mxsh-0.15,\mysh);
%
%\draw [lnovr] (-0.75-\mxsh,0) to [out=0,in=180] (0,-\mysh) to [out=0,in=180] (0.75+\mxsh,0);
%\draw (-0.75-\mxsh,0) to [out=0,in=180] (0,-\mysh) to [out=0,in=180] (0.75+\mxsh,0);
%
\draw [ptzer] (-0.15-\mxsh,-0.6-\mysh) rectangle ++(.3,1.2);
\draw [thkln] (-0.75-\mxsh,-\mysh) -- (-0.15-\mxsh,-\mysh)
node [near start,below] {$\scriptstyle \xca$};
\begin{scope}[xscale=-1]
\draw [ptzer] (-0.15-\mxsh,-0.6-\mysh) rectangle ++(.3,1.2);
\draw [thkln] (-0.75-\mxsh,-\mysh) -- (-0.15-\mxsh,-\mysh)
node [near start,below] {$\scriptstyle \xca$};
\end{scope}
\begin{scope}[yscale=-1]
\draw [ptzer] (-0.15-\mxsh,-0.6-\mysh) rectangle ++(.3,1.2);
\draw [thkln] (-0.75-\mxsh,-\mysh) -- (-0.15-\mxsh,-\mysh)
node [near start,above] {$\scriptstyle \xca$};
\end{scope}
\begin{scope}[yscale=-1,xscale=-1]
\draw [ptzer] (-0.15-\mxsh,-0.6-\mysh) rectangle ++(.3,1.2);
\draw [thkln] (-0.75-\mxsh,-\mysh) -- (-0.15-\mxsh,-\mysh)
node [near start,above] {$\scriptstyle \xca$};
\end{scope}
\end{tikzpicture}
\;\hteqv\;
\shcr^{-\hlf\xca^2}\;
\boxed{
\begin{tikzpicture}[menvone]
\draw [thkln] (-\mxsh+0.15,\mysh) -- (-\mxsh*0.3,\mysh)
node [near start,above] {$\scriptstyle \xca$}
(\mxsh*0.3,\mysh) -- (\mxsh-0.15,\mysh)
node [near end,above] {$\scriptstyle \xca$};
\draw [thkln] (-\mxsh+0.15,-\mysh) -- (-\mxsh*0.3,-\mysh)
node [near start,below] {$\scriptstyle \xca$} (\mxsh*0.3,-\mysh) -- (\mxsh-0.15,-\mysh)
node [near end,below] {$\scriptstyle \xca$};
%%%%
\draw[ptzer] (-\mxsh*0.3,-\mysh*1-0.6) rectangle ++(\mxsh*0.6,\mysh*2+1.2);
\node at (0,0) {$*$};
\end{tikzpicture}
\longrightarrow
\begin{tikzpicture}[menvone]
\draw [thkln] (-\mxsh+0.15,\mysh) --  (\mxsh-0.15,\mysh);
\draw [thkln] (-\mxsh+0.15,-\mysh) -- (\mxsh-0.15,-\mysh);
%%%%
\draw [ptzer] (-0.15-\mxsh,-0.6-\mysh) rectangle ++(.3,1.2);
\draw [thkln] (-0.75-\mxsh,-\mysh) -- (-0.15-\mxsh,-\mysh)
node [near start,below] {$\scriptstyle \xca$};
\begin{scope}[xscale=-1]
\draw [ptzer] (-0.15-\mxsh,-0.6-\mysh) rectangle ++(.3,1.2);
\draw [thkln] (-0.75-\mxsh,-\mysh) -- (-0.15-\mxsh,-\mysh)
node [near start,below] {$\scriptstyle \xca$};
\end{scope}
\begin{scope}[yscale=-1]
\draw [ptzer] (-0.15-\mxsh,-0.6-\mysh) rectangle ++(.3,1.2);
\draw [thkln] (-0.75-\mxsh,-\mysh) -- (-0.15-\mxsh,-\mysh)
node [near start,above] {$\scriptstyle \xca$};
\end{scope}
\begin{scope}[yscale=-1,xscale=-1]
\draw [ptzer] (-0.15-\mxsh,-0.6-\mysh) rectangle ++(.3,1.2);
\draw [thkln] (-0.75-\mxsh,-\mysh) -- (-0.15-\mxsh,-\mysh)
node [near start,above] {$\scriptstyle \xca$};
\end{scope}
\end{tikzpicture}
}\;,
\end{equation}
where
\begin{equation}
\label{eq:stbx}
\begin{tikzpicture}[menvone]
\draw [thkln] (-\mxsh+0.15,\mysh) -- (-\mxsh*0.3,\mysh)
node [near start,above] {$\scriptstyle \xca$}
(\mxsh*0.3,\mysh) -- (\mxsh-0.15,\mysh)
node [near end,above] {$\scriptstyle \xca$};
\draw [thkln] (-\mxsh+0.15,-\mysh) -- (-\mxsh*0.3,-\mysh)
node [near start,below] {$\scriptstyle \xca$} (\mxsh*0.3,-\mysh) -- (\mxsh-0.15,-\mysh)
node [near end,below] {$\scriptstyle \xca$};
%%%%
\draw[ptzer] (-\mxsh*0.3,-\mysh*1-0.6) rectangle ++(\mxsh*0.6,\mysh*2+1.2);
\node at (0,0) {$*$};
\end{tikzpicture}
\;\hteqv\;
\Pcnv{
%\cdots\longrightarrow
\bigoplus_{\xki=1}^{\xca}
\shcr^{\xki^2}{a \brace \xki}_{\shcr}
\begin{tikzpicture}[menvone]
%\draw [bcrc] (-0.75-\mxsh,-\mysh) -- (0.75+\mxsh,-\mysh);
%\draw [bcrc] (-0.75-\mxsh,\mysh) -- (0.75+\mxsh,\mysh);
\draw [thkln] (-\mxsh+0.15,\mysh-0.2) to [out=0,in=90] (-\mxsh+1.1,0)
% node [near end,right] {$\scriptstyle i$}
 to [out=-90,in=0] (-\mxsh+0.15,-\mysh+0.2) ;
\begin{scope}[xscale=-1]
\draw [thkln] (-\mxsh+0.15,\mysh-0.2) to [out=0,in=90] (-\mxsh+1.1,0) to [out=-90,in=0] (-\mxsh+0.15,-\mysh+0.2);
\end{scope}
\node at (-\mxsh+0.5,0) {$\scriptstyle \xki$};
\node at (\mxsh-0.5,0) {$\scriptstyle \xki$};
\draw [thkln] (-\mxsh+0.15,\mysh+0.2) -- (\mxsh-0.15,\mysh+0.2) node [midway,above] {$\scriptstyle \xca-\xki$};
\draw [thkln] (-\mxsh+0.15,-\mysh-0.2) -- (\mxsh-0.15,-\mysh-0.2) node [midway, below] {$\scriptstyle \xca-\xki$};
\draw [ptzert] (-0.15-\mxsh,-0.6-\mysh) rectangle ++(.3,1.2);
\draw [thkln] (-0.75-\mxsh,-\mysh) -- (-0.15-\mxsh,-\mysh)
node [near start,below] {$\scriptstyle \xca$};
\begin{scope}[xscale=-1]
\draw [ptzert] (-0.15-\mxsh,-0.6-\mysh) rectangle ++(.3,1.2);
\draw [thkln] (-0.75-\mxsh,-\mysh) -- (-0.15-\mxsh,-\mysh)
node [near start,below] {$\scriptstyle \xca$};
\end{scope}
\begin{scope}[yscale=-1]
\draw [ptzert] (-0.15-\mxsh,-0.6-\mysh) rectangle ++(.3,1.2);
\draw [thkln] (-0.75-\mxsh,-\mysh) -- (-0.15-\mxsh,-\mysh)
node [near start,above] {$\scriptstyle \xca$};
\end{scope}
\begin{scope}[yscale=-1,xscale=-1]
\draw [ptzert] (-0.15-\mxsh,-0.6-\mysh) rectangle ++(.3,1.2);
\draw [thkln] (-0.75-\mxsh,-\mysh) -- (-0.15-\mxsh,-\mysh)
node [near start,above] {$\scriptstyle \xca$};
\end{scope}
\end{tikzpicture}
%\longrightarrow\cdots
}
%_{ \;\xki=1}^{ \;\xca}.
\end{equation}
\end{corollary}
\begin{proof}
According to the second inequality of\rx{eq:twineq}, $\yki(0)=0$, hence the the right tangle of the cone\rx{eq:ccrc} is the one that appears at $i=0$ in the \tmcn\rx{eq:lmtcn}. According to \ex{eq:colKhbr}, this tangle has multiplicity one, so this is the only place where it may appear in that \tmcn, and it has a zero shift of \thdgr.
\end{proof}

\begin{proof}[Proof of Lemma\rw{lm:strm}]
Since $\xDp$ is constructed from $\xD$ by a removal of a single crossing (\tstrt\ of $\sBD$), we set
\def\mxsh{2}
\def\mysh{1.2}
\[
\ytngsi =
\begin{tikzpicture}[menvone]
\draw [thkln] (-\mxsh+0.15,\mysh) to [out=0,in=180] (\mxsh-0.15,-\mysh);
\draw [lnovr] (-\mxsh+0.15,-\mysh) to [out=0,in=180] (\mxsh-0.15,\mysh);
\draw [thkln] (-\mxsh+0.15,-\mysh) to [out=0,in=180] (\mxsh-0.15,\mysh);
%
%\draw [lnovr] (-0.75-\mxsh,0) to [out=0,in=180] (0,-\mysh) to [out=0,in=180] (0.75+\mxsh,0);
%\draw (-0.75-\mxsh,0) to [out=0,in=180] (0,-\mysh) to [out=0,in=180] (0.75+\mxsh,0);
%
\draw [ptzer] (-0.15-\mxsh,-0.6-\mysh) rectangle ++(.3,1.2);
\draw [thkln] (-0.75-\mxsh,-\mysh) -- (-0.15-\mxsh,-\mysh)
node [near start,below] {$\scriptstyle \xca$};
\begin{scope}[xscale=-1]
\draw [ptzer] (-0.15-\mxsh,-0.6-\mysh) rectangle ++(.3,1.2);
\draw [thkln] (-0.75-\mxsh,-\mysh) -- (-0.15-\mxsh,-\mysh)
node [near start,below] {$\scriptstyle \xca$};
\end{scope}
\begin{scope}[yscale=-1]
\draw [ptzer] (-0.15-\mxsh,-0.6-\mysh) rectangle ++(.3,1.2);
\draw [thkln] (-0.75-\mxsh,-\mysh) -- (-0.15-\mxsh,-\mysh)
node [near start,above] {$\scriptstyle \xca$};
\end{scope}
\begin{scope}[yscale=-1,xscale=-1]
\draw [ptzer] (-0.15-\mxsh,-0.6-\mysh) rectangle ++(.3,1.2);
\draw [thkln] (-0.75-\mxsh,-\mysh) -- (-0.15-\mxsh,-\mysh)
node [near start,above] {$\scriptstyle \xca$};
\end{scope}
\end{tikzpicture}\;,
\qquad
\ytngsf =
\begin{tikzpicture}[menvone]
\draw [thkln] (-\mxsh+0.15,\mysh) --  (\mxsh-0.15,\mysh);
\draw [thkln] (-\mxsh+0.15,-\mysh) -- (\mxsh-0.15,-\mysh);
%%%%
\draw [ptzer] (-0.15-\mxsh,-0.6-\mysh) rectangle ++(.3,1.2);
\draw [thkln] (-0.75-\mxsh,-\mysh) -- (-0.15-\mxsh,-\mysh)
node [near start,below] {$\scriptstyle \xca$};
\begin{scope}[xscale=-1]
\draw [ptzer] (-0.15-\mxsh,-0.6-\mysh) rectangle ++(.3,1.2);
\draw [thkln] (-0.75-\mxsh,-\mysh) -- (-0.15-\mxsh,-\mysh)
node [near start,below] {$\scriptstyle \xca$};
\end{scope}
\begin{scope}[yscale=-1]
\draw [ptzer] (-0.15-\mxsh,-0.6-\mysh) rectangle ++(.3,1.2);
\draw [thkln] (-0.75-\mxsh,-\mysh) -- (-0.15-\mxsh,-\mysh)
node [near start,above] {$\scriptstyle \xca$};
\end{scope}
\begin{scope}[yscale=-1,xscale=-1]
\draw [ptzer] (-0.15-\mxsh,-0.6-\mysh) rectangle ++(.3,1.2);
\draw [thkln] (-0.75-\mxsh,-\mysh) -- (-0.15-\mxsh,-\mysh)
node [near start,above] {$\scriptstyle \xca$};
\end{scope}
\end{tikzpicture}\;,
\qquad
\ytngsc =
\shcr^{-\hlf\xca^2}
\begin{tikzpicture}[menvone]
\draw [thkln] (-\mxsh+0.15,\mysh) -- (-\mxsh*0.3,\mysh)
node [near start,above] {$\scriptstyle \xca$}
(\mxsh*0.3,\mysh) -- (\mxsh-0.15,\mysh)
node [near end,above] {$\scriptstyle \xca$};
\draw [thkln] (-\mxsh+0.15,-\mysh) -- (-\mxsh*0.3,-\mysh)
node [near start,below] {$\scriptstyle \xca$} (\mxsh*0.3,-\mysh) -- (\mxsh-0.15,-\mysh)
node [near end,below] {$\scriptstyle \xca$};
%%%%
\draw[ptzer] (-\mxsh*0.3,-\mysh*1-0.6) rectangle ++(\mxsh*0.6,\mysh*2+1.2);
\node at (0,0) {$*$};
\end{tikzpicture}\;,
\]
so that $\xDsi = \xD$ and $\xDsf = \xDp$, while the relation\rx{eq:cnrel} comes from\rx{eq:ccrc} if we set $\xhshf=0$ and $\xqshf= -\hlf \xca^2$. Since $\yncrv{\ytngsi}=\xca^2$ , in view of Proposition\rw{prp:gestdg} it remains to establish the bound
$
\tKHmvv{i}{\hem}(\xDsc) = 0,
$
if $i\leq\xca$. Let $\xDsci$ be the diagram $\xDsc$ in which $\ytngsc$ is replaced by a constituent tangle $\ytngsci$ from the \rhs of \ex{eq:stbx}, in which one strand of a $\xki$-cable is separated from the others: %which contains two %$\xki$-cables.
\def\esh{2.5}
\[
\ytngsci =
%\;\hteqv \;
\begin{tikzpicture}[menvone]
%\draw [bcrc] (-0.75-\mxsh,-\mysh) -- (0.75+\mxsh,-\mysh);
%\draw [bcrc] (-0.75-\mxsh,\mysh) -- (0.75+\mxsh,\mysh);
\draw [thkln] (-\mxsh+0.15,\mysh-0.2) to [out=0,in=90] (-\mxsh+1.1,0)
% node [near end,right] {$\scriptstyle i$}
 to [out=-90,in=0] (-\mxsh+0.15,-\mysh+0.2) ;
\begin{scope}[xscale=-1]
\draw [thkln] (-\mxsh+0.15,\mysh-0.2) to [out=0,in=90] (-\mxsh+1.1,0) to [out=-90,in=0] (-\mxsh+0.15,-\mysh+0.2);
\end{scope}
\node at (-\mxsh+0.2,0) {$\scriptstyle \xki-1$};
\node at (\mxsh-0.5,0) {$\scriptstyle \xki$};
\draw [thkln] (-\mxsh+0.15,\mysh+0.2) -- (\mxsh-0.15,\mysh+0.2) node [midway,above] {$\scriptstyle \xca-\xki$};
\draw [thkln] (-\mxsh+0.15,-\mysh-0.2) -- (\mxsh-0.15,-\mysh-0.2) node [midway, below] {$\scriptstyle \xca-\xki$};
\draw [ptzert] (-0.15-\mxsh,-0.6-\mysh) rectangle ++(.3,1.2);
\draw [thkln] (-\esh-0.75-\mxsh,-\mysh) -- (-\esh-0.15-\mxsh,-\mysh)
node [near start,below] {$\scriptstyle \xca$};
\draw [thkln] (0.15-\esh-\mxsh,-\mysh-0.2) -- (-0.15-\mxsh,-\mysh-0.2)
node [midway, below] {$\scriptstyle \xca-1$};
\begin{scope}[xscale=-1]
\draw [ptzert] (-0.15-\mxsh,-0.6-\mysh) rectangle ++(.3,1.2);
\draw [thkln] (-0.75-\mxsh,-\mysh) -- (-0.15-\mxsh,-\mysh)
node [near start,below] {$\scriptstyle \xca$};
\end{scope}
\begin{scope}[yscale=-1]
\draw [ptzert] (-0.15-\mxsh,-0.6-\mysh) rectangle ++(.3,1.2);
\draw [thkln] (-\esh-0.75-\mxsh,-\mysh) -- (-\esh-0.15-\mxsh,-\mysh)
node [near start,above] {$\scriptstyle \xca$};
\draw [thkln] (0.15-\esh-\mxsh,-\mysh-0.2) -- (-0.15-\mxsh,-\mysh-0.2)
node [midway, above] {$\scriptstyle \xca-1$};
\end{scope}
\begin{scope}[yscale=-1,xscale=-1]
\draw [ptzert] (-0.15-\mxsh,-0.6-\mysh) rectangle ++(.3,1.2);
\draw [thkln] (-0.75-\mxsh,-\mysh) -- (-0.15-\mxsh,-\mysh)
node [near start,above] {$\scriptstyle \xca$};
\end{scope}
\draw [ptzert] (-\esh-0.15-\mxsh,-0.6-\mysh) rectangle ++(.3,1.2);
\draw [ptzert] (-\esh-0.15-\mxsh,0.6+\mysh) rectangle ++(.3,-1.2);
\draw (-\esh-\mxsh+0.15,\mysh-0.2) to [out=0,in=90] (-\esh-\mxsh+1.1,0) to [out=-90,in=0] (-\esh-\mxsh+0.15,-\mysh+0.2);
\end{tikzpicture}
\;\hteqv\;
\begin{tikzpicture}[menvone]
%\draw [bcrc] (-0.75-\mxsh,-\mysh) -- (0.75+\mxsh,-\mysh);
%\draw [bcrc] (-0.75-\mxsh,\mysh) -- (0.75+\mxsh,\mysh);
\draw [thkln] (-\mxsh+0.15,\mysh-0.2) to [out=0,in=90] (-\mxsh+1.1,0)
% node [near end,right] {$\scriptstyle i$}
 to [out=-90,in=0] (-\mxsh+0.15,-\mysh+0.2) ;
\begin{scope}[xscale=-1]
\draw [thkln] (-\mxsh+0.15,\mysh-0.2) to [out=0,in=90] (-\mxsh+1.1,0) to [out=-90,in=0] (-\mxsh+0.15,-\mysh+0.2);
\end{scope}
\node at (-\mxsh+0.5,0) {$\scriptstyle \xki$};
\node at (\mxsh-0.5,0) {$\scriptstyle \xki$};
\draw [thkln] (-\mxsh+0.15,\mysh+0.2) -- (\mxsh-0.15,\mysh+0.2) node [midway,above] {$\scriptstyle \xca-\xki$};
\draw [thkln] (-\mxsh+0.15,-\mysh-0.2) -- (\mxsh-0.15,-\mysh-0.2) node [midway, below] {$\scriptstyle \xca-\xki$};
\draw [ptzert] (-0.15-\mxsh,-0.6-\mysh) rectangle ++(.3,1.2);
\draw [thkln] (-0.75-\mxsh,-\mysh) -- (-0.15-\mxsh,-\mysh)
node [near start,below] {$\scriptstyle \xca$};
\begin{scope}[xscale=-1]
\draw [ptzert] (-0.15-\mxsh,-0.6-\mysh) rectangle ++(.3,1.2);
\draw [thkln] (-0.75-\mxsh,-\mysh) -- (-0.15-\mxsh,-\mysh)
node [near start,below] {$\scriptstyle \xca$};
\end{scope}
\begin{scope}[yscale=-1]
\draw [ptzert] (-0.15-\mxsh,-0.6-\mysh) rectangle ++(.3,1.2);
\draw [thkln] (-0.75-\mxsh,-\mysh) -- (-0.15-\mxsh,-\mysh)
node [near start,above] {$\scriptstyle \xca$};
\end{scope}
\begin{scope}[yscale=-1,xscale=-1]
\draw [ptzert] (-0.15-\mxsh,-0.6-\mysh) rectangle ++(.3,1.2);
\draw [thkln] (-0.75-\mxsh,-\mysh) -- (-0.15-\mxsh,-\mysh)
node [near start,above] {$\scriptstyle \xca$};
\end{scope}
\end{tikzpicture}
\]
The \lumps\ presentation\rx{eq:stbx} of $\ytngsc$ allows us to use Remark\rw{rmk:bndss}: it is sufficient to establish a bound
\begin{equation}
\label{eq:bdck}
\tKHmvv{i}{\hem}(\xDsci) = 0,\quad\text{if $i\leq \xca-\xki $},
\end{equation}
because the tangle $\ytngsci$ has an extra \thdgr\ shift $\shcr^{\xki^2}$ in the \lumps\ formula\rx{eq:stbx} and $\xki^2-\xki\geq 0$.

Consider the portion of $\xDsci$ between the left $\xki$-cable of $\ytngsci$ and another crossing which connects the same \tBcr s:
\def\mxash{6}
\def\mxbsh{1}
\def\myash{2}
\begin{equation}
\label{eq:ordg}
\def\mxsh{3}
\def\mysh{1.6}
\begin{tikzpicture}[menvone]
\draw [bcrc] (-0.75-\mxsh-0.3-\mxash,-\mysh) -- (\mxash+1.6,-\mysh);
\draw [bcrc] (-0.75-\mxsh-0.3-\mxash,\mysh) -- (\mxash+1.6,\mysh);
\draw [thkln] (-\mxbsh+0.15,\mysh+\myash) to [out=0,in=180] (\mxbsh-0.15,\mysh);
\draw [lnovr,middle segment=0.3cm] (-\mxbsh+0.15,\mysh) to [out=0,in=180] (\mxbsh-0.15,\mysh+\myash);
\draw [thkln] (-\mxbsh+0.15,\mysh) to [out=0,in=180] (\mxbsh-0.15,\mysh+\myash);
\draw [ptzert] (-\mxbsh-0.15,\mysh-0.6) rectangle ++(.3,1.2);
\draw [ptzert] (\mxbsh-0.15,\mysh-0.6) rectangle ++(.3,1.2);
\begin{scope}[yshift=\myash cm]
\draw [ptzert] (-\mxbsh-0.15,\mysh-0.6) rectangle ++(.3,1.2);
\draw [ptzert] (\mxbsh-0.15,\mysh-0.6) rectangle ++(.3,1.2);
\draw [thkln] (-0.75-\mxbsh,\mysh) -- (-0.15-\mxbsh,\mysh) node [near start,above] {$\scriptstyle \xca$};
\draw [thkln] (0.75+\mxbsh,\mysh) -- (0.15+\mxbsh,\mysh) node [near start,above] {$\scriptstyle \xca$};
\end{scope}
%%%%%%%%%%%%%%%%%%%%%%%%%%
\begin{scope}[xshift=-\mxash cm-0.3cm]
\draw [thkln] (-\mxsh+0.15,\mysh) to [out=0,in=180] (0,-\mysh);
\draw [lnovr,middle segment=0.3cm] (-\mxsh+0.15,-\mysh) to [out=0,in=180] (0,\mysh);
\draw [thkln] (-\mxsh+0.15,-\mysh) to [out=0,in=180] (0,\mysh);
\draw [ptzert] (-0.15-\mxsh,-0.6-\mysh) rectangle ++(.3,1.2);
\draw [ptzert] (0,-0.6-\mysh) rectangle ++(.3,1.2);
\draw [ptzert] (0,-0.6+\mysh) rectangle ++(.3,1.2);
\draw [thkln] (-0.75-\mxsh,-\mysh) -- (-0.15-\mxsh,-\mysh)
node [near start,below] {$\scriptstyle \xca$};
\begin{scope}[yscale=-1]
\draw [ptzert] (-0.15-\mxsh,-0.6-\mysh) rectangle ++(.3,1.2);
\draw [thkln] (-0.75-\mxsh,-\mysh) -- (-0.15-\mxsh,-\mysh)
node [near start,above] {$\scriptstyle \xca$};
\end{scope}
\end{scope}
%%%%%%%%%%%%%%%
\begin{scope}[xshift=\mxash cm]
\draw [ptzert] (0,-0.6-\mysh) rectangle ++(.3,1.2);
\draw [ptzert] (0,-0.6+\mysh) rectangle ++(.3,1.2);
\draw [thkln] (0.3,\mysh-0.2) to [out=0,in=90] (1,0) to [out=-90,in=0] (0.3,-\mysh+0.2);
\node at (1.5,0) {$\scriptstyle \xki$};
\draw [thkln] (0.3,\mysh+0.2) -- (1.6,\mysh+0.2) node [near end, above] {$\scriptstyle \xca - \xki $};
\draw [thkln] (0.3,-\mysh-0.2) -- (1.6,-\mysh-0.2) node [near end, below] {$\scriptstyle \xca - \xki $};
\end{scope}
\draw [thkln] (-\mxash,-\mysh) -- (-0.7,-\mysh) (0.7,-\mysh) -- (\mxash,-\mysh);
\node at (0,-\mysh) {$\cdots$};
\draw [thkln] (-\mxash,\mysh) -- (-0.5*\mxash-0.5*\mxbsh-0.7,\mysh)
(-0.5*\mxash-0.5*\mxbsh+0.7,\mysh) -- (-\mxbsh-0.15,\mysh);
\node at (-0.5*\mxash-0.5*\mxbsh,\mysh) {$\cdots$ };
\draw [thkln] (\mxash,\mysh) -- (0.5*\mxash+0.5*\mxbsh+0.7,\mysh)
(0.5*\mxash+0.5*\mxbsh-0.7,\mysh) -- (\mxbsh+0.15,\mysh);
\node at (0.5*\mxash+0.5*\mxbsh,\mysh) {$\cdots$ };
\end{tikzpicture}
\end{equation}
As usual, gray strips indicate \tBcr s of the \tBdg. We showed explicitly one of the crossings attached to a \tBcr s. Consider a modification of this diagram which results from a repeated application of
of Propositions\rw{prp:bndfs} and\rw{prp:bndsc} to these crossings:
\begin{equation}
\label{eq:mddg}
\def\mxsh{3}
\def\mysh{1.6}
\def\smshy{0.35}
\begin{tikzpicture}[menvone]
\draw [bcrc] (-0.75-\mxsh-0.3-\mxash,-\mysh) -- (\mxash+1.6,-\mysh);
\draw [bcrc] (-0.75-\mxsh-0.3-\mxash,\mysh) -- (\mxash+1.6,\mysh);
%%%%%%%%%%%%%%
\begin{scope}[yshift=\smshy cm+0.3 cm]
\draw [thkln] (-\mxbsh+0.15,\mysh+\myash-0.2) to [out=0,in=180] (\mxbsh-0.15,\mysh);
\draw [lnovr,middle segment=0.3cm] (-\mxbsh+0.15,\mysh) to [out=0,in=180] (\mxbsh-0.15,\mysh+\myash-0.2);
\draw [thkln] (-\mxbsh+0.15,\mysh) to [out=0,in=180] (\mxbsh-0.15,\mysh+\myash-0.2);
\draw [thkln] (-\mxbsh+0.15,\mysh+\myash+0.2) -- (\mxbsh-0.15,\mysh+\myash+0.2)
node [midway,above] {$\scriptstyle \xki$};
\draw [ptzert] (-\mxbsh-0.15,\mysh-0.6) rectangle ++(.3,1.2);
\draw [ptzert] (\mxbsh-0.15,\mysh-0.6) rectangle ++(.3,1.2);
%\end{scope}
\begin{scope}[yshift=\myash cm]
\draw [ptzert] (-\mxbsh-0.15,\mysh-0.6) rectangle ++(.3,1.2);
\draw [ptzert] (\mxbsh-0.15,\mysh-0.6) rectangle ++(.3,1.2);
\draw [thkln] (-0.75-\mxbsh,\mysh) -- (-0.15-\mxbsh,\mysh) node [near start,above] {$\scriptstyle \xca$};
\draw [thkln] (0.75+\mxbsh,\mysh) -- (0.15+\mxbsh,\mysh) node [near start,above] {$\scriptstyle \xca$};
\end{scope}
\end{scope}
%%%%%%%%%%%%%%%%%%%%%%%%%%
\begin{scope}[xshift=-\mxash cm-0.3cm]
\draw [thkln] (-\mxsh+0.15,\mysh) to [out=0,in=180] (0,-\mysh);
\draw [lnovr,middle segment=0.3cm] (-\mxsh+0.15,-\mysh) to [out=0,in=180] (0,\mysh);
\draw [thkln] (-\mxsh+0.15,-\mysh) to [out=0,in=180] (0,\mysh);
\draw [ptzert] (-0.15-\mxsh,-0.6-\mysh) rectangle ++(.3,1.2);
\draw [ptzert] (0,-0.6-\mysh) rectangle ++(.3,1.2);
\draw [ptzert] (0,-0.6+\mysh) rectangle ++(.3,1.2);
\draw [thkln] (-0.75-\mxsh,-\mysh) -- (-0.15-\mxsh,-\mysh)
node [near start,below] {$\scriptstyle \xca$};
\begin{scope}[yscale=-1]
\draw [ptzert] (-0.15-\mxsh,-0.6-\mysh) rectangle ++(.3,1.2);
\draw [thkln] (-0.75-\mxsh,-\mysh) -- (-0.15-\mxsh,-\mysh)
node [near start,above] {$\scriptstyle \xca$};
\end{scope}
\end{scope}
%%%%%%%%%%%%%%%
\begin{scope}[xshift=\mxash cm]
\draw [ptzert] (0,-0.4-\mysh-\smshy) rectangle ++(.3,0.8);
\draw [ptzert] (0,-0.4+\mysh+\smshy) rectangle ++(.3,0.8);
\draw [thkln] (0.3,\mysh-\smshy) to [out=0,in=90] (1,0) to [out=-90,in=0] (0.3,-\mysh+\smshy);
\node at (1.5,0) {$\scriptstyle \xki$};
\draw [thkln] (0.3,\mysh+\smshy) -- (1.6,\mysh+\smshy) node [near end, above] {$\scriptstyle \xca - \xki $};
\draw [thkln] (0.3,-\mysh-\smshy) -- (1.6,-\mysh-\smshy) node [near end, below] {$\scriptstyle \xca - \xki $};
\end{scope}
\begin{scope}[yshift=-\smshy cm]
\draw [thkln] (-\mxash,-\mysh) -- (-0.7,-\mysh) (0.7,-\mysh) -- (\mxash,-\mysh);
\node at (0,-\mysh) {$\cdots$};
\end{scope}
\draw [thkln] (-\mxash,-\mysh+\smshy) -- (0.3+\mxash,-\mysh+\smshy);;
\begin{scope}[yshift=\smshy cm]
\draw [thkln] (-\mxash,\mysh) -- (-0.5*\mxash-0.5*\mxbsh-0.7,\mysh)
(-0.5*\mxash-0.5*\mxbsh+0.7,\mysh) -- (-\mxbsh-0.15,\mysh);
\node at (-0.5*\mxash-0.5*\mxbsh,\mysh) {$\cdots$ };
\draw [thkln] (\mxash,\mysh) -- (0.5*\mxash+0.5*\mxbsh+0.7,\mysh)
(0.5*\mxash+0.5*\mxbsh-0.7,\mysh) -- (\mxbsh+0.15,\mysh);
\node at (0.5*\mxash+0.5*\mxbsh,\mysh) {$\cdots$ };
\end{scope}
\draw [thkln] (-\mxash,\mysh-\smshy) -- (\mxash+0.3,\mysh-\smshy);
\end{tikzpicture}
\end{equation}
Let $\xDscip$ be the diagram constructed from $\xDsci$ by replacing the subdiagram\rx{eq:ordg} with\rx{eq:mddg}. According to Propositions\rw{prp:bndfs} and\rw{prp:bndsc}
% detaching the single line of $\ytngsc$ all the way to the projectors of the second crossings:
%A repeated application of Propositions\rw{prp:bndfs} and\rw{prp:bndsc} produces
there is a map of shifted Khovanov homologies $\tKHm(\xDscip)\rightarrow \tKHm(\xDsci)$, which is an isomorphism on $\tKHmvv{i}{\hem}(\xDsci)$ for $i\leq \xca-\xki$. We are going to show that
\begin{equation}
\label{eq:bdckb}
\tKHmvv{i}{\hem}(\xDscip) = 0\quad \text{if $i\leq \xca-\xki $},
\end{equation}
hence this will imply the bound\rx{eq:bdck}.

Consider a sequence of homotopy equivalences:
\begin{equation}
\label{eq:xthm}
\begin{tikzpicture}[menvone]
\begin{scope}[xshift=-\mxsh cm]
\draw [thkln] (-\mxsh+0.15,\mysh) to [out=0,in=180] (0,-\mysh);
\draw [lnovr] (-\mxsh+0.15,-\mysh) to [out=0,in=180] (0,\mysh);
\draw [thkln] (-\mxsh+0.15,-\mysh) to [out=0,in=180] (0,\mysh);
\draw [ptzer] (-0.15-\mxsh,-0.6-\mysh) rectangle ++(.3,1.2);
\draw [thkln] (-0.75-\mxsh,-\mysh) -- (-0.15-\mxsh,-\mysh)
node [near start,below] {$\scriptstyle \xca$};
\begin{scope}[yscale=-1]
\draw [ptzer] (-0.15-\mxsh,-0.6-\mysh) rectangle ++(.3,1.2);
\draw [thkln] (-0.75-\mxsh,-\mysh) -- (-0.15-\mxsh,-\mysh)
node [near start,above] {$\scriptstyle \xca$};
\end{scope}
\end{scope}
\draw [thkln] (-\mxsh+0.15,\mysh-0.2) to [out=0,in=90] (-\mxsh+1.1,0)
 to [out=-90,in=0] (-\mxsh+0.15,-\mysh+0.2);
%\begin{scope}[xscale=-1]
%\draw [thkln] (-\mxsh+0.15,\mysh-0.2) to [out=0,in=90] (-\mxsh+1.1,0) to [out=-90,in=0] (-\mxsh+0.15,-\mysh+0.2);
%\end{scope}
\node at (-\mxsh+1.5,0) {$\scriptstyle \xki$};
%\node at (\mxsh-0.5,0) {$\scriptstyle \xki$};
\draw [thkln] (-\mxsh+0.15,\mysh+0.2) -- (0.3,\mysh+0.2) node [midway,above] {$\scriptstyle \xca-\xki$};
\draw [thkln] (-\mxsh+0.15,-\mysh-0.2) -- (0.3,-\mysh-0.2) node [midway, below] {$\scriptstyle \xca-\xki$};
\draw [ptzert] (-0.15-\mxsh,-0.6-\mysh) rectangle ++(.3,1.2);
%\draw [thkln] (-0.75-\mxsh,-\mysh) -- (-0.15-\mxsh,-\mysh)
%node [near start,below] {$\scriptstyle \xca$};
%\begin{scope}[xscale=-1]
%\draw [ptzert] (-0.15-\mxsh,-0.6-\mysh) rectangle ++(.3,1.2);
%\draw [thkln] (-0.75-\mxsh,-\mysh) -- (-0.15-\mxsh,-\mysh)
%node [near start,below] {$\scriptstyle \xca$};
%\end{scope}
\begin{scope}[yscale=-1]
\draw [ptzert] (-0.15-\mxsh,-0.6-\mysh) rectangle ++(.3,1.2);
%\draw [thkln] (-0.75-\mxsh,-\mysh) -- (-0.15-\mxsh,-\mysh)
%node [near start,above] {$\scriptstyle \xca$};
\end{scope}
%\begin{scope}[yscale=-1,xscale=-1]
%\draw [ptzert] (-0.15-\mxsh,-0.6-\mysh) rectangle ++(.3,1.2);
%\draw [thkln] (-0.75-\mxsh,-\mysh) -- (-0.15-\mxsh,-\mysh)
%node [near start,above] {$\scriptstyle \xca$};
%\end{scope}
\end{tikzpicture}
\;\hteqv\;
\begin{tikzpicture}[menvone]
\begin{scope}[xshift=-\mxsh cm]
\draw [thkln] (-\mxsh+0.15,\mysh-0.2) to [out=0,in=180] (0.15+0.75,-\mysh-0.2);
%\draw [lnovrtw] (-\mxsh+0.15,-\mysh+0.2) to [out=0,in=180] (0.15+0.75,\mysh+0.2);
%\draw [thkln] (-\mxsh+0.15,-\mysh+0.2) to [out=0,in=180] (0.15+0.75,\mysh+0.2);
%%%%%%%
\draw [thkln,yshift=0.4cm] (-\mxsh+0.15,\mysh-0.2) to [out=0,in=180] (0.15+0.75,-\mysh-0.2+0.6);
\draw [lnovrtw,yshift=-0.4cm] (-\mxsh+0.15,-\mysh+0.2) to [out=0,in=180] (0.15+0.75,\mysh+0.2-0.6);
\draw [thkln,yshift=-0.4cm] (-\mxsh+0.15,-\mysh+0.2) to [out=0,in=180] (0.15+0.75,\mysh+0.2-0.6);
\draw [lnovrtw] (-\mxsh+0.15,-\mysh+0.2) to [out=0,in=180] (0.15+0.75,\mysh+0.2);
\draw [thkln] (-\mxsh+0.15,-\mysh+0.2) to [out=0,in=180] (0.15+0.75,\mysh+0.2);
\draw [ptzer] (-0.15-\mxsh,-0.6-\mysh) rectangle ++(.3,1.2);
\draw [thkln] (-0.75-\mxsh,-\mysh) -- (-0.15-\mxsh,-\mysh)
node [near start,below] {$\scriptstyle \xca$};
\begin{scope}[yscale=-1]
\draw [ptzer] (-0.15-\mxsh,-0.6-\mysh) rectangle ++(.3,1.2);
\draw [thkln] (-0.75-\mxsh,-\mysh) -- (-0.15-\mxsh,-\mysh)
node [near start,above] {$\scriptstyle \xca$};
\end{scope}
\end{scope}
%%%%%%%%%%%%%%%
\draw [thkln,xshift=0.75cm] (-\mxsh+0.15,\mysh-0.2-0.6) to [out=0,in=90] (-\mxsh+0.5,0)
 to [out=-90,in=0] (-\mxsh+0.15,-\mysh+0.2+0.6) ;
%\begin{scope}[xscale=-1]
%\draw [thkln] (-\mxsh+0.15,\mysh-0.2) to [out=0,in=90] (-\mxsh+1.1,0) to [out=-90,in=0] (-\mxsh+0.15,-\mysh+0.2);
%\end{scope}
\node at (-\mxsh+1.65,0) {$\scriptstyle \xki$};
%\node at (\mxsh-0.65,0) {$\scriptstyle \xki$};
\draw [thkln] (-\mxsh+0.15+0.75,\mysh+0.2) -- (0.3,\mysh+0.2) node [midway,above] {$\scriptstyle \xca-\xki$};
\draw [thkln] (-\mxsh+0.15+0.75,-\mysh-0.2) -- (0.3,-\mysh-0.2) node [midway, below] {$\scriptstyle \xca-\xki$};
%
%\draw [ptzert] (-0.15-\mxsh,-0.6-\mysh) rectangle ++(.3,1.2);
%\draw [thkln] (-0.75-\mxsh,-\mysh) -- (-0.15-\mxsh,-\mysh)
%node [near start,below] {$\scriptstyle \xca$};
%\begin{scope}[xscale=-1]
%\draw [ptzert] (-0.15-\mxsh,-0.6-\mysh) rectangle ++(.3,1.2);
%\draw [thkln] (-0.75-\mxsh,-\mysh) -- (-0.15-\mxsh,-\mysh)
%node [near start,below] {$\scriptstyle \xca$};
%\end{scope}
\begin{scope}[yscale=-1]
%\draw [ptzert] (-0.15-\mxsh,-0.6-\mysh) rectangle ++(.3,1.2);
%\draw [thkln] (-0.75-\mxsh,-\mysh) -- (-0.15-\mxsh,-\mysh)
%node [near start,above] {$\scriptstyle \xca$};
\end{scope}
%\begin{scope}[yscale=-1,xscale=-1]
%\draw [ptzert] (-0.15-\mxsh,-0.6-\mysh) rectangle ++(.3,1.2);
%\draw [thkln] (-0.75-\mxsh,-\mysh) -- (-0.15-\mxsh,-\mysh)
%node [near start,above] {$\scriptstyle \xca$};
%\end{scope}
\end{tikzpicture}
\;
\hteqv
\;
\shcr^{\xca \xki - \shlf \xki^2}\shfr^\xki
\begin{tikzpicture}[menvone]
\draw [thkln] (-\mxsh+0.15,\mysh-0.2) to [out=0,in=90] (-\mxsh+0.8,0)
% node [near end,right] {$\scriptstyle i$}
 to [out=-90,in=0] (-\mxsh+0.15,-\mysh+0.2) ;
%\begin{scope}[xscale=-1]
%\draw [thkln] (-\mxsh+0.15,\mysh-0.2) to [out=0,in=90] (-\mxsh+0.8,0) to [out=-90,in=0] (-\mxsh+0.15,-\mysh+0.2);
%\end{scope}
\node at (-\mxsh+0.4,0) {$\scriptstyle \xki$};
%\node at (\mxsh-0.4,0) {$\scriptstyle \xki$};
\draw [thkln] (-\mxsh+0.15,\mysh+0.2) to [out=0,in=180]
%node [sloped, near start,above] {$\scriptstyle \xca-\xki$}
(1.2,-\mysh);
\draw [thkln] (1.2,-\mysh) -- (2,-\mysh)
node [near start,below] {$\scriptstyle \xca-\xki$};
\draw [lnovr] (-\mxsh+0.15,-\mysh-0.2) to [out=0,in=180] (1.2,\mysh);
\draw [thkln] (-\mxsh+0.15,-\mysh-0.2) to [out=0,in=180]
%node [sloped,near start,below] {$\scriptstyle \xca-\xki$}
(1.2,\mysh);
\draw [thkln] (1.2,\mysh) -- (2,\mysh)
node [near start,above] {$\scriptstyle \xca-\xki$};
%\draw [thkln] (-\mxsh+0.15,\mysh+0.2) -- (\mxsh-0.15,\mysh+0.2) node [midway,above] {$\scriptstyle \xca-i$};
%\draw [thkln] (-\mxsh+0.15,-\mysh-0.2) -- (\mxsh-0.15,-\mysh-0.2) node [midway, below] {$\scriptstyle \xca-i$};
%
\draw [ptzert] (-0.15-\mxsh,-0.6-\mysh) rectangle ++(.3,1.2);
\draw [thkln] (-0.75-\mxsh,-\mysh) -- (-0.15-\mxsh,-\mysh)
node [near start,below] {$\scriptstyle \xca$};
%\begin{scope}[xscale=-1]
%\draw [ptzert] (-0.15-\mxsh,-0.6-\mysh) rectangle ++(.3,1.2);
%\draw [thkln] (-0.75-\mxsh,-\mysh) -- (-0.15-\mxsh,-\mysh)
%node [near start,below] {$\scriptstyle \xca$};
%\end{scope}
\begin{scope}[yscale=-1]
\draw [ptzert] (-0.15-\mxsh,-0.6-\mysh) rectangle ++(.3,1.2);
\draw [thkln] (-0.75-\mxsh,-\mysh) -- (-0.15-\mxsh,-\mysh)
node [near start,above] {$\scriptstyle \xca$};
\end{scope}
%\begin{scope}[yscale=-1,xscale=-1]
%\draw [ptzert] (-0.15-\mxsh,-0.6-\mysh) rectangle ++(.3,1.2);
%\draw [thkln] (-0.75-\mxsh,-\mysh) -- (-0.15-\mxsh,-\mysh)
%node [near start,above] {$\scriptstyle \xca$};
%\end{scope}
\end{tikzpicture}
\end{equation}
Here the first homotopy equivalence comes from sliding $\xki$-cable projectors to the left along $\xca$-cables and then contracting double projectors into single ones, while the second homotopy comes from \eex{eq:cfrsh} and\rx{eq:projtw}. Let $\xDscipp$ be the diagram constructed from $\xDscip$ by replacing the left tangle of \ex{eq:xthm} with the right tangle. Since $\yncrv{\xDscipp} = \yncrv{\xDscip} + (\xca-\xki)^2 - \xca^2$, homotopy equivalence\rx{eq:xthm} implies the isomorphism of shifted Khovanov homologies
\[
\tKHm(\xDscip) = \shcr^{\xki(2\xca-\xki)}
%\shcr^{\xca \xki - \shlf \xki^2}
\shfr^\xki\,\tKHm(\xDscipp)
\]
and the bound\rx{eq:bdckb} follows from Theorem\rw{thm:eaest}.
\end{proof}

\appendix
\section{A single crossing of colored strands approximates a \tJWp}

Let $\xD$ be a diagram which may include both single and cabled lines as well as \tJWp s. Suppose that $\xD$ has a crossing of two $\xca$-cables with a projector on each. Let $\xDp$ be a diagram, in which the crossing is replaced with a \tJWp:
\def\mxsh{2}
\def\mysh{1.2}
\[
\begin{tikzpicture}[menvthree]
\node (l) at (-1.8,0) {};
\node (r) at (1.8,0) {};
\path[commutative diagrams/.cd, every arrow, every label]
(l) edge[commutative diagrams/squiggly] (r);
\begin{scope}[xshift=-5.5cm]
\def\mxsh{1.7}
\draw [thkln] (-\mxsh+0.15,\mysh) to [out=0,in=180] (\mxsh-0.15,-\mysh);
\draw [lnovr] (-\mxsh+0.15,-\mysh) to [out=0,in=180] (\mxsh-0.15,\mysh);
\draw [thkln] (-\mxsh+0.15,-\mysh) to [out=0,in=180] (\mxsh-0.15,\mysh);
%
%\draw [lnovr] (-0.75-\mxsh,0) to [out=0,in=180] (0,-\mysh) to [out=0,in=180] (0.75+\mxsh,0);
%\draw (-0.75-\mxsh,0) to [out=0,in=180] (0,-\mysh) to [out=0,in=180] (0.75+\mxsh,0);
%
\draw [ptzer] (-0.15-\mxsh,-0.6-\mysh) rectangle ++(.3,1.2);
\draw [thkln] (-0.75-\mxsh,-\mysh) -- (-0.15-\mxsh,-\mysh)
node [near start,below] {$\scriptstyle \xca$};
\begin{scope}[xscale=-1]
\draw [ptzer] (-0.15-\mxsh,-0.6-\mysh) rectangle ++(.3,1.2);
\draw [thkln] (-0.75-\mxsh,-\mysh) -- (-0.15-\mxsh,-\mysh)
node [near start,below] {$\scriptstyle \xca$};
\end{scope}
\begin{scope}[yscale=-1]
\draw [ptzer] (-0.15-\mxsh,-0.6-\mysh) rectangle ++(.3,1.2);
\draw [thkln] (-0.75-\mxsh,-\mysh) -- (-0.15-\mxsh,-\mysh)
node [near start,above] {$\scriptstyle \xca$};
\end{scope}
\begin{scope}[yscale=-1,xscale=-1]
\draw [ptzer] (-0.15-\mxsh,-0.6-\mysh) rectangle ++(.3,1.2);
\draw [thkln] (-0.75-\mxsh,-\mysh) -- (-0.15-\mxsh,-\mysh)
node [near start,above] {$\scriptstyle \xca$};
\end{scope}
\end{scope}
\begin{scope}[xshift=5.5cm]
%\draw [thkln] (-\mxsh+0.15,\mysh) to [out=0,in=180] (\mxsh-0.15,-\mysh);
%\draw [lnovr] (-\mxsh+0.15,-\mysh) to [out=0,in=180] (\mxsh-0.15,\mysh);
%\draw [thkln] (-\mxsh+0.15,-\mysh) to [out=0,in=180] (\mxsh-0.15,\mysh);
\draw [thkln] (-\mxsh+0.15,\mysh) -- (-\mxsh*0.3,\mysh) (\mxsh*0.3,\mysh) -- (\mxsh-0.15,\mysh);
\draw [thkln] (-\mxsh+0.15,-\mysh) -- (-\mxsh*0.3,-\mysh) (\mxsh*0.3,-\mysh) -- (\mxsh-0.15,-\mysh);
%
%\draw [lnovr] (-0.75-\mxsh,0) to [out=0,in=180] (0,-\mysh) to [out=0,in=180] (0.75+\mxsh,0);
%\draw (-0.75-\mxsh,0) to [out=0,in=180] (0,-\mysh) to [out=0,in=180] (0.75+\mxsh,0);
%
\draw [ptzer] (-0.15-\mxsh,-0.6-\mysh) rectangle ++(.3,1.2);
\draw [thkln] (-0.75-\mxsh,-\mysh) -- (-0.15-\mxsh,-\mysh)
node [near start,below] {$\scriptstyle \xca$};
\begin{scope}[xscale=-1]
\draw [ptzer] (-0.15-\mxsh,-0.6-\mysh) rectangle ++(.3,1.2);
\draw [thkln] (-0.75-\mxsh,-\mysh) -- (-0.15-\mxsh,-\mysh)
node [near start,below] {$\scriptstyle \xca$};
\end{scope}
\begin{scope}[yscale=-1]
\draw [ptzer] (-0.15-\mxsh,-0.6-\mysh) rectangle ++(.3,1.2);
\draw [thkln] (-0.75-\mxsh,-\mysh) -- (-0.15-\mxsh,-\mysh)
node [near start,above] {$\scriptstyle \xca$};
\end{scope}
\begin{scope}[yscale=-1,xscale=-1]
\draw [ptzer] (-0.15-\mxsh,-0.6-\mysh) rectangle ++(.3,1.2);
\draw [thkln] (-0.75-\mxsh,-\mysh) -- (-0.15-\mxsh,-\mysh)
node [near start,above] {$\scriptstyle \xca$};
\end{scope}
\draw[ptzer] (-\mxsh*0.3,-\mysh*1-0.6) rectangle ++(\mxsh*0.6,\mysh*2+1.2);
\end{scope}
\end{tikzpicture}
\]

\begin{theorem}
\label{thm:crpr}
There exists a map $\tKHm(\xD)\xrightarrow{\xmg}\tKHm(\xDp)$ which is an isomorphism on $\tKHmvv{i}{\hem}$ for $i\leq 2\xca-2$.
\end{theorem}
\begin{proof}
Consider three tangles
\[
\ytngsi =
\begin{tikzpicture}[menvthree]
\draw [thkln] (-\mxsh+0.15,\mysh) -- (-\mxsh*0.3,\mysh) (\mxsh*0.3,\mysh) -- (\mxsh-0.15,\mysh);
\draw [thkln] (-\mxsh+0.15,-\mysh) -- (-\mxsh*0.3,-\mysh) (\mxsh*0.3,-\mysh) -- (\mxsh-0.15,-\mysh);
\draw [ptzer] (-0.15-\mxsh,-0.6-\mysh) rectangle ++(.3,1.2);
\draw [thkln] (-0.75-\mxsh,-\mysh) -- (-0.15-\mxsh,-\mysh)
node [near start,below] {$\scriptstyle \xca$};
\begin{scope}[xscale=-1]
\draw [ptzer] (-0.15-\mxsh,-0.6-\mysh) rectangle ++(.3,1.2);
\draw [thkln] (-0.75-\mxsh,-\mysh) -- (-0.15-\mxsh,-\mysh)
node [near start,below] {$\scriptstyle \xca$};
\end{scope}
\begin{scope}[yscale=-1]
\draw [ptzer] (-0.15-\mxsh,-0.6-\mysh) rectangle ++(.3,1.2);
\draw [thkln] (-0.75-\mxsh,-\mysh) -- (-0.15-\mxsh,-\mysh)
node [near start,above] {$\scriptstyle \xca$};
\end{scope}
\begin{scope}[yscale=-1,xscale=-1]
\draw [ptzer] (-0.15-\mxsh,-0.6-\mysh) rectangle ++(.3,1.2);
\draw [thkln] (-0.75-\mxsh,-\mysh) -- (-0.15-\mxsh,-\mysh)
node [near start,above] {$\scriptstyle \xca$};
\end{scope}
\draw[ptzer] (-\mxsh*0.3,-\mysh*1-0.6) rectangle ++(\mxsh*0.6,\mysh*2+1.2);
\end{tikzpicture}
,\qquad
\def\mxsh{1.7}
\ytngsf =
\begin{tikzpicture}[menvthree]
\draw [thkln] (-\mxsh+0.15,\mysh) to [out=0,in=180] (\mxsh-0.15,-\mysh);
\draw [lnovr] (-\mxsh+0.15,-\mysh) to [out=0,in=180] (\mxsh-0.15,\mysh);
\draw [thkln] (-\mxsh+0.15,-\mysh) to [out=0,in=180] (\mxsh-0.15,\mysh);
%
%\draw [lnovr] (-0.75-\mxsh,0) to [out=0,in=180] (0,-\mysh) to [out=0,in=180] (0.75+\mxsh,0);
%\draw (-0.75-\mxsh,0) to [out=0,in=180] (0,-\mysh) to [out=0,in=180] (0.75+\mxsh,0);
%
\draw [ptzer] (-0.15-\mxsh,-0.6-\mysh) rectangle ++(.3,1.2);
\draw [thkln] (-0.75-\mxsh,-\mysh) -- (-0.15-\mxsh,-\mysh)
node [near start,below] {$\scriptstyle \xca$};
\begin{scope}[xscale=-1]
\draw [ptzer] (-0.15-\mxsh,-0.6-\mysh) rectangle ++(.3,1.2);
\draw [thkln] (-0.75-\mxsh,-\mysh) -- (-0.15-\mxsh,-\mysh)
node [near start,below] {$\scriptstyle \xca$};
\end{scope}
\begin{scope}[yscale=-1]
\draw [ptzer] (-0.15-\mxsh,-0.6-\mysh) rectangle ++(.3,1.2);
\draw [thkln] (-0.75-\mxsh,-\mysh) -- (-0.15-\mxsh,-\mysh)
node [near start,above] {$\scriptstyle \xca$};
\end{scope}
\begin{scope}[yscale=-1,xscale=-1]
\draw [ptzer] (-0.15-\mxsh,-0.6-\mysh) rectangle ++(.3,1.2);
\draw [thkln] (-0.75-\mxsh,-\mysh) -- (-0.15-\mxsh,-\mysh)
node [near start,above] {$\scriptstyle \xca$};
\end{scope}
\end{tikzpicture}
,\qquad
\def\mxsh{2}
\ytngsc= \shcr^{\hlf\xca^2}
\begin{tikzpicture}[menvthree]
\draw [thkln] (-\mxsh+0.15,\mysh) -- (-\mxsh*0.3,\mysh) (\mxsh*0.3,\mysh) -- (\mxsh-0.15,\mysh);
\draw [thkln] (-\mxsh+0.15,-\mysh) -- (-\mxsh*0.3,-\mysh) (\mxsh*0.3,-\mysh) -- (\mxsh-0.15,-\mysh);
%%%%
\def\xmlf{2.5}
\draw [thkln] (-0.15-\mxsh,-\mysh) to [out=180, in=0] (-\xmlf*\mxsh+0.15,\mysh);
\draw [lnovr] (-0.15-\mxsh,\mysh) to [out=180, in=0] (-\xmlf*\mxsh+0.15,-\mysh);
\draw [thkln] (-0.15-\mxsh,\mysh) to [out=180, in=0] (-\xmlf*\mxsh+0.15,-\mysh);
\draw [ptzer] (-0.15-\xmlf*\mxsh,-0.6-\mysh) rectangle ++(.3,1.2);
\draw [thkln] (-0.75-\xmlf*\mxsh,-\mysh) -- (-0.15-\xmlf*\mxsh,-\mysh)
node [near start,below] {$\scriptstyle \xca$};
\begin{scope}[yscale=-1]
\draw [ptzer] (-0.15-\xmlf*\mxsh,-0.6-\mysh) rectangle ++(.3,1.2);
\draw [thkln] (-0.75-\xmlf*\mxsh,-\mysh) -- (-0.15-\xmlf*\mxsh,-\mysh)
node [near start,above] {$\scriptstyle \xca$};
\end{scope}
%%%%
\draw [ptzer] (-0.15-\mxsh,-0.6-\mysh) rectangle ++(.3,1.2);
%\draw [thkln] (-0.75-\mxsh,-\mysh) -- (-0.15-\mxsh,-\mysh);
%node [near start,below] {$\scriptstyle \xca$};
\begin{scope}[xscale=-1]
\draw [ptzer] (-0.15-\mxsh,-0.6-\mysh) rectangle ++(.3,1.2);
\draw [thkln] (-0.75-\mxsh,-\mysh) -- (-0.15-\mxsh,-\mysh)
node [near start,below] {$\scriptstyle \xca$};
\end{scope}
\begin{scope}[yscale=-1]
\draw [ptzer] (-0.15-\mxsh,-0.6-\mysh) rectangle ++(.3,1.2);
%\draw [thkln] (-0.75-\mxsh,-\mysh) -- (-0.15-\mxsh,-\mysh);
%node [near start,above] {$\scriptstyle \xca$};
\end{scope}
\begin{scope}[yscale=-1,xscale=-1]
\draw [ptzer] (-0.15-\mxsh,-0.6-\mysh) rectangle ++(.3,1.2);
\draw [thkln] (-0.75-\mxsh,-\mysh) -- (-0.15-\mxsh,-\mysh)
node [near start,above] {$\scriptstyle \xca$};
\end{scope}
\draw[ptone] (-\mxsh*0.3,-\mysh*1-0.6) rectangle ++(\mxsh*0.6,\mysh*2+1.2);
\end{tikzpicture}\;,
\]
where the complex
$
\begin{tikzpicture}[scale=0.4,baseline=-2.5]
\draw[line width=\ljwp,fill=gray!30] (-0.15,-0.6) rectangle ++(0.3,1.2);
\draw [line width=\cblth] (-.65,0) -- (-0.15,0) node [near start, above] {$\scriptstyle a$}
(0.15,0) -- (.65,0);
\end{tikzpicture}
$
is defined by \ex{eq:grproj}, and set $\xKhv{\ytngsfp}=\shcr^{\hlf\xca^2}\xKhv{\ytngsf}$. These tangles have a relation\rx{eq:cnrel} which comes from a sequence of homotopy equivalences:
\[
\begin{tikzpicture}[menvthree]
\draw [thkln] (-\mxsh+0.15,\mysh) -- (-\mxsh*0.3,\mysh) (\mxsh*0.3,\mysh) -- (\mxsh-0.15,\mysh);
\draw [thkln] (-\mxsh+0.15,-\mysh) -- (-\mxsh*0.3,-\mysh) (\mxsh*0.3,-\mysh) -- (\mxsh-0.15,-\mysh);
\draw [ptzer] (-0.15-\mxsh,-0.6-\mysh) rectangle ++(.3,1.2);
\draw [thkln] (-0.75-\mxsh,-\mysh) -- (-0.15-\mxsh,-\mysh)
node [near start,below] {$\scriptstyle \xca$};
\begin{scope}[xscale=-1]
\draw [ptzer] (-0.15-\mxsh,-0.6-\mysh) rectangle ++(.3,1.2);
\draw [thkln] (-0.75-\mxsh,-\mysh) -- (-0.15-\mxsh,-\mysh)
node [near start,below] {$\scriptstyle \xca$};
\end{scope}
\begin{scope}[yscale=-1]
\draw [ptzer] (-0.15-\mxsh,-0.6-\mysh) rectangle ++(.3,1.2);
\draw [thkln] (-0.75-\mxsh,-\mysh) -- (-0.15-\mxsh,-\mysh)
node [near start,above] {$\scriptstyle \xca$};
\end{scope}
\begin{scope}[yscale=-1,xscale=-1]
\draw [ptzer] (-0.15-\mxsh,-0.6-\mysh) rectangle ++(.3,1.2);
\draw [thkln] (-0.75-\mxsh,-\mysh) -- (-0.15-\mxsh,-\mysh)
node [near start,above] {$\scriptstyle \xca$};
\end{scope}
\draw[ptzer] (-\mxsh*0.3,-\mysh*1-0.6) rectangle ++(\mxsh*0.6,\mysh*2+1.2);
\end{tikzpicture}
\;\hteqv\;
\shcr^{\hlf\xca^2}
\begin{tikzpicture}[menvthree]
\draw [thkln] (-\mxsh+0.15,\mysh) -- (-\mxsh*0.3,\mysh) (\mxsh*0.3,\mysh) -- (\mxsh-0.15,\mysh);
\draw [thkln] (-\mxsh+0.15,-\mysh) -- (-\mxsh*0.3,-\mysh) (\mxsh*0.3,-\mysh) -- (\mxsh-0.15,-\mysh);
%%%%
\def\xmlf{2.5}
\draw [thkln] (-0.15-\mxsh,-\mysh) to [out=180, in=0] (-\xmlf*\mxsh+0.15,\mysh);
\draw [lnovr] (-0.15-\mxsh,\mysh) to [out=180, in=0] (-\xmlf*\mxsh+0.15,-\mysh);
\draw [thkln] (-0.15-\mxsh,\mysh) to [out=180, in=0] (-\xmlf*\mxsh+0.15,-\mysh);
\draw [ptzer] (-0.15-\xmlf*\mxsh,-0.6-\mysh) rectangle ++(.3,1.2);
\draw [thkln] (-0.75-\xmlf*\mxsh,-\mysh) -- (-0.15-\xmlf*\mxsh,-\mysh)
node [near start,below] {$\scriptstyle \xca$};
\begin{scope}[yscale=-1]
\draw [ptzer] (-0.15-\xmlf*\mxsh,-0.6-\mysh) rectangle ++(.3,1.2);
\draw [thkln] (-0.75-\xmlf*\mxsh,-\mysh) -- (-0.15-\xmlf*\mxsh,-\mysh)
node [near start,above] {$\scriptstyle \xca$};
\end{scope}
%%%%
\draw [ptzer] (-0.15-\mxsh,-0.6-\mysh) rectangle ++(.3,1.2);
%\draw [thkln] (-0.75-\mxsh,-\mysh) -- (-0.15-\mxsh,-\mysh);
%node [near start,below] {$\scriptstyle \xca$};
\begin{scope}[xscale=-1]
\draw [ptzer] (-0.15-\mxsh,-0.6-\mysh) rectangle ++(.3,1.2);
\draw [thkln] (-0.75-\mxsh,-\mysh) -- (-0.15-\mxsh,-\mysh)
node [near start,below] {$\scriptstyle \xca$};
\end{scope}
\begin{scope}[yscale=-1]
\draw [ptzer] (-0.15-\mxsh,-0.6-\mysh) rectangle ++(.3,1.2);
%\draw [thkln] (-0.75-\mxsh,-\mysh) -- (-0.15-\mxsh,-\mysh);
%node [near start,above] {$\scriptstyle \xca$};
\end{scope}
\begin{scope}[yscale=-1,xscale=-1]
\draw [ptzer] (-0.15-\mxsh,-0.6-\mysh) rectangle ++(.3,1.2);
\draw [thkln] (-0.75-\mxsh,-\mysh) -- (-0.15-\mxsh,-\mysh)
node [near start,above] {$\scriptstyle \xca$};
\end{scope}
\draw[ptzer] (-\mxsh*0.3,-\mysh*1-0.6) rectangle ++(\mxsh*0.6,\mysh*2+1.2);
\end{tikzpicture}
\;\hteqv\;
\shcr^{\hlf\xca^2}\;
\boxed{
\shcr
\begin{tikzpicture}[menvthree]
\draw [thkln] (-\mxsh+0.15,\mysh) -- (-\mxsh*0.3,\mysh) (\mxsh*0.3,\mysh) -- (\mxsh-0.15,\mysh);
\draw [thkln] (-\mxsh+0.15,-\mysh) -- (-\mxsh*0.3,-\mysh) (\mxsh*0.3,-\mysh) -- (\mxsh-0.15,-\mysh);
%%%%
\def\xmlf{2.5}
\draw [thkln] (-0.15-\mxsh,-\mysh) to [out=180, in=0] (-\xmlf*\mxsh+0.15,\mysh);
\draw [lnovr] (-0.15-\mxsh,\mysh) to [out=180, in=0] (-\xmlf*\mxsh+0.15,-\mysh);
\draw [thkln] (-0.15-\mxsh,\mysh) to [out=180, in=0] (-\xmlf*\mxsh+0.15,-\mysh);
\draw [ptzer] (-0.15-\xmlf*\mxsh,-0.6-\mysh) rectangle ++(.3,1.2);
\draw [thkln] (-0.75-\xmlf*\mxsh,-\mysh) -- (-0.15-\xmlf*\mxsh,-\mysh)
node [near start,below] {$\scriptstyle \xca$};
\begin{scope}[yscale=-1]
\draw [ptzer] (-0.15-\xmlf*\mxsh,-0.6-\mysh) rectangle ++(.3,1.2);
\draw [thkln] (-0.75-\xmlf*\mxsh,-\mysh) -- (-0.15-\xmlf*\mxsh,-\mysh)
node [near start,above] {$\scriptstyle \xca$};
\end{scope}
%%%%
\draw [ptzer] (-0.15-\mxsh,-0.6-\mysh) rectangle ++(.3,1.2);
%\draw [thkln] (-0.75-\mxsh,-\mysh) -- (-0.15-\mxsh,-\mysh);
%node [near start,below] {$\scriptstyle \xca$};
\begin{scope}[xscale=-1]
\draw [ptzer] (-0.15-\mxsh,-0.6-\mysh) rectangle ++(.3,1.2);
\draw [thkln] (-0.75-\mxsh,-\mysh) -- (-0.15-\mxsh,-\mysh)
node [near start,below] {$\scriptstyle \xca$};
\end{scope}
\begin{scope}[yscale=-1]
\draw [ptzer] (-0.15-\mxsh,-0.6-\mysh) rectangle ++(.3,1.2);
%\draw [thkln] (-0.75-\mxsh,-\mysh) -- (-0.15-\mxsh,-\mysh);
%node [near start,above] {$\scriptstyle \xca$};
\end{scope}
\begin{scope}[yscale=-1,xscale=-1]
\draw [ptzer] (-0.15-\mxsh,-0.6-\mysh) rectangle ++(.3,1.2);
\draw [thkln] (-0.75-\mxsh,-\mysh) -- (-0.15-\mxsh,-\mysh)
node [near start,above] {$\scriptstyle \xca$};
\end{scope}
\draw[ptone] (-\mxsh*0.3,-\mysh*1-0.6) rectangle ++(\mxsh*0.6,\mysh*2+1.2);
\end{tikzpicture}
\longrightarrow
\begin{tikzpicture}[menvthree]
\def\mxsh{1.7}
\draw [thkln] (-\mxsh+0.15,\mysh) to [out=0,in=180] (\mxsh-0.15,-\mysh);
\draw [lnovr] (-\mxsh+0.15,-\mysh) to [out=0,in=180] (\mxsh-0.15,\mysh);
\draw [thkln] (-\mxsh+0.15,-\mysh) to [out=0,in=180] (\mxsh-0.15,\mysh);
%
%\draw [lnovr] (-0.75-\mxsh,0) to [out=0,in=180] (0,-\mysh) to [out=0,in=180] (0.75+\mxsh,0);
%\draw (-0.75-\mxsh,0) to [out=0,in=180] (0,-\mysh) to [out=0,in=180] (0.75+\mxsh,0);
%
\draw [ptzer] (-0.15-\mxsh,-0.6-\mysh) rectangle ++(.3,1.2);
\draw [thkln] (-0.75-\mxsh,-\mysh) -- (-0.15-\mxsh,-\mysh)
node [near start,below] {$\scriptstyle \xca$};
\begin{scope}[xscale=-1]
\draw [ptzer] (-0.15-\mxsh,-0.6-\mysh) rectangle ++(.3,1.2);
\draw [thkln] (-0.75-\mxsh,-\mysh) -- (-0.15-\mxsh,-\mysh)
node [near start,below] {$\scriptstyle \xca$};
\end{scope}
\begin{scope}[yscale=-1]
\draw [ptzer] (-0.15-\mxsh,-0.6-\mysh) rectangle ++(.3,1.2);
\draw [thkln] (-0.75-\mxsh,-\mysh) -- (-0.15-\mxsh,-\mysh)
node [near start,above] {$\scriptstyle \xca$};
\end{scope}
\begin{scope}[yscale=-1,xscale=-1]
\draw [ptzer] (-0.15-\mxsh,-0.6-\mysh) rectangle ++(.3,1.2);
\draw [thkln] (-0.75-\mxsh,-\mysh) -- (-0.15-\mxsh,-\mysh)
node [near start,above] {$\scriptstyle \xca$};
\end{scope}
\end{tikzpicture}
}
\]
Here the first homotopy equivalence comes from \ex{eq:projtw}, while the second equivalence comes from \ex{eq:projcn}.

In order to put a bound on the homological order of $\xDsc$, we purge the gray box in $\ytngsc$, that is, we contract all constituent \taTLt s, whose cup or cap is connected directly to \tJWp s sitting on $\xca$-cables. After the purge, the complex of $\ytngsc$ takes the form
\begin{eqnarray*}
\ytngsc &\hteqv&
\boxed{\cdots\longrightarrow \bigoplus_{\substack{0\leq\zmj\leq \zmi \\ \xki\geq 1} }
\prmltijk\,\shcr^\zmi \shfr^\zmj
\begin{tikzpicture}[menvthree]
\begin{scope}[xshift=-\mxsh cm]
\draw [thkln] (-\mxsh+0.15,\mysh) to [out=0,in=180] (0,-\mysh);
\draw [lnovr] (-\mxsh+0.15,-\mysh) to [out=0,in=180] (0,\mysh);
\draw [thkln] (-\mxsh+0.15,-\mysh) to [out=0,in=180] (0,\mysh);
\draw [ptzer] (-0.15-\mxsh,-0.6-\mysh) rectangle ++(.3,1.2);
\draw [thkln] (-0.75-\mxsh,-\mysh) -- (-0.15-\mxsh,-\mysh)
node [near start,below] {$\scriptstyle \xca$};
\begin{scope}[yscale=-1]
\draw [ptzer] (-0.15-\mxsh,-0.6-\mysh) rectangle ++(.3,1.2);
\draw [thkln] (-0.75-\mxsh,-\mysh) -- (-0.15-\mxsh,-\mysh)
node [near start,above] {$\scriptstyle \xca$};
\end{scope}
\end{scope}
\draw [thkln] (-\mxsh+0.15,\mysh-0.2) to [out=0,in=90] (-\mxsh+1.1,0)
 to [out=-90,in=0] (-\mxsh+0.15,-\mysh+0.2) ;
\begin{scope}[xscale=-1]
\draw [thkln] (-\mxsh+0.15,\mysh-0.2) to [out=0,in=90] (-\mxsh+1.1,0) to [out=-90,in=0] (-\mxsh+0.15,-\mysh+0.2);
\end{scope}
\node at (-\mxsh+0.5,0) {$\scriptstyle \xki$};
\node at (\mxsh-0.5,0) {$\scriptstyle \xki$};
\draw [thkln] (-\mxsh+0.15,\mysh+0.2) -- (\mxsh-0.15,\mysh+0.2) node [midway,above] {$\scriptstyle \xca-\xki$};
\draw [thkln] (-\mxsh+0.15,-\mysh-0.2) -- (\mxsh-0.15,-\mysh-0.2) node [midway, below] {$\scriptstyle \xca-\xki$};
\draw [ptzert] (-0.15-\mxsh,-0.6-\mysh) rectangle ++(.3,1.2);
%\draw [thkln] (-0.75-\mxsh,-\mysh) -- (-0.15-\mxsh,-\mysh)
%node [near start,below] {$\scriptstyle \xca$};
\begin{scope}[xscale=-1]
\draw [ptzert] (-0.15-\mxsh,-0.6-\mysh) rectangle ++(.3,1.2);
\draw [thkln] (-0.75-\mxsh,-\mysh) -- (-0.15-\mxsh,-\mysh)
node [near start,below] {$\scriptstyle \xca$};
\end{scope}
\begin{scope}[yscale=-1]
\draw [ptzert] (-0.15-\mxsh,-0.6-\mysh) rectangle ++(.3,1.2);
%\draw [thkln] (-0.75-\mxsh,-\mysh) -- (-0.15-\mxsh,-\mysh)
%node [near start,above] {$\scriptstyle \xca$};
\end{scope}
\begin{scope}[yscale=-1,xscale=-1]
\draw [ptzert] (-0.15-\mxsh,-0.6-\mysh) rectangle ++(.3,1.2);
\draw [thkln] (-0.75-\mxsh,-\mysh) -- (-0.15-\mxsh,-\mysh)
node [near start,above] {$\scriptstyle \xca$};
\end{scope}
\end{tikzpicture}
\longrightarrow\cdots
}_{\;\xki=1}^{\;\infty}
\\
&\hteqv&
\def\mxsh{3}
\boxed{
\cdots\longrightarrow \bigoplus_{\substack{0\leq\zmj\leq \zmi \\ \xki\geq 1} }
\prmltijk\,\shcr^{\zmi+\xki(\xca-\xki)+\hlf\xki^2} \shfr^{\zmj+\xki}
\begin{tikzpicture}[menvthree]
\draw [thkln] (-\mxsh+0.15,\mysh-0.2) to [out=0,in=90] (-\mxsh+0.8,0)
% node [near end,right] {$\scriptstyle i$}
 to [out=-90,in=0] (-\mxsh+0.15,-\mysh+0.2) ;
\begin{scope}[xscale=-1]
\draw [thkln] (-\mxsh+0.15,\mysh-0.2) to [out=0,in=90] (-\mxsh+0.8,0) to [out=-90,in=0] (-\mxsh+0.15,-\mysh+0.2);
\end{scope}
\node at (-\mxsh+0.4,0) {$\scriptstyle \xki$};
\node at (\mxsh-0.4,0) {$\scriptstyle \xki$};
\draw [thkln] (-\mxsh+0.15,\mysh+0.2) to [out=0,in=180]
node [sloped, near start,above] {$\scriptstyle \xca-\xki$} (\mxsh-0.15,-\mysh-0.2);
\draw [lnovr] (-\mxsh+0.15,-\mysh-0.2) to [out=0,in=180] (\mxsh-0.15,\mysh+0.2);
\draw [thkln] (-\mxsh+0.15,-\mysh-0.2) to [out=0,in=180]
node [sloped,near start,below] {$\scriptstyle \xca-\xki$}
(\mxsh-0.15,\mysh+0.2);
%\draw [thkln] (-\mxsh+0.15,\mysh+0.2) -- (\mxsh-0.15,\mysh+0.2) node [midway,above] {$\scriptstyle \xca-i$};
%\draw [thkln] (-\mxsh+0.15,-\mysh-0.2) -- (\mxsh-0.15,-\mysh-0.2) node [midway, below] {$\scriptstyle \xca-i$};
%
\draw [ptzert] (-0.15-\mxsh,-0.6-\mysh) rectangle ++(.3,1.2);
\draw [thkln] (-0.75-\mxsh,-\mysh) -- (-0.15-\mxsh,-\mysh)
node [near start,below] {$\scriptstyle \xca$};
\begin{scope}[xscale=-1]
\draw [ptzert] (-0.15-\mxsh,-0.6-\mysh) rectangle ++(.3,1.2);
\draw [thkln] (-0.75-\mxsh,-\mysh) -- (-0.15-\mxsh,-\mysh)
node [near start,below] {$\scriptstyle \xca$};
\end{scope}
\begin{scope}[yscale=-1]
\draw [ptzert] (-0.15-\mxsh,-0.6-\mysh) rectangle ++(.3,1.2);
\draw [thkln] (-0.75-\mxsh,-\mysh) -- (-0.15-\mxsh,-\mysh)
node [near start,above] {$\scriptstyle \xca$};
\end{scope}
\begin{scope}[yscale=-1,xscale=-1]
\draw [ptzert] (-0.15-\mxsh,-0.6-\mysh) rectangle ++(.3,1.2);
\draw [thkln] (-0.75-\mxsh,-\mysh) -- (-0.15-\mxsh,-\mysh)
node [near start,above] {$\scriptstyle \xca$};
\end{scope}
\end{tikzpicture}
\longrightarrow\cdots
}_{\;\xki=1}^{\;\infty}
\end{eqnarray*}
We used homotopy equivalence\rx{eq:xthm}.
Note that there are no tangles with $k=0$, because the complex\rx{eq:grproj} does not contain identity braids.

Let $\xDsci$ be the diagram $\xDsc$ in which the complex $\ytngsc$ is replaced by the tangle diagram
\[
\def\mxsh{3}
\ytngsci \;=\;
\begin{tikzpicture}[menvthree]
\draw [thkln] (-\mxsh+0.15,\mysh-0.2) to [out=0,in=90] (-\mxsh+0.8,0)
% node [near end,right] {$\scriptstyle i$}
 to [out=-90,in=0] (-\mxsh+0.15,-\mysh+0.2) ;
\begin{scope}[xscale=-1]
\draw [thkln] (-\mxsh+0.15,\mysh-0.2) to [out=0,in=90] (-\mxsh+0.8,0) to [out=-90,in=0] (-\mxsh+0.15,-\mysh+0.2);
\end{scope}
\node at (-\mxsh+0.4,0) {$\scriptstyle \xki$};
\node at (\mxsh-0.4,0) {$\scriptstyle \xki$};
\draw [thkln] (-\mxsh+0.15,\mysh+0.2) to [out=0,in=180]
node [sloped, near start,above] {$\scriptstyle \xca-\xki$} (\mxsh-0.15,-\mysh-0.2);
\draw [lnovr] (-\mxsh+0.15,-\mysh-0.2) to [out=0,in=180] (\mxsh-0.15,\mysh+0.2);
\draw [thkln] (-\mxsh+0.15,-\mysh-0.2) to [out=0,in=180]
node [sloped,near start,below] {$\scriptstyle \xca-\xki$}
(\mxsh-0.15,\mysh+0.2);
%\draw [thkln] (-\mxsh+0.15,\mysh+0.2) -- (\mxsh-0.15,\mysh+0.2) node [midway,above] {$\scriptstyle \xca-i$};
%\draw [thkln] (-\mxsh+0.15,-\mysh-0.2) -- (\mxsh-0.15,-\mysh-0.2) node [midway, below] {$\scriptstyle \xca-i$};
%
\draw [ptzert] (-0.15-\mxsh,-0.6-\mysh) rectangle ++(.3,1.2);
\draw [thkln] (-0.75-\mxsh,-\mysh) -- (-0.15-\mxsh,-\mysh)
node [near start,below] {$\scriptstyle \xca$};
\begin{scope}[xscale=-1]
\draw [ptzert] (-0.15-\mxsh,-0.6-\mysh) rectangle ++(.3,1.2);
\draw [thkln] (-0.75-\mxsh,-\mysh) -- (-0.15-\mxsh,-\mysh)
node [near start,below] {$\scriptstyle \xca$};
\end{scope}
\begin{scope}[yscale=-1]
\draw [ptzert] (-0.15-\mxsh,-0.6-\mysh) rectangle ++(.3,1.2);
\draw [thkln] (-0.75-\mxsh,-\mysh) -- (-0.15-\mxsh,-\mysh)
node [near start,above] {$\scriptstyle \xca$};
\end{scope}
\begin{scope}[yscale=-1,xscale=-1]
\draw [ptzert] (-0.15-\mxsh,-0.6-\mysh) rectangle ++(.3,1.2);
\draw [thkln] (-0.75-\mxsh,-\mysh) -- (-0.15-\mxsh,-\mysh)
node [near start,above] {$\scriptstyle \xca$};
\end{scope}
\end{tikzpicture}
\]
According to Theorem\rw{thm:smfr}, $\KHmvv{i}{\hem}(\xDsci)=0$ for $i\leq -\shlf (\xca-\xki)^2 -\shlf\yncrv{\xDsi}-1$, so, by Remark\rw{rmk:bndss}, $\KHmvv{i}{\hem}(\xDsc)=0$ for $i\leq %-\shlf\xca^2+
-\shlf\yncrv{\xDsi}+2\xca-2$ (here we used inequality $2\xca\xki - \xki^2\geq 2\xca-1$ for $\xki\geq 1$). Now the claim of Theorem\rw{thm:crpr} follows from \ex{eq:ineqo}.
%, since in our case $\hbnd=-\shlf\xca^2+2\xca-\shlf\yncrv{\xDsi}-1$ and $\xqshc=\shlf\xca^2 + 1$.
\end{proof}

\end{document}

Let $\xDsci$ be the diagram $\xDsc$ in which the \tmcn\rx{eq:stbx} is replaced by its $i$-th constituent diagram.
\begin{proposition}
\label{prp:smit}
There is a bound $\KHmvv{i}{\hem}(\xDsci)=0$ for $i\geq$.
\end{proposition}

\begin{proof}[Proof of Proposition\rw{prp:smdsc}]

\end{proof}
\begin{proof}[Proof of Proposition\rw{prp:smit}]
Consider a homotopy equivalence transformation of a constituent diagram of the \tmcn\rx{eq:stbx} composed with a crossing tangle on the left:
\[
\begin{tikzpicture}[menvone]
\begin{scope}[xshift=-\mxsh cm]
\draw [thkln] (-\mxsh+0.15,\mysh) to [out=0,in=180] (0,-\mysh);
\draw [lnovr] (-\mxsh+0.15,-\mysh) to [out=0,in=180] (0,\mysh);
\draw [thkln] (-\mxsh+0.15,-\mysh) to [out=0,in=180] (0,\mysh);
\draw [ptzer] (-0.15-\mxsh,-0.6-\mysh) rectangle ++(.3,1.2);
\draw [thkln] (-0.75-\mxsh,-\mysh) -- (-0.15-\mxsh,-\mysh)
node [near start,below] {$\scriptstyle \xca$};
\begin{scope}[yscale=-1]
\draw [ptzer] (-0.15-\mxsh,-0.6-\mysh) rectangle ++(.3,1.2);
\draw [thkln] (-0.75-\mxsh,-\mysh) -- (-0.15-\mxsh,-\mysh)
node [near start,above] {$\scriptstyle \xca$};
\end{scope}
\end{scope}
\draw [thkln] (-\mxsh+0.15,\mysh-0.2) to [out=0,in=90] (-\mxsh+1.1,0)
 to [out=-90,in=0] (-\mxsh+0.15,-\mysh+0.2) ;
\begin{scope}[xscale=-1]
\draw [thkln] (-\mxsh+0.15,\mysh-0.2) to [out=0,in=90] (-\mxsh+1.1,0) to [out=-90,in=0] (-\mxsh+0.15,-\mysh+0.2);
\end{scope}
\node at (-\mxsh+0.5,0) {$\scriptstyle i$};
\node at (\mxsh-0.5,0) {$\scriptstyle i$};
\draw [thkln] (-\mxsh+0.15,\mysh+0.2) -- (\mxsh-0.15,\mysh+0.2) node [midway,above] {$\scriptstyle \xca-i$};
\draw [thkln] (-\mxsh+0.15,-\mysh-0.2) -- (\mxsh-0.15,-\mysh-0.2) node [midway, below] {$\scriptstyle \xca-i$};
\draw [ptzert] (-0.15-\mxsh,-0.6-\mysh) rectangle ++(.3,1.2);
%\draw [thkln] (-0.75-\mxsh,-\mysh) -- (-0.15-\mxsh,-\mysh)
%node [near start,below] {$\scriptstyle \xca$};
\begin{scope}[xscale=-1]
\draw [ptzert] (-0.15-\mxsh,-0.6-\mysh) rectangle ++(.3,1.2);
\draw [thkln] (-0.75-\mxsh,-\mysh) -- (-0.15-\mxsh,-\mysh)
node [near start,below] {$\scriptstyle \xca$};
\end{scope}
\begin{scope}[yscale=-1]
\draw [ptzert] (-0.15-\mxsh,-0.6-\mysh) rectangle ++(.3,1.2);
%\draw [thkln] (-0.75-\mxsh,-\mysh) -- (-0.15-\mxsh,-\mysh)
%node [near start,above] {$\scriptstyle \xca$};
\end{scope}
\begin{scope}[yscale=-1,xscale=-1]
\draw [ptzert] (-0.15-\mxsh,-0.6-\mysh) rectangle ++(.3,1.2);
\draw [thkln] (-0.75-\mxsh,-\mysh) -- (-0.15-\mxsh,-\mysh)
node [near start,above] {$\scriptstyle \xca$};
\end{scope}
\end{tikzpicture}
\;\hteqv\;
\begin{tikzpicture}[menvone]
\begin{scope}[xshift=-\mxsh cm]
\draw [thkln] (-\mxsh+0.15,\mysh-0.2) to [out=0,in=180] (0.15+0.75,-\mysh-0.2);
%\draw [lnovrtw] (-\mxsh+0.15,-\mysh+0.2) to [out=0,in=180] (0.15+0.75,\mysh+0.2);
%\draw [thkln] (-\mxsh+0.15,-\mysh+0.2) to [out=0,in=180] (0.15+0.75,\mysh+0.2);
%%%%%%%
\draw [thkln,yshift=0.4cm] (-\mxsh+0.15,\mysh-0.2) to [out=0,in=180] (0.15+0.75,-\mysh-0.2+0.6);
\draw [lnovrtw,yshift=-0.4cm] (-\mxsh+0.15,-\mysh+0.2) to [out=0,in=180] (0.15+0.75,\mysh+0.2-0.6);
\draw [thkln,yshift=-0.4cm] (-\mxsh+0.15,-\mysh+0.2) to [out=0,in=180] (0.15+0.75,\mysh+0.2-0.6);
\draw [lnovrtw] (-\mxsh+0.15,-\mysh+0.2) to [out=0,in=180] (0.15+0.75,\mysh+0.2);
\draw [thkln] (-\mxsh+0.15,-\mysh+0.2) to [out=0,in=180] (0.15+0.75,\mysh+0.2);
\draw [ptzer] (-0.15-\mxsh,-0.6-\mysh) rectangle ++(.3,1.2);
\draw [thkln] (-0.75-\mxsh,-\mysh) -- (-0.15-\mxsh,-\mysh)
node [near start,below] {$\scriptstyle \xca$};
\begin{scope}[yscale=-1]
\draw [ptzer] (-0.15-\mxsh,-0.6-\mysh) rectangle ++(.3,1.2);
\draw [thkln] (-0.75-\mxsh,-\mysh) -- (-0.15-\mxsh,-\mysh)
node [near start,above] {$\scriptstyle \xca$};
\end{scope}
\end{scope}
%%%%%%%%%%%%%%%
\draw [thkln,xshift=0.75cm] (-\mxsh+0.15,\mysh-0.2-0.6) to [out=0,in=90] (-\mxsh+0.5,0)
 to [out=-90,in=0] (-\mxsh+0.15,-\mysh+0.2+0.6) ;
\begin{scope}[xscale=-1]
\draw [thkln] (-\mxsh+0.15,\mysh-0.2) to [out=0,in=90] (-\mxsh+1.1,0) to [out=-90,in=0] (-\mxsh+0.15,-\mysh+0.2);
\end{scope}
\node at (-\mxsh+1.65,0) {$\scriptstyle i$};
\node at (\mxsh-0.65,0) {$\scriptstyle i$};
\draw [thkln] (-\mxsh+0.15+0.75,\mysh+0.2) -- (\mxsh-0.15,\mysh+0.2) node [midway,above] {$\scriptstyle \xca-i$};
\draw [thkln] (-\mxsh+0.15+0.75,-\mysh-0.2) -- (\mxsh-0.15,-\mysh-0.2) node [midway, below] {$\scriptstyle \xca-i$};
%
%\draw [ptzert] (-0.15-\mxsh,-0.6-\mysh) rectangle ++(.3,1.2);
%\draw [thkln] (-0.75-\mxsh,-\mysh) -- (-0.15-\mxsh,-\mysh)
%node [near start,below] {$\scriptstyle \xca$};
\begin{scope}[xscale=-1]
\draw [ptzert] (-0.15-\mxsh,-0.6-\mysh) rectangle ++(.3,1.2);
\draw [thkln] (-0.75-\mxsh,-\mysh) -- (-0.15-\mxsh,-\mysh)
node [near start,below] {$\scriptstyle \xca$};
\end{scope}
\begin{scope}[yscale=-1]
%\draw [ptzert] (-0.15-\mxsh,-0.6-\mysh) rectangle ++(.3,1.2);
%\draw [thkln] (-0.75-\mxsh,-\mysh) -- (-0.15-\mxsh,-\mysh)
%node [near start,above] {$\scriptstyle \xca$};
\end{scope}
\begin{scope}[yscale=-1,xscale=-1]
\draw [ptzert] (-0.15-\mxsh,-0.6-\mysh) rectangle ++(.3,1.2);
\draw [thkln] (-0.75-\mxsh,-\mysh) -- (-0.15-\mxsh,-\mysh)
node [near start,above] {$\scriptstyle \xca$};
\end{scope}
\end{tikzpicture}
\;
\hteqv
\;
\shcr^{\xca i - \shlf i^2}\shfr^i
\begin{tikzpicture}[menvone]
\draw [thkln] (-\mxsh+0.15,\mysh-0.2) to [out=0,in=90] (-\mxsh+0.8,0)
% node [near end,right] {$\scriptstyle i$}
 to [out=-90,in=0] (-\mxsh+0.15,-\mysh+0.2) ;
\begin{scope}[xscale=-1]
\draw [thkln] (-\mxsh+0.15,\mysh-0.2) to [out=0,in=90] (-\mxsh+0.8,0) to [out=-90,in=0] (-\mxsh+0.15,-\mysh+0.2);
\end{scope}
\node at (-\mxsh+0.4,0) {$\scriptstyle i$};
\node at (\mxsh-0.4,0) {$\scriptstyle i$};
\draw [thkln] (-\mxsh+0.15,\mysh+0.2) to [out=0,in=180]
node [sloped, near start,above] {$\scriptstyle \xca-i$} (\mxsh-0.15,-\mysh-0.2);
\draw [lnovr] (-\mxsh+0.15,-\mysh-0.2) to [out=0,in=180] (\mxsh-0.15,\mysh+0.2);
\draw [thkln] (-\mxsh+0.15,-\mysh-0.2) to [out=0,in=180]
node [sloped,near start,below] {$\scriptstyle \xca-i$}
(\mxsh-0.15,\mysh+0.2);
%\draw [thkln] (-\mxsh+0.15,\mysh+0.2) -- (\mxsh-0.15,\mysh+0.2) node [midway,above] {$\scriptstyle \xca-i$};
%\draw [thkln] (-\mxsh+0.15,-\mysh-0.2) -- (\mxsh-0.15,-\mysh-0.2) node [midway, below] {$\scriptstyle \xca-i$};
%
\draw [ptzert] (-0.15-\mxsh,-0.6-\mysh) rectangle ++(.3,1.2);
\draw [thkln] (-0.75-\mxsh,-\mysh) -- (-0.15-\mxsh,-\mysh)
node [near start,below] {$\scriptstyle \xca$};
\begin{scope}[xscale=-1]
\draw [ptzert] (-0.15-\mxsh,-0.6-\mysh) rectangle ++(.3,1.2);
\draw [thkln] (-0.75-\mxsh,-\mysh) -- (-0.15-\mxsh,-\mysh)
node [near start,below] {$\scriptstyle \xca$};
\end{scope}
\begin{scope}[yscale=-1]
\draw [ptzert] (-0.15-\mxsh,-0.6-\mysh) rectangle ++(.3,1.2);
\draw [thkln] (-0.75-\mxsh,-\mysh) -- (-0.15-\mxsh,-\mysh)
node [near start,above] {$\scriptstyle \xca$};
\end{scope}
\begin{scope}[yscale=-1,xscale=-1]
\draw [ptzert] (-0.15-\mxsh,-0.6-\mysh) rectangle ++(.3,1.2);
\draw [thkln] (-0.75-\mxsh,-\mysh) -- (-0.15-\mxsh,-\mysh)
node [near start,above] {$\scriptstyle \xca$};
\end{scope}
\end{tikzpicture}
\]
Here the first homotopy equivalence comes from sliding the middle projectors to the left along $\xca$-cables and then contracting double projectors into single ones, while the second homotopy comes from \eex{eq:cfrsh} and\rx{eq:projtw}.
\end{proof}

\end{document}

In order to prove the remaining bound\rx{eq:bd4a} in the case of \tBadq\ diagram $\xD$, we observe that  apart from inequality it corresponds to the `equal' part of the bound\rx{eq:bd3a}, so relation $\xca \gvD+j =-i -\xabms$ holds only if there are equalities in the second inequality of\rx{eq:mbnd} and in the inequality \ex{eq:bndNt}:
\begin{equation}
\label{eq:threq}
\njcr = \xabms,\qquad \nscrb+\nwcrb=\xca,\qquad \xca_1+\xca_2=2\xca.
\end{equation}
Now consider the case of $\xabms>0$ and the case of $\xabms=0$. If $\xabms>0$, then the diagram $\xDs$ must have at least one \tstrtl. If that line is attached to a \tBcr\ $c$, then it is easy to see that in that circle either $\xca_1<\xca$ or $\xca_2<\xca$, and the third equation of\rx{eq:threq} can't hold. If $\xabms=0$, then $\spvr=0$ for all crossings $\xvrt$ of $\xD$, that is, $\xDs$ contains no \tstrtl s. This means that the diagram $\xDs$ consists of $\xca$-colored circles, each carrying one projector. Let us convert them into circle diagrams by using \eex{eq:projcn} and\rx{eq:grproj} for the projectors. A \twndg\ circle in a \tBcr\ $c$ contains at least two strands of the corresponding $\xca$-cable, so in this case $\nscrb + 2\nwcrb\leq\xca$, hence the second relation of\rx{eq:threq} may hold only if $\nwcrb=0$.
%
%
%\begin{equation}
%\label{eq:wcsc}
%\nscrb +
%\end{equation}
%
However, since $\xabms=0$, the condition $i+\xabms\neq 0$ of the bound\rx{eq:bd4a} implies $i>0$. A non-trivial \thdgr\ means that the circle diagrams contributing to it must have at least one \taTLt\ $\gamma$ from \ex{eq:grproj} from one of projectors. If this happens on a \tBadq\ circle $c$, then since $\gamma$ has at least one cap and one cup, there must be at least one \twndg\ circle, which contradicts $\nwcrb=0.$\qed
%\end{proof}

However, the constituent \tTLt s of \tJWp s on a \tBcr\ $c$ produce too many cups and caps, thus potentially creating too many \twndg\ circles in circle diagrams and violating the bound\rx{eq:wbd}.

In order to resolve this difficulty, we will contract the circle diagrams with too many cups and caps within the complex $\xKhv{\xDs}$. Consider a neighborhood of a particular \tBcr\ $c$. We are going to cut temporarily the \tstrtl s connecting it to the rest of the diagram, thus creating a colored tangle diagram $\xtusc$, and then perform some contractions within its Khovanov complex $\xKhv{\xtusc}$.

If the \tstt\ $\spmp$ was such that $c$ had no \tstrtl s attached to it, then we could simply use the projector property of \tJWp s and contract them all to a single projector. In the presence of \tstrtl s we, roughly speaking, contract only parts of the projectors which are connected directly to each other.

We use the following `purging' procedure. We select the initial projector on $c$, the preceding projector being called `final', and purge all other projectors on $c$ one-by-one going clockwise. Purging means that we consider the complex\rx{eq:projcn},\rx{eq:grproj} for the current projector, thus presenting $\xKhv{\xtusc}$ as a \tmcn\ of complexes, each coming from a particular \tTLt\ $\gamma$ in \ex{eq:grproj} (and the identity braid of \ex{eq:projcn}) , and contract all complexes, whose graph $\gamma$ has a cap connected directly to the initial projector or a cup connected directly to the next projector.
%%%%%%%%%%%%%%%%%%
%
%
%
%number $\njrc$ of \trlxc s within each \tBcr\ $c$ by $\xca$.  The problem with this estimate is that constituent \tTLt s of \ex{eq:grproj} contain cups and caps in addition to \txstrs s. Hence \trlxc s are of two types. \Tstrght\ circles go strictly along a \tBcr\ passing straight through all of its \tJWp s. Their number is obviously bounded by $\xca$. But then there are also \twndgc s which change their direction a few times by using cups and caps of the \tTLt s $\gamma$ of \ex{eq:grproj}. It is easy to show that every extra pair of cup-cap within those tangles comes with a shift functor $\shcr$, but the corresponding bound is insufficient to prove the bound of Theorem\rw{thm:kqbnd}. Hence we have to use a more delicate argument: we will use other projectors on the same \tBcr\ in order to contract constituent \tTLt s of the current projector which may create \twndgc s.
%
%Consider a neighborhood of a \tBcr\ $\scs$ within a diagram $\xDs$. By cutting its connecting \tstrtl s we produce a diagram $\xDscr$. In order to simplify its complex $\xKhv{\xDscr}$, we perform the following `purging' procedure by going clockwise over all \tJWp s, except the initial and the final ones. We apply \eex{eq:projcn} and\rx{eq:grproj} to the current projector and then remove (contract) all diagrams, whose graph $\gamma$ has a cup or a cap, which is attached by lines going along \tBcr\ to the initial or the next projector.
We claim that after we perform this procedure on all projectors, except the initial and final ones, we get a \tmcn
\begin{equation}
\label{eq:prcmp}
\xKhv{\xtusc} \hteqv
\boxed{
\cdots\longrightarrow\shcr^i
\bigoplus_{0\leq j\leq i}\shfr^j\left( \bigoplus_{\tau} m_{ij,\tau} \xKhv{\tau}\right)
%\bigoplus_{\substack{0\le j \le i \\ \gamma\in\sTLab,\; \wdv{\gamma}=b}} %m_{ij,\gamma}\,\shfr^j\,\xKhv{\gamma}
\longrightarrow\cdots
}_{\;i=0}^{\;\infty}
\end{equation}
where the diagrams $\tau$ are of one of two types:
%%%%%%%%%%%%%%%%%%%%%%%%%%%%%%%%
%%%%%%%%%%%%%%%%%%%%%%%%%%%%%%%%
\begin{equation}
\label{eq:twpdgs}
\begin{tikzpicture}[menvtwo]
\draw [bcrct] (-1.25,0) -- (1.25,0) (1.55,0) to [out=0,in=90] (4,-2) to [out=-90,in=0] (0,-5)
(-1.55,0) to [out=180,in=90] (-4,-2) to [out=-90,in=180] (0,-5) ;
\draw [thkln] (-1.25,0) -- (1.25,0) (1.55,0) to [out=0,in=90] (4,-2) to [out=-90,in=0] (0,-5)
(-1.55,0) to [out=180,in=90] (-4,-2) to [out=-90,in=180] (0,-5) ;
\node at (0,-5.5) {$\scriptstyle  \xca_1$};
\node at (0,0.5) {$\scriptstyle \xca_2$};
\draw [ptzert] (-1.55,-0.6) rectangle ++(0.3,1.2);
\draw [ptzert] (1.55,-0.6) rectangle ++(-0.3,1.2);
%\draw [ptzert] (4.05,-0.6) rectangle ++(0.3,1.2);
%%%%
\draw [ptzer] (0.8,2) rectangle ++(-1.6,-0.8);
\draw [thkc] (0.2,0.1+1.6) arc (0:-180:0.2);
\draw [thkln]  (-0.5,2) -- (-0.5,2.6);
\draw [thkln]  (0.5,2) -- (0.5,2.6);
%\node at (0.05,2.3) {$\cdots$};
\node at (0.05,2.4) {$\scriptstyle \cdots$};
\draw [ptzer] (0.8,-2) rectangle ++(-1.6,0.8);
\draw [thkc] (0.2,-0.1-1.6) arc (0:180:0.2);
\draw [thkln]  (-0.5,-2) -- (-0.5,-2.6);
\draw [thkln]  (0.5,-2) -- (0.5,-2.6);
%\node at (0.05,-2.3) {$\cdots$};
\node at (0.05,-2.4) {$\scriptstyle \cdots$};
\draw [thkln] (-1.25,0.35) to [out=0,in=-90] (-0.5,1.2);
\draw [thkln,xscale=-1] (-1.25,0.35) to [out=0,in=-90] (-0.5,1.2);
\draw [thkln,yscale=-1] (-1.25,0.35) to [out=0,in=-90] (-0.5,1.2);
\draw [thkln,xscale=-1,yscale=-1] (-1.25,0.35) to [out=0,in=-90] (-0.5,1.2);
\draw [thkln] (-1.55,0.35) to [out=180,in=-90] ++(-0.75,0.85) -- (-2.3,2.6);
\draw [thkln,xscale=-1] (-1.55,0.35) to [out=180,in=-90] ++(-0.75,0.85) -- (-2.3,2.6);
%\draw [thkln,xshift=5.6cm] (-1.55,0.35) to [out=180,in=-90] ++(-0.75,0.85) -- (-2.3,2.6);
%\draw [thkln,xscale=-1,yscale=-1,xshift=-2.8cm] (-1.55,0.35) to [out=180,in=-90] ++(-0.75,0.85) -- (-2.3,2.6);
\draw [decorate,decoration={brace,amplitude=4pt},xshift=0,yshift=-5pt]
(0.5cm+5pt,-2.6) -- (-0.5cm-5pt,-2.6)
node [black,midway,yshift=-0.4cm]
{\scriptsize to struts};
\draw [decorate,decoration={brace,amplitude=4pt},xshift=0,yshift=5pt]
(-2.3cm-5pt,2.6) -- (2.3cm+5pt,2.6)
node [black,midway,yshift=0.4cm]
{\scriptsize to struts};
\end{tikzpicture}
\qquad
\qquad
\begin{tikzpicture}[menvtwo]
\draw [bcrct] (-1.25,0) -- (1.25,0) (1.55,0) to [out=0,in=90] (4,-2) to [out=-90,in=0] (0,-5)
(-1.55,0) to [out=180,in=90] (-4,-2) to [out=-90,in=180] (0,-5) ;
\draw [thkln] (0,1.2) -- (0,-1.2) (1.55,0) to [out=0,in=90] (4,-2) to [out=-90,in=0] (0,-5)
(-1.55,0) to [out=180,in=90] (-4,-2) to [out=-90,in=180] (0,-5) ;
\node at (0,-5.5) {$\scriptstyle  \xca_1$};
\draw [ptzert] (-1.55,-0.6) rectangle ++(0.3,1.2);
\draw [ptzert] (1.55,-0.6) rectangle ++(-0.3,1.2);
%\draw [ptzert] (4.05,-0.6) rectangle ++(0.3,1.2);
%%%%
\draw [ptzer] (0.8,2) rectangle ++(-1.6,-0.8);
\draw [thkc] (0.2,0.1+1.6) arc (0:-180:0.2);
\draw [thkln]  (-0.5,2) -- (-0.5,2.6);
\draw [thkln]  (0.5,2) -- (0.5,2.6);
%\node at (0.05,2.3) {$\cdots$};
\node at (0.05,2.4) {$\scriptstyle \cdots$};
\draw [ptzer] (0.8,-2) rectangle ++(-1.6,0.8);
\draw [thkc] (0.2,-0.1-1.6) arc (0:180:0.2);
\draw [thkln]  (-0.5,-2) -- (-0.5,-2.6);
\draw [thkln]  (0.5,-2) -- (0.5,-2.6);
%\node at (0.05,-2.3) {$\cdots$};
\node at (0.05,-2.4) {$\scriptstyle \cdots$};
\draw [thkln] (-1.25,0.35) to [out=0,in=-90] (-0.5,1.2);
\draw [thkln,xscale=-1] (-1.25,0.35) to [out=0,in=-90] (-0.5,1.2);
\draw [thkln,yscale=-1] (-1.25,0.35) to [out=0,in=-90] (-0.5,1.2);
\draw [thkln,xscale=-1,yscale=-1] (-1.25,0.35) to [out=0,in=-90] (-0.5,1.2);
\draw [thkln] (-1.55,0.35) to [out=180,in=-90] ++(-0.75,0.85) -- (-2.3,2.6);
\draw [thkln,xscale=-1] (-1.55,0.35) to [out=180,in=-90] ++(-0.75,0.85) -- (-2.3,2.6);
%
%
%\draw [thkln,xshift=5.6cm] (-1.55,0.35) to [out=180,in=-90] ++(-0.75,0.85) -- (-2.3,2.6);
%\draw [thkln,xscale=-1,yscale=-1,xshift=-2.8cm] (-1.55,0.35) to [out=180,in=-90] ++(-0.75,0.85) -- %(-2.3,2.6);
\draw [decorate,decoration={brace,amplitude=4pt},xshift=0,yshift=-5pt]
(0.5cm+5pt,-2.6) -- (-0.5cm-5pt,-2.6)
node [black,midway,yshift=-0.4cm]
{\scriptsize to struts};
\draw [decorate,decoration={brace,amplitude=4pt},xshift=0,yshift=5pt]
(-2.3cm-5pt,2.6) -- (2.3cm+5pt,2.6)
node [black,midway,yshift=0.4cm]
{\scriptsize to struts};
\end{tikzpicture}
\end{equation}
%%%%%%%%%%%%%%%%%%%%%%%%%%%%%%%%
with $\xca_1,\xca_2\leq \xca$.
A box with an arc denotes a `\ntwd' \tTLt: for $a \geq b$,
\[
\swdv{
\begin{tikzpicture}[menvone]
\draw [ptzer] (0.4,0.8) rectangle ++(-0.8,-1.6);
\draw [thkc] (-0.1,-0.2) arc (-90:90:0.2);
\draw [thkln] (-1,0) -- (-0.4,0) node [near start,below] {$\scriptstyle a$}
(0.4,0) -- (1,0) node [near end,below] {$\scriptstyle b$};
\end{tikzpicture}
} = b.
\]
In other words, a tangle
$
\begin{tikzpicture}[menvthree,rotate=90]
\draw [ptzer] (0.4,0.8) rectangle ++(-0.8,-1.6);
\draw [thkc] (-0.1,-0.2) arc (-90:90:0.2);
\draw [thkln] (-1,0) -- (-0.4,0)  % node [near start,below] {$\scriptstyle a$}
(0.4,0) -- (1,0);  %node [near end,below] {$\scriptstyle b$};
\end{tikzpicture}
$
contains no cups, but only caps and \txstrs s.

Establishing a bound $\njrc\leq\xca$ within diagrams\rx{eq:twpdgs} is easy.  In the first diagram a \tstrghtc\ contains one strand from the $\xca_1$-cable and one strand from the $\xca_2$-cable, while a \twndgc\ contains at least two strands from one of these cables, hence the total number of \trlxc s
has a bound $\njrc\leq\hlf(\xca_1+\xca_2)\leq\xca$. The second diagram contains only \twndgc s, each of them contains at least two strands of the $\xca_1$-cable, hence there $\njrc\leq\hlf\xca_1\leq\xca$.

It remains to show that the purging process ends up with the complex\rx{eq:prcmp} generated by diagrams\rx{eq:twpdgs} with non-negative shifts of \tqdgr. We prove inductively that after we purged projectors between the initial and the current one, we get a \tmcn\
\begin{equation*}
\xKhv{\xDscr} \hteqv
\boxed{
\cdots\longrightarrow\shcr^i
\bigoplus_{0\leq j\leq i}\shfr^j\left( \bigoplus_{\tau} m'_{ij,\tau} \xKhv{\tau}\right)
%\bigoplus_{\substack{0\le j \le i \\ \gamma\in\sTLab,\; \wdv{\gamma}=b}} %m_{ij,\gamma}\,\shfr^j\,\xKhv{\gamma}
\longrightarrow\cdots
}_{\;i=0}^{\;\infty}
\end{equation*}
whose constituent diagrams $\tau$ have one of two possible forms between the initial and current projector (which lies to the left of the dashed line):
\[
\begin{tikzpicture}[menvtwo]
\draw [bcrct] (-1.55-0.6-0.5,0) -- (-1.55,0) (-1.25,0) -- (1.25,0) (1.55,0) -- (4.05,0) (4.35,0) -- (4.95+0.5,0);
\draw [thkln] (-1.55-0.6-0.5,0) -- (-1.55,0) (-1.25,0) -- (1.25,0) (1.55,0) -- (4.05,0) (4.35,0) -- (4.95+0.5,0);
\draw [ptzert] (-1.55,-0.6) rectangle ++(0.3,1.2);
\draw [ptzert] (1.55,-0.6) rectangle ++(-0.3,1.2);
\draw [ptzert] (4.05,-0.6) rectangle ++(0.3,1.2);
%%%%
\draw [ptzer] (0.8,2) rectangle ++(-1.6,-0.8);
\draw [thkc] (0.2,0.1+1.6) arc (0:-180:0.2);
\draw [thkln]  (-0.5,2) -- (-0.5,2.6);
\draw [thkln]  (0.5,2) -- (0.5,2.6);
%\node at (0.05,2.3) {$\cdots$};
\node at (0.05,2.4) {$\scriptstyle \cdots$};
\draw [ptzer] (0.8,-2) rectangle ++(-1.6,0.8);
\draw [thkc] (0.2,-0.1-1.6) arc (0:180:0.2);
\draw [thkln]  (-0.5,-2) -- (-0.5,-2.6);
\draw [thkln]  (0.5,-2) -- (0.5,-2.6);
%\node at (0.05,-2.3) {$\cdots$};
\node at (0.05,-2.4) {$\scriptstyle \cdots$};
\draw [thkln] (-1.25,0.35) to [out=0,in=-90] (-0.5,1.2);
\draw [thkln,xscale=-1] (-1.25,0.35) to [out=0,in=-90] (-0.5,1.2);
\draw [thkln,yscale=-1] (-1.25,0.35) to [out=0,in=-90] (-0.5,1.2);
\draw [thkln,xscale=-1,yscale=-1] (-1.25,0.35) to [out=0,in=-90] (-0.5,1.2);
\draw [thkln] (-1.55,0.35) to [out=180,in=-90] ++(-0.75,0.85) -- (-2.3,2.6);
\draw [thkln,xscale=-1] (-1.55,0.35) to [out=180,in=-90] ++(-0.75,0.85) -- (-2.3,2.6);
\draw [thkln,xshift=5.6cm] (-1.55,0.35) to [out=180,in=-90] ++(-0.75,0.85) -- (-2.3,2.6);
\draw [thkln,xscale=-1,yscale=-1,xshift=-2.8cm] (-1.55,0.35) to [out=180,in=-90] ++(-0.75,0.85) -- (-2.3,2.6);
\draw [decorate,decoration={brace,amplitude=4pt},xshift=0,yshift=-5pt]
(5.1cm+5pt,-2.6) -- (-0.5cm-5pt,-2.6)
node [black,midway,yshift=-0.4cm]
{\scriptsize to struts};
\draw [decorate,decoration={brace,amplitude=4pt},xshift=0,yshift=5pt]
(-2.3cm-5pt,2.6) -- (3.3cm+5pt,2.6)
node [black,midway,yshift=0.4cm]
{\scriptsize to struts};
\draw [dashed] (2.8,2.6) -- (2.8,-2.6);
\end{tikzpicture}
\;,\qquad
\begin{tikzpicture}[menvtwo]
\draw [bcrct] (-1.55-0.6-0.5,0) -- (-1.55,0) (-1.25,0) -- (1.25,0) (1.55,0) -- (4.05,0) (4.35,0) -- (4.95+0.5,0);
\draw [thkln] (-1.55-0.6-0.5,0) -- (-1.55,0) (0,1.2) -- (0,-1.2) (1.55,0) -- (4.05,0) (4.35,0) -- (4.95+0.5,0);
\draw [ptzert] (-1.55,-0.6) rectangle ++(0.3,1.2);
\draw [ptzert] (1.55,-0.6) rectangle ++(-0.3,1.2);
\draw [ptzert] (4.05,-0.6) rectangle ++(0.3,1.2);
%%%%
\draw [ptzer] (0.8,2) rectangle ++(-1.6,-0.8);
\draw [thkc] (0.2,0.1+1.6) arc (0:-180:0.2);
\draw [thkln]  (-0.5,2) -- (-0.5,2.6);
\draw [thkln]  (0.5,2) -- (0.5,2.6);
\node at (0.05,2.4) {$\scriptstyle \cdots$};
\draw [ptzer] (0.8,-2) rectangle ++(-1.6,0.8);
\draw [thkc] (0.2,-0.1-1.6) arc (0:180:0.2);
\draw [thkln]  (-0.5,-2) -- (-0.5,-2.6);
\draw [thkln]  (0.5,-2) -- (0.5,-2.6);
\node at (0.05,-2.4) {$\scriptstyle \cdots$};
\draw [thkln] (-1.25,0.35) to [out=0,in=-90] (-0.5,1.2);
\draw [thkln,xscale=-1] (-1.25,0.35) to [out=0,in=-90] (-0.5,1.2);
\draw [thkln,yscale=-1] (-1.25,0.35) to [out=0,in=-90] (-0.5,1.2);
\draw [thkln,xscale=-1,yscale=-1] (-1.25,0.35) to [out=0,in=-90] (-0.5,1.2);
\draw [thkln] (-1.55,0.35) to [out=180,in=-90] ++(-0.75,0.85) -- (-2.3,2.6);
\draw [thkln,xscale=-1] (-1.55,0.35) to [out=180,in=-90] ++(-0.75,0.85) -- (-2.3,2.6);
\draw [thkln,xshift=5.6cm] (-1.55,0.35) to [out=180,in=-90] ++(-0.75,0.85) -- (-2.3,2.6);
\draw [thkln,xscale=-1,yscale=-1,xshift=-2.8cm] (-1.55,0.35) to [out=180,in=-90] ++(-0.75,0.85) -- (-2.3,2.6);
\draw [decorate,decoration={brace,amplitude=4pt},xshift=0,yshift=-5pt]
(5.1cm+5pt,-2.6) -- (-0.5cm-5pt,-2.6)
node [black,midway,yshift=-0.4cm]
{\scriptsize to struts};
\draw [decorate,decoration={brace,amplitude=4pt},xshift=0,yshift=5pt]
(-2.3cm-5pt,2.6) -- (3.3cm+5pt,2.6)
node [black,midway,yshift=0.4cm]
{\scriptsize to struts};
\draw [dashed] (2.8,2.6) -- (2.8,-2.6);
\end{tikzpicture}
\]
In both diagrams the left projector on the grey strip is the initial one, the middle projector is the current one and the right projector is the next after the current. It is not hard to see that if we purge the current projector, then we get similar diagrams with the next projector becoming the current one (the first type diagram may become of either type after the purge, while the second type diagram remains of the same type). The \tqdgr\ shifts remain non-negative, because the purging does not produce any circles: it just makes explicit various line connections that were hidden inside the constituent \tTLt s of the current projector.

Our goal is to remove the single line circles one-by-one. Consider a particular circle. We designate one of its projectors as initial and then, going clockwise, we detach the circle from other projectors step-by-step. First, if the current projector is of the third or fourth type of\rx{eq:frmprj}, then we replace it with the first or second projector with the help of the homotopy equivalences
\begin{equation}
\label{eq:smsdr}
\begin{tikzpicture}[menvtwo]
\draw [ptzer] (-0.15,-0.6) rectangle (.15,.6);
\draw [thkln] (-0.75,0) -- (-0.15,0)
node [near start,above] {$\scriptstyle \xca$}
(0.15,0) -- (.75,0)
node [near end,below] {$\scriptstyle \xca$};
\draw (-0.75,-0.4) -- (-0.15,-0.4) (0.15,0.4) -- (0.75,0.4);
\end{tikzpicture}
\;\hteqv\;
\shcr^{\hlf\xca}
\begin{tikzpicture}[menvtwo]
\draw [thkln] (-1.25,0) -- (-0.15,0)
node [very near start,above] {$\scriptstyle \xca$}
(0.15,0) -- (.75,0)
node [near end,below] {$\scriptstyle \xca$};
\draw [lnovr]  (-1.25,-0.4) to [out=0,in=180] (-0.15,0.4);
\draw (-1.25,-0.4) to [out=0,in=180] (-0.15,0.4) (0.15,0.4) -- (0.75,0.4);
\draw [ptzer] (-0.15,-0.6) rectangle (.15,.6);
\end{tikzpicture},
\qquad
\begin{tikzpicture}[menvtwo]
\draw [thkln] (-0.75,0) -- (-0.15,0)
node [near start,below] {$\scriptstyle \xca$}
(0.15,0) -- (.75,0)
node [near end,above] {$\scriptstyle \xca$};
\draw (-0.75,0.4) -- (-0.15,0.4) (0.15,-0.4) -- (0.75,-0.4);
\draw [ptzer] (-0.15,-0.6) rectangle (.15,.6);
\end{tikzpicture}
\;\hteqv\;
\shcr^{-\hlf\xca}
\begin{tikzpicture}[menvtwo]
\draw [thkln] (-1.25,0) -- (-0.15,0)
node [very near start,below] {$\scriptstyle \xca$}
(0.15,0) -- (.75,0)
node [near end,above] {$\scriptstyle \xca$};
\draw [lnovr]  (-1.25,0.4) to [out=0,in=180] (-0.15,-0.4);
\draw (-1.25,0.4) to [out=0,in=180] (-0.15,-0.4) (0.15,-0.4) -- (0.75,-0.4);
\draw [ptzer] (-0.15,-0.6) rectangle (.15,.6);
\end{tikzpicture}
\end{equation}
(note that we keep the level of the outgoing single lines and keep single lines above the $\xca$-cables).

Now suppose that our projector junction is of the first type (the second type is treated similarly). We perform a \ltrf
\begin{equation}
\label{eq:lrdpr_o}
\ytngsi =
\begin{tikzpicture}[menvtwo]
\draw [ptzer] (-0.15,-0.6) rectangle (.15,.6);
\draw [thkln] (-0.75,0) -- (-0.15,0)
node [near start,below] {$\scriptstyle \xca$}
(0.15,0) -- (.75,0)
node [near end,below] {$\scriptstyle \xca$};
\draw (-0.75,0.4) -- (-0.15,0.4) (0.15,0.4) -- (0.75,0.4);
\end{tikzpicture},
\qquad
\ytngsf =
\begin{tikzpicture}[menvtwo]
\draw [ptzer] (-0.15,-0.6) rectangle (.15,.6);
\draw [thkln] (-0.75,0) -- (-0.15,0)
node [near start,below] {$\scriptstyle \xca$}
(0.15,0) -- (.75,0)
node [near end,below] {$\scriptstyle \xca$};
\draw (-0.75,0.8) -- (0.75,0.8);% (0.15,-0.4) -- (0.75,-0.4);
\end{tikzpicture},
\qquad
\ytngsc =
\shcr
\begin{tikzpicture}[menvtwo]
\draw[ptfour] (-.15,-0.6) rectangle ++(0.3,1.2);
\draw[thkln] (-1.35,0) -- (-.15,0)
node[near start,below] {$\scriptstyle \xca+1$ }
(.15,0) -- (1.35,0)
node [near end, below] {$\scriptstyle \xca+1$};
\end{tikzpicture}
\end{equation}
and the cone relation\rx{eq:cnrel} is \ex{eq:jwpar}.

After we performed the \ltrf\rx{eq:lrdpr} on all projectors of a given single line, except the initial one, the circle remains attached to $\xDclN$ only at the initial projector. We perform the replacement\rx{eq:smsdr} on it, if necessary. Suppose that as a result the projector is of the first type of\rx{eq:frmprj}. The single line circle passes above the $\xca$-cable in all their intersections, hence the circle can be contracted towards the remaining junction, and the junctions has the form
$
\begin{tikzpicture}[menvthree]
\draw [ptzer] (-0.15,-0.6) rectangle (.15,.6);
\draw [thkln] (-0.75,0) -- (-0.15,0)
node [near start,below] {$\scriptstyle \xca$}
(0.15,0) -- (.75,0)
node [near end,below] {$\scriptstyle \xca$};
\draw (-0.15,0.4) to [out=180,in=-90] (-0.6,0.8) to [out=90,in=180] (0,1.4)
to [out=0,in=90] (0.6,0.8) to [out=-90,in=0] (0.15,0.4);
\end{tikzpicture}
$
Now we perform a \ltrf
\[
\ytngsi=
\begin{tikzpicture}[menvtwo]
\draw [ptzer] (-0.15,-0.6) rectangle (.15,.6);
\draw [thkln] (-0.75,0) -- (-0.15,0)
node [near start,below] {$\scriptstyle \xca$}
(0.15,0) -- (.75,0)
node [near end,below] {$\scriptstyle \xca$};
\draw (-0.15,0.4) to [out=180,in=-90] (-0.6,0.8) to [out=90,in=180] (0,1.4)
to [out=0,in=90] (0.6,0.8) to [out=-90,in=0] (0.15,0.4);
\end{tikzpicture},
\qquad
\ytngsf = \shfr^{-1}
\begin{tikzpicture}[menvtwo]
\draw [ptzer] (-0.15,-0.6) rectangle (.15,.6);
\draw [thkln] (-0.75,0) -- (-0.15,0)
node [near start,below] {$\scriptstyle \xca$}
(0.15,0) -- (.75,0)
node [near end,below] {$\scriptstyle \xca$};
%\draw (-0.15,0.4) to [out=180,in=-90] (-0.6,0.8) to [out=90,in=180] (0,1.4)
%to [out=0,in=90] (0.6,0.8) to [out=-90,in=0] (0.15,0.4);
\end{tikzpicture},
\qquad
\ytngsc =
\shcr^{2\xca}\shfr
\begin{tikzpicture}[menvtwo]
\draw [color=white] (-0.15,-1) rectangle (0.15,1.6);
\draw [ptthr] (-0.15,-0.6) rectangle (.15,.6);
\draw [thkln] (-0.75,0) -- (-0.15,0)
node [near start,below] {$\scriptstyle \xca$}
(0.15,0) -- (.75,0)
node [near end,below] {$\scriptstyle \xca$};
\draw (-0.15,0.4) to [out=180,in=-90] (-0.6,0.8) to [out=90,in=180] (0,1.4)
to [out=0,in=90] (0.6,0.8) to [out=-90,in=0] (0.15,0.4);
%\draw (-0.75,0.4) -- (-0.15,0.4) (0.15,0.4) -- (0.75,0.4);
\end{tikzpicture}
\]
and the cone relation\rx{eq:cnrel} is \ex{eq:htloop}.

Performing these steps on each single line circle of $\xtDNo$ we transform it into $\xDclN$, and the composition of elementary maps yields the map
\begin{equation}
\label{eq:spmapsp}
\KHm(\xDclN)\xrightarrow{\;\mntfN\;}
\end{equation}

 Composing all maps, which lead from $\xDclN$ to $\xDclNo$, we get a \tdgpr\ map
\begin{equation}
\label{eq:spmapsp_o}
\KHm(\xDclN)\xrightarrow{\;\mntfN\;} \shcr^{(\xca+\hlf)\ncrD}\, \shfr^{\gvD}\,\KHm(\xDclNo).
\tKHm(\xDclN)\xrightarrow{\;\mntfN\;} \tKHm(\xDclNo).
\end{equation}
Note that the degree shifts of \ex{eq:smsdr} cancel each other, since a single line circle has equal number of the junctions of the third and fourth type of\rx{eq:frmprj}.
After the degree shifts\rx{eq:shdgt}, the map\rx{eq:spmapsp} becomes the map\rx{eq:spmaps}.

\section{Estimates of \thdgr\ bounds on correction diagrams}

Let $\xD$ be a diagram of a link which may involve \tJWp s. Let $\yncrD$ denote the number of single line crossings in $\xD$. If $\xD$ involves cables, then we consider them as multiple single lines for the purpose of defining $\yncrD$: a crossing of an $a$-cable and a $b$-cable contributes $ab$ to $\yncrD$.

Our estimates of \thdgr\ bounds are based on the following:
\begin{proposition}
\label{prp:dgrest}
$\KHmvv{i}{\hem}(\xD)=0$ for $i\leq-\hlf \yncrD-1$, where $\yncrD$ is the number of single line crossings in $\xD$.
\end{proposition}
\begin{proof}
The claim follows easily from the defining relations\rx{eq:dkhbr} of the \tKhbr\ and from the fact that a \cJWp\ has a presentation as a complex with only non-negative homological degrees.
\end{proof}

\subsection{\Ltrf s of the first type}

\def\xpsh{3}

\begin{proof}[Proof of Proposition\rw{prp:bndfs}]

The diagram $\xtDN$ is the result of applying \Arpl s to all crossings of $\xDclNo$. A composition of maps $\xmg$ corresponding to \ltrf s\rx{eq:frstst} produces a map of \tbdgr\ zero
\[
\shcr^{\hlf\yncrv{\xtDN}}
\shfr^{\gvD(\xca+1)}
\KHm(\xtDN) \xrightarrow{\;\xmgN\;} \tKHm(\xDclNo)
\]

We perform a finite induction over the number of \Arpl s on the diagram $\xDclNo$.
Suppose that we performed \Arpl s on first $\incra$ crossings, thus obtaining the diagram $\xDsi$, and consider the \ltrf\rx{eq:frstst} on the $(\incra+1)$-st vertex. By the assumption of induction, we have a %\tbdgr\ zero
map
%of zero \tbdgr
\[
%\shcr^{\hlf\clN^2\ncrD + \hlf(2\xca+1)(\ncrD-\incra)} \shfr^{\gvD\xca}
\shcr^{\hlf\yncrv{\xDsi}} \shfr^{\gvD(\xca+1)}
\KHm(\xDsi) \xrightarrow{\mnfi} \tKHm(\xDclNo),
%\qquad
%\yncrv{\xDsi} = \clN^2\ncrD + (2\xca+1)(\ncrD-\incra)
\]
where
\[
\yncrv{\xDsi} = \clN^2\ncrD + (2\xca+1)(\ncrD-\incra)
\]
is the number of single line crossings in $\xDsi$. The map $\mnfi$ has zero \tbdgr\ and it is an isomorphism on $\tKHmvb{i}(\xDclNo)$ for $i\leq 2\xca-1$.

In diagram $\xDsf$ the $(\incra+1)$-st vertex is \Arpld, while is the diagram $\xDsc$ this vertex is \Brpld. Hence the diagrams $\xDsf$ and $\xDsc$ have equal number of single line crossings:
\[
\yncrv{\xDsf} = \yncrv{\xDsc} = \yncrv{\xDsi} - (2\xca + 1).
%\yncr = \xca^2\ncrD + (2\xca+1)(\ncrD-\incra-1)
\]
%single line crossings, and
By Proposition\rw{prp:dgrest},
$
\KHmvb{i}(\xDsc) = 0
$
for $i\leq-\hlf\yncrv{\xDsc}+\xca-\hlf$, where $\xca+\hlf$ accounts for the degree shift of the correction tangle in\rx{eq:frstst}. Hence by Proposition\rw{prp:shest} the map $\KHmvb{i}(\xDsf)\xrightarrow{\xmg}\KHmvb{i}(\xDsi)$ is an isomorphism for
$i \leq-\hlf\yncrv{\xDsc}+\xca-\thlf$ and the composition map
\[
%\shcr^{\hlf\clN^2\ncrD + \hlf(2\xca+1)(\ncrD-\incra-1)}
\shcr^{\hlf\yncrv{\xDsf}+\xca+\hlf}
\shfr^{\gvD(\xca+1)}\KHmvb{i}(\xDsf) \xrightarrow{\mnff}\tKHmvb{i}(\xDclNo)
\]
is an isomorphism for $i\leq 2\xca-1$.
\end{proof}

\subsection{\Ltrf s of the second type}

The resulting diagram is $\xDsi$.
 and, by the assumption of the induction we have a map

 of $\xDclNo$. The resulting

Let $\xDNoa$ denote a diagram constructed from $\xDclNo$ by performing \Brpl s on crossings $\xlbb$, $1\leq \incrb \leq \incra$ (in particular,  $\xtDNo= \xDpNof$), and let $\xDpNoa$ denote a diagram constructed by performing \Brpl s on crossings $\xlbb$, $1\leq \incrb \leq \incra$ and \Arpl\ on the crossing $\xlbao$. We define their shifted \tKhom:
\[\tKHm(\xDNoa) = \shcr^{abcd} \KHm(\xDNoa).
\]

According to Lemma\rw{lm:sss}, the complexes of the picture\rx{eq:abrpl} form an exact triangle, hence there is a canonical map
\[
\shcr^{-(\xca+\hlf)}\KHm(\xDNoa)\xrightarrow{\;\xmgNa\;}
\KHm(\xDNoamo)
\]
which has zero degree with respect to all gradings. A composition of maps $\xmgNa$ for $\incra=0,\ldots,\ncrD-1$ is a zero-degree map
\[
\shcr^{-\ncrD(\xca+\hlf)}\KHm(\xDclNo) \xrightarrow{\;\xmgN\;} \KHm(\xtDNo).
\]

%
%\[
%\KHm(\xtDNo)\xrightarrow{}\cdots\xrightarrow{}\shcr^{-(\xca+\hlf)\incra}\KHm(\xDNoa)\xrightarrow{}
%\shcr^{-(\xca+hlf)(\incra-1)}\KHm(\xDNoamo)\xrightarrow{}\cdots\xrightarrow{}\KHm(\xDclNo).
%\]

and let $\xDpNoa$ denote a diagram constructed by performing \Brpl s on crossings $\xlbb$, $1\leq \incrb \leq \incra$ and \Arpl\ on the crossing $\xlba$.

\subsection{Second stage}
Consider the structure of the diagram $\xtDNo$. It consists of two parts. The first part is $\xca$-cabled diagram $\xDclN$. The second part consists of non-intersecting circles formed by single lines, which appeared when \Brpl s were applied to all vertices of $\xDclNo$. These circles are the same as the circles of the diagram $\sBD$ which is constructed by \Bsplng\ all crossings of the original diagram $\xD$.
We orient single line circles clockwise and on our pictures the clockwise orientation goes from the left to the right.

Both parts of $\xtDNo$ are joined by \cJWp s, the junctions having four possible forms:
%depending on whether the single lines appear on the same side of the $\xca$-cable or on opposite %sides:
\begin{equation}
\label{eq:frmprj1}
\begin{tikzpicture}[menvtwo]
\draw [ptzer] (-0.15,-0.6) rectangle (.15,.6);
\draw [thkln] (-0.75,0) -- (-0.15,0)
node [near start,above] {$\scriptstyle \xca$}
(0.15,0) -- (.75,0)
node [near end,above] {$\scriptstyle \xca$};
\draw (-0.75,-0.4) -- (-0.15,-0.4) (0.15,-0.4) -- (0.75,-0.4);
\end{tikzpicture},
\qquad
\begin{tikzpicture}[menvtwo]
\draw [ptzer] (-0.15,-0.6) rectangle (.15,.6);
\draw [thkln] (-0.75,0) -- (-0.15,0)
node [near start,below] {$\scriptstyle \xca$}
(0.15,0) -- (.75,0)
node [near end,below] {$\scriptstyle \xca$};
\draw (-0.75,0.4) -- (-0.15,0.4) (0.15,0.4) -- (0.75,0.4);
\end{tikzpicture},
\qquad
\begin{tikzpicture}[menvtwo]
\draw [ptzer] (-0.15,-0.6) rectangle (.15,.6);
\draw [thkln] (-0.75,0) -- (-0.15,0)
node [near start,above] {$\scriptstyle \xca$}
(0.15,0) -- (.75,0)
node [near end,below] {$\scriptstyle \xca$};
\draw (-0.75,-0.4) -- (-0.15,-0.4) (0.15,0.4) -- (0.75,0.4);
\end{tikzpicture},
\qquad
\begin{tikzpicture}[menvtwo]
\draw [thkln] (-0.75,0) -- (-0.15,0)
node [near start,below] {$\scriptstyle \xca$}
(0.15,0) -- (.75,0)
node [near end,above] {$\scriptstyle \xca$};
\draw (-0.75,0.4) -- (-0.15,0.4) (0.15,-0.4) -- (0.75,-0.4);
\draw [ptzer] (-0.15,-0.6) rectangle (.15,.6);
\end{tikzpicture}.
\end{equation}

Our strategy is to remove the circles one by one by though the following procedure applied to each circle: we choose the zeroth projector on the circle and then go from it clockwise, detaching the single line from the projectors. After we detach the single line circle from all projectors except the initial one, we

 First, if the projector is of the third or fourth type of\rx{eq:frmprj}, then we turn it into the first or second form by the homotopy equivalences
\begin{equation}
\label{eq:smsdr1}
\begin{tikzpicture}[menvtwo]
\draw [ptzer] (-0.15,-0.6) rectangle (.15,.6);
\draw [thkln] (-0.75,0) -- (-0.15,0)
node [near start,above] {$\scriptstyle \xca$}
(0.15,0) -- (.75,0)
node [near end,below] {$\scriptstyle \xca$};
\draw (-0.75,-0.4) -- (-0.15,-0.4) (0.15,0.4) -- (0.75,0.4);
\end{tikzpicture}
\;\hteqv\;
\shcr^{\hlf\xca}
\begin{tikzpicture}[menvtwo]
\draw [thkln] (-1.25,0) -- (-0.15,0)
node [very near start,above] {$\scriptstyle \xca$}
(0.15,0) -- (.75,0)
node [near end,below] {$\scriptstyle \xca$};
\draw [lnovr]  (-1.25,-0.4) to [out=0,in=180] (-0.15,0.4);
\draw (-1.25,-0.4) to [out=0,in=180] (-0.15,0.4) (0.15,0.4) -- (0.75,0.4);
\draw [ptzer] (-0.15,-0.6) rectangle (.15,.6);
\end{tikzpicture},
\qquad
\begin{tikzpicture}[menvtwo]
\draw [thkln] (-0.75,0) -- (-0.15,0)
node [near start,below] {$\scriptstyle \xca$}
(0.15,0) -- (.75,0)
node [near end,above] {$\scriptstyle \xca$};
\draw (-0.75,0.4) -- (-0.15,0.4) (0.15,-0.4) -- (0.75,-0.4);
\draw [ptzer] (-0.15,-0.6) rectangle (.15,.6);
\end{tikzpicture}
\;\hteqv\;
\shcr^{-\hlf\xca}
\begin{tikzpicture}[menvtwo]
\draw [thkln] (-1.25,0) -- (-0.15,0)
node [very near start,below] {$\scriptstyle \xca$}
(0.15,0) -- (.75,0)
node [near end,above] {$\scriptstyle \xca$};
\draw [lnovr]  (-1.25,0.4) to [out=0,in=180] (-0.15,-0.4);
\draw (-1.25,0.4) to [out=0,in=180] (-0.15,-0.4) (0.15,-0.4) -- (0.75,-0.4);
\draw [ptzer] (-0.15,-0.6) rectangle (.15,.6);
\end{tikzpicture}
\end{equation}
(note that we keep the level of the outgoing single lines and keep single lines above the $\xca$-cables).
Second, we replace the projectors of the first two types of\rx{eq:frmprj} by the diagrams
\[
\begin{tikzpicture}[menvtwo]
\draw [ptzer] (-0.15,-0.6) rectangle (.15,.6);
\draw [thkln] (-0.75,0) -- (-0.15,0)
node [near start,above] {$\scriptstyle \xca$}
(0.15,0) -- (.75,0)
node [near end,above] {$\scriptstyle \xca$};
\draw (-0.75,-0.8) -- (0.75,-0.8);% (0.15,-0.4) -- (0.75,-0.4);
\end{tikzpicture},
\qquad
\begin{tikzpicture}[menvtwo]
\draw [ptzer] (-0.15,-0.6) rectangle (.15,.6);
\draw [thkln] (-0.75,0) -- (-0.15,0)
node [near start,below] {$\scriptstyle \xca$}
(0.15,0) -- (.75,0)
node [near end,below] {$\scriptstyle \xca$};
\draw (-0.75,0.8) -- (0.75,0.8);% (0.15,-0.4) -- (0.75,-0.4);
\end{tikzpicture}
\]
After we finish, the single line circle is attached only at the zeroth projector. The final step is to `suck' it into the projector.
%
%
%: first we perform replacements
%\begin{equation}
%\label{eq:prjrpl}
%%\begin{tikzpicture}[menvtwo,decoration=snake]
%%\draw [ptzer] (-0.15,-0.6) rectangle (.15,.6);
%%\draw [thkln] (-0.75,0) -- (-0.15,0)
%%node [near start,above] {$\scriptstyle \xca$}
%%(0.15,0) -- (.75,0)
%%node [near end,above] {$\scriptstyle \xca$};
%%\draw (-0.75,-0.4) -- (-0.15,0.-4) (0.15,0.-4) -- (0.75,0.-4);
%%\draw [decorate] (2,0) -- (4,0) arc (0:180:1.5 and 1);
%%\end{tikzpicture}
%\xy 0;/r.22pc/:
%(-20,0)*{
%\begin{tikzpicture}[menvtwo]
%\draw [ptzer] (-0.15,-0.6) rectangle (.15,.6);
%\draw [thkln] (-0.75,0) -- (-0.15,0)
%node [near start,above] {$\scriptstyle \xca$}
%(0.15,0) -- (.75,0)
%node [near end,above] {$\scriptstyle \xca$};
%\draw (-0.75,-0.4) -- (-0.15,-0.4) (0.15,-0.4) -- (0.75,-0.4);
%\end{tikzpicture}
%}="1";
%(20,0)*{
%\begin{tikzpicture}[menvtwo]
%\draw [ptzer] (-0.15,-0.6) rectangle (.15,.6);
%\draw [thkln] (-0.75,0) -- (-0.15,0)
%node [near start,above] {$\scriptstyle \xca$}
%(0.15,0) -- (.75,0)
%node [near end,above] {$\scriptstyle \xca$};
%\draw (-0.75,-0.8) -- (0.75,-0.8);% (0.15,-0.4) -- (0.75,-0.4);
%\end{tikzpicture}
%}="2";
%{\ar@{~>}"1"+(12,0);"2"+(-12,0)};
%\endxy
%\end{equation}
%on its \tJWp s, except for a single projector, so that the circle is attached to $\xDclN$ only at this projector, and then `sucking in' the circle into this remaining projector.

We label the single line circles of $\xtDNo$ as $\ylbb$, $1\leq \crcb \leq \gvD$. On each circle $\ylbb$ we choose the zeroth \tJWp\ and then label other projectors clockwise as $\zlbbg$, $0 \leq \prjg \leq\pncrb-1 $, where $\pncrb$ is the total number projectors on the circle $\ylbb$ ($\pncrb$ is also equal to the number of crossings ... ). $\xtDNob$ denotes the diagram $\xtDNo$ in which the single line circles $\ylbbp$ with $\crcbp < \crcb$ are removed. Further, $\xtDNobg$ denotes the diagram $\xtDNob$ in which the circle $\ylbb$ is detached from the projectors $\zlbbgp$ with $1\leq\prjgp<\prjg$.
The homotopy equivalence\rx{eq:jwpar} determines a canonical map
\begin{equation}
\label{eq:mpgo}
\KHm(\xtDNobgo)\xrightarrow{\;\xtmgbg\;}
\KHm(\xtDNobg).
\end{equation}

In the diagram $\xtDNobf$ the cirlce $\ylbb$ is detached only at the zeroth projector, and whenever it crosses the cable (due to replacements\rx{eq:smsdr}) it passes above it. Hence the circle can be contracted towards the zeroth projector. Let $\xhDNob$ denote the resulting diagram, the vicinity of the zeroth projector there has the form:
\[
\begin{tikzpicture}[menvtwo]
\draw [ptzer] (-0.15,-0.6) rectangle (.15,.6);
\draw [thkln] (-0.75,0) -- (-0.15,0)
node [near start,below] {$\scriptstyle \xca$}
(0.15,0) -- (.75,0)
node [near end,below] {$\scriptstyle \xca$};
\draw (-0.15,0.4) to [out=180,in=-90] (-0.6,0.8) to [out=90,in=180] (0,1.4)
to [out=0,in=90] (0.6,0.8) to [out=-90,in=0] (0.15,0.4);
\end{tikzpicture}
\]
(the single line loop may also appear below, but this distinction is not important).
Thus there is an isomorphism $\KHm(\xtDNobf)\cong\KHm(\xhDNob)$ and the composition of all maps\rx{eq:mpgo} for $1\leq\prjg\leq \pncrb-1$ produces a map
\[
\KHm(\xhDNob)\xrightarrow{}\KHm(\xtDNob).
\]
The homotopy equivalence\rx{eq:htloop} determines a canonical map
\[
\KHm(\xtDNobo) \xrightarrow{} \KHm(\xhDNob).
\]

Let $\xhDNob=$ denote the diagram in which the circle $\ylbb$ is attached

%replacements\rx{eq:prjrpl} are performed on projectors $\zlbbgp$ with $\prjgp<\prjg$.

\subsection{The tail structure}
\subsubsection{Tail conjecture}

We want to study the structure of the \thead\ of the colored Jones polynomial which is the dominant part of $\pJqaL$ at $q\rightarrow 0$. In other words, we want to describe the coefficients at the highest negative powers of $q$ in $\pJqaL$. If $\xLb$ is the mirror image of $\xL$, then $\pJvv{\xca}{\xLb}(q) = \pJvv{\xca}{\xL}(\qi)$, hence all results about the behavior of $\pJqaL$ at $q\rightarrow 0$ can be applied at $q\rightarrow \infty$ by replacing $\xL$ with $\xLb$.

%It has been observed experimentally that at the large values of the color $\xca$ the \thead\ of $\pJqaL$ stabilizes: it seems that for every link $\xL$ and for $\xnu\in\{0,1\}$ there exists a  series $\zpolnqnL\in\ZZ[[q]]$  and an integer-valued function $\xfs(\xca)$ such that
%\[
%\pJqaL = q^{\xfs(\xca)} (\zpolnqnL + O(q^{2\xca})),
%\]
%if $\xnu = \xca \mod 2$. In other words, if $\xca\p - \xca$ is a positive even integer, then
%$\pJqvv{\xca\p}{\xL}$ and $\pJqaL$ share the same first $\xca$ coefficients.

It has been observed experimentally that at the large values of the color $\xca$ the \thead\ of $\pJqaL$ stabilizes: if $\xca\p - \xca$ is a positive even integer, then $\pJqvv{\xca\p}{\xL}$ and $\pJqaL$ share the same first $\xca$ coefficients.
\begin{conjecture}
\label{conj.tail}
For every link $\xL$ and for $\xnu\in\{0,1\}$ there exists a  series $\zpolnqnL\in\ZZ[[q]]$  and an integer-valued function $\xfsLN$ such that
\[
\pJqaL = q^{-\xfsLN} (\zpolnqnL + O(q^{\xca})),
\]
if $\xnu = \xca \mod 2$.
\end{conjecture}

% More formally, for every link $\xL$ and for $\xnu\in\{0,1\}$ there exists a  series $\zpolnqnL\in\ZZ[[q]]$  and an integer-valued function $\xfsLN$ such that
%\[
%\pJqaL = q^{\xfsLN} (\zpolnqnL + O(q^{2\xca})),
%\]
%if $\xnu = \xca \mod 2$.

The strongest result so far regarding this conjecture has been obtained by C.~Armond.
\begin{theorem}[C.~Armond]
If $\xL$ is an alternating link, then Conjecture\rw{conj.tail}  holds true and the \thead\ does not depend on the parity of $\xca$: $\zpolnqvv{0}{\xL}=\zpolnqvv{1}{\xL}$.
\end{theorem}

\subsubsection{Multi-layered structure}
A closer look at the experimental data suggests that the tail of $\pJqaL$ can be better described by the double Laurent series of the form
\[
q^{-\xfsLN} \sum_{m,n} c_{m,n}\, q^{m + \xn\xca}.
\]
 We are going to prove the following:
%
%Let $\xD$ be a diagram of a framed link $\xL$ (we assume blackboard framing). Let $\ncrD$ denote the number of crossings in $\xD$. Each crossing can be `spliced' in two ways, and we call them \Asplng\ and \Bsplng. Let $\nBD$ denote the number of circles in the diagram $\sBD$ which is the result of \Bsplng\ of all crossings of $\xD$.
%\begin{theorem}
%\label{thm.main}
%If a framed link $\xL$ is presented by a diagram $\xD$ with $\ncr$ crossings, then $\xL$ determines a sequence of Laurent series $\ypolqnLcr\in\Zqiq$, $\xn\geq 0$, such that
%\[
%\degq \ypolqnLcr \geq -
%\]
%and
%\[
%\pJqaL = q^{-\frth\ncr\xca^2 - \nBD\xca }
%\sum_{\xn = 0}^{?} q^{\xn\xca}\ypolqnL
%\]
%\end{theorem}
%
%
\begin{theorem}
\label{thm.main}
For a framed link $\xL$ with the minimal crossing number $\ncrL$
there exists a non-negative integers $\gvL$, $\geL$ and $\glL$
%$\xfsL\geq 1$
and a two-variable
Laurent series
\[
\ypolqtL = \xt^{-\hlf\gvL}\smxnzi\xt^\xn\, \ypolqnL,
%\qquad \ypolqnL\in\Zqiq
\]
where $\ypolqnL$ is a Laurent series with degree bounds
 % of $q$ such that
\begin{equation}
\label{eq.bnd}
\degq \ypolqnL\geq  -\shlf\geL \xn(\xn+1) - \shlf\glL\left(2 n + 1\right)^2 ,
\end{equation}
such that for any integer $\xM\geq 1$
\begin{equation}
\label{eq.mrel}
\pJqaL = \Nsgn q^{-\frth\ncrL\xca^2 - \gvL\xca }
\left(
\sum_{\xn = 0}^{\xM-1} q^{\xn\xca}\ypolqnL
+ O\left( q^{\xca(\xM-\hlf)}\right)
\right)
\end{equation}
if $\xca \geq 16\, \ncrtD^2\xM^3$.
The series $\ypolqtL$ is a topological invariant of the framed link $\xL$.
\end{theorem}

This theorem is an easy corollary of the following
\begin{theorem}
\label{thm.ma}
A link diagram $\xD$ determines a two-variable Laurent series
\[
\ypolqtD = t^{-\hlf\nBD} \smxnzi \xt^\xn\,\ypolqnD
\]
where $\ypolqnD$ are Laurent series with degree bounds
\[
\degq \ypolqnL\geq -\shlf\,\geD \xn(\xn+1) - \shlf\,\glD\left( 2n + 1\right)^2,
\]
such that for any integer $\xM\geq 1$
%if $\xca \geq 16\, \ncrtD^2\xM^3$, then
\[
\pJqaD =\Nsgn q^{-\frth\ncrD\xca^2 - \hlf\gvD \xca }
\left(
\sum_{\xn = 0}^{\xM-1} q^{\xn\xca}\ypolqnL
+ O\left( q^{\xca(\xM-\hlf)}\right)
\right)
\]
if $\xca \geq 16\, \ncrtD^2\xM^3$.
\end{theorem}

\begin{proof}[Proof of Theorem]
By choosing the diagram $\xD$ of Theorem\rw{thm.ma} to be minimal and setting
$\ypolqtL = \ypolqtD$, $\gvL=\gvD$, $\geL=\geD$ and $\glL=\glD$ we are led to the bounds and relations.

It is easy to see from the relation\rx{eq.mrel} that the colored Jones polynomial $\pJqaL$ determines the series $\ypolqtL$, hence that series is a topological invariant of $\xL$.
\end{proof}

\subsubsection{Adequate links}

The main drawback of Theorem\rw{thm.main} is that for many links the series $\ypolqtL$
%determined by the colored Jones polynomial $\pJqaL$ through \ex{eq.mrel}
is identically zero: $\ypolqtL\equiv 0$. Examples of this phenomenon are identified with the help of the following simple corollary of Theorem\rw{thm.ma}:
%\begin{corollary}
%If for a framed link $\xL$ there exists a minimal diagram $\xD$ and and numbers $\bdA$ and $\xca_0$ such that for $\xca\geq \xca_0$ the colored Jones polynomial $\pJqaL$ has a degree bound
%%If the colored Jones polynomial $\pJqaD$ of a minimal diagram $\xD$ has a degree bound
%\[
%\degq \pJqaL \geq  -\sfrth(\ncrD-1)\xca^2 + \bdA \xca,
%\]
%then $\ypolqnL\equiv 0$
%\end{corollary}
%
\begin{corollary}
\label{cor.equiv}
If for a framed link $\xL$ there exist numbers $\bdA$ and $\xca_0$ such that for $\xca\geq \xca_0$ the colored Jones polynomial $\pJqaL$ has a degree bound
%If the colored Jones polynomial $\pJqaD$ of a minimal diagram $\xD$ has a degree bound
\begin{equation}
\label{eq.dcond1}
\degq \pJqaL \geq  -\sfrth(\ncrD-1)\xca^2 + \bdA \xca,
\end{equation}
then $\ypolqnL\equiv 0$.
\end{corollary}
\begin{proof}
It is easy to see from \ex{eq.mrel} that the Jones polynomials $\pJqaL$  at $\xca\geq \xca_0$ determine the series $\ypolqnL$ and that the condition\rx{eq.dcond1} implies $\ypolqtL\equiv 0$.
\end{proof}

As an example of this Corollary, consider torus knots $\yTmn$, $1< m < n$, their framing being determined by the standard diagram of the twisted $m$-cable of the unknot. Then
\[
\degq \pJqvv{\xca}{\yTmn} \geq -\sfrth n \xca^2 + \shlf n\xca.
\]
Since $\ncrv{\yTmn} = (m-1)n$, then according to Corollary\rw{cor.equiv}, $\ypolqtv{\yTmn}\equiv 0$ if $m\geq 3$. In fact, we conjecture, that this happens to all \tnBadq\ links:
\begin{conjecture}
If a framed link $\xL$ is \tnBadq, then $\ypolqtL\equiv 0$.
\end{conjecture}
The relation $\ypolqtL=0$ means that for the link $\xL$, the series $\ypolqtL$ computed according to the prescription of , misses the actual \thead\ of $\pJqaL$ and this \thead\ has to be computed differently.

For \tBadq\ links the method of captures the actual \thead. Indeed, a theorem by O.~Dasbach and X.-S.~Lin\cx{asdf} implies that if $\xL$ is \tBadq, then $\degq\pJqaL=-\frth\ncrL\xca^2 - \gvL\xca$ while the coefficient at the lowest power of $q$ is $1$. Hence Theorem\rw{thm.main} has the following corollary:
\begin{corollary}
If a framed link $\xL$ is \tBadq, then $\ypolqtL\not\equiv 0$ and $\ypolqvv{0}{\xL} = 1 + O(q)$.
\end{corollary}

\subsubsection{The \thead s of torus links $\yTmmn$}
The bound\rx{eq.bnd} for a \tBadq\ link is stronger:
\[
\degq \ypolqnL\geq  -\shlf\geL \xn(\xn+1).
\]
The computation of the \thead s of torus links indicates  that this bound is sharp. Indeed, according to~\cite{Mort95}, the colored Jones polynomial of $\yTmmn$ is

\subsubsection{The \thead\ and Habiro series}

Still it does not permit us to perform a direct substitution of $\xt = q^\xca$ in the whole double series $\ypolqtL$.

\end{document}

Let us recall the definitions of a pair of categories and a pair of 2-categories related to the Jones polynomial and its categorification. The first category in each pair is purely topological while the second category is associated to $\xSo$ by the 3-dimensional Chern-Simons-Witten \TQFT\ and by the 4-dimensional \TQFT\ corresponding to Khovanov's categorification.

An object of the topological tangle category $\cTng$ is an oriented 2-disc $\xDon$ with $n$ marked points in its interior
%. We will assume for simplicity that the points are
placed on a special oriented `equatorial' diameter in the disc. The marked points are framed, that is, there is a choice of a tangent vector at each point. We assume that the framing vectors are tangent to the \eqdiam. The set of morphisms $\Hom_{\cTng}(\xDom,\xDon)$ consists of \ttnglmn s, that is, we glue the discs $\xDom$ and $\xDon$ together along their circle boundaries so that the orientations and end-points of \eqdiam s match, and consider a 3-ball $\xBmpn$, whose boundary is the resulting 2-sphere. Then an \ttnglmn\ is an embedding of framed (unoriented) segments and cirlces into $\xBmpn$ such that the end-points of segments map to the $m+n$ marked points on its boundary. The morphism composition rule is the obvious composition of tangles.

The category $\cTng$ can be promoted to the 2-category $\ctTng$ if  the morphism sets $\Hom(\xDom,\xDon)$ are defined as categories, the morphisms between two tangles being cobordisms.

Let $q$ be a commutative variable and let $\QQqqi$ be the algebra of Laurent polynomials of $q$.
The \tTL\ (\taTL) category $\cTL$ is an additive category over the algebra $\QQqqi$ of Laurent polynomials of $q$. Its objects are the same as those of $\cTng$, but the set of morphisms
$\aTLmn = \Hom_{\cTL}(\xDom,\xDon)$ is a module over $\QQqqi$ generated freely by \taTLt s, which are  tangles consisting of segments embedded (crossinglessly) into the equatorial disc of $\xBmpn$. The composition of \taTLt s may produce disjoint circles, but each disjoint circle is replaced by the factor $\mqpqi$. The sum of all modules $\aTL = \bigoplus_{m,n\geq 0}\aTLmn$
%\[\aTL=\bigoplus_{m,n\geq 0} \Hom_{\cTL}(\xDom,\xDon)\]
is called the \tTLa.

In order to work with \tJWp s we have to introduce the algebra of formal Laurent series $\QQqqip$. The corresponding \tTLc\ is denoted as $\cTLp$.

The \tKbr\
$
\xKbrd\colon \cTng\rightarrow\cTL
$
is the functor between categories defined by the relation
\[
\xKbrBv{\xcrsp}
\;\;=\;\;
\qvh\;\;
\xKbrBv{\xpver}
\;\;+\;\;
\qvmh\;\;
\xKbrBv{\xphor}
\]
and by the rule that disjoint circles are converted into the factors $\mqpqi$. In particular, the \tKbr\ maps a framed link into its \tJpol.

Khovanov categorified the Jones polynomial as well as the \tKbr\ for tangles. We find it convenient to use Bar-Natan's canopoly version of this categorification with slight modifications. Thus we consider the 2-category $\ctTL$, whose objects are again the discs $\xDon$. For two objects $\xDom$ and $\xDon$ consider the $\ZZZtt$-graded additive category $\tHom_{\ctTL}(\xDom,\xDon)$ generated by objects
$\xKhl$ corresponding to \TLttnglmn s $\xlam$ and by their translations $\xKhl\hgrshklm$ with respect to the $\ZZZtt$-grading. Note that $k$ and $l$ may take half-integer values, but integer and half-integer degrees `do not mix', so we can safely pretend that the grading is $\ZZZtt$.

\subsection{The center of the 2-category $\ctTL$}

Our main goal is the study of the center of the 2-category $\ctTL$. Recall that an element $\elcf$ of the center $\Zcatv{\caC}$ of a category $\caC$ is a choice of an endomorphism $\elcfv{A}$ for every object $A$ of $\caC$ such that for any morphism $\gAB\in\Hom_{\caC}(A,B)$ there is an equality $\elcfv{B} \gAB = \gAB \elcfv{A}$.
 %If $\ctaC$ is a 2-category, then this equality should be a natural equivalence of functors.
 The elements of the center can be composed, so $\Zcatv{\caC}$ is a monoid and  the center $\Zcatv{\ctaC}$ of a 2-category $\ctaC$ is a monoidal category.

The centers of the \TQFT\ categories $\cTL$ and $\ctTL$ are especially interesting, since they are related to the module and, respectively, category associated within these \TQFT s to the 2-torus $\xTt$. To illustrate this point, consider first the centers of the topological categories $\cTng$ and $\ctTng$.

Let $\xIdn$ denote the \TLttnglnn\ corresponding to the $n$-strand identity braid. For a link $\xL$ in $\xStSo$, let $\xIdnL$ denote the \ttnglnn\ constructed by wrapping the link $\xL$ as a meridian (that is, as a band) around $\xIdn$. It is easy to see that a collection of endomorphisms $\xcztL$ defined as
$\xczt_{\xDon} = \xIdnL$ is an element in the center of both $\cTL$ and $\ctTL$, because the $\xL$-band can be slid along any tangle around which it is wrapped. From the \TQFT\ perspective, since the boundary of $\xStSo$ is $\xTt$, the solid torus $\xStSo$ containing a link determines an element (an object) in the module (category) associated with $\xTt$.

Notably, $\xczt$ are not the only elements (objects) in the center of topological categories. Let
$\gbrmn$ be the \ttnglnn\ corresponding to a \trbr\ with $m$ full \clckw\ rotations of $n$ strands having zero framing. Then a collection of endomorphisms

\section{Multi-cones}

Let $\caCt$ be an additive category generated freely by a finite set of objects, that is, the objects of $\caCt$ are finite sums of generators  (we have categories $\tHom_{\ctTL}(\xDom,\xDon)$ in mind). Let $\caC = \cKomm{\caCt} $ be the homotopy category of complexes bounded from above: an object of $\caC$ is a complex
$\cmA = (\cdots \rightarrow A_i \rightarrow A_{i+1} \rightarrow\cdots\rightarrow A_k )$ and morphisms are chain maps up to homotopy.

%$\xKhl\hgrshklm$, where $\xlam$ is a \TLttnglmn, while $\hgrshklm$ denotes translation of %$\ZZZtt$-degrees. %, which we denote

A \JWp\ $\jwpn$ is a special idempotent element of the $n$-strand
\TLa\ $\cTLn$, whose defining property is
the annihilation of cap and cup tangles.
% which annihilates cap and cup tangles.
The coefficients in its expression in terms of \TLb\ tangles are
rational (rather than polynomial) functions of $q$. This suggests
that the categorification $\ctjwn$ of $\jwpn$ in the
universal tangle category $\dTLn$ constructed by D.~Bar-Natan\cx{BN1}
should be presented by a semi-infinite \chcpl. In fact, there are
two mutually dual categorifications: the complex $\ctjwpn$ which is bound from
above and the complex $\ctjwmn$ which is bound from below. We will
consider only $\ctjwpn$ in detail, since the story of $\ctjwmn$ is totally
similar.

%
%Its
%expression in terms of \TLb\ tangles involves rational functions of
%$q$, which suggests that the categorification of $\jwpn$ in the
%universal tangle category $\dTLn$, constructed by D.~Bar-Natan\cx{BN1},
%should be presented by a semi-infinite \chcpl.

The construction of $\ctjwmn$ by
B.~Cooper and S.~Krushkal\cx{CK} is based upon the
Frenkel-Khovanov formula for $\jwpn$ and requires the invention of morphisms
between constituent \TL\ tangles as well as non-trivial `thickening'
of the complex. An alternative `representation-theoretic'
approach to the categorification of the \JWp\ is developed by Igor Frenkel,
Catharina Stroppel, and Joshua Sussan\cx{FSS}.

Our approach is rather straightforward: the
categorified projector $\ctjwpn$ is a direct limit of
appropriately shifted
categorification complexes of \cbr s
(\ie braid analogs of torus links) with high \clckw\ twist (the
other projector $\ctjwmn$ comes from high \cclckw\ twists).
The limit  $\ctjwpn$
can be presented as a cone:
\xlee{eq:int1}
\ctjwpn\hteqv
%\CnBv{\wbCmnp\rightarrow\cbrmns},
\CnBv{\Ohp\big(2m(n-1)\big)\longrightarrow\cbrmns},
\xeee
where
$\gbrmn$ is a \cbr\ with $m$ full \clckw\ rotations of $n$ strands,
%$\symcat{-}$ is the categorification complex,
$\symcats{-}$ is the
categorification complex with a special grading shift, and
$\Ohp(k)$ denotes a \chcpl\ which ends at the homological degree
$-k$. Theorem\rw{th:cnpr} imposes even stronger restrictions on
%$\wbCmnp$.
the complex $\Ohp\big(2m(n-1)\big)$ in \ex{eq:int1}.

%$\cbrmns$
%is an appropriately shifted categorification
%complex of a \cbr\ which induces $m$ full rotations of its
%$n$ strands,
%%has $m$ full twists,
%while
%%$\wbCmnp$
%$\Ohp\big(2m(n-1)\big)$
%is a
%\chcpl\ which starts only at the homological degree $2m(n-1)$.
%Theorem\rw{th:cnpr} imposes even stronger restrictions on
%%$\wbCmnp$.
%this complex.

The advantage of our approach is that one can use \cbr s with high
twist as approximations to $\ctjwpn$ in a computation of \Kh\ of a
spin network which involves \JWp s:
% (\ie a spin network):
if a spin network $\xnu$ is constructed by connecting $\jwpn$ to an \ttngnn\
$\xtau$ such that $\symcat{\xtau}\hteqv\Ohp(k)$, while a spin network $\xnum$ is constructed
by replacing $\jwpn$ in $\xnu$ with $\gbrmn$, then the homology of
$\symcat{\xnu}$ coincides with the shifted homology of
$\symcat{\xnum}$ in all homological degrees $i$
such that $i> -k - 2m(n-1)$. Thus one may say that
there is a stable limit
\xlee{eq:stlimsn}
\symcat{\xnu} =
\lim_{m\rightarrow+\infty}\symcats{\xnum}.
\xeee
We will define homological limits more precisely in
subsection\rw{sss.homcal}.

The practical
importance of the relation between $\symcat{\xnu}$
and $\symcat{\xnum}$ stems from the fact that $\xnum$ is an
ordinary link and its homology
can be computed with the help of
existing efficient computer programs even for high values
of $m$.
%
%
%the formula\rx{eq:int1}
%guarantees that replacing $\ctjwn$ with $\cbrmns$ does not change
%the homology of such a `link' up to a certain homological degree which
%goes to infinity as $m\rightarrow +\infty$.
%Such an approximation is useful, since an insertion of a \cbr\ $\gbrmn$
%creates an ordinary link whose homology can be computed by
%existing efficient computer programs which can handle high values
%of $m$.

The simplest example of a spin network  is the unknot
`colored' by the $(n+1)$-dimensional representation of
$\mathrm{SU}(2)$ with the help of the projector $\jwpn$. Its
\Kh\ is approximated by the homology
of torus links $\mathrm{T}_{n,-mn}$ which appear as cyclic closures of
$\gbrmn$. The \Kh\ of torus links has been studied by Marko Stosic\cx{St}, who
observed that it stabilizes at lower degrees as $m$ grows. This is
a particular case of the `stable limit'\rx{eq:stlimsn}.
% formula for a spin network:
%%
%\ylee{eq:stlimsn}
%\symcat{\xnu} =
%\lim_{m\rightarrow+\infty}\symcats{\xnum}.
%\yeee
%%

In Section\rw{s:notres}
we explain all notations and conventions
which are used in the paper. In particular, in
subsection\rw{sss:trgr} we define a non-traditional grading of
\Kh, which is convenient for our computations.
Then we formulate our results.

%In subsection\rw{ss:not} we explain all notations and conventions
%which are used in the paper. In particular, in
%subsection\rw{sss:trgr} we define a non-traditional grading of
%\Kh, which is convenient for our computations.
%In subsection\rw{ss:res} we formulate our results.

In Section\rw{s:elhomcal} we review basic facts about homological
`calculus' required to work with limits of sequences of
complexes in a homotopy category. In Section\rw{s:cbr} we construct
a sequence of categorification complexes of
\cbr s related by special \chmp s. This sequence yields $\ctjwpn$
as its direct limit. In Section\rw{s:prfs} we use homological
calculus of Section\rw{s:elhomcal} in order to prove that
$\ctjwpn$ is a categorification of the \JWp.

\def\bbS{ \mathbb{S} }
\def\So{ \bbS^1 }
\def\St{ \bbS^2 }
\def\Sot{ \So\times\St }

\subsection*{Acknowledgements}

This paper is a spinoff of a joint project with Mikhail
Khovanov\cx{KhRS} which is
dedicated to the study of categorification complexes of \cbr s and their
relation to the categorification of the Witten-Reshetikhin-Turaev
invariant of links in $\Sot$. I am deeply indebted to Mikhail for
numerous discussions and suggestions.

I would like to thank Slava Krushkal for sharing the results of
his ongoing research. I am also indebted to organizers of the M.S.R.I.
workshop `Homology Theories of Knots and Links' which stimulated
me to write this paper.

This work is supported by the NSF grant DMS-0808974.

\end{document}

\section{Notations and results}
\label{s:notres}
\subsection{Notations}
\label{ss:not}
\subsubsection{Tangles and \TLa}

All tangles in this paper are framed and we assume the blackboard
framing in pictures. We use the symbol %$\symfr\; k$
$\xygraph{
!{0;/r1.5pc/:}
[u(0.5)]
!{\xcapv@(0)}
[u(0.45)r(0.23)]
*{\symfr\;\scriptstyle{k}}
[u(1.5)]
%*{\bullet}
}
$to indicate an
addition of $k$ framing twists to a tangle strand:
\xlee{ae1.1b}
\xygraph{
!{0;/r1.5pc/:}
[u(0.5)]
!{\hover}
!{\hcap}
[u(0.5)l(0.25)]
%*{\bullet}
}
\;\; = \;\;
%-q^{\frac{3}{2}}\;\;
\xygraph{
!{0;/r1.5pc/:}
[u(0.5)]
!{\xcapv@(0)}
[u(0.45)r(0.23)]
*{\symfr\;\scriptstyle{1}}
[u(1.5)]
%*{\bullet}
}
\xeee

A tangle is called \emph{\plnr} if it can be presented by a diagram
without crossings. A \plnr\ tangle is called \emph{connected} or
\emph{\TLb} (\TLba) if
it does not contain disjoint circles. Let $\rTNG$ denote the set of all
framed tangles,
$\rTNGmn$ -- the set of \ttngmn s and $\rTNGn$ -- the set of
\ttngnn s.
We adopt similar notations for the set  $\rTL$ of \TLba-tangles.

We use the symbol $\tcmp$ to denote the composition of tangles:
$\xtauo\tcmp\xtaut$. The same symbol is used to denote the
multiplication in \TLa\ and the composition bifunctor in the
category $\dTL$.

A \TLa\ $\cTL$ over the ring of Laurent polynomials $\Zqqi$\footnote{It is clear from our normalization of the
Kauffman  bracket relation\rx{ae1.2} that we should rather use the
ring $\Zqqhi$. However, in all expressions in this paper the
half-integer power of $q$ appears only as a common factor, so the terms with integer
and half-integer powers of $q$ do not mix. Hence
we refer to $\Zqqi$, while keeping in mind that $\qh$ may
appear as a common factor is some expressions.}
is a quiver ring. The vertices $v_n$ of the quiver are indexed by
non-negative integers $n$ and each pair of vertices $v_m$, $v_n$,
such that $m-n$ is even, is connected
by an edge $e_{mn}$. To a vertex $v_n$ we associate a ring $\cTLnn$ (also denoted as
$\cTLn$)
and to an edge $e_{mn}$ we associate a
$\cTLn\otimes\cTLm^{\mathrm{op}}$-module $\cTLmn$. As a module,
$\cTLmn$ is generated freely by elements $\clam$ corresponding to \TL\ \ttngmn s $\xlam$, while ring
and module structures come from the composition of tangles modulo
the relation
\xlee{ae1.1}
%\widehat{\lcir}
\Bsymalg{\lcir}
 = -(\qpqi),
\xeee
which is needed to remove disjoint circles that may appear in the composition
of \TLb\ tangles.

The map $\rTNG\xrightarrow{\symalg{-}}\cTL$
associates an element $\ctau$ to a tangle $\xtau$ with the help of
\ex{ae1.1} and the Kauffman bracket relation
\xlee{ae1.2}
\Bsymalg{\xcrsp}
\;\;=\;\;
\qvh\;\;
\Bsymalg{\xpver}
\;\;+\;\;
\qvmh\;\;
\Bsymalg{\xphor}.
\xeee
This relation removes crossings and disjoint circles from the
diagram of $\xtau$, hence
\xlee{ae1.2a0}
\ctau = \sltln \xcalt\, \clam,\qquad
%\xcalt\in\Zqqi.\quad
\xcalt = \sum_{i\in\ZZ}\xcalit\,q^i
\xeee
with only finitely many coefficients $\xcalit$ being non-zero.
%
%A \TLa\ $\cTLn$ over the ring of Laurent polynomials
%$\Zqqi$\footnote{It is clear from our normalization of the
%Kauffman  bracket relation\rx{ae1.2} that we should rather use the
%ring $\Zqqhi$. However, in all expressions in this paper the
%half-integer power of $q$ appears only as a common factor, so the terms with integer
%and half-integer powers of $q$ do not mix. Hence
%we refer to $\Zqqi$, while keeping in mind that $\qh$ may
%appear as a common factor is some expressions.} is
%generated by elements $\clam$ corresponding to \TL\ \ttngnn s $\xlam$, multiplication
%coming from the composition of tangles, modulo the relation
%%
%\xlee{ae1.1x}
%%\widehat{\lcir}
%\Bsymalg{\lcir}
% = -(\qpqi).
%\xeee
%%
%which is needed to remove the disjoint circles that may appear in the composition
%of \TLb\ tangles. Equivalently, $\cTLn$ is generated by elements
%$\ctau$ corresponding to
%\ttngnn s $\xtau$ modulo the relation\rx{ae1.1} and the Kauffman bracket
%relation
%%
%\xlee{ae1.2x}
%\Bsymalg{\xcrsp}
%\;\;=\;\;
%\qvh\;\;
%\Bsymalg{\xpver}
%\;\;+\;\;
%\qvmh\;\;
%\Bsymalg{\xphor}.
%\xeee
%%
%In other words, there is a map
%$\rTNGn\xrightarrow
%%{\salg}
%{\amap}
%\cTLn$
%which maps a \ttngnn\ $\xtau$ into a \TLa\ element
%%
%\xlee{ae1.2a0x}
%\ctau = \sltln \xcalt\, \clam,\qquad
%%\xcalt\in\Zqqi.\quad
%\xcalt = \sum_{i\in\ZZ}\xcalit\,q^i
%\xeee
%%
%with only finitely many coefficients $\xcalit$ being non-zero.

If two tangles differ only by the framing of their strands, then
the corresponding algebra elements differ by the $q$
power factor coming from the following relation associated with
the first Reidemeister move:
\xlee{ae1.2a}
\Bsymalg{\xvfro\hspace*{-0.2cm}}
\;\; = \;\;
-q^{\frac{3}{2}}\;\;
\Bsymalg{\;\xvert\hspace*{-0.5cm}}
\xeee

A \ttngzz\ $\xL$ is a framed link, so $\symalg{\xL}$
%the application of the map $\amap$ to it
is
the framing dependent Jones polynomial defined by the
Kauffman bracket.

%Planar \ttngmn s modulo\rx{ae1.1} or, equivalently, \ttngmn s
%modulo\rx{ae1.1} and\rx{ae1.2} generate a bimodule $\cTLmn$ over
%$\cTLm^{\op}\otimes\cTLn$ and a composition of tangles produces a
%homomorphism
%$(\cTLkm\otimes_{\cTLm}\cTLmn)\rightarrow\cTLkn$. The total \TLa\ is the
%union of all algebras and bimodules $\cTL = \bigcup_{m,n}\cTLmn$,
%the product being the composition of tangles or zero if the
%valences do not match.

%and $\cTLn$.

We use the notations $\QcTL$ and $\cTLpinf$ for \TLa s defined over
the field $\Qq$ of rational functions of $q$ and over the field
$\Zsqqi$ of Laurent power series.
A sequence of injective homomorphisms
$\Zqqi\hookrightarrow\Qq\hookrightarrow\Zsqqi$, the latter one
generated by the expansion in powers of $q$,
produce a sequence of injective homomorphisms of the corresponding
\TLa s.% and tangle bimodules.
%We will also use the ring $\Zsqiq$
%of Laurent power series in $q^{-1}$ and the corresponding \TLa\
%$\cTLminf$.

%Expansion in powers of $q$
%produces an injective homomorphism $\Qq\rightarrow\Zsqqi$ and the
%latter generates an injective homomorphism of the corresponding
%\TLa s.
%
%Let $\Zqqi$ be the ring of Laurent polynomials in $q$, $\Zsqqi$ --
%the ring of formal Laurent series and $\Qq$ -- the ring of
%rational functions of $q$. We will define three `flavors' of \TLa\ over
%these rings

%Let $\Opqm$ denote any element of $\cTLpinf$ of the form
%$\sltln\sum_{i\geq m} \xcali\,q^i\,\clam$.
%We define a \emph{\qord} of an element $\yal\in\cTLpinf$ as
%$\yordq{\yal} = \xsupv{m\colon \yal = \Opqm}$.
%%
%
%\begin{definition}
%A sequence of elements
%%
%\wlee{ae1.2a1}
%\yal_1,\yal_2,\ldots\in\cTLpinf%
%\weee
%%
%has a limit
%$\lim_{k\rightarrow \infty} \yal_k = \ybet$, $\ybet\in\cTLpinf$
%if $\lmii\yordq{\ybet-\yal_k} = +\infty$.
%\end{definition}

\subsubsection{The \JWp}

Let $\gcupni\in\rTLvv{n-2}{n}$ and $\gcapni\in\rTLvv{n}{n-2}$,
$1\leq i\leq n-1$, denote the following \TL\ tangles:
\ylee{ae1.3}
%\underbrace{
\gcupni=\xygraph{
!{0;/r1.5pc/:}
[r(0.25)u(0.5)]
!{\xcapv@(0)}
[u(0.5)r(1)]
*{\cdots}
[r(01)u(0.5)]
!{\xcapv@(0)}
[r(0.5)u(1)]
!{\vcap-}
[r(1.5)]
!{\xcapv@(0)}
[u(0.5)r(1)]
*{\cdots}
[r(01)u(0.5)]
!{\xcapv@(0)}
[u(1.5)l(3.5)]
*{\scriptstyle{i}}
[r(1)]
*{\scriptstyle{i+1}}
[l(3.5)]
*{\scriptstyle{1}}
[r(6)]
*{\scriptstyle{n}}
}
%}
,
\quad\quad
\gcapni=
\xygraph{
!{0;/r1.5pc/:}
[r(0.25)u(0.5)]
!{\xcapv@(0)}
[u(0.5)r(1)]
*{\cdots}
[r(01)u(0.5)]
!{\xcapv@(0)}
[r(0.5)]
!{\vcap}
[r(1.5)u(1)]
!{\xcapv@(0)}
[u(0.5)r(1)]
*{\cdots}
[r(01)u(0.5)]
!{\xcapv@(0)}
[d(0.5)l(3.5)]
*{\scriptstyle{i}}
[r(1)]
*{\scriptstyle{i+1}}
[l(3.5)]
*{\scriptstyle{1}}
[r(6)]
*{\scriptstyle{n}}
}
\yeee
Their compositions $\xUni = \gcupni\tcmp \gcapni$ are standard
generators of the \TLa\ $\cTLn$.

The \JWp\ $\jwpn\in\QcTLn$ is the unique non-trivial idempotent element satisfying the
condition
\xlee{ae1.4}
\acapni\;\tcmp\jwpn =0,\qquad 1\leq i\leq n-1.
\xeee
The \JWp\ also satisfies the relation
\xlee{ae1.4a}
\jwpn\tcmp\;\acupni =0,\qquad 1\leq i\leq n-1.
\xeee

We denote the idempotent element of
$\cTLpinf_n$  corresponding to $\jwpn$ as $\jwpnp$.
% denotes the corresponding idempotent element of $\cTLpinf_n$.

\subsubsection{Basic notions of homological algebra}
Let $\xChA$ be a category of \chcpls\ associated with an additive
category $\xctA$. An object of $\xChA$ is a \chcpl\
\ylee{ae1.ch1}
\xbA  = (\cdots \rightarrow
\xAi\xrightarrow{\xdi}\xAio\rightarrow\cdots),
\yeee
and a morphism
between two chain complexes is a \chmp\ defined as a \mmp
\xlee{ae1.10d}
\vcenter{\xymatrix{
\xbA \ar[d]^-{\xbf} &&
\cdots\ar[r]^-{\xdimo} & \xAi \ar[r]^-{\xdi} \ar[d]^-{\yfi} &
\xAio
\ar[r]^-{\xdio} \ar[d]^{\yfio} & \cdots
\\
\xbB &&
\cdots\ar[r]^-{\xdpimo} & \xBi \ar[r]^{\xdpi} & \xBio
\ar[r]^-{\xdpio} & \cdots
}
}
%,\qquad\qquad
%\xdpio\,\yfio = \yfi\,\xdio
\xeee
which commutes with the chain differential: $\xdpi\,\yfi = \yfio\,\xdi$ for all $i$.
The cone of a \chmp\ $\xbA\xrightarrow{\xbf}\xbB$ is a complex
\ylee{ae1.10b1}
\Cnbf
=
\lrbc{
\vcenter{
\xymatrix@C=1.5cm@R=0.5cm{
\cdots \ar[dr] \ar[r] & \xAi
\ar@{}[d] |{\oplus} \ar[r]^-{-\xdi} \ar[dr]^{-\xfi} &
\xAio
\ar@{}[d] |{\oplus}
\ar[r] \ar[dr]& \cdots
\\
\cdots \ar[r] & \xBimo\ar[r]_{\xdpimo} & \xBi \ar[r] & \cdots
}
}
}
\yeee
in which the object $\xAio\oplus\xBi$ has the homological degree
$i$.
There are two special \chmp s
%$\chdlbf\colon
$\xbB\xrightarrow{\idlbf}\Cnbf$ and
$\Cnbf[1]\xrightarrow{\chdlbf}\xbA$ associated to
the cone:
\ylee{ae1.10b2}
\xymatrix{
\xbB \ar[d]^-{\idlbf}&&
\cdots \ar[r] &
\xBi \ar[r] \ar[d]^-{0\oplus \xId}
&
\xBio \ar[r] \ar[d]^-{0\oplus \xId}
&
\cdots
\\
\Cnbf \ar[d]^-{\chdlbf} &&
\cdots \ar[r] &
\xAio \oplus \xBi \ar[r] \ar[d]^-{\xId\oplus 0} &
\xAit\oplus \xBio \ar[r] \ar[d]^-{\xId\oplus 0} &
\cdots
\\
\xbA[-1] &&
\cdots \ar[r]
&
\xAio \ar[r]
&
\xAit \ar[r]
&
\cdots
}
\yeee
These complexes and \chmp s form a \dstt:
\xlee{ae1.ch2}
\xymatrix{
\xbA\ar[r]^-{\xbf} &
\xbB \ar[r]^-{\idlbf} &
\Cnbf \ar[r]^-{\chdlbf} &
\xbA[-1]
}.
\xeee

The homotopy category of complexes $\xKhA$ has the same objects as
$\xChA$ and the morphisms are the morphisms of $\xChA$ modulo
homotopies.
%
%the objects and
%morphisms of $\xChA$ distinguished only up to homotopy
%equivalence.
%
We denote
homotopy equivalence by the sign $\hteqv$.
The notion of a cone extends to $\xKhA$ and there
are additional relations in that category: $\Cnv{\idlbf} \hteqv \xbA[-1]$ and
$\Cnv{\chdlbf} \hteqv \xbB[-1]$, so all vertices of a \dstt\ have
equal properties.

\subsubsection{A triply graded categorification of the Jones
polynomial}
\label{sss:trgr}
In his famous paper\cx{Kh1}, M.~Khovanov
%In\cx{Kh1} the first author (M.K.)
introduced a categorification of
the Jones polynomial of links. To a diagram $\xL$ of a
link he associates a complex of graded modules
\xlee{ae1.5}
\dL = \lrbc{ \cdots \rightarrow \dLi \rightarrow \dLio\rightarrow\cdots}
\xeee
so that
if two diagrams represent the same link then the corresponding
complexes are homotopy equivalent, and the graded Euler
characteristic of $\dL$ is equal to the Jones polynomial of $\xL$.

Thus, overall, the complex\rx{ae1.5} has two gradings: the first one
is
the grading related to powers of $q$ and the second one is the
homological grading of the complex itself, the corresponding
degree being equal to $i$.
In this paper we adopt a slightly different convention which is
convenient for working with framed links and tangles. It is
inspired by matrix factorization categorification\cx{KR1} and its
advantage is that it is no longer necessary to assign orientation to
link strands in order to obtain the grading of the categorification
complex\rx{ae1.5} which would make it invariant under the second
Reidemeister move.

To a framed link
diagram $\xL$ we associate a $\ZZ \oplus\ZZ\oplus\ZZ_2$-graded complex\rx{ae1.5} with
degrees $\dgo$, $\dgt$ and $\dgh$.
The first two gradings are of the same nature as in\cx{Kh1} and, in
particular, $\dgo\dLi=i$. The third grading is an inner grading of
chain modules defined modulo 2 and of homological
nature, that is, the homological parity of an element of $\dL$,
which affects various sign factors, is the sum of $\dgo$ and
$\dgh$. Both homological degrees are either integer or
half-integer simultaneously, so the homological parity is integer
and takes values in $\ZZ_2$. The $q$-degree $\dgt$ may also take
half-integer values.

%The
%first grading is the homological grading $i$: $\dgo\dLi=i$.
%The second and third gradings are inner gradings of individual chain modules
%$\dLi$. The second grading
%is the $\ZZ$-grading associated with powers of $q$. The third grading is
%of homological nature and it is defined only modulo 2. The first
%and third degrees  take integer or half-integer values simultaneously, so
% their sum, which is the total homological grading, is integer
%and takes values in $\ZZ_2$.

Let $\tgrshv{l}{m}{n}$ denote the shift of three degrees by $l$,
$m$ and $n$ units respectively\footnote{
Our degree shift is defined in such a way that if an object $M$
has a homogeneous degree $n$, then the shifted object $M[1]$ has a
homogeneous degree $n+1$.
}. We use abbreviated notations
$$
\tgrsshv{m}{l} = \tgrshv{m}{l}{0},\qquad
\qshv{m} = \tgrshv{m}{0}{0}
$$
as well as the following `power' notation:
$$
\tgrshv{m}{l}{n}^k = \tgrshv{km}{kl}{kn}.
$$

With new grading conventions, the categorification
formulas of\cx{Kh1} take the following form:
the module associated with an unknot is still $\ZZ[x]/(x^2)$ but with
a different degree assignment:
\begin{eqnarray}
\label{ae1.6}
&
\Bsymcat{\lcir}=\ZZ[x]/(x^2)\,
%[0,-1,1]
\tgrshv{-1}{0}{1},
%\qquad
\\
&\dgt 1 = 0, \quad \dgt x = 2,
\quad\dgo 1 = \dgo x = \dgh 1=\dgh x =0,
\end{eqnarray}
and the categorification complex of a crossing is the same as
in\cx{Kh1} but with a different degree shift:
\xlee{ae1.7}
%\xygraph{
%!{0;/r1.5pc/:}
%[u(0.5)]
%!{\xoverv}
%[u(1.5)r(0.53)]
%*{\smcat}
%}
\Bsymcat{\xcrsp}
\;\;=\;\;
\Bigg(\;\;
%\vcenter{{
%%%%%%\begin{CD}
%\xygraph{
%!{0;/r1.5pc/:}
%[u(0.5)]
%!{\xunoverv}
%[u(1.5)r(0.5)]
%*{\smcat}
%}
\Bsymcat{\xpver}
\;\tgrshv{\vthf}{-\vthf}{\vthf}
%\;\;+\;\;
%%%%%@>\xmrf>>
\xrightarrow{\;\;\;\;\xmrf\;\;\;\;}
%\xygraph{
%!{0;/r1.5pc/:}
%[u(0.5)]
%!{\xunoverh}
%[u(0.5)l(0.5)]
%*{\smcat}
%}
\Bsymcat{\xphor}
\;\tgrshv{-\vthf}{\vthf}{-\vthf}
\vspace*{18pt}
%%%%%%%\end{CD}
%}}
\;\;
\Bigg),
\xeee
where $f$ is either a multiplication or a comultiplication of the
ring $\ZZ[x]/(x^2)$ depending on how the arcs in the \rhs are
closed into circles.
The resulting categorification complex\rx{ae1.5} is invariant
up to homotopy under the second and third Reidemeister moves, but
it acquires a degree shift under the first Reidemeister move:
\xlee{ae1.8}
%\xygraph{
%!{0;/r1.5pc/:}
%[u(0.5)]
%!{\hover}
%!{\hcap}
%[u(0.5)l(0.25)]
%*{\smcat}
%}
%\;\; = \;\;
%%-q^{\frac{3}{2}}\;\;
%\xygraph{
%!{0;/r1.5pc/:}
%[u(0.5)]
%!{\xcapv@(0)}
%[u(1.5)]
%*{\smcat}
%}
%\!\!\!\!
%\tgrshv{\vthh}{\vthf}{\vthf}.
%%
\Bsymcat{\xvfro\hspace*{-0.2cm}}
\;\; = \;\;
\Bsymcat{\;\xvert\hspace*{-0.5cm}}
\tgrshv{\vthh}{-\vthf}{-\vthf}.
\xeee
It is easy to see that the whole categorification complex\rx{ae1.5} has a
homogeneous degree $\dgh$.
% which is equal to the sum of all linking
%numbers of $\xL$ (including the self-linking numbers):
%$\dgh\dL = $.

\subsubsection{A universal categorification of the \TLa}
D.~Bar-Natan\cx{BN1} described the universal category $\dTL$, whose
Grothendieck \Kzg\ is $\cTL$ considered as a $\Zqqi$-module.
We will use this category with obvious adjustments required by the new
grading conventions.

Let $\dTLt$ be an additive category whose objects are in
one-to-one correspondence with \TLb\ tangles, morphisms being generated
by tangle cobordisms (see\cx{BN1} for details). The universal category
$\dTL$ is the homotopy category of bounded complexes associated with
$\dTLt$. In other words, an object of $\dTL$ is a complex
\xlee{ae1.8a}
\xbC =
\lrbc{\cdots\rightarrow\xCi\rightarrow\xCipo\rightarrow\cdots},\qquad
\xCi =
\bigoplus_{j,\xmu}%_{\substack{j\in\ZZ\\ \xmu\in\ZZ_2+\xmuz}}
\oltln \cjilam\,\dlam \tgrshv{j}{0}{\mu},
\xeee
where
%$\xmuz=0$ or $\tfrac{1}{2}$, while
non-negative integers $\cjilam$ are multiplicities; since the
complex is bounded, they are non-zero for only finitely many
values of $i$.

%
%In other words, a complex as a whole is a sum
%%
%\xlee{ae1.8ab}
%\bigoplus_{\substack{i,j\in\ZZ \\ \xmu\in\ZZ_2}} % \\ \xlam\in\rTLn}}
%\sltln \cjilam\,\dlam \tgrshv{j}{i}{\mu},
%\xeee
%%
%where non-negative numbers $\cjilam$ are multiplicities; since the
%complex is bounded, they are non-zero for only finitely many
%values of $i$.

A categorification map $\rTNG\xrightarrow{\mpcat}\dTL$ turns a framed
tangle diagram $\xtau$ into a complex $\dtau$ according to the
rules\rx{ae1.6} and\rx{ae1.7}, the morphism $\xmrf$ in the
complex\rx{ae1.7} being the saddle cobordism. A composition of
tangles becomes a composition bi-functor
$\dTL\times\dTL\rightarrow\dTL$ if we apply
the categorified version of the rule\rx{ae1.1} in order to remove
disjoint circles:
\xlee{ae1.01}
%\xlam\sqcup\chlcir
% = \xlam\tgrshv{1}{0}{1} + \xlam\tgrshv{-1}{0}{1}.
%\chlcir
\Bsymcat{\lcir}= \cnot \tgrshv{1}{0}{1} + \cnot\tgrshv{-1}{0}{1},
\xeee
where $\xnot$ is the empty \TL\ \ttngzz.

A complex $\dtau$ associated to a tangle $\xtau$ is defined only up to
homotopy. We use a notation $\spcc{\dtau}$ for a particular complex
with special properties which represents $\dtau$.
%within the homotopy class $\dtau$, we use a no which has
%
%We use a notation $\spcc{\dtau}$ for a particular
%complex representing $\dtau$.

Overall,
we have the following commutative diagram:
%
%\xlee{ae1.9}
\begin{equation}
\xymatrix@C=1.5cm@R=0.3cm{
& {}\dTL \ar[dd]^{\Kz}
\\
{}\rTNG \ar[ur]^{\mpcat} \ar[dr]^{\mpalg}
\\
& {}\cTL
}
\end{equation}
where the map $\Kz$ turns the complex\rx{ae1.8a} into the
sum\rx{ae1.2a0}:
\xlee{ae1.9a}
\Kz(\xbC)
%\Kz\bigg(\bigoplus_{\substack{i,j\in\ZZ \\ \xmu\in\ZZ_2}}
%\sltln \cjilam\,\dlam \tgrshv{j}{i}{\mu}\bigg)
=
\sltln\sum_{j}
%\in\ZZ}
\xcalj \,q^j\,\clam,
\qquad
\xcalj = \sum_{i,\xmu}
%{\substack{i\in\ZZ \\ \xmu\in\ZZ_2}}
(-1)^{i+\xmu}\, \cjilam.
%\sum_{\substack{i,j\in\ZZ \\ \xmu\in\ZZ_2}} \sltln (-1)^{i+\xmu}
%\cjilam\,q^{j}\,\clam.
\xeee
Since the complex is bounded, the sum in the expression for
$\xcalj$ is finite.

%In addition to $\dTL$ we consider
%% two categories of semi-bounded complexes:
%a category $\dTLp$ of complexes
%bounded from below and a category $\dTLm$ of complexes bounded
%from above (that is, the multiplicity coefficients in the
%sum\rx{ae1.8a} are zero for $i$ below or above certain value).
%We say that a complex of $\dTLp$ is \emph{\qpb} if the multiplicity
%coefficients $\cjilam$ are zero for $j<i$ and a complex of $\dTLm$
%is \emph{\qmb} if $\cjilam$ are zero for $j>i$. For a $q^+$ or a \qmb\
%complex, the sum in the expression\rx{ae1.9a} for $\xcalj$ is
%finite and the map $\Kz$ is well defined.

In addition to $\dTL$ we consider
% two categories of semi-bounded complexes:
a category $\dTLp$ of complexes
bounded from above, that is, the multiplicity coefficients in the
sum\rx{ae1.8a} are zero if $i$ is greater than certain value.
Define the $q^+$ order of a \qcmd\ $\xCi$:
$\yordq{\xCi} = \xinfv{j\colon \exists\mu\colon\cjilam\neq
0}$.
A complex $\xbC$ in $\dTLp$ is \emph{\qpb} if $\lmii\yordq{\xCmi} =
+\infty$.
%, where $\yordq{\xCmi} = \xinfv{j\colon \exists\mu\quad\text{such that}\quad\cjilam\neq 0}$.
%
%For a complex $\xbC$ in $\dTLp$ we define $\xbnbi(\xbC) = \xinfv{j\colon \exists\xlam\text{ such that }
%\cjilam\neq 0} $. A complex $\xbC$ is \emph{\qpb} if
%$\lmii\xbnbi(\xbC) = +\infty$.
For a \qpb\ complex,
the sum in the expression\rx{ae1.9a} for $\xcalj$ is
finite, hence the element $\Kz(\xbC)$ is well defined.

%We say that a complex of $\dTLp$ is \emph{\qpb} if the multiplicity
%coefficients $\cjilam$ are zero for $j<i$ and a complex of $\dTLm$
%is \emph{\qmb} if $\cjilam$ are zero for $j>i$. For a $q^+$ or a \qmb\
%complex, the sum in the expression\rx{ae1.9a} for $\xcalj$ is
%finite and the map $\Kz$ is well defined.

\subsection{Results}
\label{ss:res}

\subsubsection{Infinite \cbr\ as a \JWp\ in a \TLa}
A braid with $n$ strands is a particular example of a \ttngnn.
A \emph{\cbr} is a braid
%is called \emph{\cylb} if it
that can be drawn on a cylinder $\So\times[0,1]$
without intersections. In fact, all \cbr s have the form
$\btcyln^m$, $m\in\ZZ$, where $\btcyln$ is the elementary
clockwise winding \cbr:
\xlee{ae1.10p}
\btcyln \;\;=\;\;
\xygraph{
!{0;/r1.5pc/:}
!{\vover}
[u(1.5)l(0.5)]
!{\xbendr[0.5]}
[u(1.25)l(1.25)]
!{\xbendd[-0.5]}
[u(1.25)l(1)]
!{\xcapv[0.25]@(0)}
[r(1.75)]
!{\xcaph@(0)}
[u(1)]
!{\vover[-1]}
[r(1)]
!{\xbendr[-0.5]}
[u(0.75)l(0.75)]
!{\xbendd[0.5]}
[u(0.5)l(0.5)]
!{\xcapv[0.25]@(0)}
[u(1)l(1.5)]
*{\cdots}
[u(1.5)l(1.5)]
*{\cdots}
[u(0.75)l(1.5)]
*{\scriptstyle{1}}
[r(2.5)]
*{\scriptstyle{n-1}}
[r(1.25)]
*{\scriptstyle{n}}
[d(3)l(3)]
*{\scriptstyle{1}}
[r(1)]
*{\scriptstyle{2}}
[r(2.75)]
*{\scriptstyle{n}}
[u(2.2)l(0.35)]
%%%%%%% negative framing !!!!!!!!
%*{\symfr\;\scriptstyle{-1} }
}
\xeee
%
%
%
%
%\ylee{ae1.10a}
%\xygraph{
%!{0;/r1.5pc/:}
%!{\vunder}
%[u(1.5)l(0.5)]
%!{\xbendr[0.5]}
%[u(1.25)l(1.25)]
%!{\xbendd[-0.5]}
%%
%[u(1.25)l(1)]
%!{\xcapv[0.25]@(0)}
%%
%[r(1.75)]
%!{\xcaph@(0)}
%[u(1)]
%!{\vunder[-1]}
%[r(1)]
%!{\xbendr[-0.5]}
%[u(0.75)l(0.75)]
%!{\xbendd[0.5]}
%%
%[u(0.5)l(0.5)]
%!{\xcapv[0.25]@(0)}
%%
%[u(1)l(1.5)]
%*{\cdots}
%[u(1.5)l(1.5)]
%*{\cdots}
%%
%[u(0.75)l(1.5)]
%*{\scriptstyle{1}}
%[r(2.5)]
%*{\scriptstyle{n-1}}
%[r(1.25)]
%*{\scriptstyle{n}}
%%
%[d(3)l(3)]
%*{\scriptstyle{1}}
%[r(1)]
%*{\scriptstyle{2}}
%[r(2.75)]
%*{\scriptstyle{n}}
%}
%\yeee
%%
%
%The negative framing twist is added to maintain a natural
%cylindrical framing: the ribbons representing framed strands of
%the braid can be drawn on the cylinder.
%
We introduce a special notation for the \cbr\ which corresponds to
$m$ full rotations of $n$ strands:
\ylee{eq:brdf}
\gbrmn = \btcyl^{mn}.
\yeee

Let $\Opqm$ denote any element of $\cTLpinf$ of the form
$\sltln\sum_{j\geq m} \xcalj\,q^j\,\clam$.
We define a \emph{\qord} of an element $\yal\in\cTLpinf$ as
$\yordq{\yal} = \xsupv{m\colon \yal = \Opqm}$.

\begin{definition}
\label{df:qlm}
A sequence of elements
\wlee{ae1.2a1}
\yal_1,\yal_2,\ldots\in\cTLpinf%
\weee
has a limit
$\lim_{k\rightarrow \infty} \yal_k = \ybet$,
%$\ybet\in\cTLpinf$
if $\lmii\yordq{\ybet-\yal_k} = +\infty$.
\end{definition}

The following theorem may be known, so we do not claim
credit for it. It appears here as a by-product
and it is an easy corollary of \ex{ae2.m4}.
\begin{theorem}
\label{th:alg}
The \TL\ element corresponding to the infinite \cbr\ equals the
\JWp:
\xlee{ae1.9}
\lim_{m\rightarrow+\infty} q^{\vthf mn(n-1)}\abrmn = \jwpnp,
%\qquad\text{in $\cTLpinf$}.
%%%,\qquad\lim_{m\rightarrow+\infty}(\btcylan)^{-mn} = \jwpn
%%%\qquad\text{in $\cTLminf$}.
\xeee
where $\jwpnp\in \cTLpinf_n$ corresponds to the \JWp\
$\jwpn\in\QcTLn$.
\end{theorem}
In fact, a more general statement is also true:
\xlee{ae1.10}
\lim_{m\rightarrow+\infty} q^{\vthf m(n-1)}\symalg{\btcyln^m} =
\jwpnp,
%\qquad\text{in $\cTLpinf$}.
%,\qquad\lim_{m\rightarrow+\infty}(\btcylan)^{-m} = \jwpn
%\qquad\text{in $\cTLminf$}.
\xeee
but its proof is more technical and we omit it here.

\subsubsection{A bit of homological calculus}
\label{sss.homcal}

Let $\xKhA$ denote the homotopy category of complexes associated
with an additive category $\xctA$ (we have in mind a particular case of $\xKhA = \dTLp$).
%An object of $\xKhA$ is a \chcpl\
%$\xbA  = (\cdots \rightarrow
%\xAio\xrightarrow{\xdio}\xAi\rightarrow\cdots)$ and a morphism
%between two chain complexes is a \chmp
%%
%\xlee{ae1.10dx}
%\vcenter{\xymatrix{
%\cdots\ar[r]^-{\xdit} & \xAio \ar[r]^-{\xdio} \ar[d]^-{\yfio} & \xAi
%\ar[r]^-{\xdi} \ar[d]^{\yfi} & \cdots
%\\
%\cdots\ar[r]^-{\xdpit} & \xBio \ar[r]^{\xdpio} & \xBi
%\ar[r]^-{\xdi} & \cdots
%}
%},
%\qquad\qquad
%\xdpio\,\yfio = \yfi\,\xdio
%\xeee
%%
%%($\xdpio\yfio = \yfi\,\xdio$ for all $i$)
%distinguished up to homotopy
%equivalence.

A \chcpl\ is considered `homologically small' if it ends at a low
(that is, high negative)
homological degree.
Let
$\Ohpm$ denote a complex which ends at $(-m)$-th homological
degree: $\Ohpm = (\cdots \xAv{-m-1} \rightarrow\xAv{-m})$. We define
a homological order of a complex $\xbA$ as $\yordh{\xbA} =
\xsupv{m\colon\xbA\hteqv\Ohpm}$.

Two complexes connected by a \chmp: $\xbA\xrightarrow{\xbf}\xbB$
are considered `homologically close' if $\Cnbf$ is homologically
small.

A \emph{\chsq} is a sequence of complexes connected by \chmp s:
\ylee{ae1.10d1}
\scA = (\xbAz\xrightarrow{\xbfz} \xbAo
\xrightarrow{\xbfo}\cdots).
\yeee
\begin{definition}
\label{df:cauchy}
A \chsq\ $\scA$ is \emph{\Cch} if $\lmii \yordh{\Cnbfi} = \pinft$.
\end{definition}
\begin{definition}
\label{df:sqlm}
A \chsq\ has a limit
\footnote{This definition differs
from the standard categorical definition of a direct limit, however
Theorem\rw{pr:spmp} indicates that our definition implies the standard one. We expect that
both definitions are equivalent.}
: $\dlm\scA = \xbA$, where $\xbA$ is a \chcpl, if
there exist \chmp s $\xbAi\xrightarrow{\xbtfi}\xbA$ such that
they form commutative triangles
\xlee{ae1.10e1}
\cmtr{\xbfi}{\xbtfio}{\xbtfi}{\xbAi}{\xbAio}{\xbA}
\xeee
%
%$\xbtfi \hteqv \xbtfio\,\xbfi$
and $\lmii\yordhr{\Cnv{\xbtfi}} =\pinft$.
\end{definition}

In Section\rw{s:elhomcal} we prove the following homology versions
of standard theorems about limits
(Propositions\rw{pr:chlm},\rw{pr:lmch} and\rw{pr:lmun}):
\begin{theorem}
\label{th:lmt}
A \chsq\ $\scA$ has a limit if and only if it is \Cch.
\end{theorem}
\begin{theorem}
\label{th:lmt2}
The limit of a \chsq\ is unique up to homotopy equivalence.
%If $\scA$ has a limit, then the limit is unique up to homotopy.
\end{theorem}

\subsubsection{Infinite \cbr\ as a \JWp\ in the universal category}

For a tangle diagram $\xtau$ let
$\dtaus$ denote the categorification complex $\dtau$
with a degree shift proportional to the number $\crsv{\xtau}$ of crossings
in the diagram $\xtau$:
\xlee{ae1.10b1}
\dtaus = \dtau\tgrshv{\vthf}{-\vthf}{\vthf}^{\crsv{\xtau}}.
\xeee
%
%where $\crsv{\xtau}$ is the number of crossings in
%the diagram $\xtau$.

In subsection\rw{ss:brchsq} we define a
\chsq\ of categorification complexes of \cbr s connected by special
\chmp s
%
%sequence of special \chmp s
%
%$\btcylcnmns \xrightarrow{\mrfm}  \btcylcnmpons$ which form a \chsq
%
%\xlee{ae1.10cx}
%\xctBn =
%\Big(\xymatrix@C=0.6cm{
%{}\cIdbn \ar[r]^-{\mrfz}
%& {}\btcylcnns \ar[r]^-{\mrfo}
%& \cdots \ar[r]^-{\mrfmmo}&{}
%\btcylcnmns
%%\symcats{\gbrmn}
%\ar[r]^-{\mrfm}
%& {}\btcylcnmpons \ar[r]^-{\mrfmo} &\cdots
%}\Big).
%\xeee
%
\begin{multline}
\label{ae1.10c}
\xctBn =
\Big(
\cidbrn
\xratv{\mrfz}
\cbrons \xratv{\mrfo}
\cdots
\\
\cdots
\xrightarrow{\mrfmmo}
\cbrmns \xraov{\mrfm}
\cbrmons \xrightarrow{\mrfmo}\cdots\Big).
\end{multline}
We prove that
%$\Cnv{\mrfm} \hteqv \Ohp\big(2m(n-1) + 1\big)$,
$\yordhr{\Cnv{\mrfm}}\geq2m(n-1) + 1$,
so
$\xctBn$ is a \Csq\ and by  Theorem\rw{th:lmt} it has a unique limit:
$\dlm\xctBn =\ctjwpn \in\dTLnp $.

\begin{theorem}
\label{th:enum}
The limiting complex $\ctjwpn$ has the following properties:
\begin{enumerate}
%\item
%There exists a \odct\ \otbl\ complex $\xbCn$ such that
%$\xbCn\tgrsshv{1}{1}\hteqv\Cnv{\xbtfz}$.

\item A composition of $\ctjwpn$ with cap- and \uptg s is contractible:
\ylee{eq:auc}
\ccapni \;\tcmp\ctjwpn \hteqv \ctjwpn\tcmp\; \ccupni\hteqv 0.
\yeee
\item The complex $\ctjwpn$ is idempotent with respect to tangle composition:
$\ctjwpn \tcmp\ctjwpn
\hteqv \ctjwpn$.

\end{enumerate}
\end{theorem}

We provide a glimpse into the structure of $\ctjwpn$.
A complex $\xbC$  in $\cTLn$ is called \emph{\odct} if
%all
%multiplicities $\xcv{j,i,\xIdbn}$ in \ex{ae1.8a} are zero, that is,
%the tangle
$\gidbrn$ never appears in \qcmds\ $\xCi$.
A complex $\xbC$ in $\cTLn$ is called \emph{\otbl} if the multiplicities
$\cjilam$ of \ex{ae1.8a} satisfy the property
\xlee{ae2.m1}
%\text{$\cjilam=0$ for $j<i$ and for $j>2i$}
\cjmilam=0\qquad\text{if $i<0$, or $j<i$, or $j>2i$.}
\xeee
%

%Let
%$$\cidbrn\xraov{\xbtfz}\ctjwpn,\qquad
%\cbrons\xraov{\xbtfo}\ctjwpn,\ldots\quad
%$$
%%

Let $\cbrmns\xraov{\xbtfm}\ctjwpn$ be \chmp s associated
with the limit $\dlm\xctBn = \ctjwpn$ in accordance with
Definition\rw{df:sqlm}.
\begin{theorem}
\label{th:cnpr}
There exist  \odct\ \otbl\ complexes $\wbCmn$ such that
$$\Cnv{\xbtfm}\hteqv\wbCmn\spshmnm\tgrsshv{1}{-1}.$$
%\tgrsshv{2mn+1}{2m(n-1)+1}$.
\end{theorem}
\noindent
In other words, there exists a distinguished triangle
%%
%\ylee{ae2.m2x}
%\wbCmnp\qsho \xratv{\chdlbtfm} \cbrmns \xrahv{\xbtfm} \ctjwpn
%\xrightarrow{\;\;\;\;\;\;\;} \wbCmnp\tgrsshv{1}{1}
%\yeee
%%
%
\ylee{ae2.m2}
\wbCmn\spshmnm\qsho \xratv{\chdlbtfm} \cbrmns \xrahv{\xbtfm} \ctjwpn
\xrightarrow{\;\;\;\;\;\;\;} \wbCmn\spshmnm\tgrsshv{1}{-1}
\yeee
so there is a presentation
%%
%\xlee{ae2.m4y}
%\ctjwpn \hteqv \CnBv{ \wbCmnp\qsho
%\xrahv{\chdlbtfm} \cbrmns},
%\xeee
%%
%
\xlee{ae2.m4}
\ctjwpn \hteqv \CnBv{ \wbCmn\spshmnm\qsho
\xrahv{\chdlbtfm} \cbrmns},
\xeee
%
%where $\wbCmnp = \wbCmn\spshmn$,
where the complex $\wbCmn$ is
\odct\ and \otbl.
%%
%\ylee{ae2.m3x}
%\ctjwpn \hteqv \CnBv{ \wbCmn\tgrsshv{2mn+1}{2m(n-1)}
%\xrahv{\chdlbtfm} \cbrmns}.
%\yeee
%%

At $m=0$ the formula\rx{ae2.m4} becomes
\xlee{ae2.m5}
\ctjwpn \hteqv \CnBv{ \wbCzn\qsho
\xrahv{\chdlbtfz} \cidbrn},
\xeee
where the complex $\wbCzn$ is \odct\ and \otbl.

%
%Let $\cidbrn\xraov{\xbtfz}\ctjwpn$ be the first \chmp\
%associated with the limit $\lim\xctBn =\ctjwpn$ in
%accordance with Definition\rw{df:sqlm}.
%
%\begin{theorem}
%\label{th:cnpr_x}
%There exists a \odct\ \otbl\ complex $\xbCn$ such that
%$\Cnv{\xbtfz}\hteqv\xbCn\tgrsshv{1}{1}$.
%%$\hteqv$.
%\end{theorem}
%%
%In other words, there is a \dstt
%%%
%%\ylee{ae2.m2x}
%%\xymatrix{
%%\xbCn\qsho \ar[r]^-{\chdlbtfz} &
%%{}\cIdbn \ar[r]^-{\xbtfz} &
%%\ctjwpn \ar[r] &
%%\xbCn\tgrsshv{1}{1}
%%}
%%\yeee
%%%
%%
%\ylee{ae2.m2x}
%\xbCn\qsho \xratv{\chdlbtfz} \cidbrn \xrahv{\xbtfz} \ctjwpn
%\xrightarrow{\;\;\;\;\;\;\;} \xbCn\tgrsshv{1}{1}
%\yeee
%%
%and $\ctjwpn \hteqv \Cnchdlbtfz$.

Since $\wbCzn$ is \otbl, the complex $\Cnchdlbtfz$ is
also \otbl\ and consequently \qpb. Hence $\Kctjwpn$ is well-defined. Also
$\Kctjwpn\neq 0$, because it contains $\cidbrn$ with coefficient 1.
Theorem\rw{th:enum} indicates that
$\Kctjwpn$ satisfies
defining properties of the \JWp, hence by uniqueness it is the \JWp:
\begin{corollary}
The complex $\ctjwpn$ categorifies the \JWp\ in
$\cTLpinf$:
\xlee{eq:catKz}
\Kctjwpn = \jwpn.
\xeee
\end{corollary}

%An obvious corollary of this theorem is that $\ctjwpn$ is
%homotopy equivalent to \qpb\ complex, hence
%$\Kctjwpn$ is well-defined. It also implies that
%$\Kctjwpn$ is non-trivial: it contains $\aIdbn$ with
%coefficient 1. The properties 2 and 3 indicate that $\Kctjwpn$ satisfies
%defining properties of the \JWp, hence by uniqueness it is the \JWp:
%%
%\begin{corollary}
%The complex $\ctjwpn$ categorifies the \JWp\ in
%$\cTLpinf$: $\Kctjwpn = \jwpn$.
%\end{corollary}

%%
%\begin{theorem}
%There exists a limit
%%
%\xlee{ae1.11}
%\lmp{m\rightarrow\infty}
%\lrbc{\btcylcn}^{mn} \sht
%%\,\shmn
%%\tgrshv{\sfnh+1}{\sfnh}{-\sfnh+1}}^{mn}
%=
%\ctjwpn \in\dTLnp
%\xeee
%%
%The semi-bounded complex is a projector:
%$(\ctjwpn)^2\xhte\ctjwpn$ and its composition with $\ccapni$ is
%contractible: $\ccapni\ctjwpn\xhte 0$. The complex $\ctjwpn$ is \qpb\ and it categorifies the
%\JWp:
%%$\Kz$ maps it to the \JWp:
%$\Kz(\ctjwpn) = \jwpn$.
%\end{theorem}

%\section{Cylindrical braids are \xgd}

\section{Elementary homological calculus}
\label{s:elhomcal}

\subsection{Limits in the category of complexes}

Consider a category $\xChA$ of \chcpls\ associated with
an additive category $\xctA$.
An $i$-th \emph{\trnc} $\xtrniv{\xbA}$ of a \chcpl\ $\xbA$ is
the \chcpl\
%$\cdots\rightarrow\xAimo\xrightarrow{\xdimo}\xAi$.
$\xAmi\xrightarrow{\xdmi}\xAmio\rightarrow\cdots$.
An
\trnci\ of a \chmp\ $\xbf$ is defined similarly.

Define an \emph{\isor}
$\ysiobf$
of a chain map $\xbA\xrightarrow{\xbf}\xbB$  as
%$\ysiobf = \xsupv{i\colon\xtrniv{\xbf}\quad\text{is an
%isomorphism}}$.
the largest number $i$ for which a truncated \chmp\ $\xtrniv{\xbf}$ is an
isomorphism of truncated complexes.

\begin{remark}
\label{rm:cnord}
Consider a \dstt\rx{ae1.ch2}.
If $\xbA\hteqv\Ohpm$, then $\ysiov{\idlbf}\geq m-1$.
\end{remark}

\begin{definition}
A \chsq\
$\scA = (\xbAo\xrightarrow{\xbfo}\xbAt\xrightarrow{\xbft}\cdots)$ is
\emph{\stblz} if
$\lim_{i\rightarrow\infty} \ysiobfi=\pinft$.
\end{definition}

\begin{definition}
A \chsq\ $\scA$ has a \tchlm\ $\chlm\scA=\xbA$ if there exist
\chmp s $\xbAi\xrightarrow{\xbtfi}\xbA$ such that
$\xbtfi = \xbtfio\,\xbfi$ and $\lmii \ysiorv{\xbtfi} = \pinft$.
\end{definition}

The following two theorems are easy to prove:
\begin{theorem}
A \chsq\ has a \tchlm\ if and only if it is \stblz.
%A \tchlm\ of a \stblz\ \chsq\ is unique.
If a \tchlm\ exists then it is unique.
\end{theorem}

%Then there exists a
%unique (up to an isomorphism) \chcpl\ $\xbA$ such that for any
%$N>0$ there exists $N\p$ such that for any
%%$i\geq N$ and any
%$i\geq N\p$ there is an isomorphism of \trnd\ complexes
%$\xtrnNv{\xbAi}\cong\xtrnNv{\xbA}$. In this case we use a notation
%%$\lmii\xbAi
%$\chlm\scA= \xbA$.
%
%The limiting complex $\xbA$ comes with a special sequence of \chmp s
%$\xbAi\xrightarrow{\xbtfi}\xbA$ ($\xbtfi = \cdots\xbtfio\xbtfi$) with the
%property $\xbtfi = \xbtfio\,\xbfi$.

\begin{theorem}
\label{pr:fnctchlm}
Suppose that $\chlm\scA = \xbA$.
Then for a complex $\xbB$ and \chmp s
$\xbAi\xrightarrow{\ybgi}\xbB$ such that $\ybgi = \ybgio\xbfi$,
 there exists a unique \chmp\ $\xbA\xrightarrow{\ybg}\xbB$
such that $\ybgi = \ybg\,\xbtfi$.
\end{theorem}

\begin{definition}
\label{df:chlmmp}
A sequence of \chmp s $\xbA\xrightarrow{\xbfz,\xbfo,\cdots}\xbB$
has a \tchlm\ $\lmii\xbfi = \xbf$ if for any $N$ there exists $N\p$
such that $\xtrnNv{\xbfi} = \xtrnNv{\xbf}$ for any $i\geq N\p$.
\end{definition}

\subsection{Limits in the homotopy category}

Definitions\rw{df:cauchy} and\rw{df:sqlm} extend the notion of a
\stblz\ \chsq\ and its limit to the homotopy category $\xKhA$:
obviously, a \stblz\ \chsq\ is \Cch, while $\chlm\scA=\xbA$ implies
$\dlm\scA=\xbA$.

\begin{proposition}
\label{pr:chlm}
A \Csq\ has a limit.
\end{proposition}
\proof
Consider a \Csq\ $\scA$. We construct a special \chcpl\
$\xbAs$ such that $\dlm\scA=\xbAs$ in accordance with
Definition\rw{df:sqlm}. Roughly speaking, we take $\xbAz$ and
attach to it the cones $\Cnbfi$ represented by homologically small
complexes, one by one. The result is a sequence $\scAs=\xbApz,\xbApo,\ldots$ of \stblz\
complexes $\xbApi$ such that $\xbApi\hteqv\xbAi$, and
$\xbAs=\chlm\scAs$ is their \tchlm.

Here is a detailed explanation.
By Definition\rw{df:cauchy}, there exist complexes $\ybCi$ such
that
\xlee{ae1.10a1}
\Cnv{\xbfi} \hteqv \ybCi[1] = \Ohpmi,\qquad\lmii m_i=+\infty.
\xeee
The complexes $\xbAi$, $\xbAio$ and $\ybCi$ form exact triangles:
\ylee{ae1.10a2}
\xymatrix{
\ybCi \ar[r]^-{\chdlbfi} & \xbAi \ar[r]^-{\xbfi} & \xbAio \ar[r] &
\ybCi[-1]
}
\yeee
and $\xbAio \hteqv \Cnv{\chdlbfi}$. We define recursively a new sequence
of complexes $\scAp = (\xbApz \xrightarrow{\idlbgz} \xbApo\xrightarrow{\idlbgo}\cdots)$ by
the relations $\xbApz = \xbAz$, $\xbApi\hteqv\xbAi$ and
$\xbApio = \Cnv{\ybgi}$, where the \chmp\
$\ybCi\xrightarrow{\ybgi}\xbApi$ is homotopy equivalent to the
\chmp\ $\chdlbfi$. In other words,
\xlee{ae1.10g1}
\xbApio = \Cnv{\ybCi\xrightarrow{\ybgi}
\underbrace{
\Cnv{\ybCimo\xrightarrow{\ybgimo}\cdots\xrightarrow{\ybgt}
\underbrace{
\Cnv{\ybCo\xrightarrow{\ybgo}
\underbrace{
\Cnv{\ybCz\xrightarrow{\chdlbfz}\xbAz
}}_{\xbApo}\;
}
}_{\xbApt}\;
}
}_{\xbApi}\;
}
\xeee

According to Remark\rw{rm:cnord}, $\ysiov{\idlbgi}\geq m_i$,
hence the sequence
$\scAp$ is \stblz, so there exists a chain limit
$\chlm\scAp = \xbAs$ and consequently there is a limit
$\dlm\scA=\xbAs$.\qed

Simply saying, the complex $\xbAs$ is an infinite \mtcn\ extension
of the complex\rx{ae1.10g1}:
\xlee{ae1.10g2}
\xbAs =
\cdots\xrightarrow{\ybgh}\Cnv{\ybCt
\xrightarrow{\ybgt}
\Cnv{\ybCo\xrightarrow{\ybgo}
\Cnv{\ybCz\xrightarrow{\chdlbfz}\xbAz
}
}
}.
\xeee

For our applications it is important to express $\Cnbtfz$ in terms
of complexes $\ybCi$. This can be done by rearranging
the infinite \mtcn\rx{ae1.10g2} with the help of associativity of
cone formation, which exists even within the category $\xChA$:
%
%The associativity of cone formation even within the category
%$\xChA$ allows us to reshuffle the infinite \mtcn\rx{ae1.10g2}:
%
\xlee{ae1.10h1}
\xbAs = \Cnv{\ybtC\xrightarrow{\ybtg}\xbAz},\qquad
\ybtC =\cdots\xrightarrow{\ybht}
\Cnv{\ybCt[1]\xrightarrow{\ybho}\Cnv{\ybCo[1]\xrightarrow{\ybhz}\ybCz}},
%\qquad
%\xbtfz \hteqv \idlv{\ybtg},
\xeee
so that $\xbtfz \hteqv \idlv{\ybtg}$, and $\Cnbtfz\hteqv\ybtC[-1]$
is expressed in terms of complexes $\ybCi$
arranged into an infinite \mtcn\ $\ybtC$. Here is a more formal
statement.
\begin{theorem}
\label{th:rshfl}
For a \Csq\ $\scA$ there exists another \Csq\
$\yctC = (\ybCz
\xrightarrow{\ybhpz} \ybtCo\xrightarrow{\ybhpo}\cdots)$
%, $\ybtCz=\ybCz$,
and \chmp s
$\ybCi[1]\xrightarrow{\ybhi} \ybtCi$ such that
$\Cnv{\ybhi}=\ybtCio$, $\ybhpi = \idlv{\ybhi}$ and for the
limiting complex $\ybtC=\chlm\yctC$ there exists a \chmp\
$\ybtC\xrightarrow{\ybtg}\xbAz$ such that $\xbAs = \Cnv{\ybtg}$,
$\;\xbtfz \hteqv \idlv{\ybtg}$ and consequently $\Cnbtfz\hteqv\ybtC[-1]$.
\end{theorem}

\proof
%The cones $\Cnbtfi$ of the \chmp s $\xbtfi$ which come with the limit $\lim\scA
%= \xbAs$ can be presented as \mtcn s built of the complexes $\ybCi$ of
%\ex{ae1.10a1}. We will consider the case of $i=0$ (others are similar).
%
%The \mtcn\ presentation of $\Cnbtfz$ comes from the associativity of
%taking a cone even within the category $\xChA$. Consider a general
%case first.
%
Let us recall the associativity of cones in a general setting.
For a \chmp\ $\xbA\xrightarrow{\xbf}\xbB$,
a \chmp\ $\xbC\xrightarrow{\ybg} \Cnbf$
%where $\xbf$ is a \chmp\ $\xbA\xrightarrow{\xbf}\xbB$,
is a sum: $\ybg = \ybgA \oplus\ybgB$
\ylee{ae1.10f1}
\xymatrix{
 & \xbA\ar[d]^-{\xbf}
\\
\xbC \ar[ur]|{[1]}^-{\ybgA} \ar[r]_-{\ybgB}
& \xbB
}
\yeee
where $\xbC\xrightarrow{\ybgA}\xbA[-1]$ is a \chmp\ and
$\xbC\xrightarrow{\ybgB}\xbB$ is a \mmp. Now it is obvious
that
\xlee{ae1.10f2}
\Cnv{\xbC\xrightarrow{\ybg}\Cnv{\xbA\xrightarrow{\xbf}\xbB}}
=\Cnv{\Cnv{\xbC[1]\xrightarrow{\ybgA}\xbA}\xrightarrow{\ybgB\oplus\xbf}\xbB
}.
\xeee

We apply the associativity relation\rx{ae1.10f2} to
\mtcn s\rx{ae1.10g1} consecutively for $i=1,2,\ldots$ in order to rearrange
them, so that $\xbApi = \Cnv{\ybtCi\xrightarrow{\ybtgi}\xbAz}$,
while the complexes $\ybtCi$ and \chmp s $\ybtgi$ are defined
recursively: $\ybtCz=\ybCz$, $\ybtgz = \chdlbfz$, $\ybtCio = \Cnv{\ybhi}$,
while the \chmp s $\ybCi[1]\xrightarrow{\ybhi} \ybtCi$ and
$\ybtCio\xrightarrow{\ybtgio}\xbAz$
are defined by applying the associativity
relation\rx{ae1.10f2} to the  double cone on the second line of
the following equation:
\begin{equation}
\label{ae1.10f3}
\begin{split}
\xbApio & = \Cnv{\ybCi\xrightarrow{\ybgi}\xbApi}
\\
& = \Cnv{\ybCi\xrightarrow{\ybgi}\Cnv{\ybtCi\xrightarrow{\ybtgi}\xbAz} }
\\
& = \Cnv{\Cnv{\ybCi[1]\xrightarrow{\ybhi} \ybtCi
} \xrightarrow{\ybtgio}
\xbAz
}
\\
& = \Cnv{\ybtCio\xrightarrow{\ybtgio}\xbAz}.
\end{split}
\end{equation}
Distinguished triangles
\ylee{ae1.10f3}
\xymatrix{
\ybCi[1]\ar[r]^-{\ybhi}
&
\ybtCi \ar[r]^-{\idlbhi}
&
\ybtCio \ar[r]
&
\ybCi
}
\yeee
determine  \chmp s $\ybhpi=\idlbhi$ of the \chsq\  $\yctC = (\ybtCz
\xrightarrow{\ybhpz} \ybtCo\xrightarrow{\ybhpo}\cdots)$. By
Remark\rw{rm:cnord} it has a \tchlm: $\chlm\yctC =
\ybtC$, which is an infinite \mtcn:
\ylee{ae1.10f4}
\ybtC =\cdots\xrightarrow{\ybht}
\Cnv{\ybCt[1]\xrightarrow{\ybho}\Cnv{\ybCo[1]\xrightarrow{\ybhz}\ybCz}}.
\yeee
The \chmp s $\ybtCi\xrightarrow{\ybhpi}\ybtCio$ satisfy a relation
$\ybtgi = \ybtgio\,\ybhpi$, so by Theorem\rw{pr:fnctchlm} there
exists a unique \chmp\ $\ybtC \xrightarrow{\ybtg}\xbAz$ such that
$\ybtgi=\ybtg\,\ybhtpi$.
It is easy to show that $\xbAs = \Cnv{\ybtC\xrightarrow{\ybtg}\xbAz}$,
and $\xbtfz = \idlv{\ybtg}$, hence $\Cnbtfz\hteqv \ybtC$.\qed

It is easy to prove the analog of Theorem\rw{pr:fnctchlm}:
\begin{theorem}
\label{pr:spmp}
For a complex $\xbB$ and \chmp s
$\xbAi\xrightarrow{\ybgi}\xbB$ such that $\ybgi \hteqv \ybgio\xbfi$,
 there exists a unique (up to homotopy) \chmp\ $\xbAs\xrightarrow{\ybg}\xbB$
which forms commutative triangles
\xlee{ae1.10f1}
\cmtr{\xbtfi}{\ybg}{\ybgi}{\xbAi}{\xbAio}{\xbB}
\xeee
%
%such that
%$\ybgi \hteqv \ybg\,\xbtfi$.
\end{theorem}

In order to complete the proof
of Theorems\rw{th:lmt} and\rw{th:lmt2},
we need two simple propositions. The first one establishes a
triangle inequality for homological orders of cones.
\begin{proposition}
If three \chmp s form a commutative triangle
\xlee{ae1.10c1}
\xymatrix{
\xbA \ar[r]_-{\xbfAB} \ar@/^1pc/[rr]^-{\xbfAC} &
\xbB \ar[r]_-{\xbfBC} &
\xbC
},\qquad \xbfAC\hteqv \xbfBC\xbfAB.
\xeee
then the homological orders of their cones satisfy the inequalities
\begin{align}
\label{eq:in1}
\yordch{\xbfAB}\geq \min\big(\yordch{\xbfAC},\yordch{\xbfBC}-1\big),
\\
\label{eq:in2}
\yordch{\xbfBC}\geq \min\big(\yordch{\xbfAB}+1,\yordch{\xbfAC}\big).
\end{align}
\end{proposition}
\proof
If \chmp s form a commutative triangle\rx{ae1.10c1}, then their cones form a \dstt
\ylee{ae1.10c2}
\Cnv{\xbfAB}\xrightarrow{\ybgo} \Cnv{\xbfAC} \xrightarrow{\ybgt} \Cnv{\xbfBC}
\xrightarrow{\ybgh}\Cnv{\xbfAB}[1],
\yeee
so the first inequality follows from the relation
$\Cnv{\xbfAB}\hteqv\Cnv{\ybgt}[1]$ and the second inequality
follows from the relation $\Cnv{\xbfBC} \hteqv \Cnv{\ybgo}$.\qed

The second proposition says that if a complex is homologically
infinitely small then it is contractible.
\begin{proposition}
\label{pr:ismc}
If $\yordh{\xbA} = +\infty$ then $\xbA$ is contractible.
\end{proposition}
\proof Since $\yordh{\xbA} = +\infty$, there exist complexes
$\xbAi\hteqv\xbA$, such that $\xbAi=\Ohpmi$ and $\lmii m_i =
+\infty$. Consider a sequence of \chmp s establishing homotopy
equivalence between the complexes:
\ylee{ae1.10c3}
\xymatrix@C=0.5cm{
\xbA \ar@<0.6ex>[r]^-{\xbfz}
&
\xbAo
\ar@<0.3ex>[l]^-{\ybgz}
\ar@<0.6ex>[r]^-{\xbfo}
&
\xbAt
\ar@<0.3ex>[l]^-{\ybgo}
\ar@<0.6ex>[r]%^-{\xbft}
&
\cdots
\ar@<0.3ex>[l]
\ar@<0.6ex>[r]
&
\xbAi
\ar@<0.3ex>[l]
\ar@<0.6ex>[r]^-{\xbfi}
&
\xbAio
\ar@<0.3ex>[l]^-{\ybgi}
\ar@<0.6ex>[r]
&
\cdots
\ar@<0.3ex>[l]
}
,\qquad
\yIdAi-\ybgi\xbfi =
%\xbdi\, \xbhi + \xbhi\, \xbdi,
\atcmv{\xbdi}{\xbhi},
\yeee
where $\yIdAi$ is the identity
\chmp\ of $\xbAi$, while $\xbAi[1]\xrightarrow{\xbhi}\xbAi$ is a homotopy \chmp\ (it does
not commute with the chain differential $\xbdi$ in the complex $\xbAi$).

Consider the compositions $\xbhfi = \xbfi\cdots\xbfo\xbfz$,
$\xbhgi = \ybgz\ybgo\cdots\ybgi$ and $\xbhhi =
\xbhgimo\,\xbhi\,\xbhfimo$. It is easy to see
that $\xbhgimo\,\xbhfimo - \xbhgi\,\xbhfi = \atcmv{\xbd}{\xbhhi}$,
hence $\yIdA - \xbhgi\,\xbhfi = \atcmv{\xbd}{\xbchi}$, where
$\xbchi = \xbhhz + \xbhho +\cdots + \xbhhi$.
There is a limit (\cf Definition\rw{df:chlmmp}) $\lmii\xbchi = \xbch$, while
$\lmii\xbhgi\,\xbhfi = 0$, hence $\yIdA = \atcmv{\xbd}{\xbch}$
which means that the complex $\xbA$ is contractible.
\qed

\begin{proposition}
\label{pr:lmch}
If a \chsq\ $\scA$ has a limit, then it is \Cch.
\end{proposition}
\proof
The inequality\rx{eq:in1} applied to the commutative
triangle\rx{ae1.10e1} says that
$$\yordch{\xbfi}\geq
\min\lrbc{\zordch{\xbtfi},\zordch{\xbtfio}-1},$$
hence the limit $\lmii\zordch{\xbtfi} = +\infty$ implies the \Cch\
property of $\scA$.

\begin{proposition}
\label{pr:lmun}
If a \chsq\ $\scA$ has a limit then it is unique.
\end{proposition}
\proof
If $\scA$ has a limit then by Proposition\rw{pr:lmch} it is \Cch.
Hence it has a special limit $\xbAs$ described in the
proof of Proposition\rw{pr:chlm}. If $\scA$ has another limit
$\xbAp$ with \chmp s $\xbAi\xrightarrow{\xbtfpi}\xbAp$ then by
Theorem\rw{pr:spmp} there is a \chmp\
$\xbAs\xrightarrow{\ybg}\xbAp$ with commutative
triangles\rx{ae1.10f1}. The inequality\rx{eq:in2} says
\ylee{ae1.10c3}
\yordch{\ybg}
\geq \min\big(\zordch{\xbtfi}+1,\zordch{\ybgi}\big).
\yeee
Since both cones in the \rhs become homologically infinitely small
at $i\rightarrow +\infty$, the cone $\Cnv{\ybg}$ is also
homologically infinitely small. Then Proposition\rw{pr:ismc} says
that $\Cnv{\ybg}$ is contractible and as a result
$\xbAp\hteqv\xbAs$.\qed

We end this section with a theorem which follows easily from
Definition\rw{df:sqlm}.
\begin{theorem}
\label{th:zl}
If a \chsq\ $\scA$ satisfies the property $\lmii\yordh{\xbAi} =
+\infty$ then its limit is contractible: $\dlm\scA = 0$.
\end{theorem}

%\section{The properties of the universal complex associated with a
%\cbr}
\section{A \chsq\ of
categorification complexes of  \cbr s} % universal complexes associated with \cbr s}
\label{s:cbr}
\subsection{A special categorification complex of a \ngbr}

Let $\xsgi$ denote an elementary negative $n$-strand braid:
\ylee{ae2.1}
%\underbrace{
\xsgi=\xygraph{
!{0;/r1.5pc/:}
[r(0.25)u(0.5)]
!{\xcapv@(0)}
[u(0.5)r(1)]
*{\cdots}
[r(01)u(0.5)]
!{\xcapv@(0)}
[r(0.5)u(1)]
%!{\vcap}
!{\vcross}
[r(1.5)u(1)]
!{\xcapv@(0)}
[u(0.5)r(1)]
*{\cdots}
[r(01)u(0.5)]
!{\xcapv@(0)}
[d(0.5)l(3.5)]
*{\scriptstyle{i}}
[r(1)]
*{\scriptstyle{i+1}}
[l(3.5)]
*{\scriptstyle{1}}
[r(6)]
*{\scriptstyle{n}}
}
\yeee

\begin{theorem}
\label{th:prop}
If an $n$-strand braid $\brb$ can be presented as a product of elementary
negative braids: $\brb = \xsgiv{k}\cdots\xsgiv{2}\xsgiv{1}$, then
its categorification complex has a special presentation $\cbrbs$:
\xlee{aea2.1}
\cbrbas =
\Big(\ldots\rightarrow\xCmt\rightarrow\xCmo\rightarrow\cidbrn\Big)
%\tgrshv{-\vthf}{\vthf}{-\vthf}^k,
\xeee
such that the complex
\xlee{aea2.2}
\xbC = (\ldots\rightarrow\xCmt\rightarrow\xCmo)\tgrsshv{-1}{1}
\xeee
is
\odct\ and \otbl.

%\begin{remark}
%\label{rm:spblx}
%The special complex $\cbrbas$ is  \otbl.
%\end{remark}

%$\cbrbs$
%such that $\cbrbs\tgrshv{\vthf}{-\vthf}{\vthf}^k$ is \oldn\ and
%\otbl.

%%
%\xlee{ae2.2}
%\cbrbs =
%%\tgrshv{\vthf}{\vthf}{-\vthf}^k=
%\lrbc{\cdots\rightarrow\xCi\rightarrow\cdots\rightarrow\xCv{1}\rightarrow\cIdbn}
%\tgrshv{-\vthf}{-\vthf}{\vthf}^k,
%\xeee
%%
%such that the multiplicities $\cjilam$ in the formula\rx{ae1.8a} for
%$\xCi$ satisfy the following properties:
%%
%\begin{enumerate}
%\item $\xcv{j,i,\xIdbn}=0$ for $i\geq 1$;
%\item $\cjilam=0$ for $j<i$ and for $j>2i$.
%\end{enumerate}
%%
\end{theorem}

More abstractly, the theorem says that there exists a
\odct\ and \otbl\ complex $\xbC$ and a \chmp\
$\xbC\rightarrow\cidbrn$ such that $\cbrba \hteqv
\CnBv{\xbC\qsho\rightarrow\cidbrn}$.

\begin{remark}
\label{rm:spbl}
Theorem\rw{th:prop} implies that the special complex
$\cbrbas$ is  \otbl.
\end{remark}

\pr{Theorem}{th:prop}
Let $\xlam$ be a \TL\ \ttngnn. Fix $i$ such that $1\leq i\leq n-1$. If the
composition $\gcapni\tcmp\xlam$ does not contain a disjoint circle,
then, in accordance with \ex{ae1.7},
we define the special categorification complex of $\xsgi\tcmp\xlam$ as
%%
%\xlee{ae2.3}
%\spcc{\symbcat{\xsgi\xlam} } =
%\Big(\symbcat{\xUni\,\xlam}\tgrshv{\vthf}{\vthf}{-\vthf}
%\rightarrow \dlam\tgrshv{-\vthf}{-\vthf}{\vthf} \Big).
%\xeee
%%
%
\xlee{ae2.3}
%\spsymbcat{\xsgi\xlam}\sht
\symcatps{\xsgi\tcmp\xlam}  =
\Big(\symbcat{\xUni\tcmp\xlam}\tgrshv{1}{-1}{1}
\rightarrow \dlam \Big)
%\tgrshv{-\vthf}{-\vthf}{\vthf}.
\xeee
%
%Here we factored out the degree shift associated
%with the number of crossings in the braid expression.
If $\gcapni\tcmp\xlam$ contains a disjoint circle, then $\xlam$ must
have the form $\gcupni\tcmp\xlamp$. Hence
$\xsgi\tcmp\xlam=\xsgi\tcmp\gcupni\tcmp\xlamp$. The tangle $\xsgi\tcmp\gcupni$
is the same as $\gcupni$ with a positive framing twist, so
according to \ex{ae1.8},
$\bsymcat{\xsgi\tcmp\gcupni} = \ccupni \tgrshv{\vthh}{-\vthf}{-\vthf}$.
Hence in this case we define the special categorification complex
of $\xsgi\tcmp\xlam$ simply as shifted $\dlam$:
\xlee{ae2.4}
%\spsymbcat
\symcatps{\xsgi\tcmp\xlam} = %\sht=
%\Big(
\dlam
\tgrshv{2}{-1}{0}.
%\Big).
%\tgrshv{-\vthf}{-\vthf}{\vthf}
\xeee
%
%and we have again factored out the crossing related degree shift.

Now we define a recursive algorithm for constructing the complex
$\cbrbas$. For $\brb = \gidbrn$ we define $\cbrbas = \cidbrn$. Let
$\brb = \xsgiv{k}\tcmp\cdots\tcmp\xsgiv{1}$ and
suppose that we have defined its special complex $\cbrbas$. We
define the special categorification complex of a
braid $\brbp=\xsgikpo\tcmp\brb$ by applying the rules\rx{ae2.3}
and\rx{ae2.4} to all constituent tangles $\xlam$ in the complex
$\cbrbs$ (see the formula\rx{ae1.8a}).

We prove the properties of $\cbrbas$ by induction over $k$.
If $k=0$ then $\brb = \gidbrn$ and the properties of $\cbrbas$ are
obvious.

Suppose that the special categorification complex
$\cbrbas$ of a braid $\brb = \xsgiv{k}\tcmp\cdots\tcmp\xsgiv{1}$
has the form\rx{aea2.1} and its tail\rx{aea2.2} is \odct\ and
\otbl.
Consider
%the special categorification complex of
a longer braid
$\brbp=\xsgikpo\tcmp\brb$. The object $\cidbrn$ may appear in $\cbrbpas$ if
and only if $\xlam=\gidbrn$ and the extra crossing $\xsgikpo$ is
\nsplcd\ in \ex{ae2.3}, hence
%$\cbrbps$ is \oldn.
$\cbrbpas$ has the form\rx{aea2.1} and its tail\rx{aea2.2} is
\odct.

If the negative crossing $\xsgikpo$ is composed with the head
$\cidbrn$ of the complex\rx{aea2.1}, then the formula\rx{ae2.3}
applies and the tangle $\xUv{n}{i_k+1}$ appearing in the tail of
$\cbrbpas$ satisfies the property\rx{ae2.m1}.

%resulting \TL\ tangles satisfy the condition

If the crossing $\xsgikpo$ is composed with a
\TL\ tangle $\xlam$ from the $(-i)$-th \qcmd\ $\xCmi$ (see \ex{ae1.8a})
in the tail of the complex $\cbrbas$ with the $q$-degree shift $j$
satisfying the inequality $i-1 \leq j-1 \leq 2(i-1)$, then the
shifted objects in the \rhs of \eex{ae2.3} and\rx{ae2.4}
also satisfy this inequality.\qed

The picture\rx{ae1.10p} presents a \cbr\ as a product of negative
crossings, hence
\begin{corollary}
\label{cr:otbl}
%A crossing-shifted special categorification complex of a \cbr\
%has the form
A \cbr\ $\gbrmn$ has a special \otbl\ categorification complex
$\cbrmnps$. In particular, for $m=1$
\xlee{ae2.5}
%\btcylcsnmns
\cbronps
 = \CnBv{\xbCon\qsho\rightarrow\cidbrn},
\xeee
where the complex
$\xbCon$ is \odct\ and \otbl.
%is \oldn\ and \otbl.
\end{corollary}
%\begin{remark}
%\label{rm:1}
%The formula\rx{ae2.5} indicates that the complex $\cbrmns$
%%$\btcylcsnmns$
%itself is \otbl.
%\end{remark}

\subsection{Special morphisms between \cbr\ complexes}
\label{ss:brchsq}

Relation\rx{ae2.5}  indicates that there is a \dstt\
%
%\ylee{ae2.6x}
%\xymatrix{
%\xbCon\qsho \ar[r] &
%{}\cIdbn \ar[r]^-{\mrfo} &
%%\btcylcsnons
%{}\btcylcnns \ar[r] &
%\xbCon\tgrsshv{1}{1}
%}
%\yeee
%
%
\ylee{ae2.6}
\xbCon\qsho \longrightarrow
\cidbrn \xratv{\mrfo}
\cbrons \longrightarrow
\xbCon\tgrsshomo
\yeee
and
\xlee{ae2.6a}
\Cnv{\mrfo} \hteqv \xbCon\tgrsshomo.
\xeee
%
%If we `tangle-compose'
%the morphism $\mrfo$ with the identity morphism
%$\btcylcnmns\xrightarrow{=}\btcylcnmns$
%
Composing both sides of the morphism $\mrfo$ with
the \cbr\ complex
%$\btcylcnmns$,
$\cbrmns$,
we get a morphism
%%
%\ylee{ae2.7x}
%\xymatrix{
%\btcylcnmns \ar[r]^-{\mrfm} & \btcylcnmpons
%}
%\yeee
%%
%
\ylee{ae2.7}
\cbrmns \xratv{\mrfm}\cbrmons
\yeee
such that
\xlee{ae2.8}
\Cnv{\mrfm} \hteqv \Cnv{\mrfo}\tcmp\cbrmns.
%\hteqv. \xbCon\,\btcylcnmns\tgrsshv{1}{1}.
\xeee
\begin{theorem}
\label{th:2.1}
The cone\rx{ae2.8} can be presented by a shifted complex
\ylee{ae2.9}
\Cnv{\mrfm} \hteqv \xbCmn \tgrsshnontm\tgrsshomo,
%\tgrsshv{2mn+1}{2m(n-1)+1},
\yeee
%
%$\xbCmn\tgrsshv{}{}$
such that $\xbCmn$ is \odct\ and \otbl.
\end{theorem}

The proof is based on a simple geometric lemma:
\begin{lemma}
\label{l:1}
For $n\geq 2$, the following two compositions of framed tangles are isotopic:
%%
%\xlee{ae2.bx}
%\xcapni (\btcyln)^n = (\btcylv{n-2})^{n-2}\;\xcapnit
%\xeee
%%
%
\xlee{ae2.b}
\gcapni \tcmp\;\gbron = \gbronmt\;\tcmp\gcapnit
\xeee
where $\gcapnit$ is the tangle $\gcapni$ with double framing twist
on the cap:
\ylee{ae2.10}
\gcapnik=
\xygraph{
!{0;/r1.5pc/:}
[r(0.25)u(0.5)]
!{\xcapv@(0)}
[u(0.5)r(1)]
*{\cdots}
[r(01)u(0.5)]
!{\xcapv@(0)}
[r(0.5)]
!{\vcap}
[r(1.5)u(1)]
!{\xcapv@(0)}
[u(0.5)r(1)]
*{\cdots}
[r(01)u(0.5)]
!{\xcapv@(0)}
[d(0.5)l(3.5)]
*{\scriptstyle{i}}
[r(1)]
*{\scriptstyle{i+1}}
[l(3.5)]
*{\scriptstyle{1}}
[r(6)]
*{\scriptstyle{n}}
[l(3)u(1)]
*{\symfr}
[u(0.5)]
*{\scriptstyle{k}}
}
\yeee
\end{lemma}
\proof
This lemma is geometrically obvious: a cap on a pair of adjacent strands slides down
through the \cbr\ to the
bottom.\qed

An immediate corollary of \ex{ae2.b} and of the framing change
formula\rx{ae1.8} is the following relation:
%%
%\xlee{ae2.11x}
%\ccapni \btcylcnmns\; \hteqv
%\Big(\btcylcnmnmts\Big)\;\ccapni
%\tgrsshv{n}{n-1}^{2m}.
%\xeee
%%
%
\xlee{ae2.11}
\bsymcats{\gcapni \tcmp\gbrmn} \hteqv \bsymcats{\gbrmnmt\;\tcmp\gcapni}
%\tgrsshv{n}{n-1}^{2m}.
\tgrsshnontm.
\xeee

%%%%%%%%%%%%%%%%%%%%%%%%%%%%%%

In order to prove Theorem\rw{th:2.1}, we need three simple
propositions.
For a positive integer $d\leq \frac{n}{2}$,
let $\stI=(i_1,\ldots,i_d)$ be a sequence of positive integer
numbers such that $i_k<n-2k+2$ for all $k\in\{1,\ldots,d\}$.
A \emph{\aptg} $\gcapnI$
%, where $\stI = (i_1,\ldots,i_d)$ is a sequence of integer numbers,
%of \apdg\ $d$
is a $(n,n-2d)$-tangle which
can be presented as a product of $d$ tangles of the form
$\gcapv{m}{i}{0.75}$:
%%
%\ylee{aes2.1ax}
%\maptgnI = \xcapvv{n-2d+2}{i_d}\cdots
%\xcapvv{n-2}{i_2}\,\xcapvv{n}{i_1}.
%\yeee
%%
%
\ylee{aes2.1a}
\gcapnI =
\gcapv{n-2d+2}{i_d}{2}\tcmp\cdots\tcmp
\gcapv{n-2}{i_2}{1.5}
\tcmp
\gcapv{n}{i_1}{0.75}.
\yeee
A \emph{\uptg} $\gcupnI$ is defined similarly:
%%
%\ylee{aes2.2a}
%\muptgnI = \xcupvv{n-2d+2}{i_d}\cdots
%\xcupvv{n-2}{i_2}\,\xcupvv{n}{i_1}.
%\yeee
%%
%
\ylee{aes2.2a}
\gcupnI =
\gcupv{n}{i_1}{-0.75}
\tcmp
\gcupv{n-2}{i_2}{-1.25}
\tcmp
\cdots
\tcmp
\gcupv{n-2d+2}{i_d}{-2.25}
.
\yeee
The first proposition is obvious:
\begin{proposition}
\label{pr:3}
Every \TL\ \ttngnn\ $\xlam$ has a presentation
\xlee{aes2.3a}
\xlam = \gcupnIp\tcmp
%\gidbrnmtd\tcmp
\gcapnI,\qquad
\nI=\nIp.
%=\cpdlam.
%\nIp=\cpd.
\xeee
\end{proposition}
The number $\cpdlam=\nI=\nIp$ is determined by the tangle $\xlam$
and we call it  the \apdg\ (or \updg) of $\xlam$.
%: $\dgap\xlam = d$.

The second proposition is also obvious:
\begin{proposition}
\label{pr:1}
If at least one of two complexes $\xbCo$ and $\xbCt$ in $\dTLn$ is
\odct\ then their composition $\xbCo\tcmp\xbCt$ is \odct.
\end{proposition}
Note that even if
both complexes are \otbl,  then their composition is not necessarily
\otbl. Indeed, in contrast to the homological degree,
the \qdgr\ is not additive with respect to the composition of
tangles:
%
%although the homological degree is
%additive with respect to the composition of tangles, the \qdgr\ is
%not:
if  the composition of two \TL\ tangles contains a disjoint
circle then the \qdgr\ shifts of the rule\rx{ae1.01}
%creates the shifts of \qdgr\ which
violate additivity. However, if the upper tangle in the composition has no caps or the
lower tangle has no cups then no circles are created and the
\otblc\ is maintained:
\begin{proposition}
\label{pr:2}
If a complex $\xbC$ in $\dTLv{n-2\cpdlam}$ is \otbl, then the complexes
%$\dcupnti\xbC$
$\ccupnI\tcmp\xbC$
and
$\xbC\tcmp\ccapnI$
%$\xbC\;\ccapnti$
are also \otbl.
\end{proposition}

%%%%%%%%%%%%%%%%%%%%%%%%%%%%%%%%%

\pr{Theorem}{th:2.1}
In order to construct the \odct\ and \otbl\ complex $\xbCmn$, we
use the presentation
\xlee{ae2.12}
\Cnv{\mrfm} \hteqv
\xbCon\tcmp\cbrmns\tgrsshomo,
\xeee
which follows from \eex{ae2.8} and\rx{ae2.6a}.
We construct $\xbCmn$ by
simplifying the complexes
%$\dlam\btcylcnmn$
$\bsymcats{\xlam\tcmp\gbrmn}$
for \TL\ \ttngnn s $\xlam$
appearing in the \qcmds\ of $\xbCon$, with the help of the
relation\rx{ae2.11}, thus creating necessary degree shifts, and then
using Corollary\rw{cr:otbl} which says that emerging \cbr s have \otbl\ categorification
complexes.

%using the special complex
%$\cbrmnmtps$
%%$\btcylcsnmnmt$
%to represent $\cbrmnmts$.

%%%%%%%%%5

%%%%%%%%%%%%%%%%

Let
%$\dlam\qshj$
$\dlam\tgrsshjmi$
be an object appearing in the $(-i)$-th \qcmd\ of
$\xbCon$ with a non-zero multiplicity (we made its homological degree explicit by
including $-i$ in the shift).
%The complex $\xbCon$ is \otbl,
%hence $i\leq j\leq 2i$. The complex $\xbCon$ is also \odct,
%hence $\cpdlam\geq 1$.
%
We apply \ex{ae2.11} consequently to every cap $\gcapnk$ appearing
in the \aptg\ $\gcapnI$ in the presentation\rx{aes2.3a} of
$\xlam$:
\begin{multline}
\label{ae2.13}
\dlam\tgrsshjmi\tcmp\cbrmns
\\
\hteqv
\Bigg(
\ccupnIp\tcmp
%\btcylcnmnmtdl
%\btcylcsnmnmtdl\sht
\bsymcatps{\gbrv{m}{n-2\cpdlam}{2.5}}\tcmp\ccapnI
\tgrsshv{\alm}{-\blm}^{2m} \tgrsshjmi\Bigg)
%\tgrsshv{n}{n-1}^{2m},
\tgrsshnontm,
\end{multline}
where
\xlee{ae2.14}
\alm = \sum_{k=1}^{\cpdlam-1}(n-2k)
,\qquad
\blm = \sum_{k=1}^{\cpdlam-1}(n-2k-1).
\xeee

The object $\dlam$ comes from the \odct\ complex $\xbCon$,
hence $\cpdlam>0$ and
%the tangles $\muptgnIp$ and $\maptgnI$ are \odct, hence
the complex
in big brackets in the \rhs of \ex{ae2.13} is \odct\ in view of
Proposition\rw{pr:1}. Proposition\rw{pr:2} implies that the complex
$\ccupnIp\tcmp
\bsymcatps{\gbrv{m}{n-2\cpdlam}{2.5}}\tcmp\ccapnI$
%$\cmuptgnIp \btcylcsnmnmtdl\sht \cmaptgnI$
is also \otbl. Since $\dlam$
comes from the \otbl\ complex $\xbCon$, the numbers $i$ and $j$
satisfy inequalities $i\geq 0$ and $i\leq j\leq 2i$. It is easy to check that the numbers
$\alm$ and $\blm$ of \ex{ae2.14} satisfy the same inequalities:
$\blm\geq 0$, $\blm\leq \alm \leq 2\blm$, hence the complex in big
brackets in the \rhs of \ex{ae2.13} is also \otbl. The
complex
%$\xbCon\,\btcylcnmns$
$\xbCon\tcmp\cbrmns$
in the \rhs of \ex{ae2.12} is
composed of complexes\rx{ae2.13}, hence Theorem\rw{th:2.1} is proved.
%this proves Theorem\rw{th:2.1}.
\qed

\section{A categorified \JWp}
\label{s:prfs}

Consider the \chsq\rx{ae1.10c}. Theorem\rw{th:2.1} implies that
%$\Cnv{\mrfm} \hteqv \Ohp\big(2m(n-1) + 1\big)$,
$\yordhr{\Cnv{\mrfm}} \geq 2m(n-1)+1$,
hence $\xctBn$ is
\Cch\ and it has a unique limit $\dlm\xctBn =\ctjwpn \in\dTLnp$.

Now we prove Theorems\rw{th:enum} and Theorem\rw{th:cnpr} which
describe the properties of $\ctjwpn$.

\pr{Theorem}{th:cnpr}
%Apply Theorem\rw{th:rshfl} to
Consider the \chsq\rx{ae1.10c} truncated from below:
\ylee{eq:np1}
\xctBmn =
\Big(
\cbrmns \xraov{\mrfm}
\cbrmons \xrightarrow{\mrfmo}\cdots\Big)\longrightarrow\ctjwpn.
\yeee
According to Theorem\rw{th:rshfl}, the limit $\ctjwpn$ can be
presented as a cone\rx{ae2.m4}, where $\wbCmnp = \wbCmn\spshmnm$
and $\wbCmn$ is an infinite
\mtcn:
\begin{multline}
%\ylee{ae3.1}
\nonumber
\wbCmn =\cdots\rightarrow\Cnv{\xbCvn{m+k}\tgrsshv{2kn}{-2k(n-1)+1}
\rightarrow
\cdots
\\
\cdots
\rightarrow
\Cnv{
\xbCvn{m+1}\tgrsshv{2n}{-2n+3}\rightarrow\xbCmn}
}
\end{multline}
with \odct\ and \otbl\ complexes $\xbCmn$ introduced in
Theorem\rw{th:2.1}. Hence the complex $\wbCmn$ itself is \odct\ and
\otbl.\qed

%Let us apply Theorem\rw{th:rshfl} to the \chsq\rx{ae1.10c}
%truncated from below:
%
%\pr{Theorem}{th:cnpr}
%According to Theorem\rw{th:rshfl}, the limit $\ctjwpn$ can be
%presented as a cone $\ctjwpn\hteqv\CnBv{\zbCn\qsho\rightarrow
%\cidbrn}$ (we removed tildes from the notations of \ex{ae1.10h1}),
%where $\zbCn$ is an infinite \mtcn:
%%
%\ylee{ae3.1x}
%\zbCn =\cdots\rightarrow\Cnv{\xbCmn\tgrsshv{2mn}{2m(n-1)-1}
%\rightarrow
%\cdots\rightarrow
%\Cnv{
%\xbCvn{1}\tgrsshv{2n}{2n-3}\rightarrow\xbCvn{0}}
%}
%\yeee
%%
%with \odct\ and \otbl\ complexes $\xbCmn$ introduced in
%Theorem\rw{th:2.1}. Hence the complex $\zbCn$ itself is \odct\ and
%\otbl.\qed

\pr{part 1 of Theorem}{th:enum}
The tangle composition with $\ccapni$ is a `continuous'
functor, that is, it can be applied to both sides of
$\dlm\xctBn = \ctjwpn$, hence $ \dlm\;
\ccapni\tcmp\xctBn = \ccapni\tcmp\ctjwpn$. According to \ex{ae2.11},
\begin{equation}
\nonumber
%\label{ae3.2}
\begin{split}
\ccapni\tcmp\xctBn & = \Big(\ccapni\tcmp\cidbrn\rightarrow \cdots\rightarrow
\ccapni\tcmp\cbrmns\rightarrow\cdots\Big)
\\
& = \Big( \ccapni\rightarrow\cdots\rightarrow
\cbrmnmts\tcmp\ccapni
\spshmnm
%\tgrsshv{n}{n-1}^{2m}
\rightarrow\cdots
\Big).
\end{split}
\end{equation}
Since
\ylee{ae3.3}
\yordhb{\cbrmnmts\tcmp\ccapni
%\tgrsshv{n}{n-1}^{2m}}
\spshmnm}
= 2m(n-1)\xrightarrow[m\rightarrow +\infty]{}+ \infty,
\yeee
according
to Theorem\rw{th:zl}, $\dlm\ccapni\tcmp\xctBn=0$, hence
$\ccapni\tcmp\ctjwpn$ is contractible.\qed

\begin{remark}
The contractibility of $\ctjwpn\tcmp\ccupni$ is proved similarly.
\end{remark}

\begin{corollary}
\label{cr:odct}
If $\xbC$ is a \odct\ complex in $\dTLnp$, then $\xbC\tcmp\ctjwpn$ is
contractible.
\end{corollary}

\pr{part 2 of Theorem}{th:enum}
According to
%$\ctjwpn\hteqv\CnBv{\xbCn\qsho\rightarrow\cidbrn}$, where the
%complex $\xbCn$ is \odct. Then
\ex{ae2.m5},
\begin{multline}
\nonumber
%\ylee{ae3.4}
\ctjwpn\tcmp\ctjwpn \hteqv \CnBv{\wbCzn\qsho\longrightarrow\cidbrn}\tcmp\ctjwpn
\\
\hteqv\CnBv{\wbCzn\tcmp\ctjwpn\qsho\longrightarrow\cidbrn\tcmp\ctjwpn}
\hteqv\ctjwpn,
%\yeee
\end{multline}
where we used the fact that $\wbCzn$ is \odct\ and Corollary\rw{cr:odct} in order to establish the last
homotopy equivalence.\qed

\pr{Theorem}{th:alg}
The complexes $\ctjwpn$, $\wbCmn$ and $\cbrmns$ in \ex{ae2.m4} are
\otbl, hence they are \qpb\ and their $\Kz$ images are
well-defined. Applying $\Kz$ to this equation and taking into account \ex{eq:catKz} and
the definition\rx{ae1.10b1}, we find
\ylee{ae3.5}
\jwpn = q^{\vthf mn(n-1)}\abrmn - q^{2mn+1}\Kz(\wbCmn).
\yeee
The complex $\wbCmn$ is \otbl, so $\yordq{\Kz(\wbCmn)}\geq 0$ and
by Definition\rw{df:qlm} there is a limit\rw{ae1.9}.\qed

\section{The other projector}
A dual of an \ttngmn\ $\xtau$ is the \ttngnm\ tangle
$\xtaud$ which is its mirror image. Duality extends to an
isomorphism $\cTL \xrightarrow{\dsym} \cTLop$ combined with the
isomorphism of the ground ring $\Zqqi\xrightarrow{\dsym}\Zqqi$, such that
$\dsymv{q} = q^{-1}$. Furthermore, duality establishes an
isomorphism $\cTLpinf\xrightarrow{\dsym}\cTLminfop$, where
$\cTLminf$ is the analog of $\cTLpinf$ constructed over the ring
$\Zsqiq$ of Laurent series in $q^{-1}$.

Since the relations\rx{ae1.4} and\rx{ae1.4a} are dual to each other,
while the idempotency condition $\jwpn\tcmp\jwpn=\jwpn$ is duality
invariant, the uniqueness of the \JWp\ implies that it is duality
invariant: $\dsymv{\jwpn} = \jwpn$. Hence the corresponding
idempotents $\jwpnp\in\cTLpinf$ and $\jwpnm\in\cTLminf$ are also
dual to each other: $\jwpnm = \dsymv{(\jwpnp)}$. Taking the dual
of \ex{ae1.9} we find that $\jwpnm$ is the limit of \cbr s with
high positive (that is, \cclckw) twist:
\xlee{ae1.9b}
\lim_{m\rightarrow+\infty} q^{-\vthf mn(n-1)}\aobrmn = \jwpnm,
\xeee
because $\dsymv{\Big(\gbrmn\Big)} = \gobrmn$.

%The dual \JWp\
%$\jwpnm = \dsymv{(\jwpnp)}\in\cTLminfop$ is also an idempotent in
%$\cTLminf$, because the relation $\jwpnm\jwpnm=\jwpnm$ remains
%intact after the reversal of multiplication order.

Duality extends to a contravariant
equivalence functor $\dTL\xrightarrow{\dsym}\dTLop$, where
$\dTLop$ is the same category as $\dTL$, except that the
composition of tangles is performed in reversed order. The functor
$\dsym$ also switches the signs of all three gradings of $\dTL$.
Applying the duality functor to the construction of $\ctjwpn$ we
find that there exists a \chsq
\begin{multline}
\label{ae1.10f}
\xctBnd =
\Big(
\cidbrn
\rxratv{\dmrfz}
\cobrons \rxratv{\dmrfo}
\cdots
\\
\cdots
\xleftarrow{\dmrfmmo}
\cobrmns \rxraov{\dmrfm}
\cobrmons \xleftarrow{\dmrfmo}\cdots\Big),
\end{multline}
where $-\xspsh$ denotes the grading shift which is opposite
to\rx{ae1.10b1}.
The system\rx{ae1.10f} is
dual to the system\rx{ae1.10c}
and it has an inductive limit $\ilm \xctBnd =\ctjwmn $, which satisfies
projector properties
\ylee{ae1.10f1}
\ctjwmn\tcmp \ctjwmn \hteqv \ctjwmn,\qquad
\ccapni \tcmp\ctjwmn \hteqv \ctjwmn\tcmp \ccupni\hteqv 0
\yeee
and has a presentation
\ylee{ae1.10f2}
\ctjwpn \hteqv \CnBv{ \dsymv{\wbCmn}\ospshmnm\qshmo
\xrahv{\dsymv{\chdlbtfm}} \cobrmns},
\yeee
%
%where $\wbCmnp = \wbCmn\spshmn$,
where the complex $\wbCmn$ is
\odct\ and \otbl. In particular, at $m=0$ we get the dual of
presentation\rx{ae2.m5}:
\ylee{ae1.10f3}
\ctjwmn \hteqv \CnBv{ \dsymv{\wbCzn}\qshmo
\xrahv{\dsymv{\chdlbtfz}} \cidbrn},
\yeee
where the complex $\wbCzn$ is \odct\ and \otbl.

%$$ \testgr$$
%$$\testgr1 $$
%$$\testgr2 $$
%$$\gbrmnmt\qquad\gbrmn\qquad\tstgro\qquad\tstgrt$$

%$$
%\begindc{\commdiag}[50]
%\obj(0,1){$A$}
%\obj(1,2){\box{$\cidbrn$}}
%\enddc
%$$

%\subsection{Proof of Theorem\rw{th:enum}}

%$\accentset{\bullet}{\rTL}$

\end{document}